\documentclass[12pt]{article}
\usepackage[margin=24truemm]{geometry}
\usepackage{amsmath}
\usepackage{float}
\usepackage{amsmath,amssymb,bm}
\usepackage{mathtools}
\usepackage{ascmac}
\usepackage{array}
\usepackage{amsthm}
\usepackage[all]{xy}
\usepackage{graphics}
\usepackage{mathrsfs}

\usepackage[backend=biber, style=ieee-alphabetic, sorting=nyt, maxbibnames=99]{biblatex}
\theoremstyle{plain}
\newtheorem{thm}{Theorem}[section]
\newtheorem*{thm*}{Theorem}
\newtheorem{cor}[thm]{Corollary}
\newtheorem{lem}[thm]{Lemma}
\newtheorem{conj}[thm]{Conjecture}
\newtheorem{prop}[thm]{Proposition}

\theoremstyle{definition}
\newtheorem{dfn}[thm]{Definition}

\newtheorem{rem}[thm]{Remark}
\newtheorem{ex}[thm]{Example}

\newcommand{\Fl}{{\mathscr{F}\!\ell}}
\newcommand{\Isom}{\ {\stackrel{\sim}{\longrightarrow}}\ }
\newcommand{\etale}{{\mathrm{\acute{e}tale}}}
\newcommand{\et}{{\mathrm{\acute{e}t}}}

\newcommand{\proet}{{\mathrm{pro\acute{e}t}}}
\newcommand{\aS}{{\mathcal{S}}}

\newcommand{\Pla}{{\Psi-\mathrm{la}}}

\usepackage[OT2,T1]{fontenc}
\DeclareSymbolFont{cyrletters}{OT2}{wncyr}{m}{n}
\DeclareMathSymbol{\Sha}{\mathalpha}{cyrletters}{"58}

\begin{document}

\begin{center}

\textbf{On the classicality theorem and its applications to the automorphy lifting theorem and the Breuil-M$\mathrm{\acute{\textbf{e}}}$zard conjecture in some $\mathrm{GL}_2(\mathbb{Q}_{p^2})$ cases}

\vspace{0.5 \baselineskip}

Kojiro Matsumoto

\vspace{0.5 \baselineskip}

\end{center}    

\begin{abstract} In this paper, we study locally analytic vectors in the ``partially'' completed cohomology of Shimura varieties associated with some rank 2 unitary groups over a totally real field $F^+$ such that $F^+_v = \mathbb{Q}_{p^2}$ for some $p$-adic places $v$ and prove a certain classicality theorem. This is a partial generalization and modification of Lue Pan's work in the modular curve case by using the works of Caraiani-Scholze, Koshikawa and Zou on mod $l$ cohomology of Shimura varieties. As applications, we prove the automorphy lifting theorem and the Breuil-M$\mathrm{\acute{e}}$zard conjecture in some $\mathrm{GL}_2(\mathbb{Q}_{p^2})$ cases. We will assume a technical regularity condition on Serre weights of residual representations, but we don't assume any technical condition on the properties of liftings of residual representations at $p$-adic places except Hodge-Tate regularity. It should be noted that previously, such results were known only when we assumed that $F^+_v$ is equal to $\mathbb{Q}_p$ for any $p$-adic place $v$ of $F^+$ so that we can use the $p$-adic Langlands correspondence of $\mathrm{GL}_2(\mathbb{Q}_p)$. Moreover, we propose a conjectural strategy to prove such results in some $\mathrm{GL}_2(\mathbb{Q}_{p^f})$ cases.\end{abstract}

\vspace{0.5 \baselineskip}

\tableofcontents

\section{Introduction}

\subsection{Background}

Let $p \ge 5$ be a prime and $\iota : \overline{\mathbb{Q}_p} \Isom \mathbb{C}$ be an isomorphism of fields.

By \cite{SW, SEI, SEII, FMRI, EMLG, HUTU, FMRR}, there exists a natural one-to-one correspondence between the following two objects. \ ($G_{\mathbb{Q}}$ is the absolute Galois group of $\mathbb{Q}$ and $k$ is a positive integer. This is the Fontaine-Mazur conjecture for $\mathrm{GL}_{2} (\mathbb{Q})$.)

\vspace{0.5 \baselineskip}

Galois side : irreducible continuous odd\footnote{The oddness of a Galois representation $r$ means that the determinant character $\mathrm{det}r$ of $r$ satisfies $\mathrm{det}r(c) = -1$ for a complex conjugation $c$. } $2$-dimensional representations of $G_{\mathbb{Q}}$ over $\overline{\mathbb{Q}}_p$ which are unramified at almost all finite places and de Rham of Hodge-Tate weight $(0, k)$ at $p$.

\vspace{0.5 \baselineskip}

Automorphic side : normalized new cuspidal eigenforms of weight $k + 1$.

\vspace{0.5 \baselineskip}

Moreover, there are already many results which partially generalize the above correspondence to higher-dimensional cases and more general number fields. In particular, the map from the automorphic side to the Galois side was constructed in many cases (cf. \cite{Kot}, \cite{HT}, \cite{Shin}, \cite{HLTT}, \cite{Torsion} and \cite{Mok}) and the correspondence of deformations of both sides, i.e., the automorphy lifting theorem was known in many cases if we assume that the Galois representations have a nice property at $p$-adic places such as ordinarity or Fontaine-Laffaille at $p$.\footnote{Technically, we also have nice automorphy lifting theorems for Galois representations which are potentially diagonalizable at any $p$-adic place (cf. \cite{CW}).}   (cf. \cite{CHT}, \cite{TII}, \cite{OG}, \cite{CW}, \cite{10}, \cite{CN} and \cite{BCGP}.) However, there are many remaining problems and in particular, there are few results about the following problem in general. \footnote{The correspondence of residual representations is also widely open, but we don't study this problem in this paper.}

\vspace{0.5 \baselineskip}

(*) \ Automorphy lifting of Galois representations which are neither ordinary nor Fontaine-Laffaille at $p$.

\vspace{0.5 \baselineskip}

In the $\mathrm{GL}_2(\mathbb{Q})$ case, the problem (*) was solved in \cite{FMRI}, \cite{EMLG}, \cite{HUTU} and \cite{FMRR} by crucially using the $p$-adic Langlands correspondence of $\mathrm{GL}_2(\mathbb{Q}_p)$, which is still not known beyond the $\mathrm{GL}_2(\mathbb{Q}_p)$ case.

\vspace{0.5 \baselineskip}

Recently, \cite{PanII} introduced another approach to the problem (*) in the $\mathrm{GL}_2(\mathbb{Q})$ case by using certain geometric properties of perfectoid modular curves. In this paper, by modifying Pan's method, we study the problem (*) in the conjugate self-dual \footnote{The reason why we assume the conjugate self-duality is that we can use unitary groups in such cases.} $\mathrm{GL}_2(F)$ case, where $F$ is a CM field such that $F_v = \mathbb{Q}_{p^2}$ or $\mathbb{Q}_p$ for any $p$-adic places $v$ of $F$.

\begin{rem}
    
Note that many arguments of this paper work without any assumptions on $F_v$ for $v \mid p$. The author hopes that the method of this paper will be generalized at least if we assume appropriate conditions on the residual representation. In this paper, we crucially use the assumption that $F_v/\mathbb{Q}_p$ is unramified in {\S} 4 and the assumption $[F_v:\mathbb{Q}_{p}] \le 2$ for any $v \mid p$ to prove Conjectures \ref{classicality conjecture} (and \ref{key diagram} in the non-parallel weight case).

\end{rem}

\subsection{Main results}

Let $p \ge 5$ be a prime, $\iota : \overline{\mathbb{Q}_p} \Isom \mathbb{C}$ be an isomorphism of fields, $F$ be a CM field or totally real field such that $F_v = \mathbb{Q}_{p^2}$ or $\mathbb{Q}_p$ for any $p$-adic place $v$ and $F^+$ be the maximal totally real subfield of $F$. (Thus $F = F^+$ if $F$ is a totally real field.) Let $c$ denote the complex conjugation on $F$ (thus $c$ is trivial if $F$ is a totally real field) and $\Psi$ be the set of $p$-adic places of $F$ such that $F_v = \mathbb{Q}_{p^2}$. The main theorems of this paper are as follows.

\begin{thm}(Automorphy lifting theorem, Theorem \ref{automorphy lifting theorem})\label{automorphy lifting theoremII}

Let  $\rho : G_{F} \rightarrow \mathrm{GL}_2(\overline{\mathbb{Q}_p})$ be an irreducible continuous representation. We assume the following conditions.

1 \ There exist a continuous character $\chi : G_{F^+} \rightarrow \overline{\mathbb{Q}_p}^{\times}$ satisfying $\chi(c_v) = \chi(c_w)$ for any complex conjugations $c_v, c_w$ at $v, w \mid \infty$ and a perfect $G_{F}$-equivariant symmetric pairing $\rho \times \rho^c \rightarrow \chi|_{G_{F}}$.
    
2 \ $\rho$ is unramified at almost all finite places.
    
3 \ $\rho$ is de Rham and Hodge-Tate regular\footnote{We say that $\rho : G_{F_v} \rightarrow \mathrm{GL}_2(\overline{\mathbb{Q}}_p)$ is Hodge-Tate regular if $\mathrm{dim}_{\overline{\mathbb{Q}_p}}(\rho \otimes_{\tau, F_v} \widehat{\overline{F_v}})^{G_{F_v}}$ is $0$ or $1$ for any $v \mid p$ and $\tau : F_v \hookrightarrow \overline{\mathbb{Q}}_p$ over $\mathbb{Q}_p$, where $\widehat{\overline{F_v}}$ denotes the completion of the algebraic closure $\overline{F}_v$ of $F_v$.} at all $p$-adic places.
    
4 \ The restriction $\overline{\rho}|_{G_{F(\zeta_p)}}$ to $G_{F(\zeta_p)}$ of the residual representation $\overline{\rho}$ of $\rho$ is absolutely irreducible, where $\zeta_p$ denotes a primitive $p$-th root of unity.
    
5 \ There exists an essentially conjugate self-dual \footnote{We say that $\sigma$ is essentially conjugate self-dual if there exists a continuous character $\psi : \mathbb{A}_{F^+}^{\times}/F^{+, \times} \rightarrow \mathbb{C}^{\times}$ such that $\psi(c_w) = \psi(c_v)$ for any $v, w \mid \infty$ and $\sigma^c \cong \sigma^{\vee} \otimes \psi \circ N_{F/F^+}$. } cohomological cuspidal automorphic representation $\sigma$ of $\mathrm{GL}_2(\mathbb{A}_F)$ such that $\overline{\rho} \cong \overline{r_{\iota}(\sigma)}$. ($r_{\iota}(\sigma)$ denotes the Galois representation corresponding to $\sigma$. See Theorem \ref{Galois representation}.)
    
6 \ For any $v \in \Psi$ and for any element $k = (k_{\tau, 1}, k_{\tau ,2})_{\tau}$ of the set of the Serre weights $W_{F_v}(\overline{\rho}|_{G_{F_v}})$ of $\overline{\rho}|_{G_{F_v}}$, we have $2 \le k_{\tau, 1} - k_{\tau, 2} \le p-5$ for any $\tau \in \mathrm{Hom}(\mathbb{F}_v, \overline{\mathbb{F}}_p)$. (See definition \ref{Serre weight} for the definition of $W_{F_v}(\overline{\rho}|_{G_{F_v}})$. We write $\mathbb{F}_v$ for the residue field at $v$.)
    
Then there exists an essentially conjugate self-dual cohomological cuspidal automorphic representation $\pi$ of $\mathrm{GL}_2(\mathbb{A}_F)$ such that $\rho \cong r_{\iota}(\pi)$.
    
\end{thm}
    
    \begin{rem}
    
    (1) \ Note that in the above theorem, we don't assume that $\rho|_{G_{F_v}}$ is ordinary or Fontaine-Laffaille for any $v \mid p$. Previously, such results were known only when $F_v = \mathbb{Q}_p$ for any $p$-adic place $v$ so that we can use the $p$-adic Langlands correspondence of $\mathrm{GL}_2(\mathbb{Q}_p)$.

    (2) \ Note that if $F$ is an imaginary quadratic field, the above condition $4$ of the above theorem is always satisfied by \cite{SEI}. (See Proposition \ref{residual automorphy}.)

    (3) \ If $F$ is totally real, then the above condition 1 follows from the oddness condition $\mathrm{det}\rho(c_v) = -1$ for any $v \mid \infty$.
    
    (4) \ The above condition $6$ is needed to use a modification of the works \cite{GEKI} and \cite{BGRT} about big $R=T$ theorems. (See Theorem \ref{eigenspace}.)

    \end{rem}
    
    From the above theorem, we can deduce the following theorem by the potential automorphy of residual representations \cite{Calabi} and \cite{CW}.
    
    \begin{thm}(Potential automorphy theorem, Theorem \ref{potential automorphy})
    
    Let $\rho : G_{F} \rightarrow \mathrm{GL}_2(\overline{\mathbb{Q}_p})$ be an irreducible continuous representation. We assume that $\rho$ satisfies the conditions of Theorem \ref{automorphy lifting theoremII} except the condition $5$.

    Then there exists a finite Galois extension $E/F$ of CM fields and an essentially conjugate self-dual cohomological cuspidal automorphic representation $\pi$ of $\mathrm{GL}_2(\mathbb{A}_E)$ such that $\rho|_{G_E} \cong r_{\iota}(\pi)$.
    
    \end{thm}
    
    By the works of \cite{FMRI}, \cite{GK} and \cite{BM}, it is known that the automorphy lifting theorem is equivalent to the Breuil-M$\mathrm{\acute{e}}$zard conjecture. Actually in the $\mathrm{GL}_2(\mathbb{Q}_p)$ case, the Breuil-M$\mathrm{\acute{e}}$zard conjecture was already known by \cite{FMRI}, \cite{HUTU}, \cite{TunI} and \cite{TunII} and thus the automorphy lifting theorem was already known. Therefore, Theorem \ref{automorphy lifting theoremII} follows from the following theorem. 
    
    \begin{thm}(Breuil-M$\acute{e}$zard conjecture for $\mathrm{GL}_2(\mathbb{Q}_{p^2})$, Theorem \ref{Breuil-Mezard})\label{Breuil-Mezard conjectureI}
    
    Let $\overline{\rho} : G_{\mathbb{Q}_{p^2}} \rightarrow \mathrm{GL}_2(\overline{\mathbb{F}}_p)$ be a continuous representation such that for any $k = (k_{\gamma, 1}, k_{\gamma ,2})_{\gamma} \in W_{\mathbb{Q}_{p^2}}(\overline{\rho})$, we have $2 \le k_{\gamma, 1} - k_{\gamma, 2} \le p-5$ for any $\gamma \in \mathrm{Hom}(\mathbb{F}_{p^2}, \overline{\mathbb{F}}_p)$. Then for any weight $\lambda$ and any inertia type $\tau$, the Breuil-M$\acute{e}$zard conjecture for $\overline{\rho}$, $\lambda$ and $\tau$ holds. (See Conjecture \ref{Breuil-Mezard conjecture} for the precise statement of the Breuil-M$\acute{e}$zard conjecture.)
            
    \end{thm}
    
    \begin{rem} Beyond the $\mathrm{GL}_2(\mathbb{Q}_p)$ case, the Breuil-M$\mathrm{\acute{e}}$zard conjecture was previously known in the following cases.
        
    1 \ The two dimensional trivial weight case. (See \cite{GK}.)
    
    2 \ When the Galois representation (including higher dimensional cases) is potentially crystalline at $p$ with a ``generic'' degeneration and a relatively small weight with respect to $p$. (See \cite{LM} and \cite{Tony}.)
    
    \end{rem}

    Note that by the work \cite{EGS}, in order to prove Theorem \ref{Breuil-Mezard conjectureI}, we may assume that $\overline{\rho}$ is ``sufficiently generic''. In fact, \cite{EGS} constructed a certain moduli stack of $p$-adic Galois representations, which is called Emerton-Gee stack and can be regarded as a globalization of $p$-adic lifting rings. As explained in \cite[Remark 8.3.7]{EGS}, we have the following implications.
    
     \ \ \ \ Breuil-M$\mathrm{\acute{e}}$zard conjecture of $p$-adic lifting rings.
     
     $\Leftarrow$ Breuil-M$\mathrm{\acute{e}}$zard conjecture of the Emerton-Gee stack.
     
     $\Leftarrow$ Breuil-M$\mathrm{\acute{e}}$zard conjecture of $p$-adic lifting rings of ``sufficiently generic'' residual representations.

This implications are essential to use the work \cite{GEKI}, which is used to prove our main theorem \ref{Breuil-Mezard conjectureI}. (See {\S} 7.2 and the proof of Corollary \ref{Breuil-Mezard}.)

    \vspace{0.5 \baselineskip}
    
    On the other hand, there exists another approach to the problem (*) in the $\mathrm{GL}_2(\mathbb{Q})$ case by \cite{EMLG} other than using the Breuil-M$\mathrm{\acute{e}}$zard conjecture. Roughly speaking, that is ``big $R=T$ theorem \ + classicality theorem''. A big $R=T$ theorem roughly means that every Galois deformation of a residual representation (not necessarily being de Rham) corresponds to an eigensystem appearing in a certain completed cohomology which is a deformation of the corresponding residual eigensystem. A classicality theorem roughly means that an eigensystem appearing in a certain completed cohomology corresponding to a de Rham representation is actually equal to a classical eigensystems, i.e., an eigensystem coming from a usual automorphic representation. By combining these two results, we can deduce the automorphy lifting theorem.
    
    Big $R=T$ theorems are known in relatively many cases (cf. \cite{AP} and \cite{density}, but actually for some technical reasons, we will use a modification of the works \cite{BGRT} and \cite{GEKI}. See arguments before Theorem \ref{eigenspaceI}.) On the other hand, in the $\mathrm{GL}_2(\mathbb{Q})$ case, Emerton crucially used the $p$-adic Langlands correspondence for $\mathrm{GL}_2(\mathbb{Q}_p)$ to prove the classicality theorem and since the $p$-adic Langlands correspondence is still not known beyond the $\mathrm{GL}_2(\mathbb{Q}_p)$ case, this method was not generalized so far. Recently, again in the $\mathrm{GL}_2(\mathbb{Q})$ case, \cite{PanII} introduced another approach to prove the classicality theorem by using perfectoid modular curves and as stated in \cite[Remark 1.1.3]{PanII}, this may be generalized by replacing modular curves with Shimura varieties. 
    
    \vspace{0.5 \baselineskip}

    In this paper, we will study locally analytic vectors only in a ``partially'' completed cohomology of some rank 2 unitary Shimura varieties after localizing at a non-Eisenstein ideal corresponding to a ``very nice'' residual representation for some technical reasons. (We will give more detailed explanations later. See arguments after Theorem \ref{mikami expansion}, Theorem \ref{eigenspace}, Proposition \ref{residual irreducibility} and Proposition \ref{comparison}.) Note that the Breuil-M$\mathrm{\acute{e}}$zard conjecture is a purely local conjecture. Therefore by using a certain globalization argument, we will prove Theorem \ref{automorphy lifting theoremII} by combining the two approaches in the following way. (We fix a local residual representation $\overline{\rho} : G_{\mathbb{Q}_{p^2}} \rightarrow \mathrm{GL}_2(\overline{{\mathbb{F}}}_p)$.)
    
\vspace{0.5 \baselineskip}

    Big $R=T$ theorem \ + classicality theorem for some special global representation $\overline{r}$ such that $\overline{r}|_{G_{\mathbb{Q}_{p^2}}} \cong \overline{\rho}$.
    
    \vspace{0.5 \baselineskip}
    
    $\Rightarrow$ Automorphy lifting theorem for such $\overline{r}$.
    
    \vspace{0.5 \baselineskip}
    
    $\Leftrightarrow$ Breuil-M$\mathrm{\acute{e}}$zard conjecture for $\overline{\rho}$.
    
    \vspace{0.5 \baselineskip}
    
    $\Leftrightarrow$ Automorphy lifting theorem for all global representations $\overline{r}$ satisfying mild conditions such that $\overline{r}|_{G_{\mathbb{Q}_{p^2}}} \cong \overline{\rho}$.
    
    \vspace{0.5 \baselineskip}

Note that ``big $R=T$ theorem \ + classicality theorem'' a priori implies usual $R = T$ theorems after inverting $p$. Thus the combination of these results doesn't suffice to prove the Breuil-M$\mathrm{\acute{e}}$zard conjecture because this is a conjecture about special fibers of $p$-adic lifting rings. In this paper, we will deduce the Breuil-M$\mathrm{\acute{e}}$zard conjecture indirectly from ``big $R=T$ theorem \ + classicality theorem'' by using the work \cite{KP}. (See {\S} 7.1 and the proof of Theorem \ref{BIGRT}.)

In order to state our precise classicality result, we fix some notations.

\begin{itemize}
\item $F_0$ is an imaginary quadratic field in which $p$ splits.
\item $F^+/\mathbb{Q}$ is a Galois totally real field (for simplicity).
\item $F:=F_0F^+$.
\item $\iota : \overline{\mathbb{Q}_p} \Isom \mathbb{C}$ is an isomorphism of fields and by using this, we identify $\mathrm{Hom}(F, \mathbb{C}) = \mathrm{Hom}(F, \overline{\mathbb{Q}}_p)$.
\item $w$ be the $p$-adic place of $F$ and $v$ be the $p$-adic place of $F_0$ induced by $\iota$. We also write $w$ for the place of $F^+$ lying below $w$. We assume that $F_w = \mathbb{Q}_p$ or $\mathbb{Q}_{p^2}$.
\item $\Phi := \mathrm{Gal}(F/F_0)$.
\item $\Psi := \mathrm{Gal}(F_w/\mathbb{Q}_p) \subset \Phi$.
\item $d:=[F_w:\mathbb{Q}_p]$.
\item $S(B)$ is a non-empty finite set of non-$p$-adic finite places of $F$ such that any prime lying below a place in $S(B)$ splits in $F_0$, $S(B) = S(B)^c$ and $\frac{1}{2}|S(B)| + d = 0 \mod 2$.
\item $B$ is the central division algebra over $F$ satisfying the following conditions : $\mathrm{dim}_FB = 4$, $\mathrm{inv}_{F_{w'}}(B_{w'}) = \mathrm{inv}_{F_{{w'}^c}}(B_{{w'}^c}) = \frac{1}{2}$ for any $w' \in S(B)$ and $\mathrm{inv}_{F_{w'}}(B_{w'}) = 0$ for any $w' \notin S(B)$.

\end{itemize}

By the assumption $\frac{1}{2}|S(B)| + d = 0 \mod 2$, there exists a positive involution of second kind $* : B \rightarrow B$ and an alternating non-degenerate pairing $\psi : B \times B \rightarrow \mathbb{Q}$ satisfying the following conditions. (See Proposition \ref{unitary groups}.)

\begin{itemize}
    \item The corresponding unitary group $U$ over $F^+$ satisfies $U \times_{F^+, \tau} \mathbb{R} = U(1,1)$ for any $\tau \in \Psi$.
    \item $U \times_{F, \tau} \mathbb{R} = U(0,2)$ for any $\tau \in \Phi \setminus \Psi$.
    \item $U_{F^+_v}$ is quasi-split for any finite place $v$ not lying below a place of $S(B)$ of $F^+$.  
\end{itemize}

Let $GU/\mathbb{Q}$ be the unitary similitude group defined by $U$ and $S_K$ be the corresponding Shimura variety, which can be regarded as a proper algebraic variety over $F$ of dimension $d$. We fix a sufficiently large finite subextension $E$ of $\overline{\mathbb{Q}}_p/\mathbb{Q}_p$ and a weight $\lambda := (\lambda_0, (\lambda_{\tau, 1}, \lambda_{\tau, 2})_{\tau \in \Phi}) := (\lambda_0, (\lambda_{\tau, 1}, \lambda_{\tau, 2})_{\tau \in \Phi \setminus \Psi}, (0, -\lambda_{\tau})_{\tau \in \Psi}) \in \mathbb{Z} \times (\mathbb{Z}^2_+)^{\Phi}$. Let $\mathcal{O}$ be the ring of integers of $E$ and $\lambda^w = (\lambda_0, (\lambda_{\tau, 1}, \lambda_{\tau, 2})_{\tau \in \Phi \setminus \Psi}, (0, 0)_{\tau \in \Psi})$, $\lambda_w = (0, (0, 0)_{\tau \in \Phi \setminus \Psi}, (0, -\lambda_{\tau})_{\tau \in \Psi})$  $\in \mathbb{Z} \times (\mathbb{Z}^2_+)^{\Phi}$. Let $\mathcal{V}_{\lambda^w}$ be an integral local system of $\mathcal{O}$-modules on $S_K$ defined by $\lambda^w$. (We will give more detailed explanations about this in {\S} 3.3.)

Let $K = \prod_{l}K_l$ be a sufficiently small open compact subgroup of $GU(\mathbb{A}_{\mathbb{Q}}^{\infty})$ and $S$ be a finite subset of primes containing $p$ and all primes lying below a place in $S(B)$ such that $K_l$ is hyperspecial for any $l \notin S$. Let $\mathbb{T}^S$ denote the subalgebra of the Hecke algebra $\mathcal{H}(GU(\mathbb{A}_{\mathbb{Q}}), K)_{\mathcal{O}}$ generated over $\mathcal{O}$ by $\mathcal{H}(GU(\mathbb{Q}_l), K_l)_{\mathcal{O}}$ for all $l \notin S$ splitting in $F_0$. Let $\mathfrak{m}$ be a decomposed generic non-Eisenstein ideal (see definition \ref{non-Eisenstein}) of the Hecke algebra $\mathbb{T}^S$. We assume the following technical conditions.

\begin{itemize}
\item $(\otimes_{\tau \in \Psi} \overline{\rho}_{\mathfrak{m}}^{\tau})$ is irreducible and $\overline{\rho}_{\mathfrak{m}}(G_{F})$ is not solvable.
\item There exists a prime $l$ splitting completely in $F$ not lying below any place in $S(B)$ such that $\overline{\rho}_{\mathfrak{m}}|_{G_{F_{w'}}}$ is irreducible and generic for any $w' \mid l$, i.e., $H^0(G_{F_{w'}}, \mathrm{ad}\overline{\rho}_{\mathfrak{m}}|_{G_{F_{w'}}}(1)) = 0$.
\end{itemize}

We assume that we have a decomposition $K_p = K_0 \times \prod_{w' \mid v} K_{w'} \subset \mathbb{Q}_p^{\times} \times \prod_{w' \mid v} \mathrm{GL}_2(F_{w'}) = GU(\mathbb{Q}_p)$.

Let $$\mathbb{T}^S(K^wK_w, \mathcal{V}_{\lambda^w}/\varpi^n)_{\mathfrak{m}} := \mathrm{Im}(\mathbb{T}^S \rightarrow \mathrm{End}_{\mathcal{O}/\varpi^n}(H^d_{\mathrm{\acute{e}t}}(S_{K^wK_w, \overline{F}}, \mathcal{V}_{\lambda^w}/\varpi^n)_{\mathfrak{m}}))$$ and $$\mathbb{T}^S(K^w, \mathcal{V}_{\lambda^w})_{\mathfrak{m}} := \varprojlim_{K_w, n} \mathbb{T}^S(K^wK_w, \mathcal{V}_{\lambda^w}/\varpi^n)_{\mathfrak{m}}$$ and we consider the following $\mathbb{T}^S(K^w, \mathcal{V}_{\lambda^w})_{\mathfrak{m}} \times \mathrm{GL}_2(F_w)$-modules $$\widehat{H}^d(S_{K^w}, \mathcal{V}_{\lambda^w})_{\mathfrak{m}} := \varprojlim_n\varinjlim_{K_w}H^d_{\mathrm{\acute{e}t}}(S_{K^wK_w, \overline{F}}, \mathcal{V}_{\lambda^w}/\varpi^n)_{\mathfrak{m}}$$ and $\widehat{H}^d(S_{K^w}, V_{\lambda^w})_{\mathfrak{m}} = \widehat{H}^d(S_{K^w}, \mathcal{V}_{\lambda^w})_{\mathfrak{m}}[\frac{1}{p}]$. Our main classicality theorem is the following.

\begin{thm} (Theorem \ref{classicality theorem})

Let $\varphi : \mathbb{T}^S(K^w, \mathcal{V}_{\lambda^w})_{\mathfrak{m}} \rightarrow \mathcal{O}$ be an $\mathcal{O}$-morphism such that the corresponding Galois representation $\rho_{\varphi}$ satisfies the condition that $\rho_{\varphi}|_{G_{F_{w'}}} : G_{F} \rightarrow \mathrm{GL}_2(\mathcal{O})$ is de Rham of the $p$-adic Hodge type $\lambda_{w'} := (\lambda_{\tau, 1}, \lambda_{\tau, 2}) \in (\mathbb{Z}_+^2)^{\mathrm{Hom}_{\mathbb{Q}_p}(F_{w'}, E)}$. (See the remark before Theorem \ref{infinitesimal character} and around the middle of Notations for the definitions of $\rho_{\varphi}$ and the $p$-adic Hodge type.) If $\widehat{H}^d(S_{K^w}, V_{\lambda^w})_{\mathfrak{m}}[\varphi] \neq 0$, then $\varphi$ is a classical eigensystem of weight $\lambda$, i.e., an eigensystem of an eigenvector in $\varinjlim_{K} H^d_{\et}(S_{K, \overline{F}}, V_{\lambda})_{\mathfrak{m}}$.

\end{thm}

\begin{rem} (1) \ Recently, a similar result in the unitary Shimura curve case was proved by \cite{unitaryshimura}.
    
(2) \ The author hopes that we can prove certain classicality theorems under more natural assumptions. In this paper, we use the technical assumptions on $\overline{\rho}_{\mathfrak{m}}$ to prove Proposition \ref{residual irreducibility} and Proposition \ref{comparison}.

\end{rem}

\subsection{Strategy of proof}

First, we recall some results and methods of \cite{PanII} to explain the strategy for proving the classicality theorem of this paper.

Let $p$ be a prime, $K = \prod_{l}K_l$ be a sufficiently small open compact subgroup of $\mathrm{GL}_2(\mathbb{A}_{\mathbb{Q}}^{\infty})$, $X_{K^pK_p}/\mathbb{Q}$ be the compact modular curve of level $K$ and $E$ be a finite extension of $\mathbb{Q}_p$. We consider the completed cohomology of modular curves $\widehat{H}^1(K^p, E) := (\varprojlim_{n}\varinjlim_{K_p} H^1_{\mathrm{\acute{e}t}}(X_{K^pK_p, \overline{\mathbb{Q}}}, \mathbb{Z}/p^n)) \otimes_{\mathbb{Z}_p} E$. This carries a natural action of $G_{\mathbb{Q}} \times \mathrm{GL}_2(\mathbb{Q}_p) \times \mathbb{T}^S$. \footnote{Here, $\mathbb{T}^S$ denotes an appropriate Hecke algebra. In \cite[Theorem 1.1.2]{PanII}, he considered open modular curves and didn't consider compact modular curves. However, our $\widehat{H}^1(K^p, E)$ is equal to $\widehat{H}^1(K^p, E)$ of \cite[Theorem 1.1.2]{PanII} by \cite[{\S} 4.4.1]{PanI}.} The main theorem of \cite{PanII} is the following.

\begin{thm} \cite[Theorem 1.1.2]{PanII} \label{Pan classicality}

Let $\rho : G_{\mathbb{Q}} \rightarrow \mathrm{GL}_2(E)$ be a continuous absolutely irreducible representation and $k \in \mathbb{Z}_{\ge 1}$. We assume the following conditions.

1 \ There exists a $G_{\mathbb{Q}}$-equivariant injection $\rho \hookrightarrow \widehat{H}^1(K^p, E)$.

2 \ $\rho|_{G_{\mathbb{Q}_p}}$ is de Rham of Hodge-Tate weight $(0, k)$.

Then there exists a cuspidal eigenform $f$ of weight $k+1$ such that $\rho \cong \rho_f$, where $\rho_f$ denotes the Galois representation corresponding to $f$.

\end{thm}

Note that $\widehat{H}^1(K^p, E)$ is an admissible representation of $\mathrm{GL}_2(\mathbb{Q}_p)$ by \cite[Theorem 0.1]{admissible}. Thus the $\rho$-isotypic subspace $\widehat{H}^1(K^p, E)[\rho]$ of $\widehat{H}^1(K^p, E)$ is also admissible and the locally analytic part $\widehat{H}^1(K^p, E)^{\mathrm{la}}[\rho]$ is nonzero by \cite[Theorem 7.1]{ST}. (See Theorem \ref{density of locally analytic vectors}.) Pan studied $\widehat{H}^1(K^p, E)^{\mathrm{la}}$ to prove Theorem \ref{Pan classicality}.

One of the most fundamental ingredients to study $\widehat{H}^1(K^p, E)^{\mathrm{la}}$ is geometric Sen theory, which was initiated by Pan \cite{PanI} and was generalized to general smooth rigid spaces by \cite{Cam}. We recall applications of geometric Sen theory used in \cite{PanI}.

Let $\mathcal{X}_{K, \mathbb{Q}_p}$ be the adic space over $\mathbb{Q}_p$ corresponding to the scheme $X_{K, \mathbb{Q}_p}:=X_K \times_{\mathbb{Q}} \mathbb{Q}_p$, $C := \widehat{\overline{\mathbb{Q}}}_p$ and $\mathcal{X}_{K}:=\mathcal{X}_{K, \mathbb{Q}_p} \times_{\mathbb{Q}_p} C$. Then by \cite{Torsion}, there exists a unique perfectoid space $\mathcal{X}_{K^p}$ over $C$ such that $\mathcal{X}_{K^p} \sim \varprojlim_{K_p} \mathcal{X}_{K^pK_p}$. (See Theorem \ref{sim} for the precise definition of $\sim$.) Moreover, by using the relative Hodge-Tate filtration $0 \rightarrow \omega_{\mathcal{X}_{K^p}}^{-1}(1) \rightarrow \mathbb{Q}_p^{\oplus 2} \otimes_{\mathbb{Q}_p} \mathcal{O}_{\mathcal{X}_{K^p}} \rightarrow \omega_{\mathcal{X}_{K^p}} \rightarrow 0$, we obtain the Hodge-Tate period map $\pi_{\mathrm{HT}} : \mathcal{X}_{K^p} \rightarrow \Fl:=\mathbb{P}^{1, \mathrm{ad}}_{C}$ \footnote{$\mathbb{P}^{1, \mathrm{ad}}_{C}$ denotes the adic space over $C$ corresponding to $\mathbb{P}^{1}_{C}$.}, which was proved to be ``affine'' in \cite{Torsion}. (See Theorem \ref{affine} for the precise meaning.) Let $\mathcal{O}_{K^p} := \pi_{\mathrm{HT} *}\mathcal{O}_{\mathcal{X}_{K^p}}$, $\omega_{K^p}:=\pi_{\mathrm{HT} *}\omega_{\mathcal{X}_{K^p}}$. Moreover, let $\mathcal{O}_{K^p}^{\mathrm{la}}$ denote the subsheaf of $\mathcal{O}_{K^p}$ defined by $\mathcal{O}_{K^p}^{\mathrm{la}}(U) := \mathcal{O}_{K^p}(U)^{\mathrm{la}}$ for any quasicompact open subset $U$ of $\Fl$. Here, $\mathcal{O}_{K^p}(U)^{\mathrm{la}}$ denotes the subspace of locally analytic vectors in $\mathcal{O}_{K^p}(U)$ with respect to the action of an open compact subgroup of $\mathrm{GL}_2(\mathbb{Q}_p)$ stabilizing $U$.

As a consequence of geometric Sen theory, we get the following result.

\begin{thm} \cite[Theorem 4.4.6, 4.2.7]{PanII} \label{fundamental lemma}

1 \ There exists a canonical $G_{\mathbb{Q}_p} \times \mathrm{GL}_2(\mathbb{Q}_p) \times \mathbb{T}^S$-equivariant isomorphism $\widehat{H}^1(K^p, \mathbb{Q}_p)^{\mathrm{la}} \widehat{\otimes}_{\mathbb{Q}_p} C \cong H^1(\Fl, \mathcal{O}_{K^p}^{\mathrm{la}}).$

2 \ Let $\mathfrak{n}^0$ be the $\mathrm{GL}_{2, C}$-equivariant line bundle on $\Fl$ defined by $\mathfrak{n} := \{ (a_{i,j})_{i, j = 1,2} \mid a_{i,j} \in C, a_{1,1} = a_{2, 2} = a_{2, 1} = 0 \} \subset \mathfrak{gl}_2(C)$. (See definition \ref{nilpotent} for the precise definition of $\mathfrak{n}^0$) Then the action of $\mathfrak{n}^0$ on $\mathcal{O}_{K^p}^{\mathrm{la}}$ induced by the embedding $\mathfrak{n}^0 \hookrightarrow \mathfrak{gl}_2(\mathbb{Q}_p) \otimes_{\mathbb{Q}_p} \mathcal{O}_{\Fl}$ is trivial.    

\end{thm}

The property 1 of the above theorem implies that we can understand $\widehat{H}^1(K^p, \mathbb{Q}_p)^{\mathrm{la}}$ by studying the sheaf $\mathcal{O}_{K^p}^{\mathrm{la}}$. Moreover, the property 2 of the above theorem implies that the locally analytic functions $\mathcal{O}_{K^p}^{\mathrm{la}}$ on $\mathcal{X}_{K^p}$ satisfies a certain differential equation \footnote{A generator of $\mathfrak{n}^0$ is called a geometric Sen operator.} and actually by using this, we can get a very explicit description of $\mathcal{O}_{K^p}^{\mathrm{la}}$. We recall that in the following.

Let $U$ be a rational open subset of $\{ (x, 1) \mid || x || \le 1 \} \subset \Fl$. Then $\pi_{\mathrm{HT}}^{-1}(U)$ is an affinoid perfectoid open of $\mathcal{X}_{K^p}$, which is equal to $\pi_{K_p}^{-1}(V_{K_p})$ for some $K_p$ and affinoid open $V_{K_p}$ of $\mathcal{X}_{K^pK_p}$, where $\pi_{K_p} : \mathcal{X}_{K^p} \rightarrow \mathcal{X}_{K^pK_p}$ denotes the canonical projection. Let $x \in \mathcal{O}_{\Fl}(U)$ be the coordinate function on $U$, $e_1$ denote the image of $\begin{pmatrix}
    1 \\ 
    0 
    \end{pmatrix}$ via the map $\mathbb{Q}_p^{\oplus 2} \rightarrow \omega_{K^p}(U)$ induced by the Hodge-Tate filtration and $t : \mathcal{X}_{K^p} \rightarrow \mathrm{Isom}(\mathbb{Z}_p, \mathbb{Z}_p(1))$ be the similitude factor. Note that $e_1$ is a generator of $\omega_{K^p}(U)$ and we regard $t$ as an element of $\Gamma(\Fl, \mathcal{O}_{K^p})^{\times}$ by fixing an isomorphism $\mathrm{Isom}(\mathbb{Z}_p, \mathbb{Z}_p(1)) \cong \mathbb{Z}_p^{\times}$. Assume $K_p = G_m := I_2 + p^m\mathrm{M}_2(\mathbb{Z}_p)$ for $m \ge 2$. Take an increasing sequence of integers $r(m+1) < r(m+2) < \cdots < r(n) < \cdots$ and $x_n, t_n \in \mathcal{O}_{\mathcal{X}_{K^pG_{r(n)}}}(V_{G_{r(n)}})$, $e_{1,n} \in \omega_{\mathcal{X}_{K^pG_{r(n)}}}(V_{G_{r(n)}})$ satisfying $|| x - x_n ||, || \frac{e_{1,n}}{e_1} - 1 ||, || t - t_n || \le p^{-n-1}$. Moreover, we may assume $|| x - x_n || = || x - x_n ||_{G_{r(n)}}, || \frac{e_{1,n}}{e_1} - 1 || = || \frac{e_{1,n}}{e_1} - 1 ||_{G_{r(n)}}, || t - t_n || = || t - t_n ||_{G_{r(n)}}$ by Lemma \ref{norm}. Let $V_{n}$ be the pullback of $V_{G_m} = V_{K_p}$ to $\mathcal{X}_{K^pG_n}$

\begin{thm}\cite[Theorem 4.3.9]{PanII}\label{mikami expansion}

For any $m$, there exists a positive integer $N$ such that for any $n \ge N$, any $G_m$-analytic function $f$ on $U$ (i.e., an element of $\mathcal{O}_{K^p}(U)^{G_m-\mathrm{an}}$) can be uniquely expressed as $f = \sum_{i, j, k} a_{i, j, k} (x-x_n)^i(\mathrm{log}(\frac{e_1}{e_{1, n}}))^j(\mathrm{log}(\frac{t}{t_n}))^k$, where $a_{i, j, k}$ are elements of $\mathcal{O}_{\mathcal{X}_{K^pG_n}}(V_{G_n})$ and satisfy $\mathrm{sup}_{i, j, k} || a_{i,j,k} ||p^{-(i+j+k)n} < \infty$.

\end{thm}

We briefly recall the proof of this. Actually, this follows from the following simple principle: if the action $\mathfrak{b} := \begin{pmatrix}
    * & * \\ 
    0 & * 
    \end{pmatrix}$ on $f \in \mathcal{O}_{K^p}(U)^{G_m-\mathrm{an}}$ is trivial, then $\mathfrak{gl}_2(\mathbb{Q}_p) f = 0$ because $\mathfrak{n}^0(U)$ is generated by $\begin{pmatrix}
x & x^2 \\ 
-1 & -x 
\end{pmatrix}$ and $\mathfrak{n}^0(U)f = 0$ by 2 of Theorem \ref{fundamental lemma}. Thus we obtain $f \in \mathcal{O}_{K^p}(U)^{G_m} = \mathcal{O}_{\mathcal{X}_{K^pG_m}}(V_{G_m})$. From this, we only need to consider the action of $\mathfrak{b}$ and we obtain the above expansion formula in the following way. (Let $u^{+}:=\begin{pmatrix}
    0 & 1 \\ 
    0 & 0 
    \end{pmatrix}, h := \begin{pmatrix}
        1 & 0 \\ 
        0 & -1 
        \end{pmatrix}, z := \begin{pmatrix}
            1 & 0 \\ 
            0 & 1 
            \end{pmatrix}$ and note that $u^+x = 1, u^+\mathrm{log}(\frac{e_1}{e_{1, n}}), u^+\mathrm{log}(\frac{t}{t_n}) = 0, h \ \mathrm{log}(\frac{e_1}{e_{1, n}}) = 1, h \ \mathrm{log}(\frac{t}{t_n}) = 0$ and $z \ \mathrm{log}(\frac{t}{t_n}) = 2$, see \cite[proof of {\S} 4.3.4 $\sim$ Theorem 4.3.9]{PanI} for more precise explanations and $N$ is a sufficiently large integer determined by the norm of $u^+, h, z$ on $\mathcal{O}_{K^p}(U)^{G_m-\mathrm{an}}$.)

            \begin{equation*}
                \begin{split}
                f & = \sum_{i=0}^{\infty} a_{i} (x-x_n)^i \ \mathrm{for \ some} \ a_i \in \mathcal{O}_{K^p}(U)^{G_n-\mathrm{an}, u^+ = 0} \\
                 &  = \sum_{i, j} a_{i, j} (\mathrm{log}(\frac{e_1}{e_{1, n}}))^j(x-x_n)^i \ \mathrm{for \ some} \ a_{i, j} \in \mathcal{O}_{K^p}(U)^{G_n-\mathrm{an}, u^+=0, h = 0} \\
                        & = \sum_{i, j, k} a_{i, j, k} (\mathrm{log}(\frac{t}{t_n}))^k(\mathrm{log}(\frac{e_1}{e_{1, n}}))^j(x-x_n)^i \ \mathrm{for \ some} \ a_{i, j, k} \in \mathcal{O}_{K^p}(U)^{G_n-\mathrm{an}, \mathfrak{b}} = \mathcal{O}_{\mathcal{X}_{K^pG_n}}(V_{G_n}).
                \end{split}
               \end{equation*}

\vspace{0.5 \baselineskip}

This description is very useful and in \cite{PanII}, this explicit description is frequently used to study $\widehat{H}^1(K^p, \mathbb{Q}_p)^{\mathrm{la}}$. Therefore, in order to generalize Pan's work to more general Shimura varieties, we should consider whether a similar description still holds or not. For general Shimura varieties, a generalization of Theorem \ref{fundamental lemma} was proved by \cite[Theorem 6.2.6]{LACC} and $\mathfrak{n}$ is replaced by the nilpotent subalgebra defined by the minuscule. However, the above construction gives a similar description only for Shimura varieties such that the Levi subgroup defined by the minuscule is a torus. Therefore, Shimura varieties associated with unitary (similitude) groups of rank $n \ge 3$ and Shimura curves over a totally real field $F^+ \neq \mathbb{Q}$ don't satisfy this property. On the other hand, Shimura varieties associated with unitary similitude groups $GU$ defined by unitary groups $U$ over a totally real field $F^+$ satisfying $U \times_{F^+, \tau} \mathbb{R} = U(1,1)$ for any $\tau : F^+ \hookrightarrow \mathbb{R}$ (we will write $GU(1,1)$, $U(1,1)$ for such $GU$, $U$) and Hilbert modular varieties satisfy this property. In this paper, we mainly study $GU(1, 1)$ cases. (Actually, we need technical modifications later.) Note that automorphic representations of rank 2 unitary groups basically correspond to conjugate self-dual automorphic representations of $\mathrm{GL}_{2}$ over a CM extension of $F^+$.

\vspace{0.5 \baselineskip}

In the $GU(1,1)$ case, we have the following difficulty.

Let $F_0$ be an imaginary quadratic field in which $p$ splits, $F^+/\mathbb{Q}$ be a Galois totally real field of degree $d$, $F:=F^+F_0$ and $\Phi := \mathrm{Hom}_{F_0}(F, \mathbb{C})$. We consider a unitary group $U(1, 1)/F^+$ constructed from a central division algebra over $F$ for simplicity and the Shimura variety $S_K/\mathbb{Q}$ associated with $G:=GU(1, 1)/\mathbb{Q}$, which is proper of dimension $d$. Let $\sigma$ be a cohomological cuspidal automorphic representation of $GU(\mathbb{A}_{\mathbb{Q}})$ whose base change $\chi \boxtimes \pi$ to $\mathbb{A}_E^{\times} \times \mathrm{GL}_2(\mathbb{A}_F)$ is cuspidal. (See Definition \ref{cohomological} and Theorem \ref{base changeII} for more detailed explanations.) Then by the Kottwitz conjecture for $S_K$, we have a $G_{F}$-equivariant isomorphism \begin{equation}\label{Kottwitz3}
(\varinjlim_{K}H^d_{\mathrm{\acute{e}t}}(S_{K, \overline{F}}, V_{\iota \lambda}))[\sigma^{\infty}]^{\mathrm{ss}} \cong ((r_{\iota}(\chi)|_{G_{F}} \otimes (\otimes_{\tau \in \Phi} r_{\iota}(\pi)^{\tau})^{\vee}(-d)))^{\oplus m},\end{equation} where $m$ is some positive integer and $\lambda \in \mathbb{Z} \times (\mathbb{Z}^2_+)^{\Phi}$ denotes the weight of $\sigma$, $V_{\iota \lambda}$ is the $p$-adic local system on $(S_{K})_{\mathrm{\acute{e}t}}$ defined by a weight $\iota \lambda \in \mathbb{Z} \times \mathbb{Z}^{\iota \Phi}$, $r_{\iota}(\pi)$ denotes the Galois representation corresponding to $\pi$ and $\mathrm{ss}$ denotes the semisimplification as a representation of $G_F$. (See {\S} 3.2 and 3.3 for more detailed explanations.) Actually, we can regard both sides as representations of $G_{F_0}$, but even though we do that, the representation $\otimes_{\tau \in \Phi} r_{\iota}(\pi)^{\tau}$ is not irreducible in general\footnote{For example, when $F^+$ is a real quadratic, $r_{\iota}(\pi)$ has a big image and extends to a representation of $G_{F_0}$.}. Therefore, the relation between the Galois module structure of a Hecke eigenspace in the completed cohomology of $S_K$ and the Galois representation corresponding to that eigensystem is not clear in general. In order to overcome this difficulty, we use the equivalence between the automorphy lifting theorem and the Breuil-M$\mathrm{\acute{e}}$zard conjecture. (cf. \cite{FMRI}, \cite{GK} and \cite{BM}.) Note that the Breuil-M$\mathrm{\acute{e}}$zard conjecture is a purely local conjecture. Therefore, we can use the following strategy. (In the following, we fix a local residual representation $\overline{\rho}_v$, a weight $\lambda_v$ and an inertia type $\tau_v$ for any $v \mid p$ and $R_{\overline{\rho}}^{\lambda, \tau}$ (resp. $\mathbb{T}^{\lambda, \tau}_{\mathfrak{m}}$) denotes an appropriate deformation ring (resp. Hecke algebra) defined by a global residual representation $\overline{\rho}$ (resp. non-Eisenstein ideal $\mathfrak{m}$), $\lambda = (\lambda_v)_{v \mid p}$ and $\tau = (\tau_v)_{v \mid p}$. See {\S} 7.3 for details.)

\vspace{0.5 \baselineskip}

Automorphy lifting theorem $R^{\lambda, \tau, \mathrm{red}}_{\overline{r}} \cong \mathbb{T}^{\lambda, \tau}_{\mathfrak{m}}$ for global Galois representations $\overline{r}$ such that $\overline{r}|_{G_{F_v}} = \overline{\rho}_v$ for any $v \mid p$ and $\overline{r}$ has very nice global properties such as the irreducibility of $\otimes_{\tau \in \Phi} \overline{r}^{\tau}$.

\vspace{0.5 \baselineskip}

$\Leftrightarrow$ Breuil-M$\mathrm{\acute{e}}$zard conjecture for $R^{\lambda_v, \tau_v}_{\overline{\rho}_v}$ for any $v \mid p$.

\vspace{0.5 \baselineskip}

$\Leftrightarrow$ Automorphy lifting theorem $R^{\lambda, \tau, \mathrm{red}}_{\overline{r}} \cong \mathbb{T}^{\lambda, \tau}_{\mathfrak{m}}$ for all global Galois representations $\overline{r}$ satisfying mild conditions such that $\overline{r}|_{G_{F_v}} = \overline{\rho}_v$.

\vspace{0.5 \baselineskip}

Note that Emerton's method ``big $R=T$ theorem + classicality theorem'' only implies $R^{\lambda, \tau, \mathrm{red}}_{\overline{\rho}}[\frac{1}{p}] \cong \mathbb{T}^{\lambda, \tau}_{\mathfrak{m}}[\frac{1}{p}]$, but we can deduce the Breuil-M$\mathrm{\acute{e}}$zard conjecture by combining Emerton's method and a modification of \cite{KP}, \cite{KR} and \cite{FS}. (See {\S} 7.1 and 7.3.) Note also that it is not known how to construct globalizations $\overline{r}$ of a local residual representation $\overline{\rho} : G_{F_w} \rightarrow \mathrm{GL}_2(\overline{\mathbb{F}}_p)$ having the irreducibility of $\otimes_{\tau \in \Phi} \overline{r}^{\tau}$. By using the method of \cite{Calegari}, the author only knows the method of constructing a globalization $\overline{r}$ of $\overline{\rho}$ having the irreducibility of $\otimes_{\tau \in \mathrm{Gal}(F_w/\mathbb{Q}_p)} \overline{r}^{\tau}$ for some fixed $w \mid p$. (Here, we regard $\Psi := \mathrm{Gal}(F_w/\mathbb{Q}_p)$ as a subset of $\Phi$.)

Therefore, we assume that $p$ is unramified in $F$ and basically consider the following situation and problem. (Let $\mathcal{O}$ be the ring of integers of $E$, $\mathbb{F}$ be the residue field of $\mathcal{O}$ and $\varpi$ be a uniformizer of $\mathcal{O}$. We fix $\lambda \in \mathbb{Z} \times (\mathbb{Z}^2_+)^{\Phi}$, put $d := [F_v:\mathbb{Q}_p]$ and forget the previous $d$.)

$\cdot$ Unitary group $U/F^+$ such that $U \times_{F^+, \tau} \mathbb{R} = U(1, 1)$ for $\tau \in \Psi$ and $U(0,2)$ for $\tau \notin \Psi$ and $U(F_w) = \mathrm{GL}_2(F_w)$.

$\cdot$ Partially completed cohomology of the Shimura variety $S_K/F$ associated with $GU/\mathbb{Q}$ with weight $\lambda^w \in \mathbb{Z} \times (\mathbb{Z}^2_+)^{\Phi \setminus \Psi}$ over $\mathcal{O}$.

$\widehat{H}^d(S_{K^w}, \mathcal{V}_{\lambda^{w}}) := \varprojlim_{n} \varinjlim_{K_w} H^d_{\mathrm{\acute{e}t}}(S_{K^wK_w, \overline{F}}, \mathcal{V}_{\lambda^{w}}/\varpi^n)$

$\cdot$ Conjugate self-dual Galois representation $\rho : G_F \rightarrow \mathrm{GL}_2(\mathcal{O})$ such that $\rho|_{G_{F_w}}$ is de Rham of $p$-adic Hodge type $\lambda_w \in (\mathbb{Z}^2_+)^{\Psi}$, $\rho|_{G_{F_v}}$ is Fontaine-Laffaille of weight $\lambda_v$ for any $p$-adic places $v \neq w$, $\otimes_{\tau \in \Psi} \overline{\rho}^{\tau}$ is irreducible and the non-Eisenstein ideal $\mathfrak{m}$ corresponding to $\overline{\rho}$ is decomposed generic and satisfies $\widehat{H}^d(S_{K^v}, \mathcal{V}_{\lambda^{v}})_{\mathfrak{m}} \neq 0$. (See definition \ref{non-Eisenstein} for the definition of decomposed genericity.)

Problem: Prove that there exists a conjugate self-dual cohomological cuspidal automorphic representation $\pi$ of $\mathrm{GL}_2(\mathbb{A}_F)$ such that $\rho \cong r_{\iota}(\pi)$.

As we have seen before, solving this problem for any $\rho$ as above implies the Breuil-M$\mathrm{\acute{e}}$zard conjecture for $R^{\lambda_w, \tau_w}_{\overline{\rho}|_{G_{F_w}}}$ for any inertia type $\tau_w$. Note that we can use the results of geometric Sen theory for the above ``partially'' completed cohomology $\widehat{H}^d(S_{K^w}, \mathcal{V}_{\lambda^{w}})$. In fact, an integral model of $S_{K^wK_w}$ is $\etale$ over that of $S_{{K'}^wK_w}$ for any $K^w \subset {K'}^w$ because the universal $p$-divisible groups at $v \neq w$ on an integral model $S_K$ are ind-$\etale$ by the moduli interpretation. Thus the perfectoidness of $S_{K^p}$ implies the perfectoidness of $S_{K^w}:=``\varprojlim_{K_w}S_{K^wK_w}$''. (Actually, we will use more technical results in {\S} 3. See Theorem \ref{geometric sen theory} and Remark \ref{perfectoidness}.) Note also that by the decomposed genericity of $\mathfrak{m}$ and the works \cite{CS} and \cite{kos}, we obtain the vanishing \begin{equation} \label{vanishing} H^i(S_K, \mathcal{V}_{\lambda^w}/\varpi)_{\mathfrak{m}} = 0 \end{equation} for any $i \neq d$. Thus $\widehat{H}^d(S_{K^w}, \mathcal{V}_{\lambda^w})_{\mathfrak{m}}$ is the $p$-adic completion of $\varinjlim_{K_w}H^d_{\mathrm{\acute{e}t}}(S_{K^wK_w, \overline{F}}, \mathcal{V}_{\lambda^w})_{\mathfrak{m}}$.

In the following, we discuss how to solve this problem by using the method ``big $R=T$ theorem + classicality theorem''. First, we should note that usual big $R=T$ theorems such as \cite{AP} and \cite{density} don't suffice for our purpose because usual big $R=T$ theorems only imply that there exists a ring morphism from an appropriate Hecke algebra to $\mathcal{O}$ corresponding to the given Galois representation and don't imply the existence of an eigenvector of the eigensystem in the completed cohomology. Emerton used the $p$-adic Langlands correspondence of $\mathrm{GL}_2(\mathbb{Q}_p)$ to prove this existence.\footnote{Note that as we have seen in Background, Emerton crucially used the $p$-adic Langlands correspondence of $\mathrm{GL}_2(\mathbb{Q}_p)$ to prove the classicality theorem.} (See \cite[Remark 6.4.3]{EMLG}. See also \cite[after Lemma 3.5.13]{FMRR} for a similar argument.)

\vspace{0.5 \baselineskip}

Recently, \cite{GN} introduced a new approach to big $R=T$ theorems such that we can simultaneously prove the existence of a non-zero eigenvector for any Hecke eigensystem assuming that the considered local deformation rings are formally smooth and the mod $p$ cohomology has a ``correct'' Gelfand-Kirillov dimension. Moreover, the computation of the Gelfand-Kirillov dimension was done by \cite{BGRT} and \cite{GEKI} in some Shimura curve cases and Shimura set over totally real field cases under the assumption that $\overline{\rho}|_{G_{F_w}}$ is ``sufficiently generic''. By using exactly the same strategy as theirs, we can obtain the ``correct'' Gelfand-Kirillov dimension in our situation under the same assumptions. (For some technical reasons, we use a variant $\widehat{\mathcal{A}}_U(K^w, V_{\lambda^w}(\sigma))_{\mathfrak{m}}$ of $\widehat{H}^d(S_{K^w}, V_{\lambda^w})_{\mathfrak{m}}$.) %(We put  $\mathbb{T}^S(K^w, \mathcal{V}_{\lambda^w})_{\mathfrak{m}} := \varprojlim_{n, K^w} \mathrm{Im}(\mathbb{T}^S \rightarrow \mathrm{End}_{\mathcal{O}/\varpi^n}(H^d_{\mathrm{\acute{e}t}}(S_{K^wK_w, \overline{F}}, \mathcal{V}_{\lambda^w}/\varpi^n)_{\mathfrak{m}}))$, $\widehat{H}^d(S_{K^w}, V_{\lambda^w})_{\mathfrak{m}} := \widehat{H}^d(S_{K^w}, \mathcal{V}_{\lambda^w})_{\mathfrak{m}}[\frac{1}{p}]$ and $H^d(K^w, \mathcal{V}_{\lambda^w}/\varpi)_{\mathfrak{m}} = \varinjlim_{K_w}H^d_{\mathrm{\acute{e}t}}(S_{K^wK_w, \overline{F}}, \mathcal{V}_{\lambda^w}/\varpi)_{\mathfrak{m}}$.)

\begin{thm}(Big $R=T$ theorem, Gelfand-Kirillov dimension, Theorem \ref{eigenspace})\label{eigenspaceI}

Under appropriate assumptions, we have the following results. (See {\S} 7.2 for the precise notations.)

1 \ $R_{\overline{\rho}, \mathcal{D}} \ \Isom \ \mathbb{T}^S_U(K^w, \mathcal{V}_{\lambda^w}(\sigma))_{\mathfrak{m}}$.
    
2 \ $\widehat{\mathcal{A}}_U(K^w, V_{\lambda^w}(\sigma))_{\mathfrak{m}}[\varphi] \neq 0$ for any $\mathcal{O}$-morphism $\varphi : R_{\overline{\rho}, \mathcal{D}} \rightarrow \mathcal{O}$.

\end{thm}

In the following, we assume that there exists an eigensystem $\varphi : \mathbb{T}^S(K^w, \mathcal{V}_{\lambda^w})_{\mathfrak{m}} \rightarrow \mathcal{O}$ corresponding to $\rho$ and we have $\widehat{H}^d(S_{K^w}, V_{\lambda^w})_{\mathfrak{m}}[\varphi] \neq 0$. We will discuss how to prove the classicality\footnote{Here, we say that $\varphi$ is classical if $\varphi$ is equal to an eigensystem associated with a conjugate self-dual cohomological automorphic representation of $\mathrm{GL}_2(\mathbb{A}_F)$.} of $\varphi$ from the de Rhamness of $\rho|_{G_{F_w}}$. In \cite{PanII}, the notion ``Fontaine operator'' was used. Let's recall a property of the Fontaine operator. Let $L$ be a finite extension of $\mathbb{Q}_p$, $W$ be a finite-dimensional representation of $G_L$ over $\mathbb{Q}_p$ which is Hodge-Tate of weights $0, k$ for some $k > 0$. Let $W \otimes_{\mathbb{Q}_p} C = W_0 \oplus W_k$ be the Hodge-Tate decomposition of $W$. Then by using the work of \cite{FO}, we can construct a natural $C$-linear map $N_W : W_0 \rightarrow W_k(k)$ such that $N_W = 0$ is equivalent to the condition that $W$ is de Rham. 

In the following, we assume that the given weight $\lambda$ is trivial for simplicity. In order to use a Fontaine operator, first we need to understand ``the Hodge-Tate structure of $\widehat{H}^d(S_{K^w}, F_w)^{\mathrm{la}} \widehat{\otimes}_{F_w} C \cong H^d(\Fl, \mathcal{O}_{K^w}^{\mathrm{la}})$''. Here we use similar notations as in the modular curve case. Thus $\Fl := \prod_{\tau \in \Psi}\mathbb{P}_C^{1, \mathrm{ad}}$. This was already studied in more general situations by \cite{IC} and \cite{LACC}. By \cite{IC}, we obtain $\widehat{H}^d(S_{K^w}, F_w)_{\mathfrak{m}}[\varphi] \subset \widehat{H}^d(S_{K^w}, F_w)^{\chi}$, where $\chi : Z(U(\mathfrak{g})) \rightarrow F_w$ denotes the infinitesimal character of the trivial representation of the Lie algebra $\mathfrak{g}:= \prod_{\tau \in \Psi} \mathfrak{gl}_{2}(F_w)_{\tau}$ over $F_w$ and $Z(U(\mathfrak{g}))$ denotes the center of the universal enveloping algebra of $\mathfrak{g}$. (See Theorem \ref{infinitesimal character}.) Moreover, by using \cite{LACC}, which is a partial generalization of \cite{PanI}, we obtain a certain action $\theta_{\mathfrak{h}}$ of $\mathfrak{h} := \prod_{\tau} \mathfrak{h}_{\tau}$ on $\mathcal{O}_{K^w}^{\mathrm{la}}$, which is called the horizontal action, commutes with the action of $\mathrm{GL}_2(F_w)$ and encodes the information of the Sen operator and the action of $Z(U(\mathfrak{g}))$ on $\mathcal{O}_{K^w}^{\mathrm{la}}$. (Here $\mathfrak{h}_{\tau}$ denotes the subalgebra of $\mathfrak{gl}_2(C)_{\tau}$ consisting of diagonal matrices.) More precisely, the Sen operator on $\mathcal{O}_{K^w}^{\mathrm{la}}$ is equal to $\theta_{\mathfrak{h}}(\begin{pmatrix}
-1 & 0 \\ 
0 &  0
\end{pmatrix})$ and we can obtain the Hodge-Tate decomposition $\displaystyle \mathcal{O}_{K^w}^{\mathrm{la}, \chi} = \oplus_{I \subset \Psi} \mathcal{O}_{K^w}^{\mathrm{la}, (1,-1)_{\tau \in I}, (0,0)_{\tau \notin I} }$ and $\displaystyle H^d(\Fl, \mathcal{O}_{K^w}^{\mathrm{la}})_{\mathfrak{m}}^{\chi} = H^d(\Fl, \mathcal{O}_{K^w}^{\mathrm{la}, \chi})_{\mathfrak{m}} = \oplus_{I \subset \Psi} H^d(\Fl, \mathcal{O}_{K^w}^{\mathrm{la}, (1,-1)_{\tau \in I}, (0,0)_{\tau \notin I} })_{\mathfrak{m}}$, where $(a_{\tau}, b_{\tau})$ denotes the character of $\mathfrak{h}_{\tau}$ defined by $\begin{pmatrix}
    c & 0 \\ 
    0 &  d
    \end{pmatrix} \longmapsto a_{\tau}c + b_{\tau}d$. Note that this Hodge-Tate decomposition is compatible with that of Galois representations $r_{\iota}(\chi)|_{G_{F}} \otimes (\otimes_{\tau \in \Psi} r_{\iota}(\pi)^{\tau})$ appearing at finite levels. (See the above isomorphism $(\ref{Kottwitz3})$ or Proposition \ref{Hodge-Tate decomposition}.)

In the following, we discuss how to prove the classicality of $\varphi$ by using the Fontaine operator. By using a generalization \cite{PanII} of \cite{FO} to Galois representations on Banach spaces and LB spaces, we can construct ``the Fontaine operator''  $N_0 : \mathcal{O}_{K^w}^{\mathrm{la}, (0, 0)} \rightarrow \oplus_{\tau \in \Psi} \mathcal{O}_{K^w}^{\mathrm{la}, (1, -1)_{\tau}, (0, 0)_{\sigma \neq \tau}}(1)$ which is a generalization of the map coming from the Fontaine operator of $\otimes_{\tau \in \Psi} \rho^{\tau}$ for a 2-dimensional Hodge-Tate of weight $(0, 1)$ representation $\rho$. The de Rhamness of $\rho$ implies that $H^d(\Fl, N_0)$ is trivial on $H^d(\Fl, \mathcal{O}_{K^w}^{\mathrm{la}, (0,0)})[\varphi]$ and thus it suffices to prove the following. (See Theorem \ref{induction step} for detailed argument.)

\vspace{0.5 \baselineskip}

\begin{conj}\label{eigenspacedecomposition} $\mathrm{Ker}H^d(\Fl, N_0)_{\mathfrak{m}}$ has a generalized eigenspace decomposition by classical eigensystems. \end{conj}

One possible strategy to calculate $\mathrm{Ker}H^d(\Fl, N_0)_{\mathfrak{m}}$ is the following.

By the description (\ref{Kottwitz3}) of Galois representations in the cohomology at finite levels, the author expects that the following sequence of maps induced by the Fontaine operator is a complex. \ $\mathcal{O}^{\mathrm{la}, (0, 0)}_{K^w} \rightarrow^{N_0} \oplus_{\tau \in \Psi} \mathcal{O}^{\mathrm{la}, (1, -1)_{\tau}, (0, 0)_{\sigma \neq \tau}}_{K^w}(1) \rightarrow \cdots \rightarrow \mathcal{O}^{\mathrm{la}, (1, -1)}_{K^w}(d)$.

(This is a generalization of the complex $\otimes_{\tau}((\rho^{\tau} \otimes_{\mathbb{Q}_p} C)_0 \rightarrow^{N_{\tau}} (\rho^{\tau} \otimes_{\mathbb{Q}_p} C)_1(1))$ for a 2-dimensional Hodge-Tate of weight $(0 ,1)_{\tau}$ representation $\rho$ of $G_{F_w}$ over $\mathbb{Q}_p$ with the Hodge-Tate decomposition $\rho^{\tau} \otimes_{\mathbb{Q}_p} C = (\rho^{\tau} \otimes_{\mathbb{Q}_p} C)_0 \oplus (\rho^{\tau} \otimes_{\mathbb{Q}_p} C)_1$ and the Fontaine operator $N_{\tau} : (\rho^{\tau} \otimes_{\mathbb{Q}_p} C)_0 \rightarrow (\rho^{\tau} \otimes_{\mathbb{Q}_p} C)_1(1)$.)

In the following, we assume that the above sequence is a complex. Actually, we don't need such a result. (See the statement after Remark \ref{artificial}.)

The key step in the calculation is comparing this complex with the following ``geometric locally analytic de Rham complex'' $GDR^{\mathrm{la}}_0$, which is a certain locally analytic extension of the usual de Rham complex at finite levels. \footnote{This observation implicitly appeared in \cite{PanII}.} (Let $\Omega_{K^w}^{1, \mathrm{sm}}$ denote the pushforward via $\pi_{\mathrm{HT}}$ of the pullback to $S_{K^w}$ of the sheaf of differential forms at finite levels. See {\S} 4.1 for the precise construction.)

$GDR^{\mathrm{la}}_{0} : \mathcal{O}^{\mathrm{la}, (0, 0)}_{K^w} \rightarrow^{d} \mathcal{O}^{\mathrm{la}, (0, 0)}_{K^w} \otimes_{\mathcal{O}_{K^w}^{\mathrm{sm}}} \Omega_{K^w}^{1, \mathrm{sm}} \rightarrow \cdots \rightarrow \mathcal{O}^{\mathrm{la}, (0, 0)}_{K^w} \otimes_{\mathcal{O}_{K^w}^{\mathrm{sm}}} \Omega_{K^w}^{d, \mathrm{sm}}$.

Note that by using a similar description of $\mathcal{O}_{K^w}^{\mathrm{la}}$ as in the modular curve case (see Propositions \ref{mikami expansionII} and \ref{mikami expansionIII} for the precise results.), we have $\mathcal{O}^{\mathrm{la}, (0, 0)}_{K^w} $``$=$''$ \mathcal{O}_{K^w}^{\mathrm{sm}} \otimes_{C} \mathcal{O}_{\Fl}$. Roughly speaking, $GDR^{\mathrm{la}}_0$ is the tensor product of the usual de Rham complex coming from finite levels and $\mathcal{O}_{\Fl}$. Note also that by the Kodaira-Spencer isomorphism $\Omega_{K^w}^{1, \mathrm{sm}} \cong \oplus_{\tau \in \Psi} \omega_{\tau, K^w}^{2, \mathrm{sm}}$, we can identify $GDR^{\mathrm{la}}_0$ as the following complex. (Here, $\omega_{\tau, K^w}^{\mathrm{sm}}$ denotes the line bundle defined by the $\tau$-component of the dual of the Lie algebra of the universal abelian scheme.)

$GDR_{0}^{\mathrm{la}} : \mathcal{O}^{\mathrm{la}, (0, 0)}_{K^w} \rightarrow^{\oplus_{\tau \in \Psi} d_{\tau}} \oplus_{\tau \in \Psi} \omega^{2, \mathrm{la}, (0, 0)}_{\tau, K^w} \rightarrow \cdots \rightarrow \otimes_{\tau \in \Psi} \omega^{2, \mathrm{la}, (0, 0)}_{\tau, K^w}$.

Then we have the following diagram.

\begin{equation}\label{commutativediagram}
\xymatrix{
    \mathcal{O}^{\mathrm{la}, (0, 0)}_{K^w} \ar[r]^d \ar[dr]^{N_0} & \oplus_{\tau \in \Psi} \omega^{2, \mathrm{la}, (0, 0)}_{\tau, K^w} \ar[r] \ar@{->>}[d]^{\oplus_{\tau} \overline{d}_{\tau} \otimes_{\mathcal{O}_{K^w}^{\mathrm{sm}}} \mathrm{id}_{\omega_{\tau, K^w}^{2, \mathrm{sm}}}} & \ar@{->>}[d] \cdots \ar[r] & \otimes_{\tau \in \Psi} \omega^{2, \mathrm{la}, (0, 0)}_{\tau, K^w} \ar@{->>}[d]^{\overline{d}_{\tau_1} \otimes_{\mathcal{O}_{K^w}^{\mathrm{sm}}} \mathrm{id}_{\omega_{\tau_1, K^w}^{2, \mathrm{sm}}}} \\
     \ & \oplus_{\tau \in \Psi} \mathcal{O}^{\mathrm{la}, (1, -1)_{\tau}, (0, 0)_{\sigma \neq \tau}}_{K^w}(1) \ar[dr] &  \ar@{->>}[d] \  & \ar@{->>}[d]^{\overline{d}_{\tau_2} \otimes_{\mathcal{O}_{K^w}^{\mathrm{sm}}} \mathrm{id}_{\omega_{\tau_2, K^w}^{2, \mathrm{sm}}}} \\
     \ & \ & \ddots \ar[dr] & \ \ar@{->>}[d]^{\overline{d}_{\tau_d} \otimes_{\mathcal{O}_{K^w}^{\mathrm{sm}}} \mathrm{id}_{\omega_{\tau_d, K^w}^{2, \mathrm{sm}}}} \vdots \\
     \ & \ & \ & \mathcal{O}^{\mathrm{la}, (1, -1)}_{K^w}(d)
    }
\end{equation}

In the above, we fix a numbering $\tau_1, \cdots, \tau_d$ of $\Psi = \{ \tau_1, \cdots, \tau_d \}$ and $\overline{d}_{\tau}$ denotes an $\mathcal{O}_{K^w}^{\mathrm{sm}}$-linear surjective extension $\mathcal{O}_{K^w}^{\mathrm{la}, (0,0)}$``$=$''$\mathcal{O}_{K^w}^{\mathrm{sm}} \otimes_{C} \mathcal{O}_{\Fl} \rightarrow \mathcal{O}_{K^w}^{\mathrm{sm}} \otimes_{C} \Omega^1_{\tau, \Fl} = \mathcal{O}_{K^w}^{\mathrm{la}, (0,0)} \otimes_{C} \omega_{\tau, \Fl}^{-2}$``$=$''$\omega_{\tau, K^w}^{-2, \mathrm{la}, (1, -1)_{\tau}, (0, 0)_{\sigma \neq \tau}}(1) $ of the derivation $\mathcal{O}_{\Fl} \rightarrow \Omega^1_{\Fl, \tau}$ along $\tau$ on $\Fl = \prod_{\tau} \mathbb{P}_C^{1, \mathrm{ad}}$. (See {\S} 4.1 for more detailed explanations.)

Then one can expect that the diagram (\ref{commutativediagram}) is commutative. In this paper, we will partially prove this by using a similar method as \cite{PanII}. 

\begin{thm}(Theorem \ref{Fontaine}) \label{compatibilityIII}

There exists $(c_{\tau})_{\tau \in \Psi} \in (F_w^{\times})^{\Psi}$ such that $N_0 = (\oplus_{\tau} c_{\tau}\overline{d}_{\tau} \otimes_{\mathcal{O}_{K^w}^{\mathrm{sm}}} \mathrm{id}_{\omega_{\tau, K^w}^{2, \mathrm{sm}}} ) \circ d$.

\end{thm}

\begin{rem}\label{artificial}
    
    Note that in the non-parallel weight case, the existence of the $N_0$ doesn't formally follow from the work of \cite{FO}. Actually, we will prove this only when $[F_w : \mathbb{Q}_{p}] = 2$ cases. (See {\S} 5.2.3.)

\end{rem}

Theorem \ref{compatibilityIII} suffices for calculating $\mathrm{Ker}H^0(\Fl, N_0)$ and by the above surjections $\overline{d}_{\tau}$'s, there uniquely exists a complex $ADR^{\mathrm{la}}_0$ having the following form such that $\overline{d}_{\tau}$'s induce a surjective map of complexes $GDR^{\mathrm{la}}_0 \twoheadrightarrow ADR^{\mathrm{la}}_0$. (See Lemma \ref{construction of arithmetic} for the proof. We call $ADR^{\mathrm{la}}_0$ an ``arithmetic locally analytic de Rham complex''.)

$ADR^{\mathrm{la}}_{0} : \mathcal{O}^{\mathrm{la}, (0, 0)}_{K^w} \rightarrow^{N_0'} \oplus_{\tau \in \Psi} \mathcal{O}^{\mathrm{la}, (1, -1)_{\tau}, (0, 0)_{\sigma \neq \tau}}_{K^w}(1) \rightarrow \cdots \rightarrow \mathcal{O}^{\mathrm{la}, (1, -1)}_{K^w}(d)$.

Note that by the result of Caraiani-Scholze \cite{CS} as we have seen in the above (\ref{vanishing}), we can deduce the vanishing of $H^i(\Fl, \mathcal{O}_{K^w}^{\mathrm{la}, (a_{\tau}, b_{\tau})_{\tau}})_{\mathfrak{m}}$ for any $i \neq d$ and any $(a_\tau , b_{\tau})_{\tau} : \mathfrak{h} \rightarrow C$. Thus we obtain $\mathrm{Ker}H^d(\Fl, N_0)_{\mathfrak{m}} = \mathrm{Ker}H^d(\Fl, N_0')_{\mathfrak{m}} \cong H^d(\Fl, ADR^{\mathrm{la}}_0)_{\mathfrak{m}}$ by the degeneration of the following spectral sequence.

$E_1^{p, q} := \oplus_{I \subset \Phi, |I|=p} H^q(\Fl, \mathcal{O}_{K^w}^{\mathrm{la}, (1, -1)_{\tau \in I}, (0,0)_{\tau \notin I}}(p))_{\mathfrak{m}} \Rightarrow H^{p+q}(\Fl, ADR^{\mathrm{la}}_0)_{\mathfrak{m}}$.

 By putting $Ker := \mathrm{Ker}(GDR^{\mathrm{la}}_0 \twoheadrightarrow ADR^{\mathrm{la}}_0)$, we have an exact sequence $$\rightarrow H^d(\Fl, GDR^{\mathrm{la}}_0) \rightarrow H^d(\Fl, ADR^{\mathrm{la}}_0) \rightarrow H^{d+1}(\Fl, Ker).$$ Therefore, in order to prove Conjecture \ref{eigenspacedecomposition}, it suffices to prove that $H^d(\Fl, GDR^{\mathrm{la}}_0)$ and $H^{d+1}(\Fl, Ker)$ have generalized Hecke eigenspace decompositions by classical eigensystems.

\vspace{0.5 \baselineskip}

First, we consider $H^d(\Fl, GDR^{\mathrm{la}}_0)$. Inspired by \cite[{\S} 5]{PanII}, we propose the following conjecture.

\begin{conj}\label{classicalityII}

$H^d(\Fl, GDR^{\mathrm{la}}_0)$ has a generalized Hecke eigenspace decomposition by classical eigensystems.

\end{conj}

By \cite[Corollary 3.5.9]{CS}, $\Fl$ has a ``Newton stratification'' $\Fl = \sqcup_{b \in B(G_{\mathbb{Q}_p}, \mu^{-1})} \Fl^b$, which is compatible with the usual Newton stratification on $S_{K^w}$ via $\pi_{\mathrm{HT}}$, where $G_{\mathbb{Q}_p}:=\mathrm{Res}_{F_w/\mathbb{Q}_p}\mathrm{GL}_{2, F_w}$ and $\mu : \mathbb{G}_{m, \mathbb{Q}_p} \rightarrow G_{\mathbb{Q}_p}$ is the minuscule induced by the minuscule of $S_K$. (See {\S} 4.3 for more detailed explanations.) By using excision sequences, we can reduce Conjecture \ref{classicalityII} to the following more precise conjecture. (Let $j_b : \Fl^b \hookrightarrow \Fl$ be the natural immersion.)

\begin{conj} \label{classicality}

For any $b \in B(G_{\mathbb{Q}_p}, \mu^{-1})$, the cohomology group $H^d(\Fl, j_{b !}j_b^*GDR^{\mathrm{la}}_0)$ has a generalized Hecke eigenspace decomposition by classical eigensystems.

\end{conj}

Roughly, the author hopes that $H^d(\Fl, j_{b, !}j_b^*GDR^{\mathrm{la}}_0)$ is equal to the $J_b(\mathbb{Q}_p)$-invariant part of (rigid cohomology of the Igusa variety $\mathrm{Igs}_{K^w}^b$) $\otimes$ (geometric locally analytic de Rham complex of the Rapoport-Zink space $\mathcal{M}_{(G_{\mathbb{Q}_p}, b, \mu)}$), which can be regarded as a locally analytic de Rham analogue of the Mantovan formula for $l$-adic $\etale$ cohomology \cite[Theorem 22]{Mant}. Then Conjecture \ref{classicality} follows from the classicality of the rigid cohomology $H^i_{\mathrm{rig}}(\mathrm{Igs}_{K^w})$ of $\mathrm{Igs}_{K^w}^b$, which is still an open problem. Note that we can prove the classicality of the alternating sum of $H^i_{\mathrm{rig}}(\mathrm{Igs}_{K^w})$ in the Grothendieck group by using a trace formula in many cases, but that doesn't suffice to prove the classicality of individual $H^i_{\mathrm{rig}}(\mathrm{Igs}_{K^w})$ because there are some cancellations in the alternating sum. In some $\mathrm{dim} \ \mathrm{Igs}_{K^w}^b \le 1$ cases, we can deduce the classicality of $H^i_{\mathrm{rig}}(\mathrm{Igs}_{K^w}^b)$ by directly calculating $H^0_{\mathrm{rig}}(\mathrm{Igs}_{K^w}^b)$. (cf. \cite[Lemma 4.4.13]{PanII}.) In this paper, we will prove the following.

\begin{thm} (Theorem \ref{classicality of geometric})\label{quadratic}
    
Conjecture \ref{classicality} is true if $[F_w : \mathbb{Q}_p] \le 2$.

\end{thm}

When $[F_w : \mathbb{Q}_p] \le 2$, the Newton strata consist of $2$ loci : the $\mu$-ordinary locus and the basic locus. The method of computation of $GDR^{\mathrm{la}}_0$ on these loci was already given in \cite[{\S} 5]{PanII}. On the basic locus, the Igusa variety $\mathrm{Igs}_K^b$ can be regarded as the Shimura set of an inner form $GI$ of $GU$ such that $GI(\mathbb{R})$ is compact modulo center and $H^d(\Fl, j_{b, !}j_b^*GDR^{\mathrm{la}}_0)$ has a Hecke eigenspace decomposition by eigensystems appearing in the space of algebraic automorphic forms $\mathcal{A}^{\mathrm{sm}}_{GI}(K^w)$ on $GI$, which is classical by definition. (See Theorem \ref{basic classicality}.) On the other hand, on the $\mu$-ordinary locus, $H^d(\Fl, j_{b, !}j_b^*GDR^{\mathrm{la}}_0)$ is equal to the locally analytic induction of the de Rham cohomology of some dagger space and the author conjectures that this de Rham cohomology is equal to the rigid cohomology of the Igusa variety $\mathrm{Igs}_{K^w}^b$ of the $\mu$-ordinary locus. Note that even if we can prove that conjecture, in the $[F_w : \mathbb{Q}_p] = 2$ case, $\mathrm{Igs}_{K^w}^b$ has dimension $2$. Thus the classicality of $H^i_{\mathrm{rig}}(\mathrm{Igs}_{K^w}^b)$ doesn't follows from the method explained above. In this paper, we will use the recent work \cite{DC} to prove the classicality of $H^d(\Fl, j_{b, !}j_b^*GDR^{\mathrm{la}}_0)$. (See Proposition \ref{classicality of Igusa}.)

\vspace{0.5 \baselineskip}

Even if we can prove Conjecture \ref{classicalityII}, we have one more problem, which didn't appear in the modular curve case. That is the computation of $H^{d+1}(\Fl, Ker)$. In the modular curve case, $Ker$ is equal to $\Omega_{K^w}^{1, \mathrm{sm}}(\mathcal{C})[-1]$, which comes from finite levels. ($\mathcal{C}$ denotes the cusp of the modular curve and actually, we have $H^2(\Fl, \Omega_{K^w}^{1, \mathrm{sm}}(\mathcal{C})[-1]) = 0$.) However, this property doesn't hold in our situation if $[F_w : \mathbb{Q}_p] \ge 2$ and the computation of the complex $Ker$ is very complicated. Note that by the definition of $Ker$, the component of the complex $Ker$ is generated by elements on which $\mathfrak{gl}_2(F_w)_{\tau}$ acts trivially for some $\tau \in \Psi$. (See Lemma \ref{kernel}.) In order to overcome this difficulty, we use the following inductive argument.

We assume that Conjecture \ref{classicalityII} holds and the considered eigensystem $\varphi$ is not a classical eigensystem. Then the localization $H^d(\Fl, GDR^{\mathrm{la}}_0)_{\varphi}$ of $H^d(\Fl, GDR^{\mathrm{la}}_0)$ by $\varphi$ is zero. Thus we obtain the following inclusion of nonzero spaces. (Note that by the property of the Fontaine operator, we can regard $\mathrm{Ker}H^d(\Fl, N_0)[\varphi]$ as the Hodge-Tate weight 0 subspace of $\widehat{H}^d(S_{K^v}, F_w)[\varphi]^{\mathrm{la}} \widehat{\otimes}_{F_w} C$.) $$H^{d+1}(\Fl, Ker)[\varphi] \supset H^d(\Fl, ADR^{\mathrm{la}}_0)[\varphi] = \mathrm{Ker}H^d(\Fl, N_0)[\varphi] \subset \widehat{H}^d(S_{K^v}, F_w)[\varphi]^{\mathrm{la}} \widehat{\otimes}_{F_w} C.$$ Since all the components of $Ker$ are generated by its $\mathfrak{gl}_2(F_w)_{\tau}$-invariant pars ($\tau \in \Psi$), we obtain $\widehat{H}^d(S_{K^w}, F_w)[\varphi]^{\mathrm{la}, \mathfrak{gl}_2(F_w)_{\tau}} \neq 0$ for some $\tau \in \Psi$. Here, one method to prove the classicality is computing $\widehat{H}^d(S_{K^v}, F_w)[\varphi]^{\mathrm{la}, \mathfrak{gl}_2(F_w)_{\tau}}$ by a similar method as \cite{PanII}. However, a variant of Theorem \ref{fundamental lemma} $$\widehat{H}^d(S_{K^w}, F_w)^{\mathrm{la}, \mathfrak{gl}_2(F_w)_{\tau}} \widehat{\otimes}_{F_w} C \cong \widehat{H}^d(S_{K^w}, \mathcal{O}_{S_{K^w}}^{\mathrm{la}, \mathfrak{gl}_2(F_w)_{\tau}})$$ is basically not true. (A variant of this property $\widehat{H}^1(K^p, \mathbb{Q}_p)^{\mathrm{sm}} \widehat{\otimes}_{\mathbb{Q}_p} C \cong H^1(\Fl, \mathcal{O}_{K^p}^{\mathrm{sm}})$ in the modular curve case is clearly not true because this contradicts the usual Hodge-Tate decomposition.) Moreover, the sheaf $\mathcal{O}_{S_{K^w}}^{\mathrm{la}, \mathfrak{gl}_2(F_w)_{\tau}}$ forgets the symmetry of the original locally analytic sheaf $\mathcal{O}_{K^w}^{\mathrm{la}}$. 

In this paper, by using the following comparison theorem of the completed cohomologies, we replace the unitary Shimura variety $S_K/F$ by another unitary Shimura variety of dimension $d-1$. Let $l \neq p$ be a prime splitting completely in $F$ and we assume that $\overline{\rho}|_{G_{F_{w'}}}$ is irreducible and generic and $U(F^+_{w'}) = \mathrm{GL}_2(F_{w'}^+)$ for any $w' \mid l$. We fix $\overline{\mathbb{Q}}_l \Isom \mathbb{C}$ and $w_0 \mid l$ (resp. $v_0 \mid l$) denotes the induced place of $F$ (resp. $F_0$). Let $U'/F^+$ be the unitary group such that $U' \times_{F^+, \tau} \mathbb{R} = U(1, 1)$ for any $\tau' \in \Psi \setminus \{ \tau \}$, $U' \times_{F^+, \tau} \mathbb{R} = U(0, 2)$ for any $\tau' \notin \Psi$ or $\tau' = \tau$, $U(\mathbb{A}_{F^+}^{\infty, l}) = U'(\mathbb{A}_{F^+}^{\infty, l})$, $U'(F^+_{w'}) = \mathrm{GL}_2(F_{w'}^+)$ for any $l$-adic place $w \neq w_0$ and $U'(F^+_{w_0}) = D_{w_0}^{\times}$, where $D_{w_0}$ be the central division algebra over $F_{w_0}$ of dimension $4$. Let $GU'/\mathbb{Q}$ be the unitary similitude group defined by $U'$ and $T_K/F$ be the Shimura varieties associated with $GU'/\mathbb{Q}$.

\begin{thm}(Corollary \ref{comparison})\label{comparisonII}

For any sufficiently small $K_l, K_l'$, there exist positive integers $n, m \in \mathbb{Z}$ and a $\mathrm{GL}_2(F_w) \times \mathbb{T}^S$-equivariant isomorphism $\widehat{H}^{d}(S_{K_lK^{l, w}}, \mathcal{O})^{\oplus m}_{\mathfrak{m}} \cong \widehat{H}^{d-1}(T_{K'_lK^{l, w}}, \mathcal{O})^{\oplus n}_{\mathfrak{m}}$.

\end{thm}

We give an outline of the proof of theorem \ref{comparisonII}. In the proof, we use the torsion Mantovan formula \cite[Theorem 7.1]{kos}, which is a torsion analogue of the usual Mantovan formula. Now we consider the irreducible parameter at every $w' \mid l$. Thus the considered cohomology groups are equal to the contributions of the basic strata :  

$$H^d_{\et}(S_{K^{l}K_l, \overline{F}}, \mathcal{O}/\varpi^n)_{\mathfrak{m}} \cong \mathcal{A}_{GI}^{\mathrm{sm}}(K^l, \mathcal{O}/\varpi^n)_{\mathfrak{m}} \otimes_{\mathcal{H}(J_b(\mathbb{Q}_l))_{\mathcal{O}/\varpi^n}}^{\mathbb{L}} R\Gamma_c(\mathcal{M}_{(GU_{\mathbb{Q}_l}, b, \mu), K_l}, \mathcal{O}/\varpi^n)[d],$$ $$H^{d-1}_{\et}(T_{K^lK'_l, \overline{F}}, \mathcal{O}/\varpi^n)_{\mathfrak{m}} \cong \mathcal{A}_{GI}^{\mathrm{sm}}(K^l, \mathcal{O}/\varpi^n)_{\mathfrak{m}} \otimes_{\mathcal{H}(J_{b'}(\mathbb{Q}_l))_{\mathcal{O}/\varpi^n}}^{\mathbb{L}} R\Gamma_c(\mathcal{M}_{(GU'_{\mathbb{Q}_l}, b', \mu'), K'_l}, \mathcal{O}/\varpi^n)[d-1].$$ (See {\S} 6.1 for details. $\mathcal{A}_{GI}^{\mathrm{sm}}(K^l, \mathcal{O}/\varpi^n)$ denotes the space of algebraic automorphic forms on an inner form $GI$ satisfying that $GI(\mathbb{R})$ is compact modulo center and $R\Gamma_c(\mathcal{M}_{(GU_{\mathbb{Q}_l}, b, \mu), K_l}, \mathcal{O}/\varpi^n)$ and $R\Gamma_c(\mathcal{M}_{(GU'_{\mathbb{Q}_l}, b', \mu'), K'_l}, \mathcal{O}/\varpi^n)$ are the cohomology of basic Rapoport-Zink spaces. In fact, we have $J_b = J_{b'}$) The important property is that in the above formulas, the only difference between the above two groups is the contribution of the torsion cohomology of Rapoport-Zink spaces. Moreover, by using the tensor product representation $\pi_l^{\mathrm{univ}} := (\otimes_{w' \mid v_0} \pi_{w'}^{\mathrm{univ}}) \otimes_{\otimes_{w' \mid v_0} R_{w'}^{\mathrm{univ}}} R_l^{\mathrm{univ}}$ over $R_{l}^{\mathrm{univ}} := \widehat{\otimes}_{w \mid v_0} R_w^{\mathrm{univ}}$ of the admissible representation $\pi_w^{\mathrm{univ}}$ over $R_w^{\mathrm{univ}}$ corresponding to the universal deformation of $\overline{\rho}|_{G_{F_w}}$, we obtain $\mathrm{Hom}_{R_l[J_b(\mathbb{Q}_l)]}(\pi^{\mathrm{univ}}_l, \mathcal{A}_{GI}^{\mathrm{sm}}(K^{l}, \mathcal{O}/\varpi^n)_{\mathfrak{m}}) \otimes_{R_l} \pi^{\mathrm{univ}}_l \Isom \mathcal{A}_{GI}^{\mathrm{sm}}(K^{l}, \mathcal{O}/\varpi^n)_{\mathfrak{m}}$. Thus\begin{gather} H^d_{\et}(S_{K^lK_l, \overline{F}}, \mathcal{O}/\varpi^m)_{\mathfrak{m}} \notag \\
    \cong (\mathrm{Hom}_{R_l[J_b(\mathbb{Q}_l)]}(\pi^{\mathrm{univ}}_l, \mathcal{A}_{GI}^{\mathrm{sm}}(K^{l}, \mathcal{O}/\varpi^m)_{\mathfrak{m}}) \otimes_{R_l} \pi^{\mathrm{univ}}_l) \otimes^{\mathbb{L}}_{\mathcal{H}(J_b(\mathbb{Q}_l))_{\mathcal{O}/\varpi^m}} R\Gamma_c(\mathcal{M}_{(GU_{\mathbb{Q}_l}, b, \mu), K_l}, \mathcal{O}/\varpi^m)[d] \notag \\
    = (\mathrm{Hom}_{R_l[J_b(\mathbb{Q}_l)]}(\pi^{\mathrm{univ}}_l, \mathcal{A}_{GI}^{\mathrm{sm}}(K^{l}, \mathcal{O}/\varpi^m)_{\mathfrak{m}}) \otimes_{R_l}^{\mathbb{L}} \pi^{\mathrm{univ}}_l) \otimes^{\mathbb{L}}_{\mathcal{H}(J_b(\mathbb{Q}_l))_{\mathcal{O}/\varpi^m}} R\Gamma_c(\mathcal{M}_{(GU_{\mathbb{Q}_l}, b, \mu), K_l}, \mathcal{O}/\varpi^m)[d] \notag \\
    = \mathrm{Hom}_{R_l[J_b(\mathbb{Q}_l)]}(\pi^{\mathrm{univ}}_l, \mathcal{A}_{GI}^{\mathrm{sm}}(K^{l}, \mathcal{O}/\varpi^m)_{\mathfrak{m}}) \otimes^{\mathbb{L}}_{R_l} (\pi^{\mathrm{univ}}_l \otimes^{\mathbb{L}}_{\mathcal{H}(J_b(\mathbb{Q}_l))_{\mathcal{O}}} R\Gamma_c(\mathcal{M}_{(GU_{\mathbb{Q}_l}, b, \mu), K_l}, \mathcal{O})[d]), \notag \\
    H^{d-1}(T_{K^lK_l', \overline{F}}, \mathcal{O}/\varpi^m)_{\mathfrak{m}} \notag \\
    \cong \mathrm{Hom}_{R_l[J_b(\mathbb{Q}_l)]}(\pi^{\mathrm{univ}}_l, \mathcal{A}_{GI}^{\mathrm{sm}}(K^{l}, \mathcal{O}/\varpi^m)_{\mathfrak{m}}) \otimes^{\mathbb{L}}_{R_l} (\pi^{\mathrm{univ}}_l \otimes^{\mathbb{L}}_{\mathcal{H}(J_b(\mathbb{Q}_l))_{\mathcal{O}}} R\Gamma_c(\mathcal{M}_{(GU'_{\mathbb{Q}_l}, b', \mu'), K'_l}, \mathcal{O})[d]) \notag
   \end{gather}

Therefore, we obtain Theorem \ref{comparisonII} by taking the limit $m \rightarrow \infty$ if we can prove that $\pi^{\mathrm{univ}}_l \otimes^{\mathbb{L}}_{\mathcal{H}(J_b(\mathbb{Q}_l))_{\mathcal{O}}} R\Gamma_c(\mathcal{M}_{(GU'_{\mathbb{Q}_l}, b', \mu'), K'_l}, \mathcal{O})$ is concentrated at $d$ and is a nonzero finite free module over $R_l$. These assumptions follow from a part of the torsion Kottwitz conjecture of the Rapoport-Zink space by \cite[Theorem 8.10]{zou}.

\vspace{0.5 \baselineskip}

As a consequence of Theorem \ref{comparisonII}, we obtain $\widehat{H}^{d-1}(T_{K^w}, F_w)[\varphi]^{\mathrm{la}, \mathfrak{gl}_2(F_w)_{\tau}} \neq 0$. For the Shimura variety $T_K$, we can prove the following.

\begin{thm} (Theorems \ref{geometric sen theory} and \ref{geometric sen theoryII}) \label{geometric sen theoryIII}

$\widehat{H}^{d-1}(T_{K^w}, F_w)_{\mathfrak{m}}^{\mathrm{la}, \mathfrak{gl}_2(F_w)_{\tau}} \widehat{\otimes}_{F_w} C \cong \widehat{H}^{d-1}(T_{K_w}, \mathcal{O}_{T_{K^w}}^{\mathrm{la}, \mathfrak{gl}_2(F_w)_{\tau}})_{\mathfrak{m}}$.

\end{thm}

We give an outline of the proof of theorem \ref{geometric sen theoryIII}. This theorem can be proved by a similar method of \cite[Theorem 4.4.6]{PanI}. In fact, as mentioned in \cite[Theorem 3.6.1]{PanI} in the locally analytic case, the essential property to prove Theorem \ref{geometric sen theoryIII} is that the Faltings extension $0 \rightarrow \mathcal{O}_{T_{K^w}}(1) \rightarrow \mathrm{gr}^1\mathcal{O}\mathbb{B}_{\mathrm{dR}, T_{K^w}}^+ \rightarrow \mathcal{O}_{T_{K^w}} \rightarrow 0$ on $T_{K^w}$ remains exact after taking the locally analytic $\mathfrak{gl}_2(F_w)_{\tau}$-invariant part. In fact, this can be proved by the following comparison between the Faltings extension and the Hodge-Tate filtration as \cite[Corollary 4.2.3]{PanI}. (See Proposition \ref{FHT}. We put $F^2 \otimes_{\mathbb{Q}} F_{w} = \oplus_{\sigma \in \mathrm{Hom}(F, F_w)} V_{\sigma}$. The right vertical map is the minus of the Kodaira-Spencer map. Note that the action of $\mathfrak{gl}_2(F_w)_{\tau}$ on $\oplus_{\sigma \in \Psi \setminus \{ \tau \}} V_{\sigma}$ is trivial.)

\xymatrix{
    \displaystyle 0 \ar[r] & \oplus_{\sigma \in \Psi \setminus \{ \tau \}} \mathcal{O}_{T_{K^w}}(1) \ar[r] \ar@{->>}[d]^{\sum_{\sigma} \mathrm{id}} & \oplus_{\sigma \in \Psi \setminus \{ \tau \}} V_{\sigma}(1) \otimes_{F_w} \omega_{T_{K^w}, \tau}  \ar[r] \ar@{->>}[d] & \oplus_{\sigma \in \Psi \setminus \{ \tau \}} \omega_{T_{K^w}, \sigma}^{2} \ar[r] \ar[d]^{-\mathrm{KS}}_{\sim} & 0 \\
    0 \ar[r] & \mathcal{O}_{T_{K^w}}(1) \ar[r] & \mathrm{gr}^1\mathcal{O}\mathbb{B}^+_{\mathrm{dR}, T_{K^w}} \ar[r]  & \mathcal{O}_{\mathcal{T}_{K^w}} \otimes_{\mathcal{O}_{T_{K^w}}^{\mathrm{sm}}} \Omega_{T_{K^w}}^{1, \mathrm{sm}} \ar[r]& 0
    }

 Moreover, we can obtain a very explicit description for $\mathcal{O}_{T_{K^w}}^{\mathrm{la}, \mathfrak{gl}_2(F_w)_{\tau}}$ as in the modular curve case because the action of $\mathfrak{gl}_2(F_w)_{\tau}$ is trivial and we have $U \times_{F^+, \sigma} \mathbb{R} = U(1,1)$ for $\tau \neq \sigma \in \Psi$. Again, if the counterpart of the classicality conjecture \ref{classicalityII} holds for $T_{K^w}$, then we can replace $T_K$ with another unitary Shimura variety of dimension $d - 2$ and by repeating this argument, we can reduce to a unitary Shimura curve case, in which case everything is exactly the same as in the modular curve case.

\subsection{Organization of this paper}

In section 2, we recall basic properties of locally analytic representations. In section 3, we recall and prove basic properties of automorphic representations of rank 2 unitary groups, cohomologies of rank 2 unitary Shimura varieties and elementary results on the locally analytic functions on the perfectoid unitary Shimura varieties. One of the key observations is a modification of geometric Sen theory (see Theorem \ref{geometric sen theoryII}) as we have seen before (cf. Theorem \ref{geometric sen theoryIII}). In section 4, we construct geometric locally analytic de Rham complexes, which are certain locally analytic extensions of de Rham complexes at finite levels and are expected to be classical (cf. Conjecture \ref{classicalityII}). Moreover, we will prove this in the case $[F_w : \mathbb{Q}_p] \le 2$ in {\S} 4.4 (cf. Theorem \ref{quadratic}). In section 5, we will prove the compatibility of the complex obtained by the Fontaine operator and the geometric locally analytic de Rham complexes in the parallel weight case and the $[F_w : \mathbb{Q}_p] \le 2$ case (cf. Theorem \ref{compatibilityIII}). In section 6, we will prove a certain comparison theorem of completed cohomologies of unitary Shimura varieties and complete the proof of our classicality theorem (cf. Theorem \ref{comparisonII}). In section 7, we will give applications to the automorphy lifting theorem and the Breuil-M$\mathrm{\acute{e}}$zard conjecture. In {\S} 7.1, we will prove a certain theorem on liftings of residual representations, which is needed to prove the Breuil-M$\mathrm{\acute{e}}$zard conjecture from ``big $R=T$ theorem + classicality theorem''. (See the discussion after Theorem \ref{Breuil-Mezard conjectureI}.) In {\S} 7.2, we will prove a big $R=T$ theorem, which is a slight modification of \cite{GEKI}. (See Theorem \ref{eigenspaceI}.) In {\S} 7.3, we will prove the automorphy lifting theorem and the Breuil-M$\mathrm{\acute{e}}$zard conjecture in some $\mathrm{GL}_2(\mathbb{Q}_{p^2})$ cases by using theorems proved in the previous sections.

\subsection{Notations} \label{p-adic Hodge type}

For a positive integer $n$, we write $\mathfrak{S}_n$ for the symmetric group of degree $n$. For a finite set $X$, we write $|X|$ for the order of $X$. We put $\mathbb{Z}^n_+ := \{ (m_1, \cdots, m_n) \in \mathbb{Z}^n \mid m_1 \ge \cdots \ge m_n \}$. For a field $K$, we write $\overline{K}$ for an algebraic closure of $K$ and $K^{\mathrm{sep}}$ for a separable closure of $K$. We put $G_K := \mathrm{Gal}(K^{\mathrm{sep}}/K)$. For schemes $X$ and $\mathrm{Spec}B$ over $\mathrm{Spec}A$, we put $X_B := X \times_{\mathrm{Spec}A} \mathrm{Spec}B$. We will use similar notations for rings and adic spaces. For a representation $V$ of a group $G$ over a ring $A$, we put $V^{\vee} := \mathrm{Hom}_A(V,A)$ with an action of $G$ defined by $g \cdot f (v) := f(g^{-1}v)$.

Let $p$ be a prime and $G$ be a profinite group. For a continuous representation $r:G \rightarrow \mathrm{GL}_n(\mathcal{O}_{\overline{\mathbb{Q}}_p})$, we write $\overline{r}$ for the composition $G \rightarrow \mathrm{GL}_n(\mathcal{O}_{\overline{\mathbb{Q}_p}}) \rightarrow \mathrm{GL}_n(\overline{\mathbb{F}_p})$. For a continuous representation $r: G \rightarrow \mathrm{GL}_n(\overline{\mathbb{Q}_p})$, we write $\overline{r}$ for the semisimplification of $\overline{grg^{-1}}: G \rightarrow \mathrm{GL}_n(\overline{\mathbb{F}_p})$ for some $g \in \mathrm{GL}_n(\overline{\mathbb{Q}_p})$ such that $\mathrm{Im}(grg^{-1}) \subset \mathrm{GL}_n(\mathcal{O}_{\overline{\mathbb{Q}_p}})$. The isomorphism class of $\overline{r}$ is independent of the choice of $g$.

Let $l$ be a prime and $D$ be a central division algebra over a finite extension of $\mathbb{Q}_l$ of dimension $4$. We write $\mathcal{O}_D$ for the ring of integers of $D$, $\mathbb{F}_D$ for the residue field of $\mathcal{O}_D$ and $q_D:=|\mathbb{F}_D|$. 

Let $K$ be a finite extension of $\mathbb{Q}_l$. When $l \neq p$, we say that a continuous representation $\overline{\rho} : G_K \rightarrow \mathrm{GL}_2(\overline{\mathbb{F}}_p)$ is generic if $H^0(K, \mathrm{ad} \overline{\rho}(1)) = 0$. We assume $p = l$. For a Hodge-Tate representation $\rho : G_K \rightarrow \mathrm{GL}_n(\overline{\mathbb{Q}}_p)$ and $\tau \in \mathrm{Hom}_{\mathbb{Q}_p}(K, \overline{\mathbb{Q}}_p)$, we put $\mathrm{HT}_{\tau}(\rho):=[\cdots, \underbrace{i-1, \cdots \cdots,  i-1}_{\mathrm{dim}_{\overline{\mathbb{Q}_p}}(\rho \otimes_{\tau, K} \widehat{\overline{K}}(i-1))^{G_K}}, \underbrace{i, \cdots \cdots,  i}_{\mathrm{dim}_{\overline{\mathbb{Q}_p}}(\rho \otimes_{\tau, K} \widehat{\overline{K}}(i))^{G_K}}, \cdots] \in (\mathbb{Z}^n/\mathfrak{S}_n)$. For example, the $p$-adic cyclotomic character $\varepsilon_p$ satisfies $\mathrm{HT}_{\tau}(\varepsilon_p)=-1$ for any $\tau$. For $\lambda \in (\mathbb{Z}^2_+)^{\mathrm{Hom}_{\mathbb{Q}_p}(K, \overline{\mathbb{Q}}_p)}$, we say that a continuous representation $\rho : G_{K} \rightarrow \mathrm{GL}_2(\overline{\mathbb{Q}}_p)$ is de Rham of $p$-adic Hodge type $\lambda$ if $\rho$ is de Rham and $\mathrm{HT}_{\tau}(\rho) = \{ \lambda_{\tau, 1} + 1, \lambda_{\tau, 2} \}$ for any $\tau \in \mathrm{Hom}_{\mathbb{Q}_p}(K, \overline{\mathbb{Q}}_p)$.

For a global field $K$ and for a finite place $v$ of $K$, let $K_v$ be the completion of $K$ at $v$, $\mathbb{F}_v$ be the residue field at $v$ and $q_v := q_{K_v}$.

Let $K$ be a complete nonarchimedean extension field of $\mathbb{Q}_p$. For Banach spaces $V, \ W$ over $K$, we write $V \widehat{\otimes}_{K} W$ for the Banach space $(\varprojlim_n (V^o \otimes_{\mathcal{O}_{K}} W^o)/p^n)[\frac{1}{p}]$ with a unit ball $(\varprojlim_n (V^o \otimes_{\mathcal{O}_{K}} W^o)/p^n)$, where $V^o$ and $W^o$ denote the unit balls of $V$ and $W$. For a profinite group $G$ and a Banach space $W$, let $C(G, W)$ denote the set of continuous functions from $G$ to $W$. This is a Banach representation of $G$ with a unit ball $C(G, W^o)$ by the action $(gf)(x) = f(g^{-1}x)$ for $x, g \in G$ and $f \in C(G, W)$. We say that a locally convex space $V$ over $K$ is a LB space if $V \cong \varinjlim_n V_n$ for an inductive system of Banach spaces $\{ V_n \}_{n \in \mathbb{Z}_{>0}}$ over $K$ with the locally convex topology. Note that if $V$ is Hausdorff, we may assume that all transition maps $V_n \rightarrow V_m$ are injective by \cite[{\S} 2.1.4]{PanI}. Thus in the following, unless we specify, when we write $V = \varinjlim_{n} V_n$ for a Hausdorff LB space $V$, we always assume that all transition maps $V_n \rightarrow V_m$ are injective. For Hausdorff LB spaces $V = \varinjlim_n V_n, W = \varinjlim_n W_n$ over $K$, we define $V \widehat{\otimes}_K W = \varinjlim_{n} V_n \widehat{\otimes}_K W_n$. This is independent of the expressions $V = \varinjlim_n V_n$ and $W = \varinjlim_n W_n$.

\subsection{Acknowledgement}

The author is grateful to his advisor Takeshi Saito for careful reading of an earlier version of this paper and for his constant support and encouragement. This work was supported by the WINGS-FMSP program at the Graduate School of Mathematical Sciences, the University of Tokyo and JSPS KAKENHI Grant number 24KJ0812.

\section{Preliminaries on locally analytic representations}

In this section, we recall some properties of Banach representations and locally analytic representations.

\begin{prop}\label{inductive limit}
    
Let $W$ be a Hausdorff LB space over $\mathbb{Q}_p$ and let $\{ W_i \}_{i \in \mathbb{Z}_{>0}}$ be a family of Banach spaces over $\mathbb{Q}_p$ with continuous injections $u_i : W_i \hookrightarrow W$ over $\mathbb{Q}_p$ for all $i$ such that $W = \cup_{i \in \mathbb{Z}_{>0}} u_i(W_i)$. Then for any Banach space $V$ over $\mathbb{Q}_p$ and any continuous map $g : V \rightarrow W$, there exist $i$ and a continuous map $g_i : V \rightarrow W_i$ over $\mathbb{Q}_p$ such that $V \xrightarrow{g_i} W_i \xrightarrow{u_i} W$ is equal to $g$. %In particular, for any inductive systems $\{ V_i \}_{i \in \mathbb{Z}_{>0}}, \{ W_i \}_{i >0}$ of Banach spaces over $\mathbb{Q}_p$ with injective transition maps such that $\varinjlim_{i} V_i \cong \varinjlim_i W_i$ and these are Hausdorff and for any $i$, there exists $j$ such that $V_i \rightarrow \varinjlim_{i} V_i \cong \varinjlim_i W_i$ factors through a continuous map $V_i \hookrightarrow W_j$.
        
\end{prop}

\begin{proof}  See \cite[Corollary 8.9]{funcana}. \end{proof}

\begin{lem}\label{Hausdorff}

Let $V = \varinjlim_{i \in \mathbb{Z}_{>0}}V_i$ be a Hausdorff LB space over a finite extension $K$ of $\mathbb{Q}_p$ and let $W$ be a Banach space over $K$. Then $V \widehat{\otimes}_K W$ is also a Hausdorff LB space over $K$.

\end{lem}

\begin{proof}

By using \cite[Lemma 4.6]{funcana}, it suffices to prove that for any $0 \neq v \in V_j \widehat{\otimes}_K W$, there exists an open lattice $L$ of $V \widehat{\otimes} W$ such that $v \notin L$. We may assume that $v \in V_j^o \widehat{\otimes}_{\mathcal{O}_K} W^o \setminus pV_j^o \widehat{\otimes}_{\mathcal{O}_K} W^o$. Then we can write $v = \sum_{i=1}^n u_i \otimes w_i + pu$ such that $u_i \in V_j^o$, $w_i \in W^o$, $u \in V_j^o \widehat{\otimes}_{\mathcal{O}_K} W^o$ and $\sum_{i} u_i \otimes w_i \notin p V_j^o \widehat{\otimes}_{\mathcal{O}_K} W^o$. We replace $u_i$'s with a basis of $V_j^{o} \cap \langle u_i \rangle_{K}$. By using the Hahn-Banach theorem \cite[Proposition 9.2]{funcana}, we have a continuous $K$-linear map $f : W \rightarrow K$ such that $f(W^o) = \mathcal{O}_K$ and $f(w_i) \in \mathcal{O}_K \setminus p\mathcal{O}_K$ for some $i$. This induces the map $1 \otimes f : V \widehat{\otimes}_{K} W \rightarrow V$ satisfying $1 \otimes f (v) \neq 0$. Since $V$ is Hausdorff, we obtain the result. \end{proof}

\begin{lem}\label{tensor product}

Let $C^0 \rightarrow C^1 \rightarrow \cdots \rightarrow C^d$ be a bounded complex of Hausdorff LB spaces over a finite extension $K$ of $\mathbb{Q}_p$ with strict maps having closed images and let $W$ be a Banach space over $K$. Then $C^0 \widehat{\otimes}_K W \rightarrow C^1 \widehat{\otimes}_K W \rightarrow \cdots \rightarrow C^d \widehat{\otimes}_K W$ satisfies the same properties as $C^{\bullet}$ and $H^i(C^{\bullet}) \widehat{\otimes}_K W = H^i(C^{\bullet} \widehat{\otimes}_K W)$.

\end{lem}

\begin{proof}

Note that if we can prove $H^i(C^{\bullet}) \widehat{\otimes}_K W = H^i(C^{\bullet} \widehat{\otimes}_K W)$,  the maps of the complex $C^{\bullet} \widehat{\otimes} W$ have closed images by Lemma \ref{Hausdorff}. By using the open mapping theorem \cite[Proposition 8.8]{funcana} and decomposing $C^0 \rightarrow C^1 \rightarrow \cdots \rightarrow C^d$ by short exact sequences, it suffices to prove that for any short exact sequence $0 \rightarrow M_1 \rightarrow M_2 \rightarrow M_3 \rightarrow 0$ of Hausdorff LB spaces over $K$, the induced sequence $0 \rightarrow M_1 \widehat{\otimes}_{K} W \rightarrow M_2 \widehat{\otimes}_{K} W \rightarrow M_3 \widehat{\otimes}_{K} W \rightarrow 0$ is exact. Again by using the open mapping theorem \cite[Proposition 8.8]{funcana}, we may assume that $M_1$, $M_2$ and $M_3$ are Banach spaces and all maps are strict. Thus we may assume that the image of $M_2^{o}$ in $M_3$ and $M_2^{o} \cap M_1$ are unit balls of $M_3$ and $M_1$ respectively. Since $M_3^o$ is $p$-torsion free, the exact sequence $0 \rightarrow M_1^o \rightarrow M_2^o \rightarrow M_3^o \rightarrow 0$ induces the exact sequence $0 \rightarrow M_1^o \widehat{\otimes}_{\mathcal{O}_K} W^{o} \rightarrow M_2^o \widehat{\otimes}_{\mathcal{O}_K} W^{o} \rightarrow M_3^o \widehat{\otimes}_{\mathcal{O}_K} W^{o} \rightarrow 0$, which implies the result. \end{proof}

For a $p$-adic Lie group $G$ and a Banach representation $V$ of $G$ over $\mathbb{Q}_p$, we can consider the subspace $V^{\mathrm{la}}$ of locally analytic vectors in $V$. We recall the definition of $V^{\mathrm{la}}$ only in the case $G = \prod_{v \in T} \mathrm{GL}_{n_v}(D_v)$, where $T$ is a finite non-empty index set, $n_v \in \mathbb{Z}_{>0}$ and $D_v$ is a central division algebra over a finite extension $F_v$ of $\mathbb{Q}_p$. We only use such $p$-adic Lie groups in this paper. (Precisely, we also consider $G=\mathbb{Z}_p^n$, but such groups can be regarded as an open subgroup of our groups by $\mathrm{exp}(p^2 \cdot) : \mathbb{Z}_p^n \Isom (1+ p^2\mathbb{Z}_p)^n$ or in that case, the definition $V^{\mathrm{la}}$ is clear.) See \cite[Section 2]{PanI} and \cite{pL} for details of general cases. We consider a subgroup $G_0 = \prod_{v \in T} (I_{n_v} + p^n\mathrm{M}_{n_v}(\mathcal{O}_{D_v}))$ for $n > 1$. Then the usual exponential map induces a homeomorphism $\mathrm{exp} : p^n\prod_{v \in T}\mathrm{M}_{n_v}(\mathcal{O}_{D_v}) \Isom G_0$. (This is called a coordinate of the first kind.) We have $G_k := G_0^{p^k} = \prod_{v \in T} (I_{n_v} + p^{n+k}\mathrm{M}_{n_v}(\mathcal{O}_{D_v}))$ for $k \ge 0$. We fix an identification $p^n\prod_{v \in T}\mathrm{M}_{n_v}(\mathcal{O}_{D_v}) = \mathbb{Z}_p^m$ over $\mathbb{Z}_p$ for some $m$. We obtain a homeomorphism $\mathrm{exp} : p^k\mathbb{Z}_p^{m} \Isom G_k$. For a Banach space $W$ over $\mathbb{Q}_p$, let $C^{\mathrm{an}}(G_k, W)$ be the subspace $C(G_k, W)$ consisting of analytic functions $f$ on $G_k$, i.e., $\displaystyle f \circ \mathrm{exp}(x_1, \cdots, x_m) = \sum_{i=(i_1, \cdots, i_m) \in \mathbb{Z}_{\ge 0}^n} a_{i} x_1^{i_1} \cdots x_{m}^{i_m}$ for any $(x_1, \cdots, x_m) \in p^k\mathbb{Z}_p^m$, where $a_i \in W$ satisfies $a_{i}p^{k(i_1 + \cdots + i_m)} \rightarrow 0$ as $i_1+\cdots+i_m \rightarrow \infty$. This definition is independent of the choice of the identification $\mathbb{Z}_p^m = p^n\prod_{v \in T}\mathrm{M}_{n_v}(\mathcal{O}_{D_v})$. This is a Banach space by the norm $|| f ||_{G_k} = \mathrm{sup}_{i} || a_i ||p^{-k(i_1 + \cdots + i_m)}$.
For $f \in C^{\mathrm{an}}(G_k, W)$ and $X \in \mathfrak{g} := \mathrm{Lie} (G) = \prod_{v \in T}\mathrm{M}_{n_v}(D_v) = \mathbb{Q}_p^m$, we define $X f(g) := \frac{d}{dt} f(g(\mathrm{exp}(tX)))|_{t=0}$. This is also independent of the choice of the identification $\mathbb{Z}_p^m = p^n\prod_{v \in T}\mathrm{M}_{n_v}(\mathcal{O}_{D_v})$ and this defines an action of the Lie algebra $\mathfrak{g}$ over $\mathbb{Q}_p$ on $C^{\mathrm{an}}(G_k, W)$. 

Assume that $W$ is a Banach representation of $G_k$. Then we have an action of $G_k$ on $C^{\mathrm{an}}(G_k, W)$ by $(g \cdot f)(x) = g (f (g^{-1}x))$ for $g, x \in G_k$ and $f \in C^{\mathrm{an}}(G_k, W)$. Then the evaluation map $C^{\mathrm{an}}(G_k, W) \rightarrow W$ at the unit element $e$ is an injection on the invariant part of this action $f : C^{\mathrm{an}}(G_k, W)^{G_k} \hookrightarrow W$. Note that $f$ is equivariant if we consider the right translation action of $G_k$ on $C^{\mathrm{an}}(G_k, W)^{G_k}$. We put $W^{G_k-\mathrm{an}}:=\mathrm{Im}f$ and call this the space of $G_k$-analytic vectors. Note that $W^{G_k- \mathrm{an}}$ is a Banach space by $C^{\mathrm{an}}(G_k, W)^{G_k} \Isom W^{G_k- \mathrm{an}}$ and by using this isomorphism, the Lie algebra $\mathfrak{g}$ acts on $W^{G_k- \mathrm{an}}$. Let $W^{\mathrm{la}} := \varinjlim_{k' \ge k} W^{G_{k'} - \mathrm{an}}$. Note that $W^{\mathrm{la}}$ is a Hausdorff LB space because we have a continuous injection $W^{\mathrm{la}} \hookrightarrow W$ of $W^{\mathrm{la}}$ into the Hausdorff space $W$. Note that $W^{\mathrm{la}}$ is equipped with the natural action of $\mathfrak{g}$ and $G_k$. Moreover, if the action of $G_k$ on $W$ extends to $G$, then $W^{\mathrm{la}}$ is $G$-stable by the construction of $W^{\mathrm{la}}$. We recall some basic properties of $W^{\mathrm{la}}$.

\begin{lem} \label{norm}
    
Let $W$ be a Banach representation of $G_k$ over $\mathbb{Q}_p$ and $x \in W^{G_k-\mathrm{an}}$. Then the following holds.

1 \ $||x||_{G_{k + 1}} \le ||x||_{G_{k}}$.

2 \ $||x||$ is equal to $||x||_{G_{k'}}$ for sufficiently large $k'$.

\end{lem}

\begin{proof}

This is clear from the definition of $|| \ ||_{G_k}$. \end{proof}

\begin{prop}\label{continuity}

For any $D \in \mathfrak{g}$ and $k$, there exists a constant $C_{D,k}>0$ such that $||D(x)||_{G_{k}} \le C_{D, k}||x||_{G_k}$ for any $x \in W^{G_k-\mathrm{an}}$.

\end{prop}

\begin{proof} See \cite[Lemma 2.6]{BC}. \end{proof}

\begin{prop}\label{locally analytic}

Let $V$ be a Banach representation of $G_k$ over $\mathbb{Q}_p$.
        
Then the following conditions are equivalent.
        
1 \ $V = V^{\mathrm{la}}$.
        
2 \ $V = V^{G_{k'}-an}$ for some $k' \ge k$.
        
3 \ For any positive integer $l$, there exist $k' \ge k$ and a $G_{k'}$-stable lattice $V^+$ of $V$ such that the action of $G_{k'}$ is trivial on $V^+/p^lV^+$ for some $k' \ge k$.
        
\end{prop}

\begin{proof} See \cite[Lemma 2.1.4]{Cam}. \end{proof}

\begin{ex} \label{example}

1 \ Any finite-dimensional representation of $G_k$ over $\mathbb{Q}_p$ satisfies the condition 3 and hence all conditions of Proposition \ref{locally analytic}.

2 \ Any complete f-adic ring $A$ which is topologically of finite type over a complete nonarchimedean extension field $K$ of $\mathbb{Q}_p$ with a continuous action of $G_k$ over $K$ also satisfies the condition 3 and hence all conditions of Proposition \ref{locally analytic}.

3 \ The $\mathbb{Z}_p(1)$-perfectoid cover $C\langle T^{\pm \frac{1}{p^{\infty}}} \rangle := \widehat{\cup_k \mathcal{O}_{C}\langle T^{\pm \frac{1}{p^{k}}} \rangle}[\frac{1}{p}]$ of $C\langle T^{\pm 1} \rangle$, which can be regarded as a representation of $\mathbb{Z}_p \cong \mathbb{Z}_p(1)$ doesn't satisfy the conditions of Proposition \ref{locally analytic}. Actually, we can prove $C\langle T^{\pm \frac{1}{p^{\infty}}} \rangle^{\mathrm{la}} = \cup_k C\langle T^{\pm \frac{1}{p^{k}}} \rangle$ by using the Tate normalized trace. (cf. \cite[example 3.1.6]{PanI}.)

\end{ex}

\begin{dfn}

Let $V$ be a Banach representation of $G_k$ over $\mathbb{Q}_p$. We say that $V$ is an admissible representation if there exists $k' \ge k$ and a positive integer $m$ such that there exists a closed $G_{k'}$-equivariant embedding of $V$ into $C(G_{k'}, \mathbb{Q}_p)^{\oplus m}$.

\end{dfn}

\begin{thm}\label{density of locally analytic vectors}

Let $V$ be an admissible Banach representation of $G_k$ over $\mathbb{Q}_p$. Then $V^{\mathrm{la}}$ is a dense subspace of $V$.

\end{thm}

\begin{proof} See \cite[Theorem 7.1]{ST}. \end{proof}

In this paper, we will use a variant of the locally analytic vector. Let $L$ be a finite extension of $\mathbb{Q}_p$ such that for $\mathrm{Hom}_{\mathbb{Q}_p}(F_v, \overline{L}) = \mathrm{Hom}_{\mathbb{Q}_p}(F_v, L)$ for any $v \in T$ and $D_v \otimes_{F_v, \tau} L = M_{m_v}(L)$ for any $\tau \in \mathrm{Hom}_{\mathbb{Q}_p}(F_v, L)$ and some $m_v \in \mathbb{Z}_{>0}$. Let $\Sigma := \sqcup_{v \in T} \mathrm{Hom}_{\mathbb{Q}_p}(F_v, L)$. Then we have a natural decomposition $\mathfrak{g}_L := \mathfrak{g} \otimes_{\mathbb{Q}_p} L = \oplus_{\tau \in \Sigma} \mathfrak{g}_{\tau}$, where $\mathfrak{g}_{\tau} = M_{n_vm_v}(L)$ if $\tau \in \mathrm{Hom}_{\mathbb{Q}_p}(F_v, L)$. Let $\Psi$ be a subset of $\Sigma$.

For a Banach representation $W$ of $G_k$ over $L$, we consider the $G_k$-analytic $\oplus_{\tau \notin \Psi} \mathfrak{g}_{\tau}$-trivial part: $W^{G_k-\Psi-\mathrm{an}} := W^{G_k-\mathrm{an}, \oplus_{\tau \notin \Psi} \mathfrak{g}_{\tau}} = (C^{\mathrm{an}}(G_k, W)^{\oplus_{\tau \notin \Psi} \mathfrak{g}_{\tau}})^{G_k} = (C^{\mathrm{an}}(G_k, L)^{\oplus_{\tau \notin \Psi} \mathfrak{g}_{\tau}} \widehat{\otimes}_L W)^{G_k}$. Note that we consider $\mathbb{Q}_p$-analytic functions and don't consider $L$-analytic functions. This is a closed subspace of $W^{G_k-\mathrm{an}}$ and $\mathfrak{g}_{\Psi}:=\oplus_{\tau \in \Psi} \mathfrak{g}_{\tau}$ continuously acts on this by Proposition \ref{continuity}. We also consider $W^{\Psi-\mathrm{la}} := W^{\mathrm{la}, \oplus_{\tau \notin \Psi} \mathfrak{g}_{\tau}} = \varinjlim_{k' \ge k}(C^{\mathrm{an}}(G_{k'}, W)^{\oplus_{\tau \notin \Psi} \mathfrak{g}_{\tau}})^{G_{k'}}$.

If $\Psi = \Phi$, then $W^{\Psi-\mathrm{la}} = W^{\mathrm{la}}$ and if $\Psi$ is empty, then $W^{\Psi-\mathrm{la}} = W^{\mathrm{sm}}$. For subsets $\Psi' \subset \Psi$, we have $M^{\Psi'-G_n-\mathrm{an}} \subset M^{\Psi-G_n-\mathrm{an}}$. 

\vspace{0.5 \baselineskip}

We introduce the ``derived'' functor $R^i\Psi\mathfrak{LA}(W) := \varinjlim_k H^i(G_k, C^{\mathrm{an}}(G_k, L)^{\oplus_{\tau \notin \Psi} \mathfrak{g}_{\tau}} \widehat{\otimes}_{L} W)$ of the functor $W \mapsto W^{\Psi-\mathrm{la}}$, which is a variant of $R^i\mathfrak{LA}$ introduced in \cite[{\S} 2.2]{PanI}. Note that the category of Banach representations is not abelian. The author doesn't know any appropriate universal property which characterizes $R^i\Psi\mathfrak{LA}$. However, by the definition of $R^i\Psi\mathfrak{LA}$, we can obtain the following, which are very similar to results in \cite[{\S} 2.2]{PanI}. 

\begin{lem}\label{fundamental}

    1 \ Let $W$ be a Banach representation of $G_k$ over $L$. Then $W^{\Psi-\mathrm{la}} = R^0\Psi\mathfrak{LA}(W)$.

    2 \ Let $0 \rightarrow W_1 \rightarrow W_2 \rightarrow W_3 \rightarrow 0$ be a short exact sequence of Banach representations of $G_k$ over $L$. Then we have a long exact sequence $0 \rightarrow W_1^{\Psi-\mathrm{la}} \rightarrow W_2^{\Psi-\mathrm{la}} \rightarrow W_3^{\Psi-\mathrm{la}} \rightarrow R^1\Psi\mathfrak{LA}(W_1) \rightarrow R^1\Psi\mathfrak{LA}(W_2) \rightarrow \cdots$.

    3 Let $0 \rightarrow M \rightarrow W^0 \rightarrow W^1 \rightarrow \cdots \rightarrow W^s \rightarrow 0$ be an exact sequence of Banach representations of $G_k$ over $L$ with strict maps such that $R^i\Psi\mathfrak{LA}(W^j) = 0$ for any $i \ge 1$ and $j \ge 0$. Then we have $R^i\Psi\mathfrak{LA}(M) \cong H^i((W^{\cdot} )^{\Psi-\mathrm{la}})$ for any $i \ge 0$.
    
    4 \ Let $0 \rightarrow W^0 \rightarrow W^1 \rightarrow \cdots \rightarrow W^s \rightarrow 0$ be a bounded complex of Banach representations of $G_k$ over $L$ with strict maps such that $R^i\Psi\mathfrak{LA}(W^j) = 0$ for any $i \ge 1$ and any $j$. Then we have a spectral sequence $E_2^{u, t} := R^u\Psi\mathfrak{LA}(H^t(W^{\cdot})) \Rightarrow H^{u+t}((W^{\cdot})^{\Psi-\mathrm{la}})$.
    
    %In particular, for an integer $0 \le d \le k$, if $H^i(W^{\cdot}) = 0$ for any $i \neq d$, then we have $H^i((W^{\cdot})^{\Psi-\mathrm{la}}) = 0$ for $i < d$, $H^d((W^{\cdot})^{\Psi-\mathrm{la}}) = H^d(W^{\cdot})^{\Psi-\mathrm{la}}$ and $H^i((W^{\cdot})^{\Psi-\mathrm{la}}) = R^{i-d}\Psi\mathfrak{LA}(H^d(W^{\cdot}))$ for any $i > d$.
    
    \end{lem}

\begin{proof}

The property 1 is clear. The property 2 follows from the exactness of $$0 \rightarrow C^{\mathrm{an}}(G_k, L)^{\oplus_{\tau \notin \Psi} \mathfrak{g}_{\tau}} \widehat{\otimes}_L W_1 \rightarrow C^{\mathrm{an}}(G_k, L)^{\oplus_{\tau \notin \Psi} \mathfrak{g}_{\tau}} \widehat{\otimes}_L W_2 \rightarrow C^{\mathrm{an}}(G_k, L)^{\oplus_{\tau \notin \Psi} \mathfrak{g}_{\tau}} \widehat{\otimes}_L W_3 \rightarrow 0. $$

The statement 3 is clear. The spectral sequence of 4 is induced by the following bicomplex $C^{\cdot, \cdot}$.

\xymatrix{ \ & \ & \ & \\
\varinjlim_{k' \ge k} C(G_{k'}^2, C^{\mathrm{an}}(G_{k'}, L)^{\oplus_{\tau \notin \Psi} \mathfrak{g}_{\tau}} \widehat{\otimes}_LW^0) \ar[r] \ar[u] & \cdots \ar[r] \ar[u] & \varinjlim_{k' \ge k}C(G_{k'}^2, C^{\mathrm{an}}(G_{k'}, L)^{\oplus_{\tau \notin \Psi} \mathfrak{g}_{\tau}} \widehat{\otimes}_L W^s) \ar[u] \\
\varinjlim_{k' \ge k} C(G_{k'}, C^{\mathrm{an}}(G_{k'}, L)^{\oplus_{\tau \notin \Psi} \mathfrak{g}_{\tau}} \widehat{\otimes}_LW^0) \ar[r] \ar[u] & \cdots \ar[r] \ar[u] & \varinjlim_{k' \ge k} C(G_{k'}, C^{\mathrm{an}}(G_{k'}, L)^{\oplus_{\tau \notin \Psi} \mathfrak{g}_{\tau}} \widehat{\otimes}_L W^s) \ar[u] \\
\varinjlim_{k' \ge k} C^{\mathrm{an}}(G_k, L)^{\oplus_{\tau \notin \Psi} \mathfrak{g}_{\tau}} \widehat{\otimes}_L W^0 \ar[r] \ar[u] & \cdots \ar[u] \ar[r] & \varinjlim_{k' \ge k} C^{\mathrm{an}}(G_k, L)^{\oplus_{\tau \notin \Psi} \mathfrak{g}_{\tau}} \widehat{\otimes}_L W^s \ar[u] 
}

Note that $(W^{\cdot})^{\Psi-\mathrm{la}} \rightarrow C^{\mathrm{tot}}$ is a quasi-isomorphism by the assumption $R^i\Psi\mathfrak{LA}(W^j) = 0$ for any $i \ge 1$ and any $j$.   \end{proof}

Finally, we state a certain formal expansion principle, which was used in \cite[proofs of Theorems 3.6.1 and 4.3.9]{PanI}.

Let $B$ be a commutative Banach algebra over $L$ with a continuous action of $G_k$ over $L$ and let $\mathfrak{b}$ be a solvable Lie subalgebra of $\mathfrak{g}_{\Psi}$. We fix a sequence of subalgebras $0=\mathfrak{b}_0 \subset \mathfrak{b}_1 \subset \cdots \subset \mathfrak{b}_s$ such that $\mathfrak{b}_{i-1}$ is an ideal of $\mathfrak{b}_{i}$ and $\mathfrak{b}_i/\mathfrak{b}_{i-1}$ has dimension $1$ and for any $i$, fix a representative $X_i \in \mathfrak{b}_i$ of a generator of $\mathfrak{b}_i/\mathfrak{b}_{i-1}$. Assume that $B^{\mathrm{sm}}$ is dense in $B$ and for any $i$, there exist $v_1, \cdots, v_s \in B^{\Psi-G_{k}-\mathrm{la}}$ such that $X_i(v_i) = 1$ and $X_j(v_i) = 0$ for any $i > j$.

Take an increasing sequence of integers $k \le r(1) < r(2) < \cdots < r(n)$\footnote{We forget the previous $n$.} $< \cdots $ and take $v_{i, n} \in B^{G_{r(n)}}$ such that $|| v_i - v_{i,n} || \le p^{-n - 1}$. Moreover, after replacing $r(n)$ if necessary, we may assume $|| v_i - v_{i,n} ||_{G_{r(n)}} = || v_i - v_{i, n}||$ by Lemma \ref{norm}.

\begin{lem}\label{formal expansion} In the above situation, we obtain the following statements.

1 \ There exists $n_0 > 0$ such that for any $n \ge n_0$ and any $f \in B^{G_{k}-\Psi-\mathrm{an}}$, we have $$f = \sum_{j:=(j_1, \cdots, j_s) \in \mathbb{Z}^s_{\ge 0}} a_{j} \prod_{i} (v_i - v_{i,n})^{j_i}$$ in $B^{G_{r(n)}-\Psi-\mathrm{an}}$ for some $a_j \in B^{G_{r(n)}-\Psi-\mathrm{an}, \mathfrak{b}}$ satisfying $\mathrm{sup}_j || a_j ||_{G_{r(n)}}p^{-n(j_1 + \cdots + j_s)} \le || f ||_{G_k}$.

Moreover, $a_j$'s are uniquely determined by $f$.

2 \ The subspace $B^{G_{r(n)}-\Psi-\mathrm{an}, \mathfrak{b}}\lbrace v_1 - v_{1, n}, \cdots, v_s - v_{s, n} \rbrace := \{ f = \sum_{j:=(j_1, \cdots, j_s) \in \mathbb{Z}^s_{\ge 0}} a_{j} \prod_{i} (v_i - v_{i, n})^{j_i} \mid a_j \in B^{G_{r(n)}-\Psi-\mathrm{an}, \mathfrak{b}}, \ \mathrm{sup}_j || a_j ||_{G_{r(n)}}p^{-n(j_1 + \cdots + j_s)} < \infty \}$ of $B^{\Psi-G_{r(n)}-\mathrm{an}}$ is a Banach space by the norm $|| f ||_{v_{\cdot}} := \mathrm{sup}_{j} || a_j ||_{G_{r(n)}}p^{-(n+1)(j_1 + \cdots + j_s)}$ and by this Banach space structure, the natural inclusion $B^{G_{r(n)}-\Psi-\mathrm{an}, \mathfrak{b}}\lbrace v_1- v_{1, n}, \cdots, v_s- v_{s, n} \rbrace \hookrightarrow B^{G_{r(n)}-\Psi-\mathrm{an}}$ is continuous and the natural injection $B^{G_{k}-\Psi-\mathrm{an}} \hookrightarrow B^{G_{r(n)}-\Psi-\mathrm{an}}$ factors through a continuous injection $B^{G_{k}-\Psi-\mathrm{an}} \hookrightarrow  B^{G_{r(n)}-\Psi-\mathrm{an}, \mathfrak{b}}\lbrace v_1 - v_{1, n} ,\cdots v_s - v_{s,n} \rbrace$.

\end{lem}

%\begin{rem}

%For any $f = \sum_j a_j \prod_{i = 1}^s(v_i - v_{i, n})^{j_i} \in B^{G_{r(n)}-\Psi-\mathrm{an}, \mathfrak{b}}\lbrace v_1 - v_{1, n}, \cdots, v_s - v_{s, n} \rbrace$, we can consider an integration of $f$ along $v_i - v_{i, n}$ : $\sum_j \frac{a_j}{j_i+1} (v_i - v_{i, n})^{j_i+1}\prod_{i' \neq i}(v_{i'} - v_{i', n})^{j_{i'}}$ in $B^{G_{r(n)}-\Psi-\mathrm{la}}$ (but not necessarily in $B^{G_{r(n)}-\Psi-\mathrm{an}, \mathfrak{b}}\lbrace v_1 - v_{1, n}, \cdots, v_s - v_{s, n} \rbrace$) by the assumption $|| v_i - v_{i, n} ||_{G_{r(n)}} \le p^{-n - 1}$.

%\end{rem}

\begin{proof} See \cite[{\S} 4.3.5 $\sim$ Theorem 4.3.9]{PanI}. \end{proof}

\section{Preliminaries on rank 2 unitary Shimura varieties}

In this section, we prepare and prove some fundamental results on unitary groups which will be used later.

\subsection{Rank 2 unitary groups and its automorphic representations}

From here to {\S} 6, we consider the following situation. 

\begin{itemize}
    \item $F_0$ is an imaginary quadratic field.
    \item $F^+$ is a totally real field.
    \item $F:=F_0F^+$.
    \item $\Phi := \mathrm{Hom}_{F_0}(F, \mathbb{C})$. 
    \item $\Psi$ is a subset of $\Phi$.
    \item $d:=|\Psi|$.
    \item $S(B)$ is a finite set of finite places of $F$ such that any prime lying below a place of $S(B)$ splits in $F_0$, $S(B) = S(B)^c$ and $\frac{1}{2}|S(B)| + d = 0 \mod 2$.
    \item $B$ is a central simple algebra over $F$ defined by $\mathrm{dim}_{F}B = 4$, $\mathrm{inv}_{F_v}(B_{F_v}) = \frac{1}{2}$ for any $v \in S(B)$ and $\mathrm{inv}_{F_v}(B_{F_v}) = 0$ for any $v \notin S(B)$. (Thus if $S(B)$ is empty, we have $B = M_2(F)$.)
    \item $V:=B$.
    \item If $S(B)$ is empty, we assume that $F^+ \neq \mathbb{Q}$ and $\Psi \neq \Phi$. These assumptions are needed to use Theorem \ref{base changeIII} and Theorem \ref{kottwitz conjecture}.
    \end{itemize}

By the assumption $\frac{1}{2}|S(B)| + d = 0 \mod 2$, we have the following. (Let $c$ be the complex conjugation on $F$.)

\begin{prop}\label{unitary groups}

 1 \ There exist an anti-involution $* : B \rightarrow B$ and an invertible element $\beta \in B^{* = -1}$ such that $*|_{F} = c$, $\mathrm{tr}_{B/\mathbb{Q}}(bb^*) > 0$ for any $0 \neq b \in B$ and the unitary group $U$ over $F^+$ defined by the following formula having the properties that $U \times_{F^+, \tau} \mathbb{R} = U(1,1)$ for any $\tau \in \Psi$, $U \times_{F^+, \tau} \mathbb{R} = U(0,2)$ for any $\tau \notin \Psi$ and $U_{F^+_v}$ is quasi-split for any finite place $v \notin S(B)$.

$U(R):=\{ g \in ( B^{\mathrm{op}} \otimes_{F^+} R )^{\times} \mid \psi(gv, gw) = \psi(v,w) \ \mathrm{for \ any \ } v, w \in V \otimes_{F^+} R \}$ for any $F^+$-algebra $R$, where $\psi(v, w) := \mathrm{tr}_{B/F}(v \beta w^*)$. (Note that we regard $V$ as a $B^{\mathrm{op}}$-module by the right multiplication. Actually, we have $B^{\mathrm{op}} \cong B$ because $B^{\mathrm{op}}_{F_v} \cong B_{F_v}$ for any finite place $v$.)

2 \ Let $*'$ and $\beta'$ also satisfy the above conditions and let $U'$ be the unitary group defined by $*'$ and $\beta'$. Then there exists $b \in B^{\times}$ such that $\beta' g^{*'} \beta^{' -1} = b \beta b^{*}g^{*}(b^{*})^{-1} \beta^{-1}b^{-1} $ for any $g \in B$. In particular, we have an isomorphism $U' \Isom U, g \mapsto b^{-1}gb$.

    \end{prop}
    
\begin{rem} If $S(B)$ is empty, we can define $g^* := ^t\!g^c$ for any $g \in B = M_2(F)$. \end{rem}

\begin{proof} See \cite[proof of Lemma I.7.1]{HT} and \cite[proof of Theorem 5.4]{LRC}.  \end{proof}

\begin{dfn}\label{cohomological}

Let $G$ be a connected reductive group over $\mathbb{Q}$, $B$ be a Borel subgroup of $G_{\mathbb{C}}$, $T$ be a maximal torus of $G_{\mathbb{C}}$, $K_{\infty}$ be a maximal compact subgroup of $G(\mathbb{R})$, $\mathrm{Lie}(G(\mathbb{R}))$ be the Lie algebra of the Lie group $G(\mathbb{R})$, $\lambda$ be a dominant weight of $T$ , $V_{\lambda}$ be an irreducible algebraic representation of $G_{\mathbb{C}}$ corresponding to $\lambda$ and $\pi$ be an automorphic representation of $G(\mathbb{A}_{\mathbb{Q}})$.

1 \ We say that $\pi$ is cohomological of weight $\lambda$ if $H^i(\mathrm{Lie}(G(\mathbb{R})) \otimes_{\mathbb{R}} \mathbb{C}, K_{\infty}; \pi_{\infty} \otimes V_{\lambda}) \neq 0$ for some $i$.

2 \ We simply say that $\pi$ is cohomological if there exists a dominant weight $\lambda$ such that $\pi$ is cohomological of weight $\lambda$.

For any connected reductive group $G$ over a number field $L$, we define the notion of being cohomological for automorphic representations of $G(\mathbb{A}_{L})$ by the identification $G(\mathbb{A}_L) = (\mathrm{Res}_{L/\mathbb{Q}}G)(\mathbb{A}_{\mathbb{Q}})$.

\end{dfn}

Note that since $U$ is anisotropic modulo center, all discrete automorphic representations of $U(\mathbb{A}_{F^+})$ are always cuspidal. Note also that for any finite place $v$ of $F^+$ splitting $v = ww^c$ in $F$, we have a natural identification $U(F^+_v) = B_{F_w}^{\mathrm{op}, \times}$. By taking a Borel subgroup of $U(F^+ \otimes_{\mathbb{Q}} \mathbb{C}) = \prod_{\tau \in \Phi} \mathrm{GL}_{2}(\mathbb{C})$ as the set of upper triangular matrices, we identify the set of dominant weight with $(\mathbb{Z}_+^{2})^{\Phi}$.

\begin{thm} \label{base changeIII}

1 \ For any cohomological cuspidal automorphic representation $\sigma$ of $U(\mathbb{A}_{F^+})$, there exists a unique conjugate self-dual\footnote{$\pi$ is conjugate self-dual if and only if $\pi^c \cong \pi^{\vee}$.} cohomological isobaric automorphic representation $BC(\sigma) := \pi$ of $\mathrm{GL}_2(\mathbb{A}_F)$ such that $\pi_w = \sigma_v$ for any prime $l$ splitting in $F_0$ not lying below a place in $S(B)$ and any place $v \mid l$ of $F^+$ and any place $w \mid v$ of $F$. 

2 \ For $\pi$ and $\sigma$ as in the statement 1, if $\pi$ is cuspidal, then we obtain the following properties. (Let $\lambda = (\lambda_{\tau, 1}, \lambda_{\tau, 2})_{\tau} \in (\mathbb{Z}_+^{2})^{\Phi}$ be the weight of $\sigma$.)

(a) \ The weight $\mu =((\mu_{\tau, 1}, \mu_{\tau, 2})_{\tau \in \Phi}, (\mu_{\tau, 1}', \mu_{\tau, 2}')_{\tau \in \Phi}) \in (\mathbb{Z}^2_{+})^{\Phi} \times (\mathbb{Z}^2_{+})^{c\Phi}$ of $\pi$ is equal to $$((\lambda_{\tau, 1}, \lambda_{\tau, 2})_{\tau \in \Phi}, (-\lambda_{\tau, 2}, -\lambda_{\tau, 1})_{\tau \in \Phi} ).$$

(b) \ Let $v = w w^c$ be a place of $F^+$ splitting in $F$ such that $w \notin S(B)$. Then we have $\sigma_w = \pi_v$.

(c) \ For any place $v$ of $F^+$ which is inert in $F$, if $\sigma_{v}$ is unramified, then $\pi_v$ is unramified. (Here, we also regard $v$ as the place of $F$.)

(d) \ For any place $v$ of $F^+$ and any $w \mid v$ satisfying $w \in S(B)$, we have $\mathrm{JL}_{F^+_v}(\sigma_v) = \pi_w$, where $\mathrm{JL}_{F^+_v}$ denotes the Jacquet-Langlands transfer from $U(F_{v}^+) = B_{F_w}^{\mathrm{op}, \times}$ to $\mathrm{GL}_2(F_w)$.

(e) \ If $F/F^+$ is unramified at any finite place and $\sigma_v$ is unramified for any place $v$ of $F^+$ which is inert in $F$, the multiplicity $m(\sigma)$ of $\sigma$ is $1$. Moreover, such a $\sigma$ is uniquely determined by $\pi$ up to isomorphism.

3 \ Conversely, for any conjugate self-dual cohomological cuspidal automorphic representation $\pi$ of $\mathrm{GL}_2(\mathbb{A}_{F})$ such that $\pi_w$ is discrete series for any $w \in S(B)$, there exists a cohomological cuspidal automorphic representation $\sigma$ of $U(\mathbb{A}_{F^+})$ such that $\pi = BC(\sigma)$ and for any finite place $v$ of $F^+$ which is inert in $F$, if $\pi_{v}$ is unramified, then $\sigma_v$ is unramified. 

\end{thm}

\begin{proof}
        
Statement 1 follows from \cite[Theorem 1.1]{Shin} as in \cite[Theorem 2.3.3]{10} in the case that $S(B)$ is empty and from \cite[Theorem A.5.2]{CL} in the case that $S(B)$ is non-empty. 

The statements 2 and 3 were proved by \cite[Theorems 5.3 and 5.4]{Lab} in the case that $S(B)$ is empty. (Note that in order to use \cite{Lab}, we need to assume $F^+ \neq \mathbb{Q}$ and in \cite[Theorem 5.3 and 5.4]{Lab}, he assumed a certain technical condition (*), but under our assumption that $\pi$ is cuspidal, we can check that (*) always holds as in \cite[Corollary 8.4]{KS} by using Theorem \ref{kottwitz conjecture} later and \cite[Lemma 2.1.1]{CS2}.) In the case that $S(B)$ is non-empty, we can also prove statements 2 and 3 as in \cite[Theorem VI.2.1 and 2.9]{HT} by using \cite[Theorem A.3.1]{CL} and \cite[Theorem (a) and (b)]{JacLan} except for (e) of 2. (\cite[Theorem VI.2.1 and 2.9]{HT} states the result for a certain unitary group, but the same argument works in our case.) (e) is shown by the same argument as \cite[Theorem 5.4]{Lab} and using \cite[Theorem A.3.1]{CL} instead of \cite[Theorem 5.1]{Lab}. \end{proof}

Let $GU/\mathbb{Q}$ be the unitary similitude group defined by $U/F^+$, i.e., $GU(R) = \{ (\lambda, g) \in R^{\times} \times (B^{\mathrm{op}} \otimes_{\mathbb{Q}} R)^{\times} \mid \psi(g v, gw) = \lambda\psi(v, w) \ \mathrm{for \ any \ } v, w \in V \otimes_{\mathbb{Q}} R \}$ for any $\mathbb{Q}$-algebra $R$ and we put $\nu : GU \rightarrow \mathbb{G}_{m}, \ \ (\lambda, g) \mapsto \lambda$ and then $\mathrm{Ker}\nu = \mathrm{Res}_{F^+/\mathbb{Q}}U$.

For a finite place $v$ of $F_0$ splitting over $\mathbb{Q}$ lying above $l$, we have an identification $GU(\mathbb{Q}_{l}) \cong \mathbb{Q}_l \times \prod_{w \mid v} B_{F_w}^{\mathrm{op}, \times} = F_{0, v^c} \times \prod_{w \mid v} B_{F_w}^{\mathrm{op}, \times}$. Moreover, $\Phi$ induces the identification $GU(\mathbb{C}) = \mathbb{C}^{\times} \times \prod_{\tau \in \Phi} \mathrm{GL}_2(\mathbb{C})$. By taking a Borel subgroup of $G(\mathbb{C})$ as the set of upper triangular matrices of $GU(\mathbb{C})$, we identify the set of dominant weights with $\mathbb{Z} \times (\mathbb{Z}_+^{2})^{\Phi}$.

\begin{thm} \label{base changeII}

1 \ For any cohomological cuspidal automorphic representation $\sigma$ of $GU(\mathbb{A}_{\mathbb{Q}})$, there exists a unique cohomological isobaric automorphic representation $BC(\sigma) := \chi \boxtimes \pi$ of $\mathbb{A}_{F_0}^{\times} \times \mathrm{GL}_2(\mathbb{A}_F)$ satisfying that $\pi^c \cong \pi^{\vee}$, $\omega_{\pi}|_{\mathbb{A}_{F_0}^{\times}} = \chi^c/\chi$ and for any place $v \mid l$ of $F_0$ splitting over $\mathbb{Q}$ not lying below places of $S(B)$, the representation $\sigma_l = \psi_{l} \boxtimes (\boxtimes_{w \mid v} \sigma_{w})$, $(\chi \boxtimes \pi)_v = \chi_v \boxtimes (\boxtimes_{w \mid v} \pi_{w})$ and $(\chi \boxtimes \pi)_{v^c} = \chi_{v^c} \boxtimes (\boxtimes_{w \mid v^c} \pi_{w})$ satisfy $\psi_l = \chi_{v^c}$ and $\sigma_w = \pi_w$ for any $w \mid v$.

2 For $\sigma$, $\chi$ and $\pi$ as in the statement 1, if $\pi$ is cuspidal, then we have the following properties. (Let $\lambda = (\lambda_0, (\lambda_{\tau, 1}, \lambda_{\tau, 2})_{\tau \in \Phi}) \in \mathbb{Z} \times (\mathbb{Z}_+^{2})^{\Phi}$ be the weight of $\sigma$.)
    
(a) \ The weight $\mu = ((\mu, \mu'), (\mu_{\tau, 1}, \mu_{\tau, 2})_{\tau \in \Phi}, (\mu_{\tau, 1}', \mu_{\tau, 2}')_{\tau \in c\Phi}) \in \mathbb{Z}^{\mathrm{id}} \times \mathbb{Z}^c \times (\mathbb{Z}^2_{+})^{\Phi} \times (\mathbb{Z}^2_{+})^{c\Phi}$ of $\chi \boxtimes \pi_{\infty}$ is equal to $((\sum_{\tau} (\lambda_{\tau, 1} + \lambda_{\tau, 2}) + \lambda_0, \lambda_0), (\lambda_{\tau, 1}, \lambda_{\tau, 2})_{\tau \in \Phi}, (-\lambda_{\tau, 2}, -\lambda_{\tau, 1})_{\tau \in \Phi} )$.

(b) \ For any prime $l$ which is inert in $F_0$ and unramified in $F$, if $\sigma_{l}$ is unramified, then $\pi_l$ is unramified.

(c) \ For any place $v \mid l$ of $F_0$ splitting over $\mathbb{Q}$, the representation $\sigma_l = \psi_{l} \boxtimes (\boxtimes_{w \mid v} \sigma_{w})$, $(\chi \boxtimes \pi)_v = \chi_v \boxtimes (\boxtimes_{w \mid v} \pi_{w})$ and $(\chi \boxtimes \pi)_{v^c} = \chi_{v^c} \boxtimes (\boxtimes_{w \mid v^c} \pi_{w})$ satisfy $\psi_l = \chi_{v^c}$, $\sigma_w = \pi_w$ for any $w \mid v$ if $w \notin S(B)$ and $\mathrm{JL}_{F_w}(\sigma_w) = \pi_w$ if $w \in S(B)$.

(d) \ Let $v \mid l$ be a place of $F_0$ not splitting over $\mathbb{Q}$ such that $l$ splits $l = v_1 v_1^c$ in an imaginary quadratic field $F_1$ contained in $F$. Then the representations $\sigma_l = \psi_{l} \boxtimes (\boxtimes_{w \mid v_1} \sigma_{w})$, $\chi_v$ and $\pi_w$ for $w \mid v_1$ satisfy $\chi_v = (\psi_{l} \circ N_{F_{0, v}/\mathbb{Q}_l} )(\prod_{w \mid v_1} \omega_{\pi_w}|_{F_{0,v}^{\times}})$ and $\sigma_w = \pi_{w}$.

3 \ Conversely, for any conjugate self-dual cohomological cuspidal automorphic representation $\pi$ of $\mathrm{GL}_2(\mathbb{A}_{F})$ such that $\pi_w$ is discrete series for any $w \in S(B)$ and any algebraic Hecke character $\chi : \mathbb{A}_{F_0}^{\times}/F_0^{\times} \rightarrow \mathbb{C}^{\times}$ satisfying $\omega_{\pi}|_{\mathbb{A}_{F_0}^{\times}} = \chi^c/\chi$, there exists a cohomological cuspidal automorphic representation $\sigma$ of $GU(\mathbb{A}_{\mathbb{Q}})$ such that $BC(\sigma) = \pi$ and for any prime $l$ which is inert in $F_0$ and unramified in $F$, if $\pi_{l}$ is unramified, then $\sigma_l$ is unramified.

\end{thm}

\begin{rem}

In the above statement 2, the author doesn't know what the multiplicity of $\sigma$ is. (cf. See \cite[Theorem 1.2.4]{CHL} for the proof of the multiplicity one in the odd unitary similitude group case.)

\end{rem}

\begin{proof}

This follows from Theorem \ref{base changeIII} as in \cite[Theorems VI.2.1 and VI.2.9]{HT}. \end{proof}

\begin{thm}\label{Galois representation}

Let $\pi$ be a conjugate self-dual\footnote{i.e., $\pi^{\vee} \cong \pi^{c}$.} cuspidal automorphic representation of $\mathrm{GL}_2(\mathbb{A}_{F})$ of weight $\lambda \in (\mathbb{Z}_+^{2})^{\mathrm{Hom}(F, \mathbb{C})}$ and $\iota : \overline{\mathbb{Q}}_p \Isom \mathbb{C}$ be an isomorphism of fields. Then there exists an irreducible continuous representation $r_{\iota}(\pi) : G_{F} \rightarrow \mathrm{GL}_2(\overline{\mathbb{Q}_p})$ unique up to isomorphism satisfying the following conditions.

1 \ There exists a $G_{F}$-equivariant perfect symmetric pairing $r_{\iota}(\pi) \times r_{\iota}(\pi)^c \rightarrow \varepsilon_p^{-1}$.

2 \ For any $v \mid p$, the representation $r_{\iota}(\pi)|_{G_{F_v}}$ is de Rham of $p$-adic Hodge type $\lambda_v:=(\lambda_{\iota \tau})_{\tau} \in (\mathbb{Z}_+^2)^{\mathrm{Hom}_{\mathbb{Q}_p}(F_v, \overline{\mathbb{Q}}_p)}$.

3 \ For any finite places $v$, we have $\iota \mathrm{WD}(r_{\iota}(\pi)|_{G_{F_v}})^{F-ss} \cong \mathrm{rec}_{F_v}(\pi_v|\mathrm{det}|_v^{-\frac{1}{2}})$ and these representations are pure. (See \cite[around Lemma 1.4]{TY} or \cite[definition 2.21]{matsumoto} for the definition of purity.)

\end{thm}

\begin{proof}

Except for the irreducibility of $r_{\iota}(\pi)$, this was proved by many people including higher dimensional cases. See \cite[Theorem 2.1.1]{CW} and \cite[Theorem 1.1]{Caraiani} for the results and see \cite[Theorem 2.1.1]{CW} for appropriate references. Note that in our case, we can easily prove the irreducibility of $r_{\iota}(\pi)$ because if $r_{\iota}(\pi)$ is a direct sum of characters $\psi_1, \psi_2$, then $\psi_1$ and $\psi_2$ are de Rham. Therefore, we have algebraic Hecke characters $\chi_1$ and $\chi_2$ such that $\psi_i = r_{\iota}(\chi_i)$ for $i = 1, 2$. Thus $\pi_1 = \chi_1 || \ ||_{F}^{\frac{1}{2}} \boxplus \chi_2 || \ ||_{F}^{\frac{1}{2}}$ and this contradicts the cuspidality of $\pi$. (See \cite[Lemma 7.1.1]{10} and \cite[proof of Theorem 5.5.2]{CW} for more general results.) \end{proof}

\subsection{Cohomologies of rank 2 unitary Shimura varieties}

We use the notations of the previous subsection.  In the following, we assume that $\Psi$ is not empty. For any $\tau \in \Phi$, we fix an identification $B \otimes_{F^+, \tau} \mathbb{R} = M_2(\mathbb{C})$ such that the involution $*$ induces $c$ on $M_2(\mathbb{C})$. By the Morita equivalence, the $B \otimes_{F^+, \tau} \mathbb{R}$-module $V \otimes_{F^+, \tau} \mathbb{R}$ corresponds to the $\mathbb{C}$-vector space $\mathbb{C}^{2}$ and $\psi$ induces a $c$-Hermitian non-degenerate bilinear form on $\mathbb{C}^{2}$, for which we write $\psi_{\tau}$. By the assumption on $\psi$, by replacing a basis of $\mathbb{C}^{2}$ if necessary, we obtain $\psi_{\tau}((x_1, x_2), (y_1, y_2)) = \sqrt{-1}(- x_1\overline{y_1} + x_2\overline{y_2})$ if $\tau \in \Psi$ and $\psi_{\tau}((x_1, x_2), (y_1, y_2)) = \sqrt{-1}(- x_1\overline{y_1} - x_2\overline{y_2})$ if $\tau \in \Phi \setminus \Psi$. Let $h : \mathbb{C} \rightarrow \mathrm{End}_{B}(V) \otimes_{\mathbb{Q}} \mathbb{R} = \prod_{\tau \in \Phi} M_2(\mathbb{C}), \ z \mapsto \begin{pmatrix}\begin{pmatrix}
    z & 0 \\
    0 & \overline{z} \end{pmatrix}_{\tau \in \Psi}, \begin{pmatrix}
        \overline{z} & 0 \\
        0 & \overline{z} 
        \end{pmatrix}_{\tau \in \Phi \setminus \Psi}
    \end{pmatrix}$. Then $\mathrm{tr}_{F/\mathbb{Q}} \circ \psi( \ , h(\sqrt{-1}) \ )$ is positive definite. Thus $(B, *, V, \mathrm{tr}_{F/\mathbb{Q}} \circ \psi, h)$ is a PEL datum. (See \cite{Kottwitz} for details of PEL data.) Let $\mu : \mathbb{G}_{m, \mathbb{C}} \rightarrow GU(\mathbb{C})$ be the minuscule corresponding to $h$.
    
    Thus we obtain the Shimura variety $S_K$ for any neat open compact subgroup $K$ of $GU(\mathbb{A}_{\mathbb{Q}}^{\infty})$, which is a proper smooth variety of dimension $d$ over its reflex field $E_0$, which is smaller than the Galois closure $\tilde{F}$ of $F/\mathbb{Q}$. (See \cite[definition 1.4.1.8]{Lan} for the definition of neatness.) We have $S_K(\mathbb{C}) = GU(\mathbb{Q}) \setminus GU(\mathbb{A}_{\mathbb{Q}})/KK_{\infty}$, where $K_{\infty} = \mathbb{R}_{>0} (\prod_{\tau \in \Psi} (U(1) \times U(1)) \times \prod_{\tau \in \Phi \setminus \Psi} U(0, 2))$. (Calculations of the dimension and the reflex field of $S_K$ are easy. The properness follows from the fact that $GU/\mathbb{Q}$ is anisotropic modulo center and \cite[Lemma 3.1.5]{anisotropic}.) We fix a weight $\lambda = (\lambda_0, (\lambda_{\tau, 1}, \lambda_{\tau, 2})_{\tau \in \Phi}) \in \mathbb{Z} \times (\mathbb{Z}_+^2)^{\Phi}$. Then $V_{\lambda}$ defines a local system of $\mathbb{C}$-modules on $S_K$. (See \cite[Proposition 3.3]{Milne}.) We will see the moduli interpretation of $S_K$ and $V_{\lambda}$ in {\S} 3.3. 

\begin{thm}(Matsushima formula)\label{Zuker}

We have the following decomposition as $GU(\mathbb{A}_{\mathbb{Q}}^{\infty})$-modules.

$\varinjlim_{K}H^*(S_{K, \mathbb{C}}, V_{\lambda}) = \oplus_{\sigma} (\sigma_f \otimes H^*(\mathrm{Lie}(GU(\mathbb{R})) \otimes_{\mathbb{R}} \mathbb{C}, K_{\infty}; \sigma_{\infty} \otimes V_{\lambda}))^{m(\sigma)}$, where $\sigma$ runs through cohomological cuspidal automorphic representations of weight $\lambda$ of $GU(\mathbb{A}_{\mathbb{Q}})$ and $m(\sigma)$ denotes the multiplicity of $\sigma$.

\end{thm}

\begin{rem}

This also holds when $\Psi$ is empty if we put $S_{K, \mathbb{C}} := GU(\mathbb{Q}) \setminus GU(\mathbb{A}_{\mathbb{Q}}^{\infty})/K$. This is a finite set, which is called a Shimura set of $GU$ of level $K$. Then $H^0(S_K, V_{\lambda})$ can be regarded as the space of algebraic automorphic forms $\mathcal{A}_{GU}(K, V_{\lambda})$ of weight $\lambda$. (See \cite[proof of Proposition 3.3.2]{CHT} for a more explicit description.)

\end{rem}

\begin{proof}

See \cite[{\S} 2]{Arthur} for example. \end{proof}

Let $p$ be a prime which splits in $F_0$ and doesn't lie below any place in $S(B)$ and $\iota : \overline{\mathbb{Q}_p} \Isom \mathbb{C}$ be an isomorphism of fields. By using $\iota$, we regard $\Psi$ and $\Phi$ as subsets of $\mathrm{Hom}(F, \overline{\mathbb{Q}_p})$. By using $\iota$, we can regard $V_{\lambda}$ as a $\overline{\mathbb{Q}}_p$-local system on $S_K$.

\vspace{0.5 \baselineskip}

Let $\rho : G_{F} \rightarrow \mathrm{GL}_2(\overline{\mathbb{Q}}_p)$ be a continuous representation. This can be regarded as a continuous morphism $\rho : G_{F_0} \rightarrow \prod_{\tau \in \Phi}\mathrm{GL}_2(\overline{\mathbb{Q}}_p) \rtimes G_{F_0}$ satisfying that the composition of this with the second projection is equal to $\mathrm{id}$ by \cite[Proposition 1.7]{PR}. Let $\psi : G_{F_0} \rightarrow \overline{\mathbb{Q}}_p^{\times}$ be a continuous character.

On the other hand, $\mu$ can be regarded as a dominant character of a maximal torus of the dual group $\widehat{GU}_{\overline{\mathbb{Q}}_p}$ of $GU_{\overline{\mathbb{Q}}_p}$. This induces the irreducible algebraic representation $\widehat{GU}(\overline{\mathbb{Q}}_p) := \overline{\mathbb{Q}}_p^{\times} \times \prod_{\tau \in \Phi} \mathrm{GL}_2(\overline{\mathbb{Q}}_p) \rightarrow \mathrm{GL}_{n_{\mu}}(\overline{\mathbb{Q}}_p)$ over $\overline{\mathbb{Q}}_p$ with a highest weight $\mu$ and this is uniquely extended to $(\overline{\mathbb{Q}}_p^{\times} \times \prod_{\tau \in \Phi} \mathrm{GL}_2(\overline{\mathbb{Q}}_p)) \rtimes G_{F_0E_{0}} \rightarrow \mathrm{GL}_{n_{\mu}}(\overline{\mathbb{Q}}_p)$ such that the action of $G_{F_0E_{0}}$ is trivial on the highest weight space. Thus we obtain the representation $r_{\mu} \circ  (\psi \boxtimes \rho)|_{G_{E_0F_0}} : G_{E_{0}F_0} \rightarrow \mathrm{GL}_{n_{\mu}}(\overline{\mathbb{Q}}_p)$. The restriction $r_{\mu} \circ  (\psi \boxtimes \rho)|_{G_{\tilde{F}}}$ to the Galois closure $\tilde{F}$ of $F$ is equal to $\psi|_{G_{\tilde{F}}} \otimes (\otimes_{\tau \in \Psi} (\rho|_{G_{\tilde{F}}})^{\tau})$, where $(\rho|_{G_{\tilde{F}}})^{\tau}$ is the representation $G_{\tilde{F}} \rightarrow \mathrm{GL}_2(\overline{\mathbb{Q}}_p)$ defined by $(\rho|_{G_{\tilde{F}}})^{\tau}(g):= \rho(\tilde{\tau}^{-1} g \tilde{\tau})$ for $g \in G_{\tilde{F}}$ where $\tilde{\tau} \in G_{\mathbb{Q}}$ is a lift $\tau$. Note that the isomorphism class of $(\rho|_{G_{\tilde{F}}})^{\tau}$ is independent of the choice of $\tilde{\tau}$.

\begin{thm} \label{kottwitz conjecture}

Let $\pi$ be a cohomological cuspidal automorphic representation of $\mathrm{GL}_2(\mathbb{A}_{F})$ and $\chi : \mathbb{A}_{F_0}^{\times} \rightarrow \mathbb{C}^{\times}$ be an algebraic Hecke character such that $\omega_{\pi}|_{\mathbb{A}_{F_0}^{\times}} = \chi^c/\chi$. Then we have the following isomorphism of $G_{F_0E_0}$-modules up to semisimplification.

$\oplus_{\sigma^{\infty}}(\varinjlim_{K}H^d_{\et}(S_{K, \overline{F_0E_0}}, V_{\lambda}))[\sigma^{\infty}] \cong (r_{\mu} \circ (r_{\iota}(\chi)^c \boxtimes r_{\iota}(\pi)))^{\vee}(-d)^{\oplus m}$, where $m$ is a certain positive integer, $\sigma^{\infty}$ runs through the irreducible admissible representations of $GU(\mathbb{A}_{\mathbb{Q}}^{\infty})$ such that $\sigma := \sigma^{\infty} \otimes \sigma_{\infty}$ is a cohomological automorphic representation of $GU(\mathbb{A}_{\mathbb{Q}})$ such that $BC(\sigma) = \chi \boxtimes \pi$ for some $\sigma_{\infty}$.

\end{thm}

\begin{proof} See \cite[Theorem 1]{Kot} and \cite[Theorem 1.1, Theorem 1.3 and Corollary 10.6]{SS}. Note that in order to use \cite{SS}, we need to assume that if $S(B)$ is empty, then $F^+ \neq \mathbb{Q}$ and $\Psi \neq \Phi$. \end{proof}

Let $E$ be a finite subextension of $\overline{\mathbb{Q}}_p/\mathbb{Q}_p$ such that $\mathrm{Hom}(F, E) = \mathrm{Hom}(F, \overline{\mathbb{Q}}_p)$, $\mathcal{O}$ be the ring of integers of $E$, $\varpi$ be a uniformizer of $\mathcal{O}$ and $\mathbb{F}$ be the residue field of $\mathcal{O}$. Then $V_{\lambda}$ is defined over $E$, for which we also write $V_{\lambda}$ and has a model $\mathcal{V}_{\lambda}$ over $\mathcal{O}$. Let $S$ be a finite set of primes containing $p$, all primes lying below places in $S(B)$ or all primes ramified in $F$.  

Let $K = \prod_{l}K_l$ be a neat open compact subgroup of $GU(\mathbb{A}_{\mathbb{Q}}^{\infty})$ such that $K_l$ is a hyperspecial maximal compact subgroup of $GU(\mathbb{Q}_l)$ for any prime $l \notin S$. We may assume $K_l = \mathcal{O}_{F_{0, v^c}}^{\times} \times \prod_{w \mid v}\mathrm{GL}_2(\mathcal{O}_{F_w})$ for any $v \mid l$ splitting over $\mathbb{Q}$ not contained in $S$. Let $\mathbb{T}^S$ denote the subalgebra of the Hecke algebra $\mathcal{H}(GU(\mathbb{A}_{\mathbb{Q}}^{\infty}), K)_{\mathcal{O}}$ generated over $\mathcal{O}$ by $\mathcal{H}(GU(\mathbb{Q}_l), K_l)_{\mathcal{O}}$ for any $l \notin S$ splitting in $F_0$. 

Let $v \mid l$ be a place of $F_0$ splitting over $\mathbb{Q}$ not lying above $S$.  For a finite place $w \mid v$ of $F$, we define $P_w(X) := X^2 - T_{w,1}X + q_vT_{w,2} \in \mathcal{H}(\mathrm{GL}_2(F_{w}), \mathrm{GL}_2(\mathcal{O}_{F_w}))_{\mathbb{Z}}[X] \subset \mathbb{T}^S[X]$,  where $$T_{w,1} := [\mathrm{GL}_2(\mathcal{O}_{F_w})\begin{pmatrix}
    \varpi_w & 0 \\
    0 & 1 \end{pmatrix} \mathrm{GL}_2(\mathcal{O}_{F_w})], T_{w,2} := [\mathrm{GL}_2(\mathcal{O}_{F_w})\begin{pmatrix}
        \varpi_w & 0 \\
        0 & \varpi_w \end{pmatrix}\mathrm{GL}_2(\mathcal{O}_{F_w})].$$ 
        
Note that for any unramified representation $\pi_w$ of $\mathrm{GL}_2(F_w)$ over $\mathbb{C}$ with an eigensystem $\varphi_{\pi_w} : \mathcal{H}(\mathrm{GL}_2(F_{w}), \mathrm{GL}_2(\mathcal{O}_{F_w}))_{\mathbb{Z}} \rightarrow \mathbb{C}$, we have $\mathrm{det}(X - \mathrm{rec}_{F_w}(\pi_w|\mathrm{det}|_w^{-\frac{1}{2}})(\mathrm{Frob}_w)) = \varphi_{\pi_w} \circ P_w(X)$. (See \cite[Corollary 3.1.2]{CHT}.) Let $T_v := [\mathcal{O}_{F_{0, v}}^{\times} \varpi_v \mathcal{O}_{F_{0,v}}^{\times}] \in \mathbb{T}^S$ for a finite place $v$ of $F_0$ splitting over $\mathbb{Q}$. 

\begin{dfn} \label{non-Eisenstein} Let $\mathfrak{m}$ be a maximal ideal of $\mathbb{T}^S$ such that $\mathbb{T}^S/\mathfrak{m}$ is a finite field.

1 \ We say that $\mathfrak{m}$ is decomposed generic if there exists a prime $l \notin S$ splitting completely in $F$ such that for any finite place $w \mid l$ of $F$, the eigenvalues $\alpha_w, \beta_w$ of $P_w(X) \mod \mathfrak{m}$ satisfies $\alpha_w/\beta_w \neq l^{\pm 1}$.

2 \ We say that $\mathfrak{m}$ is non-Eisenstein if there exists an irreducible continuous representation $\overline{\rho}_{\mathfrak{m}} : G_{F} \rightarrow \mathrm{GL}_2(\mathbb{T}^S/\mathfrak{m})$ such that for almost all $l \notin S$ splitting in $F_0$ and any $w \mid l$, the restriction $\overline{\rho}_{\mathfrak{m}}|_{G_{F_w}}$ is unramified and $\mathrm{det}(X - \overline{\rho}_{\mathfrak{m}}(\mathrm{Frob}_w)) = P_w(X)$. (Precisely, we need to write $P_w(X) \mod \mathfrak{m}$, but we also omit such notations in the following.)

\end{dfn}

For decomposed generic maximal ideals, we have the following very nice vanishing theorem. 

\begin{thm} \label{torsion vanishing}

Let $\mathfrak{m}$ be a decomposed generic maximal ideal of $\mathbb{T}^S$. Then we have the following.

1 \  $H^i_{\et}(S_{K, \overline{F}}, \mathcal{V}_{\lambda}/\varpi)_{\mathfrak{m}} = 0$ for any $i \neq d$.

2 \ $H^d_{\et}(S_{K, \overline{F}}, \mathcal{V_{\lambda}})_{\mathfrak{m}}$ is a finite free $\mathcal{O}$-module.

3 \ $H^d_{\et}(S_{K, \overline{F}}, \mathcal{V_{\lambda}})_{\mathfrak{m}}/\varpi^n \Isom H^d_{\et}(S_{K, \overline{F}}, \mathcal{V_{\lambda}}/\varpi^n)_{\mathfrak{m}}$.
    
\end{thm}

\begin{proof} See \cite[Theorem 6.3.1]{CS} and \cite[Theorem 1.4]{kos} for the statement 1. The statements 2 and 3 are formal consequences of 1. \end{proof}

We consider $\widehat{H}^i(S_{K^p}, \mathcal{O}) := \varprojlim_{n} \varinjlim_{K_p} H^i_{\et}(S_{K^pK_p, \overline{F}}, \mathcal{O}/{\varpi}^n)$. This has an action of $G_F \times G(\mathbb{Q}_p)$, where $G := GU_{\mathbb{Q}_p}$. More generally, we consider $\widehat{H}^i(S_{K^p}, \mathcal{V}_{\lambda}) := \varprojlim_{n} \varinjlim_{K_p} H^i_{\et}(S_{K^pK_p, \overline{F}}, \mathcal{V_{\lambda}}/{\varpi}^n)$. We put $\widehat{H}^i(S_{K^p}, V_{\lambda}) := \widehat{H}^i(S_{K^p}, \mathcal{V}_{\lambda})[\frac{1}{p}]$.

\begin{prop}\label{admissibility}

1 \ Let $n$ be a positive integer. Then we have $$H^i_{\et}(S_{K^pK_p, \overline{F}}, \mathcal{V}_{\lambda}/\varpi^n) = H^i_{\et}(S_{K^pK_p, \overline{F}}, \mathcal{O}/\varpi^n) \otimes_{\mathcal{O}/\varpi^n} \mathcal{V}_{\lambda}/\varpi^n$$ for any $K_p$ acting trivially on $\mathcal{V}_{\lambda}/\varpi^n$.

2 \ $\widehat{H}^i(S_{K^p}, \mathcal{V}_{\lambda}) = \widehat{H}^i(S_{K^p}, \mathcal{O}) \otimes_{\mathcal{O}} \mathcal{V}_{\lambda}$.

3 \ $\widehat{H}^i(S_{K^p}, \mathcal{V}_{\lambda})$ has bounded $p$-torsions and $\widehat{H}^i(S_{K^p}, V_{\lambda})$ is an admissible Banach representation of $G(\mathbb{Q}_p)$ whose unit ball is the maximal $p$-torsion free quotient $\widehat{H}^i(S_{K^p}, \mathcal{V}_{\lambda})/(p-\mathrm{tor})$ of $\widehat{H}^i(S_{K^p}, \mathcal{V}_{\lambda})$. 

\end{prop}

\begin{proof}

The statement 1 follows from the fact that $\mathcal{V}_{\lambda}/\varpi^n$ is a constant sheaf of finite free $\mathcal{O}/\varpi^n$-modules on $S_{K^pK_p}$ if $K_p$ acts trivially on $\mathcal{V}_{\lambda}/\varpi^n$.

2 \ $\widehat{H}^i(S_{K^p}, \mathcal{V}_{\lambda}) := \varprojlim_{n} \varinjlim_{K_p} H^i_{\et}(S_{K^pK_p}, \mathcal{V_{\lambda}}/{\varpi}^n) = \varprojlim_{n} \varinjlim_{K_p} H^i_{\et}(S_{K^pK_p}, \mathcal{O}/{\varpi}^n) \otimes_{\mathcal{O}} \mathcal{V_{\lambda}}$. See \cite[Theorem 0.1]{emerton} and \cite[Proposition 6.2.6]{emertonII} for the statement 3. \end{proof}

Let us fix an open compact subgroup $K_{p, 0}$ of $GU(\mathbb{Q}_p)$ and for a normal open compact subgroup $K_p \subset K_{p, 0}$, let $$\mathbb{T}^S(K_pK^p, \mathcal{V}_{\lambda}/\varpi^n) := \mathrm{Im}(\mathbb{T}^S \rightarrow \mathrm{End}_{D(\mathcal{O}/\varpi^n[K_{0, p}/K_{p}])}(R\Gamma_{\et}(S_{K_pK^p, \overline{F}}, \mathcal{V}_{\lambda}/\varpi^n)))$$ and $\mathbb{T}^S(K^p, \mathcal{V}_{\lambda}):=\varprojlim_{K_p, n}\mathbb{T}^S(K_pK^p, \mathcal{V}_{\lambda}/\varpi^n).$ Note that $\mathbb{T}^S(K_pK^p, \mathcal{V}_{\lambda}/\varpi^n)$ a priori depends on the choice of $K_{p, 0}$, but a natural map $\mathbb{T}^S(K_pK^p, \mathcal{V}_{\lambda}/\varpi^n) \rightarrow \mathrm{End}_{D(\mathcal{O}/\varpi^n)}(R\Gamma_{\et}(S_{K^pK_p, \overline{F}}, \mathcal{V}_{\lambda}/\varpi^n))$ has a nilpotent kernel annihilated by $(d+1)$ power. (See \cite[Lemma 2.2.4]{10}.) Note also that by 2 of Proposition \ref{admissibility}, we have $\mathbb{T}^S(K^p, \mathcal{V}_{\lambda}) = \mathbb{T}^S(K^p, \mathcal{O})$.

\begin{prop}\label{maximal ideal}

1 \ If $K_p$ is pro-$p$, then the ring morphism $\mathbb{T}^S(K^p, \mathcal{V}_{\lambda}) \twoheadrightarrow \mathbb{T}^S(K^pK_p, \mathcal{V}_{\lambda}/\varpi)$ induces a bijection between the sets of maximal ideals. Therefore, the set of maximal ideals of $\mathbb{T}^S(K^p, \mathcal{V}_{\lambda})$ is a finite set $\{ \mathfrak{m}_1, \cdots, \mathfrak{m}_r \}$. 

2 \ Let $J:=\cap_{i=1}^{r} \mathfrak{m}_i$. Then $\mathbb{T}^S(K^p, \mathcal{V}_{\lambda})$ is $J$-adically complete. Thus we have a finite direct product decomposition $\mathbb{T}^S(K^p, \mathcal{V}_{\lambda}) = \prod_{i=1}^r \mathbb{T}^S(K^p, \mathcal{V}_{\lambda})_{\mathfrak{m}_i}$ such that $\mathbb{T}^S(K^p, \mathcal{V}_{\lambda})_{\mathfrak{m}_i}$ is $\mathfrak{m}_i$-adically complete for any $i$ and a finite direct sum decomposition $\widehat{H}^i(S_{K^p}, \mathcal{V}_{\lambda}) = \oplus_{i = 1}^{r} \widehat{H}^i(S_{K^p}, \mathcal{V}_{\lambda})_{\mathfrak{m}_i}$.

3 \ Let $\mathfrak{m}$ be a decomposed generic maximal ideal of $\mathbb{T}^S(K^p, \mathcal{V}_{\lambda})$. (Precisely, the pullback of $\mathfrak{m}$ to $\mathbb{T}^S$ is a decomposed generic maximal ideal of $\mathbb{T}^S$. In the following, we also use such conventions.) Then $\widehat{H}^i(S_K, \mathcal{V}_{\lambda})_{\mathfrak{m}} = 0$ for any $i \neq d$ and $\widehat{H}^d(S_{K^p}, \mathcal{V}_{\lambda})_{\mathfrak{m}}$ is the $p$-adic completion of the torsion free $\mathcal{O}$-module $\varinjlim_{K_p} H^d_{\et}(S_{K^pK_p}, \mathcal{V}_{\lambda})_{\mathfrak{m}}$.

\end{prop}

\begin{proof}

See \cite[proof of Lemma 2.1.14]{GN} for the statements 1 and 2. The statement 3 follows from Theorem \ref{torsion vanishing}. \end{proof}

\begin{prop}\label{invariant part}

For any decomposed generic maximal ideal $\mathfrak{m}$ of $\mathbb{T}^S$, we have $$H^d_{\et}(S_{K^pK_p', \overline{F}}, \mathcal{V}_{\lambda}/\varpi^n)_{\mathfrak{m}}^{K_p} = H^d_{\et}(S_{K^pK_p, \overline{F}}, \mathcal{V}_{\lambda}/\varpi^n)_{\mathfrak{m}}$$ for any integer $n$ and any open normal subgroup $K'_p$ of $K_p$ and $\widehat{H}^d(S_{K^p}, \mathcal{V}_{\lambda})_{\mathfrak{m}}^{K_p} = H^d_{\et}(S_{K^pK_p, \overline{F}}, \mathcal{V}_{\lambda})_{\mathfrak{m}}$.

\end{prop}

\begin{proof}

The first statement follows from Theorem \ref{torsion vanishing} and the second statement follows from the first statement. \end{proof}

In the rest of this subsection, we fix a decomposed generic non-Eisenstein maximal ideal $\mathfrak{m}$ of $\mathbb{T}^S(K^p, \mathcal{V}_{\lambda})$. We put $$\mathbb{T}^S(K, \mathcal{V}_{\lambda}/\varpi^n)_{\mathfrak{m}} := \mathrm{Im}(\mathbb{T}^S \rightarrow \mathrm{End}_{\mathcal{O}/\varpi^n}(H^d_{\et}(S_{K, \overline{F}}, \mathcal{V}_{\lambda}/\varpi^n)_{\mathfrak{m}}))$$ and $\mathbb{T}^S(K^p, \mathcal{V}_{\lambda})_{\mathfrak{m}}:=\varprojlim_{K_p, n}\mathbb{T}^S(K^pK_p, \mathcal{V}_{\lambda}/\varpi^n)_{\mathfrak{m}}$. This is an abuse of notations because this has a surjection from the localization of $\mathbb{T}^S(K^p, \mathcal{V}_{\lambda})$ at $\mathfrak{m}$ with a nilpotent kernel annihilated by $(d+1)$-power but this may not be an isomorphism. In the following, we will not consider the localization of $\mathbb{T}^S(K^p, \mathcal{V}_{\lambda})$ at $\mathfrak{m}$. Thus there are no risks of confusion. We put $\mathbb{T}^S(K, \mathcal{V}_{\lambda})_{\mathfrak{m}}:=\mathrm{Im}(\mathbb{T}^S \rightarrow \mathrm{End}_{\mathcal{O}}(H^d_{\et}(S_{K, \overline{F}}, \mathcal{V}_{\lambda})_{\mathfrak{m}}))$. By Proposition \ref{maximal ideal}, we have $\mathbb{T}^S(K^p, \mathcal{V}_{\lambda})_{\mathfrak{m}}=\varprojlim_{K_p}\mathbb{T}^S(K^pK_p, \mathcal{V}_{\lambda})_{\mathfrak{m}}$. 

\begin{thm}\label{Galois completed cohomology}
 
1 \ There exists a continuous representation $\rho_{\mathfrak{m}} : G_{F, S} \rightarrow \mathrm{GL}_2(\mathbb{T}^S(K^p, \mathcal{V}_{\lambda})_{\mathfrak{m}})$ unique up to isomorphism such that $\rho_{\mathfrak{m}}^c \cong \rho_{\mathfrak{m}}^{\vee}\varepsilon_p^{-1}$ and $\mathrm{det}(T - \rho_{\mathfrak{m}}(\mathrm{Frob}_w)) = P_w(T)$ for any $l \notin S$ splitting in $F_0$ and any places $w \mid l$.

2 \ There exists a unique continuous character $\chi_{\mathfrak{m}} : G_{F_0, S} \rightarrow \mathbb{T}^S(K^p, \mathcal{V}_{\lambda})_{\mathfrak{m}}^{\times}$ such that $\chi_{\mathfrak{m}}^c/\chi_{\mathfrak{m}} \circ \mathrm{Art}_{F_{0}} = \varepsilon_p \mathrm{det}\rho_{\mathfrak{m}} \circ \mathrm{Art}_{F}|_{\mathbb{A}_{F_0}^{\infty}}$ and $\chi_{\mathfrak{m}}(\mathrm{Frob}_v) = T_v$ for any $l \notin S$ splitting in $F_0$ and any place $v \mid l$ of $F_0$.

\end{thm}

\begin{rem} Let $\sigma$ be a cohomological cuspidal automorphic representation of $GU(\mathbb{A}_{\mathbb{Q}})$ such that $\sigma^{\infty, K} \neq 0$, $\chi \boxtimes \pi := BC(\sigma)$ and $\varphi_{\sigma, \iota} : \mathbb{T}^S(K^p, \mathcal{V}_{\lambda})_{\mathfrak{m}} \rightarrow \overline{\mathbb{Q}}_p$ be the eigensystem corresponding to $\overline{\mathbb{Q}}_p \otimes_{\iota^{-1}, \mathbb{C}} \sigma^{\infty, K}$. Then we have $r_{\iota}(\pi) = \varphi_{\sigma, \iota *} \rho_{\mathfrak{m}}$ and $r_{\iota}(\chi) = \varphi_{\sigma, \iota *} \chi_{\mathfrak{m}}$, where $\varphi_{\sigma, \iota *} \rho_{\mathfrak{m}}$ (resp. $\varphi_{\sigma, \iota *} \chi_{\mathfrak{m}}$) denotes the composition of the following map $G_{F, S} \rightarrow \mathrm{GL}_2(\mathbb{T}^S(K^p, \mathcal{V}_{\lambda})_{\mathfrak{m}}) \xrightarrow{\mathrm{GL}_2(\varphi_{\sigma, \iota})} \mathrm{GL}_2(\overline{\mathbb{Q}}_p)$ (resp. $G_{F, S} \rightarrow \mathbb{T}^S(K^p, \mathcal{V}_{\lambda})_{\mathfrak{m}}^{\times} \xrightarrow{\varphi_{\sigma, \iota}} \overline{\mathbb{Q}}_p^{\times})$. \end{rem}

\begin{proof}

1 \ By using Theorem \ref{base changeII} and Theorem \ref{Galois representation}, we can construct a continuous representation $\rho_{\mathfrak{m}, K_p} : G_{F, S} \rightarrow \mathrm{GL}_2(\mathbb{T}^S(K^pK_p, \mathcal{V}_{\lambda})_{\mathfrak{m}})$ satisfying $\rho_{\mathfrak{m}, K_p}^c \cong \rho_{\mathfrak{m},K_p}^{\vee}\varepsilon_p^{-1}$, $\mathrm{det}(T - \rho_{\mathfrak{m}}(\mathrm{Frob}_w)) = P_w(T)$ for any $l \notin S$ splitting in $F_0$ and any places $w \mid l$ and that for any $K_p' \subset K_p$, the pushforward of $\rho_{\mathfrak{m}, K_p'}$ to $\mathbb{T}^S(K^pK_p, \mathcal{V}_{\lambda})_{\mathfrak{m}}$ is equal to $\rho_{\mathfrak{m}, K_p}$ compatible with isomorphisms $\rho_{\mathfrak{m}, K_p}^c \cong \rho_{\mathfrak{m},K_p}^{\vee}\varepsilon_p^{-1}$. See \cite[proof of Proposition 3.4.4]{CHT} for details. Thus by taking $\varprojlim_{K_p}$, we obtain the result. The proof of 2 is easier. \end{proof}

\begin{prop} \label{residual irreducibility}

We put $s_{\mathfrak{m}} := \chi_{\mathfrak{m}}^c|_{G_{\tilde{F}, S}} \otimes (\otimes_{\tau \in \Psi} (\rho_{\mathfrak{m}}|_{G_{\tilde{F}, S}})^{\tau})$ and assume the following conditions.

(1) \ $\overline{s}_{\mathfrak{m}} := s_{\mathfrak{m}} \mod \mathfrak{m}$ is absolutely irreducible.

(2) \ $\overline{\rho}_{\mathfrak{m}}(G_{F})$ is not solvable. 

 Then the evaluation map $$\mathrm{ev}_{\mathcal{V}_{\lambda}} : s_{\mathfrak{m}}^{\vee}(-d) \otimes_{\mathbb{T}^S(K^p, \mathcal{V}_{\lambda})_{\mathfrak{m}}} \mathrm{Hom}_{\mathbb{T}^S(K^p, \mathcal{V}_{\lambda})_{\mathfrak{m}}[G_{\tilde{F}}]}(s_{\mathfrak{m}}^{\vee}(-d), \widehat{H}^d(S_{K^p}, \mathcal{V}_{\lambda})_{\mathfrak{m}}) \rightarrow \widehat{H}^d(S_{K^p}, \mathcal{V}_{\lambda})_{\mathfrak{m}}$$ is an isomorphism.

\end{prop}

\begin{proof} We may assume that the residue field of $\mathfrak{m}$ is equal to $\mathbb{F}$. We write $H^d_{\et}(S_{K^p, \overline{\tilde{F}}}, \mathcal{V}_{\lambda}/\varpi)$ for $\varinjlim_{K_p}H^d_{\et}(S_{K^pK_p, \overline{\tilde{F}}}, \mathcal{V}_{\lambda}/\varpi)$. 
    
First, we have the injectivity of a map $$\overline{s}_{\mathfrak{m}}^{\vee}(-d) \otimes_{\mathbb{F}} \mathrm{Hom}_{\mathbb{F}[G_{\tilde{F}}]}(\overline{s}_{\mathfrak{m}}^{\vee}(-d), H^d(S_{K^p}, \mathcal{V}_{\lambda}/\varpi)_{\mathfrak{m}}[\mathfrak{m}]) \rightarrow H^d(S_{K^p}, \mathcal{V}_{\lambda}/\varpi)_{\mathfrak{m}}[\mathfrak{m}]$$ by the assumption that $\overline{s}_{\mathfrak{m}}$ is absolutely irreducible. Thus we obtain the injectivity of $$(s_{\mathfrak{m}}^{\vee}(-d) \otimes_{\mathbb{T}^S(K^p, \mathcal{O})_{\mathfrak{m}}} \mathrm{Hom}_{\mathbb{T}^S(K^p, \mathcal{V}_{\lambda})_{\mathfrak{m}}[G_{\tilde{F}}]}(s_{\mathfrak{m}}^{\vee}(-d), \widehat{H}^d(S_{K^p}, \mathcal{V}_{\lambda})_{\mathfrak{m}}))/\varpi[\mathfrak{m}] \rightarrow H^d(S_{K^p}, \mathcal{V}_{\lambda}/\varpi)_{\mathfrak{m}}[\mathfrak{m}].$$ This implies the injectivity of $\mathrm{ev}_{\mathcal{V}_{\lambda}} \otimes_{\mathcal{O}} \mathbb{F}$ because any element of $$(s_{\mathfrak{m}}^{\vee}(-d) \otimes_{\mathbb{T}^S(K^p, \mathcal{V}_{\lambda})_{\mathfrak{m}}} \mathrm{Hom}_{\mathbb{T}^S(K^p, \mathcal{V}_{\lambda})_{\mathfrak{m}}[G_{\tilde{F}}]}(s_{\mathfrak{m}}^{\vee}(-d), \widehat{H}^d(S_{K^p}, \mathcal{V}_{\lambda})_{\mathfrak{m}}))/\varpi$$ is annihilated by a power of $\mathfrak{m}$. Therefore, ev$_{\mathcal{V}_{\lambda}}$ is injective and the cokernel is torsion free over $\mathcal{O}$. Thus, $\mathrm{ev}_{\mathcal{V}_{\lambda}}[\frac{1}{p}]$ has a closed image.
    
    By Theorem \ref{kottwitz conjecture}, if we can prove the semisimplicity of the action of $G_{\tilde{F}}$ on $H^d_{\et}(S_{K^pK_p, \overline{\tilde{F}}}, V_{\lambda})_{\mathfrak{m}}$ for any $K_p$, then we obtain the density of $\mathrm{Im}(\mathrm{ev}_{\mathcal{V}_{\lambda}})[\frac{1}{p}]$ in $\widehat{H}^d(S_{K^p}, V_{\lambda})_{\mathfrak{m}}$, which implies the result. Therefore, it suffices to prove that for any cohomological cuspidal automorphic representation $\sigma$ of $GU(\mathbb{A}_{\mathbb{Q}})$ with weight $\lambda$ such that $\chi \boxtimes \pi := BC(\sigma)$ satisfies $\overline{r_{\iota}(\pi)} = \overline{\rho}_{\mathfrak{m}}$, the representation $\varinjlim_{K} H^d_{\et}(S_{K, \overline{\tilde{F}}}, V_{\lambda})[\sigma^{\infty}]$ is semisimple. This follows from \cite[Theorem 2.20, (2.12) and (2.13)]{NJ}. Note that our assumption on $\overline{\rho}_{\mathfrak{m}}$ implies that $r_{\iota}(\pi)$ is strongly irreducible by \cite[Lemma 7.1.2]{10} and thus the assumption of \cite[Theorem 2.20]{NJ} is satisfied. \end{proof}

Let $\mathfrak{g} := \mathrm{Lie}(G(\mathbb{Q}_p)) \otimes_{\mathbb{Q}_p} E$ and $v \mid p$ be the finite place of $F_0$ induced by $F_0 \hookrightarrow F \hookrightarrow \mathbb{C} \xrightarrow{\iota^{-1}} \overline{\mathbb{Q}}_p$. Then we have $G(\mathbb{Q}_p) = \mathbb{Q}_p^{\times} \times \prod_{w \mid v} \mathrm{GL}_2(F_w)$ and after enlarging $E$, the Lie algebra $\mathfrak{g}$ is identified with $E \times \prod_{\tau \in \Phi} \mathfrak{gl}_2(E)_{\tau}$. (Precisely, we put $\mathfrak{gl}_2(E)_{\tau} := \mathfrak{gl}_2(E)$. We use this notation for expressing the considered index.) For an $\mathcal{O}$-morphism $\varphi : \mathbb{T}^S(K^p, \mathcal{V}_{\lambda})_{\mathfrak{m}} \rightarrow \mathcal{O}$, we put $\rho_{\varphi} := \varphi_*\rho_{\mathfrak{m}}$ and $\chi_{\varphi} := \varphi_*\chi_{\mathfrak{m}}$.

\begin{thm}\label{infinitesimal character}

Let $\varphi : \mathbb{T}^S(K^p, \mathcal{O})_{\mathfrak{m}} \rightarrow \mathcal{O}$ be an $\mathcal{O}$-morphism such that $\rho_{\varphi}|_{G_{F_v}}$ is de Rham of $p$-adic Hodge type $\lambda_w:=(\lambda_{\iota\tau, 1}, \lambda_{\iota\tau, 2}) \in (\mathbb{Z}^2_{+})^{\mathrm{Hom}_{\mathbb{Q}_p}(F_w, E)}$ for any $w \mid v$ and $\chi_{\varphi}|_{G_{F_{v^c}}}$ is de Rham of $p$-adic Hodge type $\lambda_0$ and let $\chi_{\lambda^{\vee}} : Z(U(\mathfrak{g})) \rightarrow E$ be the infinitesimal character of $V_{\lambda}^{\vee}$. Here, we regard $V_{\lambda}$ as the irreducible algebraic representation $\mathrm{GL}_{2, E}$ with dominant weight $\lambda$ and $Z(U(\mathfrak{g}))$ denotes the center of the universal enveloping algebra $U(\mathfrak{g})$ of $\mathfrak{g}$.

Then we have $\widehat{H}^d(S_{K^p}, E)_{\mathfrak{m}}^{\mathrm{la}}[\varphi] \subset \widehat{H}^d(S_{K^p}, E)^{\mathrm{la}, \chi_{\lambda^{\vee}}}_{\mathfrak{m}}$.
    
\end{thm}

\begin{proof} See \cite[{\S} 9.10]{IC}. \end{proof}

\subsection{Hodge-de Rham structure of universal abelian varieties}

Let $L$ be an extension field of $\tilde{F}$ contained in $\mathbb{C}$ such that $B \otimes_{F} L = M_2(L)$. Thus the minuscule $\mu$ is defined over $L$ and we have a decomposition $V \otimes_{\mathbb{Q}} L = V_1 \oplus V_0$, where $\mu (z) = z$ on $V_1$ and $\mu(z) = 1$ on $V_0$. For simplicity, we only recall the moduli interpretation of $S_K$ over $L$. (See \cite[{\S} 5]{Kottwitz} or \cite[{\S} 1.4.2]{Lan} for more details. Note that we have $\mathrm{ker}^1(\mathbb{Q}, GU) = 1$ in our situation by \cite[{\S} 7]{Kottwitz} and we replace the usual Kottwitz determinant condition by the following condition 4 by using \cite[Lemma 1.2.5.13]{Lan}.)

Let $T$ be a connected locally noetherian scheme over $L$ with a geometric point $t$. Then $S_{K, L}(T)$ parametrizes the equivalence classes of tuples $(A,\lambda, i, \overline{\eta})$, where

\vspace{0.5 \baselineskip}

1 \  $A$ is an abelian scheme over $T$.

\vspace{0.5 \baselineskip}

2 \  $\lambda : A \rightarrow A^{\vee}$ is a polarization.

\vspace{0.5 \baselineskip}

3 \ $i : B \rightarrow \mathrm{End}(A) \otimes_{\mathbb{Z}} \mathbb{Q}$ is a ring morphism such that $i(b^*)^{\vee} \circ \lambda = \lambda \circ i(b)$.

\vspace{0.5 \baselineskip}

4 \  Zariski locally on $T$, the Lie algebra $\mathrm{Lie}A$ of $A$ is isomorphic to $V_1 \otimes_{L} \mathcal{O}_T$ as $B \otimes_{\mathbb{Q}} \mathcal{O}_{T}$-modules.

\vspace{0.5 \baselineskip}

5 \ $\overline{\eta}$ is a $K$-level structure, i.e., a $\pi_1(T, t)$-stable $K$-orbit of isomorphisms of $B \otimes_{\mathbb{Q}} \mathbb{A}_{\mathbb{Q}}^{\infty}$-modules $\eta : V \otimes_{\mathbb{Q}} \mathbb{A}_{\mathbb{Q}}^{\infty} \Isom VA_{t}$ which is compatible with the pairings under some identification $g : \mathbb{A}_{\mathbb{Q}}^{\infty} \Isom \mathbb{A}_{\mathbb{Q}}^{\infty}(1)$ of $\mathbb{A}_{\mathbb{Q}}^{\infty}$ (note that $g$ is uniquely determined by $\eta$), where $VA_{t} := (\varprojlim_{n}A_{t}[n](k(t))) \otimes_{\widehat{\mathbb{Z}}} \mathbb{A}_{\mathbb{Q}}$.

\vspace{0.5 \baselineskip}

We say that $(A_1,\lambda_1, i_1, \overline{\eta}_1)$ and $(A_2,\lambda_2, i_2, \overline{\eta}_2)$ are equivalent if there exist an isogeny $\alpha : A_1 \rightarrow A_2$ and $r \in \mathbb{Q}_{>0}$ such that $\alpha \lambda_1 = r\lambda_2 \alpha$, $\alpha \circ i_1(b) = i_2(b) \circ \alpha$ for any $b \in B$ and $\eta_2^{-1} \circ V(\alpha) \circ \eta_1$ is contained in the $K$-orbit of $\mathrm{id}$ and $g_2^{-1} \circ g_1$ is contained in the $\nu(K)$-orbit of $r$, where $\eta_1$ (resp. $\eta_2$) is any representative of $\overline{\eta}_1$ (resp. $\overline{\eta}_2$) and $g_1$ (resp. $g_2$) is an isomorphism $\mathbb{A}_{\mathbb{Q}}^{\infty} \Isom \mathbb{A}_{\mathbb{Q}}^{\infty}(1)$ defined by $\eta_1$ (resp. $\eta_2$) as in the above condition 5.

\vspace{0.5 \baselineskip}

We fix a representative $f : A_{K, L} \rightarrow S_{K, L}$ of the universal isogeny class of abelian schemes over $S_{K, L}$. Actually, we can see that the following construction is independent of this choice. The decomposition $B \otimes_{\mathbb{Q}} L = \prod_{\tau \in \mathrm{Hom}(F, L)} M_{2}(L)$ induces the decompositions $(R^1f_{*} \Omega_{A_{K,L}/S_{K, L}}^{\bullet})^{\vee} = \oplus_{\tau \in \mathrm{Hom}(F, L)} D_{\tau, S_{K, L}}'$ and $f_*\Omega_{A_{K,L}/S_{K,L}}^1 = \oplus_{\tau} \omega_{\tau, S_{K, L}}'$. Moreover, by the Morita equivalence, $D_{\tau, S_{K, L}}'$ and $\omega_{\tau, S_{K, L}}'$ induce a rank 2 filtered vector bundle $D_{\tau, S_{K, L}}$ on $S_{K, L}$ with an integrable connection satisfying Griffiths transversality and a vector bundle $\omega_{\tau, S_{K, L}}$ on $S_{K, L}$ satisfying the following by the above conditions 4 and 5. 

\begin{itemize}
\item For any $\tau \in \Phi \setminus \Psi$, we have $\omega_{\tau, S_{K, L}} = 0$ and $D_{\tau, S_{K, L}} = \mathrm{Fil}^0D_{\tau, S_{K, L}} \supset \mathrm{Fil}^1D_{\tau, S_{K, L}} = 0$.
\item For any $\tau \in \Psi$, the vector bundle $\omega_{\tau, S_{K, L}}$ is a line bundle and we have \begin{equation*} \mathrm{gr}^iD_{\tau, S_{K, L}}  = \begin{cases} 0 \ \mathrm{if} \ i \ge 1. \\ \omega_{\tau, S_{K, L}} \otimes_{\mathcal{O}_{S_{K, L}}} \wedge^2 D_{\tau, S_{K, L}}  \ \mathrm{if} \ i = 0. \\ \omega_{\tau, S_{K, L}}^{-1} \ \mathrm{if} \ i = -1. \\ 0 \ \mathrm{if} \ i \le -2. \end{cases} \end{equation*}
\end{itemize}

\begin{prop}(Kodaira-Spencer isomorphism)\label{KS}

For any $\tau \in \Psi$, the composition of maps $$\omega_{\tau, S_{K, L}} \otimes_{\mathcal{O}_{S_{K, L}}} \wedge^2 D_{\tau, S_{K, L}} \hookrightarrow D_{\tau, S_{K, L}} \xrightarrow{\nabla} D_{\tau, S_{K, L}} \otimes_{\mathcal{O}_{S_{K, L}}} \Omega^1_{S_{K, L}} \twoheadrightarrow \omega_{\tau, S_{K, L}}^{-1} \otimes_{\mathcal{O}_{S_{K, L}}} \Omega_{S_{K, L}}^1$$ is an injection of $\mathcal{O}_{S_{K, L}}$-modules and the induced map \begin{equation} \oplus_{\tau \in \Psi} (\omega_{\tau, S_{K, L}}^2 \otimes \wedge^2 D_{\tau, S_{K, L}}) \rightarrow \Omega_{S_{K, L}}^1 \end{equation} is an isomorphism. (Here, we put $\Omega^1_{S_{K, L}} := \Omega^1_{S_{K, L}/L}$ for simplicity. In the following, we will use this notation.)

\end{prop}

\begin{proof}

We may assume that $L = \mathbb{C}$. We put $S := GU(\mathbb{R})/K_{\infty}$ and thus we have $S_K(\mathbb{C}) = GU(\mathbb{Q}) \setminus S \times GU(\mathbb{A}_{\mathbb{Q}}^{\infty})/K$. Note that the flag variety $\mathrm{Fl}/\mathbb{C}$ corresponding to $(GU_{\mathbb{C}}, \mu)$ is regarded as $\prod_{\tau \in \Psi} \mathbb{P}^1_{\mathbb{C}}$ and classifies 1-dimensional quotients of $V_{\tau} := \mathbb{C}^2$ for all $\tau \in \Psi$. Let $\omega_{\tau, \mathrm{Fl}}$ be the dual of the universal 1-dimensional subsheaf of $V_{\tau} \otimes_{\mathbb{C}} \mathcal{O}_{\mathrm{Fl}}$ for any $\tau \in \Psi$. Thus we have an exact sequence $0 \rightarrow \omega_{\tau, \mathrm{Fl}}^{-1} \rightarrow V_{\tau} \otimes_{\mathbb{C}} \mathcal{O}_{\mathrm{FL}} \rightarrow \omega_{\tau, \mathrm{FL}}  \otimes \wedge^2 V_{\tau} \rightarrow 0$. Note that we have a trivialization of the pullback of $D_{\tau, S_{K, \mathbb{C}}}$ to $S$ and the map from $S$ to $\mathrm{Fl}$ induced by the exact sequence $0 \rightarrow \omega_{\tau, S_{K, \mathbb{C}}}  \otimes_{\mathcal{O}_{S_{K, \mathbb{C}}}} \wedge^2 D_{\tau, S_{K,\mathbb{C}}} \rightarrow D_{\tau, S_{K, \mathbb{C}}} \rightarrow \omega_{\tau, S_{K, \mathbb{C}}}^{-1} \rightarrow 0$ is an open immersion. (See \cite[chapter III, Proposition 1.1]{Milne}.) Thus, it suffice to prove the counterpart of the result on $U := \prod_{\tau \in \Psi} \mathbb{A}^1_{\mathbb{C}} \subset \mathrm{Fl}$. Here, we regard $\mathbb{A}^1_{\mathbb{C}} \Isom \{ [(x_{\tau}, 1)] \} \subset \mathbb{P}^1_{\mathbb{C}}$. Thus the image $e_{1, \tau}$ of $\begin{pmatrix}
    1 \\ 0 
   \end{pmatrix}_{\tau}$ via the map $V_{\tau} \rightarrow \omega_{\tau}(U)$ is a generator of $\omega_{\tau}(U) \otimes \wedge^2 V_{\tau}$. 

We have the following exact sequences. (We fix a numbering $\tau_1, \cdots, \tau_d$ of $\Psi = \{ \tau_1, \cdots, \tau_d \}$.)

\xymatrix{
0 \ar[r] & \omega_{\tau, \mathrm{Fl}}^{-1}(U) \ar[r] & V_{\tau} \otimes_{\mathbb{C}} \mathcal{O}_{\mathrm{FL}}(U) \ar[r] \ar[d]^{\mathrm{id} \otimes \nabla} & \omega_{\tau, \mathrm{FL}}(U)  \otimes \wedge^2 V_{\tau} \ar[r] & 0. \\
0 \ar[r] &  \oplus_{i = 1}^d \omega_{\tau, \mathrm{Fl}}^{-1}(U) dx_{\tau_i} \ar[r] & \oplus_{i=1}^d (V_{\tau} \otimes_{\mathbb{C}} \mathcal{O}_{\mathrm{FL}}(U) dx_{\tau_i}) \ar[r] & \oplus_{i=1}^d (\omega_{\tau, \mathrm{FL}}(U)  \otimes \wedge^2 V_{\tau} dx_{\tau_i}) \ar[r] & 0.
}

Note that the subspace $\omega_{\tau, \mathrm{Fl}}^{-1}(U)$ of $V_{\tau} \otimes_{\mathbb{C}} \mathcal{O}_{\mathrm{FL}}(U)$ is generated by $x_{\tau}\begin{pmatrix}
    1 \\ 0 
   \end{pmatrix}_{\tau} - \begin{pmatrix}
    0 \\ 1 
   \end{pmatrix}_{\tau}$ and for $f(x_{\tau_1}, \cdots, x_{\tau_d}) \in \mathbb{C}[x_{\tau_1}, \cdots, x_{\tau_d}]$ the image of $f(x_{\tau_1}, \cdots, x_{\tau_d}) (x_{\tau}\begin{pmatrix}
    1 \\ 0 
   \end{pmatrix}_{\tau} - \begin{pmatrix}
    0 \\ 1 
   \end{pmatrix}_{\tau}) \in \omega_{\mathrm{Fl}, \tau}^{-1}(U)$ in $\oplus_{i=1}^d (\omega_{\mathrm{FL}, \tau}(U) \otimes \wedge^2 V_{\tau} dx_{\tau_i})$ is equal to $(0, \cdots, 0, f(x_{\tau_1}, \cdots, x_{\tau_d})e_{1, \tau}dx_{\tau}, 0, \cdots, 0)$. Thus we obtain the result. \end{proof}

For $\lambda = (\lambda_0, (\lambda_{\tau, 1}, \lambda_{\tau, 2})) \in \mathbb{Z} \times (\mathbb{Z}_+^2)^{\Phi}$, we have a filtered $\mathcal{O}_{S_{K, L}}$-module with an integrable connection \begin{equation} D_{\lambda, S_{K,L}} := D_{0, S_{K, L}}^{\lambda_0} \otimes (\otimes_{\tau \in \Phi} ((\wedge^2D_{\tau, S_{K, L}})^{\lambda_{\tau, 2}} \otimes \mathrm{Sym}^{\lambda_{\tau, 1} - \lambda_{\tau, 2}}D_{\tau, S_{K, L}})), \end{equation} where $D_{0, S_{K,L}}$ is the filtered $\mathcal{O}_{S_{K, L}}$-module with an integrable connection corresponding to the similitude representation (see \cite[Proposition 3.3]{Milne}). Thus this is a line bundle and we have $\mathrm{Fil}^{-1}D_{0, S_{K,L}} = D_{0, S_{K, L}} \supset \mathrm{Fil}^0D_{0, S_{K,L}} = 0$. \footnote{We use the notation $\mathcal{L}^n := \mathcal{L}^{\otimes n}$ for a line bundle $\mathcal{L}$ on a scheme.}

We regard $\mathbb{Z} \times (\mathbb{Z}_+^2)^{\Phi \setminus \Psi}$ (resp. $(\mathbb{Z}_+^2)^{\Psi}$) as a subset of $\mathbb{Z} \times (\mathbb{Z}_+^2)^{\Phi}$ by $(\lambda_0, (\lambda_{\tau, 1}, \lambda_{\tau, 2})_{\tau \in \Phi \setminus \Psi})$ (resp. $(\lambda_{\tau, 1}, \lambda_{\tau, 2})_{\tau \in \Psi}$)  $\mapsto (\lambda_0, (\lambda_{\tau, 1}, \lambda_{\tau, 2})_{\tau \in \Psi}, (0, 0)_{\tau \in \Psi})$ (resp. $(0, (0,0)_{\tau \in \Phi \setminus \Psi}, (\lambda_{\tau, 1}, \lambda_{\tau, 2})_{\tau \in \Psi})$). Moreover, for $\lambda^{\Psi} = (\lambda_0, (\lambda_{\tau, 1}, \lambda_{\tau, 2})_{\tau \in \Phi \setminus \Psi}) \in (\mathbb{Z}^2_{+})^{\Phi \setminus \Psi}$ and $\lambda_{\Psi} = (\lambda_{\tau, 1}, \lambda_{\tau, 2})_{\tau \in \Psi} \in (\mathbb{Z}^2_{+})^{\Psi}$, we put $\lambda^{\Psi}\lambda_{\Psi} := (\lambda_0, (\lambda_{\tau, 1}, \lambda_{\tau, 2})_{\tau \in \Phi}) \in \mathbb{Z} \times (\mathbb{Z}_{+}^{2})^{\Phi}$. In the following, we fix $\lambda^{\Psi} = (\lambda_0, (\lambda_{\tau, 1}, \lambda_{\tau, 2})_{\tau \notin \Psi}) \in \mathbb{Z} \times (\mathbb{Z}_+^2)^{\Phi \setminus \Psi}$, $\lambda_{\Psi} = (0, -\lambda_{\tau})_{\tau \in \Psi} \in (\mathbb{Z}_+^{2})^{\Psi}$ and $\lambda := \lambda^{\Psi}\lambda_{\Psi}$. 

\vspace{0.5 \baselineskip}

In the following, we fix a numbering $\tau_1, \cdots, \tau_n$ of $\Psi = \{ \tau_{1}, \cdots, \tau_{n} \}$. If we change this numbering, some maps $f$ in the following construction and the constructions {\S} 4.1 and 4.2 later are changed to $-f$.

By \cite[Corollary 4.30]{Lan}, the de Rham complex $(D_{\lambda, S_{K,L}} \otimes_{\mathcal{O}_{S_{K,L}}} \Omega_{S_{K, L}}^{\bullet}, \nabla_{\bullet})$ of weight $\lambda$ is quasi-isomorphic to a complex having the following form.

\begin{align}\label{BGG form} D_{\lambda^{\Psi}, S_{K,L}} \otimes (\otimes_{\tau \in \Psi} (\omega_{\tau, S_{K, L}}^{-\lambda_{\tau}} \otimes (\wedge^2 D_{\tau, S_{K, L}})^{-\lambda_{\tau}})) \rightarrow \nonumber \\ \oplus_{\sigma \in \Psi} D_{\lambda^{\Psi}, S_{K,L}} \otimes \omega_{\sigma, S_{K, L}}^{\lambda_{\sigma}+2} \otimes (\wedge^2 D_{\sigma, S_{K, L}}) \otimes (\otimes_{\tau \neq \sigma} (\omega_{\tau, S_{K, L}}^{-\lambda_{\tau}} \otimes (\wedge^2 D_{\tau, S_{K, L}})^{-\lambda_{\tau}})) \rightarrow \nonumber \\ \oplus_{I \subset \Psi, |I|=2} D_{\lambda^{\Psi}, S_{K, L}} \otimes (\otimes_{\tau \in I} (\omega_{\tau, S_{K, L}}^{\lambda_{\tau}+2} \otimes (\wedge^2 D_{\tau, S_{K, L}}))) \otimes (\otimes_{\tau \notin I} (\mathcal{\omega}_{\tau, S_{K, L}}^{-\lambda_{\tau}} \otimes (\wedge^2D_{\tau, S_{K, L}})^{-\lambda_{\tau}})) \rightarrow \nonumber \\ \cdots \rightarrow \nonumber \\ D_{S_{K,L}, \lambda^{\Psi}} \otimes (\otimes_{\tau \in \Psi} (\mathcal{\omega}_{\tau, S_{K, L}}^{\lambda_{\tau} + 2} \otimes (\wedge^2D_{\tau, S_{K, L}}))).\end{align}

However, we need more explicit construction of $GDR_{\lambda, S_{K, L}}$ in {\S} 4.2. In order to give a construction, first we see that the above quasi-isomorphism can be easily constructed over $\mathbb{C}$ by using the holomorphic Borel embedding and reducing the construction to that on the flag variety $\prod_{\tau \in \Psi} \mathbb{P}^1_{\mathbb{C}}$. 

\vspace{0.5 \baselineskip}

$\textbf{Construction over} \ \mathbb{C}$

In the following, we use the same notations as in the proof of Proposition \ref{KS}. Let $\mathrm{pr}_{\tau} : \mathrm{Fl}  = \prod_{\tau \in \Psi} \mathbb{P}^1_{\mathbb{C}} \rightarrow \mathbb{P}^1_{\mathbb{C}}$ be the projection and $\mathrm{pr}_{\tau}^{-1}(\mathcal{F})$ be the inverse image of a sheaf $\mathcal{F}$ via $\mathrm{pr}_{\tau}$. $(D_{\lambda, S_{K,\mathbb{C}}} \otimes_{\mathcal{O}_{S_{K,\mathbb{C}}}} \Omega_{S_{K, \mathbb{C}}}^{\bullet}, \nabla_{\bullet})$ can be identified with the pullback of the complex $$V_{\lambda^{\Psi}} \otimes_{\mathbb{C}} (\otimes_{\tau \in \Psi} (\mathrm{pr}_{\tau}^{-1}(\mathrm{Sym}^{\lambda_{\tau}}V_{\tau}^{\vee} \otimes_{\mathbb{C}} \mathcal{O}_{\mathbb{P}^1_{\mathbb{C}}}) \rightarrow \mathrm{pr}_{\tau}^{-1}(\mathrm{Sym}^{\lambda_{\tau}}V_{\tau}^{\vee} \otimes_{\mathbb{C}} \Omega_{\mathbb{P}^1_{\mathbb{C}}}^1)))$$ on $\mathrm{Fl}$. Moreover, $\mathrm{Sym}^{\lambda_{\tau}}V_{\tau}^{\vee} \otimes_{\mathbb{C}} \mathcal{O}_{\mathbb{P}^1_{\mathbb{C}}} \rightarrow \mathrm{Sym}^{\lambda_{\tau}}V_{\tau}^{\vee} \otimes_{\mathbb{C}} \Omega_{\mathbb{P}^1_{\mathbb{C}}}^1$ induces an isomorphism $\mathrm{Fil}^1\mathrm{Sym}^{\lambda_{\tau}}V_{\tau}^{\vee} \otimes_{\mathbb{C}} \mathcal{O}_{\mathbb{P}^1_{\mathbb{C}}} \Isom (\mathrm{Sym}^{\lambda_{\tau}}V_{\tau}^{\vee} \otimes_{\mathbb{C}} \mathcal{O}_{\mathbb{P}^1_{\mathbb{C}}}/\mathrm{Fil}^{\lambda_{\tau}}(\mathrm{Sym}^{\lambda_{\tau}}V_{\tau}^{\vee} \otimes_{\mathbb{C}} \mathcal{O}_{\mathbb{P}^1_{\mathbb{C}}})) \otimes_{\mathcal{O}_{\mathbb{P}^1_{\mathbb{C}}}} \Omega_{\mathbb{P}^1_{\mathbb{C}}}^1$ by the same proof as in Proposition \ref{KS} when $\lambda_{\tau} = 1$ case and by induction on $\lambda_{\tau}$ in general. This induces a splitting $\omega_{\tau, \mathbb{P}^1_{\mathbb{C}}}^{\lambda_{\tau}} \hookrightarrow \mathrm{Sym}^{\lambda_{\tau}}V_{\tau}^{\vee} \otimes_{\mathbb{C}} \mathcal{O}_{\mathbb{P}^1_{\mathbb{C}}}$ giving a quasi-isomorphism $(\mathrm{Sym}^{\lambda_{\tau}}V_{\tau}^{\vee} \otimes_{\mathbb{C}} \mathcal{O}_{\mathbb{P}^1_{\mathbb{C}}} \rightarrow \mathrm{Sym}^{\lambda_{\tau}}V_{\tau}^{\vee} \otimes_{\mathbb{C}} \Omega_{\mathbb{P}^1_{\mathbb{C}}}^1) \cong (\omega_{\tau, \mathbb{P}^1_{\mathbb{C}}}^{\lambda_{\tau}} \rightarrow \omega_{\tau, \mathbb{P}^1_{\mathbb{C}}}^{-\lambda_{\tau}+2} \otimes (\wedge^2 V_{\tau})^{-\lambda_{\tau}+1})$ by the same proof as in Proposition \ref{KS} when $\lambda_{\tau} = 1$ case and by induction on $\lambda_{\tau}$ in general. This induces the above description (\ref{BGG form}) of $(D_{\lambda, S_{K,L}} \otimes_{\mathcal{O}_{S_{K,L}}} \Omega_{S_{K, L}}^{\bullet}, \nabla_{\bullet})$ over $\mathbb{C}$.

\vspace{0.5 \baselineskip}

In the following, we give a direct (but subtle) construction of $GDR_{\lambda, S_{K, L}}$ over general $L$. Note that by the result over $\mathbb{C}$, it suffices to construct the maps realizing the above maps over $\mathbb{C}$. 

By Proposition \ref{KS}, the numbering $\tau_1, \cdots, \tau_d$ induces the natural identification $\Omega_{S_{K, L}}^{n} = \oplus_{I \subset \Psi, |I|=n} ( \otimes_{\tau \in I} (\omega_{\tau, S_{K, L}}^2 \otimes (\wedge^2D_{\tau, S_{K, L}})))$.

In the de Rham complex $(D_{\lambda, S_{K,L}} \otimes_{\mathcal{O}_{S_{K,L}}} \Omega_{S_{K, L}}^{\bullet}, \nabla_{\bullet})$ of weight $\lambda$, we have a map $$\nabla_n : D_{\lambda^{\Psi}, S_{K,L}} \otimes (\otimes_{\tau \in \Psi} \mathrm{Sym}^{\lambda_{\tau}}D_{\tau, S_{K, L}}^{\vee}) \otimes_{\mathcal{O}_{S_{K,L}}} \Omega_{S_{K, L}}^{n} \rightarrow D_{\lambda^{\Psi}, S_{K,L}} \otimes (\otimes_{\tau \in \Psi} \mathrm{Sym}^{\lambda_{\tau}}D_{\tau, S_{K, L}}^{\vee}) \otimes_{\mathcal{O}_{S_{K,L}}} \Omega_{S_{K, L}}^{n+1}.$$ Thus for $I, J \subset \Psi$ such that $|I| = n$ and $|J| = n+1$, by using the natural inclusion from the $I$-component $i_{I} : D_{\lambda^{\Psi}, S_{K,L}} \otimes (\otimes_{\tau \in \Psi} \mathrm{Sym}^{\lambda_{\tau}}D_{\tau, S_{K, L}}^{\vee}) \otimes_{\mathcal{O}_{S_{K, L}}} (\otimes_{\tau \in I} (\omega_{\tau, S_{K, L}}^2 \otimes (\wedge^2D_{\tau, S_{K, L}}))) \rightarrow D_{\lambda^{\Psi}, S_{K, L}} \otimes (\otimes_{\tau \in \Psi} \mathrm{Sym}^{\lambda_{\tau}}D_{\tau, S_{K, L}}^{\vee}) \otimes_{\mathcal{O}_{S_{K,L}}} \Omega_{S_{K, L}}^{n}$ and the natural projection to the $J$-component $j_J : D_{\lambda^{\Psi}, S_{K,L}} \otimes (\otimes_{\tau \in \Psi} \mathrm{Sym}^{\lambda_{\tau}}D_{\tau, S_{K, L}}^{\vee}) \otimes_{\mathcal{O}_{S_{K,L}}} \Omega_{S_{K, L}}^{n+1} \rightarrow D_{\lambda^{\Psi}, S_{K,L}} \otimes (\otimes_{\tau \in \Psi} \mathrm{Sym}^{\lambda_{\tau}}D_{\tau, S_{K, L}}^{\vee}) \otimes_{\mathcal{O}_{S_{K,L}}} (\otimes_{\tau \in J} (\omega_{\tau, S_{K, L}}^2 \otimes (\wedge^2D_{\tau, S_{K, L}})))$, we obtain a map $\nabla_{n, I, J} := j_{J} \circ \nabla_n \circ i_{I} : D_{\lambda^{\Psi}, S_{K,L}} \otimes (\otimes_{\tau \in \Psi} \mathrm{Sym}^{\lambda_{\tau}}D_{\tau, S_{K, L}}^{\vee}) \otimes (\otimes_{\tau \in I} (\omega_{\tau, S_{K, L}}^2 \otimes (\wedge^2D_{\tau, S_{K, L}}))) \rightarrow D_{\lambda^{\Psi}, S_{K,L}} \otimes (\otimes_{\tau \in \Psi} \mathrm{Sym}^{\lambda_{\tau}}D_{\tau, S_{K, L}}^{\vee}) \otimes (\otimes_{\tau \in J} (\omega_{\tau, S_{K, L}}^2 \otimes (\wedge^2D_{\tau, S_{K, L}}))).$

Note that we have a subquotient $D_{\lambda^{\Psi}, S_{K, L}} \otimes (\otimes_{\tau \in I} (\omega_{\tau, S_{K, L}}^{\lambda_{\tau}+2} \otimes (\wedge^2 D_{\tau, S_{K, L}}))) \otimes (\otimes_{\tau \notin I} (\mathcal{\omega}_{\tau, S_{K, L}}^{-\lambda_{\tau}} \otimes (\wedge^2D_{\tau, S_{K, L}})^{-\lambda_{\tau}})) \twoheadleftarrow D_{\lambda^{\Psi}, S_{K, L}} \otimes (\otimes_{\tau \in I} (\omega_{\tau, S_{K, L}}^{\lambda_{\tau}+2} \otimes (\wedge^2 D_{\tau, S_{K, L}}))) \otimes (\otimes_{\tau \notin I} \mathrm{Sym}^{\lambda_{\tau}}D_{\tau, S_{K, L}}^{\vee}) \hookrightarrow D_{\lambda^{\Psi}, S_{K,L}} \otimes (\otimes_{\tau \in \Psi} \mathrm{Sym}^{\lambda_{\tau}}D_{\tau, S_{K, L}}^{\vee}) \otimes (\otimes_{\tau \in I} (\omega_{\tau, S_{K, L}}^2 \otimes (\wedge^2D_{\tau, S_{K, L}})))$. In the following, we will construct a section $s_I$ of this quotient map for any $I$ whose base change to $\mathbb{C}$ is equal to the section previously constructed above (Note that this implies that we obtain $\mathrm{Im}(\nabla_{n, I, J} \circ s_{I}) \subset \mathrm{Im}(s_{J})$ and $s_I$'s induce the complex having a form as above (\ref{BGG form}).)

\begin{lem}\label{constdeRham}

Let $I$ and $J$ as above. Then we have the following.

(1) \  $\nabla_{n, I, J}$ is zero unless $J = I \cup \{ \sigma \}$ for some $\sigma \in \Psi \setminus I$.

\vspace{0.5 \baselineskip}

In the following, we also assume that $J = I \cup \{ \sigma \}$ for some $\sigma \in \Psi \setminus I$.

(2) \ $\nabla_{n, I, J}$ sends the subspace $D_{\lambda^{\Psi}, S_{K, L}} \otimes (\otimes_{\tau \in I} (\omega_{\tau, S_{K, L}}^{\lambda_{\tau}+2} \otimes (\wedge^2 D_{\tau, S_{K, L}}))) \otimes (\otimes_{\tau \notin I} \mathrm{Sym}^{\lambda_{\tau}}D_{\tau, S_{K, L}}^{\vee})$ to the subspace $D_{\lambda^{\Psi}, S_{K,L}} \otimes (\otimes_{\tau \in I} (\omega_{\tau, S_{K, L}}^{\lambda_{\tau}+2} \otimes (\wedge^2 D_{\tau, S_{K, L}}))) \otimes (\otimes_{\tau \notin I} \mathrm{Sym}^{\lambda_{\tau}}D_{\tau, S_{K, L}}^{\vee}) \otimes (\omega_{\sigma, S_{K, L}}^2 \otimes (\wedge^2D_{\sigma, S_{K, L}}))$.

In the following, we fix $M \subset \Psi \setminus J$. Actually, we have the following stronger result.

(3) \ $\nabla_{n, I, J}$ sends the subspace $D_{\lambda^{\Psi}, S_{K, L}} \otimes (\otimes_{\tau \in I} (\omega_{\tau, S_{K, L}}^{\lambda_{\tau}+2} \otimes (\wedge^2 D_{\tau, S_{K, L}}))) \otimes (\otimes_{\tau \notin I \cup M} \mathrm{Sym}^{\lambda_{\tau}}D_{\tau, S_{K, L}}^{\vee}) \otimes (\otimes_{\tau \in M} \mathrm{Fil}^1(\mathrm{Sym}^{\lambda_{\tau}}D_{\tau, S_{K, L}}^{\vee}))$ to the subspace $D_{\lambda^{\Psi}, S_{K,L}} \otimes (\otimes_{\tau \in I} (\omega_{\tau, S_{K, L}}^{\lambda_{\tau}+2} \otimes (\wedge^2 D_{\tau, S_{K, L}}))) \otimes (\otimes_{\tau \notin I \cup M} \mathrm{Sym}^{\lambda_{\tau}}D_{\tau, S_{K, L}}^{\vee}) \otimes (\otimes_{\tau \in M} \mathrm{Fil}^1(\mathrm{Sym}^{\lambda_{\tau}}D_{\tau, S_{K, L}}^{\vee})) \otimes (\omega_{\sigma, S_{K, L}}^2 \otimes (\wedge^2D_{\sigma, S_{K, L}}))$.

(4) \ $\nabla_{n, I, J, M}' : D_{\lambda^{\Psi}, S_{K, L}} \otimes (\otimes_{\tau \in I} (\omega_{\tau, S_{K, L}}^{\lambda_{\tau}+2} \otimes (\wedge^2 D_{\tau, S_{K, L}}))) \otimes (\otimes_{\tau \notin I \cup M} \mathrm{Sym}^{\lambda_{\tau}}D_{\tau, S_{K, L}}^{\vee}) \otimes (\otimes_{\tau \in M} \omega_{\tau, S_{K, L}}^{-\lambda_{\tau}} \otimes (\wedge^2 D_{\tau, S_{K, L}})^{-\lambda_{\tau}}) \rightarrow D_{\lambda^{\Psi}, S_{K,L}} \otimes (\otimes_{\tau \in I} (\omega_{\tau, S_{K, L}}^{\lambda_{\tau}+2} \otimes (\wedge^2 D_{\tau, S_{K, L}}))) \otimes (\otimes_{\tau \notin I \cup M} \mathrm{Sym}^{\lambda_{\tau}}D_{\tau, S_{K, L}}^{\vee}) \otimes (\otimes_{\tau \in M} \omega_{\tau, S_{K, L}}^{-\lambda_{\tau}} \otimes (\wedge^2 D_{\tau, S_{K, L}})^{-\lambda_{\tau}}) \otimes (\omega_{\sigma, S_{K, L}}^2 \otimes (\wedge^2D_{\sigma, S_{K, L}}))$ induced from (3) induces an isomorphism $D_{\lambda^{\Psi}, S_{K, L}} \otimes (\otimes_{\tau \in I} (\omega_{\tau, S_{K, L}}^{\lambda_{\tau}+2} \otimes (\wedge^2 D_{\tau, S_{K, L}}))) \otimes \mathrm{Fil}^1(\mathrm{Sym}^{\lambda_{\sigma}}D_{\sigma, S_{K, L}}^{\vee}) \otimes (\otimes_{\tau \notin J \cup M} \mathrm{Sym}^{\lambda_{\tau}}D_{\tau, S_{K, L}}^{\vee}) \otimes (\otimes_{\tau \in M} \omega_{\tau, S_{K, L}}^{-\lambda_{\tau}} \otimes (\wedge^2 D_{\tau, S_{K, L}})^{-\lambda_{\tau}}) \Isom D_{\lambda^{\Psi}, S_{K,L}} \otimes (\otimes_{\tau \in I} (\omega_{\tau, S_{K, L}}^{\lambda_{\tau}+2} \otimes (\wedge^2 D_{\tau, S_{K, L}}))) \otimes (\mathrm{Sym}^{\lambda_{\sigma}}D_{\sigma, S_{K, L}}^{\vee}/ \mathrm{Fil}^{\lambda_{\sigma}}(\mathrm{Sym}^{\lambda_{\sigma}}D_{\sigma, S_{K, L}}^{\vee})) \otimes (\otimes_{\tau \notin J \cup M} \mathrm{Sym}^{\lambda_{\tau}}D_{\tau, S_{K, L}}^{\vee}) \otimes (\otimes_{\tau \in M} \omega_{\tau, S_{K, L}}^{-\lambda_{\tau}} \otimes (\wedge^2 D_{\tau, S_{K, L}})^{-\lambda_{\tau}}) \otimes (\omega_{\sigma, S_{K, L}}^2 \otimes (\wedge^2D_{\sigma, S_{K, L}}))$. More strongly, the map $\nabla_{n, I, J, M}'$ induces the $\mathcal{O}_{S_{K, L}}$-linear isomorphisms of the following graded pieces spaces for any $1 \le i \le \lambda_{\sigma}$. 

$D_{\lambda^{\Psi}, S_{K, L}} \otimes (\otimes_{\tau \in I} (\omega_{\tau, S_{K, L}}^{\lambda_{\tau}+2} \otimes (\wedge^2 D_{\tau, S_{K, L}}))) \otimes \mathrm{gr}^i(\mathrm{Sym}^{\lambda_{\sigma}}D_{\sigma, S_{K, L}}^{\vee}) \otimes (\otimes_{\tau \notin J \cup M} \mathrm{Sym}^{\lambda_{\tau}}D_{\tau, S_{K, L}}^{\vee}) \otimes (\otimes_{\tau \in M} \omega_{\tau, S_{K, L}}^{-\lambda_{\tau}} \otimes (\wedge^2 D_{\tau, S_{K, L}})^{-\lambda_{\tau}}) \Isom D_{\lambda^{\Psi}, S_{K,L}} \otimes (\otimes_{\tau \in I} (\omega_{\tau, S_{K, L}}^{\lambda_{\tau}+2} \otimes (\wedge^2 D_{\tau, S_{K, L}}))) \otimes \mathrm{gr}^{i-1}(\mathrm{Sym}^{\lambda_{\sigma}}D_{\sigma, S_{K, L}}^{\vee}) \otimes (\otimes_{\tau \notin J \cup M} \mathrm{Sym}^{\lambda_{\tau}}D_{\tau, S_{K, L}}^{\vee}) \otimes (\otimes_{\tau \in M} \omega_{\tau, S_{K, L}}^{-\lambda_{\tau}} \otimes (\wedge^2 D_{\tau, S_{K, L}})^{-\lambda_{\tau}}) \otimes (\omega_{\sigma, S_{K, L}}^2 \otimes (\wedge^2D_{\sigma, S_{K, L}}))$.

\end{lem}

\begin{proof}

We may assume $L = \mathbb{C}$ and by using the holomorphic Borel embedding, it suffices to prove the corresponding results on $\mathrm{Fl} = \prod_{\tau \in \Psi} \mathbb{P}^1_{\mathbb{C}}$. (1) is trivial because $\nabla_{n, I, J}$ is simply the derivation along $\sigma$ on $\prod_{\tau \in \Psi} \mathbb{P}^1_{\mathbb{C}}$ and $\otimes_{\tau \in I} (\omega_{\tau, S_{K, L}}^2 \otimes (\wedge^2D_{\tau, S_{K, L}})) \subset \Omega_{S_{K, L}}^n$ corresponds to the subspace generated by $\wedge_{\tau \in I} dx_{\tau}$. Here we use the coordinate $x_{\tau}$ of $\tau$-component $\mathbb{P}^1_{\mathbb{C}}$  of $\mathrm{Fl}$ as in the proof of Proposition \ref{KS}. (2) and (3) follow from the same argument. (4) follows from the same argument as in the proof of Proposition \ref{KS} when $\lambda_{\sigma} = 1$ and by induction on $\lambda_{\sigma}$ in general. \end{proof}

We put $\Psi \setminus I = \{ \tau_{i_1}, \cdots, \tau_{i_{k}} \}$ and $i _1 < \cdots < i_k$. Let $\nabla_{n, I, I \cup \{ \tau_{i_1}\}, \emptyset }' : D_{\lambda^{\Psi}, S_{K, L}} \otimes (\otimes_{\tau \in I} (\omega_{\tau, S_{K, L}}^{\lambda_{\tau}+2} \otimes (\wedge^2 D_{\tau, S_{K, L}}))) \otimes (\otimes_{\tau \notin I} \mathrm{Sym}^{\lambda_{\tau}}D_{\tau, S_{K, L}}^{\vee}) \rightarrow D_{\lambda^{\Psi}, S_{K,L}} \otimes (\otimes_{\tau \in I} (\omega_{\tau, S_{K, L}}^{\lambda_{\tau}+2} \otimes (\wedge^2 D_{\tau, S_{K, L}}))) \otimes (\otimes_{\tau \notin I} \mathrm{Sym}^{\lambda_{\tau}}D_{\tau, S_{K, L}}^{\vee}) \otimes (\omega_{\tau_{i_1}, S_{K, L}}^2 \otimes (\wedge^2D_{\tau_{i_1}, S_{K, L}}))$ induced from (3) of Lemma \ref{constdeRham}. Then by (4) of Lemma \ref{constdeRham}, the quotient map \begin{align}\label{section}D_{\lambda^{\Psi}, S_{K, L}} \otimes (\otimes_{\tau \in I} (\omega_{\tau, S_{K, L}}^{\lambda_{\tau}+2} \otimes (\wedge^2 D_{\tau, S_{K, L}}))) \otimes (\otimes_{\tau \notin I} \mathrm{Sym}^{\lambda_{\tau}}D_{\tau, S_{K, L}}^{\vee}) \twoheadrightarrow \nonumber \\ D_{\lambda^{\Psi}, S_{K, L}} \otimes (\otimes_{\tau \in I} (\omega_{\tau, S_{K, L}}^{\lambda_{\tau}+2} \otimes (\wedge^2 D_{\tau, S_{K, L}}))) \otimes (\otimes_{\tau \notin I \cup \{ \tau_{i_1} \}} \mathrm{Sym}^{\lambda_{\tau}}D_{\tau, S_{K, L}}^{\vee}) \otimes (\omega_{\tau_{i_1}, S_{K, L}}^{-\lambda_{\tau_{i_1}}} \otimes (\wedge^2 D_{\tau_{i_1}, S_{K, L}})^{-\lambda_{\tau_{i_1}}})\end{align} induces the isomorphism $\nabla^{' -1}_{n, I, I \cup \{ \tau_{i_1} \}, \emptyset}(D_{\lambda^{\Psi}, S_{K,L}} \otimes (\otimes_{\tau \in I} (\omega_{\tau, S_{K, L}}^{\lambda_{\tau}+2} \otimes (\wedge^2 D_{\tau, S_{K, L}}))) \otimes (\otimes_{\tau \notin I \cup \{ \tau_{i_1} \}} \mathrm{Sym}^{\lambda_{\tau}}D_{\tau, S_{K, L}}^{\vee}) \otimes \mathrm{Fil}^{\lambda_{\tau_{i_{1}}}}(\mathrm{Sym}^{\lambda_{\tau}}D_{\tau, S_{K, L}}^{\vee}) \otimes (\omega_{\tau_{i_1}, S_{K, L}}^2 \otimes (\wedge^2D_{\tau_{i_1}, S_{K, L}}))) \Isom D_{\lambda^{\Psi}, S_{K, L}} \otimes (\otimes_{\tau \in I} (\omega_{\tau, S_{K, L}}^{\lambda_{\tau}+2} \otimes (\wedge^2 D_{\tau, S_{K, L}}))) \otimes (\otimes_{\tau \notin I \cup \{ \tau_{i_1} \}} \mathrm{Sym}^{\lambda_{\tau}}D_{\tau, S_{K, L}}^{\vee}) \otimes (\omega_{\tau_{i_1}, S_{K, L}}^{-\lambda_{\tau_{i_1}}} \otimes (\wedge^2 D_{\tau_{i_1}, S_{K, L}})^{-\lambda_{\tau_{i_1}}})$. This induces the section $s_{I, 1} : D_{\lambda^{\Psi}, S_{K, L}} \otimes (\otimes_{\tau \in I} (\omega_{\tau, S_{K, L}}^{\lambda_{\tau}+2} \otimes (\wedge^2 D_{\tau, S_{K, L}}))) \otimes (\otimes_{\tau \notin I \cup \{ \tau_{i_1} \}} \mathrm{Sym}^{\lambda_{\tau}}D_{\tau, S_{K, L}}^{\vee}) \otimes (\omega_{\tau_{i_1}, S_{K, L}}^{-\lambda_{\tau_{i_1}}} \otimes (\wedge^2 D_{\tau_{i_1}, S_{K, L}})^{-\lambda_{\tau_{i_1}}}) \hookrightarrow D_{\lambda^{\Psi}, S_{K, L}} \otimes (\otimes_{\tau \in I} (\omega_{\tau, S_{K, L}}^{\lambda_{\tau}+2} \otimes (\wedge^2 D_{\tau, S_{K, L}}))) \otimes (\otimes_{\tau \notin I} \mathrm{Sym}^{\lambda_{\tau}}D_{\tau, S_{K, L}}^{\vee})$ of the above quotient map (\ref{section}). 

Again, (3) of Lemma \ref{constdeRham} induces the map $\nabla_{n, I, I \cup \{ \tau_{i_2} \}, \{ \tau_{i_1} \}}' : D_{\lambda^{\Psi}, S_{K, L}} \otimes (\otimes_{\tau \in I} (\omega_{\tau, S_{K, L}}^{\lambda_{\tau}+2} \otimes (\wedge^2 D_{\tau, S_{K, L}}))) \otimes (\otimes_{\tau \notin I \cup \{ \tau_{i_1} \}} \mathrm{Sym}^{\lambda_{\tau}}D_{\tau, S_{K, L}}^{\vee}) \otimes (\omega_{\tau_{i_1}, S_{K, L}}^{-\lambda_{\tau_{i_1}}} \otimes (\wedge^2 D_{\tau_{i_1}, S_{K, L}})^{-\lambda_{\tau_{i_1}}}) \rightarrow D_{\lambda^{\Psi}, S_{K,L}} \otimes (\otimes_{\tau \in I} (\omega_{\tau, S_{K, L}}^{\lambda_{\tau}+2} \otimes (\wedge^2 D_{\tau, S_{K, L}}))) \otimes (\otimes_{\tau \notin I \cup \{ \tau_{i_1} \}} \mathrm{Sym}^{\lambda_{\tau}}D_{\tau, S_{K, L}}^{\vee}) \otimes (\omega_{\tau_{i_1}, S_{K, L}}^{-\lambda_{\tau_{i_1}}} \otimes (\wedge^2 D_{\tau_{i_1}, S_{K, L}})^{-\lambda_{\tau_{i_1}}}) \otimes (\omega_{\tau_{i_2}, S_{K, L}}^2 \otimes (\wedge^2D_{\tau_{i_2}, S_{K, L}}))$. Again by (4) of Lemma \ref{constdeRham}, we obtain the section $s_{I, 2} : D_{\lambda^{\Psi}, S_{K, L}} \otimes (\otimes_{\tau \in I} (\omega_{\tau, S_{K, L}}^{\lambda_{\tau}+2} \otimes (\wedge^2 D_{\tau, S_{K, L}}))) \otimes (\otimes_{\tau \notin I \cup \{ \tau_{i_1} \}} \mathrm{Sym}^{\lambda_{\tau}}D_{\tau, S_{K, L}}^{\vee}) \otimes (\otimes_{i=1}^2 \omega_{\tau_{i}, S_{K, L}}^{-\lambda_{\tau_{i}}} \otimes (\wedge^2 D_{\tau_{i}, S_{K, L}})^{-\lambda_{\tau_{i}}}) \hookrightarrow D_{\lambda^{\Psi}, S_{K, L}} \otimes (\otimes_{\tau \in I} (\omega_{\tau, S_{K, L}}^{\lambda_{\tau}+2} \otimes (\wedge^2 D_{\tau, S_{K, L}}))) \otimes (\otimes_{\tau \notin I \cup \{ \tau_{i_1} \}} \mathrm{Sym}^{\lambda_{\tau}}D_{\tau, S_{K, L}}^{\vee}) \otimes (\omega_{\tau_{i_1}, S_{K, L}}^{-\lambda_{\tau_{i_1}}} \otimes (\wedge^2 D_{\tau_{i_1}, S_{K, L}})^{-\lambda_{\tau_{i_1}}})$ of the quotient map. Thus we obtain the section $s_{I, 1} \circ s_{I, 2} : D_{\lambda^{\Psi}, S_{K, L}} \otimes (\otimes_{\tau \in I} (\omega_{\tau, S_{K, L}}^{\lambda_{\tau}+2} \otimes (\wedge^2 D_{\tau, S_{K, L}}))) \otimes (\otimes_{\tau \notin I \cup \{ \tau_{i_1} \}} \mathrm{Sym}^{\lambda_{\tau}}D_{\tau, S_{K, L}}^{\vee}) \otimes (\otimes_{i=1}^2 \omega_{\tau_{i}, S_{K, L}}^{-\lambda_{\tau_{i}}} \otimes (\wedge^2 D_{\tau_{i}, S_{K, L}})^{-\lambda_{\tau_{i}}}) \hookrightarrow D_{\lambda^{\Psi}, S_{K, L}} \otimes (\otimes_{\tau \in I} (\omega_{\tau, S_{K, L}}^{\lambda_{\tau}+2} \otimes (\wedge^2 D_{\tau, S_{K, L}}))) \otimes (\otimes_{\tau \notin I} \mathrm{Sym}^{\lambda_{\tau}}D_{\tau, S_{K, L}}^{\vee})$. By considering $\nabla_{n, I, I \cup \{ \tau_i \}, \{ \tau_{i_1}, \cdots, \tau_{i-1} \}}'$ repeatedly, we obtain the section $s_I : D_{\lambda^{\Psi}, S_{K, L}} \otimes (\otimes_{\tau \in I} (\omega_{\tau, S_{K, L}}^{\lambda_{\tau}+2} \otimes (\wedge^2 D_{\tau, S_{K, L}}))) \otimes (\otimes_{\tau \notin I} (\mathcal{\omega}_{\tau, S_{K, L}}^{-\lambda_{\tau}} \otimes (\wedge^2D_{\tau, S_{K, L}})^{-\lambda_{\tau}})) \hookrightarrow D_{\lambda^{\Psi}, S_{K, L}} \otimes (\otimes_{\tau \in I} (\omega_{\tau, S_{K, L}}^{\lambda_{\tau}+2} \otimes (\wedge^2 D_{\tau, S_{K, L}}))) \otimes (\otimes_{\tau \notin I} \mathrm{Sym}^{\lambda_{\tau}}D_{\tau, S_{K, L}}^{\vee})$. 

By construction, the base change of $s_{I}$ to $\mathbb{C}$ coincides with the section constructed previously. Thus the sections $s_I$ satisfy $\mathrm{Im}(\nabla_{n, I, J} \circ s_{I}) \subset \mathrm{Im}(s_{J})$ and $s_I$'s induce the following complex.

\begin{align*} GDR_{\lambda, S_{K, L}} : D_{\lambda^{\Psi}, S_{K,L}} \otimes (\otimes_{\tau \in \Psi} (\omega_{\tau, S_{K, L}}^{-\lambda_{\tau}} \otimes (\wedge^2 D_{\tau, S_{K, L}})^{-\lambda_{\tau}})) \rightarrow \\ \oplus_{\sigma \in \Psi} D_{\lambda^{\Psi}, S_{K,L}} \otimes \omega_{\sigma, S_{K, L}}^{\lambda_{\sigma}+2} \otimes (\wedge^2 D_{\sigma, S_{K, L}}) \otimes (\otimes_{\tau \neq \sigma} (\omega_{\tau, S_{K, L}}^{-\lambda_{\tau}} \otimes (\wedge^2 D_{\tau, S_{K, L}})^{-\lambda_{\tau}})) \rightarrow \\ \oplus_{I \subset \Psi, |I|=2} D_{\lambda^{\Psi}, S_{K, L}} \otimes (\otimes_{\tau \in I} (\omega_{\tau, S_{K, L}}^{\lambda_{\tau}+2} \otimes (\wedge^2 D_{\tau, S_{K, L}}))) \otimes (\otimes_{\tau \notin I} (\mathcal{\omega}_{\tau, S_{K, L}}^{-\lambda_{\tau}} \otimes (\wedge^2D_{\tau, S_{K, L}})^{-\lambda_{\tau}})) \rightarrow \\ \cdots \rightarrow \\ D_{S_{K,L}, \lambda^{\Psi}} \otimes (\otimes_{\tau \in \Psi} (\mathcal{\omega}_{\tau, S_{K, L}}^{\lambda_{\tau} + 2} \otimes (\wedge^2D_{\tau, S_{K, L}}))).\end{align*}

By construction, we have a natural map from $GDR_{\lambda, S_{K, L}}$ to $(D_{\lambda, S_{K,L}} \otimes_{\mathcal{O}_{S_{K,L}}} \Omega_{S_{K, L}}^{\bullet}, \nabla_{\bullet})$ which is a quasi-isomorphism after taking the base change to $\mathbb{C}$. Thus we have the following.

\begin{prop}(Dual BGG resolution) \label{BGG}

The above complex $GDR_{\lambda, S_{K, L}}$ is quasi-isomorphic to the de Rham complex $(D_{\lambda, S_{K,L}} \otimes_{\mathcal{O}_{S_{K,L}}} \Omega_{S_{K, L}}^{\bullet}, \nabla_{\bullet})$ of weight $\lambda$.

\end{prop}

By using $\iota : \overline{\mathbb{Q}_p} \Isom \mathbb{C}$, we regard $L$ as a subfield of $\overline{\mathbb{Q}}_p$ and we assume $L = E$, where $E$ is a finite extension of $\mathbb{Q}_p$ appeared in {\S} 3.2. Let $\aS_{K, L}$ be the adic space over $L$ corresponding to $S_{K, L}$. Then we have vector bundles $D_{\tau, \aS_{K, L}}$ and $\omega_{\tau, \aS_{K, L}}$ on $\aS_{K, L}$ corresponding to $D_{\tau, S_{K, L}}$ and $\omega_{\tau, S_{K, L}}$. We also write $D_{\tau, \aS_{K, L}}$ and $\omega_{\tau, \aS_{K, L}}$ on $\aS_{K, L}$ for the corresponding vector bundles on $(\aS_{K, L})_{\mathrm{pro\et}}$. Let $f : \mathcal{A}_{K, L} \rightarrow \mathcal{S}_{K, L}$ be the adic space over $L$ corresponding to the previous abelian scheme $A_{K, L} \rightarrow S_{K, L}$.

Again, $B \otimes_{\mathbb{Q}} L = \prod_{\tau \in \mathrm{Hom}(F, L)} M_{2}(L)$ induces the decomposition $(R^1f_{\mathrm{pro\acute{e}t}, *}\mathbb{Q}_p \otimes_{\mathbb{Q}_p} L)^{\vee} = \oplus_{\tau \in \mathrm{Hom}(F, L)} \mathbb{L}_{\tau}'$ and by the Morita equivalence, $\mathbb{L}_{\tau}'$ induces a rank 2 smooth $L$-sheaf $\mathbb{L}_{\tau}$ on $(\aS_{K, L})_{\proet}$. Let $\mathbb{L}_{0}$ be the rank 1 smooth $L$-sheaf on $(\aS_{K, L})_{\proet}$ corresponding to the similitude representation (see \cite[Proposition 3.3]{Milne}). Then the previous smooth $L$-sheaf $V_{\lambda'}$ in {\S} 3.2 is equal to $\mathbb{L}_0^{\lambda_0'} \otimes (\otimes_{\tau \in \Phi} (\wedge^2\mathbb{L}_{\tau})^{\lambda_{\tau, 2}'} \otimes \mathrm{Sym}^{\lambda_{\tau, 1}' - \lambda_{\tau, 2'}'}\mathbb{L}_{\tau})$ for $\lambda' = (\lambda_0', (\lambda_{\tau, 1}', \lambda_{\tau, 2}')) \in \mathbb{Z} \times (\mathbb{Z}_+^2)^{\Phi}$. 

Let $\widehat{\mathcal{O}}_{\aS_{K,L}}$, $\mathbb{B}_{\mathrm{dR}}^{+}$, $\mathbb{B}_{\mathrm{dR}}$, $\mathcal{O}\mathbb{B}_{\mathrm{dR}}^{+}$ and $\mathcal{O}\mathbb{B}_{\mathrm{dR}}$ be the sheaves on $(\aS_{K,L})_{\proet}$ defined in \cite[{\S} 6]{pH}.

\begin{prop}\label{Hodge de Rham}

1 \ For any $\tau \in \Phi$, we have $\mathbb{L}_{\tau} \otimes_{L} \mathcal{O}\mathbb{B}_{\mathrm{dR}} \cong D_{\tau, \aS_{K, L}} \otimes_{\mathcal{O}_{\aS_{K, L}}} \mathcal{O}\mathbb{B}_{\mathrm{dR}}$.

2 \ For any $\tau \in \Phi \setminus \Psi$, we have $D_{\tau, \aS_{K, L}} \otimes_{\mathcal{O}_{\aS_{K,L}}} \widehat{\mathcal{O}}_{\aS_{K, L}} \cong V_{\tau} \otimes_{L} \widehat{\mathcal{O}}_{\aS_{K, L}}$.

3 \ For any $\tau \in \Psi$, we have $\mathbb{L}_{\tau} \otimes_{L} \mathbb{B}_{\mathrm{dR}}^{+} \supset (D_{\tau, \aS_{K, L}} \otimes_{\mathcal{O}_{\aS_{K, L}}} \mathcal{O}\mathbb{B}^+_{\mathrm{dR}})^{\nabla = 0}(1) \supset \mathbb{L}_{\tau}(1) \otimes_{L} \mathbb{B}^+_{\mathrm{dR}} = (\mathrm{Fil}^0(D_{\tau, \aS_{K, L}} \otimes_{\mathcal{O}_{\aS_{K, L}}} \mathcal{O}\mathbb{B}^+_{\mathrm{dR}}))^{\nabla = 0}(1)$. Moreover, we have $\mathbb{L}_{\tau} \otimes_{L} \mathbb{B}^+_{\mathrm{dR}}/(D_{\tau, \aS_{K, L}} \otimes_{\mathcal{O}_{\aS_{K, L}}} \mathcal{O}\mathbb{B}^+_{\mathrm{dR}})^{\nabla = 0}(1) = \omega_{\tau, \aS_{K, L}} \otimes \wedge^2 D_{\tau, \aS_{K, L}} \otimes_{\mathcal{O}_{\aS_{K, L}}} \widehat{\mathcal{O}}_{\aS_{K, L}}$ and $(D_{\tau, \aS_{K, L}} \otimes_{\mathcal{O}_{\aS_{K, L}}} \mathcal{O}\mathbb{B}^+_{\mathrm{dR}})^{\nabla = 0}(1)/\mathbb{L}_{\tau}(1) \otimes_{L} \mathbb{B}^+_{\mathrm{dR}} = \omega_{\tau, \aS_{K, L}}^{-1} \otimes \widehat{\mathcal{O}}_{\aS_{K, L}}(1)$.

Therefore, we obtain the following exact sequence.

$0 \rightarrow \omega_{\tau, \aS_{K, L}}^{-1} \otimes \widehat{\mathcal{O}}_{\aS_{K, L}}(1) \rightarrow \mathbb{L}_{\tau} \otimes_{L} \widehat{\mathcal{O}}_{\aS_{K, L}} \rightarrow \omega_{\tau, \mathcal{S}_{K, L}} \otimes \widehat{\mathcal{O}}_{\aS_{K, L}} \otimes_{\mathcal{O}_{S_{K, L}}} \wedge^2 D_{\tau, S_{K, L}} \rightarrow 0$.
        
4 \ For any $\tau \in \Psi$, we have $\wedge^2 \mathbb{L}_{\tau} \otimes_{L} \widehat{\mathcal{O}}_{\aS_{K, L}} \cong \wedge^2 D_{\tau, \aS_{K, L}}(1) \otimes_{\mathcal{O}_{\aS_{K, L}}} \widehat{\mathcal{O}}_{\aS_{K, L}}$.

5 \ $\mathbb{L}_0 \otimes_{L} \widehat{\mathcal{O}}_{\aS_{K, L}} \cong D_{0, \aS_{K,L}}(1) \otimes_{\mathcal{O}_{\aS_{K, L}}} \widehat{\mathcal{O}}_{\aS_{K, L}}$.

\end{prop}

\begin{proof} See \cite[Proposition 7.9 and Theorem 8.8]{pH} and \cite[Proposition 5.2.10.]{RH}. Note that we consider the trivial Hodge structure along $\tau \in \Phi \setminus \Psi$. \end{proof}

\begin{prop}\label{FHT}

For any $\tau \in \Psi$, we have the following canonical commutative diagram of $\widehat{\mathcal{O}}_{\aS_{K, L}}$-modules with exact rows, where $-\mathrm{KS}$ denotes the $-1$-time of the Kodaira-Spencer map in Proposition \ref{KS}, $\widehat{\omega}_{\tau, \aS_{K, L}} := \omega_{\tau, \aS_{K, L}} \otimes_{\mathcal{O}_{\aS_{K, L}}} \widehat{\mathcal{O}}_{\aS_{K, L}}$ and the below sequence is the Faltings extension \cite[Corollary 6.14]{pH}.

        \xymatrix{
        0 \ar[r] & \widehat{\mathcal{O}}_{\aS_{K, L}}(1) \ar[r] \ar[d]^{\mathrm{id}} & \mathbb{L}_{\tau} \otimes_{L} \widehat{\omega}_{\tau, \aS_{K, L}}  \ar[r] \ar[d] & \widehat{\omega}_{\tau, \aS_{K, L}}^{2} \otimes_{\mathcal{O}_{\aS_{K, L}}} \wedge^2 D_{\tau, \aS_{K, L}} \ar[r] \ar[d]^{-\mathrm{KS}} & 0.\\
        0 \ar[r] & \widehat{\mathcal{O}}_{\aS_{K, L}}(1) \ar[r] & \mathrm{gr}^1\mathcal{O}\mathbb{B}^+_{\mathrm{dR}} \ar[r]  & \widehat{\mathcal{O}}_{\aS_{K, L}} \otimes_{\mathcal{O}_{\aS_{K, L}}} \Omega_{\aS_{K, L}}^{1} \ar[r]& 0
        }

    \end{prop}
    
    \begin{proof}
    
    This follows from exactly the same argument as \cite[Theorem 4.2.2]{PanI} by using 3 of Proposition \ref{Hodge de Rham}. In {\S} 5, we will use the construction of this map. Thus we recall the proof. By 3 of Proposition \ref{Hodge de Rham}, we have a map $\mathbb{L}_{\tau} \otimes_{L} \mathcal{O}\mathbb{B}_{\mathrm{dR}}^+ \rightarrow \mathrm{Fil}^0( D_{\tau, \aS_{K, L}} \otimes_{\mathcal{O}_{\aS_{K, L}}} \mathcal{O}\mathbb{B}_{\mathrm{dR}}^+)$. By taking the graded $0$, we obtain the map $\mathbb{L}_{\tau} \otimes_{L} \widehat{\mathcal{O}}_{\aS_{K, L}} \rightarrow \mathrm{gr}^0( D_{\tau, \aS_{K, L}} \otimes_{\mathcal{O}_{\aS_{K, L}}} \mathcal{O}\mathbb{B}_{\mathrm{dR}}^+) = \omega_{\tau, \aS_{K, L}}^{-1} \otimes \mathrm{gr}^1\mathcal{O}\mathbb{B}_{\mathrm{dR}}^{+} \oplus \widehat{\omega}_{\tau, \aS_{K, L}} \otimes_{\mathcal{O}_{\aS_{K, L}}} (\wedge^2 D_{\tau_1, \aS_{K ,L}}) \rightarrow \omega_{\tau, \aS_{K, L}}^{-1} \otimes_{\mathcal{O}_{\aS_{K, L}}} \mathrm{gr}^1\mathcal{O}\mathbb{B}_{\mathrm{dR}}^{+}$. This is (the middle map of this proposition) $\otimes \ \mathrm{id}_{\omega_{\tau, \aS_{K, L}}^{-1}}$. Note that we get the quotient map of the Hodge-Tate filtration by $\mathbb{L}_{\tau} \otimes_{L} \widehat{\mathcal{O}}_{\aS_{K, L}} \rightarrow \omega_{\tau, \aS_{K, L}}^{-1} \otimes \mathrm{gr}^1\mathcal{O}\mathbb{B}_{\mathrm{dR}}^{+} \oplus \widehat{\omega}_{\tau, \aS_{K, L}} \otimes_{\mathcal{O}_{\aS_{K, L}}} (\wedge^2 D_{\tau, \aS_{K, L}}) \rightarrow \widehat{\omega}_{\tau, \aS_{K, L}} \otimes_{\mathcal{O}_{\aS_{K, L}}} (\wedge^2 D_{\tau, \aS_{K, L}})$.
    
By the above construction, the compatibility of the subspaces is clear. The compatibility of the quotient spaces follows from the fact $\mathrm{Im}(\mathbb{L}_{\tau} \otimes_{L} \widehat{\mathcal{O}}_{\aS_{K, L}} \rightarrow \omega_{\tau, \aS_{K, L}}^{-1} \otimes_{\mathcal{O}_{\aS_{K, L}}} \mathrm{gr}^1\mathcal{O}\mathbb{B}_{\mathrm{dR}}^{+} \oplus \omega_{\tau, \aS_{K, L}} \otimes_{\mathcal{O}_{\aS_{K, L}}} (\wedge^2 D_{\tau, \aS_{K, L}}) \otimes_{\mathcal{O}_{\aS_{K, L}}} \widehat{\mathcal{O}}_{\aS_{K, L}}) \subset (\omega_{\tau, \aS_{K, L}}^{-1} \otimes_{\mathcal{O}_{\aS_{K, L}}} \mathrm{gr}^1\mathcal{O}\mathbb{B}_{\mathrm{dR}}^{+} \oplus \omega_{\tau, \aS_{K, L}} \otimes_{\mathcal{O}_{\aS_{K, L}}} (\wedge^2 D_{\tau, \aS_{K, L}}) \otimes_{\mathcal{O}_{\aS_{K, L}}} \widehat{\mathcal{O}}_{\aS_{K, L}})^{\nabla = 0}$. In fact, if $x \in V_{\tau} \otimes_{L} \widehat{\mathcal{O}}_{\aS_{K, L}}$ is mapped to $(y, z) \in (\omega_{\tau, \aS_{K, L}}^{-1} \otimes_{\mathcal{O}_{\aS_{K, L}}} \mathrm{gr}^1\mathcal{O}\mathbb{B}_{\mathrm{dR}}^{+} \oplus \omega_{\tau, \aS_{K, L}} \otimes_{\mathcal{O}_{\aS_{K, L}}} (\wedge^2 D_{\tau, \aS_{K, L}}) \otimes_{\mathcal{O}_{\aS_{K, L}}} \widehat{\mathcal{O}}_{\aS_{K, L}})^{\nabla = 0}$, then we obtain $\nabla(y) = -\nabla(z) = - \mathrm{KS}(z)$. \end{proof}

\subsection{Perfectoid Shimura varieties and geometric Sen theory}

In this subsection, we recall fundamental results on perfectoid Shimura varieties and geometric Sen theory. We put $C:=\widehat{\overline{L}}$ and $\aS_{K} := \aS_{K, C}:=\aS_{K, L} \times_{L} C$. Again $B \otimes_{\mathbb{Q}} L = \prod_{\tau \in \mathrm{Hom}(F, L)} M_{2}(L)$ induces the decomposition $V = \oplus_{\tau \in \Phi} V_{\tau}'$ and by the Morita equivalence, $V_{\tau}'$ induces a $2$-dimensional $L$-vector space $V_{\tau}$. We fix an identification $V_{\tau} = L^{\oplus 2}$.

\begin{thm}\label{sim}

There exists a unique perfectoid space $\aS_{K^p}$ over $C$ up to isomorphism such that $\aS_{K^p} \sim \varprojlim_{K_p}\aS_{K^pK_p}$. More precisely, there exist compatible maps $\pi_{K_p} : \aS_{K^p} \rightarrow \aS_{K^pK_p}$ over $C$ for any $K_p$ satisfying the following conditions.

1 \ $\pi_{K_p}$'s induce the homeomorphism of base spaces $|\aS_{K^p}| \Isom \varprojlim_{K_p} |\aS_{K^pK_p}|$.

2 \ There exist $K_p$ and an affinoid open covering $\{ V_{i, K_p} \}$ of $\aS_{K^pK_p}$ such that for any $i$, the inverse image $V_{i, \infty}$ of $V_{i, K_p}$ in $\aS_{K^p}$ is an affinoid perfectoid and $\mathcal{O}_{\aS_{K^p}}^{+}(V_{i, \infty})$ is the $p$-adic completion of $\varinjlim_{K_p' \subset K_p} \mathcal{O}_{\aS_{K^pK_p'}}^{+}(V_{i, K_p'})$, where $V_{i, K_p'}$ denotes the inverse image of $V_{i, K_p}$ in $\aS_{K^pK_p'}$.

\end{thm}

\begin{proof} See \cite[Theorem 4.1.1]{Torsion}. \end{proof}

We can regard $\aS_{K^p}$ as an object of $(\aS_{K, L})_{\proet}$. Then we have a trivialization $V_{\tau} = \mathbb{L}_{\tau}$ on $\aS_{K^p}$. Since $\aS_{K^p}$ is perfectoid, we have the following relative Hodge-Tate filtration on $\aS_{K^p}$ for any $\tau \in \Psi$ by restricting the exact sequence 3 of Proposition \ref{Hodge de Rham} to $\aS_{K^p}$. ($\omega_{\tau, \aS_{K^p}}$ denotes the pullback of $\omega_{\tau, \aS_{K,L}}$ to $\aS_{K^p}$, $\mathcal{O}_{\aS_{K^p}}^{\mathrm{sm}} := \varinjlim_{K_p} \pi^{-1}_{K_p}\mathcal{O}_{\aS_{K^pK_p}}$ and we use a similar notation such as $D_{\tau, \aS_{K^p}}^{\mathrm{sm}}$ and $\omega_{\tau, \aS_{K^p}}^{\mathrm{sm}}$.)

$0 \rightarrow \omega_{\tau, \aS_{K^p}}^{-1}(1) \rightarrow V_{\tau} \otimes_{L} \mathcal{O}_{\aS_{K^p}} \rightarrow \omega_{\tau, \aS_{K^p}} \otimes_{\mathcal{O}_{\aS_{K^p}}^{\mathrm{sm}}} \wedge^2 D_{\aS_{K^p}, \tau}^{\mathrm{sm}} \rightarrow 0$.

This induces a map $\pi_{\mathrm{HT}} : \aS_{K^p} \rightarrow \Fl := \prod_{\tau \in \Psi} \mathbb{P}^{1}_{C, \tau}$ of adic spaces over $C$. (Precisely, we need to replace $\mathbb{P}^{1}_{C, \tau}$ by the adic space corresponding to $\mathbb{P}^{1}_{C, \tau}$, but we also use this notation in the following.) Note that we have a natural identification $\omega_{\tau, \aS_{K^p}}(-1) \cong \pi_{\mathrm{HT}}^{*} \omega_{\tau, \Fl}$ for any $\tau \in \Psi$, where $\omega_{\tau, \Fl}$ denotes the dual of the universal rank 1 subsheaf of $V_{\tau} \otimes_{L} \mathcal{O}_{\Fl}$.

Note that we use a normalization that $[(x_{\tau}, y_{\tau})] \in \mathbb{P}^1_{C, \tau}$ denotes the subspace $x_{\tau}\begin{pmatrix}
    1 \\ 0 
   \end{pmatrix}_{\tau} - y_{\tau}\begin{pmatrix}
    0 \\ 1 
   \end{pmatrix}_{\tau}$. Note also that in order to obtain the $G(\mathbb{Q}_p)$-equivariance of $\pi_{\mathrm{HT}}$, we should consider an action of $G(\mathbb{Q}_p) = \mathbb{Q}_p^{\times} \times \prod_{w \mid v} \mathrm{GL}_2(F_w)$ on $\Fl = \prod_{\tau \in \Psi} \mathbb{P}^1_{C, \tau}$ as in the following. The first factor $\mathbb{Q}_p^{\times}$ acts trivially and for $w \mid v$ and $\tau \in \Psi$, the action $g = \begin{pmatrix}
    a & b \\
    c & d \end{pmatrix} \in \mathrm{GL}_2(F_w)$ sends $[(x_{\tau}, y_{\tau})] \in \mathbb{P}^1_{C, \tau}$ to $$\begin{cases} [(x_{\tau}, y_{\tau})] \ \mathrm{if} \ \tau \notin \mathrm{Hom}_{\mathbb{Q}_p}(F_w, L) \\ [(\tau(d)x_{\tau} + \tau(b)y_{\tau}, \tau(c)x_{\tau} + \tau(a)y_{\tau})] \ \mathrm{if} \ \tau \in \mathrm{Hom}_{\mathbb{Q}_p}(F_w, L) . \end{cases}$$

The Hodge-Tate period map $\pi_{\mathrm{HT}}$ has the following affineness property. 

\begin{thm} \label{affine}

There exists an open base $\mathcal{B}_0$ of $\Fl$ consisting of affinoid open subsets of $\Fl$ satisfying the following properties.

1 \ For any $U \in \mathcal{B}_0$ and any rational open $V$ of $U$, we have $V \in \mathcal{B}_0$.

2 \ For any $U, V \in \mathcal{B}_0$, we have $U \cap V \in \mathcal{B}_0$.

3 \ For any $U \in \mathcal{B}_0$, the inverse image $V_{\infty} := \pi_{\mathrm{HT}}^{-1}(U)$ is an affinoid perfectoid space and equal to $\pi_{K_p}^{-1}(V_{K_p})$ for some $K_p$ and an affinoid open $V_{K_p}$ of $S_{K^pK_p}$ and $\mathcal{O}_{\aS_{K^p}}^{+}(V_{\infty})$ is the $p$-adic completion of $\varinjlim_{K_p' \subset K_p}\mathcal{O}_{\aS_{K^pK_p'}}^+(V_{K_p'})$, where $V_{K_p'}$ denotes the inverse image of $V_{K_p}$ in $\aS_{K^pK_p'}$.

\end{thm}

\begin{proof} See \cite[Theorem 4.1.1]{Torsion}. \end{proof}

We put $U_{1, \tau} := \{ [(x_{\tau}, 1)] \in \mathbb{P}^1_{C, \tau} \mid ||x_{\tau}|| \le 1  \}$, $U_{2, \tau} := \{ [(1, y_{\tau})] \in \mathbb{P}^1_{C, \tau} \mid ||y_{\tau}|| \le 1  \}$ and $U_1 := \prod_{\tau \in \Psi} U_{1, \tau} \subset \Fl = \prod_{\tau \in \Psi} \mathbb{P}^1_{C, \tau}$. 

\begin{lem}\label{open unit ball}

Let $\mathcal{B}$ be the subset of $\mathcal{B}_0$ consisting of $U \in \mathcal{B}_0$ such that there exists $g \in GU(\mathbb{Q}_p)$ such that $U \subset U_1g$. Then $\mathcal{B}$ is an open base of $\Fl$, stable under taking rational open subsets and intersections.
    
\end{lem}

\begin{proof}

It suffices to prove that $\mathcal{B}$ is an open base of $\Fl$. Note that for $\gamma := \begin{pmatrix}
    1 & 0 \\
    0 & p \end{pmatrix} \in \mathrm{GL}_2(\mathbb{Q}_p)$ and any $n \in \mathbb{Z}_{>0}$, we have $U_{1, \tau}\gamma^{-n} := \{ [(x_{\tau}, 1)] \in \mathbb{P}^1_{C, \tau} \mid ||x_{\tau}|| \le p^{n} \}$ and for any rank 1 point $z_{\tau} = [(x_{\tau}, y_{\tau})] \in \mathbb{P}_{C, \tau}^1$ satisfying $y_{\tau} \neq 0$, there exists $n \in \mathbb{Z}_{\ge 0}$ such that $z_{\tau} \in U_{1, \tau}\gamma^{-n}$. On the other hand, for any rank 1 point $z = (z_{\tau})_{\tau} = [(x_{\tau}, y_{\tau})] \in \Fl = \prod_{\tau \in \Psi} \mathbb{P}^1_{C}$, the elements $b \in p\mathbb{Z}_p$ satisfying that $z_{\tau}\begin{pmatrix}
        1 & 0 \\
        b & 1 \end{pmatrix} = [(x_{\tau}, x_{\tau}b + y_{\tau})]$ has the property $x_{\tau}b + y_{\tau} = 0$ for some $\tau \in \Psi$, form a finite set. Thus for any rank 1 point $z \in \Fl$, there exists $\gamma \in GU(\mathbb{Q}_p)$ such that $z \in U_1 \gamma$. This implies the result because any point $w \in \Fl$ has a unique rank 1 point $z \in \Fl$ which specialize $w$. By the above argument, there exists $g \in G(\mathbb{Q}_p)$ such that $wg \in U_1$. Thus $zg$ is contained in the closure $\overline{U_1}$ of $U_1$, which is contained in $U_1\gamma^{-1}$. \end{proof}
    
    \begin{dfn}

Let $\tilde{\mathcal{B}}$ be the set of affinoid open subsets $U$ of $\Fl$ such that $Ug \in \mathcal{B}$ for some $g \in G(\mathbb{Q}_p)$. Note that $\tilde{\mathcal{B}}$ may not be stable under taking intersections but all elements of $\tilde{\mathcal{B}}$ satisfy the property 3 of Theorem \ref{affine}.

    \end{dfn}

For any quasicompact open $V$ of $\aS_{K^p}$, there exists a sufficiently small open subgroup $K_p = (1 + p^k \mathbb{Z}_p) \times \prod_{w \mid v} (1 + p^k M_2(\mathcal{O}_{F_w})) \subset \mathbb{Q}_p^{\times} \times \prod_{w \mid v} \mathrm{GL}_2(F_w)$ for some $k \in \mathbb{Z}_{>0}$ acting on $\mathcal{O}_{K^p}(V)$. Thus we can consider $\mathcal{O}_{K^p}(V)^{\Psi-\mathrm{la}}$. (See the definition after Theorem \ref{density of locally analytic vectors} for the definition of $\Psi-\mathrm{la}$.) Let $\mathcal{O}_{S_{K^p}}^{\Psi-\mathrm{la}}$ be a sheaf on $\aS_{K^p}$ defined by $V \mapsto \mathcal{O}_{S_{K^p}}(V)^{\Psi-\mathrm{la}}$ for any quasicompact open subset $V$ of $\aS_{K^p}$, which is smaller than $\mathcal{O}_{S_{K^p}}^{\mathrm{la}}$ studied in \cite{PanI}, \cite{PanII} and \cite{Cam}. Actually, $\mathcal{O}_{S_{K^p}}^{\Psi-\mathrm{la}}$ has a very concrete description and has a very nice property as in \cite{PanI} and \cite{PanII} (see Theorem \ref{geometric sen theory} and Proposition \ref{mikami expansionII}) while we don't have such a concrete description for $\mathcal{O}_{S_{K^p}}^{\mathrm{la}}$ in general and studying $\mathcal{O}_{K^p}^{\Psi-\mathrm{la}}$ suffices for our applications. 

We say that an affinoid open $V$ of $\aS_{K^pK_p}$ is small if there exists an $\etale$ map $V \rightarrow \mathbb{T}^d$ which is a composition of rational open immersions and finite $\etale$ maps, where $$\mathbb{T}^d := \mathrm{Spa}(C\langle T_1^{\pm 1}, \cdots, T_d^{\pm 1} \rangle, \mathcal{O}_C\langle T_1^{\pm 1}, \cdots, T_d^{\pm 1} \rangle).$$

Let $U \in \tilde{\mathcal{B}}$, let $V_{K_p}$ be an affinoid open of $\aS_{K^pK_p}$ such that $\pi_{\mathrm{HT}}^{-1}(U) = \pi_{K_p}^{-1}(V_{K_p})$, let $V$ be a rational open of $V_{K_p}$ and let $V_{\infty} = \mathrm{Spa}(B, B^+)$ be the inverse image of $V$ in $\aS_{K^p}$.

\begin{thm}\label{geometric sen theory}

Assume that $V$ is small. Then we have $R^i\Psi\mathfrak{LA}(B) = 0$ for any $i > 0$.

\end{thm}

\begin{proof}

We consider a similar situation as \cite[{\S} 3.6]{PanI} and \cite[{\S} 3]{Cam}. 

By the smallness of $V$, we have an $\mathrm{\acute{e}}$tale map $V \rightarrow \mathbb{T}^d$ which is a composition of rational open immersions and finite $\etale$ maps and consider a usual perfectoid pro-finite $\Gamma := \mathbb{Z}_p^d$-Galois $\etale$ cover $\mathbb{T}_{\infty}^d \rightarrow \mathbb{T}^d$, where $\mathcal{O}_{C} \langle T_1^{\pm \frac{1}{p^{\infty}}}, \cdots, T_d^{\pm \frac{1}{p^{\infty}}} \rangle$ denotes the $p$-adic completion of $\cup_{k >0} \mathcal{O}_{C} \langle T_1^{\pm \frac{1}{p^{k}}}, \cdots, T_d^{\pm \frac{1}{p^{k}}} \rangle$ and $$\mathbb{T}_{\infty}^d := \mathrm{Spa}(\mathcal{O}_{C}\langle T_1^{\pm \frac{1}{p^{\infty}}}, \cdots, T_d^{\pm \frac{1}{p^{\infty}}} \rangle[\frac{1}{p}], \mathcal{O}_{C} \langle T_1^{\pm \frac{1}{p^{\infty}}}, \cdots, T_d^{\pm \frac{1}{p^{\infty}}} \rangle).$$ We put $V_{\infty} \times_{\mathbb{T}^d} \mathbb{T}^d_{\infty} =: \mathrm{Spa}(B_{\infty},B^+_{\infty})$. Here, we regard this as an object of $(\mathbb{T}^d)_{\proet}$. This is a perfectoid pro-finite $K_p \times \Gamma$-Galois $\etale$ cover of $V$. Note that we have $H^i(K_p, B \widehat{\otimes}_{L}C^{\mathrm{an}}(K_p, L)^{\oplus_{\tau \notin \Psi} \mathfrak{gl}_2(L)_{\tau}}) = H^i(K_p \times \Gamma, B_{\infty} \widehat{\otimes}_{L}C^{\mathrm{an}}(K_p, L)^{\oplus_{\tau \notin \Psi} \mathfrak{gl}_2(L)_{\tau}})$ because we have $H^i(\Gamma, B_{\infty} \widehat{\otimes}_{L}C^{\mathrm{an}}(K_p, L)^{\oplus_{\tau \notin \Psi} \mathfrak{gl}_2(L)_{\tau}}) = 0$ for any $i > 0$ by the almost purity theorem \cite[Proposition 7.9 (iii)]{Per}. On the other hand, \cite[Corollary 2.5.1 and Proposition 3.2.3]{Cam} and our assumption of the smallness imply $$H^i(K_p \times \Gamma, B_{\infty} \widehat{\otimes}_{L}C^{\mathrm{an}}(K_p, L)^{\oplus_{\tau \notin \Psi} \mathfrak{gl}_2(L)_{\tau}}) = H^i(\mathrm{Lie}\Gamma, B_{\infty}^{K_p-\Psi-\mathrm{an}, \Gamma-\mathrm{la}})^{\Gamma}.$$ Consequently, we obtain $H^i(K_p, B \widehat{\otimes}_{L}C^{\mathrm{an}}(K_p, L)^{\prod_{\tau \notin \Psi} \mathfrak{gl}_2(L)_{\tau}}) = H^i(\mathrm{Lie}\Gamma, B_{\infty}^{K_p-\Psi-\mathrm{an}, \Gamma-\mathrm{la}})^{\Gamma}$ and thus $R^i\Psi\mathfrak{LA}(B) = H^i(\mathrm{Lie}\Gamma, B_{\infty}^{\Psi-\mathrm{la}, \Gamma-\mathrm{la}})^{\Gamma}$.

By Proposition \ref{FHT}, we have the following commutative diagram.

\xymatrix{
    0 \ar[r] & \oplus_{\tau \in \Psi} \mathcal{O}_{\aS_{K^p}}(1) \ar[r] \ar[d]^{\sum_{\tau} \mathrm{id}} & \oplus_{\tau \in \Psi} (V_{\tau} \otimes_{L} \omega_{\tau, \aS_{K^p}})  \ar[r] \ar[d] & \oplus_{\tau \in \Psi} (\omega_{\tau, \aS_{K^p}}^{2} \otimes (\wedge^2 D_{\tau, \aS_{K^p}}^{\mathrm{sm}})^{-1}) \ar[r] \ar[d]_{\wr}^{-\mathrm{KS}} & 0\\
    0 \ar[r] & \mathcal{O}_{\aS_{K^p}}(1) \ar[r] & \mathrm{gr}^1\mathcal{O}\mathbb{B}^+_{\mathrm{dR}} \ar[r]  & \mathcal{O}_{\aS_{K^p}} \otimes_{\mathcal{O}_{\aS_{K^p}}^{\mathrm{sm}}} \Omega_{\aS_{K^p}}^{1, \mathrm{sm}} \ar[r]& 0.
    }

Since $$\oplus_{\tau \in \Psi} (V_{\tau} \otimes_{L} \omega_{\tau, \aS_{K^p}}^{\Psi-\mathrm{la}}(V_{\infty}) \otimes (\wedge^2 D_{\tau, \aS_{K^p}}^{\mathrm{sm}})^{-1})(V_{\infty}) = (\oplus_{\tau \in \Psi} (V_{\tau} \otimes_{L} \omega_{\tau, \aS_{K^p}}(V_{\infty}) \otimes (\wedge^2 D_{\tau, \aS_{K^p}}^{\mathrm{sm}})^{-1})(V_{\infty}))^{\Psi-\mathrm{la}}$$ and the map $$\oplus_{\tau \in \Psi} (V_{\tau} \otimes_{L} \omega_{\tau, \aS_{K^p}}^{\Psi-\mathrm{la}}(V_{\infty}) \otimes (\wedge^2 D_{\tau, \aS_{K^p}}^{\mathrm{sm}})^{-1})(V_{\infty}) \rightarrow \oplus_{\tau \in \Psi} (\omega_{\tau, \aS_{K^p}}^{2, \Psi-\mathrm{la}} \otimes (\wedge^2 D_{\tau, \aS_{K^p}}^{\mathrm{sm}})^{-1})(V_{\infty})$$ is surjective, the map $\mathrm{gr}^1\mathcal{O}\mathbb{B}^{+, \Psi-\mathrm{la}}_{\mathrm{dR}}(V_{\infty}) \rightarrow \mathcal{O}_{\mathcal{S}_{K^p}}^{\Psi-\mathrm{la}} \otimes_{\mathcal{O}_{\aS_{K^p}}^{\mathrm{sm}}} \Omega_{\aS_{K^p}}^{1, \mathrm{sm}}(V_{\infty})$ is surjective.

By \cite[Proposition 6.10]{pH}, the Faltings extension on $V_{\infty}$ can be constructed by the following way. Let $W := \mathbb{Q}_p^{\oplus (d+1)}$ and $e_0, \cdots, e_{d}$ be the standard basis. We fix a basis $\gamma_1, \cdots \gamma_d$ of $\Gamma$ and we define the action of $\Gamma$ on $W$ by $\gamma_{i} e_j = e_j + \delta_{i, j}e_0$, where $\delta_{i, j} := 0$ if $i \neq j$ and $1$ if $i = j$. Then we have a non-split extension of $\Gamma$-representation $0 \rightarrow \mathbb{Q}_pe_0 \rightarrow W \rightarrow W' \rightarrow 0$, where $\Gamma$ acts on $W'$ trivially. By taking $\otimes_{\mathbb{Q}_p} B_{\infty}$ and $\Gamma$-invariant part, we obtain an exact sequence $0 \rightarrow Be_0 \rightarrow (W \otimes_{\mathbb{Q}_p} B_{\infty})^{\Gamma} \rightarrow B \otimes_{\mathbb{Q}_p} W' \rightarrow 0$, which is exact by the almost purity Theorem \cite[Proposition 7.9 (iii)]{Per} and this can be regarded as the Faltings extension on $V_{\infty}$. 

Therefore, by the surjectivity of $\mathrm{gr}^1\mathcal{O}\mathbb{B}^{+, \Psi-\mathrm{la}}_{\mathrm{dR}}(V_{\infty}) \rightarrow \mathcal{O}_{\mathcal{S}_{K^p}}^{\Psi-\mathrm{la}} \otimes_{\mathcal{O}_{\aS_{K^p}}^{\mathrm{sm}}} \Omega_{\aS_{K^p}}^{1, \mathrm{sm}}(V_{\infty})$, after shrinking $K_p$ if necessary, we obtain $x_1, \cdots, x_d \in B^{K_p-\Psi-\mathrm{an}, \Gamma-\mathrm{an}}_{\infty}$ such that $e_j + x_je_0 \in (W \otimes_{\mathbb{Q}_p} B_{\infty})^{\Gamma}$ for any $j = 1, \cdots, d$. Thus we obtain $\gamma_{j} x_i = x_i - \delta_{i,j}$ and thus $(\mathrm{log}\gamma_{j}) x_i = -\delta_{i,j}$, where $\delta_{i, j} = 1$ if $i = j$ and $0$ if $i \neq j$.

By Lemma \ref{formal expansion}, for any $n > 0$, there exist $K_p' \subset K_p$, $m \ge n$ and $x_{i, n} \in B_{\infty}^{K_p', p^n\Gamma}$ such that $|| x_{i} - x_{i, n} || = || x_{i} - x_{i, n} ||_{K_p' \times p^m\Gamma} \le p^{-n-1}$ and $B_{\infty}^{K_p-\Psi-\mathrm{an}, p^n\Gamma-\mathrm{an}} \subset \{ \sum_{k \in \mathbb{Z}_{\ge 0}^d} a_k \prod_{i=1}^d (x_i -x_{i,n})^{k_i} \mid a_k \in B_{\infty}^{K_p'-\Psi-\mathrm{an}, p^m\Gamma} \ s.t. \ \mathrm{sup}_{k}||a_k||_{K_p' \times p^m\Gamma}p^{-n(k_{1} + \cdots + k_d)} < \infty \} \subset B_{\infty}^{K_p'-\Psi-\mathrm{an}, p^m\Gamma-\mathrm{an}}$. (Precisely, we need to take $K_p'$ having a form $(1 + p^k \mathbb{Z}_p) \times \prod_{w \mid v} (1 + p^k M_2(\mathcal{O}_{F_w}))$ for some $k$. In the following, we only consider open compact subgroups $K_p'$ of $G(\mathbb{Q}_p)$ having forms like this when we take $K_p'$-$\Psi$ analytic vectors.) This implies the surjectivity of $\mathrm{log}\gamma_i : B_{\infty}^{\Psi-\mathrm{la}, \Gamma-\mathrm{la}, \mathrm{log}\gamma_1 = \cdots = \mathrm{log}\gamma_{i-1} = 0} \rightarrow B_{\infty}^{\Psi-\mathrm{la}, \Gamma-\mathrm{la}, \mathrm{log}\gamma_1 = \cdots = \mathrm{log}\gamma_{i-1} = 0}$ for any $i$. Therefore, we obtain the vanishing of $H^i(\mathrm{Lie}\Gamma, B_{\infty}^{\Psi-\mathrm{la}, \Gamma-\mathrm{la}})$ for any $i > 0$ by using the Hochschild-Serre spectral sequence. (Or we can prove this vanishing by using more direct computation as in the proof of Poincare lemma.) \end{proof}

\begin{rem}\label{perfectoidness}

1 \ Let $$S_0 := \{ w \mid w \ \mathrm{divides} \ v \ \mathrm{and} \ \mathrm{the} \ w \ \mathrm{component} \ \mu_w \ \mathrm{of \ the \ minuscule} \  \mu \ \mathrm{is \ trivial}. \}.$$ Then we can prove that $\aS_{K^pK_{S_{0}}K_0} := \varprojlim_{K^{S_0}} \aS_{K^pK_{S_0}K^{S_0}K_0}$ is perfectoid or more precisely there exists a perfectoid space $\aS_{K^pK_{S_{0}}K_0}$ such that $\aS_{K^pK_{S_0}K_0} \sim \varprojlim_{K^{S_0}} \aS_{K^pK_{S_0}K^{S_0}K_0}$ for any open compact subgroups $K_{S_0}K_0$ of $\prod_{w \in S_0} \mathrm{GL}_2(F_w) \times \mathbb{Q}_p^{\times}$. Thus if $\Psi = \sqcup_{w \notin S_0} \mathrm{Hom}_{\mathbb{Q}_p}(F_w, \overline{\mathbb{Q}}_p)$, we can prove the above theorem as a corollary of \cite[Theorem 3.4.5]{Cam}. This follows from the following commutative diagram. (Precisely, we can only consider the following diagram locally on $\aS_{K^pK_p^0}$. We put $K_{p}^{o} := \prod_{w \mid v} \mathrm{GL}_2(\mathcal{O}_{F_w}) \times \mathbb{Z}_p^{\times}$.) $$\xymatrix{
    \aS_{K^pK_{S_0}^oK_0^o} \ar[d] & \ar[l] \aS_{K^p} \ar[d] \\
    \aS_{K^pK_p^o}  & \ar[l] \aS_{K^pK_{p}^{S_0, o}}.
    }$$

Note that by using the moduli interpretation, $\aS_{K^pK_p^{S_0, o}K_{S_0}K_0}$ has a proper flat formal model $\mathfrak{S}_{K^pK_p^{S_0, o}K_{S_0}K_0}$ over $\mathcal{O}_C$ for any $K_TK_0$ and the transition map $\mathfrak{S}_{K^pK_p^{S_0, o}K_{S_0}K_0} \rightarrow \mathfrak{S}_{K^pK_p^{S_0, o}K_{S_0}'K_0'}$ is finite $\etale$ for any $K_{S_0}K_0 \subset K_{S_0}'K_0'$. (See {\S} 4.3 for more details.) By the definition of perfectoidness, we obtain the perfectoidness of $\aS_{K^pK_{S_0}K_0}$ from that of $\aS_{K^p}$. (Roughly speaking, $\aS_{K^p}$ is integrally profinite Galois $\etale$ over $\aS_{K^pK_{S_0}K_0}$.)

2 \ We have a variant of the above result. Actually, we can prove that taking $\mathfrak{n}$-invariant part of $B^{\Psi-\mathrm{la}}$ has no higher cohomology if $V_{\infty} \subset \pi_{\mathrm{HT}}^{-1}(U_1)$ by the same argument. (See the comment before Definition \ref{nilpotent} for the definition of $\mathfrak{n}$. This can also be proved by using an explicit structure of $\mathcal{O}_{\aS_{K^p}}^{\Psi-\mathrm{la}}$. See Theorem \ref{mikami expansionII}.) The author thinks that this also can be proved by using the result \cite[Proposition 3.2.34]{Torsion}, which says that the $\Gamma_1(p^{\infty})$-level (Siegel) Shimura variety is perfectoid around a strict neighborhood of the anticanonical tower. 

\end{rem}

\begin{cor}\label{nontrivial geometric}

Assume that $V$ is small. Then we have $R^i\Psi\mathfrak{LA}(D_{\lambda^{\Psi}, \aS_{K^p}}(V_{\infty})) = 0$ for any $i > 0$.

\end{cor}

\begin{proof}

Note that $D_{\lambda^{\Psi}, S_{K^p}}(V_{\infty}) = D_{\lambda^{\Psi}}(V) \otimes_{\mathcal{O}(V)} \mathcal{O}(V_{\infty})$ and $D_{\lambda^{\Psi}}(V)$ is a finite projective $\mathcal{O}(V)$-module. Thus the result follows from Theorem \ref{geometric sen theory}.  \end{proof}

Let $P := \prod_{\tau \in \Psi} B_{2, C} \times \prod_{\tau \notin \Psi} \mathrm{GL}_{2, C} \times \mathbb{G}_{m, C}$, $\mathfrak{n} := \prod_{\tau \in \Psi} \mathfrak{n}_{C}$ and $\mathfrak{b} := \prod_{\tau \in \Psi} \mathfrak{b}_{C}$, where $\mathfrak{n}_{C}$ (resp. $\mathfrak{b}_{C}$) denotes the Lie subalgebra of $\mathfrak{gl}_{2,C}$ consisting of the upper triangle nilpotent (resp. upper triangle) matrices.

\begin{dfn}\label{nilpotent}

$\mathfrak{n}^0$ (resp. $\mathfrak{b}^0$) denotes the $G_C$-equivariant vector bundle on $\Fl$ associated with the algebraic representation $\mathfrak{n}$ (resp. $\mathfrak{b}$) of $P$.

Note that for any affinoid open subset $U$ of $\Fl$, we have an exact sequence $0 \rightarrow \omega_{\tau, \Fl}^{-1}(U) \rightarrow V_{\tau} \rightarrow \omega_{\tau, \Fl}(U) \otimes \wedge^2V_{\tau} \rightarrow 0$. By using this, we obtain more explicit formulas $\mathfrak{n}^0(U) = \{ f = (f_{\tau})_{\tau} \in \oplus_{\tau \in \Psi} \mathrm{End}_{\mathcal{O}_{\Fl}(U)}(V_{\tau} \otimes_L \mathcal{O}_{\Fl}(U)) = \oplus_{\tau \in \Psi} \mathfrak{gl}_{2}(L) \otimes_L \mathcal{O}_{\Fl}(U) \mid f_{\tau}(\omega_{\tau, \Fl}^{-1}(U)) = 0, f(V_{\tau} \otimes_L \mathcal{O}_{\Fl}(U)) \subset \omega_{\tau, \Fl}^{-1}(U) \ \mathrm{for \ any} \ \tau \in \Psi \}$ and $\mathfrak{b}^0(U) = \{ f = (f_{\tau})_{\tau} \in \oplus_{\tau \in \Psi} \mathrm{End}_{\mathcal{O}_{\Fl}(U)}(V_{\tau} \otimes_L \mathcal{O}_{\Fl}(U)) \mid f(\omega_{\tau, \Fl}^{-1}(U)) \subset \omega_{\tau, \Fl}^{-1}(U) \ \mathrm{for \ any} \ \tau \in \Psi  \}$.

\end{dfn}

Note that $\mathcal{O}_{\Fl} = \mathcal{O}_{\Fl}^{\Psi-\mathrm{la}}$ by 2 of Example \ref{example} and thus we have a natural map $\mathcal{O}_{\Fl}(U) \rightarrow \mathcal{O}_{\aS_{K^p}}(V_{\infty})^{\Psi-\mathrm{la}}$ and this induces actions of $\mathfrak{n}^0(U)$, $\mathfrak{b}^0(U)$ and $\prod_{\tau \in \Psi} \mathfrak{gl}_2(L) \otimes_L \mathcal{O}_{\Fl}(U)$ on $\mathcal{O}_{\aS_{K^p}}^{\Psi-\mathrm{la}}(V_{\infty})$.

\begin{prop}\label{sen operator}
    
$\mathfrak{n}^0(U)$ acts trivially on $\mathcal{O}_{\aS_{K^p}}^{\Psi-\mathrm{la}}(V_{\infty})$.

\end{prop}

\begin{proof}

This follows from \cite[Corollary 6.1.12]{Cam} or we can prove this by using the same method as \cite[Theorem 4.2.7]{PanI}. \end{proof}

After shrinking $K_p$ if necessary, we may assume $G_0 := K_p = (1 + p^k \mathbb{Z}_p) \times \prod_{w \mid v} (1 + p^k M_2(\mathcal{O}_{F_w}))$ for some $k \ge 2$. Let $G_m := (1 + p^{k+m} \mathbb{Z}_p) \times \prod_{w \mid v} (1 + p^{k+m} M_2(\mathcal{O}_{F_w}))$ for $m \in \mathbb{Z}_{\ge 0}$ and $V_{G_m}$ be the inverse image of $V$ in $\aS_{K^pG_m}$. Assume that $U \subset U_1$. Let $x_{\tau}$ be the coordinate function $\mathcal{O}_{\Fl}(U)$ for any $\tau \in \Psi$, $e_{1, \tau} \in \omega_{\tau, \Fl}(U) \otimes \wedge^2V_{\tau}$ be the image of the canonical basis $\begin{pmatrix}
    1 \\ 0 
   \end{pmatrix}_{\tau}$ via $V_{\tau} \rightarrow \omega_{\tau, \Fl}(U) \otimes \wedge^2V_{\tau}$ and $f_{\tau}$ denote a basis of $\wedge^2V_{\tau}(-1)$. We can regard $f_{\tau}$ as a generator of $\wedge^2D_{\tau, \aS_{K^p}}(\pi_{\mathrm{HT}}^{-1}(U))$ via $\wedge^2D_{\tau, \aS_{K^p}} \cong \wedge^2V_{\tau}(-1) \otimes_L \mathcal{O}_{\aS_{K^p}}$. Note that $e_{1, \tau}$ is a generator of $\omega_{\tau, \Fl}(U) \otimes \wedge^2V_{\tau}$.

By using the property 3 of Theorem \ref{affine}, take an increasing sequence of positive integers $r(1) < r(2) < \cdots < r(n) < \cdots$ and take $x_{\tau, n} \in \mathcal{O}_{\aS_{K^pG_{r(n)}}}(V_{G_{r(n)}})$, $e_{\tau, 1,n} \in \omega_{\tau, \aS_{K^pG_{r(n)}}} \otimes \wedge^2D_{\tau, \aS_{K^pG_{r(n)}}}(V_{G_{r(n)}})$ and $f_{\tau, n} \in \wedge^2D_{\tau, \aS_{K^pG_{r(n)}}}(V_{G_{r(n)}})$ satisfying $|| x_{\tau} - x_{\tau, n} ||, || \frac{e_{\tau, 1,n}}{e_{\tau, 1}} - 1 ||, || \frac{f_{\tau, n}}{f_{\tau}} - 1 || \le p^{-n-1}$. Moreover, after replacing $r(n)$, we may assume $|| x_{\tau} - x_{\tau, n} || = || x_{\tau} - x_{\tau, n} ||_{G_{r(n)}}, || \frac{e_{\tau, 1,n}}{e_{\tau, 1}} - 1 || = || \frac{e_{\tau, 1, n}}{e_{\tau, 1}} - 1 ||_{G_{r(n)}}$ and $|| \frac{f_{\tau, n}}{f_{\tau}} - 1 || = || \frac{f_{\tau, n}}{f_{\tau}} - 1 ||_{G_{r(n)}}$ by Lemma \ref{norm}. Then $e_{\tau, 1, n}$ is a generator of $\omega_{\tau, \aS_{K^p}}(V_{\infty}) \otimes \wedge^2V_{\tau}$. Note that $\mathfrak{n}^0(U)$ is generated by $\begin{pmatrix}
x_{\tau} & x_{\tau}^2 \\ 
-1 & -x_{\tau} 
\end{pmatrix}$ ($\tau \in \Psi$). Let $\mathcal{O}_{\aS_{K^pG_{r(n)}}}(V_{G_{r(n)}})\lbrace x_{\tau} - x_{\tau, n}, \mathrm{log}\frac{e_{1, \tau}}{e_{1, \tau, n}}, \mathrm{log}\frac{f_{\tau}}{f_{\tau, n}} \rbrace_{\tau} := \{ f = \sum_{(i, j, k) \in (\mathbb{Z}^{\Psi}_{\ge 0})^{3}} a_{i, j, k} \prod_{\tau \in \Psi}(\mathrm{log}\frac{e_{1, \tau}}{e_{1, \tau, n}})^{i_{\tau}} (\mathrm{log}\frac{f_{\tau}}{f_{\tau, n}})^{j_{\tau}} (x_{\tau}-x_{\tau, n})^{k_{\tau}} \in \mathcal{O}_{\aS_{K^p}}(V_{\infty})^{G_{r(n)}-\Psi-\mathrm{an}} \mid a_{i, j, k} \in \mathcal{O}_{\aS_{K^pG_{r(n)}}}(V_{G_{r(n)}}) \ \mathrm{s. t.} \ \mathrm{sup}_{i,j,k}|| a_{i, j, k} ||p^{-n(\sum_{\tau}(i_{\tau} + j_{\tau} + k_{\tau}))} < \infty \}$.

This is a Banach space by the norm $|| f ||_{x_{\tau}, e_{1, \tau}, f_{\tau}} := \mathrm{sup}_{i, j, k} || a_{i, j, k} ||p^{-(n+1)(\sum_{\tau}(i_{\tau} + j_{\tau} + k_{\tau}))}$ and by this Banach space structure, the natural inclusion $\mathcal{O}_{\aS_{K^pG_{r(n)}}}(V_{G_{r(n)}})\lbrace x_{\tau} - x_{\tau, n}, \mathrm{log}\frac{e_{1, \tau}}{e_{1, \tau, n}}, \mathrm{log}\frac{f_{\tau}}{f_{\tau, n}} \rbrace_{\tau} \hookrightarrow \mathcal{O}_{\aS_{K^p}}(V_{\infty})^{G_{r(n)}-\Psi-\mathrm{an}}$ is continuous.

\begin{prop} \label{mikami expansionII}

There exists an integer $m$ such that for any $n \ge m$, the natural inclusion $$\mathcal{O}_{\aS_{K^p}}(V_{\infty})^{G_{0}-\Psi-\mathrm{an}} \hookrightarrow \mathcal{O}_{\aS_{K^p}}(V_{\infty})^{G_{r(n)}-\Psi-\mathrm{an}}$$ factors through a continuous injection $$\mathcal{O}_{\aS_{K^p}}(V_{\infty})^{G_{0}-\Psi-\mathrm{an}} \hookrightarrow \mathcal{O}_{\aS_{K^pG_{r(n)}}}(V_{G_{r(n)}})\lbrace x_{\tau} - x_{\tau, n}, \mathrm{log}\frac{e_{1, \tau}}{e_{1, \tau, n}}, \mathrm{log}\frac{f_{\tau}}{f_{\tau, n}} \rbrace_{\tau}.$$
        
\end{prop}

\begin{proof} Note that $||\mathrm{log}\frac{e_{1, \tau}}{e_{1, \tau, n}}|| = ||\frac{e_{1, \tau}}{e_{1, \tau, n}} - 1||$ and $\mathcal{O}_{\aS_{K^p}}(V_{\infty})^{G_{m}-\mathrm{an}, \mathfrak{b}} = \mathcal{O}_{\aS_{K^pG_m}}(V_{G_m})$ by Proposition \ref{sen operator}. Thus the result follows from Lemma \ref{formal expansion}. \end{proof}

We define $D_{\lambda^{\Psi}, \aS_{K^pG_{r(n)}}}(V_{G_{r(n)}})\lbrace x_{\tau} - x_{\tau, n}, \mathrm{log}\frac{e_{1, \tau}}{e_{1, \tau, n}}, \mathrm{log}\frac{f_{\tau}}{f_{\tau, n}} \rbrace_{\tau}$ similarly.

\begin{cor}\label{nontrivial mikami}

There exists an integer $m$ such that for any $n \ge m$, the natural inclusion $$D_{\lambda^{\Psi}, \aS_{K^p}}(V_{\infty})^{G_{0}-\Psi-\mathrm{an}} \hookrightarrow D_{\lambda^{\Psi}, \aS_{K^p}}(V_{\infty})^{G_{r(n)}-\Psi-\mathrm{an}}$$ factors through a continuous injection $$D_{\lambda^{\Psi}, \aS_{K^p}}(V_{\infty})^{G_{0}-\Psi-\mathrm{an}} \hookrightarrow D_{\lambda^{\Psi} , \aS_{K^pG_{r(n)}}}(V_{G_{r(n)}})\lbrace x_{\tau} - x_{\tau, n}, \mathrm{log}\frac{e_{1, \tau}}{e_{1, \tau, n}}, \mathrm{log}\frac{f_{\tau}}{f_{\tau, n}} \rbrace_{\tau}.$$

\end{cor}

Let $\mathcal{O}_{K^p} := \pi_{\mathrm{HT} *} \mathcal{O}_{\aS_{K^p}}$ and we use similar notations such as $D_{\tau, K^p}$ and $\omega_{\tau, K^p}$. In the following, we assume $V=V_{K_p}$.

\begin{prop} \label{derived calculation I}

We have $R^i\Psi\mathfrak{LA}(D_{\lambda^{\Psi}, K^p}(U)) = H^i(U, D_{\lambda^{\Psi}, K^p}^{\Psi-\mathrm{la}}) = 0$ for any $i > 0$.

\end{prop}

\begin{proof}

The equality $R^i\Psi\mathfrak{LA}(D_{\lambda^{\Psi}, K^p}(U)) = \check{H}^i(U, D_{\lambda^{\Psi}, K^p}^{\Psi-\mathrm{la}})$ follows from Corollary \ref{nontrivial geometric} and 3 of Lemma \ref{fundamental}. We may assume that $U \subset U_1$ by Lemma \ref{open unit ball}. In the following, we use the previous notations. 

The sheaf $$D_{\lambda^{\Psi}, \aS_{K^pG_{r(n)}}}(\cdot)\lbrace x_{\tau} - x_{\tau, n}, \mathrm{log}\frac{e_{1, \tau}}{e_{1, \tau, n}}, \mathrm{log}\frac{f_{\tau}}{f_{\tau, n}} \rbrace_{\tau} : W \mapsto D_{\lambda^{\Psi}, \aS_{K^pG_{r(n)}}}(W)\lbrace x_{\tau} - x_{\tau, n}, \mathrm{log}\frac{e_{1, \tau}}{e_{1, \tau, n}}, \mathrm{log}\frac{f_{\tau}}{f_{\tau, n}} \rbrace_{\tau}$$ on $V_{G_{r(n)}}$, where $W$ is a rational open of $V_{G_{r(n)}}$ can be regarded as $W \mapsto (\prod_{i, j, k} D_{\aS_{K^pG_{r(n)}}, \lambda^{\Psi}}(W)^o)[\frac{1}{p}]$. 

 We have the vanishing of $\check{H}^i(V_{G_{r(n)}}, D_{\lambda^{\Psi}, \aS_{K^pG_{r(n)}}}( \cdot )\lbrace x_{\tau} - x_{\tau, n}, \mathrm{log}\frac{e_{1, \tau}}{e_{1, \tau, n}}, \mathrm{log}\frac{f_{\tau}}{f_{\tau, n}} \rbrace_{\tau})$ by using the fact that $V_{G_{r(n)}}$ is affinoid and the open mapping theorem. This implies the vanishing of $\check{H}^i(U, D_{\lambda^{\Psi}, K^p}^{\Psi-\mathrm{la}})$ for any $i > 0$ by Corollary \ref{nontrivial mikami}. \end{proof}

\begin{thm} \label{geometric sen theoryII} Let $\mathfrak{m}$ be a decomposed generic maximal ideal of $\mathbb{T}^S$.

1 \ We have a canonical isomorphism $\widehat{H}^d(S_{K^p}, V_{\lambda^{\Psi}}(-\lambda_0))_{\mathfrak{m}}^{\Psi-\mathrm{la}} \widehat{\otimes}_{L} C \cong H^d(\Fl, D_{\lambda^{\Psi}, K^p}^{\Psi-\mathrm{la}})_{\mathfrak{m}}$ of $G_{L} \times G(\mathbb{Q}_p) \times 
    \mathbb{T}^S$-modules and these can be computed by the $\check{C}$ech complex $C(\mathcal{B}', D_{\lambda^{\Psi}, K^p}^{\Psi-\mathrm{la}})_{\mathfrak{m}}$ for any finite covering $\mathcal{B}' \subset \mathcal{B}$ of $\Fl$.

2 \ For any $i \neq d$, we have $H^i(\Fl, D_{\lambda^{\Psi}, K^p}^{\Psi-\mathrm{la}})_{\mathfrak{m}} = 0$.

3 \ For any $i > 0$, we have $R^i\Psi\mathfrak{LA}(H^d(\Fl, D_{\lambda^{\Psi}, K^p})_{\mathfrak{m}}) = 0$.

\end{thm}

\begin{proof}

By the primitive comparison \cite[Theorem 1.3]{pH} and \cite[proof of Theorem 4.4.3]{PanI}, we obtain a canonical isomorphism $\widehat{H}^i(S_{K^p}, \mathbb{Q}_p) \widehat{\otimes}_{\mathbb{Q}_p} C \cong H^i(\aS_{K^p}, \mathcal{O}_{\aS_{K^p}})$. By taking $\otimes_{L} V_{\lambda^{\Psi}}(-\lambda_0)$, we obtain $\widehat{H}^i(S_{K^p}, V_{\lambda^{\Psi}}(-\lambda_0)) \widehat{\otimes}_{L} C \cong H^i(\aS_{K^p}, V_{\lambda^{\Psi}}(-\lambda_0) \otimes_{L} \mathcal{O}_{\aS_{K^p}})$. This is equal to $H^i(\Fl, D_{\lambda^{\Psi}, K^p})$ by Proposition \ref{Hodge de Rham} and the affineness of $\pi_{\mathrm{HT}}$. Thus we have $\widehat{H}^i(S_{K^p}, V_{\lambda^{\Psi}}(-\lambda_0))_{\mathfrak{m}}^{\Psi-\mathrm{la}} \widehat{\otimes}_{L} C \cong H^i(\Fl, D_{\lambda^{\Psi}, K^p})_{\mathfrak{m}}^{\Psi-\mathrm{la}}$ and $H^i(\Fl, D_{\lambda^{\Psi}, K^p})_{\mathfrak{m}} = 0$ for any $i \neq d$ by Proposition \ref{maximal ideal}. Note that $H^i(\Fl, D_{\lambda^{\Psi}, K^p})$ can be computed by the $\mathrm{\check{C}}$ech complex $C(\mathcal{B}', D_{\lambda^{\Psi}, K^p})$.

By Proposition \ref{derived calculation I}, we have $H^{i}(C(\mathcal{B}', D_{\lambda^{\Psi}, K^p}^{\Psi-\mathrm{la}})) = H^i(\Fl, D_{\lambda^{\Psi}, K^p}^{\Psi-\mathrm{la}})$. By Lemma \ref{fundamental} and Proposition \ref{derived calculation I}, we have a $\mathbb{T}^S$-equivariant spectral sequence $$E_2^{u, t} := R^u\Psi\mathfrak{LA}(H^t(\Fl, D_{\lambda^{\Psi}, K^p})) \Rightarrow H^{u+t}(\Fl, D_{\lambda^{\Psi}, K^p}^{\Psi-\mathrm{la}}).$$ Note that since $H^i(\Fl, D_{\lambda^{\Psi}, K^p}) \cong \widehat{H}^i(S_{K^p}, V_{\lambda^{\Psi}}(-\lambda_0))_{\mathfrak{m}}^{\Psi-\mathrm{la}} \widehat{\otimes}_{L} C$ has a finite direct sum decomposition by maximal ideals of $\mathbb{T}^S$ (see Proposition \ref{maximal ideal}), we have $R^u\Psi\mathfrak{LA}(H^t(\Fl, D_{\lambda^{\Psi}, K^p}))_{\mathfrak{m}} = R^u\Psi\mathfrak{LA}(H^t(\Fl, D_{\lambda^{\Psi}, K^p})_{\mathfrak{m}})$. Thus by localizing the above spectral sequence at $\mathfrak{m}$, we obtain $H^i(\Fl, D_{\lambda^{\Psi}, K^p}^{\Psi-\mathrm{la}})_{\mathfrak{m}} = 0$ for any $i < d$, we have $H^d(\Fl, D_{\lambda^{\Psi}, K^p})_{\mathfrak{m}}^{\Psi-\mathrm{la}} \cong H^d(\Fl, D_{\lambda^{\Psi}, K^p}^{\Psi-\mathrm{la}})_{\mathfrak{m}}$ and $R^i\Psi\mathfrak{LA}(H^d(\Fl, D_{\lambda^{\Psi}, K^p})_{\mathfrak{m}}) \cong H^{i+d}(\Fl, D_{\lambda^{\Psi}, K^p}^{\Psi-\mathrm{la}})_{\mathfrak{m}} = 0$ for any $i > 0$.  \end{proof}

\begin{rem}

In the previous theorem, we used the complex $C(\mathcal{B}', D_{\lambda^{\Psi}, K^p}^{\Psi-\mathrm{la}})_{\mathfrak{m}}$, whose components may not be LB spaces. We need not use this complex if we can prove that $\pi_{\mathrm{HT}}$ is well-behaved as in 3 of Theorem \ref{affine} for an affinoid open $\prod_{j_{\tau}} U_{j_{\tau}, \tau}$ ($j_{\tau} = 1, 2$) of $\Fl$ and use the same method as \cite[Theorem IV 3.1]{Torsion}.

\end{rem}

\subsection{Hodge-Tate structure of completed cohomology}

In the following, we assume that $\zeta_{p^2} \in L$ and $L$ is Galois over $\mathbb{Q}_p$ for simplicity.  We put $L_{\mathrm{cyc}} := \cup_n L(\zeta_{p^n})$, $H_L := \mathrm{Gal}(\overline{L}/L_{\mathrm{cyc}})$ and $\Gamma_L:=\mathrm{Gal}(L_{\mathrm{cyc}}/L)$. Then we have an open and closed immersion $\Gamma_L \hookrightarrow 1 + p^2 \mathbb{Z}_p$ via $\varepsilon_p$. Thus $\Gamma_L$ is isomorphic to $\mathbb{Z}_p$. First, we recall basic definitions and properties.

\begin{dfn}\label{Hodge-Tate}

Let $W$ be a Banach space over $C$ with a semilinear continuous action of $G_L$ and we put $W^L := W^{H_L, \Gamma_L-\mathrm{an}}$.

1 \ We say that $W$ is decomposable if there exists a finite extension $L'/L$ such that the canonical map $W^{L'} \widehat{\otimes}_{L'} C \rightarrow W$ is an isomorphism.

2 \ If $W$ is decomposable, then the Sen operator $\theta_{\mathrm{Sen}}$ means the operator $1 \widehat{\otimes}_{L'} \mathrm{id}_{C} \in \mathrm{End}_{C}(W)$ defined by the $C$-linear extension of an operator $1 \in \mathbb{Q}_p = \mathrm{Lie}\Gamma_{L'}$ on $W^{L'}$. Note that this is independent of the choice of $L'$.

3 \ We say that $W$ is Hodge-Tate of weight $w_1, \cdots, w_k \in \mathbb{Z}$ if $W$ is decomposable and $\theta_{\mathrm{Sen}}$ is semisimple with eigenvalues $-w_1, \cdots, -w_k$.

4 \ We say that $W$ is Hodge-Tate if $W$ is Hodge-Tate of some integers $w_1, \cdots, w_k$.

\end{dfn}

\begin{lem}\label{HTE} Let $W$ be a Banach space over $C$ with a semilinear continuous action of $G_L$.

1 \ If $\mathrm{dim}_CW < \infty$, then $W$ is always decomposable.

2 \ Let $V$ be a finite-dimensional representation of $G_L$ over $\mathbb{Q}_p$ such that $W = V \otimes_{\mathbb{Q}_p} C$. Then $V$ is Hodge-Tate in the usual sense if and only if $W$ is Hodge-Tate in the sense of Definition \ref{Hodge-Tate}. 

3 \ Let $V$ be a finite-dimensional representation of $G_L$ over $L$. Then $V$ is Hodge-Tate in the usual sense if and only if $V^{\sigma} \otimes_{L} C$ is Hodge-Tate in the above sense for any $\sigma \in \mathrm{Gal}(L/\mathbb{Q}_p)$.

4 \ Let $M$ be a Banach representation of $G_L$ over $L$ such that $M = M^{L}$. If $W = M \widehat{\otimes}_{L} C$, then $W^L = M$ and $W$ is decomposable.

5 \ If $W$ is Hodge-Tate of weight $w_1, \cdots, w_k$, then $W = \oplus_{i=1}^k W(w_i)^{G_L} \widehat{\otimes}_L C(-w_i)$.

6 \ For a finite extension $L'/L$, $W$ is Hodge-Tate if and only if $W|_{G_{L'}}$ is Hodge-Tate. 

7 \ If $W$ is Hodge-Tate, then the natural map $W^{L} \widehat{\otimes}_L C \rightarrow W$ is an isomorphism.

8 \ If $W$ is decomposable, then $H^i(H_L, W) = 0$ and $R^i\mathfrak{LA}(W^{H_L}) = 0$ for any $i > 0$. (Here, we regard $W^{H_L}$ as a Banach representation of $\Gamma_L$ over $\mathbb{Q}_p$.)

9 \ Let $0 \rightarrow W_1 \rightarrow W_2 \rightarrow W_3$ be an exact sequence of Banach spaces over $C$ with semilinear continuous actions of $G_L$. If $W_2$ and $W_3$ are Hodge-Tate (resp. decomposable), then $W_1$ is Hodge-Tate (resp. decomposable). On the other hand, if $W_2 \rightarrow W_3$ is surjective and $W_1$ and $W_2$ are Hodge-Tate (resp. decomposable), then $W_3$ is Hodge-Tate (resp. decomposable).

\end{lem}

\begin{proof}

See \cite[Theorem 3 and Corollary of Theorem 6]{Sen} for proofs of 1 and 2. 

By 2, $V$ is Hodge-Tate in the usual sense if and only if $V \otimes_{L, \sigma} C$ is Hodge-Tate in the above sense for any $\sigma \in \mathrm{Gal}(L/\mathbb{Q}_p)$. Thus the statement 3 follows from an isomorphism $(V \otimes_{L, \sigma} C)^{\sigma} \Isom V^{\sigma} \otimes_L C, \ v \otimes x \rightarrow v \otimes \sigma^{-1}(x)$.

4 follows from \cite[Proposition 3.3.3]{PanII}. 5 and 6 follow from \cite[Remark 3.3.7]{PanII}. 7 follows from 5. 8 follows from \cite[Theorem 2.4.3]{Cam}. 9 follows from \cite[Corollary 3.3.4]{PanII}. \end{proof}

For $U \in \tilde{\mathcal{B}}$, we take $G_0 = K_p$, $G_m$ and $V_{K_p} = V$ as in the previous section. (See the discussions after Proposition \ref{sen operator}.) After extending $L$ if necessary, we have a natural action of $G_L$ on $\mathcal{O}_{K^p}^{\Psi-\mathrm{la}}(U)$.

\begin{lem}\label{LBHT}

Let $W$ be a Banach space over $C$ with a continuous action of $G_0$ over $C$ and a continuous semilinear action of $G_L$. Assume that the two actions of $G_0$ and $G_L$ are commutative and we have a sequence of continuous injections of Banach spaces of $W_1 \hookrightarrow W_2 \hookrightarrow \cdots \hookrightarrow W^{\Psi-\mathrm{la}}$ such that for any $n$, there exists $k >0$ and a finite extension $M$ of $L$ such that $W_n$ has a continuous action of $G_k$ and $G_M$ compatible with the action on $W$ such that $W_n^{M} \widehat{\otimes}_{M} C \rightarrow W_n$ is an isomorphism and $\varinjlim_n W_n = W^{\Psi-\mathrm{la}}$. Then $W^{G_k-\Psi-\mathrm{an}}$ is decomposable for any $k$.

\end{lem}

\begin{proof}

Let $k$ be a nonnegative integer. By Proposition \ref{inductive limit}, there exists $n$, $k' \ge k$ and a finite extension $M$ of $L$ such that $W_n$ is $G_M$-stable, $W_n^{M} \widehat{\otimes}_{M} C \xrightarrow{\sim} W_n$ is an isomorphism and we obtain the following diagram with commutative maps.$$\xymatrix{
    W^{G_k-\Psi-\mathrm{an}, M} \widehat{\otimes}_M C \ar[d] \ar[r] & \ar[r] W_n^M \widehat{\otimes}_M C \ar[d]^{\sim} & W^{G_{k'}-\Psi-\mathrm{an}, M} \widehat{\otimes}_M C \ar[d] \ar[r] & W^{\Psi-\mathrm{la}, M} \widehat{\otimes}_M C \ar[d] \\
   W^{G_k-\Psi-\mathrm{an}} \ar@{^{(}-_>}[r] \ar@{^{(}-_>}[ur] & \ar@{^{(}-_>}[r] W_n & W^{G_{k'}-\Psi-\mathrm{an}} \ar@{^{(}-_>}[r] & W^{\Psi-\mathrm{la}}.
    }$$

In particular, we have $$\xymatrix{
    W^{G_k-\Psi-\mathrm{an}, M} \widehat{\otimes}_M C \ar[d] \ar[r] & W^{G_{k'}-\Psi-\mathrm{an}, M} \widehat{\otimes}_M C \ar[d] \\
    W^{G_k-\Psi-\mathrm{an}} \ar@{^{(}-_>}[r] \ar@{^{(}-_>}[ur] & W^{G_{k'}-\Psi-\mathrm{an}}.
    }$$

By taking the $G_{k}$-$\Psi$-analytic parts, we obtain $W^{G_k-\Psi-\mathrm{an}, M} \widehat{\otimes}_M C \Isom W^{G_k-\Psi-\mathrm{an}}$. \end{proof}

Let $\mathfrak{h} = \prod_{\tau \in \Psi} \mathfrak{h}_{\tau}$ be the Lie subalgebra of $\prod_{\tau \in \Psi} \mathfrak{gl}_2(C)$ consisting of diagonal matrices. Then we have a canonical isomorphism $\mathfrak{h} \cong H^0(\Fl, \mathfrak{b}^0/\mathfrak{n}^0)$ by using a filtration $0 \rightarrow \omega_{\tau, \Fl}^{-1} \rightarrow V_{\tau} \otimes \mathcal{O}_{\Fl} \rightarrow \omega_{\tau, \Fl} \otimes \wedge^2V_{\tau} \rightarrow 0$ and this acts on $\mathcal{O}_{K^p}^{\Psi-\mathrm{la}}$ by Corollary \ref{sen operator}. We call this action the horizontal action $\theta_{\mathfrak{h}}$ of $\mathfrak{h}$ on $\mathcal{O}_{K^p}^{\Psi-\mathrm{la}}$, which commutes with the action of $G(\mathbb{Q}_p)$. Note that if $U \subset U_1$, then $\theta_{\mathfrak{h}}(\begin{pmatrix}
    c_{\tau} & 0 \\
     0 & d_{\tau} \end{pmatrix}_{\tau}) = (\begin{pmatrix}
        d_{\tau} & x_{\tau}(d_{\tau} - c_{\tau}) \\
         0 & c_{\tau} \end{pmatrix}_{\tau}) \mod \mathfrak{n}^o(U) \in \mathfrak{b}^o(U)/\mathfrak{n}^o(U)$. Thus $\mathfrak{h}$ acts on $\mathcal{O}_{\Fl}$ trivially.

\begin{prop}\label{Sen operator} 

For any $k$, $D_{\lambda^{\Psi}, K^p}(U)^{G_{k}-\Psi-\mathrm{an}}$ is decomposable and the Sen operator $\theta_{\mathrm{Sen}}$ is equal to $\theta_{\mathfrak{h}}(\begin{pmatrix}
-1 & 0 \\
 0 & 0 \end{pmatrix}_{\tau})$.

\end{prop}

\begin{proof} We may assume that $U \subset U_1$. By Corollary \ref{nontrivial mikami} and Lemma \ref{LBHT}, it suffices to prove that $D_{\aS_{K^pG_{r(n)}}, \lambda^{\Psi}}(V_{G_{r(n)}})\lbrace x_{\tau} - x_{\tau, n}, \mathrm{log}\frac{e_{1, \tau}}{e_{1, \tau, n}}, \mathrm{log}\frac{f_{\tau}}{f_{\tau, n}} \rbrace_{\tau}$ is decomposable and the Sen operator $\theta_{\mathrm{Sen}}$ is equal to $\theta_{\mathfrak{h}}(\begin{pmatrix}
    -1 & 0 \\
     0 & 0 \end{pmatrix}_{\tau})$.

After extending $L$, we may assume that $x_{\tau, n}$, $e_{1, \tau, n}$ and $f_{\tau, n}$ are defined over $L$. Thus we have the equality between $D_{\aS_{K^pG_{r(n)}}, \lambda^{\Psi}}(V_{G_{r(n)}})\lbrace x_{\tau} - x_{\tau, n}, \mathrm{log}\frac{e_{1, \tau}}{e_{1, \tau, n}}, \mathrm{log}\frac{f_{\tau}}{f_{\tau, n}} \rbrace_{\tau}$, $D_{\aS_{K^pG_{r(n)}}, \lambda^{\Psi}}(V_{G_{r(n)}})^{G_L}\lbrace x_{\tau} - x_{\tau, n}, \mathrm{log}\frac{e_{1, \tau}}{e_{1, \tau, n}}, \mathrm{log}\frac{f_{\tau}}{f_{\tau, n}} \rbrace_{\tau} \widehat{\otimes}_{L} C$ and $D_{\aS_{K^pG_{r(n)}}, \lambda^{\Psi}}(V_{G_{r(n)}})\lbrace x_{\tau} - x_{\tau, n}, \mathrm{log}\frac{e_{1, \tau}}{e_{1, \tau, n}}, \mathrm{log}\frac{f_{\tau}}{f_{\tau, n}} \rbrace_{\tau}^{L} \widehat{\otimes}_{L} C$. Thus we obtain the result. \end{proof}

\begin{cor}\label{sen operatorII}

For any $k$, $\widehat{H}^d(S_{K^p}, V_{\lambda^{\Psi}}(-\lambda_0))_{\mathfrak{m}}^{G_k-\Psi-\mathrm{an}} \widehat{\otimes}_{L} C \cong H^d(\Fl, D_{\lambda^{\Psi}, K^p}^{\Psi-\mathrm{la}})_{\mathfrak{m}}^{G_{k}-\Psi-\mathrm{an}}$ is decomposable and the Sen operator $\theta_{\mathrm{Sen}}$ is equal to $\theta_{\mathfrak{h}}(\begin{pmatrix}
        -1 & 0 \\
         0 & 0 \end{pmatrix}_{\tau})$.

\end{cor}

\begin{proof}
The same argument as \cite[Remark 5.1.16]{PanI} implies that $$\widehat{H}^d(S_{K^p}, V_{\lambda^{\Psi}}(-\lambda_0))_{\mathfrak{m}}^{G_k-\Psi-\mathrm{an}} \widehat{\otimes}_{L} C$$ is decomposable. The formula of the Sen operator follows from Proposition \ref{Sen operator}. \end{proof}

For $(a_{\tau}, b_{\tau}) \in (\mathbb{Z}^2)^{\Psi}$, we have a character $\mathfrak{h} = \prod_{\tau \in \Psi} \mathfrak{h}_{\tau} \rightarrow C, \ \begin{pmatrix}
    c_{\tau} & 0 \\
     0 & d_{\tau} \end{pmatrix}_{\tau} \mapsto \sum_{\tau} (a_{\tau} c_{\tau} + b_{\tau} d_{\tau})$. We also write $(a_{\tau}, b_{\tau})$ for this character. 

\begin{prop}\label{derived calculation II} Let $(a_{\tau}, b_{\tau}) \in (\mathbb{Z}^2)^{\Psi}$.

1 \ $H^i(\mathfrak{h}, D_{\lambda^{\Psi}, K^p}^{\Psi-\mathrm{la}}(U)(-(a_{\tau}, b_{\tau}))) = \mathrm{Ext}^i_{\mathfrak{h}}((a_{\tau}, b_{\tau}), D_{\lambda^{\Psi}, K^p}^{\Psi-\mathrm{la}}(U)) = 0$ for any $i > 0$.

2 \ We have $H^i(U, D_{\lambda^{\Psi}, K^p}^{\Psi-\mathrm{la}, (a_{\tau}, b_{\tau})}) = 0$ for any $i > 0$.

\end{prop}

\begin{proof} We may assume $U \subset U_1$. The statement 1 can be proved by using explicit formula Corollary \ref{nontrivial mikami} and the Hochschild-Serre spectral sequence as in \cite[Lemma 5.1.2]{PanI}. Note that we have $R\Gamma(U, R\Gamma(\mathfrak{h}, D_{\lambda^{\Psi}, K^p}^{\Psi-\mathrm{la}}(-(a_{\tau}, b_{\tau})))) = R\Gamma(\mathfrak{h}, R\Gamma(U, D_{\lambda^{\Psi}, K^p}^{\Psi-\mathrm{la}})(-(a_{\tau}, b_{\tau})))$. Thus the statement 2 follows from the statement 1 and Proposition \ref{derived calculation I}. \end{proof}

\begin{prop} \label{horizontal action} Let $\mathfrak{m}$ be a decomposed generic maximal ideal of $\mathbb{T}^S$ and $(a_{\tau}, b_{\tau}) \in (\mathbb{Z}^2)^{\Psi}$.

1 \ We have $H^d(\Fl, D_{\lambda^{\Psi}, K^p}^{\Psi-\mathrm{la}})_{\mathfrak{m}}^{(a_{\tau}, b_{\tau})} = H^d(\Fl, D_{\lambda^{\Psi}, K^p}^{\Psi-\mathrm{la}, (a_{\tau}, b_{\tau})})_{\mathfrak{m}}$ and these can be computed by the $\check{C}$ech complex $C(\mathcal{B}', D_{\lambda^{\Psi}, K^p}^{\Psi-\mathrm{la}, (a_{\tau}, b_{\tau})})_{\mathfrak{m}}$ for any finite covering $\mathcal{B}' \subset \mathcal{B}$ of $\Fl$.

2 \ $H^d(\Fl, D_{\lambda^{\Psi}, K^p}^{\Psi-\mathrm{la}, (a_{\tau}, b_{\tau})})_{\mathfrak{m}}^{G_k-\Psi-\mathrm{an}}$ is Hodge-Tate of weight $\sum_{\tau} a_{\tau}$ for any $k$.

3 \ For any $i \neq d$, we have $H^i(\Fl, D_{\lambda^{\Psi}, K^p}^{\Psi-\mathrm{la}, (a_{\tau}, b_{\tau})})_{\mathfrak{m}} = 0$.

\end{prop}

\begin{proof} By Proposition \ref{derived calculation II} and Theorem \ref{geometric sen theoryII}, we have $R\Gamma(\mathfrak{h}, H^d(\Fl, D^{\Psi-\mathrm{la}}_{\lambda^{\Psi}, K^p})_{\mathfrak{m}}[-d](-a_{\tau}, -b_{\tau})) = R\Gamma(\mathfrak{h}, R\Gamma(\Fl, D^{\Psi-\mathrm{la}}_{\lambda^{\Psi}, K^p})_{\mathfrak{m}}(-a_{\tau}, -b_{\tau})) = R\Gamma(\Fl, R\Gamma(\mathfrak{h}, D^{\Psi-\mathrm{la}}_{\lambda^{\Psi}, K^p}(-a_{\tau}, -b_{\tau}))_{\mathfrak{m}} = R\Gamma(\Fl, D^{\Psi-\mathrm{la}, (a_{\tau}, b_{\tau})}_{\lambda^{\Psi}, K^p})_{\mathfrak{m}}$. This implies 1 and 3. The statement 2 follows from 1 and Corollary \ref{sen operatorII} \end{proof}

Assume $U \subset U_1$.

Let  
$$((\otimes_{\tau} (\omega^{a_{\tau} - b_{\tau}}_{\tau, \aS_{K^pG_{r(n)}}} \otimes (\wedge^2D_{\tau, \aS_{K^pG_{r(n)}}})^{-b_{\tau}})) \otimes D_{\lambda^{\Psi}, \aS_{K^pG_{r(n)}}})(V_{G_{r(n)}})\lbrace x_{\tau} - x_{\tau, n} \rbrace_{\tau}(\prod_{\tau}f_{\tau}^{a_{\tau} }e_{1, \tau}^{b_{\tau} - a_{\tau}})$$ 

$:= \{ f = (\prod_{\tau}f_{\tau}^{a_{\tau} }e_{1, \tau}^{b_{\tau} - a_{\tau}}) \sum_{i = (i_{\tau}) \in \mathbb{Z}^{\Psi}_{\ge 0}} c_{i} \prod_{\tau \in \Psi}(x_{\tau}-x_{\tau, n})^{i_{\tau}} \in D_{\lambda^{\Psi}, K^p}(U)^{G_{r(n)}-\Psi-\mathrm{an}, (a_{\tau}, b_{\tau})}$ 

$\mid c_{i} \in ((\otimes_{\tau} (\omega^{a_{\tau} - b_{\tau}}_{\aS_{K^pG_{r(n)}}, \tau} \otimes (\wedge^2D_{\tau, \aS_{K^pG_{r(n)}}})^{-b_{\tau}})) \otimes D_{\lambda^{\Psi}, \aS_{K^pG_{r(n)}}})(V_{G_{r(n)}})\ \mathrm{s. \ t.} \ \mathrm{sup}_{i}|| c_{i} ||p^{-n(\sum_{\tau} i_{\tau})} < \infty \}.$ 

This is a Banach space by the norm $|| f ||_{x_{\tau}} := \mathrm{sup}_{i} || c_{i} ||p^{-(n+1)(\sum_{\tau}i_{\tau})}$. Then the natural inclusion $((\otimes_{\tau} (\omega^{a_{\tau} - b_{\tau}}_{\tau, \aS_{K^pG_{r(n)}}} \otimes (\wedge^2D_{\tau, \aS_{K^pG_{r(n)}}})^{-b_{\tau}})) \otimes D_{\lambda^{\Psi}, \aS_{K^pG_{r(n)}}})(V_{G_{r(n)}})\lbrace x_{\tau} - x_{\tau, n} \rbrace_{\tau}(\prod_{\tau}f_{\tau}^{a_{\tau} }e_{1, \tau}^{b_{\tau} - a_{\tau}}) \hookrightarrow D_{\lambda^{\Psi}, K^p}(U)^{G_{r(n)}-\Psi-\mathrm{an}, (a_{\tau}, b_{\tau})}$ is continuous. 

\begin{prop}\label{mikami expansionIII}

There exists a positive integer $m$ such that for any $n \ge m$, the natural inclusion $D_{\lambda^{\Psi}, K^p}(U)^{G_0-\Psi-\mathrm{an}, (a_{\tau}, b_{\tau})} \hookrightarrow D_{\lambda^{\Psi}, K^p}(U)^{G_{r(n)}-\Psi-\mathrm{an}, (a_{\tau}, b_{\tau})}$ factors through a continuous injection $D_{\lambda^{\Psi}, K^p}(U)^{G_0-\Psi-\mathrm{an}, (a_{\tau}, b_{\tau})} \hookrightarrow ((\otimes_{\tau} (\omega^{a_{\tau} - b_{\tau}}_{\tau, \aS_{K^pG_{r(n)}}} \otimes (\wedge^2D_{\tau, \aS_{K^pG_{r(n)}}})^{-b_{\tau}})) \otimes D_{\lambda^{\Psi}, \aS_{K^pG_{r(n)}}})(V_{G_{r(n)}})\lbrace x_{\tau} - x_{\tau, n} \rbrace_{\tau}(\prod_{\tau}f_{\tau}^{a_{\tau} }e_{1, \tau}^{b_{\tau} - a_{\tau}})$.

\end{prop}

\begin{rem}

 Exactly the same proof gives a similar explicit description for $D_{\lambda^{\Psi}, \mathcal{S}_{K^p}}(W_{\infty})$ for the inverse image $W_{\infty}$ in $\aS_{K^p}$ of a rational open $W$ of $V$, i.e., locally on $\aS_{K^p}$ not locally on $\Fl$. This will be used in {\S} 4.5.

\end{rem}

\begin{proof} This follows from Corollary \ref{nontrivial mikami}. \end{proof}

For convenience, we record a simpler form of the above result in the case $\lambda_{\Psi} = (0, \lambda_{\tau})_{\tau}$, which we will assume from the next section.

\begin{cor}\label{citation}

There exists a positive integer $m$ such that for any $n \ge m$, the natural inclusion $D_{\lambda^{\Psi}, K^p}(U)^{G_0-\Psi-\mathrm{an}, (0, \lambda_{\tau})} \hookrightarrow D_{\lambda^{\Psi}, K^p}(U)^{G_{r(n)}-\Psi-\mathrm{an}, (0, \lambda_{\tau})}$ factors through a continuous injection $D_{\lambda^{\Psi}, K^p}(U)^{G_0-\Psi-\mathrm{an}, (0, \lambda_{\tau})} \hookrightarrow ((\otimes_{\tau} (\omega^{- \lambda_{\tau}}_{\tau, \aS_{K^pG_{r(n)}}} \otimes (\wedge^2 D_{\tau, \aS_{K^pG_{r(n)}}})^{-\lambda_{\tau}})) \otimes D_{\lambda^{\Psi}, \aS_{K^pG_{r(n)}}})(V_{G_{r(n)}})\lbrace x_{\tau} - x_{\tau, n} \rbrace_{\tau}(\prod_{\tau}e_{1, \tau}^{\lambda_{\tau}})$.

\end{cor}

\begin{lem}\label{horizontal weight}

For any $(m_{\tau}, n_{\tau})_{\tau} \in (\mathbb{Z}^2)^{\Psi}$, we have a canonical isomorphism $\mathcal{O}_{K^p}^{\Psi-\mathrm{la}, (a_{\tau}, b_{\tau})_{\tau}} \otimes_{\mathcal{O}_{\Fl}} (\otimes_{\tau \in \Psi} (\omega_{\tau, \Fl}^{n_{\tau}} \otimes (\wedge^2 V_{\tau})^{m_{\tau}}))  \cong \mathcal{O}_{K^p}^{\Psi-\mathrm{la}, (a_{\tau}+m_{\tau} - n_{\tau}, b_{\tau}+m_{\tau})_{\tau}} \otimes (\otimes_{\tau \in \Psi} (\omega_{\tau, K^p}^{n_{\tau}, \mathrm{sm}} \otimes (\wedge^2 D_{\tau, K^p}^{\mathrm{sm}})^{m_{\tau}}))(\sum_{\tau}(m_{\tau}-n_{\tau}))$. 

\end{lem}

\begin{proof}

By the definition of $\pi_{\mathrm{HT}}$, we have $\pi_{\mathrm{HT}}^{*} \omega_{\tau, \Fl} = \omega_{\tau, K^p}(-1)$. Thus we have $\mathcal{O}_{K^p}^{\Psi-\mathrm{la}} \otimes_{\mathcal{O}_{\Fl}} (\otimes_{\tau \in \Psi} (\omega_{\tau, \Fl}^{n_{\tau}} \otimes (\wedge^2 V_{\tau})^{m_{\tau}}))  \cong \mathcal{O}_{K^p}^{\Psi-\mathrm{la}} \otimes (\otimes_{\tau \in \Psi} (\omega_{\tau, K^p}^{n_{\tau}, \mathrm{sm}} \otimes (\wedge^2 D_{\tau, K^p}^{\mathrm{sm}})^{m_{\tau}}))(\sum_{\tau}(m_{\tau}-n_{\tau}))$ by Proposition \ref{Hodge de Rham}. We claim that this induces the isomorphism of the subsheaves $\mathcal{O}_{K^p}^{\Psi-\mathrm{la}, (a_{\tau}, b_{\tau})_{\tau}} \otimes_{\mathcal{O}_{\Fl}} (\otimes_{\tau \in \Psi} (\omega_{\tau, \Fl}^{n_{\tau}} \otimes (\wedge^2 V_{\tau})^{m_{\tau}}))  \cong \mathcal{O}_{K^p}^{\Psi-\mathrm{la}, (a_{\tau}+m_{\tau} - n_{\tau}, b_{\tau}+m_{\tau})_{\tau}} \otimes (\otimes_{\tau \in \Psi} (\omega_{\tau, K^p}^{n_{\tau}, \mathrm{sm}} \otimes (\wedge^2 D_{\tau, K^p}^{\mathrm{sm}})^{m_{\tau}}))(\sum_{\tau}(m_{\tau}-n_{\tau}))$. To see this, the left hand side is contained in the right hand side because the right hand side is the $(a_{\tau}+m_{\tau} - n_{\tau}, b_{\tau}+m_{\tau})_{\tau}$-isotropic part of the whole sheaf. The converse also follows from the same reasoning by twisting $\otimes_{\mathcal{O}_{\Fl}} (\otimes_{\tau \in \Psi} (\omega_{\tau, \Fl}^{-n_{\tau}} \otimes (\wedge^2 V_{\tau})^{-m_{\tau}}))$. \end{proof}

Let $\mathfrak{g}_{\Psi} := \oplus_{\tau \in \Psi} \mathfrak{gl}_2(L)_{\tau}$. Let $W := \prod_{\tau \in \Psi} \langle \begin{pmatrix}
    0 & 1 \\
     1 & 0 \end{pmatrix}_{\tau} \rangle_{\tau} \subset \prod_{\tau \in \Psi} \mathrm{GL}_2(L)$, which can be naturally regarded as a Weyl group of $\prod_{\tau \in \Psi} \mathrm{GL}_2(L)$, let $w_0 = \begin{pmatrix}
    0 & 1 \\
     1 & 0 \end{pmatrix}_{\tau} \in W$ be the longest element and $\mathrm{Sym} \mathfrak{h}$ be the symmetric algebra of $\mathfrak{h}$. We have the dot action on $\mathfrak{h}^{\vee}$ of $W$, i.e., for any $w = (w_{\tau}) \in W$ and $h = (c_{\tau}, d_{\tau})_{\tau} \in \mathfrak{h}^{\vee}$, we put $w \cdot h := w(h + \rho)w^{-1} -\rho$, where $\rho := (\frac{1}{2}, -\frac{1}{2})_{\tau} \in \mathfrak{h}^{\vee}$. This induces an action of $W$ on the symmetric algebra $\mathrm{Sym}\mathfrak{h}$ of $\mathfrak{h}$ by identification $\mathfrak{h}^{\vee} = \mathrm{Hom}_{C-\mathrm{alg}}(\mathrm{Sym}\mathfrak{h}, C)$. We have the unnormalized Harish-Chandra isomorphism $\mathrm{HC} : Z(U(\mathfrak{g}_{\Psi}))_C \Isom (\mathrm{Sym} \mathfrak{h})^W$, where $Z(U(\mathfrak{g}_{\Psi}))$ denotes the center of the universal enveloped algebra $U(\mathfrak{g}_{\Psi})$ of $\mathfrak{g}_{\Psi}$.

\begin{lem}\label{infinitesimal character and horizontal}

 The natural action of the center of the universal enveloped algebra $Z(U(\mathfrak{g}_{\Psi}))_C$ of $\mathfrak{g}_{\Psi}$ on $D_{\lambda^{\Psi}, K^p}^{\Psi-\mathrm{la}}$ is equal to $\theta_{\mathfrak{h}}(w_0\mathrm{HC}( \cdot )w_0^{-1})$.

\end{lem}

\begin{proof} Same as \cite[proof of Corollary 4.2.8]{PanI}. \end{proof}

\begin{prop} (Hodge-Tate decomposition of locally analytic completed cohomology) \label{Hodge-Tate decomposition} Let $\chi_{\lambda_{\Psi}} : Z(U(\mathfrak{g}_{\Psi})) \rightarrow L$ be the infinitesimal character of $V_{\lambda_{\Psi}}^{\vee}$.

1 \ $D^{\Psi-\mathrm{la}, \chi_{\lambda_{\Psi}}}_{\lambda^{\Psi}, K^p} = \oplus_{I \subset \Psi} D^{\Psi-\mathrm{la}, (0, \lambda_{\tau})_{\tau \in I}, (1 +\lambda_{\tau}, -1)_{\tau \notin I}}_{\lambda^{\Psi}, K^p}$.

2 \ Let $\mathfrak{m}$ be a decomposed generic maximal ideal of $\mathbb{T}^S$.

Then $\widehat{H}^d(S_{K^p}, V_{\lambda^{\Psi}}(-\lambda_0))_{\mathfrak{m}}^{\Psi-\mathrm{la}, \chi_{\lambda_{\Psi}}} \widehat{\otimes}_{L} C \cong \oplus_{I \subset \Psi} H^d(\Fl, D^{\Psi-\mathrm{la}, (0, \lambda_{\tau})_{\tau \in I}, ( 1 +\lambda_{\tau}, -1)_{\tau \notin I}}_{K^p, \lambda^{\Psi}})_{\mathfrak{m}}$.

\end{prop}

\begin{rem}

Note that this structure is compatible with the Galois structure at finite levels $r_{\iota}(\chi)|_{G_{\tilde{F}}} \otimes (\otimes_{\tau \in \Psi} (r_{\iota}(\pi)|_{G_{\tilde{F}}})^{\tau})$ appeared in Theorem \ref{kottwitz conjecture}.

\end{rem}

\begin{proof}

This follows from Lemma \ref{infinitesimal character and horizontal} and Proposition \ref{horizontal action}. In fact, $(a_{\tau}, b_{\tau})_{\tau} : \mathrm{Sym}\mathfrak{h} \rightarrow C$ induces the same character on $\mathrm{Im}(\mathrm{HC})$ as $(\lambda_{\tau}, 0)_{\tau}$ if and only if $(a_{\tau}, b_{\tau})_{\tau}$ is contained in the $W$-orbit of $(\lambda_{\tau}, 0)_{\tau}$, i.e., $\{ ((\lambda_{\tau}, 0)_{\tau \in I}, ( -1, 1 +\lambda_{\tau})_{\tau \notin I}) \mid I \subset \Psi \}$.
\end{proof}

\section{Geometric locally analytic de Rham complexes}

The purpose of this section is to construct a certain locally analytic extension of the de Rham complexes at finite levels and to calculate that complex on the $\mu$-ordinary locus and the basic locus. In particular, we prove that the cohomology groups of that complex have generalized eigenspace decompositions by eigensystems associated to classical automorphic representations in some cases. (See ${\S} 4.4$.)

\subsection{Construction of geometric locally analytic de Rham complexes}

\begin{lem}\label{locally analytic extension of derivation}

There exists a unique continuous, $G(\mathbb{Q}_p)$-equivariant and $\mathcal{O}_{\Fl}$-linear derivation $d^{\Psi-\mathrm{la}} : \mathcal{O}_{K^p}^{\Psi-\mathrm{la}, (0, 0)_{\tau}} \rightarrow \mathcal{O}_{K^p}^{\Psi-\mathrm{la}, (0, 0)_{\tau}} \otimes_{\mathcal{O}_{K^p}^{\mathrm{sm}}} \Omega_{K^p}^{1, \mathrm{sm}}$ satisfying that $d^{\Psi-\mathrm{la}}|_{\mathcal{O}_{K^p}^{\mathrm{sm}}}$ is equal to the usual derivation $d^{\mathrm{sm}}$ on finite levels.

\end{lem}

\begin{proof}

The uniqueness of $d^{\Psi-\mathrm{la}}$ is clear from the explicit description in Corollary \ref{citation}. In the following, we use the notations of {\S} 3.3.

For any element $U$ of $\tilde{\mathcal{B}}$ such that $U$ is a rational subset of $U_1$, we have a continuous derivation $d^{\Psi-\mathrm{la}} : \mathcal{O}_{\aS_{K^pG_{r(n)}}}(V_{G_{r(n)}})\lbrace x_{\tau} - x_{\tau, n} \rbrace_{\tau} \rightarrow \mathcal{O}_{\aS_{K^pG_{r(n)}}}(V_{G_{r(n)}})\lbrace x_{\tau} - x_{\tau, n} \rbrace_{\tau} \otimes_{\mathcal{O}(V_{G_{r(n)}})} \Omega^1(V_{G_{r(n)}}), \ \sum_{i = (i_{\tau}) \in \mathbb{Z}_{\ge 0}^{\Psi}} c_{i} \prod_{\tau}(x_{\tau} - x_{\tau, n})^{i_{\tau}} \mapsto$ $$\sum_{i = (i_{\tau}) \in \mathbb{Z}_{\ge 0}^{\Psi}} (d^{\mathrm{sm}}(c_{i}) \prod_{\tau}(x_{\tau} - x_{\tau, n})^{i_{\tau}} - \sum_{\sigma}i_{\sigma}d^{\mathrm{sm}}(x_{\sigma, n})c_{i} (x_{\sigma} - x_{\sigma, n})^{i_{\sigma}-1}\prod_{\tau \neq \sigma}(x_{\tau} - x_{\tau, n})^{i_{\tau}}),$$ where $c_i \in \mathcal{O}_{\aS_{K^pG_{r(n)}}}(V_{G_{r(n)}})$ satisfies $\mathrm{sup}_i||c_i||p^{-n(\sum_{\tau}i_{\tau})} < \infty$. By Corollary \ref{citation}, this induces a continuous derivation $\mathcal{O}_{K^p}^{\Psi-\mathrm{la}, (0, 0)_{\tau}}(U) \rightarrow \mathcal{O}_{K^p}^{\Psi-\mathrm{la}, (0, 0)_{\tau}} \otimes_{\mathcal{O}_{K^p}^{\mathrm{sm}}} \Omega_{K^p}^{1, \mathrm{sm}}(U)$ characterized by $$\sum_{i = (i_{\tau}) \in \mathbb{Z}_{\ge 0}^{\Psi}} c_{i} \prod_{\tau}x_{\tau}^{i_{\tau}} \mapsto \sum_{i = (i_{\tau}) \in \mathbb{Z}_{\ge 0}^{\Psi}} d^{\mathrm{sm}}(c_{i}) \prod_{\tau}x_{\tau}^{i_{\tau}}$$ on the dense subspace $\{ \sum_{i = (i_{\tau}) \in \mathbb{Z}_{\ge 0}^{\Psi}} c_{i} \prod_{\tau}x_{\tau}^{i_{\tau}} \mid c_{i} \in \mathcal{O}_{K^p}^{\mathrm{sm}}(U) \ \mathrm{s.t.} \ c_{i} = 0 \ \mathrm{for \ almost \ all} \ i \}$ of $\mathcal{O}_{K^p}^{\Psi-\mathrm{la}, (0, 0)}(U)$. Note that this is $\mathcal{O}_{\Fl}(U)$-linear because this is $\mathcal{O}_{\Fl}(U_1)$-linear and $U$ is a rational subset of $U_1$.

Moreover, from the above characterization and the $G(\mathbb{Q}_p)$-equivariance of $d^{\mathrm{sm}}$, the above map is uniquely extended to a $G(\mathbb{Q}_p)$-equivariant $\mathcal{O}_{\Fl}$-linear map $d^{\Psi-\mathrm{la}} : \mathcal{O}_{K^p}^{\Psi-\mathrm{la}, (0,0)_{\tau}} \rightarrow \mathcal{O}_{K^p}^{\Psi-\mathrm{la}, (0,0)_{\tau}} \otimes_{\mathcal{O}_{K^p}^{\mathrm{sm}}} \Omega_{K^p}^{1, \mathrm{sm}}$ satisfying that $d^{\Psi-\mathrm{la}}|_{\mathcal{O}_{K^p}^{\mathrm{sm}}} = d^{\mathrm{sm}}$. \end{proof}

By Lemma \ref{locally analytic extension of derivation}, we can define a locally analytic extension $(D_{\lambda, K^p}^{\Psi-\mathrm{la}, (0, 0)_{\tau}} \otimes_{\mathcal{O}_{K^p}^{\mathrm{sm}}} \Omega_{K^p}^{\bullet, \mathrm{sm}}, \nabla^{\Psi-\mathrm{la}}_{\bullet})$ of the usual de Rham complex $(D_{\lambda, K^p}^{\mathrm{sm}} \otimes_{\mathcal{O}_{K^p}^{\mathrm{sm}}} \Omega_{K^p}^{\bullet, \mathrm{sm}}, \nabla^{\mathrm{sm}}_{\bullet})$ of finite levels by the formula $\nabla^{\Psi-\mathrm{la}}_n : D_{\lambda, K^p}^{\Psi-\mathrm{la}, (0, 0)_{\tau}} \otimes_{\mathcal{O}_{K^p}^{\mathrm{sm}}} \Omega_{K^p}^{n, \mathrm{sm}} = \mathcal{O}_{K^p}^{\Psi-\mathrm{la}, (0, 0)_{\tau}} \otimes_{\mathcal{O}_{K^p}^{\mathrm{sm}}} D_{\lambda, K^p}^{\mathrm{sm}} \otimes_{\mathcal{O}_{K^p}^{\mathrm{sm}}} \Omega_{K^p}^{n, \mathrm{sm}} \rightarrow \mathcal{O}_{K^p}^{\Psi-\mathrm{la}, (0, 0)_{\tau}} \otimes_{\mathcal{O}_{K^p}^{\mathrm{sm}}} D_{\lambda, K^p}^{\mathrm{sm}} \otimes_{\mathcal{O}_{K^p}^{\mathrm{sm}}} \Omega_{K^p}^{n + 1, \mathrm{sm}} = D_{\lambda, K^p}^{\Psi-\mathrm{la}, (0, 0)_{\tau}} \otimes_{\mathcal{O}_{K^p}^{\mathrm{sm}}} \Omega_{K^p}^{n+1, \mathrm{sm}}, \ f \otimes v \otimes w \mapsto f \otimes v \otimes d(w) + (-1)^n(d^{\Psi-\mathrm{la}}(f) \otimes v + f \otimes \nabla^{\mathrm{sm}}(v)) \wedge w$. This is actually shown to be a complex by considering a dense subspace as in the proof of Lemma \ref{locally analytic extension of derivation}.

Moreover, we have the following variant of Lemma \ref{constdeRham}. 

For $I, J \subset \Psi$ such that $|I| = n$ and $|J| = n+1$, by using the natural inclusion from the $I$-component $i_{I} : D_{\lambda^{\Psi}, K^p}^{\Psi-\mathrm{la}} \otimes (\otimes_{\tau \in \Psi} \mathrm{Sym}^{\lambda_{\tau}}D_{\tau, K^p}^{\vee, \mathrm{sm}}) \otimes_{\mathcal{O}_{K^p}^{\mathrm{sm}}} (\otimes_{\tau \in I} (\omega_{\tau, K^p}^{2, \mathrm{sm}} \otimes (\wedge^2D_{\tau, K^p}^{\mathrm{sm}}))) \rightarrow D_{\lambda^{\Psi}, K^p}^{\Psi-\mathrm{la}, (0, 0)_{\tau}} \otimes (\otimes_{\tau \in \Psi} \mathrm{Sym}^{\lambda_{\tau}}D_{\tau, K^p}^{\vee, \mathrm{sm}}) \otimes_{\mathcal{O}_{K^p}^{\mathrm{sm}}} \Omega_{K^p}^{n, \mathrm{sm}}$ and the natural projection to the $J$-component $j_J : D_{\lambda^{\Psi}, K^p}^{\Psi-\mathrm{la}, (0, 0)_{\tau}} \otimes (\otimes_{\tau \in \Psi} \mathrm{Sym}^{\lambda_{\tau}}D_{\tau, K^p}^{\vee, \mathrm{sm}}) \otimes_{\mathcal{O}_{K^p}^{\mathrm{sm}}} \Omega_{K^p}^{n+1, \mathrm{sm}} \rightarrow D_{\lambda^{\Psi}, K^p}^{\Psi-\mathrm{la}, (0, 0)_{\tau}} \otimes (\otimes_{\tau \in \Psi} \mathrm{Sym}^{\lambda_{\tau}}D_{\tau, K^p}^{\vee, \mathrm{sm}}) \otimes_{\mathcal{O}_{K^p}} (\otimes_{\tau \in J} (\omega_{\tau, K^p}^{2, \mathrm{sm}} \otimes (\wedge^2D_{\tau, K^p}^{\mathrm{sm}})))$, we obtain a map $\nabla_{n, I, J}^{\Psi-\mathrm{la}} := j_{J} \circ \nabla_n \circ i_{I} : D_{\lambda^{\Psi}, K^p}^{\Psi-\mathrm{la}, (0, 0)_{\tau}} \otimes (\otimes_{\tau \in \Psi} \mathrm{Sym}^{\lambda_{\tau}}D_{\tau, K^p}^{\vee, \mathrm{sm}}) \otimes (\otimes_{\tau \in I} (\omega_{\tau, K^p}^{2, \mathrm{sm}} \otimes (\wedge^2D_{\tau, K^p}^{\mathrm{sm}}))) \rightarrow D_{\lambda^{\Psi}, K^p}^{\Psi-\mathrm{la}, (0, 0)_{\tau}} \otimes (\otimes_{\tau \in \Psi} \mathrm{Sym}^{\lambda_{\tau}}D_{\tau, K^p}^{\vee, \mathrm{sm}}) \otimes_{\mathcal{O}_{K^p}} (\otimes_{\tau \in J} (\omega_{\tau, K^p}^{2, \mathrm{sm}} \otimes (\wedge^2D_{\tau, K^p}^{\mathrm{sm}}))).$

Note that we have a subquotient $D_{\lambda^{\Psi}, K^p}^{\Psi-\mathrm{la}, (0, 0)_{\tau}} \otimes (\otimes_{\tau \in I} (\omega_{\tau, K^p}^{\lambda_{\tau}+2, \mathrm{sm}} \otimes (\wedge^2 D_{\tau, K^p}^{\mathrm{sm}}))) \otimes (\otimes_{\tau \notin I} (\mathcal{\omega}_{\tau, K^p}^{-\lambda_{\tau}, \mathrm{sm}} \otimes (\wedge^2D_{\tau, K^p}^{\mathrm{sm}})^{-\lambda_{\tau}})) \twoheadleftarrow D_{\lambda^{\Psi}, K^p}^{\Psi-\mathrm{la}, (0, 0)_{\tau}} \otimes (\otimes_{\tau \in I} (\omega_{\tau, K^p}^{\lambda_{\tau}+2, \mathrm{sm}} \otimes (\wedge^2 D_{\tau, K^p}^{\mathrm{sm}}))) \otimes (\otimes_{\tau \notin I} \mathrm{Sym}^{\lambda_{\tau}}D_{\tau, K^p}^{\vee, \mathrm{sm}}) \hookrightarrow D_{\lambda^{\Psi}, K^p}^{\Psi-\mathrm{la}, (0, 0)_{\tau}} \otimes (\otimes_{\tau \in \Psi} \mathrm{Sym}^{\lambda_{\tau}}D_{\tau, K^p}^{\vee, \mathrm{sm}}) \otimes (\otimes_{\tau \in I} (\omega_{\tau, K^p}^{2, \mathrm{sm}} \otimes (\wedge^2D_{\tau, K^p}^{\mathrm{sm}})))$. In the following, we will construct a $\mathcal{O}_{\Fl}$-linear section $s_I^{\Psi-\mathrm{la}}$ of this quotient map for any $I$ which is an extension of $s_I$. 

%Note that if such a section exists, that is unique and we have $\mathrm{Im}(\nabla_{n, I, J} \circ s_{I}) \subset \mathrm{Im}(s_{J})$ by considering the dense subspace as in the proof of Lemma \ref{locally analytic extension of derivation} and $s_I^{\Psi-\mathrm{la}}$'s induce a complex as in the case $s_{I}$'s.

\begin{lem}\label{constdeRham2}

Let $I$ and $J$ as above. Then we have the following.

(1) \  $\nabla_{n, I, J}^{\Psi-\mathrm{la}}$ is zero unless $J = I \cup \{ \sigma \}$ for some $\sigma \in \Psi \setminus I$.

\vspace{0.5 \baselineskip}

In the following, we also assume that $J = I \cup \{ \sigma \}$ for some $\sigma \in \Psi \setminus I$.

(2) \ $\nabla_{n, I, J}^{\Psi-\mathrm{la}}$ sends the subspace $D_{\lambda^{\Psi}, K^p}^{\Psi-\mathrm{la}, (0, 0)_{\tau}} \otimes (\otimes_{\tau \in I} (\omega_{\tau, K^p}^{\lambda_{\tau}+2, \mathrm{sm}} \otimes (\wedge^2 D_{\tau, K^p}^{\mathrm{sm}}))) \otimes (\otimes_{\tau \notin I} \mathrm{Sym}^{\lambda_{\tau}}D_{\tau, K^p}^{\vee, \mathrm{sm}})$ to the subspace $D_{\lambda^{\Psi}, K^p}^{\Psi-\mathrm{la}, (0, 0)_{\tau}} \otimes (\otimes_{\tau \in I} (\omega_{\tau, K^p}^{\lambda_{\tau}+2, \mathrm{sm}} \otimes (\wedge^2 D_{\tau, K^p}^{\mathrm{sm}}))) \otimes (\otimes_{\tau \notin I} \mathrm{Sym}^{\lambda_{\tau}}D_{\tau, K^p}^{\vee, \mathrm{sm}}) \otimes (\omega_{\sigma, K^p}^{2, \mathrm{sm}} \otimes (\wedge^2D_{\sigma, K^p}^{\mathrm{sm}}))$.

In the following, we fix $M \subset \Psi \setminus J$. Actually, we have the following stronger result.

(3) \ $\nabla_{n, I, J}^{\Psi-\mathrm{la}}$ sends the subspace $D_{\lambda^{\Psi}, K^p}^{\Psi-\mathrm{la}, (0, 0)_{\tau}} \otimes (\otimes_{\tau \in I} (\omega_{\tau, K^p}^{\lambda_{\tau}+2, \mathrm{sm}} \otimes (\wedge^2 D_{\tau, K^p}^{\mathrm{sm}}))) \otimes (\otimes_{\tau \notin I \cup M} \mathrm{Sym}^{\lambda_{\tau}}D_{\tau, K^p}^{\vee, \mathrm{sm}}) \otimes (\otimes_{\tau \in M} \mathrm{Fil}^1(\mathrm{Sym}^{\lambda_{\tau}}D_{\tau, K^p}^{\vee, \mathrm{sm}}))$ to the subspace $D_{\lambda^{\Psi}, K^p}^{\Psi-\mathrm{la}, (0, 0)_{\tau}} \otimes (\otimes_{\tau \in I} (\omega_{\tau, K^p}^{\lambda_{\tau}+2, \mathrm{sm}} \otimes (\wedge^2 D_{\tau, K^p}^{\mathrm{sm}}))) \otimes (\otimes_{\tau \notin I \cup M} \mathrm{Sym}^{\lambda_{\tau}}D_{\tau, K^p}^{\vee, \mathrm{sm}}) \otimes (\otimes_{\tau \in M} \mathrm{Fil}^1(\mathrm{Sym}^{\lambda_{\tau}}D_{\tau, K^p}^{\vee, \mathrm{sm}})) \otimes (\omega_{\sigma, K^p}^{2, \mathrm{sm}} \otimes (\wedge^2D_{\sigma, K^p}^{\mathrm{sm}}))$.

(4) \ $\nabla_{n, I, J, M}^{\Psi-\mathrm{la}'} : D_{\lambda^{\Psi}, K^p}^{\Psi-\mathrm{la}, (0, 0)_{\tau}} \otimes (\otimes_{\tau \in I} (\omega_{\tau, K^p}^{\lambda_{\tau}+2, \mathrm{sm}} \otimes (\wedge^2 D_{\tau, K^p}^{\mathrm{sm}}))) \otimes (\otimes_{\tau \notin I \cup M} \mathrm{Sym}^{\lambda_{\tau}}D_{\tau, K^p}^{\vee, \mathrm{sm}}) \otimes (\otimes_{\tau \in M} \omega_{\tau, K^p}^{-\lambda_{\tau}, \mathrm{sm}} \otimes (\wedge^2 D_{\tau, K^p}^{\mathrm{sm}})^{-\lambda_{\tau}}) \rightarrow D_{\lambda^{\Psi}, K^p}^{\Psi-\mathrm{la}, (0, 0)_{\tau}} \otimes (\otimes_{\tau \in I} (\omega_{\tau, K^p}^{\lambda_{\tau}+2, \mathrm{sm}} \otimes (\wedge^2 D_{\tau, K^p}^{\mathrm{sm}}))) \otimes (\otimes_{\tau \notin I \cup M} \mathrm{Sym}^{\lambda_{\tau}}D_{\tau, K^p}^{\vee, \mathrm{sm}}) \otimes (\otimes_{\tau \in M} \omega_{\tau, K^p}^{-\lambda_{\tau}, \mathrm{sm}} \otimes (\wedge^2 D_{\tau, K^p}^{\mathrm{sm}})^{-\lambda_{\tau}}) \otimes (\omega_{\sigma, K^p}^{2, \mathrm{sm}} \otimes (\wedge^2D_{\sigma, K^p}^{\mathrm{sm}}))$ induced from (3) induces an isomorphism $D_{\lambda^{\Psi}, K^p}^{\Psi-\mathrm{la}, (0, 0)_{\tau}} \otimes (\otimes_{\tau \in I} (\omega_{\tau, K^p}^{\lambda_{\tau}+2, \mathrm{sm}} \otimes (\wedge^2 D_{\tau, K^p}^{\mathrm{sm}}))) \otimes \mathrm{Fil}^1(\mathrm{Sym}^{\lambda_{\sigma}}D_{\sigma, K^p}^{\vee, \mathrm{sm}}) \otimes (\otimes_{\tau \notin J \cup M} \mathrm{Sym}^{\lambda_{\tau}}D_{\tau, K^p}^{\vee, \mathrm{sm}}) \otimes (\otimes_{\tau \in M} \omega_{\tau, K^p}^{-\lambda_{\tau}, \mathrm{sm}} \otimes (\wedge^2 D_{\tau, K^p})^{-\lambda_{\tau}, \mathrm{sm}}) \Isom D_{\lambda^{\Psi}, K^p}^{\Psi-\mathrm{la}, (0 ,0)_{\tau}} \otimes (\otimes_{\tau \in I} (\omega_{\tau, K^p}^{\lambda_{\tau}+2, \mathrm{sm}} \otimes (\wedge^2 D_{\tau, K^p}^{\mathrm{sm}}))) \otimes (\mathrm{Sym}^{\lambda_{\sigma}}D_{\sigma, K^p}^{\vee, \mathrm{sm}}/ \mathrm{Fil}^{\lambda_{\sigma}}(\mathrm{Sym}^{\lambda_{\sigma}}D_{\sigma, K^p}^{\vee, \mathrm{sm}})) \otimes (\otimes_{\tau \notin J \cup M} \mathrm{Sym}^{\lambda_{\tau}}D_{\tau, K^p}^{\vee, \mathrm{sm}}) \otimes (\otimes_{\tau \in M} \omega_{\tau, K^p}^{-\lambda_{\tau}, \mathrm{sm}} \otimes (\wedge^2 D_{\tau, K^p}^{\mathrm{sm}})^{-\lambda_{\tau}}) \otimes (\omega_{\sigma, K^p}^{2, \mathrm{sm}} \otimes (\wedge^2D_{\sigma, K^p}^{\mathrm{sm}}))$. More strongly, the map $\nabla_{n, I, J, M}^{\Psi-\mathrm{la}'}$ induces the $\mathcal{O}_{K^p}$-linear isomorphisms of the following graded spaces for any $1 \le i \le \lambda_{\sigma}$. 

$D_{\lambda^{\Psi}, K^p}^{\Psi-\mathrm{la}, (0, 0)_{\tau}} \otimes (\otimes_{\tau \in I} (\omega_{\tau, K^p}^{\lambda_{\tau}+2, \mathrm{sm}} \otimes (\wedge^2 D_{\tau, K^p}^{\mathrm{sm}}))) \otimes \mathrm{gr}^i(\mathrm{Sym}^{\lambda_{\sigma}}D_{\sigma, K}^{\vee, \mathrm{sm}}) \otimes (\otimes_{\tau \notin J \cup M} \mathrm{Sym}^{\lambda_{\tau}}D_{\tau, K^p}^{\vee, \mathrm{sm}}) \otimes (\otimes_{\tau \in M} \omega_{\tau, K^p}^{-\lambda_{\tau}, \mathrm{sm}} \otimes (\wedge^2 D_{\tau, K^p}^{\mathrm{sm}})^{-\lambda_{\tau}}) \Isom D_{\lambda^{\Psi}, K^p}^{\Psi-\mathrm{la}, (0, 0)_{\tau}} \otimes (\otimes_{\tau \in I} (\omega_{\tau, K^p}^{\lambda_{\tau}+2, \mathrm{sm}} \otimes (\wedge^2 D_{\tau, K^p}^{\mathrm{sm}}))) \otimes \mathrm{gr}^{i-1}(\mathrm{Sym}^{\lambda_{\sigma}}D_{\sigma, K^p}^{\vee, \mathrm{sm}}) \otimes (\otimes_{\tau \notin J \cup M} \mathrm{Sym}^{\lambda_{\tau}}D_{\tau, K^p}^{\vee, \mathrm{sm}}) \otimes (\otimes_{\tau \in M} \omega_{\tau, K^p}^{-\lambda_{\tau}, \mathrm{sm}} \otimes (\wedge^2 D_{\tau, K^p}^{\mathrm{sm}})^{-\lambda_{\tau}}) \otimes (\omega_{\sigma, K^p}^{2, \mathrm{sm}} \otimes (\wedge^2D_{\sigma, K^p}^{\mathrm{sm}}))$.

\end{lem}

\begin{proof}

These results are consequences of Lemma \ref{constdeRham}. In fact, by using the dense subspaces as in the proof of Lemma \ref{locally analytic extension of derivation} and using (1), (2) and (3) of Lemma \ref{constdeRham}, the results (1), (2) and (3) follow from the following general fact : for a continuous morphism $f : X \rightarrow Y$ of Hausdorff topological groups, if $f$ is trivial on a dense subspace of $X$, then $f$ is trivial. In order to prove (4), first note that the given map induces the map of the following filtrations for any $1 \le i \le \lambda_{\sigma}$ : $D_{\lambda^{\Psi}, K^p}^{\Psi-\mathrm{la}, (0, 0)_{\tau}} \otimes (\otimes_{\tau \in I} (\omega_{\tau, K^p}^{\lambda_{\tau}+2, \mathrm{sm}} \otimes (\wedge^2 D_{\tau, K^p}^{\mathrm{sm}}))) \otimes \mathrm{Fil}^i(\mathrm{Sym}^{\lambda_{\sigma}}D_{\sigma, K}^{\vee, \mathrm{sm}}) \otimes (\otimes_{\tau \notin J \cup M} \mathrm{Sym}^{\lambda_{\tau}}D_{\tau, K^p}^{\vee, \mathrm{sm}}) \otimes (\otimes_{\tau \in M} \omega_{\tau, K^p}^{-\lambda_{\tau}, \mathrm{sm}} \otimes (\wedge^2 D_{\tau, K^p}^{\mathrm{sm}})^{-\lambda_{\tau}}) \rightarrow D_{\lambda^{\Psi}, K^p}^{\Psi-\mathrm{la}, (0, 0)_{\tau}} \otimes (\otimes_{\tau \in I} (\omega_{\tau, K^p}^{\lambda_{\tau}+2, \mathrm{sm}} \otimes (\wedge^2 D_{\tau, K^p}^{\mathrm{sm}}))) \otimes (\mathrm{Fil}^{i-1}(\mathrm{Sym}^{\lambda_{\sigma}}D_{\sigma, K^p}^{\vee, \mathrm{sm}})/\mathrm{Fil}^{\lambda_{\sigma}}(\mathrm{Sym}^{\lambda_{\sigma}}D_{\sigma, K^p}^{\vee, \mathrm{sm}})) \otimes (\otimes_{\tau \notin J \cup M} \mathrm{Sym}^{\lambda_{\tau}}D_{\tau, K^p}^{\vee, \mathrm{sm}}) \otimes (\otimes_{\tau \in M} \omega_{\tau, K^p}^{-\lambda_{\tau}, \mathrm{sm}} \otimes (\wedge^2 D_{\tau, K^p}^{\mathrm{sm}})^{-\lambda_{\tau}}) \otimes (\omega_{\sigma, K^p}^{2, \mathrm{sm}} \otimes (\wedge^2D_{\sigma, K^p}^{\mathrm{sm}}))$. In fact, (4) of Lemma \ref{constdeRham} and the same reasoning as in the proof of (1), (2) and (3) of this lemma imply this. Thus, we obtain the induced map on graded pieces $D_{\lambda^{\Psi}, K^p}^{\Psi-\mathrm{la}, (0, 0)_{\tau}} \otimes \mathrm{gr}^i\mathrm{Sym}^{\lambda_{\sigma}}D_{\sigma, K^p}^{\vee, \mathrm{sm}} \otimes (\otimes_{\gamma \in I} (\omega_{\gamma, K^p}^{\lambda_{\gamma}+2, \mathrm{sm}} \otimes (\wedge^2 D_{\gamma, K^p}^{\mathrm{sm}}))) \otimes (\otimes_{\gamma \notin J} (\omega_{\gamma, K^p}^{-\lambda_{\gamma}, \mathrm{sm}} \otimes (\wedge^2 D_{\gamma, K^p})^{-\lambda_{\gamma}, \mathrm{sm}})) \rightarrow D_{\lambda^{\Psi}, K^p}^{\Psi-\mathrm{la}, (0, 0)_{\tau}} \otimes (\mathrm{gr}^{i-1}\mathrm{Sym}^{\lambda_{\sigma}}D_{\sigma, K^p}^{\vee, \mathrm{sm}} \otimes (\otimes_{\gamma \in I} (\omega_{\gamma, K^p}^{\lambda_{\gamma}+2, \mathrm{sm}} \otimes (\wedge^2 D_{\gamma, K^p}^{\mathrm{sm}}))) \otimes (\otimes_{\gamma \notin J} (\omega_{\gamma, K^p}^{-\lambda_{\gamma}, \mathrm{sm}} \otimes (\wedge^2 D_{\gamma, K^p}^{\mathrm{sm}})^{-\lambda_{\gamma}})) \otimes (\omega_{\sigma, K^p}^{2, \mathrm{sm}} \otimes (\wedge^2 D_{\sigma, K^p}^{\mathrm{sm}}))$, which is equal to (the second map of (4) of Lemma \ref{constdeRham}) $\otimes_{\mathcal{O}_{K^p}^{\mathrm{sm}}} \mathcal{O}_{K^p}^{\Psi-\mathrm{la}, (0, 0)_{\tau}}$. Thus this is an isomorphism. \end{proof}

As after Lemma \ref{constdeRham}, by using Lemma \ref{constdeRham2}, we obtain the $G(\mathbb{Q}_p)$-equivariant $\mathcal{O}_{\Fl}$-linear section $s_I^{\Psi-\mathrm{la}} : D_{\lambda^{\Psi}, K^p}^{\Psi-\mathrm{la}, (0, 0)_{\tau}} \otimes (\otimes_{\tau \in I} (\omega_{\tau, K^p}^{\lambda_{\tau}+2, \mathrm{sm}} \otimes (\wedge^2 D_{\tau, K^p}^{\mathrm{sm}}))) \otimes (\otimes_{\tau \notin I} (\mathcal{\omega}_{\tau, K^p}^{-\lambda_{\tau}, \mathrm{sm}} \otimes (\wedge^2D_{\tau, K^p}^{\mathrm{sm}})^{-\lambda_{\tau}})) \hookrightarrow D_{\lambda^{\Psi}, K^p}^{\Psi-\mathrm{la}, (0, 0)_{\tau}} \otimes (\otimes_{\tau \in I} (\omega_{\tau, K^p}^{\lambda_{\tau}+2, \mathrm{sm}} \otimes (\wedge^2 D_{\tau, K^p}^{\mathrm{sm}}))) \otimes (\otimes_{\tau \notin I} \mathrm{Sym}^{\lambda_{\tau}}D_{\tau, K^p}^{\vee, \mathrm{sm}})$ which is an extension of $s_I$. Note that such a section is shown to be unique by considering the dense subspace as in the proof of Lemma \ref{locally analytic extension of derivation}. Thus, we have $\mathrm{Im}(\nabla_{n, I, J}^{\Psi-\mathrm{la}} \circ s_{I}^{\Psi-\mathrm{la}}) \subset \mathrm{Im}(s_{J}^{\Psi-\mathrm{la}})$ by considering the dense subspace as in the proof of Lemma \ref{locally analytic extension of derivation} and $s_I^{\Psi-\mathrm{la}}$'s induce the following complex, whose maps are $G(\mathbb{Q}_p)$-equivariant and $\mathcal{O}_{\Fl}$-linear.\footnote{The author thinks that the complex $GDR^{\Psi-\mathrm{la}'}_{\lambda}$ is quasi-isomorphic to $(D_{\lambda, K^p}^{\Psi-\mathrm{la}, (0, 0)} \otimes_{\mathcal{O}_{K^p}^{\mathrm{sm}}} \Omega_{K^p}^{\bullet, \mathrm{sm}}, \nabla^{\Psi-\mathrm{la}}_{\bullet})$. But we don't need such a result. Thus we don't pursue this problem in this paper.}
 
\begin{align*}GDR_{\lambda}^{\Psi-\mathrm{la}'} : D_{\lambda^{\Psi}, K^p}^{\Psi-\mathrm{la}, (0, 0)_{\tau}} \otimes  (\otimes_{\tau \in \Psi} (\mathcal{\omega}_{\tau, K^p}^{-\lambda_{\tau}, \mathrm{sm}} \otimes (\wedge^2 D_{\tau, K^p}^{\mathrm{sm}})^{-\lambda_{\tau}})) \rightarrow \\ \oplus_{\sigma \in \Psi} D_{\lambda^{\Psi}, K^p}^{\Psi-\mathrm{la}, (0, 0)_{\tau}} \otimes \omega_{\sigma, K^p}^{2\lambda_{\sigma}+2, \mathrm{sm}} \otimes (\wedge^2 D_{\sigma, K^p}^{\mathrm{sm}})^{\lambda_{\sigma}+1} \otimes  (\otimes_{\tau \in \Psi} (\mathcal{\omega}_{\tau, K^p}^{-\lambda_{\tau}, \mathrm{sm}} \otimes (\wedge^2 D_{\tau, K^p}^{\mathrm{sm}})^{-\lambda_{\tau}})) \rightarrow \\ \oplus_{I \subset \Psi, |I|=2} D_{\lambda^{\Psi}, K^p}^{\Psi-\mathrm{la}, (0, 0)_{\tau}} \otimes (\otimes_{\sigma \in I} (\omega_{\sigma, K^p}^{2\lambda_{\sigma}+2, \mathrm{sm}} \otimes (\wedge^2 D_{\sigma, K^p}^{\mathrm{sm}})^{\lambda_{\sigma}+1})) \otimes (\otimes_{\tau \in \Psi} (\mathcal{\omega}_{\tau, K^p}^{-\lambda_{\tau}, \mathrm{sm}} \otimes (\wedge^2 D_{\tau, K^p}^{\mathrm{sm}})^{-\lambda_{\tau}})) \rightarrow \\ \cdots \rightarrow  \\ D_{\lambda^{\Psi}, K^p}^{\Psi-\mathrm{la}, (0, 0)} \otimes (\otimes_{\tau \in \Psi} (\mathcal{\omega}_{\tau, K^p}^{2\lambda_{\tau} + 2, \mathrm{sm}} \otimes (\wedge^2 D_{\tau, K^p}^{\mathrm{sm}})^{\lambda_{\tau}+1})) \otimes  (\otimes_{\tau \in \Psi} (\mathcal{\omega}_{\tau, K^p}^{-\lambda_{\tau}, \mathrm{sm}} \otimes (\wedge^2 D_{\tau, K^p}^{\mathrm{sm}})^{-\lambda_{\tau}})).\end{align*} 

\vspace{0.5 \baselineskip}

By taking $\otimes_{\mathcal{O}_{\Fl}} (\otimes_{\tau \in \Phi} (\mathcal{\omega}_{\tau, \Fl}^{\lambda_{\tau}} \otimes (\wedge^2 V_{\tau})^{\lambda_{\tau}}))$ and using Lemma \ref{horizontal weight}, we obtain the following complex.
 
\begin{align*} GDR_{\lambda}^{\Psi-\mathrm{la}} : D_{\lambda^{\Psi}, K^p}^{\Psi-\mathrm{la}, (0, \lambda_{\tau})_{\tau}} \rightarrow \\ \oplus_{\sigma \in \Psi} D_{\lambda^{\Psi}, K^p}^{\Psi-\mathrm{la}, (0, \lambda_{\tau})_{\tau}} \otimes \omega_{\sigma, K^p}^{2\lambda_{\sigma}+2, \mathrm{sm}} \otimes (\wedge^2 D_{\sigma, K^p}^{\mathrm{sm}})^{\lambda_{\sigma}+1} \rightarrow \\ \oplus_{I \subset \Psi, |I| = 2} D_{\lambda^{\Psi}, K^p}^{\Psi-\mathrm{la}, (0, \lambda_{\tau})_{\tau}} \otimes (\otimes_{\sigma \in I} (\omega_{\sigma, K^p}^{2\lambda_{\sigma}+2, \mathrm{sm}} \otimes (\wedge^2 D_{\sigma, K^p}^{\mathrm{sm}})^{\lambda_{\sigma}+1})) \rightarrow \\ \cdots \rightarrow \\ D_{\lambda^{\Psi}, K^p}^{\Psi-\mathrm{la}, (0, \lambda_{\tau})} \otimes (\otimes_{\tau \in \Psi} (\mathcal{\omega}_{\tau, K^p}^{2\lambda_{\tau} + 2, \mathrm{sm}} \otimes (\wedge^2 D_{\tau, K^p}^{\mathrm{sm}})^{\lambda_{\tau}+1})). \end{align*} 

\begin{rem}\label{pullback de Rham}

If we consider locally on $\aS_{K^p}$ not on $\Fl$, exactly the same construction gives us a complex $GDR^{\Psi-\mathrm{la}}_{\lambda, \aS_{K^p}}$ such that $GDR^{\Psi-\mathrm{la}}_{\lambda} = {\pi_{\mathrm{HT}}}_*GDR^{\Psi-\mathrm{la}}_{\lambda, \aS_{K^p}}$. This will be used in {\S} 4.5.

\end{rem}

\subsection{Construction of geometric locally analytic anti-de Rham complexes}

On the other hand, by using the same method as \cite[{\S} 4.2]{PanII}, we also obtain a certain $\mathcal{O}_{K^p}^{\mathrm{sm}}$-linear extension of the de Rham complex on $\Fl$. 

\begin{lem}\label{extension anti}

For any $I \subset \Psi$ and $\sigma \in \Psi \setminus I$, there exists a unique continuous, $G(\mathbb{Q}_p)$-equivariant, $\mathcal{O}_{K^p}^{\mathrm{sm}}$-linear map $\overline{d}_{\sigma, I} : D_{\lambda^{\Psi}, K^p}^{\Psi-\mathrm{la}, (0, 0)_{\tau}} \otimes (\otimes_{\tau \in I} (\omega_{\tau, \Fl}^{-2} \otimes (\wedge^2V_{\tau})^{-1})) \rightarrow D_{\lambda^{\Psi}, K^p}^{\Psi-\mathrm{la}, (0, 0)_{\tau}} \otimes (\otimes_{\tau \in I} (\omega_{\tau, \Fl}^{-2} \otimes (\wedge^2V_{\tau})^{-1})) \otimes (\omega_{\sigma, \Fl}^{-2} \otimes (\wedge^2V_{\sigma})^{-1})$ such that for any $U \in \tilde{\mathcal{B}}$ such that $U \subset U_1$, we have the following formula on $U$. $\overline{d}_{\sigma, I}(\sum_{i \in (\mathbb{Z}_{\ge 0})^{\Psi}} a_{i} \prod_{\tau \in \Psi}(x_{\tau} - x_{\tau, n})^{i_{\tau}} \otimes (\otimes_{\tau \in I} dx_{\tau})) = (\sum_{i \in (\mathbb{Z}_{\ge 0})^{\Psi}} i_{\sigma}a_{i} (x_{\sigma} - x_{\sigma, n})^{i_{\sigma} - 1}\prod_{\tau \neq \sigma}(x_{\tau} - x_{\tau, n})^{i_{\tau}}) \otimes (\otimes_{\tau \in I} dx_{\tau}) \otimes dx_{\sigma}$. (Here, we use the description of Corollary \ref{citation}. Thus $a_i \in \mathcal{O}_{\aS_{K^pK_p}}(V_{K_p})$ for some $K_p$ and affinoid open $V_{K_p}$ of $\aS_{K^pK_p}$ such that $\pi_{\mathrm{HT}}^{-1}(U) = \pi_{K_p}^{-1}(V_{K_p})$. Note that by the same proof as Proposition \ref{KS}, we have $\Omega_{\sigma, \Fl}^1 \cong \omega_{\sigma, \Fl}^{-2} \otimes (\wedge^2 V_{\sigma})^{-1}$ for the pullback $\Omega_{\sigma, \Fl}^1$ of the sheaf of differential forms on $\mathbb{P}^1_{C}$ via the projection map $\Fl = \prod_{\tau \in \Psi} \mathbb{P}^1_{C} \rightarrow \mathbb{P}^1_{C}$ to the $\sigma$-component.)
    
\end{lem}

\begin{proof} The uniqueness follows from the above formula and the $G(\mathbb{Q}_p)$-equivariance. For any $U \in \tilde{\mathcal{B}}$ such that $U \subset U_1$, let $\overline{d}_{\sigma, I, U} : D_{\lambda^{\Psi}, K^p}^{\Psi-\mathrm{la}, (0, 0)_{\tau}} \otimes (\otimes_{\tau \in I} (\omega_{\tau, \Fl}^{-2} \otimes (\wedge^2V_{\tau})^{-1}))|_{U} \rightarrow D_{\lambda^{\Psi}, K^p}^{\Psi-\mathrm{la}, (0, 0)_{\tau}} \otimes (\otimes_{\tau \in I} (\omega_{\tau, \Fl}^{-2} \otimes (\wedge^2V_{\tau})^{-1})) \otimes (\omega_{\sigma, \Fl}^{-2} \otimes (\wedge^2V_{\sigma})^{-1})|_{U}$ be the map defined by the above formula. In order to extend $\overline{d}_{\sigma, I, U}$ to the $G(\mathbb{Q}_p)$-equivariant map $\overline{d}_{\sigma, I}$ on $\Fl$, it suffices to prove that for any $g \in G(\mathbb{Q}_p)$ such that $Ug \subset U_1$, we have $\overline{d}_{\sigma, I, U} \circ g^* = g^* \circ \overline{d}_{\sigma, I, Ug}$. Let $w$ be the $p$-adic place of $F$ lying above $v$ satisfying $\sigma \in \mathrm{Hom}_{\mathbb{Q}_p}(F_w, L)$ and $g_w = \begin{pmatrix}
    a & b \\
    c & d \end{pmatrix} \in \mathrm{GL}_2(F_w)$ be the $w$-component of $g$. We may assume that $I$ is empty because we have $\overline{d}_{\sigma, I, U}(f \otimes (\otimes_{\tau \in I} dx_{\tau})) = \overline{d}_{\sigma, \emptyset, U}(f) \otimes (\otimes_{\tau \in I} dx_{\tau})$. By the definition of $\overline{d}_{\sigma, \emptyset, U}$, it suffices to prove $\overline{d}_{\sigma, \emptyset, U} \circ g^*(x_{\tau}) = g^* \circ \overline{d}_{\sigma, \emptyset, Ug}(x_{\tau})$. The left hand side is $\overline{d}_{\sigma, \emptyset, U} \circ g^*(x_{\tau}) = \overline{d}_{\sigma, \emptyset, U}(\frac{dx_{\tau} + b}{cx_{\tau} + a}) = \frac{\mathrm{det}g}{(cx_{\tau}+a)^2} \otimes dx_{\tau}$. On the other hand, the right-hand side is $g^* \circ \overline{d}_{\sigma, \emptyset, Ug}(x_{\tau}) = 1 \otimes d(\frac{dx_{\tau} + b}{cx_{\tau} + a}) = \frac{\mathrm{det}g}{(cx_{\tau}+a)^2} \otimes dx_{\tau}$. Thus we obtain the result. \end{proof}

We recall that we fixed a numbering $\tau_1, \cdots, \tau_d$ of $\Psi = \{ \tau_1, \cdots, \tau_d \}$. 

For $I, J \subset \Psi$ such that $|J| = |I| + 1$, we define a map $\overline{d}_{I,J} : D_{\lambda^{\Psi}, K^p}^{\Psi-\mathrm{la}, (0, 0)_{\tau}} \otimes (\otimes_{\tau \in I} (\omega_{\tau, \Fl}^{-2} \otimes (\wedge^2V_{\tau})^{-1})) \rightarrow D_{\lambda^{\Psi}, K^p}^{\Psi-\mathrm{la}, (0, 0)_{\tau}} \otimes (\otimes_{\tau \in J} (\omega_{\tau, \Fl}^{-2} \otimes (\wedge^2V_{\tau})^{-1}))$ by the following.

\begin{equation*} \overline{d}_{I, J} := \begin{cases} 0 \ \mathrm{unless} \ J = I \cup \{ \sigma \} \ for \ some \ \sigma \in \Psi \setminus I. \\ (-1)^{|\{ \tau_i \in I \mid i < j \}|}\overline{d}_{I, \sigma} \ \mathrm{if} \ J = I \cup \{ \tau_j \} \ for \ some \ \tau_j \in \Psi \setminus I. \end{cases} \end{equation*}

By using these maps, we obtain the following sequence of maps, which is a complex by the explicit formula of $\overline{d}_{\sigma, I}$.

\begin{align*}\label{anticomplex}
\overline{GDR}^{\Psi-\mathrm{la}}_0 : D_{\lambda^{\Psi}, K^p}^{\Psi-\mathrm{la}, (0, 0)_{\tau}} \rightarrow \\ \oplus_{\tau \in \Psi} D_{\lambda^{\Psi}, K^p}^{\Psi-\mathrm{la}, (0, 0)} \otimes \omega_{\tau, \Fl}^{-2} \otimes (\wedge^2 V_{\tau})^{-1} \rightarrow \\ \oplus_{I \subset \Psi, |I| = 2} D_{\lambda^{\Psi}, K^p}^{\Psi-\mathrm{la}, (0, 0)} \otimes (\otimes_{\tau \in I} (\omega_{\tau, \Fl}^{-2} \otimes (\wedge^2 V_{\tau})^{-1})) \rightarrow \\ \cdots \rightarrow \\ D_{\lambda^{\Psi}, K^p}^{\Psi-\mathrm{la}, (0, 0)_{\tau}} \otimes (\otimes_{\tau \in \Psi} (\omega_{\tau, \Fl}^{-2} \otimes (\wedge^2 V_{\tau})^{-1}))\end{align*}

\begin{lem}\label{BGG flag}

Let $\chi_{\lambda_{\Psi}} : Z(U(\mathfrak{g}_{\Psi})) \rightarrow L$ be the infinitesimal character of $V_{\lambda^{\Psi}}^{\vee}$ and let $I \subset \Psi$. Then $D_{\lambda^{\Psi}, K^p}^{\Psi-\mathrm{la}, (0, 0)_{\tau}} \otimes (\otimes_{\tau \notin I} \mathrm{Sym}^{\lambda_{\tau}}V_{\tau}) \otimes (\otimes_{\tau \in I} \mathrm{Sym}^{\lambda_{\tau}}V_{\tau}) \otimes (\otimes_{\tau \in I} (\omega_{\tau, \Fl}^{-2} \otimes (\wedge^2 V_{\tau})^{-1}))$ has a generalized eigenspace decomposition with respect to the action of $Z(U(\mathfrak{g}_{\Psi}))$ and the $\chi_{\lambda_{\Psi}}$-generalized eigenspace is identified with $D_{\lambda^{\Psi}, K^p}^{\Psi-\mathrm{la}, (0, 0)_{\tau}}  \otimes (\otimes_{\tau \notin I} (\omega_{\tau, \Fl}^{\lambda_{\tau}} \otimes (\wedge^2 V_{\tau})^{\lambda_{\tau}})) \otimes( \otimes_{\tau \in I} (\omega_{\tau, \Fl}^{-\lambda_{\tau} - 2} \otimes (\wedge^2 V_{\tau})^{-1})) = D_{\lambda^{\Psi}, K^p}^{\Psi-\mathrm{la}, (0, 0)_{\tau}} \otimes( \otimes_{\tau \in I} (\omega_{\tau, \Fl}^{-2\lambda_{\tau} - 2} \otimes (\wedge^2 V_{\tau})^{-\lambda_{\tau}-1})) \otimes (\otimes_{\tau \in \Psi} (\omega_{\tau, \Fl}^{\lambda_{\tau}} \otimes (\wedge^2 V_{\tau})^{\lambda_{\tau}}))$.
    
\end{lem}

\begin{proof}

Note that $D_{\lambda^{\Psi}, K^p}^{\Psi-\mathrm{la}, (0, 0)_{\tau}} \otimes (\otimes_{\tau \notin I} \mathrm{Sym}^{\lambda_{\tau}}V_{\tau}) \otimes (\otimes_{\tau \in I} \mathrm{Sym}^{\lambda_{\tau}}V_{\tau}) \otimes (\otimes_{\tau \in I} (\omega_{\tau, \Fl}^{-2} \otimes (\wedge^2 V_{\tau})^{-1}))$ has a filtration whose graded pieces are given by $D_{\lambda^{\Psi}, K^p}^{\Psi-\mathrm{la}, (0, 0)_{\tau}} \otimes (\otimes_{\tau \notin I} (\omega_{\tau, \Fl}^{\lambda_{\tau}-2i_{\tau}} \otimes (\wedge^2 V_{\tau})^{\lambda_{\tau}-i_{\tau}})) \otimes (\otimes_{\tau \in I} (\omega_{\tau, \Fl}^{\lambda_{\tau}-2i_{\tau}- 2} \otimes (\wedge^2 V_{\tau})^{\lambda_{\tau}-i_{\tau}-1}))$. This is equal to $D_{\lambda^{\Psi}, K^p}^{\Psi-\mathrm{la}, (i_{\tau}, \lambda_{\tau} - i_{\tau})_{\tau \notin I}, (i_{\tau} + 1, \lambda_{\tau} - i_{\tau} - 1)_{\tau \in I}} \otimes (\otimes_{\tau \notin I} (\omega_{\tau, K^p}^{\lambda_{\tau}-2i_{\tau}, \mathrm{sm}}(-\lambda_{\tau}+2i_{\tau}) \otimes (\wedge^2 D_{\tau, K^p}(1))^{\lambda_{\tau} - i_{\tau}, \mathrm{sm}})) \otimes (\otimes_{\tau \in I} (\omega_{\tau, K^p}^{\lambda_{\tau}-2i_{\tau}- 2, \mathrm{sm}}(-\lambda_{\tau} +2i_{\tau} + 2) \otimes (\wedge^2 D_{\tau, K^p}^{\mathrm{sm}}(1))^{\lambda_{\tau}- i_{\tau} - 1}))$ for $0 \le i_{\tau} \le \lambda_{\tau}$ by Lemma \ref{horizontal weight}. By Lemma \ref{infinitesimal character and horizontal}, $D_{\lambda^{\Psi}, K^p}^{\Psi-\mathrm{la}, (i_{\tau}, \lambda_{\tau} - i_{\tau})_{\tau \notin I}, (i_{\tau} + 1, \lambda_{\tau} - i_{\tau} - 1)_{\tau \in I}} \otimes (\otimes_{\tau \notin I} (\omega_{\tau, K^p}^{\lambda_{\tau}-2i_{\tau}, \mathrm{sm}} \otimes (\wedge^2 D_{\tau, K^p}(1))^{\lambda_{\tau} - i_{\tau}, \mathrm{sm}})) \otimes (\otimes_{\tau \in I} (\omega_{\tau, K^p}^{\lambda_{\tau}-2i_{\tau}- 2, \mathrm{sm}} \otimes (\wedge^2 D_{\tau, K^p}^{\mathrm{sm}}(1))^{\lambda_{\tau} - i_{\tau} -1}))$ has an infinitesimal character and the characters $(a_{\tau}, b_{\tau})$ and $(c_{\tau}, d_{\tau})$ of $\mathrm{Sym}\mathfrak{h}_{\tau}$ give the same infinitesimal character if and only if $(a_{\tau}, b_{\tau}) = (c_{\tau}, d_{\tau})$ or $(a_{\tau}, b_{\tau}) = (d_{\tau} + 1, c_{\tau} - 1)$. Thus the multiplicity of the infinitesimal characters given by $((0, \lambda_{\tau})_{\tau \in I}, (\lambda_{\tau} + 1, -1)_{\tau \in I})$ is one. \end{proof}

By taking $\otimes (\otimes_{\tau \in \Psi} \mathrm{Sym}^{\lambda_{\tau}} V_{\tau})$, the $\chi_{\lambda_{\Psi}}$-isotypic parts and $\otimes_{\mathcal{O}_{K^p}^{ \mathrm{sm}}} (\otimes_{\tau \in \Psi} \omega_{\tau, K^p}^{-\lambda_{\tau}, \mathrm{sm}} \otimes (\wedge^2 D_{\tau, K^p}^{\mathrm{sm}})^{-\lambda_{\tau}})$, the above complex $\overline{GDR}^{\Psi-\mathrm{la}}_0$ induces the following complex, whose maps are continuous, $G(\mathbb{Q}_p)$-equivariant and $\mathcal{O}_{K^p}^{\mathrm{sm}}$-linear: \footnote{Note that we use the identification Lemma \ref{horizontal weight}.}
 
\begin{align*} \overline{GDR}_{\lambda}^{\Psi-\mathrm{la}} : D_{\lambda^{\Psi}, K^p}^{\Psi-\mathrm{la}, (0, \lambda_{\tau})_{\tau}} \rightarrow \\ \oplus_{\sigma \in \Psi} D_{\lambda^{\Psi}, K^p}^{\Psi-\mathrm{la}, (0, \lambda_{\tau})_{\tau \neq \sigma}, (1 + \lambda_{\sigma}, -1)} \otimes \omega_{\sigma, K^p}^{-2\lambda_{\sigma}-2, \mathrm{sm}}(\lambda_{\sigma} + 1) \otimes (\wedge^2 D_{\sigma, K^p}^{\mathrm{sm}})^{-\lambda_{\sigma}-1} \rightarrow \\ \oplus_{I \subset \Psi, |I| = 2} D_{\lambda^{\Psi}, K^p}^{\Psi-\mathrm{la}, (0, \lambda_{\tau})_{\tau \notin I}, (1+\lambda_{\tau}, -1)_{\tau \in I}} \otimes (\otimes_{\sigma \in I} (\omega_{\sigma, K^p}^{-2\lambda_{\sigma}-2, \mathrm{sm}}(\lambda_{\sigma}+1) \otimes (\wedge^2 D_{\sigma, K^p}^{\mathrm{sm}})^{-\lambda_{\sigma}-1})) \rightarrow \\ \cdots \rightarrow D_{\lambda^{\Psi}, K^p}^{\Psi-\mathrm{la}, (1+\lambda_{\tau}, -1)_{\tau}} \otimes (\otimes_{\tau \in \Phi} (\mathcal{\omega}_{\tau, K^p}^{-2\lambda_{\tau} - 2, \mathrm{sm}}(\lambda_{\tau}+1) \otimes (\wedge^2 D_{\tau, K^p}^{\mathrm{sm}})^{-\lambda_{\tau}-1})). \end{align*}

\begin{prop}\label{anti}

$\overline{GDR}^{\Psi-\mathrm{la}}_{\lambda}$ is quasi-isomorphic to $(\otimes_{\tau \in \Psi} \mathrm{Sym}^{\lambda_{\tau}} V_{\tau}) \otimes (\otimes_{\tau \in \Psi} \omega_{\tau, K^p}^{-\lambda_{\tau}, \mathrm{sm}} \otimes (\wedge^2 D_{\tau, K^p}^{\mathrm{sm}})^{-\lambda_{\tau}}) \otimes_{\mathcal{O}_{K^p}^{\mathrm{sm}}} \overline{GDR}^{\Psi-\mathrm{la}}_{0}$ and $H^i(\overline{GDR}_{\lambda}^{\Psi-\mathrm{la}}) = 0$ for any $i > 0$ and $H^0(\overline{GDR}_{\lambda}^{\Psi-\mathrm{la}}) = D_{\lambda^{\Psi}, K^p}^{\mathrm{sm}} \otimes_{\mathcal{O}_{K^p}^{\mathrm{sm}}} (\otimes_{\tau \in \Psi} \mathrm{Sym}^{\lambda_{\tau}}V_{\tau} \otimes_L \omega_{\tau, K^p}^{-\lambda_{\tau}, \mathrm{sm}} \otimes (\wedge^2 D_{\tau, K^p}^{\mathrm{sm}})^{-\lambda_{\tau}})$.
    
    \end{prop}
    
    \begin{proof}
    
    By the explicit formula of $\overline{d}_{I, \sigma}$ and the Poincare lemma (cf. \cite[Proposition 7.1]{HARTS}), we have $H^i(\overline{GDR}^{\Psi-\mathrm{la}'}_{0}) = 0$ for any $i >0$ and $H^0(\overline{GDR}^{\Psi-\mathrm{la}'}_{0}) = D_{\lambda^{\Psi}, K^p}^{\mathrm{sm}}$. This implies the result by the construction of $\overline{GDR}^{\Psi-\mathrm{la}}_{\lambda}$. (We can also prove this by using the following formula.) \end{proof}

\begin{lem}\label{anti-holomorphic} Let $I, J \subset \Psi$ and $\sigma \in \Psi \setminus I$ such that $J = I \cup \{ \sigma \}$. Then there exists $c \in \mathbb{Q}^{\times}$ such that $D_{\lambda^{\Psi}, K^p}^{\Psi-\mathrm{la}, (0, \lambda_{\tau})_{\tau \notin I}, (1+\lambda_{\tau}, -1)_{\tau \in I}} \otimes (\otimes_{\tau \in I} (\omega_{\sigma, K^p}^{-2\lambda_{\tau}-2, \mathrm{sm}}(\lambda_{\tau}+1) \otimes (\wedge^2 D_{\tau, K^p}^{\mathrm{sm}})^{-\lambda_{\tau}-1})) \rightarrow D_{\lambda^{\Psi}, K^p}^{\Psi-\mathrm{la}, (1+\lambda_{\tau}, -1)_{\tau \in J}, (0, \lambda_{\tau})_{\tau \notin J}} \otimes (\otimes_{\tau \in J} (\mathcal{\omega}_{\tau, K^p}^{-2\lambda_{\tau} - 2, \mathrm{sm}}(\lambda_{\tau}+1) \otimes (\wedge^2 D_{\tau, K^p}^{\mathrm{sm}})^{-\lambda_{\tau}-1}))$ in the complex $\overline{GDR}^{\Psi-\mathrm{la}}_{\lambda}$ is equal to the following map on $U \in \tilde{\mathcal{B}}$ such that $U \subset U_1$. (Here, we use the description of Proposition \ref{mikami expansionIII}. Thus $a_i \in (\otimes_{\tau \notin I} (\omega_{\tau, \aS_{K^pK_p}}^{-\lambda_{\tau}} \otimes (\wedge^2 D_{\tau, \aS_{K^pK_p}})^{-\lambda_{\tau}})(V_{K^pK_p})) \otimes (\otimes_{\tau \in I} (\omega_{\tau, \aS_{K^pK_p}}^{\lambda_{\tau} + 2} \otimes (\wedge^2 D_{\tau, \aS_{K^pK_p}}))(V_{K^pK_p}))$ for some $K_p$ and affinoid open $V_{K^pK_p}$ of $\aS_{K^pK_p}$.)

$(\prod_{\tau \in I}e_{1, \tau}^{\lambda_{\tau}})(\prod_{\tau \in I} f_{\tau}^{1+\lambda_{\tau}} e_{\tau}^{-2-\lambda_{\tau}})\sum_{i \in (\mathbb{Z}_{\ge 0})^{\Psi}} a_{i} \prod_{\tau \in \Psi}(x_{\tau} - x_{\tau, n})^{i_{\tau}} \otimes (\otimes_{\tau \in I} dx_{\tau}^{\lambda_{\tau} + 1})$

$\mapsto c(\prod_{\tau \in I}e_{1, \tau}^{\lambda_{\tau}})(\prod_{\tau \in I} f_{\tau}^{1+\lambda_{\tau}} e_{\tau}^{-2-\lambda_{\tau}})\sum_{i \in (\mathbb{Z}_{\ge 0})^{\Psi}} (\prod_{j = 0}^{\lambda_{\sigma} + 1}(i_{\sigma} - j))a_{i} (x_{\sigma} - x_{\sigma, n})^{i_{\sigma} - \lambda_{\sigma} - 1}\prod_{\tau \neq \sigma}(x_{\tau} - x_{\tau, n})^{i_{\tau}}) \otimes (\otimes_{\tau \in I} dx_{\tau}^{\lambda_{\tau} + 1}).$
 
\end{lem}

\begin{proof} Direct calculation as in the modular curve case \cite[Theorem 4.2.7]{PanII}.  \end{proof}

\begin{rem}

If we consider locally on $\aS_{K^p}$ not on $\Fl$, exactly the same construction gives us a complex $\overline{GDR}^{\Psi-\mathrm{la}}_{\lambda, \aS_{K^p}}$ such that $\overline{GDR}^{\Psi-\mathrm{la}}_{\lambda} = {\pi_{\mathrm{HT}}}_*\overline{GDR}^{\Psi-\mathrm{la}}_{\lambda, \aS_{K^p}}$. This will be used in {\S} 4.5.

\end{rem}

\subsection{Integral models and Newton stratifications}

In the rest of {\S} 4, we assume that $p$ is unramified in $F$. We need this assumption to obtain Theorem \ref{affinoid ordinary}. We recall that $v$ is the $p$-adic place of $F_0$ induced by $\iota : \overline{\mathbb{Q}}_p \stackrel{\sim}{\rightarrow} \mathbb{C}$.

As in \cite[{\S} 5]{PanII}, we will calculate $GDR^{\Psi-\mathrm{la}}_{\lambda}$ by using the Newton stratification on $S_{K}$. Let $S_0 := \{ w \mid w \mathrm{\ divides \ } v \ \mathrm{and} \ \mathrm{the} \ w\mathrm{-component} \ \mu_w \ \mathrm{of} \ \mu \mathrm{\ is \ trivial.} \}$, $K_p^o := \mathbb{Z}_p^{\times} \times \prod_{w \mid v} \mathrm{GL}_2(\mathcal{O}_{F_w}) \supset \Gamma_p((n_w)_{w \in S_0}) := \mathbb{Z}_p^{\times} \times \prod_{w \notin S_0} \mathrm{GL}_2(\mathcal{O}_{F_w}) \times \mathrm{Ker}(\prod_{w \in S_0} \mathrm{GL}_2(\mathcal{O}_{F_w}) \rightarrow \prod_{w \in S_0} \mathrm{GL}_2(\mathcal{O}_{F_w}/\varpi_w^{n_w}))$.  Let $\mathcal{O}_B$ be a $*$-stable order of $B$ containing $\mathcal{O}_F$ such that $\mathcal{O}_{B} \otimes_{\mathcal{O}_F} \mathcal{O}_{F_w} = M_2(\mathcal{O}_{F_w})$ for $w \mid v$ of $F$. Let $V^o_p$ be a $\mathcal{O}_{B_p} := \mathcal{O}_B \otimes_{\mathbb{Z}} \mathbb{Z}_p$-stable self-dual lattice of $V_p = V \otimes_{\mathbb{Q}} \mathbb{Q}_p$ and $V_{0}^o$ and $V_1^{o}$ be $\mathcal{O}_B \otimes_{\mathbb{Z}} \mathcal{O}_L$-stable lattices of $V_{0}$ and $V_1$ such that $V^o_p \otimes_{\mathbb{Z}_p} \mathcal{O}_L = V_0^o \oplus V_1^o$. (See {\S} 3.3 for the definitions of $V_0$ and $V_1$.)

\vspace{0.5 \baselineskip}

First, we recall the definition of integral models of $S_{K, L}$. 

Let $S_{K^p\Gamma_p((n_w)_{w \in S_0}), \mathcal{O}_L}$ be the smooth projective integral model over $\mathcal{O}_L$ of $S_{K^p\Gamma_p((n_w)_{w \in S_0}), L}$ defined in \cite[{\S} 5]{Kot}. (See also \cite[Lemma 4.1]{HT} and \cite[Theorems 5.2 and 5.8]{LK} for details. Note that for any $w \in S_0$, we consider the trivial Hodge structure. Thus we obtain a projective smooth model without assuming $n_w = 0$ for any $w \in S_0$.) This has the following moduli interpretation. For a connected noetherian scheme $T$ over $\mathcal{O}_L$ with a geometric point $t$, $S_{K^p\Gamma_p((n_w)_{w \in S_0}), \mathcal{O}_L}(T)$ parametrizes the equivalence classes $[(A, \lambda, i, \overline{\eta}^p, (\eta_{n_w})_{w \in S_0})]$ of tuples $(A, \lambda, i, \overline{\eta}^p, (\eta_{n_w})_{w \in S_0})$, where

\vspace{0.5 \baselineskip}

1 \ $A$ is an abelian scheme over $T$.

\vspace{0.5 \baselineskip}

2 \ $\lambda : A \rightarrow A^{\vee}$ is a prime-to-$p$ polarization.

\vspace{0.5 \baselineskip}

3 \ $i : \mathcal{O}_{B} \rightarrow \mathrm{End}(A) \otimes \mathbb{Z}_{(p)}$ is a ring morphism such that $i(b^*)^{\vee} \circ \lambda = \lambda \circ i(b)$.

\vspace{0.5 \baselineskip}

4 \ Zariski locally on $T$, the Lie algebra $\mathrm{Lie}A$ of $A$ is isomorphic to $V_1^o \otimes_{\mathcal{O}_L} \mathcal{O}_T$ as $\mathcal{O}_{B} \otimes_{\mathbb{Z}} \mathcal{O}_{T}$-modules.

\vspace{0.5 \baselineskip}

5 $\overline{\eta}^p$ is a $K^p$-level structure, i.e., a $\pi_1(T, t)$-stable $K^p$-orbit of isomorphisms of $B \otimes_{\mathbb{Q}} \mathbb{A}_{\mathbb{Q}}^{p, \infty}$-modules $\eta : V \otimes_{\mathbb{Q}} \mathbb{A}_{\mathbb{Q}}^{p, \infty} \Isom V^pA_{t}$ compatible with the pairings under some identification $g : \mathbb{A}_{\mathbb{Q}}^{p, \infty} \Isom \mathbb{A}_{\mathbb{Q}}^{p, \infty}(1)$ of $\mathbb{A}_{\mathbb{Q}}^{p, \infty}$-module (note that $g$ is determined by $f$), where $V^pA_{t} := (\varprojlim_{n \nmid p}A_{t}[n](k(t))) \otimes_{\widehat{\mathbb{Z}}^p} \mathbb{A}_{\mathbb{Q}}^{p, \infty}$.

\vspace{0.5 \baselineskip}

6 \ For any $w \in S_0$, $\eta_{n_w} : V_w^0/\varpi_w^{n_w} \Isom  A[\varpi_w^{n_w}]$ is an $\mathcal{O}_{B_{F_w}}$-equivariant isomorphism for any $w \mid v$.

\vspace{0.5 \baselineskip}

We say that $(A_1,\lambda_1, i_1, \overline{\eta}_1^p, (\eta_{n_w, 1})_{w \in S_0})$ and $(A_2,\lambda_2, i_2, \overline{\eta}_2^p, (\eta_{n_w, 2})_{w \in S_0})$ are equivalent if there exist a prime-to-$p$ isogeny $\alpha : A_1 \rightarrow A_2$ and $r \in \mathbb{Z}_{(p), >0}$ such that $\alpha \lambda_1 = r\lambda_2 \alpha$, $\alpha \circ i_1(b) = i_2(b) \circ \alpha$ for any $b \in B$, $\alpha \circ \eta_{n_w, 1} = \eta_{n_w, 2}$ for any $w \in S_0$, $\eta_2^{p, -1} \circ V(\alpha) \circ \eta_1^p$ is contained in the $K^p$-orbit of $\mathrm{id}$ and $g_2^{-1} \circ g_1$ is contained in the $\nu(K^p)$-orbit of $r$, where $\eta_1^p$ (resp. $\eta_2^p$) is any representative of $\overline{\eta}_1^p$ (resp. $\overline{\eta}_2^p$) and $g_1$ (resp. $g_2$) is an isomorphism $\mathbb{A}_{\mathbb{Q}}^{p, \infty} \Isom \mathbb{A}_{\mathbb{Q}}^{p, \infty}(1)$ defined by $\eta_1^p$ (resp. $\eta_2^p$) as in the above condition 5.

\begin{rem}

    Let $L'$ be a perfectoid field containing $L(\zeta_{p^{\infty}})$ and fix an isomorphism $\mu_{p^{\infty}} \Isom \mathbb{Q}_p/\mathbb{Z}_p$ over $L'$. Then $S_{K^p\Gamma_p((n_w)_{w \in S_0}), \mathcal{O}_{L'}} \times_{\mathcal{O}_{L'}} (\mathbb{Z}_p/p^{n_0})^{\times}$ is an integral model of $S_{K^p\Gamma_p(n_0, (n_w)_{w \in S_0}), \mathcal{O}_{L'}}$, where $\Gamma_p(n_0, (n_w)_{w \in S_0}) := \mathrm{Ker}(\mathbb{Z}_p^{\times} \rightarrow (\mathbb{Z}_p/p^{n_0})^{\times}) \times \prod_{w \notin S_0} \mathrm{GL}_2(\mathcal{O}_{F_w}) \times \mathrm{Ker}(\prod_{w \in S_0} \mathrm{GL}_2(\mathcal{O}_{F_w}) \rightarrow \prod_{w \in S_0} \mathrm{GL}_2(\mathcal{O}_{F_w}/\varpi_w^{n_w}))$. Here, we regard $(\mathbb{Z}_p/p^{n_0})^{\times}$ as a constant scheme over $\mathcal{O}_{L'}$.
    
    \end{rem}

We recall the definition of the Newton stratification on Shimura varieties.

\begin{dfn}\label{Newton stratification}

Let $M$ be a finite extension of $\mathbb{Q}_p$, $H$ be a connected reductive group over $M$, $\breve{M}$ be the completion of the maximal unramified extension of $M$ and $\sigma \in \mathrm{Aut}(\breve{M}/M)$ be the (lift of) arithmetic Frobenius.
 
1 \ We say that elements $g_1, g_2 \in H(\breve{M})$ are $\sigma$-conjugate if there exists $h \in H(\breve{M})$ such that $g_1 = hg_2\sigma(h)^{-1}$.

2 \ Let $B(H)$ be the set of $\sigma$-conjugacy classes of elements in $H(\breve{M})$.

3 \ For $b \in H(\breve{M})$, let $J_b$ be the algebraic group over $M$ defined by $R \mapsto \{ g \in H(R \otimes_{M} \breve{M}) \mid gb\sigma(g)^{-1} = b \}$.

4 \ For a cocharacter $\mu : \mathbb{G}_{m, \overline{M}} \rightarrow H_{\overline{M}}$, let $B(H, \mu)$ be the subset of $B(H)$ defined in \cite[{\S} 6]{Iso}. (See also \cite[{\S} 3.1]{CS}.)

\end{dfn}

\begin{rem}

$J_b$ is an inner form of a Levi subgroup of $H$ by \cite[{\S} 1.11]{RR} and its isomorphism class depends only on its class in $B(H)$.

\end{rem}

\begin{lem}\label{scalar restriction}

Let $M$ be a finite extension of $\mathbb{Q}_p$, $H$ be a connected reductive group over $M$ and put $G := \mathrm{Res}_{M/\mathbb{Q}_p}H$. Then the map $H(\breve{M}) \rightarrow G(\breve{\mathbb{Q}}_p) = \prod_{\mathrm{Hom}_{\mathbb{Q}_p}(M_0, \breve{M})}H(\breve{M}), \ g \mapsto (g, 1, \cdots, 1)$ induce a bijection $B(H) \Isom B(G)$, where $M_0$ denotes the maximal unramified subextension of $M/\mathbb{Q}_p$ and the first component corresponds to the natural inclusion $M_0 \hookrightarrow \breve{M}$.

\end{lem}

\begin{proof} See \cite[{\S} 1.6]{Iso}. \end{proof}

By \cite[Theorem 3.6]{RR}, we have a Newton stratification $S_{K^pK_p^{0}, \mathbb{F}_{L}} = \sqcup_{b \in B(G)} S_{K^pK_p^{o}, \mathbb{F}_{L}}^b$, where $S_{K^pK_p^{o}, \mathbb{F}_{L}}^b$ is a locally closed subset of $S_{K^pK_p^{o}, \mathbb{F}_{L}}$ satisfying $S_{K^pK_p^{o}, \mathbb{F}_{L}}^b(\overline{\mathbb{F}}_L) = \{ x \in S_{K^pK_p^{o}, \mathbb{F}_{L}}(\overline{\mathbb{F}}_L) \mid D(A_{x}[p^{\infty}]) \cong V_b \}$ for $b \in B(G)$. Here, we use the following notations: $V_b$ denotes the $G$-isocrystal\footnote{See \cite[definition 3.1.1]{CS} for the definition of $G$-isocrystal. For example, for a finite extension $M$ of $\mathbb{Q}_p$ and $H := \mathrm{Res}_{M/\mathbb{Q}_p}\mathrm{GL}_{n, M}$, an $H$-isocrystal $V$ over $\overline{\mathbb{F}}_p$ is naturally regarded as a finite free $\breve{\mathbb{Q}}_p \otimes_{\mathbb{Q}_p} M$-module of rank $n$ with a $\sigma$-semilinear $M$-linear isomorphism $\varphi : V \Isom V$. Similary, a $H$-$p$-divisible group $\mathcal{G}$ over a $p$-complete ring $R$ means that $\mathcal{G}$ is a $p$-divisible group with an action of $\mathcal{O}_M$ such that $\mathcal{G}[\varpi_M^n]$ has rank $q_M^{n}$ for any $n$ for a uniformizer $\varpi_M$ of $\mathcal{O}_M$. (See \cite[{\S} II.1]{HT} for details.) As our $G$ is a product $H_1 \times \cdots \times H_k$ of $H_i$'s as above $H$, we say that $\mathcal{G}$ is a $G$-$p$-divisible group if $\mathcal{G}$ is a product $\mathcal{G}_1 \times \cdots \times \mathcal{G}_k$ such that $\mathcal{G}_i$ is an $H_i$-$p$-divisible group for any $i$. A $G$-Dieudonne modules mean a $G$-isocrystal which comes from a $G$-$p$-divisible group.} corresponding to $b$ and $D(A_x[p^{\infty}])$ denotes the rational covariant $G$-Dieudonne module of the $G$-$p$-divisible group $A_x[p^{\infty}]$ corresponding to $x$. Note that we use the normalization of \cite{CS} on the Dieudonne modules. Let $\mathfrak{S}_{K^pK_p^{o}, \mathcal{O}_L}$ be the $p$-adic completion of $S_{K^pK_p^{o}, \mathcal{O}_L}$, $\mathfrak{S}^{b}_{K^pK_p^{o}, \mathcal{O}_L}$ be the formal subscheme of $\mathfrak{S}_{K^pK_p^{o}, \mathcal{O}_L}$ defined by $S^{b}_{K^pK_p^{o}, \mathbb{F}_L}$ and $\aS_{K^pK_p^o, L}^{b}$ be the generic fiber of $\mathfrak{S}^{b}_{K^pK_p^{o}, \mathcal{O}_L}$. For any $K_p \subset K_p^o$, let $\aS^{b}_{K^pK_p, L}$ (resp. $\aS^{b}_{K^pK_p}$, $\aS^{b}_{K^p}$) be the inverse image of $\aS_{K^pK_p^o, L}^{b}$ in $\aS_{K^pK_p, L}$ (resp. $\aS_{K^pK_p}$, $\aS_{K^p}$). By \cite[Theorem 4.2]{RR}, if $S_{K^pK_p^0, \overline{\mathbb{F}}_L}^b$ is not empty, then we have $b \in B(G, \mu^{-1})$.

On the other hand, by \cite[Theorem 1.11]{CS}, we also have a Newton stratification $\Fl := \sqcup_{b \in B(G, \mu^{-1})} \Fl^b$. Note that by \cite[Theorem B]{p-div}, for any complete algebraically closed nonarchimedean extension field $C_1$ of $C$, $\Fl(C_1, \mathcal{O}_{C_1})$ parametrizes the equivalence classes of $G$-$p$-divisible groups $\mathcal{G}$ over $\mathcal{O}_{C_1}$ with $\alpha : V_p \Isom V_p\mathcal{G}$ (resp. $\alpha : V_p^o \Isom T_p\mathcal{G}$) such that $\mathrm{Lie}G \otimes_{\mathcal{O}_{C_1}} C_1$ is isomorphic to $V_1 \otimes_{L} C_1$ as $B \otimes_{\mathbb{Q}} C_1$-module by sending $(\mathcal{G}, \alpha)$ to the subspace $\alpha^{-1}(\mathrm{Lie}\mathcal{G} \otimes_{\mathcal{O}_{C_1}} C_1)$ of $V_p \otimes_{\mathbb{Q}_p} C_1$. Here, we say that $(\mathcal{G}_1, \alpha_1)$ and $(\mathcal{G}_2, \alpha_2)$ are equivalent if there exists a quasi-isogeny $\beta : \mathcal{G}_1 \rightarrow \mathcal{G}_2$ of $G$-$p$-divisible groups such that $\alpha_2 = V_p(\beta) \circ \alpha_1$ (resp. isomorphism $\beta : \mathcal{G}_1 \Isom \mathcal{G}_2$ of $G$-$p$-divisible groups such that $\alpha_2 = T_p(\beta) \circ \alpha_1$). Moreover, $\Fl^b(C_1, \mathcal{O}_{C_1})$ consists of the elements $[(\mathcal{G}, \alpha)] \in \Fl(C_1, \mathcal{O}_{C_1})$ satisfying $D(\mathcal{G}_{k_1}) \cong V_{b} \otimes_{\breve{L}} W(k_1)[\frac{1}{p}]$, where $k_1$ denotes the residue field of $\mathcal{O}_{C_1}$. Thus all rank 1 points of $\aS^b_{K^p}$ are sent to $\Fl^b$ via $\pi_{\mathrm{HT}}$ for any $b \in B(G, \mu^{-1})$.

We have a certain closure relation for $S_{K^pK_p^{0}, \mathbb{F}_{L}}^b$ corresponding to the order of $B(G, \mu^{-1})$. In particular, $S_{K^pK_p^{0}, \mathbb{F}_{L}}^{b_0}$ is open for the $\mu$-ordinary element $b_0 \in B(G, \mu^{-1})$ and $S_{K^pK_p^{0}, \mathbb{F}_{L}}^{b_1}$ is closed for the basic element $b_1 \in B(G, \mu^{-1})$. (See Proposition \ref{calculation of isocrystals} for explicit descriptions of basic and $\mu$-ordinary elements.) On the other hand, $\Fl^{b_0}$ is closed and $\Fl^{b_1}$ is open.

We will calculate $B(G, \mu^{-1})$. It suffices to calculate $B(\mathrm{Res}_{M/\mathbb{Q}_p}\mathrm{GL}_{2, M}, \mu^{-1})$ for a finite extension $M$ of $\mathbb{Q}_p$ and a minuscule $\mu$ because $G$ has a form $\mathbb{G}_{m, \mathbb{Q}_p} \times \prod_{i} \mathrm{Res}_{M_i/\mathbb{Q}_p}\mathrm{GL}_{2, M_i}$. Let $\varpi$ be a uniformizer of $\mathcal{O}_M$ and let $T$ be the maximal torus of $\mathrm{Res}_{M/\mathbb{Q}_p}\mathrm{GL}_{2, M}$ consisting of diagonal matrices. In this paper, it suffices to consider a minuscules $\mu$ whose conjugacy class corresponds to an element $((1,0)_{i \in I}, (0, 0)_{i \notin I}) \in X_*(T_{\overline{\mathbb{Q}}_p}) = \prod_{\mathrm{Hom}_{\mathbb{Q}_p}(M, \overline{\mathbb{Q}}_p)} \mathbb{Z}^{2}_{+}$ for some $I \subset \mathrm{Hom}_{\mathbb{Q}_p}(M, \overline{\mathbb{Q}}_p)$. Let $m := |I|$. In the following, by using Lemma \ref{scalar restriction}, we regard $B(\mathrm{Res}_{M/\mathbb{Q}_p}\mathrm{GL}_{2, M}, \mu^{-1}) \subset B(\mathrm{GL}_{2, M})$. Note that by \cite[example 1.10]{RR}, every element of $B(\mathrm{GL}_{2, M})$ is represented by $\begin{pmatrix}
\varpi^{-s} & 0 \\
0 & \varpi^{-t} 
\end{pmatrix}$ or $\begin{pmatrix}
    0 & 1 \\
    \varpi^{-n} & 0 
    \end{pmatrix}$ for some $s, t \in \mathbb{Z}$, $n \in 2\mathbb{Z} + 1$ such that $s \ge t$.

\begin{prop} \label{calculation of isocrystals}

1 \ Complete representatives of nonbasic elements of $B(\mathrm{Res}_{M/\mathbb{Q}_p}\mathrm{GL}_{2, M}, \mu^{-1})$ are given by the following elements. ($s, t \in \mathbb{Z}_{\ge 0}$ such that $s > t$ and $s + t = m$.)

$$\begin{pmatrix}
\varpi^{-s} & 0 \\
0 & \varpi^{-t} 
\end{pmatrix}.$$

In particular, $\begin{pmatrix}
\varpi^{-m} & 0 \\
0 & 1
\end{pmatrix}$ is the $\mu$-ordinary element of $B(\mathrm{Res}_{M/\mathbb{Q}_p}\mathrm{GL}_{2, M}, \mu^{-1})$.

For any nonbasic $b \in B(\mathrm{Res}_{M/\mathbb{Q}_p}\mathrm{GL}_{2, M}, \mu^{-1})$, we have $J_b = T$.

2 \ A representative of the unique basic element of $B(\mathrm{Res}_{M/\mathbb{Q}_p}\mathrm{GL}_{2, M}, \mu^{-1})$ is given by the following elements. $$\begin{cases}
    \begin{pmatrix}
        \varpi^{-\frac{m}{2}} & 0 \\
        0 & \varpi^{-\frac{m}{2}} 
        \end{pmatrix} \ \text{if $m$ is even,} \\
        \begin{pmatrix}
        0 & 1 \\
        \varpi^{-m} & 0 
        \end{pmatrix} \ \text{if $m$ is odd.}
\end{cases}$$

For the basic element $b \in B(\mathrm{Res}_{M/\mathbb{Q}_p}\mathrm{GL}_{2, M}, \mu^{-1})$, we have $$J_b = \begin{cases}
    \mathrm{Res}_{M/\mathbb{Q}_p}\mathrm{GL}_{2, M} & \text{if $m$ is even,} \\
    \mathrm{Res}_{M/\mathbb{Q}_p}D^{\times} & \text{if $m$ is odd.}
  \end{cases}$$

\end{prop}

\begin{proof} The calculation of $J_b$ is easy.

Let $s, t \in \mathbb{Z}$ such that $s \ge t$. By the definition of $B(\mathrm{Res}_{M/\mathbb{Q}_p}\mathrm{GL}_{2, M}, \mu^{-1})$, the condition that $\begin{pmatrix}
\varpi^{-s} & 0 \\
0 & \varpi^{-t} 
\end{pmatrix}$ is contained in $B(\mathrm{Res}_{M/\mathbb{Q}_p}\mathrm{GL}_{2, M}, \mu^{-1})$ is equivalent to the following two relations : $\nu(b) \le \overline{\mu}^{-1}$ and $k(b) = \mu^{-1, \flat}$. (See \cite[{\S} 3.1]{CS} for the definitions of these notations.) This is equivalent to $(0, -m) - (-t, -s) = r(1, -1)$ for some $r \in \mathbb{Q}_{\ge 0}$ and $-s - t = -m$. This is also equivalent to $t \ge 0$ and $s + t = m$. If $s=t$, this is a basic element and $m \in 2\mathbb{Z}$ and if $s \neq t$, this is a non-basic element.

Let $\begin{pmatrix}
    0 & 1 \\
    \varpi^{-n} & 0
    \end{pmatrix}$ be a basic element of $B(\mathrm{Res}_{M/\mathbb{Q}_p}\mathrm{GL}_{2, M}, \mu^{-1})$, where $n\in 2\mathbb{Z} + 1$. Then the relation $\nu(b) \le \overline{\mu}^{-1}$ and $k(b) = \mu^{-1, \flat}$ is equivalent to $(0, -m) - (-\frac{n}{2}, -\frac{n}{2}) = r(1, -1)$ for some $r \in \mathbb{Q}_{\ge 0}$ and $-n = -m$. This is equivalent to $n \ge 0$ and $n = m \in 2\mathbb{Z} + 1$. \end{proof}

\begin{cor}\label{GL iso}

    Let $w \mid v$ be a place of $F$ and assume $\Psi \subset \mathrm{Hom}_{\mathbb{Q}_p}(F_w, \overline{\mathbb{Q}}_p)$. Then we have $|B(G, \mu^{-1})| = \lfloor \frac{d+1}{2} \rfloor +1$. Thus if $d \le 2$, then $B(G, \mu^{-1})$ consists of only the basic element and the $\mu$-ordinary element.
    
    \end{cor}

In the following, we will calculate $GDR^{\Psi-\mathrm{la}}_{\lambda}$ only on the $\mu$-ordinary locus and on the basic locus by using the same strategy as \cite[{\S} 5]{PanII}.

\subsection{$\mu$-ordinary locus.}

First, we consider a local situation and prove some fundamental results. Let $M$ be a finite extension of $\mathbb{Q}_p$.

\begin{lem} \label{connected-etale}

Let $S$ be a scheme on which $p$ is Zariski locally nilpotent and $\mathcal{G}$ be a $p$-divisible group over $S$ with an action of $\mathcal{O}_M$ such that $|\mathcal{G}[p](k(s))|$ is independent of the choice of a geometric point $s$ of $S$. Then we have a unique exact sequence of $p$-divisible groups over $S$ with actions of $\mathcal{O}_M$: $0 \rightarrow \mathcal{G}^0 \rightarrow \mathcal{G} \rightarrow \mathcal{G}^{\et} \rightarrow 0$, where $\mathcal{G}^0$ is a formal $p$-divisible group over $S$ with an action of $\mathcal{O}_M$ and $\mathcal{G}^{\et}$ is an ind-$\acute{e}tale$ $p$-divisible group over $S$ with an action of $\mathcal{O}_M$.

\end{lem}

\begin{proof}

See \cite[Proposition 4.9]{Messing}. \end{proof}

Let $H:=\mathrm{Res}_{M/\mathbb{Q}_p}\mathrm{GL}_{2, M}$, $\mu : \mathbb{G}_{m, \overline{M}} \rightarrow H_{\overline{M}}$ be the minuscule given by $((1, 0)_{I}, (0, 0)_{\mathrm{Hom}_{\mathbb{Q}_p}(M, \overline{M}) \setminus I}) \in (\mathbb{Z}^2)^{\mathrm{Hom}_{\mathbb{Q}_p}(M, \overline{M})}$ as in the previous subsection, $m := |I|$ and $b \in B(H, \mu^{-1})$ be a $\mu$-ordinary element. Assume $m \neq 0$. Let $\mathbb{X}_{b}$ be a $G$-$p$-divisible group over $\overline{\mathbb{F}}_p$ corresponding to $b$. Note that by proposition \ref{calculation of isocrystals}, we may assume that we have a decomposition $\mathbb{X}_b = \mathbb{Y}_{b} \oplus M/\mathcal{O}_{M}$, where $\mathbb{Y}_{b}$ is the $p$-divisible group with an action of $\mathcal{O}_M$ given by the element of $B(\mathrm{Res}_{M/\mathbb{Q}_p}(\mathrm{GL}_{1, M}))$ represented by $\varpi^{-m} \in \mathrm{GL}_1(M)$.

\begin{lem}\label{isogeny}

Let $k$ be an algebraically closed field of characteristic $p$ and $\mathcal{G}$ be an $H$-$p$-divisible group over $k$. Then $\mathcal{G} \cong \mathbb{X}_{b_0} \times_{\overline{\mathbb{F}}_p} k$ if and only if $D(\mathcal{G}) \cong V_{b} \otimes_{\breve{\mathbb{Q}}_p} W(k)[\frac{1}{p}]$.

\end{lem}

\begin{proof}

Let $0 \rightarrow \mathcal{G}^0 \rightarrow \mathcal{G} \rightarrow \mathcal{G}^{\et} \rightarrow 0$ be the exact sequence given by Lemma \ref{connected-etale}. Note that we have a section of $\mathcal{G} \rightarrow \mathcal{G}^{\et}$ by taking the $k$-valued points. Thus the result follows from the fact that any two $\mathrm{Res}_{M/\mathbb{Q}_p}(\mathrm{GL}_{1, M})$-$p$-divisible groups which are isogenous are isomorphic by Dieudonne theory. \end{proof}

Let $C_1$ be a completed algebraic closure of $M$ and $\mathbb{Y}_{b, \mathcal{O}_{C_1}}$ be a $\mathrm{Res}_{M/\mathbb{Q}_p}(\mathrm{GL}_{1, M})$-$p$-divisible group over $\mathcal{O}_{C_1}$ which is a lift of $\mathbb{Y}_{b}$. This can be obtained from a character of $G_M$ of Hodge-Tate weight $((-1)_{i \in I}, (0)_{i \in \mathrm{Hom}_{\mathbb{Q}_p}(M, \overline{M}) \setminus I})$. We fix an isomorphism of $\mathcal{O}_{M}$-modules $\mathcal{O}_{M} \Isom T_p\mathbb{Y}_{b, \mathcal{O}_{C_1}}$ and thus $\alpha_0 : \mathcal{O}_{M}^{\oplus 2} \Isom T_p\mathbb{Y}_{b, \mathcal{O}_{C_1}} \oplus \mathcal{O}_M$. The $H$-$p$-divisible group $\mathbb{X}_{b, \mathcal{O}_{C_1}} := \mathbb{Y}_{b, \mathcal{O}_{C_1}} \oplus M/\mathcal{O}_M$ is the lift of $\mathbb{X}_b$ over $\mathcal{O}_{C_1}$. This defines a point $\infty = [(\mathbb{X}_{b, \mathcal{O}_{C_1}}, \alpha_0)] \in \Fl^{b}_{(H, \mu)}(C_1, \mathcal{O}_{C_1})$\footnote{$[(\mathbb{X}_{b, \mathcal{O}_{C_1}}, \alpha_0)]$ denotes the equivalence class of $(\mathbb{X}_b, \alpha_0)$.} by the moduli interpretation of $\Fl_{(H, \mu)}$ as we have seen in {\S} 4.3, where $\Fl_{(H, \mu)} = \prod_{\tau \in I} \mathbb{P}^1_{C_1}$ denotes the flag variety corresponding to $(H, \mu)$. Let $B_H$ denote the Borel subgroup of $H$ consisting of the upper triangular matrices.

\begin{lem}\label{transitive}

The right action of $H(\mathbb{Q}_p)$ on $\Fl^{b}$ is transitive and the stabilizer group at $\infty$ is equal to $B_H(\mathbb{Q}_p)$. Moreover, the map $B_H(\mathbb{Q}_p) \setminus H(\mathbb{Q}_p) \rightarrow \Fl^{b}, \ B_H(\mathbb{Q}_p)g \rightarrow \infty g$ is a homeomorphism.

\end{lem}

\begin{proof}

In order to prove the first property, it suffices to prove that for any algebraically closed complete non-archimedean field extension $C_1'$ of $C_1$, the orbit map $H(\mathbb{Q}_p) \rightarrow \Fl^{b_0}(C_1', \mathcal{O}_{C_1'}), \ g \mapsto \infty g$ induces a bijection $B_H(\mathbb{Q}_p) \setminus H(\mathbb{Q}_p) \Isom \Fl^{b}(C_1', \mathcal{O}_{C_1'})$. In fact, $\Fl^{b}$ is a closed subset of $\Fl$, dimension zero and thus doesn't contain higher rank points by \cite[Corollary 3.5.9 and Proposition 4.2.23]{CS}. We may assume $C_1 = C_1'$. Let $k_1$ be the residue field of $\mathcal{O}_{C_1}$.

Let $[(\mathcal{G}, \alpha)] \in \Fl^{b}_{(H, \mu)}(C_1, \mathcal{O}_{C_1})$. Then by the moduli interpretation of $\Fl^{b}_{(H, \mu)}$, we have $D(\mathcal{G}_{k_1}) \cong V_{b} \otimes_{\breve{\mathbb{Q}}_p} W(k_1)[\frac{1}{p}]$. Thus we have $\mathcal{G}_{k_1} \cong \mathbb{X}_{b_0} \times_{\overline{\mathbb{F}}_p} k_1$ by Lemma \ref{isogeny}. By Lemma \ref{connected-etale}, we have a filtration $0 \rightarrow \mathcal{G}^0 \rightarrow \mathcal{G} \rightarrow \mathcal{G}^{\et} \rightarrow 0$. This induces an exact sequence $0 \rightarrow V_p\mathcal{G}^0 \rightarrow V_p\mathcal{G} \rightarrow V_p\mathcal{G}^{\et} \rightarrow 0$. Note that $\mathrm{Lie}\mathcal{G}^0 = \mathrm{Lie}\mathcal{G}$ and $\mathcal{G}^0$ and $\mathcal{G}^{\et}$ are $\mathrm{Res}_{M/\mathbb{Q}_p}(\mathrm{GL}_{1, M})$-$p$-divisible groups by $\mathcal{G}_{k_1} \cong \mathbb{X}_{b_0} \times_{\overline{\mathbb{F}}_p} k_1$. Since $V_p\mathcal{G}^0$ (resp. $V_p\mathbb{Y}_{b, \mathcal{O}_{C_1}}$) is a one dimensional subspace of the 2-dimensional $M$-vector space $V_p\mathcal{G} \xrightarrow[\alpha]{\sim} M^{\oplus 2}$ (resp. $V_p\mathbb{X}_{b, \mathcal{O}_{C_1}} \xrightarrow[\alpha_0]{\sim} M^{\oplus 2}$), there exists $g \in H(\mathbb{Q}_p) = \mathrm{GL}_2(M)$ such that $g \alpha^{-1}(V_p\mathcal{G}^{0}) = \alpha_{0}^{-1}(V_p\mathbb{Y}_{b, \mathcal{O}_{C_1}})$. Thus we obtain $(\mathcal{G}, \alpha \circ g) = \infty$. The statement about the stabilizer of $\infty$ is trivial. The second statement is clear because $B_H(\mathbb{Q}_p) \setminus H(\mathbb{Q}_p)$ is compact and $\Fl^b \subset \Fl(C_1, \mathcal{O}_{C_1})$ is Hausdorff. \end{proof}

Now, we come back to our global situation. Let $b \in B(G, \mu^{-1})$ be the $\mu$-ordinary element. Let $S^{\mu-\mathrm{ord}}_{K^pK_p^{o}, \overline{\mathbb{F}}_L}$, $\mathfrak{S}^{\mu-\mathrm{ord}}_{K^pK_p^{o}, \mathcal{O}_L}$, $\aS_{K^pK_p^o, L}^{\mu-\mathrm{ord}}$, $\aS^{\mu-\mathrm{ord}}_{K^pK_p, L}$, $\aS^{\mu-\mathrm{ord}}_{K^pK_p}$ and $\aS^{\mu-\mathrm{ord}}_{K^p}$ be $S_{K^pK_p^{o}, \overline{\mathbb{F}}_{L}}^{b}$, $\mathfrak{S}^{b}_{K^pK_p^{o}, \mathcal{O}_L}$, $\aS_{K^pK_p^o, L}^{b}$, $\aS^{b}_{K^pK_p}$ and $\aS^{b}_{K^p}$ respectively.

Note that we have a natural identification $B(G, \mu^{-1}) = \prod_{w \mid v, w \notin S_0} B(\mathrm{Res}_{F_w/\mathbb{Q}_p} \mathrm{GL}_{2, F_w}, \mu_w^{-1})$. For any $w \mid v$ such that $w \notin S_0$, let $\mathbb{Y}_w$, $\mathbb{Y}_{w, \mathcal{O}_C}$ and $\alpha_w$ as above $\mathbb{Y}_b$, $\mathbb{Y}_{b, \mathcal{O}_{C_1}}$ and $\alpha_0$. This defines the point $\infty \in \Fl(C, \mathcal{O}_C)$. The $\mu$-ordinary locus $S^{\mu-\mathrm{ord}}_{K^pK_p^o, \overline{\mathbb{F}}_L}$ coincides with the ordinary locus of $S_{K^pK_p^o, \overline{\mathbb{F}}_L}$ if and only if $\mathbb{Y}_{w} \cong \mu_{p^{\infty}} \otimes_{\mathbb{Z}_p} \mathcal{O}_{F_w}$ for any $w \mid v$ and $w \notin S_0$.
    
    \vspace{0.5 \baselineskip}
    
    We will construct an analogue of the canonical locus in the Siegel modular variety. Let $\mathfrak{A}_{K^pK_p^o, \mathcal{O}_{L}}^{\mu-\mathrm{ord}}$ be the universal abelian scheme over $\mathfrak{S}_{K^pK_p^o, \mathcal{O}_{L}}^{\mu-\mathrm{ord}}$. Then the decomposition $\mathcal{O}_B \otimes_{\mathbb{Z}} \mathbb{Z}_p = \prod_{w \mid p} M_2(\mathcal{O}_{F_w})$ induces a decomposition $\mathfrak{A}_{K^pK_p^o, \mathcal{O}_{L}}^{\mu-\mathrm{ord}}[p^{\infty}] = \oplus_{w \mid p} \mathfrak{A}_{K^pK_p^o, \mathcal{O}_{L}}^{\mu-\mathrm{ord}}[\varpi_w^{\infty}]$. By Morita equivalence, for any $w \mid p$, we get a $p$-divisible group $\mathfrak{G}[\varpi_w^{\infty}]$ with an action of $\mathcal{O}_{F_w}$ of height $2[F_w:\mathbb{Q}_p]$ over $\mathfrak{S}_{K^pK_p^o, \mathcal{O}_{L}}^{\mu-\mathrm{ord}}$ with an isomorphism $\mathfrak{G}[\varpi_w^{\infty}]^{\vee} \cong \mathfrak{G}[\varpi_{w^c}^{\infty}]$ under the identification $c : F_w \Isom F_{w^c}$. By Lemma \ref{connected-etale} and Lemma \ref{isogeny}, we have an exact sequence $0 \rightarrow \mathfrak{G}[\varpi_w^{\infty}]^{o} \rightarrow \mathfrak{G}[\varpi_w^{\infty}] \rightarrow \mathfrak{G}[\varpi_w^{\infty}]^{\mathrm{\et}} \rightarrow 0$. We write $0 \rightarrow \mathcal{G}[\varpi_w^{\infty}]^{o} \rightarrow \mathcal{G}[\varpi_w^{\infty}] \rightarrow \mathcal{G}[\varpi_w^{\infty}]^{\mathrm{\et}} \rightarrow 0$ for its generic fibers. Let $\Gamma_{0}(p^m) := \mathbb{Z}_p^{\times} \times \prod_{w \mid v, v \notin S_0} \{g \in \mathrm{GL}_2(\mathcal{O}_{F_w}) \mid g \equiv \begin{pmatrix} * & * \\
 0 & *  \end{pmatrix} \mod p^{m}  \} \times \prod_{w \in S_0} \mathrm{GL}_2(\mathcal{O}_{F_w})$ and $\Gamma(p^m) := \mathrm{Ker}(\mathbb{Z}_p^{\times} \times \prod_{w \mid v} \mathrm{GL}_2(\mathcal{O}_{F_w}) \rightarrow (\mathbb{Z}/p^m)^{\times} \times \prod_{w \mid v} \mathrm{GL}_2(\mathcal{O}_{F_w}/p^{m}))$ for $m \in \mathbb{Z}_{\ge 0}$. 
    
    Then we have a map $\aS_{K^pK_p^o, L}^{\mu-\mathrm{ord}} \hookrightarrow \aS_{K^p\Gamma_{0}(p^m), L}^{\mu-\mathrm{ord}}, \ [A] \mapsto [(A, (\mathcal{G}[\varpi_w^{m}]^{o})_{w \in S_0})]$. (Precisely, we also need to consider extra structures.) Note that this is an open and closed immersion because the transition map $\aS_{K^p\Gamma_{0}(p^m), L}^{\mu-\mathrm{ord}} \rightarrow \aS_{K^pK_p^o, L}^{\mu-\mathrm{ord}}$ is finite $\etale$ and the composition of these two maps is the identity. Let $\aS_{K^p\Gamma_0(p^m), L}^{\mu-\mathrm{can}}$ denote the image of the above map in $\aS_{K^p\Gamma_{0}(p^m), L}^{\mu-\mathrm{ord}}$ and $\aS_{K^p\Gamma(p^m), L}^{\mu-\mathrm{can}}$ (resp. $\aS_{K^p\Gamma(p^m)}^{\mu-\mathrm{can}}$) be the inverse image of $\aS_{K^p\Gamma_0(p^m), L}^{\mu-\mathrm{can}}$ in $\aS_{K^p\Gamma(p^m), L}$ (resp. $\aS_{K^p\Gamma(p^m)}$).

    \begin{lem} \label{canonical locus}
        
$\pi_{\mathrm{HT}}^{-1}(\infty)$ is equal to the closure $\overline{\varprojlim_{n}\aS^{\mu-\mathrm{can}}_{K^p\Gamma(p^n)}}$ of $\varprojlim_{n}\aS^{\mu-\mathrm{can}}_{K^p\Gamma(p^n)}$ in $\aS_{K^p}$.
        
    \end{lem}
    
    \begin{proof}

Since $\pi_{\mathrm{HT}}^{-1}(\infty)$ is closed, it suffices to prove that for any algebraically closed complete nonarchimedean extension $C_1$ of $C$, we have $\pi_{\mathrm{HT}}^{-1}(\infty)(C_1, \mathcal{O}_{C_1}) = \varprojlim_{n}\aS^{\mu-\mathrm{can}}_{K^p\Gamma(p^n)}(C_1, \mathcal{O}_{C_1})$. This follows from the moduli interpretations. In fact, both sides parametrize the object $[(A, \lambda, i, \overline{\eta}^p, \eta_p)] \in \aS_{K^p}(C_1, \mathcal{O}_{C_1})$ such that the subspace of $V_p \otimes_{\mathbb{Q}_p} C_1 = \prod_{\tau} V_{\tau} \otimes_L C_1$ generated by $e_{1, \tau}$ ($\tau \in \Psi$) corresponds $\mathrm{Lie}A$ to via $\eta_p : V_p \Isom V_pA$, where $\lambda$ denotes the polarization, $i$ denotes the action of $B$, $\overline{\eta}^p$ denotes the $K^p$-level structure outside $p$ and $\eta_p$ denotes the level structure at $p$. \end{proof}

\begin{cor}\label{topology}

1 \ For any affinoid open $U$ of $\Fl$ containing $\infty$, there exist $n$ and a quasicompact open $V_n$ of $\aS_{K^p\Gamma(p^n)}$ satisfying $V_n \supset \overline{\aS_{K^p\Gamma(p^n)}^{\mu-\mathrm{can}}}$ such that $\pi_{\Gamma(p^n)}^{-1}(V_n) \subset \pi_{\mathrm{HT}}^{-1}(U)$.
        
2 \ Conversely, for any $n$ and open $V_n$ satisfying $V_n \supset \overline{\aS_{K^p\Gamma(p^n)}^{\mu-\mathrm{can}}}$, there exists an affinoid open $U$ of $\Fl$ containing $\infty$ such that $\pi_{\Gamma(p^n)}^{-1}(V_n) \supset \pi_{\mathrm{HT}}^{-1}(U)$.

\end{cor}
    
\begin{proof}

This formally follows from Lemma \ref{canonical locus} as in \cite[Lemma 3.3.8 etc]{Torsion}. By Lemma \ref{canonical locus}, we have $\pi_{\mathrm{HT}}^{-1}(\infty) = \cap_{\infty \in U} \pi_{\mathrm{HT}}^{-1}(U) = \overline{\varprojlim_{n}\aS^{\mu-\mathrm{can}}_{K^p\Gamma(p^n)}} = \cap_{n, V_n \supset \overline{\aS_{K^p\Gamma(p^n)}^{\mu-\mathrm{ord}}}} \pi_{\Gamma(p^n)}^{-1}(V_n)$. Thus for any $U$ containing $\infty$, we obtain $\pi_{\mathrm{HT}}^{-1}(U) \supset \cap_{n, V_n \supset \overline{\aS_{K^p\Gamma(p^n)}^{\mu-\mathrm{ord}}}} \pi_{\Gamma(p^n)}^{-1}(V_n)$. For any quasicompact open $V_n$, $\pi_{\Gamma(p^n)}^{-1}(V_n)$ is a quasicompact open. Thus $\aS_{K^p} \setminus \pi_{\Gamma(p^n)}^{-1}(V_n)$ is open for the constructible topology. This implies that there exist $n$ and a quasi-compact open $V_n \supset \overline{\aS_{K^p\Gamma(p^n)}^{\mu-\mathrm{ord}}}$ of $\aS_{K^p\Gamma(p^n)}$ such that $\pi_{\mathrm{HT}}^{-1}(U) \supset \pi_{\Gamma(p^n)}^{-1}(V_n)$. We can obtain the converse by exactly the same argument. \end{proof}
    
    \begin{prop}\label{affinoid ordinary}

$\mathcal{S}^{\mu-\mathrm{can}}_{K^p\Gamma(p^n), L}$ is affinoid.

\end{prop}

\begin{proof} By \cite[Proposition 9.21]{Igusa} (see also \cite[Proposition 1.9]{CS2}) and Lemma \ref{isogeny}, we obtain that $\aS_{K^pK_p^0, L}^{\mu-\mathrm{ord}}$ is affinoid. This implies the result. \end{proof}

Here, we recall some foundational results about dagger spaces. Let $M$ be a complete nonarchimedean rank 1 valuation field. By taking the completion of dagger algebras, we obtain a natural faithful functor $( \ )'$ from the category of dagger spaces over $M$ to the category of rigid spaces over $M$. (See \cite[Theorem 2.19]{overconvergent} for details.) Note that the category of quasiseparated rigid spaces over $M$ is equivalent to the category of quasiseparated adic space locally topologically of finite type over $M$ and for a quasiseparated rigid space $X$ over $M$, admissible quasicompact open subsets of $X$ are identified with quasicompact open subsets of $X^{\mathrm{ad}}$.\footnote{$X^{\mathrm{ad}}$ denotes the adic space over $M$ corresponding to $X$.} (See \cite[(1.1.11)]{Huber}.) We use these identifications in the following.

\begin{lem}\label{dagger variety}

1 \ For any dagger space $X$ over $M$, we have a natural map $X' \rightarrow X$ as ringed sites which induces an isomorphism as sites.

2 \ The functor induces an equivalence between the category of proper dagger spaces over $M$ and the category of proper rigid spaces over $M$. Moreover, for a proper rigid space over $M$ and the corresponding proper dagger space $X^{\dagger}$, the categories of coherent sheaves on both sides are naturally equivalent.

3 \ For any affinoid smooth rigid space $X$ over $M$, there exists an affinoid smooth dagger space $X^{\dagger}$ over $M$ such that $(X^{\dagger})' = X$.

Moreover, if $X$ is an open of a proper adic space $Y$, then we can take $X^{\dagger}$ as an admissible open of $Y^{\dagger}$, where $Y^{\dagger}$ denotes the dagger space corresponding to $Y$ given by the statement 2.

In the following, we fix a proper rigid space $Y$ over $M$ and a quasicompact open $U$ of $Y^{\dagger}$.

4 \ For a coherent sheaf $\mathcal{F}$ on $Y$, we have $H^i(U, \mathcal{F}^{\dagger}) = \varinjlim_{V \supset \overline{U}} H^i(V, \mathcal{F})$. Here, $\mathcal{F}^{\dagger}$ denotes the coherent sheaf on $Y^{\dagger}$ corresponding to $\mathcal{F}$ given by the statement 2 and $V$ runs through quasicompact open subsets of $Y$ containing the closure $\overline{U}$ of $U$.

5 \ $H^i_{dR}(U/M) = \varinjlim_{V \supset \overline{U}} H^i_{dR}(V/M)$, where $V$ runs through quasicompact open subsets of $Y$ containing the closure $\overline{U}$ of $U$. The left hand side is the de Rham cohomology of the dagger variety $U$ over $M$ and the right hand side is the de Rham cohomology of the rigid space $V$ over $M$.

\end{lem}

\begin{proof}

1 is \cite[(3) of Theorem 2.19]{overconvergent}.

2 is a special case of \cite[Theorems 2.26 and 2.27]{overconvergent}.

3 follows from \cite[(4) of Theorem 2.19]{overconvergent} and \cite[Theorem 7]{Elkik}.

4 \ We have a natural map $\varinjlim_{V \supset \overline{U}} H^i(V, \mathcal{F}) \rightarrow H^i(U, \mathcal{F}^{\dagger})$ by \cite[proof of Theorems 2.26 and 2.27]{overconvergent}. This is an isomorphism by \cite[(2.23) and proof of Theorem 2.27]{overconvergent} if $U$ is an affinoid open of $Y$ and there exists an affinoid open $V$ of $Y$ such that $U \subset^{\dagger} V$, i.e., $U$ is a rational subset of $V$ having the following form: $U = (\mathrm{Spa}M\langle X_1, \cdots, X_k \rangle/I) (\frac{X_1^s, \cdots, X_k^s}{\varpi})$\footnote{This denotes the rational subset of $\mathrm{Spa}M\langle X_1, \cdots, X_k \rangle/I$ defined by $|X_i^s \ \mathrm{mod} \ I| \le |\varpi|$ for any $i$.} and $V = \mathrm{Spa}M\langle X_1, \cdots, X_k \rangle/I$ for some pseudo-uniformizer $\varpi$ of $\mathcal{O}_M$, $s \in \mathbb{Z}_{>0}$ and some ideal $I$ of $M\langle X_1, \cdots, X_k \rangle$. Note that if $W$ is a rational open of $V$ contained in $U$ as above, then there exists a rational subset $W'$ of $V$ such that $W \subset^{\dagger} W'$ by using explicit descriptions of rational open subsets. Since $Y$ is proper, we have admissible affinoid open coverings $\{ U_{i} \}_{i=1}^k$ and $\{ V_{i} \}_{i=1}^k$ of $Y$ such that $U_i \subset^{\dagger} V_i$. (See \cite[Remark 1.3.19]{Huber} for equivalent conditions of the properness.) \cite[proof of Lemma 5.3]{pH} implies that there exists a covering $\{ W_j \}_{j=1}^s$ of $U$ such that $W_j$ is a rational subset of $V_i$ for some $i$ contained in $U_i$ and the intersection $W_j \cap W_k$ is a rational subset of $W_j$ for any $j, k$. Note that we have $\overline{W_j} \cap \overline{W_k} = \overline{W_j \cap W_k}$ because the sets of rank 1 points of both sides are equal to the set $A$ of rank 1 points in $W_j \cap W_k$ and both sets are equal to the closure of $A$. Thus the result follows from the $\mathrm{\check{C}ech}$ spectral sequence.

5 follows from $4$ by the Hodge-de Rham spectral sequences. \end{proof}

By 3 of Lemma \ref{dagger variety} and Proposition \ref{affinoid ordinary}, there exists an open affinoid dagger subspace $\mathcal{S}_{K^p\Gamma(p^n)}^{\mu-\mathrm{ord}, \dagger}$ of $\mathcal{S}_{K^p\Gamma(p^n)}^{\dagger}$ corresponding to $\aS_{K^p\Gamma(p^n)}^{\mu-\mathrm{ord}}$. Moreover, by 4 of Lemma \ref{dagger variety} and \cite[Proposition 3.1]{overconvergent}, we obtain $H^i_{\mathrm{dR}}(\aS_{K^p\Gamma(p^n)}^{\mu-\mathrm{can}, \dagger}, D^{\dagger}_{\lambda, S_{K^p\Gamma(p^n)}}) = H^i(\Gamma(\aS_{K^p\Gamma(p^n)}^{\mu-\mathrm{can}, \dagger}, D^{\dagger}_{\lambda, S_{K^p\Gamma(p^n)}} \otimes \Omega_{\aS_{K^p\Gamma(p^n)}^{\mu-\mathrm{can}, \dagger}}^{\bullet})) = \varinjlim_{V \supset \overline{\aS_{K^p\Gamma(p^n)}^{\mu-\mathrm{can}}}}H^i(\Gamma(V, GDR_{\lambda, S_{K^p\Gamma(p^n)}}))$ for any $i$, where $V$ runs through quasicompact open subsets of $\aS_{K^p\Gamma(p^n)}$ containing $\overline{\aS_{K^p\Gamma(p^n)}^{\mu-\mathrm{can}}}$. (See {\S}3.3. for the definition of $GDR_{\lambda, S_{K^p\Gamma(p^n)}}$.)
    
    \begin{prop}\label{finitedimensional}
    
$H^i_{\mathrm{dR}}(\aS_{K^p\Gamma(p^n)}^{\mu-\mathrm{can}, \dagger}, D^{\dagger}_{\lambda, S_{K^p\Gamma(p^n)}})$ is a finite-dimensional $C$-vector space for any $n$.
    
    \end{prop}

\begin{proof}

Let $m$ be a nonnegative integer, let $V$ be a quasicompact open of $\aS_{K^p\Gamma(p^n), L}$ containing $\overline{\aS_{K^p\Gamma(p^n), L}^{\mu-\mathrm{can}}}$, let $\mathcal{A} \rightarrow \aS_{K^p\Gamma(p^n), L}$ be as after \ref{BGG} and $f : \mathcal{A}^m \rightarrow \aS_{K^p\Gamma(p^n), L}$ be the $m$-times self-fiber product of $\mathcal{A}$ over $\aS_{K^p\Gamma(p^n), L}$.

We define a decreasing filtration $\{F^{p} \Omega_{\mathcal{A}^m/L}^{\bullet}\}_{p}$ on $\Omega_{\mathcal{A}^m/L}^{\bullet}$ by $F^{p} \Omega_{\mathcal{A}^m/L}^{\bullet} := \mathrm{Im}(f^* \Omega_{\aS_{K^p\Gamma(p^n)}/L}^{p} \otimes_{\mathcal{O}_{\mathcal{A}^m}} \Omega_{\mathcal{A}^m/L}^{\bullet - p} \rightarrow \Omega_{\mathcal{A}^m/L}^{\bullet})$. Thus we have $\mathrm{gr}^p \Omega_{\mathcal{A}^m/L}^{\bullet} = f^{*} \Omega_{\aS_{K^p\Gamma(p^n)}}^p \otimes \Omega_{\mathcal{A}^m/\aS_{K^p\Gamma(p^n)}}^{\bullet - p}$ and a spectral sequence $E_1^{p, q} := H^{p+q}(f^{-1}(V), \mathrm{gr}^p\Omega_{\mathcal{A}^m/L}^{\bullet}) \Rightarrow H^{p+q}(f^{-1}(V), \Omega_{\mathcal{A}^m/L}^{\bullet})$. We have equalities $E_1^{p, q} = H^{p+q}(f^{-1}(V), \ \mathrm{gr}^p\Omega_{\mathcal{A}^m/L}^{\bullet}) =  H^{q}(f^{-1}(V), f^*\Omega_{\aS_{K^p\Gamma(p^n),L}/L}^{p} \otimes_{\mathcal{O}_{\mathcal{A}^m}} \Omega_{\mathcal{A}^m/\aS_{K^p\Gamma(p^n), L}}^{\bullet})$   $= H^{q}(V, Rf_{*} (f^*\Omega_{\aS_{K^p\Gamma(p^n)/L}}^{p} \otimes_{\mathcal{O}_{\mathcal{A}^m}} \Omega_{\mathcal{A}^m/\aS_{K^p\Gamma(p^n), L}}^{\bullet})) = H^{q}(V, \Omega_{\aS_{K^p\Gamma(p^n), L}/L}^{p} \otimes_{\mathcal{O}_{\aS_{K^p\Gamma(p^n), L}}} Rf_{*}\Omega_{\mathcal{A}^m/\aS_{K^p\Gamma(p^n), L}}^{\bullet})$. 

The same argument as \cite[{\S} III.2]{HT} and \cite[{\S} 2.1]{Caraiani} shows that there exist a nonnegative integer $m$ and $e \in \mathrm{End}(\mathcal{A}^m/\aS_{K^p\Gamma(p^n), L}) \otimes \mathbb{Q}$ which is independent of the choice of $V$ such that $e$ is an idempotent on $Rf_{*} \Omega_{\mathcal{A}^{m}/\aS_{K^p\Gamma(p^n), L}}$ and $H^{n}(f^{-1}(V), \Omega_{\mathcal{A}^m/L}^{\bullet})$ and  $eRf_*\Omega_{\mathcal{A}^{m}/\aS_{K^p\Gamma(p^n), L}} = D_{\lambda, \aS_{K^p\Gamma(p^n), L}}[-m]$. (Here, we ignore the Hodge filtrations and Hecke actions on both sides.)
 
By taking $\varinjlim_{V \supset \overline{\aS_{K^p\Gamma(p^n), L}^{\mu-\mathrm{can}}}} e$, we obtain the spectral sequence

$\displaystyle E_1^{p, q} := \varinjlim_{V \supset \overline{\aS_{K^p\Gamma(p^n), L}^{\mu-\mathrm{can}}}} H^{q-m}(V, \Omega_{\aS_{K^p\Gamma(p^n)/L}}^{p} \otimes_{\mathcal{O}_{\aS_{K^p\Gamma(p^n), L}}} D_{\lambda, \aS_{K^p\Gamma(p^n), L}}) \Rightarrow \varinjlim_{V \supset \overline{\aS_{K^p\Gamma(p^n), L}^{\mu-\mathrm{can}}}} eH^{p+q}(f^{-1}(V), \Omega_{\mathcal{A}^m/L}^{\bullet})$.

Since $\mathcal{S}^{\mu-\mathrm{can}}_{K^p\Gamma(p^n), L}$ is affinoid by Proposition \ref{affinoid ordinary}, we obtain $E_1^{p, q} = 0$ for any $q \neq m$ by \cite[Proposition 3.1]{overconvergent}. This implies $$\varinjlim_{V \supset \overline{\aS_{K^p\Gamma(p^n), L}^{\mu-\mathrm{can}}}} H^{p}(\Gamma(V, \Omega_{\aS_{K^p\Gamma(p^n)/L}}^{\bullet} \otimes_{\mathcal{O}_{\aS_{K^p\Gamma(p^n), L}}} D_{\lambda, \aS_{K^p\Gamma(p^n), L}})) \cong e \varinjlim_{V \supset \overline{\aS_{K^p\Gamma(p^n), L}^{\mu-\mathrm{can}}}} H^{p+m}(f^{-1}(V), \Omega_{\mathcal{A}^m/L}^{\bullet}).$$ By \cite[Theorem A]{deRham} and 5 of Lemma \ref{dagger variety}, the cohomology $$H^{p+m}_{\mathrm{dR}}(\mathcal{A}^{m, \mu-\mathrm{can}, \dagger}) := \varinjlim_{V \supset \overline{\aS_{K^p\Gamma(p^n)}^{\mu-\mathrm{can}}}}H^{p+m}(f^{-1}(V), \Omega_{\mathcal{A}^m/L}^{\bullet})$$ is finite-dimensional over $L$. Consequently, $$H^{p}_{\mathrm{dR}}(\aS_{K^p\Gamma(p^n), L}^{\mu-\mathrm{ord}, \dagger}, D^{\dagger}_{\lambda, \aS_{K^p\Gamma(p^n),L}}) := \varinjlim_{V \supset \overline{\aS_{K^p\Gamma(p^n), L}^{\mu-\mathrm{can}}}} H^{p}(\Gamma(V, \Omega_{\aS_{K^p\Gamma(p^n)/L}}^{\bullet} \otimes_{\mathcal{O}_{\aS_{K^p\Gamma(p^n), L}}} D_{\lambda, \aS_{K^p\Gamma(p^n), L}}))$$ is finite-dimensional over $L$. Thus by Lemma \ref{tensor product}, we obtain $$H^i_{\mathrm{dR}}(\aS_{K^p\Gamma(p^n)}^{\mu-\mathrm{can}, \dagger}, D^{\dagger}_{\lambda, S_{K^p\Gamma(p^n)}}) = H^i_{\mathrm{dR}}(\aS_{K^p\Gamma(p^n), L}^{\mu-\mathrm{can}, \dagger}, D^{\dagger}_{\lambda, \aS_{K^p\Gamma(p^n), L}}) \otimes_L C$$ are finite-dimensional over $C$. \end{proof}

Let $P_{S_0} := \mathbb{Q}_p^{\times} \times \prod_{w \mid v, w \notin S_0} B_2(F_w) \times \prod_{w \mid v, w \in S_0} \mathrm{GL}_2(F_w) \subset G(\mathbb{Q}_p)$, where $B_2$ denotes the Borel subgroup of $\mathrm{GL}_2$ consisting of upper triangular matrices. For $\tau \in \Psi$, let $\chi_{\tau}$ be a one dimensional $C$-vector space with the trivial Hecke action, the trivial Galois action and an action of $P_{S_0}$ defined by $(g_0, (g_w)_{w \mid v}) \rightarrow \tau(d_{w_{\tau}})$. Here, $w_{\tau}$ denotes the unique place of $F$ satisfying $w_{\tau} \mid v$ and $\tau \in \mathrm{Hom}_{\mathbb{Q}_p}(F_{w_{\tau}}, L)$ and we put $g_{w_{\tau}} := \begin{pmatrix}
a_{w_{\tau}} & b_{w_{\tau}} \\
0 & d_{w_{\tau}} 
\end{pmatrix}.$ Let $i : \Fl^b \hookrightarrow \Fl$ be the natural closed immersion.

\vspace{0.5 \baselineskip}

By Lemma \ref{transitive}, we have a natural identification $P_{S_0} \setminus G(\mathbb{Q}_p) \Isom \Fl^b, \ P_{S_0}g \mapsto \infty g$. Let $\pi : G(\mathbb{Q}_p) \rightarrow P_{S_0} \setminus G(\mathbb{Q}_p) = \Fl^b$ be the natural projection. 

For a locally analytic Hausdorff LB representation $M$ of $P_{S_0}$, let $\mathcal{F}_{M}$ be the $G(\mathbb{Q}_p)$-equivariant sheaf on $\Fl^b$ (with respect to the discrete topology on $G(\mathbb{Q}_p)$) defined by $U \mapsto \{ \phi : \pi^{-1}(U) \rightarrow M \ \mathrm{locally \ analytic} \mid \phi(p \cdot ) = p \phi( \cdot ) \ \mathrm{for \ any \ } p \in P_{S_0} \}$. (See \cite[{\S} 2.1.2 $\sim$ 2.1.6]{PanII} for the definition of the locally analytic representation on LB spaces and locally analytic functions with values in LB spaces.) Then $\mathcal{F}_M(\Fl^b)$ is equal to the locally analytic induction $\mathrm{Ind}^{G(\mathbb{Q}_p)}_{P_{S_0}}(M)$ of $M$ from $P_{S_0}$ to $G(\mathbb{Q}_p)$. (See \cite[{\S} 2]{locallyanalyticinduction} for the definition of the locally analytic induction.) Note that $\mathcal{F}_{M}$ is a $\mathcal{O}_{\Fl}|_{\Fl^b}$-module.

\begin{lem}\label{induction1}

We have a natural $\mathcal{O}_{\Fl}|_{\Fl^b}$-linear isomorphism $\otimes_{\tau \in \Psi} (\omega_{\tau, \Fl}^{\lambda_{\tau}} \otimes (\wedge^2V_{\tau})^{\lambda_{\tau}})|_{\Fl^b} \Isom \mathcal{F}_{\otimes_{\tau \in \Psi} \chi_{\tau}^{\lambda_{\tau}}}$ of $G(\mathbb{Q}_p)$-equivariant sheaf.

\end{lem}

\begin{proof}

First note that we have the $\mathcal{O}_{\Fl, \infty}$-linear isomorphism $f_{\infty} : \otimes_{\tau \in \Psi} (\omega_{\tau, \Fl, \infty}^{\lambda_{\tau}} \otimes (\wedge^2V_{\tau})^{\lambda_{\tau}}) \Isom \mathcal{F}_{\otimes_{\tau \in \Psi} \chi_{\tau}^{\lambda_{\tau}}, \infty}$ defined by $\otimes_{\tau \in \Psi} e_{2, \tau}^{\lambda_{\tau}} \mapsto [((a_{0}, (\begin{pmatrix}
        a_{w} & b_{w} \\
        c_{w} & d_{w} 
        \end{pmatrix})_{w \mid v}) \mapsto \prod_{\tau \in \Psi}\tau(d_{w_{\tau}})^{\lambda_{\tau}})]$, where $[ \ \ ]$ denotes the equivalence class. By using this, we obtain a map $f : \otimes_{\tau \in \Psi} (\omega_{\tau, \Fl}^{\lambda_{\tau}} \otimes (\wedge^2V_{\tau})^{\lambda_{\tau}})|_{\Fl^b} \rightarrow \mathcal{F}_{\otimes_{\tau \in \Psi} \chi_{\tau}^{\lambda_{\tau}}}$ defined by $f(U)(w)(g) = f_{\infty}(g^*w_g)(\infty)$ for any $U \subset \Fl^b$ and any $w \in \otimes_{\tau \in \Psi} (\omega_{\tau, \Fl}^{\lambda_{\tau}} \otimes (\wedge^2V_{\tau})^{\lambda_{\tau}})|_{\Fl^b}(U)$. Here, $w_g$ denotes the image of $w$ in the stalk at $\infty g$. Note that by using the above same formula, $f_{\infty}$ can be extended to any open $U \subset \Fl^b \cap U_2$, which is equivariant with respect to the action of an open compact of $G(\mathbb{Q}_p)$. Thus $f(U)(w)$ is a locally analytic function. $f$ is a map of $G(\mathbb{Q}_p)$-equivariant sheaves, which is an isomorphism at $\infty$. This implies that $f$ is an isomorphism. \end{proof}

For a smooth representation $M$ of $P_{S_0}$, let $\mathcal{F}_{M}^{\mathrm{sm}}$ be the $G(\mathbb{Q}_p)$-equivariant sheaf on $\Fl^b$ defined by $U \mapsto \{ \phi : \pi^{-1}(U) \rightarrow M \ \mathrm{smooth} \mid \phi(p \cdot ) = p \phi( \cdot ) \ \mathrm{for \ any \ } p \in P_{S_0} \}$. Then $\mathcal{F}_M^{\mathrm{sm}}(\Fl^b)$ is equal to the smooth induction $\mathrm{Ind}^{G(\mathbb{Q}_p)}_{P_{S_0}}(M)^{\mathrm{sm}}$ of $M$ from $P_{S_0}$ to $G(\mathbb{Q}_p)$. Note that $\mathcal{F}_{M}$ is a $\mathcal{O}_{\Fl}|_{\Fl^b}$-module.

We put $H^i_{\mathrm{dR}}(\aS_{K^p}^{\mu-\mathrm{can}, \dagger}, D^{\dagger}_{\lambda}) := \varinjlim_{n}H^i_{\mathrm{dR}}(\aS_{K^p\Gamma(p^n)}^{\mu-\mathrm{can}, \dagger}, D^{\dagger}_{\lambda, \aS_{K^p\Gamma(p^n)}})$. This has a canonical LB space structure because $H^i_{\mathrm{dR}}(\aS_{K^p\Gamma(p^n)}^{\mu-\mathrm{can}, \dagger}, D^{\dagger}_{\lambda, \aS_{K^p\Gamma(p^n)}})$ is finite dimensional by Proposition \ref{finitedimensional}.  

\begin{lem}\label{induction2}

1 \ $\mathcal{H}^m(GDR^{\mathrm{sm}}_{\lambda})|_{\Fl^b} \cong \mathcal{F}_{H^m_{\mathrm{dR}}(\aS_{K^p}^{\mu-\mathrm{can}, \dagger}, D_{\lambda}^{\dagger})}^{\mathrm{sm}}$, where $\mathcal{H}^m$ denotes the $m$-th cohomology sheaf.

2 \ $H^m(\Fl^b, GDR^{\mathrm{sm}}_{\lambda}) \cong \mathrm{Ind}_{P_{S_0}}^{G(\mathbb{Q}_p)}(H^m_{\mathrm{dR}}(\aS_{K^p}^{\mu-\mathrm{can}, \dagger}, D_{\lambda}^{\dagger}))^{\mathrm{sm}}$.

3 \ Then the natural map $\mathcal{F}^{\mathrm{sm}}_{H^m_{\mathrm{dR}}(\aS_{K^p}^{\mu-\mathrm{can}, \dagger}, D_{\lambda}^{\dagger})} \otimes_C \mathcal{F}_{\otimes_{\tau \in \Psi} \chi_{\tau}^{\lambda_{\tau}}} \rightarrow \mathcal{F}_{H^m_{\mathrm{dR}}(\aS_{K^p}^{\mu-\mathrm{can}, \dagger}, D_{\lambda}^{\dagger}) \otimes_{C} (\otimes_{\tau \in \Psi} \chi_{\tau}^{\lambda_{\tau}})}$ is an isomorphism of $G(\mathbb{Q}_p)$-equivariant sheaves.

\end{lem}

\begin{proof} Note that $\Fl^b$ is a profinite space. This implies $H^m(\Fl^b, GDR^{\mathrm{sm}}_{\lambda}) \cong \Gamma(\Fl^b, \mathcal{H}^m(GDR^{\mathrm{sm}}_{\lambda}))$. Thus $2$ follows from 1. By the same way as Lemma \ref{induction1}, we obtain 1. By definition, we obtain $\mathcal{F}^{\mathrm{sm}}_{H^m_{\mathrm{dR}}(\aS_{K^p}^{\mu-\mathrm{can}, \dagger}, D_{\lambda}^{\dagger}), \infty} = \varinjlim_{K_p} \mathrm{Ind}^{K_p}_{K_p \cap P_{S_0}}(H^m_{\mathrm{dR}}(\aS_{K^p}^{\mu-\mathrm{can}, \dagger}, D_{\lambda}^{\dagger}))^{\mathrm{sm}} = H^m_{\mathrm{dR}}(\aS_{K^p}^{\mu-\mathrm{can}, \dagger}, D_{\lambda}^{\dagger})$ and $\mathcal{F}_{H^m_{\mathrm{dR}}(\aS_{K^p}^{\mu-\mathrm{can}, \dagger}, D_{\lambda}^{\dagger}) \otimes_{C} (\otimes_{\tau \in \Psi} \chi_{\tau}^{\lambda_{\tau}}), \infty} = \varinjlim_{K_p} \mathrm{Ind}^{K_p}_{K_p \cap P_{S_0}}(H^m_{\mathrm{dR}}(\aS_{K^p}^{\mu-\mathrm{can}, \dagger}, D_{\lambda}^{\dagger}) \otimes_{C} (\otimes_{\tau \in \Psi} \chi_{\tau}^{\lambda_{\tau}})) \cong H^m_{\mathrm{dR}}(\aS_{K^p}^{\mu-\mathrm{can}, \dagger}, D_{\lambda}^{\dagger}) \otimes_{C} \varinjlim_{K_p} \mathrm{Ind}^{K_p}_{K_p \cap P_{S_0}}(\otimes_{\tau \in \Psi} \chi_{\tau}^{\lambda_{\tau}})$. Note that the second equality comes from the fact that $H^m_{\mathrm{dR}}(\aS_{K^p}^{\mu-\mathrm{can}, \dagger}, D_{\lambda}^{\dagger})$ is a countable filtered colimit of finite dimensional subspaces. This implies the result. \end{proof}

\begin{thm} \label{ordinary cohomology}
    
We have the following isomorphisms as $G_L \times G(\mathbb{Q}_p) \times \mathbb{T}^S$-modules.

    1 \ $\mathcal{H}^m(GDR_{\lambda}^{\Psi-\mathrm{la}})_{\infty} \cong H^m_{\mathrm{dR}}(\aS_{K^p}^{\mu-\mathrm{can}, \dagger}, D_{\lambda}^{\dagger}) \otimes (\otimes_{\tau \in \Psi} (\omega_{\tau, \Fl}^{\lambda_{\tau}} \otimes (\wedge^2 V_{\tau})^{\lambda_{\tau}} ))_{\infty}$ for any $m$.
    
    2 \ $H^m(\Fl, i_*i^*GDR_{\lambda}^{\Psi-\mathrm{la}}) \cong \mathrm{Ind}^{G(\mathbb{Q}_p)}_{P_{S_0}}(H^m_{\mathrm{dR}}(\aS_{K^p}^{\mu-\mathrm{can}, \dagger}, D_{\lambda}^{\dagger}) \otimes (\otimes_{\tau \in \Psi} \chi_{\tau}^{\lambda_{\tau}}))$ for any $m$.
    
    \end{thm}

    \begin{proof} 1 \ By Corollary \ref{topology} and the construction of $GDR^{\Psi-\mathrm{la}}_{\lambda}$, we obtain $\mathcal{H}^m(GDR_{\lambda}^{\Psi-\mathrm{la}})_{\infty} = \varinjlim_{\infty \in U} H^m(GDR_{\lambda}^{\Psi-\mathrm{la}}(U)) = \varinjlim_{\infty \in U, n, V_n \supset \overline{\aS_{K^p\Gamma(p^n)}^{\mu-\mathrm{can}}}} H^i(GDR_{\lambda, \aS_{K^p\Gamma(p^n)}}(V_n) \widehat{\otimes}_{C} (\otimes_{\tau} \omega_{\tau, \Fl}^{\lambda_{\tau}} \otimes (\wedge^2 V_{\tau})^{\lambda_{\tau}})(U)).$  Note that $(\otimes_{\tau} (\omega_{\tau, \Fl}^{\lambda_{\tau}}\otimes (\wedge^2 V_{\tau})^{\lambda_{\tau}} ))$ is defined over $L$ and the complex $\varinjlim_{V_n \supset \overline{\aS_{K^p\Gamma(p^n)}^{\mu-\mathrm{can}}}} GDR_{\lambda, \aS_{K^p\Gamma(p^n)}}(V_n)$ satisfies the assumptions of Lemma \ref{tensor product} by Proposition \ref{finitedimensional}. Thus we obtain the property 1.
        
    2 \ Since $\Fl^{b}$ is a profinite space, we have $H^m(\Fl, i_*i^*GDR_{\lambda}^{\Psi-\mathrm{la}}) = H^m(\Fl^b, i^*GDR_{\lambda}^{\Psi-\mathrm{la}}) = \Gamma(\Fl^b, \mathcal{H}^m(i^*GDR_{\lambda}^{\Psi-\mathrm{la}})))$. Therefore, it suffices to construct $\mathcal{H}^m(i^*GDR_{\lambda}^{\Psi-\mathrm{la}}) \cong \mathcal{F}_{M}$, where $M := H^m_{\mathrm{dR}}(\aS_{K^p}^{\mu-\mathrm{can}, \dagger}, D_{\lambda}^{\dagger}) \otimes (\otimes_{\tau \in \Psi} \chi_{\tau}^{\lambda_{\tau}})$. Let $N := \otimes_{\tau \in \Psi} \chi_{\tau}^{\lambda_{\tau}}$. First we have a canonical map $\mathcal{H}^m(GDR^{\mathrm{sm}}_{\lambda})|_{\Fl^b} \otimes_C (\otimes_{\tau \in \Psi} (\omega_{\tau, \Fl}^{\lambda_{\tau}} \otimes (\wedge^2 V_{\tau})^{\lambda_{\tau}}))|_{\Fl^b} \rightarrow \mathcal{H}^m(GDR^{\Psi-\mathrm{la}}_{\lambda})|_{\Fl^b}$ of $G(\mathbb{Q}_p)$-equivariant sheaves, which is an isomorphism because this is isomorphism at $\infty$ by 1. By Lemmas \ref{induction1} and \ref{induction2}, $\mathcal{H}^m(GDR^{\mathrm{sm}}_{\lambda})|_{\Fl^b} \otimes_C (\otimes_{\tau \in \Psi} (\omega_{\tau, \Fl}^{\lambda_{\tau}} \otimes (\wedge^2 V_{\tau})^{\lambda_{\tau}}))|_{\Fl^b} \cong \mathcal{F}_M$. Thus we obtain the result. \end{proof}

\begin{rem}

It seems that $\aS_{K^p\Gamma(p^n)}^{\mu-\mathrm{can}}$ is closely related to the Igusa varieties. However, we don't need such facts. Thus we don't pursue this problem in this paper.

\end{rem}

\subsection{Basic locus}

Let $b \in B(G, \mu^{-1})$ be the basic element and $j : \Fl^{b} \hookrightarrow \Fl$ be the natural open immersion. Then we can essentially regard $\aS_{K^p}^b$ as the basic Rapoport-Zink space of infinite level associated with $(G, b, \mu)$ by the basic uniformization (Theorem \ref{uniformization}) and we can calculate $GDR^{\Psi-\mathrm{la}}_{\lambda}|_{\Fl^{b}}$ by the duality theorem of basic Rapoport-Zink spaces (Theorem \ref{duality}). 

Let $\mathbb{X}_{b}$ be the $G$-$p$-divisible group over $\overline{\mathbb{F}}_L$ corresponding to $b$ and $\mathcal{M}_{\mathcal{O}_{\breve{L}}}$ be the basic Rapoport-Zink space over $\breve{L}$ defined by $(\mathcal{O}_{B_p} , V_p^o, \mathbb{X}_b, \mu).$ (See \cite[{\S} 6.5]{p-div} for precise definitions and note that this can be regarded as an EL-type.) Let $\mathcal{M}_{K_p, \breve{L}}$ denote the finite $\etale$ cover of the generic fiber $\mathcal{M}_{\breve{L}}$ of $\mathcal{M}_{\mathcal{O}_{\breve{L}}}$ defined by $K_p \subset K_p^o$ and $\mathcal{M}_{K_p} = \mathcal{M}_{K_p,\breve{L}} \times_{\breve{L}} C$. These have natural actions of $J_b(\mathbb{Q}_p)$ and $J_b(\mathbb{Q}_p)$ has a form $\mathbb{Q}_p^{\times} \times \prod_{i} \mathrm{GL}_2(F_i) \times \prod_{j} D_j^{\times}$ by Proposition \ref{calculation of isocrystals}, where $F_j$ are finite extensions of $\mathbb{Q}_p$ and $D_j$ are central division algebras over finite extensions of $\mathbb{Q}_p$ of dimension $4$.

The decomposition $B_p \otimes_{\mathbb{Q}_p} \breve{L} = \prod_{\tau \in \Phi} M_2(\breve{L})$ induces a decomposition of the $G$-Dieudonne module corresponding to the universal $G$-$p$-divisible group on $\mathcal{M}_{K_p}$: $D(\mathbb{X}_{b}) \otimes_{\breve{\mathbb{Q}}_p} \mathcal{O}_{\mathcal{M}_{K}} = \oplus_{\tau \in \Phi} D(\mathbb{X}_{b})_{\tau}' \otimes_{\breve{L}} \mathcal{O}_{\mathcal{M}_{K}}$. Note that, we have an identification $D(\mathbb{X}_{b}) = V_p \otimes_{\mathbb{Q}_p} \breve{\mathbb{Q}}_p$ as $B_p \otimes_{\mathbb{Q}_p} \breve{\mathbb{Q}}_p$-modules and $D(\mathbb{X}_{b})$ has a natural action of $J_b(\mathbb{Q}_p)$. By Morita equivalence, $D(\mathbb{X}_{b})_{\tau}'$ induces a 2-dimensional $\breve{L}$-vector space $D(\mathbb{X}_{b})_{\tau}$ with a natural action of $J_b(\mathbb{Q}_p)$ and the Grothendieck-Messing filtration induces an exact sequence \begin{equation}\label{GMfil}0 \rightarrow \omega_{\tau, \mathcal{M}_{K_p}} \otimes (\wedge^2 D(\mathbb{X}_{b})_{\tau}) \rightarrow D(\mathbb{X}_{b})_{\tau} \otimes_{\breve{L}} \mathcal{O}_{\mathcal{M}_{K_p}} \rightarrow \omega_{\tau, \mathcal{M}_{K_p}}^{-1} \rightarrow 0 \end{equation} for any $\tau \in \Psi$, where $\omega_{\tau, \mathcal{M}_{K_p}}$ denotes the $\tau$-component of the dual of the Lie algebra of universal $G$-$p$-divisible group similarly as in the global situation. Let $\Fl_{\mathrm{GM}, \mathcal{M}}$ be the flag variety which parametrizes 1-dimensional quotients of $D(\mathbb{X}_{b})_{\tau} \otimes_{\breve{L}} C$ for all $\tau \in \Psi$. Thus we obtain the map $\pi_{\mathrm{GM}, \mathcal{M}} : \mathcal{M}_{K_p} \rightarrow \Fl_{\mathrm{GM}, \mathcal{M}}$, which is $\etale$ by the Grothendieck-Messing theory and is $J_b(\mathbb{Q}_p)$-equivariant. (See \cite[Proposition 23.3.3]{p-adicgeometry} for details.)

\begin{thm} (\cite[Theorem 6.5.4]{p-div}.)\label{Rapoport prefd}

There exists a perfectoid space $\mathcal{M}_{\infty}$ over $C$ unique up to isomorphism such that $\mathcal{M}_{\infty} \sim \varprojlim_{K_p} \mathcal{M}_{K_p}$.

\end{thm}

On the other hand, the Hodge-Tate filtration on $\mathcal{M}_{\infty}$ induces the filtration \begin{equation}\label{local Hodge-Tate} 0 \rightarrow \omega_{\tau, \mathcal{M}_{\infty}}^{-1}(1) \rightarrow V_{\tau} \otimes_{L} \mathcal{O}_{\mathcal{M}_{\infty}} \rightarrow \omega_{\tau, \mathcal{M}_{\infty}} \otimes (\wedge^2 D(\mathbb{X}_{b})_{\tau}) \rightarrow 0 \end{equation} for any $\tau \in \Psi$. (See \cite[Proposition 7.1.1]{p-div}.)  Let $\Fl_{\mathrm{HT}, \mathcal{M}}$ be the flag variety which parametrizes 1-dimensional quotients of $V_{\tau} \otimes_{L} C$ for all $\tau \in \Psi$. Let $\pi_{\mathrm{HT}, \mathcal{M}} : \mathcal{M}_{\infty} \rightarrow \mathrm{\Fl}_{\mathrm{HT}, \mathcal{M}}$ be the induced map.

\vspace{0.5 \baselineskip}

Let $\check{\mathcal{M}}_{\check{K}_p, \breve{L}}$ be the Rapoport-Zink space of EL-type of level $\check{K}_p$ over $\breve{L}$ given by the dual EL datum $(\check{G}, \check{b}, \check{\mu})$ of $(G, b, \mu)$. (See \cite[{\S} 7.2]{p-div} for more details.) Let $\check{G}$ and $J_{\check{b}}$ be the reductive groups associated with $\check{\mathcal{M}}_{\check{K}_p, \check{L}}$ as in the case $\mathcal{M}$. Then we have canonical identifications $\check{G} = J_b$ and $G = J_{\check{b}}$. Note that for $\check{\mathcal{M}}_{\check{K}_p, \breve{L}}$, Theorem \ref{Rapoport prefd} is also true i.e., there exists a perfectoid space $\check{\mathcal{M}}_{\infty}$ over $C$ such that $\check{\mathcal{M}}_{\infty} \sim \varprojlim_{\check{K}_P} \check{\mathcal{M}}_{\check{K}_p}$, where $\check{\mathcal{M}}_{\check{K}_p} := \check{\mathcal{M}}_{\check{K}_p, \breve{L}} \times_{\breve{L}} C$. Moreover, we also have the Grothendieck-Messing filtration and the Hodge-Tate filtration, which again give filtrations on rank 2 vector bundles by rank 1 bundles for any $\tau \in \Psi$. Moreover we have natural identifications $\Fl_{\mathrm{GM}, \mathcal{M}} = \Fl_{\mathrm{HT}, \check{\mathcal{M}}} = \prod_{\tau \in \Psi} \mathbb{P}^1_{C, \tau}$ and $\Fl_{\mathrm{HT}, \mathcal{M}} = \Fl_{\mathrm{GM}, \check{\mathcal{M}}} = \prod_{\tau \in \Psi} \mathbb{P}^1_{C, \tau}$, which is compatible with actions of $G(\mathbb{Q}_p)$ and $J_b(\mathbb{Q}_p)$ by \cite[Proposition 7.2.2]{p-div}.

\begin{thm}(Duality of basic Rapoport-Zink spaces, \cite[Theorem 7.2.3]{p-div}) \label{duality}

 There exists a canonical $G(\mathbb{Q}_p) \times J_b(\mathbb{Q}_p)$-equivariant isomorphism $\mathcal{M}_{\infty} \cong \check{\mathcal{M}}_{\infty}$ inducing identifications $\pi_{\mathrm{GM}, \mathcal{M}} = \pi_{\mathrm{HT}, \check{\mathcal{M}}}$ and $\pi_{\mathrm{HT}, \mathcal{M}} = \pi_{\mathrm{GM}, \check{\mathcal{M}}}$ under the above identifications $\Fl_{\mathrm{HT}, \mathcal{M}} = \Fl_{\mathrm{GM}, \check{\mathcal{M}}}$ and $\Fl_{\mathrm{GM}, \mathcal{M}} = \Fl_{\mathrm{HT}, \check{\mathcal{M}}}$.
    
\end{thm}

In the following, we will prove variants of results in {\S} 3.4 and 3.5, 4.1 and 4.2 for $\mathcal{M}_{\infty}$ and $\check{\mathcal{M}}_{\infty}$. Actually, we can give purely local proofs by the same way as in the global case, but for simplicity and completeness of proofs, we will use global results proved in {\S} 3.4 and 3.5. Note that by the following lemma, the dual Rapoport-Zink space $\check{\mathcal{M}}_{\check{K}_p}$ also can be obtained from a global unitary Shimura datum and its basic element. Thus we can also use global unitary Shimura variety to study $\check{\mathcal{M}}$.

\begin{lem}\label{division unitary}

There exists a set of finite places $S(\check{B})$ of $F$ (maybe containing $p$-adic places) satisfying the following. 

1 \ Every prime lying below $S(\check{B})$ splits in $F_0$.

2 \ $S(\check{B})^c = S(\check{B})$.

3 \ $\frac{1}{2}|S(\check{B})| + d = 0 \ \mod 2$.

4 \ Let $GU'/\mathbb{Q}$ be the unitary similitude group given by the data $S(\check{B})$ and $\Psi$ via Proposition \ref{unitary groups}, let $\mu' : \mathbb{G}_{m, \overline{\mathbb{Q}}_p} \rightarrow GU'_{\overline{\mathbb{Q}}_p}$ be the minuscule induced from the minuscule of $GU'$ and $\iota$ and let $b' \in B(GU'_{\mathbb{Q}_p}, \mu')$ be the basic element. Then we have an identification $(GU'_{\mathbb{Q}_p}, b', \mu') = (\check{G}, \check{b}, \check{\mu})$.

\end{lem}

\begin{rem}

Note that the above unitary group might not be studied in {\S} 3.4 and 3.5 because a factor of $\check{G}(\mathbb{Q}_p) = J_b(\mathbb{Q}_p)$ may be the unit group of a central division algebra over a finite extension of $\mathbb{Q}_p$ of dimension $4$. However, almost all results holds in {\S} 3.4 and 3.5 for such a unitary group. The main difference is that we may not have Lemma \ref{open unit ball} for such a unitary group. However, we can actually prove almost all results holds in {\S} 3.4 and 3.5 for such a unitary group without Lemma \ref{open unit ball} by using $g_{1, \tau}$ and $g_{2, \tau}$ appearing after Proposition \ref{local Sen operator} instead of $e_{1, \tau}$ and $e_{2, \tau}$. We will see how to modify results in the following.

\end{rem}

\begin{proof}

Take $S(\check{B})$ satisfying 1, 2 and 3 such that $S(\check{B}) \cap \{ w \mid v \}$ is equal to the set of $p$-adic places $w$ of $F$ dividing $v$ such that the $w$-factor of $\check{G}(\mathbb{Q}_p)$ is the unit group of a central division algebra over $F_w$ of dimension $4$. Then we obtain the property 4 because the basic element is unique in $B(GU'_{\mathbb{Q}_p}, \mu')$ and the minuscule $\check{\mu}$ is given by $\iota$ and the minuscule of $GU$.  \end{proof}

We recall relations between basic Rapoport-Zink spaces and Shimura varieties.

By Proposition \ref{unitary groups}, we have a unitary group $I/F^+$ such that $GI(\mathbb{A}_{\mathbb{Q}}^{\infty, p}) \cong GU(\mathbb{A}_{\mathbb{Q}}^{\infty, p})$ and $GI(\mathbb{Q}_p) = J_b(\mathbb{Q}_p)$ and $GI(\mathbb{R}) = G(\prod_{\tau \in \Phi} U(0, 2))$ and such a group is unique up to isomorphism. Let $GI/\mathbb{Q}$ be the unitary similitude groups induced from $I$. 

%\footnote{Note that $GI$ is defined up to outer automorphisms defined by conjugations as in Proposition \ref{unitary groups}. This doesn't matter because in the following, we only consider representations at places of $F$ splitting over $F^+$ and thus such conjugations are inner automorphisms.} 

\begin{lem}\label{properly discontinuous}

1 \ The left action of $GI(\mathbb{Q})$ on $\mathcal{M} \times GI(\mathbb{A}^{\infty, p})/K^p$ has the following property.

For any quasicompact open $U$ of $\mathcal{M}$ and any $x \in GI(\mathbb{A}_{\mathbb{Q}}^{\infty, p})$, there exists an open compact $K^p_1$ of $K^p$ such that for any $1 \neq g \in GI(\mathbb{Q})$, we have $g(U \times \{xK^p_1\}) \cap (U \times \{xK^p_1\}) = \emptyset$.

2 \ The right action of $J_b(\mathbb{Q}_p)$ on $\mathcal{M} \times [GI(\mathbb{Q}) \setminus GI(\mathbb{A}_{\mathbb{Q}}^{\infty})/K^p]$ has the following property.
 
For any quasicompact open $U$ of $\mathcal{M}$ and any $x \in GI(\mathbb{A}_{\mathbb{Q}}^{\infty})$, there exists an open compact $K^p_1$ of $K^p$ such that for any $1 \neq g \in J_b(\mathbb{Q}_p)$, we have $(U \times \{GI(\mathbb{Q})xK^p_1 \})g \cap (U \times \{GI(\mathbb{Q})xK^p_1 \}) = \emptyset$.

\end{lem}

\begin{rem}

The author doesn't know whether this property is true or not without taking $K^p_1$.

\end{rem}

\begin{proof}

See \cite[proofs of Theorem 6.23 and Proposition 2.37]{RZ}. \end{proof}

We consider the perfectoid space $\mathcal{S} := \varprojlim_{K^p_1} \mathcal{S}_{K^p_1}$. We put $\pi_{\mathrm{HT}, \mathcal{S}} : \mathcal{S} \rightarrow \mathcal{S}_{K^p} \rightarrow \Fl$.

\begin{thm} (Basic uniformization, \cite[Theorem 6.30]{RZ}) \label{uniformization}

1 \ There exist natural identifications $\pi_{\mathrm{HT}, \mathcal{S}}^{-1}(\Fl^{b}) = GI(\mathbb{Q}) \setminus \varprojlim_{K^p_1}[\mathcal{M}_{\infty} \times GI(\mathbb{A}_{\mathbb{Q}}^{p, \infty})/ K^p_1] = [\varprojlim_{K^p_1} \mathcal{M}_{\infty} \times [GI(\mathbb{Q}) \setminus GI(\mathbb{A}_{\mathbb{Q}}^{\infty})/ K^p_1]]/J_b(\mathbb{Q}_p)$. Here, we consider the quotients in the naive sense by using Lemma \ref{properly discontinuous}.

2 \ More precisely, for any quasicompact open $U$ of $\mathcal{M}$ and for any $x \in GI(\mathbb{A}_{\mathbb{Q}}^{\infty})$, there exists an open compact subgroup $K_1^p$ of $K^p$ and a quasicompact open $V$ of $\aS_{K^p_1K_p^o}^{b}$ such that for any $K_p$, the above isomorphism descends to an isomorphism $V_{K_p} \cong [U_{K_p} \times \{ GI(\mathbb{Q}) x K^p_1\}]J_b(\mathbb{Q}_p)/J_b(\mathbb{Q}_p) \cong U_{K_p}$, where $V_{K_p}$ (resp. $U_{K_p}$) denotes the inverse images of $V$ (resp. $U$) in $\aS_{K^pK_p}$ (resp. $\mathcal{M}_{K_p}$).

\end{thm}

\begin{rem} By the above result 2, the local structure of $\mathcal{M}_{\infty}$ is identified with that of $\aS_{K^p_1}$ for some $K_1^p$, which was already studied. \end{rem}

%In the following, let $\mathcal{N}$, $\mathcal{N}_{K_p}$, and $\mathcal{N}_{\infty}$ be $\mathcal{M}$, $\mathcal{M}_{K_p}$ and $\mathcal{M}_{\infty}$ respectively or $\check{\mathcal{M}}$, $\check{\mathcal{M}}_{K_p}$ and $\check{\mathcal{M}}_{\infty}$ respectively. We also use the same notations for $\check{\mathcal{M}}$, $\check{\mathcal{M}}_{K_p}$ and $\check{\mathcal{M}}_{\infty}$ such as $\mathbb{X}_{b}$ and 

In the following to Proposition \ref{anti trivial}, we always consider the action of $G(\mathbb{Q}_p)$ and don't consider the action of $J_b(\mathbb{Q}_p)$. Let $U$ be a sufficiently small affinoid open of $\Fl_{\mathrm{HT}, \mathcal{M}}$ such that for any $\tau \in \Psi$, there exists a basis $g_{1, \tau}, g_{2, \tau}$ of $H^0(\Fl_{\mathrm{HT}, \mathcal{M}}, \omega_{\Fl_{\mathrm{HT}, \mathcal{M}}} \otimes \wedge^2V_{\tau})$ such that $g_{1, \tau}|_{U}$ is a generator of $(\omega_{\Fl_{\mathrm{HT}, \mathcal{M}}} \otimes \wedge^2V_{\tau})|_{U}$. (Precisely, the sufficiently smallness means that $U$ is an element of $\mathcal{B}$ as in {\S} 3.4 for some unitary Shimura variety related to $\mathcal{M}$ in the sense of Theorem \ref{uniformization}.) Let $V$ be a small affinoid open of $\mathcal{M}_{K_p}$ such that the inverse image $V_{\infty} = \mathrm{Spa}(B, B^+)$ of $V$ in $\mathcal{M}_{\infty}$ is an affinoid perfectoid open of $\mathcal{M}_{\infty}$ and $V \subset \pi_{\mathrm{HT}, \mathcal{M}}^{-1}(U)$.

The following results are variants of results in {\S} 3.3. and 3.4.

\begin{prop}

We have $R^i\Psi\mathfrak{LA}(B) = 0$ for any $i > 0$.

\end{prop}

\begin{proof} This follows from Theorem \ref{uniformization} and Theorem \ref{geometric sen theory}. We can also give a purely local proof by the same argument as Theorem \ref{geometric sen theory}. \end{proof}

Let $\mathfrak{n}^{o}$ (resp. $\mathfrak{b}^o$) be the $G_{C}$-equivariant vector bundle on $\Fl_{\mathrm{HT}, \mathcal{M}}$ defined by the Lie algebra $\mathfrak{n} := \prod_{\tau \in \Psi} \mathfrak{n}_{C}$ and $\mathfrak{b} := \prod_{\tau \in \Psi} \mathfrak{b}_{C}$, where $\mathfrak{n}_{C}$ (resp. $\mathfrak{b}_{C}$) denotes the Lie subalgebra of $\prod_{\tau \in \Psi} \mathfrak{gl}_{2,C}$ consisting of the upper triangle nilpotent (resp. upper triangle) matrices. (See Definition \ref{nilpotent} for details.) Note that $\mathcal{O}_{\Fl_{\mathrm{HT}, \mathcal{M}}} = \mathcal{O}_{\Fl_{\mathrm{HT}, \mathcal{M}}}^{\Psi-\mathrm{la}}$ by 2 of Example \ref{example}. Thus we have a natural map $\mathcal{O}_{\Fl_{\mathrm{HT}, \mathcal{M}}}(U) \rightarrow \mathcal{O}_{\mathcal{M}_{\infty}}^{\Psi-\mathrm{la}}(V_{\infty})$ and this induces actions of $\mathfrak{n}^0(U)$, $\mathfrak{b}^0(U)$ and $\prod_{\tau \in \Psi} \mathfrak{gl}_2(L) \otimes_L \mathcal{O}_{\Fl_{\mathrm{HT}, \mathcal{M}}}(U)$ on $\mathcal{O}_{\mathcal{M}_{\infty}}^{\Psi-\mathrm{la}}(V_{\infty})$.

\begin{prop}\label{local Sen operator}

The action of $\mathfrak{n}^o(U)$ on $\mathcal{O}_{\mathcal{M}_{\infty}}^{\Psi-\mathrm{la}}(V_{\infty})$ is trivial.

\end{prop}

\begin{proof} This follows from Theorem \ref{uniformization} and Theorem \ref{sen operator}. Again, we can also give a purely local proof. \end{proof}

After shrinking $K_p$ if necessary, we may assume that $G_0 := K_p$ is sufficiently small as after Lemma \ref{tensor product}. Let $G_m := (G_0)^{p^m}$ for $m \in \mathbb{Z}_{\ge 0}$ and $V_{G_m}$ be the inverse image of $V$ in $\mathcal{M}_{G_m}$. (Note that $G_m$ is an open subgroup of $G_0$. See the discussions after Lemma \ref{tensor product}.) Let $x_{\tau} := g_{2, \tau}/g_{1, \tau} \in \mathcal{O}_{\Fl_{\mathrm{HT}, \mathcal{M}}}(U)$ and $f_{\tau}$ denote a basis of $\wedge^2V_{\tau}(-1)$ for any $\tau \in \Psi$. We can regard $f_{\tau}$ as a generator of $\wedge^2D(\mathbb{X}_b)_{\tau} \otimes \mathcal{O}_{\mathcal{M}_{\infty}}$ via $\wedge^2D(\mathbb{X}_b)_{\tau} \otimes \mathcal{O}_{\mathcal{M}_{\infty}} \cong \wedge^2V_{\tau}(-1) \otimes_L \mathcal{O}_{\mathcal{M}_{\infty}}$ coming from the above (\ref{local Hodge-Tate}).

Take an increasing sequence of positive integers $r(1) < r(2) < \cdots < r(n) < \cdots$ and take $x_{\tau, n} \in \mathcal{O}_{\mathcal{M}_{G_{G_{r(n)}}}}(V_{G_{r(n)}})$, $g_{\tau, 1,n} \in \omega_{\tau, \mathcal{M}_{G_{G_{r(n)}}}} \otimes \wedge^2D(\mathbb{X}_b)_{\tau}(1) \otimes \mathcal{O}_{\mathcal{M}_{G_{G_{r(n)}}}}(V_{G_{r(n)}})$ and $f_{\tau, n} \in \wedge^2D(\mathbb{X}_b)_{\tau} \otimes \mathcal{O}_{\mathcal{M}_{G_{r(n)}}}(V_{G_{r(n)}})$ satisfying $|| x_{\tau} - x_{\tau, n} ||, || \frac{g_{\tau, 1,n}}{g_{\tau, 1}} - 1 ||, || \frac{f_{\tau, n}}{f_{\tau}} - 1 || \le p^{-n-1}$. Moreover, after replacing $r(n)$, we may assume $|| x_{\tau} - x_{\tau, n} || = || x_{\tau} - x_{\tau, n} ||_{G_{r(n)}}, || \frac{g_{\tau, 1,n}}{g_{\tau, 1}} - 1 || = || \frac{g_{\tau, 1, n}}{g_{\tau, 1}} - 1 ||_{G_{r(n)}}$ and $|| \frac{f_{\tau, n}}{f_{\tau}} - 1 || = || \frac{f_{\tau, n}}{f_{\tau}} - 1 ||_{G_{r(n)}}$ by Lemma \ref{norm}. Then $g_{\tau, 1, n}$ is a generator of $\omega_{\tau, \mathcal{M}_{\infty}}(V_{\infty}) \otimes \wedge^2V_{\tau}$. Let $\mathcal{O}_{\mathcal{M}_{G_{r(n)}}}(V_{G_{r(n)}})\lbrace x_{\tau} - x_{\tau, n}, \mathrm{log}\frac{g_{1, \tau}}{g_{1, \tau, n}}, \mathrm{log}\frac{f_{\tau}}{f_{\tau, n}} \rbrace_{\tau} := \{ f = \sum_{(i, j, k) \in (\mathbb{Z}^{\Psi}_{\ge 0})^{3}} a_{i, j, k} \prod_{\tau \in \Psi}(\mathrm{log}\frac{g_{1, \tau}}{g_{1, \tau, n}})^{i_{\tau}} (\mathrm{log}\frac{f_{\tau}}{f_{\tau, n}})^{j_{\tau}} (x_{\tau}-x_{\tau, n})^{k_{\tau}} \in \mathcal{O}_{\mathcal{M}_{\infty}}(V_{\infty})^{G_{r(n)}-\Psi-\mathrm{an}} \mid a_{i, j, k} \in \mathcal{O}_{\mathcal{M}_{G_{r(n)}}}(V_{G_{r(n)}}) \ \mathrm{s. t.} \ \mathrm{sup}_{i,j,k}|| a_{i, j, k} ||p^{-n(\sum_{\tau}(i_{\tau} + j_{\tau} + k_{\tau}))} < \infty \}$.

This is a Banach space by the norm $|| f ||_{x_{\tau}, g_{1, \tau}, f_{\tau}} := \mathrm{sup}_{i, j, k} || a_{i, j, k} ||p^{-(n+1)(\sum_{\tau}(i_{\tau} + j_{\tau} + k_{\tau}))}$ and by this Banach space structure, the natural inclusion $\mathcal{O}_{\mathcal{M}_{G_{r(n)}}}(V_{G_{r(n)}})\lbrace x_{\tau} - x_{\tau, n}, \mathrm{log}\frac{e_{1, \tau}}{e_{1, \tau, n}}, \mathrm{log}\frac{f_{\tau}}{f_{\tau, n}} \rbrace_{\tau} \hookrightarrow \mathcal{O}_{\mathcal{M}_{\infty}}(V_{\infty})^{G_{r(n)}-\Psi-\mathrm{an}}$ is continuous.

\begin{prop} \label{local expansion formula}

There exists an integer $m$ such that for any $n \ge m$, the natural inclusion $$\mathcal{O}_{\mathcal{M}_{\infty}}(V_{\infty})^{G_{0}-\Psi-\mathrm{an}} \hookrightarrow \mathcal{O}_{\mathcal{M}_{\infty}}(V_{\infty})^{G_{r(n)}-\Psi-\mathrm{an}}$$ factors through a continuous injection $$\mathcal{O}_{\mathcal{M}_{\infty}}(V_{\infty})^{G_{0}-\Psi-\mathrm{an}} \hookrightarrow \mathcal{O}_{\mathcal{M}_{G_{r(n)}}}(V_{G_{r(n)}})\lbrace x_{\tau} - x_{\tau, n}, \mathrm{log}\frac{g_{1, \tau}}{g_{1, \tau, n}}, \mathrm{log}\frac{f_{\tau}}{f_{\tau, n}} \rbrace_{\tau}.$$
        
\end{prop}

\begin{proof} The same proof as Proposition \ref{mikami expansionII} works by Proposition \ref{local Sen operator}. \end{proof}

Let $\mathfrak{h} = \prod_{\tau \in \Psi} \mathfrak{h}_{\tau}$ be the Lie subalgebra of $\prod_{\tau \in \Psi} \mathfrak{gl}_2(C)$ consisting of diagonal matrices. Then we have a canonical isomorphism $\mathfrak{h} \cong H^0(\Fl_{\mathrm{HT}, \mathcal{M}}, \mathfrak{b}^0/\mathfrak{n}^0)$ by using a filtration $0 \rightarrow \omega_{\tau, \Fl_{\mathrm{HT}, \mathcal{M}}}^{-1} \rightarrow V_{\tau} \otimes \mathcal{O}_{\Fl_{\mathrm{HT}, \mathcal{M}}} \rightarrow \omega_{\tau, \Fl_{\mathrm{HT}, \mathcal{M}}} \otimes \wedge^2V_{\tau} \rightarrow 0$ and this acts on $\mathcal{O}_{\mathcal{M}_{\infty}}^{\Psi-\mathrm{la}}$ by Proposition \ref{local Sen operator}. We write this action for $\theta_{\mathfrak{h}, \mathcal{M}}$. As in the global situation, every element $(a_{\tau}, b_{\tau})_{\tau} \in (\mathbb{Z}^2)^{\Psi}$ gives a character $(a_{\tau}, b_{\tau})_{\tau} : \mathfrak{h} \rightarrow C$. Some of the following results hold in the general weight case, but we only need the trivial weight case.

\begin{prop} \label{local expansion formula 2}

There exists an integer $m$ such that for any $n \ge m$, the natural inclusion $$\mathcal{O}_{\mathcal{M}_{\infty}}(V_{\infty})^{G_{0}-\Psi-\mathrm{an}, (0, 0)_{\tau}} \hookrightarrow \mathcal{O}_{\mathcal{M}_{\infty}}(V_{\infty})^{G_{r(n)}-\Psi-\mathrm{an}, (0, 0)_{\tau}}$$ factors through a continuous injection $$\mathcal{O}_{\mathcal{M}_{\infty}}(V_{\infty})^{G_{0}-\Psi-\mathrm{an}, (0, 0)_{\tau}} \hookrightarrow \mathcal{O}_{\mathcal{M}_{G_{r(n)}}}(V_{G_{r(n)}})\lbrace x_{\tau} - x_{\tau, n} \rbrace_{\tau}.$$ Here, we use the same notation as Proposition \ref{mikami expansionIII}.

\end{prop}

\begin{proof} This follows from Proposition \ref{local expansion formula}. \end{proof}

The following results are local variants of results in {\S} 4.1 and 4.2. In the following, we also use the same notation $\mathcal{F}$ for the inverse image via $\pi_{\mathrm{HT}, \mathcal{M}}$ (resp. $\pi_{\mathrm{GM}, \mathcal{M}}$) of a sheaf $\mathcal{F}$ on $\Fl_{\mathrm{HT}, \mathcal{M}}$ (resp. $\Fl_{\mathrm{GM}, \mathcal{M}}$).

\begin{prop}\label{local de Rham complex}

There exists a unique continuous $G(\mathbb{Q}_p)$-equivariant and $\mathcal{O}_{\Fl_{\mathrm{HT}, \mathcal{M}}}$-linear derivation $d^{\Psi-\mathrm{la}}_{\mathcal{M}} : \mathcal{O}_{\mathcal{M}_{\infty}}^{\Psi-\mathrm{la}, (0, 0)_{\tau}} \rightarrow \mathcal{O}_{\mathcal{M}_{\infty}}^{\Psi-\mathrm{la}, (0, 0)_{\tau}} \otimes_{\mathcal{O}_{\mathcal{M}_{\infty}}^{\mathrm{sm}}} \Omega_{\mathcal{M}_{\infty}}^{1, \mathrm{sm}}$ such that $d^{\Psi-\mathrm{la}}|_{\mathcal{O}_{\mathcal{M}_{\infty}}^{\mathrm{sm}}}$ is equal to the usual derivation of finite levels. 

\end{prop}

\begin{proof} The same proof as Lemma \ref{locally analytic extension of derivation} works in our local situation by using Proposition \ref{local expansion formula 2}.  \end{proof}

As after Lemma \ref{locally analytic extension of derivation}, by Proposition \ref{local de Rham complex}, we obtain the locally analytic extension $GDR^{\Psi-\mathrm{la}}_{0, \mathcal{M}_{\infty}} := (\mathcal{O}^{\Psi-\mathrm{la}, (0, 0)_{\tau}}_{\mathcal{M}_{\infty}} \otimes_{\mathcal{O}_{\mathcal{M}_{\infty}}^{\mathrm{sm}}} \Omega_{\mathcal{M}_{\infty}}^{\bullet, \mathrm{sm}}, \nabla^{\Psi-\mathrm{la}}_{\bullet})$ of the usual de Rham complex $GDR_{0, \mathcal{M}_{\infty}}^{\mathrm{sm}} := (\Omega_{\mathcal{M}_{\infty}}^{\bullet, \mathrm{sm}}, \nabla^{\mathrm{sm}}_{\bullet})$ of finite levels by the formula $\nabla^{\Psi-\mathrm{la}}_n : \mathcal{O}^{\Psi-\mathrm{la}, (0, 0)_{\tau}}_{\mathcal{M}_{\infty}} \otimes_{\mathcal{O}_{\mathcal{M}_{\infty}}^{\mathrm{sm}}} \Omega_{\mathcal{M}_{\infty}}^{n, \mathrm{sm}} \rightarrow \mathcal{O}^{\Psi-\mathrm{la}, (0, 0)_{\tau}}_{\mathcal{M}_{\infty}} \otimes_{\mathcal{O}_{\mathcal{M}_{\infty}}^{\mathrm{sm}}} \Omega_{\mathcal{M}_{\infty}}^{n+1, \mathrm{sm}}, \ f \otimes w \mapsto f \otimes d(w) + (-1)^nd^{\Psi-\mathrm{la}}_{\mathcal{M}}(f) \wedge w$. This is actually shown to be a complex by considering a dense subspace as in the proof of Lemma \ref{locally analytic extension of derivation}.

By using the Kodaira-Spencer isomorphism $\Omega_{\mathcal{M}_{\infty}}^{1, \mathrm{sm}} \cong \oplus_{\tau \in \Psi} (\omega_{\tau, \mathcal{M}_{\infty}}^{2, \mathrm{sm}} \otimes \wedge^2D(\mathbb{X}_b)_{\tau})$ and the identification $\omega_{\tau, \mathcal{M}_{\infty}}^{2, \mathrm{sm}} \otimes \wedge^2D(\mathbb{X}_b)_{\tau} \cong \mathcal{O}_{\mathcal{M}_{\infty}}^{\mathrm{sm}} \otimes_{\mathcal{O}_{\Fl_{\mathrm{GM}, \mathcal{M}}}} \omega_{\tau, \Fl_{\mathrm{GM}, \mathcal{M}}}^{-2} \otimes (\wedge^2 D(\mathbb{X}_b)_{\tau})^{-1}$ coming from (\ref{GMfil}), the complex $GDR^{\Psi-\mathrm{la}}_{0, \mathcal{M}_{\infty}}$ can be written as follows.

\begin{align*} GDR^{\Psi-\mathrm{la}}_{0, \mathcal{M}_{\infty}} : \mathcal{O}_{\mathcal{M}_{\infty}}^{\Psi-\mathrm{la}, (0, 0)} \rightarrow \\ 
\oplus_{\tau \in \Psi} \mathcal{O}_{\mathcal{M}_{\infty}}^{\Psi-\mathrm{la}, (0, 0)_{\tau}} \otimes_{\mathcal{O}_{\Fl_{\mathrm{GM}, \mathcal{M}}}} \omega_{\tau, \Fl_{\mathrm{GM}, \mathcal{M}}}^{-2} \otimes (\wedge^2 D(\mathbb{X}_b)_{\tau})^{-1} \rightarrow \\
 \cdots \rightarrow \\ 
 \mathcal{O}_{\mathcal{M}_{\infty}}^{\Psi-\mathrm{la}, (0, 0)_{\tau}} \otimes_{\mathcal{O}_{\Fl_{\mathrm{GM}, \mathcal{M}}}} (\otimes_{\tau \in \Psi} (\omega_{\tau, \Fl_{\mathrm{GM}, \mathcal{M}}}^{-2} \otimes (\wedge^2 D(\mathbb{X}_b)_{\tau})^{-1})) \end{align*}

\begin{prop}\label{local anti}

For any $I \subset \Psi$ and $\sigma \in \Psi \setminus I$, there exists a unique continuous, $G(\mathbb{Q}_p)$-equivariant, $\mathcal{O}_{\mathcal{M}_{\infty}}^{\mathrm{sm}}$-linear map $\overline{d}_{\sigma, I} : \mathcal{O}_{\mathcal{M}_{\infty}}^{\Psi-\mathrm{la}, (0, 0)_{\tau}} \otimes (\otimes_{\tau \in I} (\omega_{\tau, \Fl_{\mathrm{HT}, \mathcal{M}}}^{-2} \otimes (\wedge^2V_{\tau})^{-1})) \rightarrow \mathcal{O}_{\mathcal{M}_{\infty}}^{\Psi-\mathrm{la}, (0, 0)_{\tau}} \otimes (\otimes_{\tau \in I} (\omega_{\tau, \Fl_{\mathrm{HT}, \mathcal{M}}}^{-2} \otimes (\wedge^2V_{\tau})^{-1})) \otimes (\omega_{\sigma, \Fl_{\mathrm{HT}, \mathcal{M}}}^{-2} \otimes (\wedge^2V_{\sigma})^{-1})$ such that for any $U$, $V$ and $x_{\tau}$ as before Proposition \ref{local expansion formula}, we have the following formula : $\overline{d}_{\sigma, I}(\sum_{i \in (\mathbb{Z}_{\ge 0})^{\Psi}} a_{i} \prod_{\tau \in \Psi}(x_{\tau} - x_{\tau, n})^{i_{\tau}} \otimes (\otimes_{\tau \in I} dx_{\tau})) = (\sum_{i \in (\mathbb{Z}_{\ge 0})^{\Psi}} i_{\sigma}a_{i} (x_{\sigma} - x_{\sigma, n})^{i_{\sigma} - 1}\prod_{\tau \neq \sigma}(x_{\tau} - x_{\tau, n})^{i_{\tau}}) \otimes (\otimes_{\tau \in I} dx_{\tau}) \otimes dx_{\sigma}$. (Here, we use the explicit description Proposition \ref{local expansion formula 2}. Note that by the same proof as Proposition \ref{KS}, we have $\Omega_{\sigma, \Fl_{\mathrm{HT}, \mathcal{M}}}^1 \cong \omega_{\sigma, \Fl_{\mathrm{HT}, \mathcal{M}}}^{-2} \otimes (\wedge^2 V_{\sigma})^{-1}$ for the pullback $\Omega_{\sigma, \Fl_{\mathrm{HT}, \mathcal{M}}}^1$ of the sheaf of differential forms on $\mathbb{P}^1_{C}$ via the projection map $\Fl_{\mathrm{HT}, \mathcal{M}} = \prod_{\tau \in \Psi} \mathbb{P}^1_{C} \rightarrow \mathbb{P}^1_{C}$ to the $\sigma$-component.)

\end{prop}

\begin{proof} The uniqueness is clear. For the existence, it suffices to prove that the map defined by the formula in the statement of this lemma is independent of the choice of the basis $g_{1, \tau}$ and $g_{2, \tau}$. This is shown by the same calculation as Lemma \ref{extension anti}. \end{proof}

By using Proposition \ref{local anti}, we obtain the complex $\overline{GDR}^{\Psi-\mathrm{la}}_{0, \mathcal{M}_{\infty}}$ having the following form as after Lemma \ref{extension anti}. (Note that we now consider the trivial weight case.)

\begin{align*}  \overline{GDR}^{\Psi-\mathrm{la}}_{0, \mathcal{M}_{\infty}} : \mathcal{O}_{\mathcal{M}_{\infty}}^{\Psi-\mathrm{la}, (0, 0)} \rightarrow \\ 
\oplus_{\tau \in \Psi} \mathcal{O}_{\mathcal{M}_{\infty}}^{\Psi-\mathrm{la}, (0, 0)_{\tau}} \otimes_{\mathcal{O}_{\Fl_{\mathrm{HT}, \mathcal{M}}}} \omega_{\tau, \Fl_{\mathrm{HT}, \mathcal{M}}}^{-2} \otimes (\wedge^2 V_{\tau})^{-1} \rightarrow \\
 \cdots \rightarrow \\ 
\mathcal{O}_{\mathcal{M}_{\infty}}^{\Psi-\mathrm{la}, (0, 0)_{\tau}} \otimes_{\mathcal{O}_{\Fl_{\mathrm{HT}, \mathcal{M}}}} (\otimes_{\tau \in \Psi} (\omega_{\tau, \Fl_{\mathrm{HT}, \mathcal{M}}}^{-2} \otimes (\wedge^2 V_{\tau})^{-1})) \end{align*}

As in the global case, we have the following.

\begin{prop}\label{anti trivial}

$H^i(\overline{GDR}^{\Psi-\mathrm{la}}_{0, \mathcal{M}_{\infty}}) = 0$ for any $i > 0$ and $H^0(\overline{GDR}^{\Psi-\mathrm{la}}_{0, \mathcal{M}_{\infty}}) = \mathcal{O}_{\mathcal{M}_{\infty}}^{\mathrm{sm}}$.

\end{prop}

\begin{proof} As in the global case, this follows from the explicit formula in Proposition \ref{local anti} and Poincare lemma.  \end{proof}

We give a certain comparison theorem between $\mathcal{M}_{\infty}$ and $\check{\mathcal{M}}_{\infty}$. We recall that we have the identification $\mathcal{M}_{\infty} = \check{\mathcal{M}}_{\infty}$ by Theorem \ref{duality}. In the following, we use the notation $G(\mathbb{Q}_p)-\Psi-\mathrm{la}$ (resp. $J_b(\mathbb{Q}_p)-\Psi-\mathrm{la}$ ) instead of $\Psi-\mathrm{la}$ when we consider the action of $G(\mathbb{Q}_p)$ (resp. $J_b(\mathbb{Q}_p)$).

\begin{prop}\label{locally analytic duality}

1 \ $\mathcal{O}_{\mathcal{M}_{\infty}}^{G(\mathbb{Q}_p)-\Psi-\mathrm{la}} = \mathcal{O}_{\check{\mathcal{M}}_{\infty}}^{J_b(\mathbb{Q}_p)-\Psi-\mathrm{la}}$.

2 \ $\theta_{\mathfrak{h}, \mathcal{M}} = \theta_{\mathfrak{h}, \check{\mathcal{M}}}$.

\end{prop}

\begin{proof} 1 \ By Proposition \ref{local expansion formula}, we have $\mathcal{O}_{\mathcal{M}_{\infty}}^{G(\mathbb{Q}_p)-\Psi-\mathrm{la}} \subset \mathcal{O}_{\check{\mathcal{M}}_{\infty}}^{J_b(\mathbb{Q}_p)-\Psi-\mathrm{la}}$ as in \cite[Corollary 5.2.10 and 5.3.9]{PanII}. The converse is shown by the same way.
    
2 Note that $\theta_{\mathfrak{h}, \mathcal{M}}$ and $\theta_{\mathfrak{h}, \check{\mathcal{M}}}$ are derivations. Thus it suffices to check the equality for all elements in $\mathcal{O}_{\mathcal{M}_{G_{r(n)}}}(V_{G_{r(n)}})$, $x_{\tau}$, $g_{1, \tau}$ and $f_{\tau}$ appearing in Proposition \ref{local expansion formula}. Note that since $\pi_{\mathrm{GM}, \mathcal{M}}$ is $\et$ale, it suffices to consider all elements in $\mathcal{O}_{\Fl_{\mathrm{GM}, \mathcal{M}}}$ instead of $\mathcal{O}_{\mathcal{M}_{G_{r(n)}}}(V_{G_{r(n)}})$. Thus it suffices to check the equality for $f_{\tau}$'s and all elements in $\mathcal{O}_{\Fl_{\mathrm{GM}, \mathcal{M}}}$, $\omega_{\tau, \Fl_{\mathrm{GM}, \mathcal{M}}}$, $\mathcal{O}_{\Fl_{\mathrm{HT}, \mathcal{M}}}$ and $\omega_{\tau,\Fl_{\mathrm{HT}, \mathcal{M}}}$. The calculations are the same as \cite[Corollary 5.3.13]{PanII}. \end{proof}

\begin{prop}\label{locally analytic de Rham duality}

We have a canonical identification $GDR^{G(\mathbb{Q}_p)-\Psi-\mathrm{la}}_{0, \mathcal{M}_{\infty}} = \overline{GDR}^{J_b(\mathbb{Q}_p)-\Psi-\mathrm{la}}_{0, \check{\mathcal{M}}_{\infty}}$.

\end{prop}

\begin{proof}

Note that we have a canonical identification of the components of $GDR^{G(\mathbb{Q}_p)-\Psi-\mathrm{la}}_{0, \mathcal{M}_{\infty}}$ and $\overline{GDR}^{J_b(\mathbb{Q}_p)-\Psi-\mathrm{la}}_{0, \check{\mathcal{M}}_{\infty}}$ by Proposition \ref{locally analytic duality} and the identification $\Fl_{\mathrm{GM}, \mathcal{M}} = \Fl_{\mathrm{HT}, \check{\mathcal{M}}}$. Thus it suffices to prove that all maps are equal. Since $\pi_{\mathrm{GM}, \mathcal{M}}$ and $\pi_{\mathrm{GM}, \check{\mathcal{M}}}$ are $\et$ale, as we used in the proof of Proposition \ref{locally analytic duality}, the result follows from the fact that all maps are $\mathcal{O}_{\Fl_{\mathrm{HT}, \mathcal{M}}} = \mathcal{O}_{\Fl_{\mathrm{GM}, \check{\mathcal{M}}}}$-linear derivations which are equal to usual derivation on $\Fl_{\mathrm{GM}, \mathcal{M}} = \Fl_{\mathrm{HT}, \check{\mathcal{M}}}$.  \end{proof}

Now, we come back to our global situation. Note that the underlying topological space of $[\varprojlim_{K^p} \mathcal{M}_{\infty} \times [GI(\mathbb{Q}) \setminus GI(\mathbb{A}_{\mathbb{Q}}^{\infty})/ K^p]]/J_b(\mathbb{Q}_p)$ is not changed if we consider the profinite topology on $[GI(\mathbb{Q}) \setminus GI(\mathbb{A}_{\mathbb{Q}}^{\infty})/ K^p]$ because we divide the space by the action of $J_b(\mathbb{Q}_p)$. In the rest of this subsection, we always consider the profinite topology on $[GI(\mathbb{Q}) \setminus GI(\mathbb{A}_{\mathbb{Q}}^{\infty})/ K^p]$ and we only consider the underlying topological space. 

For a topological space $Y$ on which the locally profinite group $J_b(\mathbb{Q}_p)$ acts continuously, let $\mathrm{Shv}_{J_b(\mathbb{Q}_p)}(Y)$ denote the category of $J_b(\mathbb{Q}_p)$-equivariant sheaves on $Y$. (See \cite[{\S} 1]{Schneider} for details). If $J_b(\mathbb{Q}_p)$ acts trivially on $Y$, then let $R\Gamma^{\mathrm{sm}}(J_{b}(\mathbb{Q}_p), \ )$ denote the right derived functor of the functor $\mathrm{Shv}_{J_b(\mathbb{Q}_p)}(Y) \rightarrow \mathrm{Shv}(Y), \ \mathcal{F} \mapsto \mathcal{F}^{J_b(\mathbb{Q}_p)}$. Let $X_{K} := GI(\mathbb{Q}) \setminus GI(\mathbb{A}_{\mathbb{Q}}^{\infty})/ K$, $X := \varprojlim_{K} X_K$ and $s : \mathcal{M}_{\infty} \times X \rightarrow [\mathcal{M}_{\infty} \times X]/J_b(\mathbb{Q}_p) \cong \pi_{\mathrm{HT}}^{-1}(\Fl^{b})$ and $\mathrm{pr}_1 : \mathcal{M}_{\infty} \times X \rightarrow \mathcal{M}_{\infty}$ be the natural projections.

\begin{lem}\label{equivariant sheaves}
    
The category $\mathrm{Shv}_{J_b(\mathbb{Q}_p)}(\mathcal{M}_{\infty} \times X)$ is equivalent to $\mathrm{Shv}([\mathcal{M}_{\infty} \times X]/J_b(\mathbb{Q}_p))$ by $\mathcal{F} \mapsto (s_*\mathcal{F})^{J_b(\mathbb{Q}_p)}$ and $\mathcal{G} \mapsto s^{-1}\mathcal{G}$.

\end{lem}

\begin{proof}

We will prove that the natural maps $s^*(s_*\mathcal{F})^{J_b(\mathbb{Q}_p)} \rightarrow \mathcal{F}$ and $\mathcal{G} \rightarrow (s_*(s^* \mathcal{G}))^{J_b(\mathbb{Q}_p)}$ are isomorphisms. Thus it suffices to check these locally on $[\mathcal{M}_{\infty} \times X]/J_b(\mathbb{Q}_p)$. By Theorem \ref{uniformization}, any point of $[\mathcal{M}_{\infty} \times X]/J_b(\mathbb{Q}_p)$ has a neighborhood having a form like $(\sqcup_{\gamma \in K_p \setminus J_b(\mathbb{Q}_p)} [U \times gK]\gamma)/J_b(\mathbb{Q}_p)$ for some $K$ and a $K_p$-stable quasicompact open $U$ of $\mathcal{M}_{\infty}$. We have natural identifications $(\sqcup_{\gamma \in K_p \setminus J_b(\mathbb{Q}_p)} [U \times gK]\gamma)/J_b(\mathbb{Q}_p) \cong [U \times gK]/K_p$ and $\mathrm{Shv}_{J_b(\mathbb{Q}_p)}((\sqcup_{\gamma \in K_p \setminus J_b(\mathbb{Q}_p)} [U \times gK]\gamma)) \cong \mathrm{Shv}_{K_p}([U \times gK])$. Thus it suffices to prove $\mathrm{Shv}_{K_p}([U \times gK]) \cong \mathrm{Shv}([U \times gK]/K_p)$. This follows from the definition of $K_p$-equivariant sheaves and the identification $[U \times gK] = \varprojlim_{K_p'}[U \times gK]/K_p'$. \end{proof}

Let $GDR^{\Psi-\mathrm{la}}_{\lambda, \mathcal{S}} := \varinjlim_{K^p_1} \pi_{K^p_1}^{-1}GDR^{\Psi-\mathrm{la}}_{\lambda, \mathcal{S}_{K^p_1}}$, where $\pi_{K^p_1} : \mathcal{S} \rightarrow \mathcal{S}_{K^p}$ denotes the canonical projection. (See Remark \ref{pullback de Rham} for the definition of $GDR^{\Psi-\mathrm{la}}_{\lambda, \mathcal{S}_{K^p}}$.) Note that we have $GDR^{\Psi-\mathrm{la}}_{\lambda, \mathcal{S}_{K^p}} = ({\pi_{K^p}}_{*}GDR^{\Psi-\mathrm{la}}_{\lambda, \mathcal{S}})^{K^p}$.  

Let $D_{\lambda}(\mathbb{X}_{b}) := D_{0}(\mathbb{X}_b)^{\lambda_0} \otimes (\otimes_{\tau \in \Phi} ((\wedge^2D(\mathbb{X}_b)_{\tau})^{\lambda_{\tau, 2}} \otimes \mathrm{Sym}^{\lambda_{\tau, 1} - \lambda_{\tau, 2}}D(\mathbb{X}_b)_{\tau}))$, where $D_{0}(\mathbb{X}_b)$ denotes the similitude representation of $J_b(\mathbb{Q}_p)$ with the trivial action of $G(\mathbb{Q}_p)$.

Let $\mathcal{F}$ be a $J_b(\mathbb{Q}_p)$-equivariant sheaf on $\mathcal{M}_{\infty} \times X$ defined by $U \times V \mapsto \varinjlim_{K' \subset K}\mathrm{Map}_{K'}(V, D_{\lambda}(\mathbb{X}_{b}) \otimes_{\breve{L}} (\otimes_{\tau \in \Psi} (\omega_{\tau, \check{\mathcal{M}}_{\infty}}^{-\lambda_{\tau}, J_b(\mathbb{Q}_p)-\mathrm{sm}} \otimes (\wedge^2V_{\tau})^{-\lambda_{\tau}}))) = \varinjlim_{K' \subset K}\mathrm{Map}_{K'}(V, D_{\lambda}(\mathbb{X}_{b})) \otimes_{\breve{L}} (\otimes_{\tau \in \Psi} (\omega_{\tau, \check{\mathcal{M}}_{\infty}}^{-\lambda_{\tau}, J_b(\mathbb{Q}_p)-\mathrm{sm}} \otimes (\wedge^2V_{\tau})^{-\lambda_{\tau}}))$ for any $K$, any $K_p$-stable open subset $U \subset \mathcal{M}_{\infty}$ and any $K$-stable subset $V \subset X$. Here $\mathrm{Map}_{K'}$ denotes the set of $K'$-equivariant maps.

% and we consider the trivial action of $K'^{p}$ on $GDR^{\Psi-\mathrm{la}}_{\lambda, \mathcal{M}_{\infty}}(U)$

\begin{prop}\label{de Rham uniformization}

$GDR^{\Psi-\mathrm{la}}_{\lambda, \mathcal{S}}|_{\pi_{\mathrm{HT}, \mathcal{S}}^{-1}(\Fl^{b})}$ is quasi-isomorphic to $s_*(\mathcal{F})^{J_b(\mathbb{Q}_p)}$.

\end{prop}

\begin{proof} By Lemma \ref{equivariant sheaves}, it suffices to prove that $s^{-1}GDR^{\Psi-\mathrm{la}}_{\lambda, \mathcal{S}}|_{\pi_{\mathrm{HT}, \mathcal{S}}^{-1}(\Fl^{b})}$ is quasi-isomorphic to $\mathcal{F}$ as $J_b(\mathbb{Q}_p)$-equivariant sheaves. By the identification of de Rham cohomologies of abelian schemes and Dieudonne modules of $p$-divisible group obtained by abelian schemes (see \cite[(2.5.6) and (3.3.7)]{BBM}), the pullback to $\mathcal{M}_{\infty} \times X$ of the vector bundle with the integrable connection $\varinjlim_{K^p_1}(D_{\lambda, \aS_{K^p_1}}^{\mathrm{sm}}, \nabla)$ is equal to the pullback of $(D_{\lambda}(\mathbb{X}_b) \otimes \mathcal{O}_{\mathcal{M}_{\infty}}^{\mathrm{sm}}, \mathrm{id} \otimes d_{\mathcal{O}_{\mathcal{M}_{\infty}}^{\mathrm{sm}}})$ via $\mathrm{pr}_1 : \mathcal{M}_{\infty} \times X \rightarrow \mathcal{M}_{\infty}$. Thus the pullback of $\varinjlim_{K^p_1}\pi_{K_1^p}^{-1}(D_{\lambda, \aS_{K^p_1}}^{\Psi-\mathrm{la}, (0, 0)} \otimes \Omega_{\aS_{K_1^p}}^{\bullet, \mathrm{sm}}, \nabla^{\Psi-\mathrm{la}})$ is identified with the pullback of $D_{\lambda}(\mathbb{X}_b) \otimes GDR^{\Psi-\mathrm{la}}_{0, \mathcal{M}_{\infty}}$. (See {\S} 4.1 for the definition of $(D_{\lambda, \aS_{K^p_1}}^{\Psi-\mathrm{la}, (0, 0)} \otimes \Omega_{\aS_{K_1^p}}^{\bullet}, \nabla^{\Psi-\mathrm{la}})$.) In {\S} 4.1, we took certain sections to $(D_{\lambda, \aS_{K^p_1}}^{\Psi-\mathrm{la}, (0, 0)} \otimes \Omega_{\aS_{K_1^p}}^{\bullet, \mathrm{sm}}, \nabla^{\Psi-\mathrm{la}})$ from certain subquotients and took $\otimes_{\mathcal{O}_{\Fl}} (\otimes_{\tau} (\omega_{\tau, \Fl}^{\lambda_{\tau}} \otimes (\wedge^2 V_{\tau})^{\lambda_{\tau}}))$ to get $GDR^{\Psi-\mathrm{la}}_{\lambda, \aS_{K^p_1}}$. However, on $\mathcal{M}_{\infty}$, the corresponding sections of the components of $D_{\lambda}(\mathbb{X}_b) \otimes GDR^{\Psi-\mathrm{la}}_{0, \mathcal{M}_{\infty}}$ are given by taking the $\chi_{\lambda_{\Psi}}$-isotypic part of the components with respect to the action of $J_b(\mathbb{Q}_p)$ by the same calculation as in the proof of Lemma \ref{BGG flag}. \footnote{Roughly speaking, the constructions of {\S} 4.1 and {\S} 4.2 are equal under the duality isomorphism.} Here, we identify the Lie algebra of $G$ over $L$ and the Lie algebra of the inner form $J_b$ of $G$ over $L$. Thus the result follows from the quasi-isomorphism $GDR^{\Psi-\mathrm{la}}_{0, \mathcal{M}_{\infty}} \cong \overline{GDR}^{\Psi-\mathrm{la}}_{0, \check{\mathcal{M}}_{\infty}} \cong \mathcal{O}_{\mathcal{M}_{\infty}}^{\mathrm{sm}}$, which is deduced from Proposition \ref{locally analytic de Rham duality} and Proposition \ref{anti trivial}. \end{proof}

Note that the space $\mathcal{A}^{\mathrm{sm}}_{GI}(V_{\lambda})_{C} := \varinjlim_{K}\mathcal{A}_{GI}(K, V_{\lambda})_{C}$ (resp. $\mathcal{A}^{\mathrm{sm}}_{GI}(K^p, V_{\lambda})_{C} := \varinjlim_{K_p}\mathcal{A}_{GI}(K^pK_p, V_{\lambda})_{C}$) of algebraic automorphic forms of weight $\lambda$ on $GI(\mathbb{A}_{\mathbb{Q}})$ can be regarded as $\varinjlim_{K} \mathrm{Map}_{K}(X, D_{\lambda}(\mathbb{X}_{b})) \otimes_{\breve{L}} C$ (resp. $\varinjlim_{K_p} \mathrm{Map}_{K_p}(X_{K^p}, D_{\lambda}(\mathbb{X}_{b})) \otimes_{\breve{L}} C$).

\begin{thm} \label{basic cohomology}

We have the following natural quasi-isomorphism.

$GDR^{\Psi-\mathrm{la}}_{\lambda}|_{\Fl^{b}} \cong R\Gamma^{\mathrm{sm}}(J_{b}(\mathbb{Q}_p), R\pi_{\mathrm{HT}, \mathcal{M} *}(\mathcal{A}_{GI}^{\mathrm{sm}}(K^p, V_{\lambda})_C \otimes_{C} (\otimes_{\tau \in \Psi}( \omega_{\tau, \check{\mathcal{M}}}^{-\lambda_{\tau}, J_b(\mathbb{Q}_p)-\mathrm{sm}} \otimes (\wedge^2 V_{\tau})^{-\lambda_{\tau}})))).$
        
    \end{thm} 

    \begin{proof} By Proposition \ref{derived calculation II}, we have $GDR^{\Psi-\mathrm{la}}_{\lambda} = (R\pi_{\mathrm{HT}, \aS *}GDR^{\Psi-\mathrm{la}}_{\lambda, \aS})^{K^p}$. By Proposition \ref{de Rham uniformization} and Lemma \ref{equivariant sheaves}, we have $GDR^{\Psi-\mathrm{la}}_{\lambda, \aS}|_{\pi_{\mathrm{HT}}^{-1}(\Fl^b)} = (s_* \mathcal{F})^{J_b(\mathbb{Q}_p)} = R\Gamma^{\mathrm{sm}}(J_b(\mathbb{Q}_p), Rs_* \mathcal{F})$. 

Thus we obtain \begin{align*} & \ \ \ \ \ GDR^{\Psi-\mathrm{la}}_{\lambda}|_{\Fl^b} \\
    &= (R(\pi_{\mathrm{HT}, \aS}|_{\pi_{\mathrm{HT}}^{-1}(\Fl^b)})_*(GDR^{\Psi-\mathrm{la}}_{\lambda, \aS_{K^p}}|_{\pi_{\mathrm{HT}}^{-1}(\Fl^b)}))^{K^p} \\ 
    &= (R(\pi_{\mathrm{HT}, \aS}|_{\pi_{\mathrm{HT}}^{-1}(\Fl^b)})_*(R\Gamma^{\mathrm{sm}}(J_b(\mathbb{Q}_p), Rs_* \mathcal{F})))^{K^p} \\ 
    &= R\Gamma^{\mathrm{sm}}(J_b(\mathbb{Q}_p), R(\pi_{\mathrm{HT}, \aS}|_{\pi_{\mathrm{HT}}^{-1}(\Fl^b)})_* Rs_* \mathcal{F})^{K^p} \\ 
    &= R\Gamma^{\mathrm{sm}}(J_b(\mathbb{Q}_p), R\pi_{\mathrm{HT}, \mathcal{M} *} R\mathrm{pr}_{1*}(\mathcal{F}))^{K^p}. \end{align*} 

Note that $R{\mathrm{pr}_1}_{*} = {\mathrm{pr}_{1}}_{*}$. In fact, by using the fact that $X$ is a profinite space, we see that every surjection $\mathcal{G} \twoheadrightarrow \mathcal{G}'$ of sheaves on $\mathcal{M}_{\infty} \times X$ gives a surjection ${\mathrm{pr}_{1}}_{*}\mathcal{G} \twoheadrightarrow {\mathrm{pr}_{1}}_{*}\mathcal{G}'$.

Thus for any $K$ and any $K_p$-stable quasicompact open subset $U$ of $\mathcal{M}_{\infty}$, we have \begin{align*}R\mathrm{pr}_{1 *}(\mathcal{F})(U)  & = \mathrm{pr}_{1 *}(\mathcal{F})(U) \\ 
    &= \varinjlim_{K_1 \subset K}\mathrm{Map}_{K_{1}}(X, D_{\lambda}(\mathbb{X}_{b}) \otimes_{\breve{L}} (\otimes_{\tau \in \Psi} (\omega_{\tau, \check{\mathcal{M}}_{\infty}}^{-\lambda_{\tau}, J_b(\mathbb{Q}_p)-\mathrm{sm}} \otimes (\wedge^2V_{\tau})^{-\lambda_{\tau}}))(U)) \\ 
    &= \mathcal{A}_{GI}^{\mathrm{sm}}(V_{\lambda})_C \otimes_{C} (\otimes_{\tau \in \Psi} (\omega_{\tau, \check{\mathcal{M}}_{\infty}}^{-\lambda_{\tau}, J_b(\mathbb{Q}_p)-\mathrm{sm}} \otimes (\wedge^2 V_{\tau})^{-\lambda_{\tau}})(U)). \end{align*} Thus we obtain the result because taking $K^p$-invariant part is exact for any smooth representations of $K^p$. \end{proof}

The following result is useful later. Let $GDR^{\mathrm{sm}}_{0, \mathcal{M}_{\infty}}$ (resp. $GDR^{\mathrm{sm}}_{0, \check{\mathcal{M}}_{\infty}}$) be the colimit of the pullbacks to $\mathcal{M}_{\infty}$ (resp. $\check{\mathcal{M}}_{\infty}$) of the de Rham complexes of finite levels.

\begin{thm} (\cite[Theorem 5.2.2.]{DC}) \label{de Rham duality}

We have a canonical $G(\mathbb{Q}_p) \times J_b(\mathbb{Q}_p)$-equivariant isomorphism $GDR^{\mathrm{sm}}_{0, \mathcal{M}_{\infty}} \cong GDR^{\mathrm{sm}}_{0, \check{\mathcal{M}}_{\infty}}$. In particular, the action of $J_{b}(\mathbb{Q}_p)$ on $H^i(GDR^{\mathrm{sm}}_{0, \mathcal{M}_{\infty}})$ (resp. $G(\mathbb{Q}_p)$ on $H^i(GDR^{\mathrm{sm}}_{0, \check{\mathcal{M}}_{\infty}})$) is smooth.

\end{thm}

\begin{rem}

    \cite[Theorem 0.6]{CDN} first proved a similar result in the Lubin-Tate curve and Drinfeld curve case and recently, an elementary proof of a similar result was given by \cite[Theorem 0.4.11]{LD} in the Lubin-Tate space and Drinfeld space case. The author think that the same argument also works in our situation.

\end{rem}

\subsection{Classicality of geometric locally analytic de Rham complexes}

In the following, we use various unitary groups. In order to give identifications of local components of different unitary groups, we use the following convention.

\vspace{0.5 \baselineskip}

\textbf{Convention} \label{convention}

\vspace{0.5 \baselineskip}

We consider $(S(B_1), \Psi_1)$ and $(S(B_2), \Psi_2)$ as $(S(B), \Psi)$ in {\S} 3.1. These data induce central division algebras $B_1$ and $B_2$ over $F$ of dimension $4$ and unitary similitude groups $GU_1/\mathbb{Q}$ and $GU_2/\mathbb{Q}$ via Proposition \ref{unitary groups}. For any prime $r$ splitting $uu^c$ in $F_0$ not lying below a place in $S(B_1) \cup S(B_2)$, we fix an identification between $GU(\mathbb{Q}_r) = \mathbb{Q}_r^{\times} \times \prod_{w \mid u} B_{1, w}^{\mathrm{op}, \times}$ and $GU(\mathbb{Q}_r) = \mathbb{Q}_r^{\times} \times \prod_{w \mid u} B_{2, w}^{\mathrm{op}, \times}$ by $B_{1, w} \cong M_2(F_w) \cong B_{2, w}$. This identification is unique up to inner automorphisms and thus isomorphism classes of representations are independent of the choice of this identification.

\begin{dfn}

Let $\varphi : \mathbb{T}^S \rightarrow C$ be an $\mathcal{O}$-morphism.

We say that $\varphi$ is a classical eigensystem of weight $\lambda$ if after replacing $S$ if necessary, $\varphi$ is an eigensystem associated with a cohomological automorphic representation $\chi \boxtimes \pi$ of $\mathbb{A}_{F_0}^{\times} \times \mathrm{GL}_2(\mathbb{A}_{F})$ which is a base change of a cohomological cuspidal automorphic representation of weight $\lambda$ of a unitary similitude group $GU'/\mathbb{Q}$ defined in {\S} 3.1 for some $S'(B)$ and $\Psi'$. Note that we fix data $F^+$, $F_0$ and $\iota$.

\end{dfn}

Inspired by \cite[{\S} 5]{PanII}, we propose the following conjecture.

\begin{conj}\label{classicality conjecture}
For any $b \in B(G, \mu^{-1})$, $H^d(\Fl, j_{b, !}j_{b}^{*}GDR^{\Psi-\mathrm{la}}_{\lambda})$ has a generalized Hecke eigenspace decomposition by classical eigensystems of weight $\lambda$, where $j_b : \Fl^b \hookrightarrow \Fl$ denotes the natural immersion.

Consequently, $H^d(\Fl, GDR^{\Psi-\mathrm{la}}_{\lambda})$ has a generalized Hecke eigenspace decomposition by classical eigensystems of weight $\lambda$. \end{conj}

The author roughly expects that $H^d(\Fl, j_{b, !}j_{b}^{*}GDR^{\Psi-\mathrm{la}}_{\lambda})$ can be written by the $J_b(\mathbb{Q}_p)$-invariant part of (the rigid cohomology of the Igusa variety) $\otimes$ (geometric locally analytic de Rham complex of the infinite level of the Rapoport-Zink space), which is an analogue of the Mantovan formula for the $l$-adic $\etale$ cohomology of Shimura varieties (see \cite[Theorem 22]{Mant}) and the classicality follows from the classicality of the rigid cohomology of the Igusa variety, which is still open in general (see \cite[proof of Theorem 5.5.7]{PanII} in the modular curve case).

\vspace{0.5 \baselineskip}

We only prove the above conjecture on the basic locus in general and on the $\mu$-ordinary locus in some cases.

\begin{prop}\label{algebraic automorphic forms}

$\mathcal{A}^{\mathrm{sm}}_{GI}(K^p, V_{\lambda})_C$ has a Hecke eigenspace decomposition by classical eigensystems of weight $\lambda$.

\end{prop}

\begin{proof} This follows from Theorem \ref{Zuker}. \end{proof}

\begin{cor}\label{basic classicality}

For the basic element $b \in B(G, \mu^{-1})$, $H^d(\Fl, j_{b, !}j_{b}^{*}GDR^{\Psi-\mathrm{la}}_{\lambda})$ has a Hecke eigenspace decomposition by classical eigensystems of weight $\lambda$.

\end{cor}

\begin{proof}

This follows from Theorem \ref{basic cohomology} and Proposition \ref{algebraic automorphic forms}. \end{proof}

\begin{prop}\label{classicality of Igusa} Let $w \mid v$ be a place of $F$ and assume $\Psi \subset \mathrm{Hom}_{\mathbb{Q}_p}(F_w, \overline{\mathbb{Q}}_p)$ and $|\Psi| \le 2$. Then for any $i$, the de Rham cohomology $H^i_{\mathrm{dR}}(\mathcal{S}_{K^p}^{\mu-\mathrm{can}, \dagger}, D_{\lambda}^{\dagger})$ defined after Lemma \ref{dagger variety} has a generalized Hecke eigenspace decomposition by classical eigensystems of weight $\lambda$. 

\end{prop}

\begin{proof}

Note that under our assumption, we have $\Fl = \Fl^{b_0} \sqcup \Fl^{b_1}$ by Corollary \ref{GL iso}, where $b_0$ (resp. $b_1$) denotes the $\mu$-ordinary element (resp. the basic element) of $B(G, \mu^{-1})$. Let $i : \Fl^{b_0} \hookrightarrow \Fl$ (resp. $j : \Fl^{b_1} \hookrightarrow \Fl$) be the natural immersion.

Thus we have a distinguished triangle $j_{!}j^*GDR_{\lambda}^{\mathrm{sm}} \rightarrow GDR_{\lambda}^{\mathrm{sm}} \rightarrow i_*i^*GDR_{\lambda}^{\mathrm{sm}} \rightarrow$ and an exact sequence \begin{align}\label{excision}\cdots \rightarrow H^i(\Fl, j_{!}j^*GDR_{\lambda}^{\mathrm{sm}}) \rightarrow H^i(\Fl, GDR_{\lambda}^{\mathrm{sm}}) \rightarrow H^i(\Fl, i_{*}i^*GDR_{\lambda}^{\mathrm{sm}}) \rightarrow \cdots.\end{align} In the following, we calculate these cohomology groups.

First, $H^i(\Fl, GDR_{\lambda}^{\mathrm{sm}}) = \varinjlim_{K_p} H^i(S_{K^pK_p, C}, GDR_{\lambda, S_{K^pK_p, C}})$ and by using the usual de Rham comparison over $\mathbb{C}$, Theorem \ref{Zuker} and Proposition \ref{BGG}, this has a Hecke eigenspace decomposition by classical eigensystems of weight $\lambda$. 

By using Theorem \ref{uniformization}, we obtain $j^*GDR_{\lambda}^{\mathrm{sm}} = (R\Gamma^{\mathrm{sm}}(J_{b}(\mathbb{Q}_p), R\pi_{\mathrm{HT}, \mathcal{M} *}(\mathcal{G})))^{K^p}$ as in the proof of Theorem \ref{basic cohomology}, where $\mathcal{G}$ denotes the complex of sheaves defined by $$U \mapsto \varinjlim_{K' \subset K} \mathrm{Map}_{K'}(X_{{K'}^p}, D_{\lambda}(\mathbb{X}_b) \otimes GDR^{\mathrm{sm}}_{0, \mathcal{M}_{\infty}}(U))$$ for any $K$ and any $K_p$-stable quasicompact open $U$ of $\mathcal{M}_{\infty}$. Theorem \ref{de Rham duality} implies the following quasi-isomorphisms.

\begin{align*} \varinjlim_{K' \subset K} \mathrm{Map}_{K_p'}(X_{{K'}^p}, D_{\lambda}(\mathbb{X}_b) \otimes GDR^{\mathrm{sm}}_{0, \mathcal{M}_{\infty}}) &= \varinjlim_{K' \subset K} \mathrm{Map}_{K_p'}(X_{{K'}^p}, D_{\lambda}(\mathbb{X}_b) \otimes GDR^{\mathrm{sm}}_{0, \check{\mathcal{M}}_{\infty}}) \\ 
    &= \varinjlim_{K' \subset K} \mathrm{Map}_{K_p'}(X_{{K'}^p}, D_{\lambda}(\mathbb{X}_b)) \otimes GDR^{\mathrm{sm}}_{0, \check{\mathcal{M}}_{\infty}} \\
    & = \mathcal{A}_{GI}^{\mathrm{sm}}(V_{\lambda})_C \otimes_C GDR^{\mathrm{sm}}_{0, \check{\mathcal{M}}_{\infty}}. \end{align*} 

This also has a Hecke eigenspace decomposition by classical eigensystems of weight $\lambda$ by Proposition \ref{algebraic automorphic forms}. Thus $H^i(\Fl, j_{!}j^*GDR_{\lambda}^{\mathrm{sm}})$ also has a Hecke eigenspace decomposition by classical eigensystems of weight $\lambda$.

By 2 of Lemma \ref{induction2}, we obtain $$H^i(\Fl, i_{*}i^*GDR_{\lambda}^{\mathrm{sm}}) = \mathrm{Ind}_{P_{S_0}}^{G(\mathbb{Q}_p)}(\varinjlim_{n}H^i_{\mathrm{dR}}(\aS_{K^p\Gamma(p^n)}^{\mu-\mathrm{can}, \dagger}, D^{\dagger}_{\lambda}))^{\mathrm{sm}}.$$ Thus we obtain the result by the above exact sequence (\ref{excision}). \end{proof}

\begin{cor}\label{classicality of ordinary}

Under the same assumption as Proposition \ref{classicality of Igusa}, for the $\mu$-ordinary element $b_0 \in B(G, \mu^{-1})$, $H^i(\Fl, j_{b_0, !}j_{b_0}^*GDR_{\lambda}^{\Psi-\mathrm{la}})$ has a generalized Hecke eigenspace decomposition by classical eigensystems of weight $\lambda$.

\end{cor}

\begin{proof} This follows from Proposition \ref{classicality of Igusa} and Theorem \ref{ordinary cohomology}. \end{proof}

\begin{cor}\label{classicality of geometric}

Under the same assumption as Proposition \ref{classicality of Igusa}, for any $i$, the cohomology group $H^i(\Fl, GDR^{\Psi-\mathrm{la}}_{\lambda})$ has a generalized Hecke eigenspace decomposition by classical eigensystems of weight $\lambda$.

\end{cor}

\begin{proof} This follows from Corollaries \ref{classicality of ordinary} and \ref{basic classicality} because under our assumption, we only have the basic locus and the $\mu$-ordinary locus on $\Fl$ by Corollary \ref{GL iso}. \end{proof}

\section{Arithmetic locally analytic de Rham complexes}

In this section, we construct a complex $ADR^{\Psi-\mathrm{la}}_{\lambda}$, which is closely related to $GDR^{\Psi-\mathrm{la}}_{\lambda}$ and is called the arithmetic locally analytic de Rham complex of weight $\lambda$ in {\S} 5.1. Moreover, in some cases, we prove that all de Rham representations of $G_F$ of $p$-adic Hodge type $\lambda$ appearing in the $\Psi$-locally analytic completed cohomology contribute to the cohomology of $ADR^{\Psi-\mathrm{la}}_{\lambda}$ in {\S} 5.2.

\subsection{Arithmetic and geometric locally analytic de Rham complexes}

For $k \in \mathbb{Z}_{\ge 0}$, let $D_k := (GDR^{\Psi-\mathrm{la}}_{\lambda})_k = \oplus_{I \subset \Psi, |I|=k} D_{\lambda^{\Psi}, K^p}^{\Psi-\mathrm{la}, (0, \lambda_{\tau})_{\tau}} \otimes (\otimes_{\tau \in I} (\mathcal{\omega}_{\tau, K^p}^{2\lambda_{\tau} + 2, \mathrm{sm}} \otimes (\wedge^2 D_{\tau, K^p}^{\mathrm{sm}})^{\lambda_{\tau}+1}))$ and $C_k := \oplus_{I \subset \Psi, |I|=k} D_{K^p}^{\Psi-\mathrm{la}, (1 + \lambda_{\tau}, -1)_{\tau \in I}, (0, \lambda_{\tau})_{\tau \notin I}}((\sum_{\tau \in I}(\lambda_{\tau} + 1)))$.

Note that for any $\tau \in \Psi$, we have the following $\mathcal{O}_{K^p}^{\mathrm{sm}}$-linear $G(\mathbb{Q}_p)$-equivariant surjection. (The first map is $(\overline{GDR}^{\Psi-\mathrm{la}}_{\lambda})_0 \rightarrow (\overline{GDR}^{\Psi-\mathrm{la}}_{\lambda})_1$ and the second map is the projection to the $\tau$-component. The surjectivity follows from the explicit description of the first map in Lemma \ref{anti-holomorphic}.)

$\overline{d}_{\tau, 1}^{\lambda} : D_{\lambda^{\Psi}, K^p}^{\Psi-\mathrm{la}, (0, \lambda_{\tau})_{\tau}} \rightarrow \oplus_{\sigma \in \Psi} D_{\lambda^{\Psi}, K^p}^{\Psi-\mathrm{la}, (1 + \lambda_{\sigma}, -1), (0, \lambda_{\gamma})_{\gamma \neq \sigma}} \otimes \omega_{\sigma, K^p}^{-2\lambda_{\sigma}-2, \mathrm{sm}}(\lambda_{\sigma} + 1) \otimes (\wedge^2 D_{\sigma, K^p}^{\mathrm{sm}})^{-\lambda_{\sigma}-1} \twoheadrightarrow D_{\lambda^{\Psi}, K^p}^{\Psi-\mathrm{la}, (1 + \lambda_{\tau}, -1), (0, \lambda_{\gamma})_{\gamma \neq \tau}} \otimes \omega_{\tau, K^p}^{-2\lambda_{\tau}-2, \mathrm{sm}}(\lambda_{\tau} + 1) \otimes (\wedge^2 D_{\tau, K^p}^{\mathrm{sm}})^{-\lambda_{\tau}-1}.$ 

By using these maps, for any $I = \{ \tau_{k_1}, \cdots, \tau_{k_s} \} \subset \Psi$ with $k_1 < \cdots < k_s$, we obtain a surjection $(\overline{d}_{\tau_{k_1}, 1}^{\lambda} \otimes \mathrm{id}) \circ \cdots \circ \overline{d}_{\tau_{k_s}, 1}^{\lambda} : D_{\lambda^{\Psi}, K^p}^{\Psi-\mathrm{la}, (0, \lambda_{\tau})_{\tau}} \twoheadrightarrow D_{\lambda^{\Psi}, K^p}^{\Psi-\mathrm{la}, (1 + \lambda_{\tau}, -1)_{\tau \in I}, (0, \lambda_{\gamma})_{\gamma \notin I}} \otimes (\otimes_{\tau \in I} (\omega_{\tau, K^p}^{-2\lambda_{\tau}-2, \mathrm{sm}} \otimes (\wedge^2 D_{\tau, K^p}^{\mathrm{sm}})^{-\lambda_{\tau}-1}))(\sum_{\tau \in I} (\lambda_{\tau} + 1))$. Thus by taking $ \otimes (\otimes_{\tau \in I} (\mathcal{\omega}_{\tau, K^p}^{2\lambda_{\tau} + 2, \mathrm{sm}} \otimes (\wedge^2 D_{\tau, K^p}^{\mathrm{sm}})^{\lambda_{\tau}+1}))$, we obtain \begin{equation}\label{antiderivation}\overline{d}_{I}^{\lambda} : D_{\lambda^{\Psi}, K^p}^{\Psi-\mathrm{la}, (0, \lambda_{\tau})_{\tau}} \otimes (\otimes_{\tau \in I} (\mathcal{\omega}_{\tau, K^p}^{2\lambda_{\tau} + 2, \mathrm{sm}} \otimes (\wedge^2 D_{\tau, K^p}^{\mathrm{sm}})^{\lambda_{\tau}+1})) \twoheadrightarrow D_{K^p}^{\Psi-\mathrm{la}, (1+\lambda_{\tau}, -1)_{\tau \in I}, (0, \lambda_{\tau})_{\tau \notin I}}((\sum_{\tau \in I}(\lambda_{\tau} + 1)))\end{equation} and consequently $D_k \twoheadrightarrow C_k$.

\begin{lem}\label{explicit kernel}

$\mathrm{Ker}\overline{d}_{I}^{\lambda} = \langle \mathrm{Sym}^{\lambda_{\tau}}V_{\tau} \otimes \mathcal{\omega}_{\tau, K^p}^{\lambda_{\tau} + 2, \mathrm{sm}} \otimes (\wedge^2 D_{\tau, K^p}^{\mathrm{sm}}) \otimes (D_{\lambda^{\Psi}, K^p}^{\Psi-\mathrm{la}, (0,0)_{\tau}, (0, \lambda_{\sigma})_{\sigma \neq \tau}})^{\mathfrak{gl}_2(L)_{\tau}} \otimes (\otimes_{\tau \neq \sigma \in I} (\mathcal{\omega}_{\sigma, K^p}^{2\lambda_{\sigma} + 2, \mathrm{sm}} \otimes (\wedge^2 D_{\sigma, K^p}^{\mathrm{sm}})^{\lambda_{\sigma}+1})) \mid \tau \in I \rangle_{C}$.

\end{lem}

\begin{proof} The universal quotient map $V_{\tau} \otimes_{L} \mathcal{O}_{\Fl} \rightarrow \omega_{\tau, \Fl} \otimes_{\mathcal{O}_{\Fl}} \wedge^2 V_{\tau}$ induces a morphism $\mathrm{Sym}^{\lambda_{\tau}}V_{\tau} \otimes_{L} C \rightarrow H^0(\Fl, (\omega_{\tau, \Fl} \otimes_{\mathcal{O}_{\Fl}} \wedge^2 V_{\tau})^{\lambda_{\tau}})$. This induces a map $\mathrm{Sym}^{\lambda_{\tau}}V_{\tau} \otimes \mathcal{\omega}_{\tau, K^p}^{\lambda_{\tau} + 2, \mathrm{sm}} \otimes (\wedge^2 D_{\tau, K^p}^{\mathrm{sm}}) \otimes (D_{\lambda^{\Psi}, K^p}^{\Psi-\mathrm{la}, (0,0)_{\tau}, (0, \lambda_{\sigma})_{\sigma \neq \tau}})^{\mathfrak{gl}_2(L)_{\tau}} \otimes (\otimes_{\tau \neq \sigma \in I} \mathcal{\omega}_{\sigma, K^p}^{2\lambda_{\sigma} + 2, \mathrm{sm}} \otimes (\wedge^2 D_{\sigma, K^p}^{\mathrm{sm}})^{\lambda_{\sigma}+1}) \rightarrow D_{\lambda^{\Psi}, K^p}^{\Psi-\mathrm{la}, (0, \lambda_{\tau})_{\tau}} \otimes (\otimes_{\tau \in I} \mathcal{\omega}_{\tau, K^p}^{2\lambda_{\tau} + 2, \mathrm{sm}} \otimes (\wedge^2 D_{\tau, K^p}^{\mathrm{sm}})^{\lambda_{\tau}+1})$. It suffices to check the injectivity of this map and the equality in the statement on $U \in \tilde{\mathcal{B}}$ satisfying $U \subset U_1$ because all considered sheaves are $G(\mathbb{Q}_p)$-equivariant sheaves. By the explicit formula Lemma \ref{anti-holomorphic}, we obtain the result because $\overline{d}_{\tau_i, 1}^{\lambda'}$ is equal to $(\frac{\partial}{\partial x_{\tau_i}})^{\lambda_{\tau_i}} \otimes dx_{\tau_i}^{\lambda_{\tau_i}}$ up to a scalar of $\mathbb{Q}_p^{\times}$. \end{proof}

\begin{lem}\label{construction of arithmetic}

   There exists a unique family of morphisms $(N_k : C_k \rightarrow C_{k+1})_{0 \le k \le d-1}$ such that $AGR^{\Psi-\mathrm{la}}_{\lambda} := ( C_{k}, N_{k})$ is a complex and the above map $D_k \rightarrow C_k$ (\ref{antiderivation}) induces the surjective morphism $GDR^{\Psi-\mathrm{la}}_{\lambda} \twoheadrightarrow ADR^{\Psi-\mathrm{la}}_{\lambda}$ of complexes.
\end{lem}

\begin{proof}

The uniqueness follows from the surjectivity. For the existence, it suffices to prove that $\mathrm{Ker}(D_k \twoheadrightarrow C_k)$ is sent to $\mathrm{Ker}(D_{k+1} \twoheadrightarrow C_{k+1})$ via the map $D_k \rightarrow D_{k+1}$ and this follows from the explicit description of $\overline{d}^{\lambda}_{I}$. (See the proof of Lemma \ref{explicit kernel}.) \end{proof}

The author hopes that this complex is equal to the complex obtained by the ``Fontaine operator''. In the following, we will explain the precise meaning of this. We assume that $F^+/\mathbb{Q}$ is Galois, $p$ is inert in $F^+$ and $\lambda = 0$ for a moment for simplicity. The vanishing of the ``Fontaine operator'' characterizes the de Rham representations among the Hodge-Tate representations. (We will see details about the Fontaine operator in {\S} 5.2.) More precisely, for a Hodge-Tate representation $\rho : G_{F_v} \rightarrow \mathrm{GL}_2(F_v)$ of $\tau$-Hodge-Tate weight $(0, 1)$ for any $\tau \in \Psi$ with a Hodge-Tate decomposition $\rho^{\tau} \otimes_{F_v} C = W_0^{\tau} \oplus W_1^{\tau}$, we can construct natural $C$-linear maps $N_{\tau} : W_0^{\tau} \rightarrow W_1^{\tau}(1)$ such that $N_{\tau} = 0$ for any $\tau$ is equivalent to the condition that $\rho$ is de Rham. By using this operator, we obtain the following complex \begin{equation}\label{Fontaine complex} \otimes_{\tau \in \Psi} (W_0^{\tau} \xrightarrow{N_{\tau}} W_1^{\tau}(1)).\end{equation} %or we can construct this complex more directly from the Fontaine operator on $\otimes_{\tau \in \Psi} \rho^{\tau}$ by changing sign appropriately. \footnote{Imagine the de Rham complex of affine spaces.}% 

By developing the theory of the Fontaine operator to Banach representations and LB representations, we can purely Galois-theoretically construct a sequence of maps $\mathcal{O}^{\Psi-\mathrm{la}, (0, 0)}_{K^p} \rightarrow \oplus_{\tau \in \Psi} \mathcal{O}^{\Psi-\mathrm{la}, (1, -1)_{\tau}, (0, 0)_{\sigma \neq \tau}}_{K^p}(1) \rightarrow \cdots \rightarrow \mathcal{O}^{\Psi-\mathrm{la}, (1, -1)}_{K^p}(d)$ such that the sequence of maps $$H^d(\Fl, \mathcal{O}^{\Psi-\mathrm{la}, (0, 0)}_{K^p})[\varphi] \rightarrow \oplus_{\tau \in \Psi} H^d(\Fl, \mathcal{O}^{\Psi-\mathrm{la}, (1, -1)_{\tau}, (0, 0)_{\sigma \neq \tau}}_{K^p})(1)[\varphi] \rightarrow \cdots \rightarrow H^d(\Fl, \mathcal{O}^{\Psi-\mathrm{la}, (1, -1)}_{K^p})(d)[\varphi]$$ is equal to the complex obtained by the above construction (\ref{Fontaine complex}) when $\rho$ is associated with an eigensystem $\varphi : \mathbb{T}^S(K^p, \mathcal{O})_{\mathfrak{m}} \rightarrow \mathcal{O}$ for a certain nice maximal ideal $\mathfrak{m}$. (We recall that we have the Hodge-Tate decomposition $\widehat{H}^d(S_{K^p}, L)^{\chi_{(0, 0)_{\tau}}} \widehat{\otimes}_{L} C = \otimes_{I \subset \Psi} H^d(\Fl, \mathcal{O}_{K^p}^{\Psi-\mathrm{la}, (1, -1)_{\tau \in I}, (0, 0)_{\tau \notin I}})$ and we have the concrete Galois structure of the completed cohomology $\widehat{H}^d(K^p, L)[\varphi] = \chi_{\varphi}^c|_{G_{F}} \otimes (\otimes_{\tau \in \Phi} \rho_{\varphi}^{\tau}) \otimes_L M$ for some smooth character $\chi_{\varphi}$ and some vector space $M$ over $L$ with the trivial Galois action. See Theorem \ref{residual irreducibility} and Proposition \ref{Hodge-Tate decomposition} for details.) The author hopes that this is a complex and equal to $ADR^{\Psi-\mathrm{la}}_{\lambda}$. - (*)

Let $N : (ADR^{\Psi-\mathrm{la}}_{\lambda})_0 \rightarrow (ADR^{\Psi-\mathrm{la}}_{\lambda})_1$ be the first map in the complex $ADR^{\Psi-\mathrm{la}}_{\lambda}$. Thus this is equal to $N = \oplus_{\tau \in \Psi} \overline{d}_{\tau}^{\lambda} \circ d_{\tau}^{\lambda}$, where the map $\overline{d}_{\tau}^{\lambda} := \overline{d}_{\{ \tau \}}^{\lambda}$ is defined before Lemma \ref{explicit kernel} and $d^{\lambda}_{\tau}$ denotes the composite of the first map $D_{\lambda^{\Psi}, K^p}^{\Psi-\mathrm{la},(0, \lambda_{\gamma})_{\gamma}} = (GDR^{\Psi-\mathrm{la}}_{\lambda})_0 \rightarrow (GDR^{\Psi-\mathrm{la}}_{\lambda})_1 = \oplus_{\sigma \in \Psi} D_{\lambda^{\Psi}, K^p}^{\Psi-\mathrm{la},(0, \lambda_{\tau})_{\gamma}} \otimes (\omega_{\sigma, K^p}^{2, \mathrm{sm}} \otimes \wedge^2 D_{\sigma, K^p})$ of $GDR^{\Psi-\mathrm{la}}_{\lambda}$ and the projection $\oplus_{\sigma \in \Psi} D_{\lambda^{\Psi}, K^p}^{\Psi-\mathrm{la},(0, \lambda_{\gamma})_{\gamma}} \otimes (\omega_{\sigma, K^p}^{2, \mathrm{sm}} \otimes \wedge^2 D_{\sigma, K^p}) \rightarrow D_{\lambda^{\Psi}, K^p}^{\Psi-\mathrm{la},(0, \lambda_{\gamma})_{\gamma}} \otimes (\omega_{\tau, K^p}^{2, \mathrm{sm}} \otimes \wedge^2 D_{\tau, K^p})$. In rest of this subsection, we fix a decomposed generic non-Eisenstein ideal $\mathfrak{m}$ of $\mathbb{T}^S(K^p, \mathcal{O})$ such that $\overline{\rho}_{\mathfrak{m}}(G_{F})$ is not solvable and $\otimes_{\tau \in \Psi} (\overline{\rho}_{\mathfrak{m}}|_{G_{\tilde{F}}})^{\tau}$ is irreducible. For any $\mathcal{O}$-morphism $\varphi : \mathbb{T}^S(K^p, \mathcal{O})_{\mathfrak{m}} \rightarrow \mathcal{O}$, we put $\rho_{\varphi} := \varphi_* \rho_{\mathfrak{m}}$ and $\chi_{\varphi} := \varphi_* \chi_{\mathfrak{m}}$. Inspired by the above hope (*), we propose the following conjecture.

\begin{conj}\label{key diagram}

Let $\varphi : \mathbb{T}^S(K^p, \mathcal{O})_{\mathfrak{m}} \rightarrow \mathcal{O}$ be an eigensystem such that $\rho_{\varphi}|_{G_{F_w}}$ is de Rham of $p$-adic Hodge type $\lambda_w$ for any $w \mid v$ and $\chi_{\varphi}|_{G_{F_{0, v^c}}}$ is de Rham of $p$-adic Hodge type $\lambda_0$. If $\widehat{H}^d(S_{K^p}, V_{\lambda^{\Psi}})_{\mathfrak{m}}^{\Psi-\mathrm{la}}[\varphi] \neq 0$, then $\mathrm{Ker} H^d(\Fl, N)_{\mathfrak{m}}[\varphi] \neq 0$.

\end{conj}

\begin{rem}

Actually, as in \cite{Che}, by using \cite[Theorem B]{monodromy}, we can prove that $\rho_{\varphi}|_{G_{F_w}}$ is always de Rham of $p$-adic Hodge type $\lambda_w$ for any $w \mid v$ not contained in $S_0$. Thus we can replace the above assumption on $\rho_{\varphi}$ by the assumption that $\rho_{\varphi}|_{G_{F_w}}$ is de Rham of $p$-adic Hodge type $\lambda_w$ for any $w \mid v$ not contained in $S_0$. However, in our application, we will not mind such a technical difference. Thus the author decided to state the above weaker conjecture.

\end{rem}

In {\S} 5.2, we will prove this conjecture when one of the following conditions holds by using a similar method as \cite[{\S} 6]{PanII}.

\begin{itemize}
\item $\lambda_{\Psi}$ is parallel.
\item $|\Psi| = 2$.
\end{itemize} 

In the following of this subsection, we will see a consequence of the above conjecture.

\begin{prop}\label{degeneration}

$\mathrm{Ker}H^d(\Fl, N)_{\mathfrak{m}} = H^d(\Fl, ADR_{\lambda}^{\Psi-\mathrm{la}})_{\mathfrak{m}}$.

\end{prop}

\begin{proof}

We have the following spectral sequence.

$E_1^{p, q} := \oplus_{I \subset \Psi, |I|=p} H^q(\Fl, D_{\lambda^{\Psi}, K^p}^{\Psi-\mathrm{la}, (0, \lambda_{\tau})_{\tau \in I}, (1+\lambda_{\tau}, -1)_{\tau \notin I}}(\sum_{\tau \in I}(\lambda_{\tau} + 1)))_{\mathfrak{m}} \Rightarrow H^{p+q}(\Fl, ADR^{\Psi-\mathrm{la}}_{\lambda})_{\mathfrak{m}}$ 

By Theorem \ref{geometric sen theoryII}, we have $E_1^{p, q} = 0$ for any $q \neq d$. Therefore, we obtain $H^d(\Fl, ADR_{\lambda}^{\Psi-\mathrm{la}})_{\mathfrak{m}} = E_2^{0, d} = \mathrm{Ker}H^d(\Fl, N)_{\mathfrak{m}}$. \end{proof}

Let $Ker := \mathrm{Ker}(GDR_{\lambda}^{\Psi-\mathrm{la}} \twoheadrightarrow ADR^{\Psi-\mathrm{la}}_{\lambda})$. Let $\{ \mathrm{Fil}^i Ker \}_{i = 1, \cdots, d}$ be the exhaustive increasing filtration on $Ker$ defined by $(\mathrm{Fil}^i Ker)_{k} := \langle \mathrm{Sym}^{\lambda_{\tau}}V_{\tau} \otimes \mathcal{\omega}_{\tau, K^p}^{\lambda_{\tau} + 2, \mathrm{sm}} \otimes (\wedge^2 D_{\tau, K^p}^{\mathrm{sm}}) \otimes (D_{\lambda^{\Psi}, K^p}^{\Psi-\mathrm{la}, (0,0)_{\tau}, (0, \lambda_{\sigma})_{\sigma \neq \tau}})^{\mathfrak{gl}_2(L)_{\tau}} \otimes (\otimes_{\tau \neq \sigma \in I} (\mathcal{\omega}_{\sigma, K^p}^{2\lambda_{\sigma} + 2, \mathrm{sm}} \otimes (\wedge^2 D_{\sigma, K^p}^{\mathrm{sm}})^{\lambda_{\sigma}+1})) \mid |I| = k, \tau \in I \cap \{ \tau_1, \cdots, \tau_i \} \rangle_{C}$. (See Lemma \ref{explicit kernel}. This is actually a subcomplex by using the explicit formula Lemma \ref{anti-holomorphic}.) By definition, we have the following.

\begin{lem}\label{kernel}

For any $i$, we have $\mathrm{Sym}^{\lambda_{\tau_i}}V_{\tau_i} \otimes_L \mathrm{Hom}_{\mathfrak{gl}_2(L)_{\tau_i}}(\mathrm{Sym}^{\lambda_{\tau_i}}V_{\tau_i}, \mathrm{gr}^i Ker) \Isom \mathrm{gr}^i Ker$.

\end{lem}

\begin{thm}\label{induction step}

Assume Conjectures \ref{classicality conjecture} and \ref{key diagram}. Let $\varphi : \mathbb{T}^S(K^p, \mathcal{O})_{\mathfrak{m}} \rightarrow \mathcal{O}$ be an $\mathcal{O}$-morphism. Assume that $\varphi$ is not a classical eigensystem of weight $\lambda$, $\rho_{\varphi}|_{G_{F_w}}$ is de Rham of $p$-adic Hodge type $\lambda_w$ for any $w \mid v$, $\chi_{\varphi}|_{G_{F_{0, v^c}}}$ is de Rham of $p$-adic Hodge type $\lambda_0$ and $\widehat{H}^d(S_{K^p}, V_{\lambda^{\Psi}})^{\Psi-\mathrm{la}}_{\mathfrak{m}}[\varphi] \neq 0$. Then there exists $\tau \in \Psi$ such that $\widehat{H}^d(S_{K^p}, V_{\lambda^{\Psi \setminus \{ \tau \}}})_{\mathfrak{m}}^{\Psi \setminus \{ \tau \}-\mathrm{la}}[\varphi] \neq 0$.

\end{thm}

\begin{rem} 

In {\S} 6, we will prove a certain comparison theorem on the completed cohomologies under a certain technical assumption on $\mathfrak{m}$. (See Theorem \ref{comparison}.) This implies that we may replace $S_K$ by a unitary Shimura variety $T_{K}$ of dimension $d-1$ associated with $\Psi \setminus \{ \tau \}$. By applying Theorem \ref{induction step} for $T_K$ again and repeating this argument, we can reduce the proof of the classicality of $\varphi$ to the unitary Shimura curve case and deduce the classicality of $\varphi$.

\end{rem}

\begin{proof} By Proposition \ref{degeneration}, we have $H^d(\Fl, ADR^{\Psi-\mathrm{la}}_{\lambda})_{\mathfrak{m}}[\varphi] = \mathrm{Ker}H^d(\Fl, N)_{\mathfrak{m}}[\varphi] \neq 0$ because we assume Conjecture \ref{key diagram}. We have an exact sequence $H^d(\Fl, GDR^{\Psi-\mathrm{la}}_{\lambda}) \rightarrow H^d(\Fl, ADR_{\lambda}^{\Psi-\mathrm{la}}) \rightarrow H^{d+1}(\Fl, Ker)$.

The localization $H^d(\Fl, GDR^{\Psi-\mathrm{la}}_{\lambda})_{\varphi}$ of $H^d(\Fl, GDR^{\Psi-\mathrm{la}}_{\lambda})$ at the prime ideal corresponding to $\varphi$ vanishes since we assume Conjecture \ref{classicality conjecture}. Therefore, we obtain $H^{d+1}(\Fl, Ker)_{\mathfrak{m}}[\varphi] \supset H^d(\Fl, ADR^{\Psi-\mathrm{la}}_{\lambda})_{\mathfrak{m}}[\varphi] = \mathrm{Ker}H^d(\Fl, N)_{\mathfrak{m}}[\varphi]$. The filtration $\{ \mathrm{Fil}^iKer \}_{0 \le i \le d}$ on $Ker$ induces the increasing exhaustive filtration $\{ \mathrm{Fil}^i \}_{0 \le i \le d}$ on $\mathrm{Ker}H^d(\Fl, N)_{\mathfrak{m}}[\varphi]$ such that $$\mathrm{Sym}^{\lambda_{\tau_i}}V_{\tau_i} \otimes_L \mathrm{Hom}_{\mathfrak{gl}_2(L)_{\tau_i}}(\mathrm{Sym}^{\lambda_{\tau_i}}V_{\tau_i}, \mathrm{gr}^i) \Isom \mathrm{gr}^i$$ for any $i$. Let $i_0$ be the minimum $i$ such that $\mathrm{gr}^{i} \neq 0$. Thus $\mathrm{gr}^{i_0} = \mathrm{Fil}^{i_0}$. Then we have $$\mathrm{Hom}_{\mathfrak{gl}_2(L)_{\tau_{i_0}}}(\mathrm{Sym}^{\lambda_{\tau_{i_0}}}V_{\tau_{i_0}}, \mathrm{Ker}H^d(\Fl, N)_{\mathfrak{m}}[\varphi]) = (\mathrm{Ker}H^d(\Fl, N)_{\mathfrak{m}}[\varphi] \otimes \mathrm{Sym}^{\lambda_{\tau_{i_0}}}V_{\tau_{i_0}}^{\vee})^{\mathfrak{gl}_2(L)_{\tau_{i_0}}} \neq 0.$$

Note that we have $\widehat{H}^d(S_{K^p}, V_{\lambda^{\Psi}}(-\lambda_0))_{\mathfrak{m}}^{\Psi-\mathrm{la}, \chi_{\lambda_{\Psi}}} \widehat{\otimes}_{L} C \cong \oplus_{I \subset \Psi} H^d(\Fl, D^{\Psi-\mathrm{la}, (0, \lambda_{\tau})_{\tau \in I}, ( 1 +\lambda_{\tau}, -1)_{\tau \notin I}}_{K^p, \lambda^{\Psi}})_{\mathfrak{m}}$ by Proposition \ref{Hodge-Tate decomposition}. Thus $\mathrm{Ker}H^d(\Fl, N)_{\mathfrak{m}}[\varphi] \subset \widehat{H}^d(S_{K^p}, V_{\lambda^{\Psi}}(-\lambda_0))_{\mathfrak{m}}^{\Psi-\mathrm{la}, \chi_{\lambda_{\Psi}}} \widehat{\otimes}_{L} C$. In particular, we obtain $(\widehat{H}^d(S_{K^p}, V_{\lambda^{\Psi}})_{\mathfrak{m}}^{\Psi-\mathrm{la}}[\varphi] \otimes \mathrm{Sym}^{\lambda_{\tau_{i_0}}}V_{\tau}^{\vee})^{\mathfrak{gl}_2(L)_{\tau_{i_0}}} = \widehat{H}^d(S_{K^p}, V_{\lambda^{\Psi \setminus \{\tau_{i_0} \}}})_{\mathfrak{m}}^{\Psi \setminus \{\tau_{i_0}\}-\mathrm{la}}[\varphi] \neq 0$. \end{proof}

\subsection{Fontaine operators and arithmetic locally analytic de Rham complexes}

The purpose of this subsection is to prove Conjecture \ref{key diagram} in the parallel weight case and when $d \le 2$.

\subsubsection{Preliminaries on the Fontaine operator and period sheaves}

Let $m$ and $k$ be positive integers. We put $\varepsilon := (1, \zeta_{p}, \zeta_{p^2}, \cdots ) \in \mathcal{O}_{C}^{\flat}:=\varprojlim_{x \mapsto x^p}\mathcal{O}_C/p$, $t:=\mathrm{log}([\varepsilon])$ and $B_{\mathrm{dR}, m}^+ := B_{\mathrm{dR}}^+/t^m$. Let $W$ be a Banach space over $\mathbb{Q}_p$ with a continuous action $B_{\mathrm{dR}, m}^+ \times W \rightarrow W$. We put $\mathrm{Fil}^iW := t^iW$ and $\mathrm{gr}^iW:=t^iW/t^{i+1}W$.

We assume that $t: \mathrm{gr}^iW(1) \rightarrow \mathrm{gr}^{i+1}W$ is an isomorphism for any $i = 0, \cdots, m-2$ and $\mathrm{gr}^0W$ is Hodge-Tate with weights in $[0, k]$, i.e., $\mathrm{gr}^0W$ is Hodge-Tate of some integers $w_1, \cdots, w_r \in [0, k]$. Note that $t^iW$ is a closed subspace of $W$ by the above first assumption. 

We recall the notations $H_L := \mathrm{Gal}(\overline{L}/L^{\mathrm{cycl}})$, $\Gamma_L := \mathrm{Gal}(L^{\mathrm{cycl}}/L)$ and $W^L := W^{H_L, \Gamma_L-\mathrm{an}}$ defined in {\S} 3.5. Note that by (8) of Lemma \ref{HTE}, we have $\mathrm{gr}^i(W^L) = (\mathrm{gr}^iW)^L$.

Let $\theta_{W} : W^L \rightarrow W^L$ be the Sen operator, i.e., the action of $1 \in \mathbb{Q}_p = \mathrm{Lie}\Gamma_{L}$. Let $E_0(W^L) := W^{L, \theta_W-\mathrm{nilp}}$ be the subspace of $W^L$ consisting of elements $x$ such that $\theta_W^s(x) = 0$ for some integer $s > 0$ and $N_W := \theta_W|_{E_0(W^L)}$. Note that $t^i$ induces $(\mathrm{gr}^0W(i))^{G_L} \Isom \mathrm{gr}^i E_0(W^L)$ and $N_W$ stabilizes the filtration on $E_0(W^L)$ induced from the $t$-adic filtration on $W$. %Thus we obtain $\mathrm{gr}^i N_W : (\mathrm{gr}^0W(i))^{G_L} \rightarrow (\mathrm{gr}^0W(i+1))^{G_L}$ and $\mathrm{gr}^i N_W \widehat{\otimes}_{L} C : (\mathrm{gr}^0W(i))^{G_L} \widehat{\otimes}_{L} C \rightarrow (\mathrm{gr}^0W(i+1))^{G_L} \widehat{\otimes}_{L} C$.

\begin{prop}

Let $V$ be a finite-dimensional representation of $G_L$ over $\mathbb{Q}_p$ which is Hodge-Tate with weights in $[0, k]$. We put $W := V \otimes_{\mathbb{Q}_p} B_{\mathrm{dR}, m}^+$.

If $V$ is de Rham, then $N_{W} = 0$.

Conversely, if $m > k$ and $N_W = 0$, then $V$ is de Rham.

\end{prop}

\begin{proof} We may assume $m > k$. Since we assume that $V$ is Hodge-Tate, we have $\mathrm{dim}_LE_0(W^L) = \mathrm{dim}_{\mathbb{Q}_p} V$ by the above identification $(\mathrm{gr}^0W(i))^{G_L} = \mathrm{gr}^i E_0(W^L)$. Thus $V$ is de Rham if and only if $\mathrm{dim}_{L} E_0(W^L) = \mathrm{dim}_LD_{\mathrm{dR}}(V)$. Note that $D_{\mathrm{dR}}(V) = E_0(W^L)^{N_W=0}$ and thus the above condition is equivalent to $E_0(W^L) = E_0(W^L)^{N_W=0}$ and this means $N_W = 0$. \end{proof}

\begin{cor}\label{deRhamness}

Let $V$ be a finite-dimensional representation of $G_L$ over $L$ which is Hodge-Tate of weights in $[0, k]$. We put $W_{\sigma} := V^{\sigma} \otimes_{L} B_{\mathrm{dR}, m}^+$ for any $\sigma \in \mathrm{Gal}(L/\mathbb{Q}_p)$. ($V^{\sigma}$ denotes the conjugate representation of $V$ by $\sigma$. See the comments before Theorem \ref{kottwitz conjecture}.)
    
If $V$ is de Rham, then $N_{W^{\sigma}} = 0$ for any $\sigma \in \mathrm{Gal}(L/\mathbb{Q}_p)$.

Conversely, if $m > k$ and $N_{W^{\sigma}} = 0$ for any $\sigma$, then $V$ is de Rham.
    
\end{cor}

\begin{proof}

We have a decomposition $V \otimes_{\mathbb{Q}_p} L \cong \oplus_{\sigma \in \mathrm{Gal}(L/\mathbb{Q}_p)} V \otimes_{L, \sigma} L$. Thus the result follows from $(V \otimes_{L, \sigma} B_{\mathrm{dR}, m})^{\sigma} \cong V^{\sigma} \otimes_{L} B_{\mathrm{dR}, m}$. \end{proof}

We recall that we have sheaves $\mathbb{B}^+_{\mathrm{dR}}$ and $\mathcal{O}\mathbb{B}_{\mathrm{dR}}^+$ on $(\aS_{K, L})_{\proet}$. (See \cite[{\S} 6]{pH}.) We also use the same notations for the restrictions to $\aS_{K^p}$. Let $\mathbb{B}_{\mathrm{dR}, m}^+ := \mathbb{B}_{\mathrm{dR}}^+/\mathrm{Fil}^m \mathbb{B}_{\mathrm{dR}}^+$. We also write $\mathbb{B}_{\mathrm{dR}, m}^+$ for the pushforward of $\mathbb{B}_{\mathrm{dR}, m}^+$ via $\pi_{\mathrm{HT}}$. Let $\varinjlim_{L'' \supset L} E_0((V_{\lambda^{\Psi}}(-\lambda_0) \otimes_L \mathbb{B}_{\mathrm{dR}, m}^+)^{\Psi-\mathrm{la}, \chi_{\lambda_{\Psi}}, L''})$ be the sheaf on $\Fl$ defined by $$U \mapsto \varinjlim_{L'' \supset L'} E_0((V_{\lambda^{\Psi}}(-\lambda_0) \otimes_L \mathbb{B}_{\mathrm{dR}, m}^+(U))^{\Psi-\mathrm{la}, \chi_{\lambda_{\Psi}}, L''})$$ for any quasicompact open $U$ of $\Fl$ and any sufficiently large finite extension $L'$ of $L$ such that $\mathbb{B}_{\mathrm{dR}, m}^+(U)$ has a natural $G_{L'}$-action and $L''$ runs through finite extensions of $L'$. Let $N_{\lambda}$ be the Fontaine operator on $\varinjlim_{L'' \supset L} E_0((V_{\lambda^{\Psi}}(-\lambda_0) \otimes_L \mathbb{B}_{\mathrm{dR}, m}^+)^{\Psi-\mathrm{la}, \chi_{\lambda_{\Psi}}, L''})$

\begin{lem}\label{de Rham structure}

1 \ $\widehat{H}^i(S_{K^p}, V_{\lambda^{\Psi}}(-\lambda_0))_{\mathfrak{m}}^{\Psi-\mathrm{la}, \chi_{\lambda_{\Psi}}} \widehat{\otimes}_{L} B^+_{\mathrm{dR}, m} \cong H^i(\Fl, (V_{\lambda^{\Psi}}(-\lambda_0) \otimes_L \mathbb{B}_{\mathrm{dR}, m}^+)^{\Psi-\mathrm{la}, \chi_{\lambda_{\Psi}}})_{\mathfrak{m}}$ and these vanish for any $i \neq d$. Moreover, the latter cohomology group can be computed by the $\check{C}$ech complex $C(\mathcal{B}', (V_{\lambda^{\Psi}}(-\lambda_0) \otimes_L \mathbb{B}_{\mathrm{dR}, m}^+)^{\Psi-\mathrm{la}, \chi_{\lambda_{\Psi}}})_{\mathfrak{m}}$ for any finite covering $\mathcal{B}' \subset \mathcal{B}$.

2 \ We have a natural map $$H^d(\Fl, \varinjlim_{L'' \supset L} E_0((V_{\lambda^{\Psi}}(-\lambda_0) \otimes_L \mathbb{B}_{\mathrm{dR}, m}^+)^{\Psi-\mathrm{la}, \chi_{\lambda_{\Psi}}, L''}))_{\mathfrak{m}} \rightarrow \varinjlim_{L'' \supset L} E_0((\widehat{H}^d(S_{K^p}, V_{\lambda^{\Psi}}(-\lambda_0))_{\mathfrak{m}}^{\Psi-\mathrm{la}, \chi_{\lambda_{\Psi}}} \widehat{\otimes}_{L} B^+_{\mathrm{dR}, m})^{L''})$$ which is compatible with the Fontaine operator on $\varinjlim_{L'' \supset L}(\widehat{H}^d(S_{K^p}, V_{\lambda^{\Psi}}(-\lambda_0))_{\mathfrak{m}}^{\Psi-\mathrm{la}, \chi_{\lambda_{\Psi}}} \widehat{\otimes}_{L} B^+_{\mathrm{dR}, m})^{L''}$ and $H^d(\Fl, N_{\lambda})$ on $H^d(\Fl, \varinjlim_{L'' \supset L} E_0((V_{\lambda^{\Psi}}(-\lambda_0) \otimes_L \mathbb{B}_{\mathrm{dR}, m}^+)^{\Psi-\mathrm{la}, \chi_{\lambda_{\Psi}}, L''}))_{\mathfrak{m}}$.

\end{lem}

\begin{rem}

The author doesn't know whether the map in 2 is an isomorphism or not, but we don't need such a result.

\end{rem}

\begin{proof} 1 follows from the primitive comparison \cite[Theorem 1.3]{pH}, Proposition \ref{Hodge-Tate decomposition} and taking graded pieces. 2 follows from the fact that the given map is induced from the Galois equivariant continuous map $Z^i(\mathcal{B}', (V_{\lambda^{\Psi}}(-\lambda_0) \otimes_L \mathbb{B}_{\mathrm{dR}, m}^+)^{\Psi-\mathrm{la}, \chi_{\lambda_{\Psi}}}) \rightarrow \widehat{H}^i(S_{K^p}, V_{\lambda^{\Psi}}(-\lambda_0))_{\mathfrak{m}}^{\Psi-\mathrm{la}, \chi_{\lambda_{\Psi}}} \widehat{\otimes}_{L} B^+_{\mathrm{dR}, m}$ given in 1, where $Z^i(\mathcal{B}', (V_{\lambda^{\Psi}}(-\lambda_0) \otimes_L \mathbb{B}_{\mathrm{dR}, m}^+)^{\Psi-\mathrm{la}, \chi_{\lambda_{\Psi}}})$ denotes the kernel of $C^i(\mathcal{B}', (V_{\lambda^{\Psi}}(-\lambda_0) \otimes_L \mathbb{B}_{\mathrm{dR}, m}^+)^{\Psi-\mathrm{la}, \chi_{\lambda_{\Psi}}}) \rightarrow C^{i+1}(\mathcal{B}', (V_{\lambda^{\Psi}}(-\lambda_0) \otimes_L \mathbb{B}_{\mathrm{dR}, m}^+)^{\Psi-\mathrm{la}, \chi_{\lambda_{\Psi}}})$.   \end{proof}

Note the following lemma, which is useful later.

\begin{lem}\label{formal power}

Let $K$ be a complete discrete valuation field extension of $\mathbb{Q}_p$ with a perfect residue field and $X$ be a smooth rigid analytic space over $K$. We put $\mathcal{O}\mathbb{B}^+_{\mathrm{dR},k} := \mathcal{O}\mathbb{B}^+_{\mathrm{dR}}/\mathrm{Fil}^k\mathcal{O}\mathbb{B}_{\mathrm{dR}}^+$ and $\mathbb{B}_{\mathrm{dR}, k}^+ := \mathbb{B}_{\mathrm{dR}}^+/\mathrm{Fil}^k\mathbb{B}^+_{\mathrm{dR}}$. (Precisely, we should write $\mathcal{O}\mathbb{B}^+_{\mathrm{dR}, X}$ and $\mathbb{B}^+_{\mathrm{dR}, X}$ to distinguish the previous $\mathcal{O}\mathbb{B}^+_{\mathrm{dR}}$ and $\mathbb{B}^+_{\mathrm{dR}}$ on $(\aS_{K, L})_{\proet}$.)

Then we have the following results.

1 \ There exists a canonical isomorphism of filtered $\widehat{\mathcal{O}}_X = \mathbb{B}_{\mathrm{dR}, 1}^+$-modules $\mathcal{O}\mathbb{B}_{\mathrm{dR}, k}^+ \otimes_{\mathbb{B}_{\mathrm{dR}, k}^+} \widehat{\mathcal{O}}_X \cong \oplus_{i=0}^{k-1} \widehat{\mathcal{O}}_{X} \otimes_{\mathcal{O}_{X}} \mathrm{Sym}^i\Omega_{X}^1$. Here, we consider the filtration $\{ \mathrm{Fil}^j \}_{j=0}^{k-1}$ on the right hand side given by $\mathrm{Fil}^j = \oplus_{i=j}^{k-1} \widehat{\mathcal{O}}_{X} \otimes_{\mathcal{O}_{X}} \mathrm{Sym}^i\Omega_{X}^1$.

2 \ The induced connection $\nabla : \mathcal{O}\mathbb{B}_{\mathrm{dR}, k}^+ \otimes_{\mathbb{B}_{\mathrm{dR}, k}^+} \widehat{\mathcal{O}}_X = \oplus_{i=0}^{k-1} \widehat{\mathcal{O}}_{X} \otimes_{\mathcal{O}_{X}} \mathrm{Sym}^i\Omega_{X}^1 \rightarrow \mathcal{O}\mathbb{B}_{\mathrm{dR}, k-1}^+ \otimes_{\mathbb{B}_{\mathrm{dR}, k-1}^+} \widehat{\mathcal{O}}_X \otimes_{\mathcal{O}_X} \Omega_{X}^1 = \oplus_{i=0}^{k-1} \widehat{\mathcal{O}}_{X} \otimes_{\mathcal{O}_{X}} \mathrm{Sym}^i\Omega_{X}^1 \otimes \Omega_X^1$ respects the natural graded pieces of both sides and is given by $\widehat{\mathcal{O}}_{X} \otimes_{\mathcal{O}_{X}} \mathrm{Sym}^i\Omega_{X}^1 \hookrightarrow \widehat{\mathcal{O}}_{X} \otimes_{\mathcal{O}_{X}} \mathrm{Sym}^{i-1}\Omega_{X}^1 \otimes_{\mathcal{O}_X} \Omega_{X}^1, \ f \otimes x_1 \cdots x_i \mapsto \sum_{j=1}^i f \otimes x_1 \cdots \breve{x_{j}} \cdots x_i \otimes x_j$ and the zero map on $i = 0$. Here, $\breve{x_{j}}$ means that we omit $x_{j}$ and we define the degree of $\widehat{\mathcal{O}}_{X} \otimes_{\mathcal{O}_{X}} \mathrm{Sym}^i\Omega_{X}^1$ to be $i$ and the degree of $\widehat{\mathcal{O}}_{X} \otimes_{\mathcal{O}_{X}} \mathrm{Sym}^i\Omega_{X}^1 \otimes \Omega_X^1$ to be $i+1$.

\end{lem}
    
\begin{rem}
    
Note that the $\mathcal{O}_X$-module structure on $\mathcal{O}\mathbb{B}_{\mathrm{dR}, k}^+ \otimes_{\mathbb{B}^+_{\mathrm{dR}, k}} \widehat{\mathcal{O}}_{X}$ induced by $\mathcal{O}\mathbb{B}_{\mathrm{dR}, k}^+$ is not equal to the $\mathcal{O}_X$-module structure induced by the $\widehat{\mathcal{O}}_{X} = \mathbb{B}_{\mathrm{dR}, 1}$-module structure. 

\end{rem}

\begin{proof}

We will prove this by induction on $k$. This is trivial if $k=1$. Assume that $k \ge 2$ and the above statement holds for $k-1$. 

First, we recall a pro$\et$ale local description of $\mathcal{O}\mathbb{B}_{\mathrm{dR}}^+$. We take a small affinoid open $U$ of $X$ and an $\et$ale map $U \rightarrow \mathbb{T}^d_K:=\mathrm{Spa}(K\langle X_1^{\pm 1}, \cdots, X_d^{\pm 1} \rangle, \mathcal{O}_K\langle X_1^{\pm 1}, \cdots, X_d^{\pm 1} \rangle)$, which is a composition of finite $\et$ale maps and rational open immersions. Let $$\mathbb{T}^d_{\widehat{\overline{K}}, \infty} :=\mathrm{Spa}(\widehat{\overline{K}}\langle X_1^{\pm\frac{1}{p^{\infty}}}, \cdots, X_d^{\pm\frac{1}{p^{\infty}}} \rangle, \mathcal{O}_{\widehat{\overline{K}}}\langle X_1^{\pm\frac{1}{p^{\infty}}}, \cdots, X_d^{\pm\frac{1}{p^{\infty}}} \rangle)$$ be the usual affinoid perfectoid cover of $\mathbb{T}^d_K$, where $\widehat{\overline{K}}$ denotes the completed algebraic closure of $K$. Let $U_{\infty} := U \times_{\mathbb{T}^d_K} \mathbb{T}^d_{\widehat{\overline{K}}, \infty}$ be the fiber product in $(\mathbb{T}^d_K)_{\mathrm{pro\et}}$, which is an affinoid perfectoid space. Then the $\mathbb{B}_{\mathrm{dR}}^+|_{U_{\infty}}$-algebra map $\mathbb{B}_{\mathrm{dR}}^{+}|_{U_{\infty}}[[Y_1, \cdots, Y_d]] \rightarrow \mathcal{O}\mathbb{B}_{\mathrm{dR}}^+|_{U_{\infty}}$ defined by $Y_i \mapsto X_i - [X_i^{\flat}]$ is a filtered isomorphism by \cite[Proposition 6.10]{pH}, where $X_i^{\flat} := (X_i, X_i^{\frac{1}{p}}, X_i^{\frac{1}{p^2}}, \cdots, )$ and $[X_i^{\flat}]$ denotes the Teichmuller lift of $X_i^{\flat}$. Here, on the right hand side, we put $\mathrm{Fil}^i\mathbb{B}_{\mathrm{dR}}^{+}|_{U_{\infty}}[[Y_1, \cdots, Y_d]] := (t, Y_1, \cdots, Y_d)^i$. Note that we have $\nabla(Y_i) = dX_i$. Thus the result is clearly true on $U_{\infty}$ by identifying $Y_i$ with $dX_i$. In the following, we will give a canonical construction of this identification.

The Faltings extension $0 \rightarrow \widehat{\mathcal{O}}_X(1) \rightarrow \mathrm{gr}^1\mathcal{O}\mathbb{B}_{\mathrm{dR}}^+ \rightarrow \widehat{\mathcal{O}}_X \otimes_{\mathcal{O}_X} \Omega_X^1 \rightarrow 0$ induces the exhaustive increasing filtration $\{ \mathrm{Fil}^s \}_{s=0}^{i}$ on $\mathrm{gr}^i\mathcal{O}\mathbb{B}_{\mathrm{dR}}^+ = \mathrm{Sym}^i\mathrm{gr}^1\mathcal{O}\mathbb{B}_{\mathrm{dR}}^+$ such that $\mathrm{gr}^s = \widehat{\mathcal{O}}_X \otimes_{\mathcal{O}_X} \mathrm{Sym}^s\Omega_{X}^1(i-s)$.

By the above local description of $\mathcal{O}\mathbb{B}_{\mathrm{dR}}^+$, the map $\mathrm{gr}^{k-1}\mathcal{O}\mathbb{B}_{\mathrm{dR}}^+ \hookrightarrow \mathcal{O}\mathbb{B}_{\mathrm{dR}, k}^+ \twoheadrightarrow \mathcal{O}\mathbb{B}_{\mathrm{dR}, k}^+ \otimes_{\mathbb{B}_{\mathrm{dR}, k}^+} \widehat{\mathcal{O}}_X$ factor through $\mathrm{gr}^{k-1}\mathcal{O}\mathbb{B}_{\mathrm{dR}}^+ \twoheadrightarrow \widehat{\mathcal{O}}_X \otimes_{\mathcal{O}_X} \mathrm{Sym}^{k-1}\Omega_{X}^1$ and induces the exact sequence $0 \rightarrow \widehat{\mathcal{O}}_X \otimes_{\mathcal{O}_X} \mathrm{Sym}^{k-1}\Omega_{X}^1 \rightarrow \mathcal{O}\mathbb{B}_{\mathrm{dR}, k}^+ \otimes_{\mathbb{B}_{\mathrm{dR}, k}^+} \widehat{\mathcal{O}}_X \rightarrow \mathcal{O}\mathbb{B}_{\mathrm{dR}, k-1}^+ \otimes_{\mathbb{B}_{\mathrm{dR}, k-1}^+} \widehat{\mathcal{O}}_X \rightarrow 0$. By our induction hypothesis, we have $\mathcal{O}\mathbb{B}_{\mathrm{dR}, k-1}^+ \otimes_{\mathbb{B}_{\mathrm{dR}, k-1}^+} \widehat{\mathcal{O}}_X = \oplus_{i=0}^{k-2} \widehat{\mathcal{O}}_{X} \otimes_{\mathcal{O}_{X}} \mathrm{Sym}^i\Omega_{X}^1$. Moreover, the image of the map $\mathcal{O}\mathbb{B}_{\mathrm{dR}, k}^+ \otimes_{\mathbb{B}_{\mathrm{dR}, k}^+} \widehat{\mathcal{O}}_X \xrightarrow{\nabla} (\mathcal{O}\mathbb{B}_{\mathrm{dR}, k-1}^+ \otimes_{\mathbb{B}_{\mathrm{dR}, k-1}^+} \widehat{\mathcal{O}}_X) \otimes_{\mathcal{O}_X} \Omega_{X}^1 = \oplus_{i=0}^{k-2} \widehat{\mathcal{O}}_{X} \otimes_{\mathcal{O}_{X}} \mathrm{Sym}^i\Omega_{X}^1 \otimes \Omega_{X}^1 \rightarrow \widehat{\mathcal{O}}_{X} \otimes_{\mathcal{O}_{X}} \mathrm{Sym}^{k-2}\Omega_{X}^1 \otimes \Omega_{X}^1$ is isomorphic to the subsheaf $\widehat{\mathcal{O}}_X \otimes_{\mathcal{O}_X} \mathrm{Sym}^{k-1}\Omega_{X}^1$ of $\mathcal{O}\mathbb{B}_{\mathrm{dR}, k}^+ \otimes_{\mathbb{B}_{\mathrm{dR}, k}^+} \widehat{\mathcal{O}}_X$ by restricting this map to this subsheaf, which is shown by using the above local description. This fact induces the splitting of the above exact sequence $0 \rightarrow \widehat{\mathcal{O}}_X \otimes_{\mathcal{O}_X} \mathrm{Sym}^{k-1}\Omega_{X}^1 \rightarrow \mathcal{O}\mathbb{B}_{\mathrm{dR}, k}^+ \otimes_{\mathbb{B}_{\mathrm{dR}, k}^+} \widehat{\mathcal{O}}_X \rightarrow \mathcal{O}\mathbb{B}_{\mathrm{dR}, k-1}^+ \otimes_{\mathbb{B}_{\mathrm{dR}, k-1}^+} \widehat{\mathcal{O}}_X \rightarrow 0$, which is shown to be the desired splitting by using the above local description. \end{proof}

Let $U \in \tilde{\mathcal{B}}$, $K_p$, $V_{K_p}$ and $G_m$ be as after Proposition \ref{sen operator}. After enlarging $L$ if necessary, we can take an $L$-structure $V_0$ of $V_{K_p}$. In the following, for the simplicity of notations, we use notations like $D_{{\lambda^{\Psi}}, K^p}^{\Psi-\mathrm{la}, (0, \lambda)_{\tau}, G_L}$ and $\mathcal{O}_{V_0}$ instead of $D_{{\lambda^{\Psi}}, K^p}^{\Psi-\mathrm{la}, (-\lambda, 0)_{\tau}}(U)^{G_L}$ and $\mathcal{O}_{V_0}(V_0)$. Let $\chi_{\lambda_{\Psi}}$ be the infinitesimal character of $V_{\lambda_{\Psi}}^{\vee}$.

\begin{lem} \label{derived calculationIII}
 
    We have $\mathrm{Ext}^i_{Z(U(\mathfrak{g}))}(\chi_{\lambda_{\Psi}}, D_{K^p}^{\Psi-\mathrm{la}}) = 0$ for any $i > 0$. \end{lem}

\begin{proof} By Lemma \ref{infinitesimal character and horizontal} and 1 of Proposition \ref{derived calculation II}, we have $$\mathrm{Ext}^i_{Z(U(\mathfrak{g}))}(\chi_{\lambda_{\Psi}}, D_{K^p}^{\Psi-\mathrm{la}}) = \oplus_{I \subset \Psi} H^i(\mathfrak{h}, D_{K^p}^{\Psi-\mathrm{la}}((0, -\lambda_{\tau})_{\tau \in I}, (-\lambda_{\tau} - 1, 1)_{\tau \in \Psi \setminus I})) = 0.$$ \end{proof}

We also need variants of 1 of Proposition \ref{derived calculation II} and Lemma \ref{derived calculationIII}.

For any $\tau \in \Psi$, we write $\mathfrak{g}_{\tau}$ for the $\tau$-factor of $\mathfrak{g}_{\Psi} = \oplus_{\tau \in \Psi} \mathfrak{gl}_2(L)_{\tau}$ and $\mathfrak{h}_{\tau}$ for the maximal torus of $\mathfrak{g}_{\tau}$ consisting of the diagonal matrices. Let $(a_{\tau}, b_{\tau}) \in \mathbb{Z}^2$ and we regard this as a character $(a_{\tau}, b_{\tau}) : \mathrm{Sym}(\mathfrak{h}_{\tau}) \rightarrow C$ and let $\chi_{(a_{\tau}, b_{\tau})} : Z(U(\mathfrak{g}_{\tau})) \xrightarrow{\mathrm{HC}_{\tau}} \mathrm{Sym}(\mathfrak{h}_{\tau}) \xrightarrow{\mathrm{Ad}(w_0)} \mathrm{Sym}(\mathfrak{h}_{\tau}) \xrightarrow{(a_{\tau}, b_{\tau})} L$, where $\mathrm{HC}_{\tau}$ denotes the unnormalized Harish-Chandra morphism. For a $\mathfrak{h}_{\tau}$-module (resp. $Z(\mathfrak{g}_{\tau})_C$-module) $M$, we put $M^{(a_{\tau}, b_{\tau})-\mathrm{nilp}} := \{ m \in M \mid \forall h \in \mathrm{Ker}(a_{\tau}, b_{\tau}), \ \exists k > 0, \mathrm{s.t.} \ h^{k}m = 0 \} = \{ m \in M \mid \exists k > 0 \mathrm{\ s.t. \ } \forall h \in \mathrm{Ker}(a_{\tau}, b_{\tau}), \ h^{k}m = 0 \}$ (resp. $M^{\chi_{(a_{\tau}, b_{\tau})}-\mathrm{nilp}} := \{ m \in M \mid \forall g \in \mathrm{Ker}\chi_{(a_{\tau}, b_{\tau})}, \ \exists k > 0, \mathrm{s.t.} \ g^{k}m = 0 \} = \{ m \in M \mid \exists k > 0, \mathrm{s.t.} \ \forall g \in \mathrm{Ker}\chi_{(a_{\tau}, b_{\tau})}, \  g^{k}m = 0 \}$). Let $R\Gamma((a_{\tau}, b_{\tau})-\mathrm{nilp}, \ \ ) : D^+(\mathrm{Mod}(\mathfrak{h}_{\tau})) \rightarrow D^+(\mathrm{Mod}(C))$ (resp. $R\Gamma(\chi_{(a_{\tau}, b_{\tau})}-\mathrm{nilp}, \ \ ) : D^+(\mathrm{Mod}(Z(\mathfrak{g}_{\tau})_C)) \rightarrow D^+(\mathrm{Mod}(C))$) be the right derived functor. 

\vspace{0.5 \baselineskip}

\begin{lem}\label{elementary}

Let $M$ be a $\mathfrak{h}_{\tau}$-module.

1 \ $R\Gamma((a_{\tau}, b_{\tau})-\mathrm{nilp}, M) = R\Gamma(\chi_{(a_{\tau}, b_{\tau})}-\mathrm{nilp}, M)$ if $(a_{\tau}, b_{\tau}) = (b_{\tau} + 1, a_{\tau}-1)$.

2 \ $R\Gamma((a_{\tau}, b_{\tau})-\mathrm{nilp}, M) \oplus R\Gamma((b_{\tau} + 1, a_{\tau}-1)-\mathrm{nilp}, M) = R\Gamma(\chi_{(a_{\tau}, b_{\tau})}-\mathrm{nilp}, M)$ if $(a_{\tau}, b_{\tau})\neq (b_{\tau} + 1, a_{\tau}-1)$. 

\end{lem}

\begin{proof}

1 follows from the facts that $\mathrm{Sym}\mathfrak{h}/\mathrm{Ker}\chi_{(a_{\tau}, b_{\tau})}\mathrm{Sym}\mathfrak{h}$ is a local ring whose maximal ideal is $\mathrm{Ker}(a_{\tau}, b_{\tau})/\mathrm{Ker}\chi_{(a_{\tau}, b_{\tau})}\mathrm{Sym}\mathfrak{h}$ and $(\mathrm{Ker}(a_{\tau}, b_{\tau})/\mathrm{Ker}\chi_{(a_{\tau}, b_{\tau})}\mathrm{Sym}\mathfrak{h})^2 = 0$.

2 follows from the fact that $\mathrm{Sym}\mathfrak{h}/\mathrm{Ker}\chi_{(a_{\tau}, b_{\tau})}\mathrm{Sym}\mathfrak{h}$ is $C \times C$, whose maximal ideals are $\mathrm{Ker}(a_{\tau}, b_{\tau})/\mathrm{Ker}\chi_{(a_{\tau}, b_{\tau})}\mathrm{Sym}\mathfrak{h}$ and $\mathrm{Ker}(b_{\tau} + 1, a_{\tau} - 1)/\mathrm{Ker}\chi_{(a_{\tau}, b_{\tau})}\mathrm{Sym}\mathfrak{h}$. \end{proof}

\begin{lem} \label{derived calculation1}

(1) \ $H^i((a_{\tau}, b_{\tau})-\mathrm{nilp}, D_{\lambda^{\Psi}, K^p}^{\Psi-\mathrm{la}}) = 0$ for any $i > 0$.

(2) \ $H^i(\mathfrak{h}_{\tau}, D_{\lambda^{\Psi}, K^p}^{\Psi-\mathrm{la}, (a_{\tau}, b_{\tau})-\mathrm{nilp}}(-a_{\tau}, -b_{\tau})) = 0$ for any $i > 0$.

(3) \ $H^i(\chi_{(a_{\tau}, b_{\tau})}-\mathrm{nilp}, D_{\lambda^{\Psi}, K^p}^{\Psi-\mathrm{la}})) = 0$ for any $i > 0$.

(4) \ If $(a_{\tau}, b_{\tau}) \neq (b_{\tau} + 1, a_{\tau} - 1)$, we have $\mathrm{Ext}^i_{Z(U(\mathfrak{g}_{\tau}))}(\chi_{(a_{\tau}, b_{\tau})}, D_{\lambda^{\Psi}, K^p}^{\Psi-\mathrm{la}, (a_{\tau}, b_{\tau})-\mathrm{nilp}}) = 0$ for any $i > 0$.

(5) \ If $(a_{\tau}, b_{\tau}) \neq (b_{\tau} + 1, a_{\tau} - 1)$, we have $\mathrm{Ext}^i_{Z(U(\mathfrak{g}_{\tau}))}(\chi_{(a_{\tau}, b_{\tau})}, D_{\lambda^{\Psi}, K^p}^{\Psi-\mathrm{la}, \chi_{(a_{\tau}, b_{\tau})}-\mathrm{nilp}}) = 0$ for any $i > 0$.

\end{lem}

\begin{proof} We may assume $U \subset U_1$. In the following, we will use the notations after Proposition \ref{sen operator}.

\underline{(1) and (2)} \ By using the isomorphism $D_{\lambda^{\Psi}, K^p}^{\Psi-\mathrm{la}}(-a_{\tau}, -b_{\tau}) \Isom D_{\lambda^{\Psi}, K^p}^{\Psi-\mathrm{la}}$ defined by $f \mapsto (\frac{e_{1, \tau}}{e_{1, \tau, n}})^{a_{\tau}-b_{\tau}} (\frac{f_{\tau}}{f_{\tau, n}})^{-a_{\tau}}f,$ we may assume that $(a_{\tau}, b_{\tau}) = (0, 0)$. We claim $$R\Gamma((0, 0)-\mathrm{nilp}, D_{\lambda^{\Psi}, K^p}^{\Psi-\mathrm{la}}) = \varinjlim_{n} R\Gamma(s_{\tau}^{n}, R\Gamma(h_{\tau}^n, D_{\lambda^{\Psi}, K^p}^{\Psi-\mathrm{la}})) = D_{\lambda^{\Psi}, K^p}^{\Psi-\mathrm{la}, (0, 0)-\mathrm{nilp}}$$ and $$R\Gamma(\mathfrak{h}_{\tau}, D_{\lambda^{\Psi}, K^p}^{\Psi-\mathrm{la}, (0, 0)-\mathrm{nilp}}) = R\Gamma(s_{\tau}, R\Gamma(h_{\tau}, D_{\lambda^{\Psi}, K^p}^{\Psi-\mathrm{la}, (0, 0)-\mathrm{nilp}})) = D_{\lambda^{\Psi}, K^p}^{\Psi-\mathrm{la}, (0, 0)},$$ where $s_{\tau} := \begin{pmatrix} 1 & 0 \\
0 & 0    
\end{pmatrix}$ and $h_{\tau} := \begin{pmatrix} -1 & 0 \\
0 & 1    
\end{pmatrix}$. In fact, by the formulas \begin{align*} \theta_{\mathfrak{h}}(s_{\tau})(\sum_{(i, j, k) \in (\mathbb{Z}^{\Psi}_{\ge 0})^{3}} a_{i, j, k} \prod_{\tau \in \Psi}(\mathrm{log}\frac{e_{1, \tau}}{e_{1, \tau, n}})^{i_{\tau}} (\mathrm{log}\frac{f_{\tau}}{f_{\tau, n}})^{j_{\tau}} (x_{\tau}-x_{\tau, n})^{k_{\tau}}) \\
= \sum_{(i, j, k) \in (\mathbb{Z}^{\Psi}_{\ge 0})^{3}} j_{\tau}a_{i, j, k} \prod_{\tau \in \Psi}(\mathrm{log}\frac{e_{1, \tau}}{e_{1, \tau, n}})^{i_{\tau}} (\mathrm{log}\frac{f_{\tau}}{f_{\tau, n}})^{j_{\tau}-1} (x_{\tau}-x_{\tau, n})^{k_{\tau}} \end{align*}

and 

\begin{align*}\theta_{\mathfrak{h}}(h_{\tau})(\sum_{(i, j, k) \in (\mathbb{Z}^{\Psi}_{\ge 0})^{3}} a_{i, j, k} \prod_{\tau \in \Psi}(\mathrm{log}\frac{e_{1, \tau}}{e_{1, \tau, n}})^{i_{\tau}} (\mathrm{log}\frac{f_{\tau}}{f_{\tau, n}})^{j_{\tau}} (x_{\tau}-x_{\tau, n})^{k_{\tau}}) \\
= \sum_{(i, j, k) \in (\mathbb{Z}^{\Psi}_{\ge 0})^{3}} i_{\tau}a_{i, j, k} \prod_{\tau \in \Psi}(\mathrm{log}\frac{e_{1, \tau}}{e_{1, \tau, n}})^{i_{\tau}-1} (\mathrm{log}\frac{f_{\tau}}{f_{\tau, n}})^{j_{\tau}} (x_{\tau}-x_{\tau, n})^{k_{\tau}},\end{align*} we can see that $h_{\tau}^n$ (resp. $s_{\tau}^n$, $h_{\tau}$ and $s_{\tau}$) is a surjection on $D_{\lambda^{\Psi}, K^p}^{\Psi-\mathrm{la}}$ (resp. $D_{\lambda^{\Psi}, K^p}^{\Psi-\mathrm{la}, h_{\tau}^n}$, $D_{\lambda^{\Psi}, K^p}^{\Psi-\mathrm{la}, (0, 0)-\mathrm{nilp}}$ and $D_{\lambda^{\Psi}, K^p}^{\Psi-\mathrm{la}, (0, 0)-\mathrm{nilp}, h_{\tau}}$)\footnote{Roughly speaking, $D_{\lambda^{\Psi}, K^p}^{\Psi-\mathrm{la}, (0, 0)-\mathrm{nilp}}$ consists of the polynomial part with respect to $\tau$.}. (Here, we used the explicit description Corollary \ref{nontrivial mikami} Thus, $a_{i, j, k} \in D_{\lambda^{\Psi}, \aS_{K^pK_p}}(V_{K_p})$ for some $K_p$ and affinoid open $V_{K_p}$ of $\aS_{K^pK_p}$. The explicit formula of the action $\theta$ is given before Proposition \ref{Sen operator}.)
 
(3) follows from (1) and Lemma \ref{elementary}.

(4) \ By 2 of Lemma \ref{elementary}, $\mathrm{Ext}^i_{Z(U(\mathfrak{g}_{\tau}))}(\chi_{(a_{\tau}, b_{\tau})}, D_{\lambda^{\Psi}, K^p}^{\Psi-\mathrm{la}, (a_{\tau}, b_{\tau})-\mathrm{nilp}})$ is equal to $$H^i(\mathfrak{h}, D_{\lambda^{\Psi}, K^p}^{\Psi-\mathrm{la}, (a_{\tau}, b_{\tau})-\mathrm{nilp}} (-a_{\tau}, -b_{\tau})) \oplus H^i(\mathfrak{h}, D_{\lambda^{\Psi}, K^p}^{\Psi-\mathrm{la}, (a_{\tau}, b_{\tau})-\mathrm{nilp}}(-b_{\tau} - 1, -a_{\tau} + 1)).$$ By (2) of this lemma, we have $H^i(\mathfrak{h}, D_{\lambda^{\Psi}, K^p}^{\Psi-\mathrm{la}, (a_{\tau}, b_{\tau})-\mathrm{nilp}} (-a_{\tau}, -b_{\tau})) = 0$. We claim that $$H^i(\mathfrak{h}, D_{\lambda^{\Psi}, K^p}^{\Psi-\mathrm{la}, (b_{\tau} + 1, a_{\tau}-1)-\mathrm{nilp}}(-a_{\tau}, -b_{\tau}))$$ also vanishes. In fact, the identiny map on this space is nilpotent because $\mathrm{Ker}(a_{\tau}, b_{\tau}) + \mathrm{Ker}(b_{\tau}+1, a_{\tau}-1) = \mathrm{Sym}\mathfrak{h}$, $\mathrm{Ker}(b_{\tau}+1, a_{\tau}-1)$ acts nilpotently on this and $\mathrm{Ker}(a_{\tau}, b_{\tau})$ acts trivially on this. 

(5) \  We have \begin{align*} \mathrm{Ext}^i_{Z(U(\mathfrak{g}_{\tau}))}(\chi_{(a_{\tau}, b_{\tau})}, D_{\lambda^{\Psi}, K^p}^{\Psi-\mathrm{la}, \chi_{(a_{\tau}, b_{\tau})}-\mathrm{nilp}}) = \mathrm{Ext}^i_{Z(U(\mathfrak{g}_{\tau}))}(\chi_{(a_{\tau}, b_{\tau})}, D_{\lambda^{\Psi}, K^p}^{\Psi-\mathrm{la}, (a_{\tau}, b_{\tau})-\mathrm{nilp}} \oplus D_{\lambda^{\Psi}, K^p}^{\Psi-\mathrm{la}, (b_{\tau} + 1, a_{\tau}-1)-\mathrm{nilp}}).\end{align*} By (4) of this lemma, we have \begin{align*}\mathrm{Ext}^i_{Z(U(\mathfrak{g}_{\tau}))}(\chi_{(a_{\tau}, b_{\tau})}, D_{\lambda^{\Psi}, K^p}^{\Psi-\mathrm{la}, (a_{\tau}, b_{\tau})-\mathrm{nilp}}) = 0\end{align*} for $i > 0$. By the same proof as (4), we have \begin{align*}\mathrm{Ext}^i_{Z(U(\mathfrak{g}_{\tau}))}(\chi_{(a_{\tau}, b_{\tau})}, D_{\lambda^{\Psi}, K^p}^{\Psi-\mathrm{la}, (b_{\tau} + 1, a_{\tau}-1)-\mathrm{nilp}}) = 0\end{align*} for any $i > 0$. \end{proof}

\subsubsection{Parallel weight case}

In this subsection, we assume $\lambda_{\tau} = \lambda_{\tau'}$ for any $\tau, \tau' \in \Psi$. Let $k$ be this integer.

In this case, we have $$E_0((V_{\lambda^{\Psi}}(-\lambda_0) \otimes_L \mathrm{Fil}^1\mathbb{B}^{+}_{\mathrm{dR}, k + 2 })^{\Psi-\mathrm{la}, \chi_{\lambda_{\Psi}}, L}) = \oplus_{\tau \in \Psi} D_{{\lambda^{\Psi}}, K^p}^{\Psi-\mathrm{la}, (0, k)_{\sigma \neq \tau}, (1+k, -1)_{\tau}}(k+1)^{G_L}.$$ Thus we have the following exact sequence $0 \rightarrow \oplus_{\tau \in \Psi} D_{{\lambda^{\Psi}}, K^p}^{\Psi-\mathrm{la}, (0, k)_{\sigma \neq \tau}, (1+k, -1)_{\tau}}(k+1)^{G_L} \rightarrow E_0((V_{\lambda^{\Psi}}(-\lambda_0) \otimes_L \mathbb{B}^{+}_{\mathrm{dR}, k + 2 })^{\Psi-\mathrm{la}, \chi_{\lambda_{\Psi}}, L}) \rightarrow D_{{\lambda^{\Psi}}, K^p}^{\Psi-\mathrm{la}, (0, k)_{\tau}, G_L} \rightarrow 0$. Since the image of the Fontaine operator $N_{\lambda}$ on $E_0((V_{\lambda^{\Psi}}(-\lambda_0) \otimes_L \mathbb{B}^{+}_{\mathrm{dR}, k + 2 })^{\Psi-\mathrm{la}, \chi_{\lambda_{\Psi}}, L})$ is contained in $\oplus_{\tau \in \Psi} D_{{\lambda^{\Psi}}, K^p}^{\Psi-\mathrm{la}, (0, k)_{\sigma \neq \tau}, (1+k, -1)_{\tau}}(k+1)^{G_L}$ and the kernel of $N_{\lambda}$ contains $\oplus_{\tau \in \Psi} D_{{\lambda^{\Psi}}, K^p}^{\Psi-\mathrm{la}, (0, k)_{\sigma \neq \tau}, (1+k, -1)_{\tau}}(k+1)^{G_L}$, $N_{\lambda}$ induces $$N_{\lambda}^0 : D_{{\lambda^{\Psi}}, K^p}^{\Psi-\mathrm{la}, (0, k)_{\tau}, G_L} \rightarrow \oplus_{\tau} D_{{\lambda^{\Psi}}, K^p}^{\Psi-\mathrm{la}, (0, k)_{\sigma \neq \tau}, (1+k, -1)_{\tau}}(k+1)^{G_L}$$ and thus $$N_{\lambda, C}^0 := N_{\lambda}^0 \widehat{\otimes}_L C : D_{{\lambda^{\Psi}}, K^p}^{\Psi-\mathrm{la}, (0, k)_{\tau}} \rightarrow \oplus_{\tau} D_{{\lambda^{\Psi}}, K^p}^{\Psi-\mathrm{la}, (0, k)_{\sigma \neq \tau}, (1+k, -1)_{\tau}}(k+1)$$ by 7 of Lemma \ref{HTE}. 

On the cohomology side, we have the Hodge-Tate decomposition $$\widehat{H}^d(S_{K^p}, V_{\lambda^{\Psi}}(-\lambda_0))_{\mathfrak{m}}^{\Psi-\mathrm{la}, \chi_{\lambda_{\Psi}}} \widehat{\otimes}_{L} C \cong \oplus_{I \subset \Psi} H^d(\Fl, D^{\Psi-\mathrm{la}, (0, k)_{\tau \in I}, (1 +k, -1)_{\tau \notin I}}_{\lambda^{\Psi}, K^p})_{\mathfrak{m}}$$ by Proposition \ref{Hodge-Tate decomposition}. %From Lemma \ref{de Rham structure}, the map induced by the Fontaine operator on $E_0((\widehat{H}^d(S_{K^p}, V_{\lambda^{\Psi}}(-\lambda_0))_{\mathfrak{m}}^{\Psi-\mathrm{la}, \chi_{\lambda_{\Psi}}} \widehat{\otimes}_L B_{\mathrm{dR}, \lambda + 2}^+)^L)$ is equal to $$H^d(\Fl, N_{\lambda}^0 \otimes_{L} C) : H^d(\Fl, D_{{\lambda^{\Psi}}, K^p}^{\Psi-\mathrm{la}, (0, \lambda)_{\tau}})_{\mathfrak{m}} \rightarrow \oplus_{\tau} H^d(\Fl, D_{{\lambda^{\Psi}}, K^p}^{\Psi-\mathrm{la}, (0, \lambda)_{\sigma \neq \tau}, (1+\lambda, -1)_{\tau}}(\lambda+1))_{\mathfrak{m}}.$$

Our purpose of this subsection is to prove the following.

\begin{thm}\label{Fontaine}

There exists $(c_{\tau})_{\tau} \in (L^{\times})^{\Psi}$ such that $N_{\lambda}^0 = \oplus_{\tau} c_{\tau}\overline{d}_{\tau}^{\lambda} \circ d_{\tau}^{\lambda}$. (See the comments before Conjecture \ref{key diagram} for the definition of $\overline{d}_{\tau}^{\lambda}$ and $d_{\tau}^{\lambda}$.)

\end{thm}

Before proving the above theorem, we prove Conjecture \ref{key diagram} by using the above theorem.

\begin{cor}\label{parallel comparison}

Let $\mathfrak{m}$ be a decomposed generic non-Eisenstein ideal of $\mathbb{T}^S(K^p, \mathcal{O})$ such that $\overline{\rho}_{\mathfrak{m}}(G_{F})$ is not solvable and $s_{\mathfrak{m}} := \chi_{\mathfrak{m}}^c|_{G_{\tilde{F}}} \otimes (\otimes_{\tau \in \Psi} (\rho_{\mathfrak{m}}|_{G_{\tilde{F}}})^{\tau})$ satisfies the condition that $\overline{s}_{\mathfrak{m}}$ is absolutely irreducible. (We recall that $\tilde{F}$ denotes the Galois closure of $F$ over $\mathbb{Q}$.) Let $\varphi : \mathbb{T}^S(K^p, \mathcal{O})_{\mathfrak{m}} \rightarrow \mathcal{O}$ be an $\mathcal{O}$-morphism such that $\rho_{\varphi}|_{G_{F_w}}$ is de Rham of $p$-adic Hodge type $\lambda_w$ for any $w \mid v$ and $\chi_{\varphi}|_{G_{F_{0, v^c}}}$ is de Rham of $p$-adic Hodge type $\lambda_0$. If $\widehat{H}^d(S_{K^p}, V_{\lambda^{\Psi}})_{\mathfrak{m}}^{\Psi-\mathrm{la}}[\varphi] \neq 0$, then $\mathrm{Ker} H^d(\Fl, N)_{\mathfrak{m}}[\varphi] \neq 0$, where $N = \oplus_{\tau \in \Psi} \overline{d}_{\tau}^{\lambda} \circ d_{\tau}^{\lambda}$.

\end{cor}

\begin{proof}

Let $s_{\mathcal{\varphi}} := \varphi_*s_{\mathfrak{m}}$. By Proposition \ref{residual irreducibility}, the evaluation map $$s_{\varphi}^{\vee}(-d-\lambda_{0}) \otimes_L \mathrm{Hom}_{L[G_{\tilde{F}}]}(s_{\varphi}^{\vee}(-d-\lambda_0), \widehat{H}^d(S_{K^p}, V_{\lambda^{\Psi}}(-\lambda_0))_{\mathfrak{m}}^{\Psi-\mathrm{la}}[\varphi]) \Isom \widehat{H}^d(S_{K^p}, V_{\lambda^{\Psi}}(-\lambda_0))_{\mathfrak{m}}^{\Psi-\mathrm{la}}[\varphi]$$ is an isomorphism. We claim that $\widehat{H}^d(S_{K^p}, V_{\lambda^{\Psi}}(-\lambda_0))_{\mathfrak{m}}^{\Psi-\mathrm{la}}[\varphi] \subset \widehat{H}^d(S_{K^p}, V_{\lambda^{\Psi}}(-\lambda_0))_{\mathfrak{m}}^{\Psi-\mathrm{la}, \chi_{\lambda_{\Psi}}}$. Actually, by Theorem \ref{infinitesimal character}, we have $\widehat{H}^d(S_{K^p}, V_{\lambda^{\Psi}}(-\lambda_0))_{\mathfrak{m}}^{\mathrm{la}}[\varphi] = \widehat{H}^d(S_{K^p}, L)_{\mathfrak{m}}^{\mathrm{la}}[\varphi] \otimes_L V_{\lambda^{\Psi}}(-\lambda_0) \subset \widehat{H}^d(S_{K^p}, L)_{\mathfrak{m}}^{\mathrm{la}, \chi_{\lambda_{\Psi}}} \otimes_L V_{\lambda^{\Psi}}(-\lambda_0) = (\widehat{H}^d(S_{K^p}, L)_{\mathfrak{m}}^{\mathrm{la}} \otimes_L V_{\lambda^{\Psi}}(-\lambda_0))^{\chi_{\lambda_{\Psi}}}$. This implies the claim.

We have the Hodge-Tate decomposition $$\widehat{H}^d(S_{K^p}, V_{\lambda^{\Psi}}(-\lambda_0))_{\mathfrak{m}}^{\Psi-\mathrm{la}, \chi_{\lambda_{\Psi}}} \widehat{\otimes}_{L} C \cong \oplus_{I \subset \Psi} H^d(\Fl, D^{\Psi-\mathrm{la}, (0, k)_{\tau \in I}, ( 1 +k, -1)_{\tau \notin I}}_{K^p, \lambda^{\Psi}})_{\mathfrak{m}}$$ by Proposition \ref{Hodge-Tate decomposition}. By Corollary \ref{sen operatorII}, we have $$(\widehat{H}^d(S_{K^p}, V_{\lambda^{\Psi}}(-\lambda_0))_{\mathfrak{m}}^{\Psi-\mathrm{la}, \chi_{\lambda_{\Psi}}} \widehat{\otimes}_{L} C)^{\theta_{\mathrm{Sen}} = 0} = H^d(\Fl, D^{\Psi-\mathrm{la}, (0, k)_{\tau}, }_{K^p, \lambda^{\Psi}})_{\mathfrak{m}},$$ where $\theta_{\mathrm{Sen}}$ denotes the Sen operator. This space contains $(s_{\varphi}^{\vee}(-d-\lambda_{0}) \otimes_L C)^{\theta_{\mathrm{Sen}} = 0} \widehat{\otimes}_L \mathrm{Hom}_{L[G_{\tilde{F}}]}(s_{\varphi}^{\vee}(-d-\lambda_0), \widehat{H}^d(S_{K^p}, V_{\lambda^{\Psi}}(-\lambda_0))_{\mathfrak{m}}^{\Psi-\mathrm{la}}[\varphi]) = (\widehat{H}^d(S_{K^p}, V_{\lambda^{\Psi}}(-\lambda_0))_{\mathfrak{m}}^{\Psi-\mathrm{la}}[\varphi] \widehat{\otimes} C)^{\theta_{\mathrm{Sen}} = 0}.$ By our assumption, we have $(s_{\varphi}^{\vee}(-d-\lambda_{0}) \otimes_L C)^{\theta_{\mathrm{Sen}} = 0} \neq 0$. Thus $H^d(\Fl, D_{\lambda^{\Psi}, K^p}^{\Psi-\mathrm{la}, (0, k)_{\tau}})[\varphi] \neq 0$. By Corollary \ref{deRhamness}, $H^d(\Fl, N_{\lambda, C}^0)_{\mathfrak{m}}$ is zero on $H^d(\Fl, D_{\lambda^{\Psi}, K^p}^{\Psi-\mathrm{la}, (0, k)_{\tau}})[\varphi]$ since $s_{\varphi}$ is de Rham.  Consequently, we obtain $0 \neq H^d(\Fl, D_{\lambda^{\Psi}, K^p}^{\Psi-\mathrm{la}, (0, k)_{\tau}})[\varphi] \subset \mathrm{Ker}H^d(\Fl, N_{\lambda, C}^0)_{\mathfrak{m}}[\varphi]$. By Theorem \ref{Fontaine}, we have $\mathrm{Ker} H^d(\Fl, N_{\lambda, C}^0)_{\mathfrak{m}} = \mathrm{Ker}H^d(\Fl, N)_{\mathfrak{m}}$. This implies the result. \end{proof}

%we obtain $H^d(\Fl, D_{\lambda^{\Psi}, K^p}^{\Psi-\mathrm{la}, (0, k)_{\tau}})[\varphi] \neq 0$.

%By comparing the Hodge-Tate decomposition $\widehat{H}^d(S_{K^p}, V_{\lambda^{\Psi}}(-\lambda_0))_{\mathfrak{m}}^{\Psi-\mathrm{la}, \chi_{\lambda_{\Psi}}} \widehat{\otimes}_{L} C \cong \oplus_{I \subset \Psi} H^d(\Fl, D^{\Psi-\mathrm{la}, (0, k)_{\tau \in I}, ( 1 +k, -1)_{\tau \notin I}}_{K^p, \lambda^{\Psi}})_{\mathfrak{m}}$ of Proposition \ref{Hodge-Tate decomposition} and the Hodge-Tate decomposition of $$s_{\varphi}^{\vee}(-d-\lambda_{0}) \otimes_L \mathrm{Hom}_{L[G_{\tilde{F}}]}(s_{\varphi}^{\vee}(-d-\lambda_0), \widehat{H}^d(S_{K^p}, V_{\lambda^{\Psi}}(-\lambda_0))_{\mathfrak{m}}^{\Psi-\mathrm{la}}[\varphi]) = \widehat{H}^d(S_{K^p}, V_{\lambda^{\Psi}}(-\lambda_0))_{\mathfrak{m}}^{\Psi-\mathrm{la}}[\varphi],$$ we obtain $H^d(\Fl, D_{\lambda^{\Psi}, K^p}^{\Psi-\mathrm{la}, (0, k)_{\tau}})[\varphi] \neq 0$. 

We will prove Theorem \ref{Fontaine} by using a similar method as \cite[{\S} 6]{PanII}. We may assume $U \subset U_1$. We have the following variant of the key diagram of \cite[{\S} 6.3.11]{PanII}, whose second and third rows and the left column are exact.

\tiny

\xymatrix{
 0 & \\
D_{{\lambda^{\Psi}}, K^p}^{\Psi-\mathrm{la}, (0, k)_{\tau}, G_L} \ar[u] \\
 E_0((V_{\lambda^{\Psi}} \otimes_L \mathbb{B}^+_{\mathrm{dR}, k + 2})^{\Psi-\mathrm{la}, \chi_{\lambda_{\Psi}}, L}) \ar[u] \ar@{^{(}-_>}[r] & \ar[lu] D_{\lambda^{\Psi}, V_0} \otimes_{\mathcal{O}_{V_0}} E_0( \mathcal{O}\mathbb{B}^{+, \Psi-\mathrm{la}, \chi_{\lambda_{\Psi}}, L}_{\mathrm{dR}, k + 2} ) \ar[r]^{\nabla} & D_{\lambda^{\Psi}, V_0} \otimes_{\mathcal{O}_{V_0}} E_0(\mathcal{O}\mathbb{B}^{+, \Psi-\mathrm{la}, \chi_{\lambda_{\Psi}}, L}_{\mathrm{dR}, k + 1}) \otimes_{\mathcal{O}_{V_0}} \Omega_{V_0}^1 \\
 \oplus_{\tau} D_{{\lambda^{\Psi}}, K^p}^{\Psi-\mathrm{la}, (0, k)_{\sigma \neq \tau}, (1+k, -1)_{\tau}}(k+1)^{G_L} \ar@{^{(}-_>}[r] \ar[u] & D_{k^{\Psi}, V_0} \otimes_{\mathcal{O}_{V_0}} E_0(\mathrm{gr}^{k + 1}\mathcal{O}\mathbb{B}^{+, \Psi-\mathrm{la}, \chi_{\lambda_{\Psi}}, L}_{\mathrm{dR}}) \ar[r]^{\nabla} \ar@{^{(}-_>}[u] & D_{\lambda^{\Psi}, V_0} \otimes_{\mathcal{O}_{V_0}} E_0(\mathrm{gr}^{k}\mathcal{O}\mathbb{B}^{+, \Psi-\mathrm{la}, \chi_{\lambda_{\Psi}}, L}_{\mathrm{dR}}) \otimes_{\mathcal{O}_{V_0}} \Omega_{V_0}^1 \ar@{^{(}-_>}[u] \\
 0 \ar[u]
 }

 \normalsize

\begin{lem}\label{application of formal series}

1 \ We have $E_0(t \mathcal{O}\mathbb{B}^{+, \Psi-\mathrm{la}, \chi_{\lambda_{\Psi}}, L}_{\mathrm{dR}, k + 1}) = 0$. 

2 \ $D_{\lambda^{\Psi}, V_0} \otimes_{\mathcal{O}_{V_0}} E_0(\mathcal{O}\mathbb{B}^{+, \Psi-\mathrm{la}, \chi_{\lambda_{\Psi}}, L}_{\mathrm{dR}, k + 1}) \otimes_{\mathcal{O}_{V_0}} \Omega_{V_0}^1 = \oplus_{i=0}^{k} D_{{\lambda^{\Psi}}, K^p}^{\Psi-\mathrm{la}, (0, k)_{\tau}, G_L} \otimes_{\mathcal{O}_{V_0}} \mathrm{Sym}^i\Omega_{V_0}^1 \otimes_{\mathcal{O}_{V_0}} \Omega_{V_0}^1$.

3 \ $D_{\lambda^{\Psi}, V_0} \otimes_{\mathcal{O}_{V_0}} E_0(\mathrm{gr}^{k}\mathcal{O}\mathbb{B}^{+, \Psi-\mathrm{la}, \chi_{\lambda_{\Psi}}, L}_{\mathrm{dR}}) \otimes_{\mathcal{O}_{V_0}} \Omega_{V_0}^1 = D_{{\lambda^{\Psi}}, K^p}^{\Psi-\mathrm{la}, (0, k)_{\tau}, G_L} \otimes_{\mathcal{O}_{V_0}} \mathrm{Sym}^{k}\Omega_{V_0}^1 \otimes_{\mathcal{O}_{V_0}} \Omega_{V_0}^1$.

4 \ The images of $D_{\lambda^{\Psi}, V_0} \otimes_{\mathcal{O}_{V_0}} E_0( \mathcal{O}\mathbb{B}^{+, \Psi-\mathrm{la}, \chi_{\lambda_{\Psi}}, L}_{\mathrm{dR}, k + 2} )$ and $D_{\lambda^{\Psi}, V_0} \otimes_{\mathcal{O}_{V_0}} E_0(\mathrm{gr}^{k + 1}\mathcal{O}\mathbb{B}^{+, \Psi-\mathrm{la}, \chi_{\lambda_{\Psi}}, L}_{\mathrm{dR}})$ via $\nabla$ are $\oplus_{i=1}^{k+1} D_{{\lambda^{\Psi}}, K^p}^{\Psi-\mathrm{la}, (0, k)_{\tau}, G_L} \otimes_{\mathcal{O}_{V_0}} \mathrm{Sym}^i\Omega_{V_0}^1$ and $D_{{\lambda^{\Psi}}, K^p}^{\Psi-\mathrm{la}, (0, k)_{\tau}, G_L} \otimes_{\mathcal{O}_{V_0}} \mathrm{Sym}^{k + 1}\Omega_{V_0}^1$ respectively. Note that these are regarded as subspaces of $\oplus_{i=0}^{k} D_{{\lambda^{\Psi}}, K^p}^{\Psi-\mathrm{la}, (0, k)_{\tau}, G_L} \otimes_{\mathcal{O}_{V_0}} \mathrm{Sym}^i\Omega_{V_0}^1 \otimes_{\mathcal{O}_{V_0}} \Omega_{V_0}^1$ and $D_{{\lambda^{\Psi}}, K^p}^{\Psi-\mathrm{la}, (0, k)_{\tau}, G_L} \otimes_{\mathcal{O}_{V_0}} \mathrm{Sym}^{k}\Omega_{V_0}^1 \otimes_{\mathcal{O}_{V_0}} \Omega_{V_0}^1$ as in Lemma \ref{formal power}.

\end{lem}

\begin{proof} 1 is clear. 2, 3 and 4 follow from Lemma \ref{formal power}.  \end{proof}

By Lemma \ref{application of formal series}, we obtain the following diagram.

\tiny

\xymatrix{
 0 & \\
D_{{\lambda^{\Psi}}, K^p}^{\Psi-\mathrm{la}, (0, k)_{\tau}, G_L} \ar[u] \\
 E_0((V_{\lambda^{\Psi}} \otimes_L \mathbb{B}^+_{\mathrm{dR}, k + 2})^{\Psi-\mathrm{la}, \chi_{\lambda_{\Psi}}, L}) \ar[u] \ar@{^{(}-_>}[r] & \ar[lu] D_{\lambda^{\Psi}, V_0} \otimes_{\mathcal{O}_{V_0}} E_0( \mathcal{O}\mathbb{B}^{+, \Psi-\mathrm{la}, \chi_{\lambda_{\Psi}}, L}_{\mathrm{dR}, k + 2} ) \ar[r]^{\nabla} & \oplus_{i=0}^{k} D_{{\lambda^{\Psi}}, K^p}^{\Psi-\mathrm{la}, (0, k)_{\tau}, G_L} \otimes_{\mathcal{O}_{V_0}} \mathrm{Sym}^i\Omega_{V_0}^1 \otimes_{\mathcal{O}_{V_0}} \Omega_{V_0}^1 \\
 \oplus_{\tau} D_{{\lambda^{\Psi}}, K^p}^{\Psi-\mathrm{la}, (0, k)_{\sigma \neq \tau}, (1+k, -1)_{\tau}}(k+1)^{G_L} \ar@{^{(}-_>}[r] \ar[u] & D_{\lambda^{\Psi}, V_0} \otimes_{\mathcal{O}_{V_0}} E_0(\mathrm{gr}^{k + 1}\mathcal{O}\mathbb{B}^{+, \Psi-\mathrm{la}, \chi_{\lambda_{\Psi}}, L}_{\mathrm{dR}}) \ar[r]^{\nabla} \ar@{^{(}-_>}[u] & D_{{\lambda^{\Psi}}, K^p}^{\Psi-\mathrm{la}, (0, k)_{\tau}, G_L} \otimes_{\mathcal{O}_{V_0}} \mathrm{Sym}^{k}\Omega_{V_0}^1 \otimes_{\mathcal{O}_{V_0}} \Omega_{V_0}^1 \ar@{^{(}-_>}[u] \\
 0 \ar[u]
 }

 \normalsize

Since  $\oplus_{\tau \in \Psi} (D_{{\lambda^{\Psi}}, K^p}^{\Psi-\mathrm{la}, (0, k)_{\sigma}, G_L} \otimes_{\mathcal{O}_{V_0}} \omega^{2k + 2}_{\tau, V_0} \otimes (\wedge^2D_{\tau, V_0})^{k + 1}) \subset \nabla(D_{\lambda^{\Psi}, V_0} \otimes_{\mathcal{O}_{V_0}} E_0(\mathrm{gr}^{k + 1}\mathcal{O}\mathbb{B}^{+, \Psi-\mathrm{la}, \chi_{\lambda_{\Psi}}, L}_{\mathrm{dR}}))$  by Lemma \ref{application of formal series}, the Sen operator $\theta_{\mathrm{Sen}}$ on $D_{\lambda^{\Psi}, V_0} \otimes_{\mathcal{O}_{V_0}} E_0(\mathrm{gr}^{k + 1}\mathcal{O}\mathbb{B}^{+, \Psi-\mathrm{la}, \chi_{\lambda_{\Psi}}, L}_{\mathrm{dR}})$ induces a map $$N'_{\lambda} : \oplus_{\tau \in \Psi} (D_{{\lambda^{\Psi}}, K^p}^{\Psi-\mathrm{la}, (0, k)_{\sigma}, G_L} \otimes_{\mathcal{O}_{V_0}} \omega^{2k + 2}_{\tau, V_0} \otimes (\wedge^2D_{\tau, V_0})^{k + 1}) \rightarrow \oplus_{\tau} D_{{\lambda^{\Psi}}, K^p}^{\Psi-\mathrm{la}, (0, k)_{\sigma \neq \tau}, (1+k, -1)_{\tau}}(k+1)^{G_L}$$ as in the case $N_{\lambda}^0$. Concretely, $N_{\lambda}'(x) = \theta_{\mathrm{Sen}}(y)$ for a lift $y$ of $x$ in $D_{\lambda^{\Psi}, V_0} \otimes_{\mathcal{O}_{V_0}} E_0(\mathrm{gr}^{k + 1}\mathcal{O}\mathbb{B}^{+, \Psi-\mathrm{la}, \chi_{\lambda_{\Psi}}, L}_{\mathrm{dR}})$.

Later, we will construct a map $s_{k} : D_{{\lambda^{\Psi}}, K^p}^{\Psi-\mathrm{la}, (0, k)_{\tau}, G_L} \rightarrow D_{\lambda^{\Psi}, V_0} \otimes_{\mathcal{O}_{V_0}} \mathcal{O}\mathbb{B}^{+, \Psi-\mathrm{la}, \chi_{\lambda_{\Psi}}, G_L}_{\mathrm{dR}, k + 2}$ which is a section of the canonical surjection $D_{\lambda^{\Psi}, V_0} \otimes_{\mathcal{O}_{V_0}} E_0( \mathcal{O}\mathbb{B}^{+, \Psi-\mathrm{la}, \chi_{\lambda_{\Psi}}, L}_{\mathrm{dR}, k + 2} ) \twoheadrightarrow D_{{\lambda^{\Psi}}, K^p}^{\Psi-\mathrm{la}, (0, k)_{\tau}, G_L}$.

\vspace{0.5 \baselineskip}

In order to prove Theorem \ref{Fontaine}, it suffices to prove the following results.

\begin{prop} \label{1st}
    
$\nabla \circ s_{k} = (-1)^k\oplus_{\tau \in \Psi} d_{\tau}^{\lambda}$.
    
\end{prop}

\begin{prop} \label{2nd}
    
There exists $(c_{\tau})_{\tau} \in (L^{\times})^{\Psi}$ such that $N'_{\lambda} = \oplus_{\tau \in \Psi} c_\tau \overline{d}_{\tau}$.
    
\end{prop}

Actually, by using Proposition \ref{1st}, we have $\mathrm{Im}(\nabla \circ s_{k}) \subset \oplus_{\tau \in \Psi} (D_{{\lambda^{\Psi}}, K^p}^{\Psi-\mathrm{la}, (0, k)_{\tau}, G_L} \otimes_{\mathcal{O}_{V_0}} \omega^{2k + 2}_{\tau, V_0} \otimes (\wedge^2D_{\tau, V_0})^{k + 1})$. Thus Propositions \ref{1st} and \ref{2nd} implies Theorem \ref{Fontaine} by the following proposition.

\begin{prop}\label{composition}
    
Assume $\mathrm{Im}(\nabla \circ s_{k}) \subset \oplus_{\tau \in \Psi} (D_{{\lambda^{\Psi}}, K^p}^{\Psi-\mathrm{la}, (0, k)_{\tau}, G_L} \otimes_{\mathcal{O}_{V_0}} \omega^{2k + 2}_{\tau, V_0} \otimes (\wedge^2D_{\tau, V_0})^{k + 1})$. Then we have $N_{\lambda}^0 = - N'_{\lambda} \circ ( \nabla \circ s_{k})$.

\end{prop}
    
\begin{proof} Let $\theta_{\mathrm{Sen}}$ denote the Sen operator. Let $x \in D_{{\lambda^{\Psi}}, K^p}^{\Psi-\mathrm{la}, (0, k)_{\tau}, G_L}$ and we take a lift $y \in D_{\lambda^{\Psi}, V_0} \otimes_{\mathcal{O}_{V_0}} E_0(\mathrm{gr}^{k + 1}\mathcal{O}\mathbb{B}^{+, \Psi-\mathrm{la}, \chi_{\lambda_{\Psi}}, L}_{\mathrm{dR}})$ of $\nabla(s_{k}(x))$. Then $s_{k}(x) - y$ is contained in $E_0((V_{\lambda^{\Psi}} \otimes_L \mathbb{B}^+_{\mathrm{dR}, k + 2})^{\Psi-\mathrm{la}, \chi_{\lambda_{\Psi}}, L})$. Thus $N_{\lambda}^0(x) = \theta_{\mathrm{Sen}}(s_{k}(x)-y)$ by the definition of $N_{\lambda}^0$. Since $s_{k}(x)$ is $G_L$-invariant, we have $\theta_{\mathrm{Sen}}(s_{k}(x)-y) = -\theta_{\mathrm{Sen}}(y)$. On the other hand, we have $\theta_{\mathrm{Sen}}(y) = N_{\lambda}'(\nabla(s_{k}(x)))$ by the definitions of $y$ and $N_{\lambda}'$. \end{proof}

\paragraph{Proof of Proposition \ref{1st}.}

\subparagraph{$k = 0$ case.}

By Proposition \ref{Hodge de Rham}, we have $V_{\tau} \otimes_L \mathcal{O}\mathbb{B}_{\mathrm{dR}}^+[\frac{1}{t}] \cong D_{\tau, V_0} \otimes_{\mathcal{O}_{V_0}} \mathcal{O}\mathbb{B}_{\mathrm{dR}}^+[\frac{1}{t}]$. More precisely, we have $V_{\tau}(1) \otimes_L \mathbb{B}_{\mathrm{dR}}^+ \subset (D_{\tau, V_0} \otimes_{\mathcal{O}_{V_0}} \mathcal{O}\mathbb{B}_{\mathrm{dR}}^+)^{\nabla = 0}(1) \subset V_{\tau} \otimes_L \mathbb{B}_{\mathrm{dR}}^+$, $(D_{\tau, V_0} \otimes_{\mathcal{O}_{V_0}} \mathcal{O}\mathbb{B}_{\mathrm{dR}}^+)^{\nabla = 0}(1) / (V_{\tau}(1) \otimes_L \mathbb{B}_{\mathrm{dR}}^+) = \omega_{\tau, K^p}^{-1}(1)$ and $V_{\tau} \otimes_L \mathbb{B}_{\mathrm{dR}}^+/(D_{\tau, V_0} \otimes_{\mathcal{O}_{V_0}} \mathcal{O}\mathbb{B}_{\mathrm{dR}}^+)^{\nabla = 0}(1) = \omega_{\tau, K^p} \otimes_{\mathcal{O}_{K^p}^{\mathrm{sm}}} \wedge^2 D_{\tau, K_p}^{\mathrm{sm}}$.

Thus we obtain a map $l_0 : D_{\tau, V_0}^{\vee} \otimes_L V_{\tau} \rightarrow \mathcal{O}\mathbb{B}_{\mathrm{dR}}^{+, \Psi-\mathrm{la}, G_L}$ compatible with filtrations and connections. Thus the composite of this map with the projection $\mathcal{O}\mathbb{B}_{\mathrm{dR}}^{+} \rightarrow \mathcal{O}_{K^p}$ factors through $D_{\tau, V_0}^{\vee} \otimes_L V_{\tau} \rightarrow \omega_{\tau, V_0}^{-1} \otimes_{\mathcal{O}_{V_0}} (\wedge^2 D_{\tau, V_0})^{-1} \otimes_L V_{\tau} \rightarrow \omega_{\tau, V_0}^{-1} \otimes_{\mathcal{O}_{V_0}} (\wedge^2 D_{\tau, V_0})^{-1} \otimes_{\mathcal{O}_{V_0}} \omega_{\tau, K^p} \otimes_{\mathcal{O}_{V_0}} (\wedge^2 D_{\tau, V_0}) = \mathcal{O}_{K^p}$, where $V_{\tau} \rightarrow \omega_{\tau, K^p} \otimes_{\mathcal{O}_{V_0}} (\wedge^2 D_{\tau, V_0})$ is the map coming from the Hodge-Tate filtration.

Let $f_{\tau} \in D_{\tau, V_0}^{\vee}$ be a lift of a generator of $\omega_{\tau, V_0}^{-1} \otimes_{\mathcal{O}_{V_0}} (\wedge^2 D_{\tau, V_0})^{-1}$. Then $l_0(f_{\tau} \otimes \begin{pmatrix}
    1 \\ 0 
   \end{pmatrix}_{\tau})$ is a unit element of $\mathcal{O}\mathbb{B}_{\mathrm{dR}}^{+}$ because $\mathcal{O}\mathbb{B}_{\mathrm{dR}}^+$ is complete with respect to $\mathrm{Ker}(\mathcal{O}\mathbb{B}_{\mathrm{dR}}^+ \twoheadrightarrow \mathcal{O}_{K^p})$-adic topology and the image of $l_0(f_{\tau} \otimes \begin{pmatrix}
    1 \\ 0 
   \end{pmatrix}_{\tau})$ via the map $\mathcal{O}\mathbb{B}_{\mathrm{dR}}^+ \twoheadrightarrow \mathcal{O}_{K^p}$ is invertible. Note that we now assume $U \subset U_1$. Let $\tilde{x}_{\tau} := \frac{l_0(f_{\tau} \otimes \begin{pmatrix}
    0 \\ 1 
   \end{pmatrix}_{\tau})}{l_0(f_{\tau} \otimes \begin{pmatrix}
    1 \\ 0 
   \end{pmatrix}_{\tau})} \in \mathcal{O}\mathbb{B}_{\mathrm{dR}}^{+}$, which is a lift of $x_{\tau}$. By using this, we define a section $s_{0} : D_{\lambda^{\Psi}, K^p}^{\Psi-\mathrm{la}, (0, 0), G_L} \rightarrow D_{\lambda^{\Psi}, V_0} \otimes_{\mathcal{O}_{V_0}} \mathcal{O}\mathbb{B}^{+, \Psi-\mathrm{la}, \chi_{0}, G_L}_{\mathrm{dR}, m}$ of the canonical quotient map for any $m > 0$ by the following way. By Corollary \ref{citation}, any element of $f \in \varinjlim_{L' \supset L : \mathrm{finite \ extension}}D_{\lambda^{\Psi}, K^p}^{\Psi-\mathrm{la}, (0, 0), G_{L'}}$ is equal to $\sum_i a_i \prod_{\tau \in \Psi}(x_{\tau}-x_{\tau, n})^{i_{\tau}}$ for some $a_i \in D_{\lambda^{\Psi}, \aS_{K^pG_m}}(V_{G_m})^{G_{L'}}$ and $x_{\tau, m} \in \mathcal{O}_{\aS_{K^pG_m}}(V_{G_m})^{G_{L'}}$ satisfying $||x_{\tau} - x_{\tau, m}|| \le p^{-m}$ and $\mathrm{sup}_i|| a_{i} ||p^{-m(\sum_{\tau}i_{\tau})} < \infty$. Note that $\sum_i a_i \prod_{\tau \in \Psi}(\tilde{x}_{\tau}-x_{\tau, n})^{i_{\tau}}$ converges in $D_{\lambda^{\Psi}, V_0} \otimes_{\mathcal{O}_{V_0}} \mathcal{O}\mathbb{B}_{\mathrm{dR}}^{+}$ because the image of $\sum_i a_i \prod_{\tau \in \Psi}(\tilde{x}_{\tau}-x_{\tau, n})^{i_{\tau}}$ (precisely, we regard this as not an element but a sequence of elements) via the map $D_{\lambda^{\Psi}, V_0} \otimes_{\mathcal{O}_{V_0}} \mathcal{O}\mathbb{B}_{\mathrm{dR}}^+ \twoheadrightarrow D_{\lambda^{\Psi}, K^p}$ converges and $\mathcal{O}\mathbb{B}_{\mathrm{dR}}^+$ is complete with respect to $\mathrm{Ker}(\mathcal{O}\mathbb{B}_{\mathrm{dR}}^+ \twoheadrightarrow \mathcal{O}_{K^p})$-adic topology. Note that $\sum_i a_i \prod_{\tau \in \Psi}(\tilde{x}_{\tau}-x_{\tau, n})^{i_{\tau}}$ is actually an element of the subspace $D_{V_0} \otimes_{\mathcal{O}_{V_0}} \mathcal{O}\mathbb{B}_{\mathrm{dR}}^{+, \Psi-\mathrm{la}, \chi_{0}, G_{L'}}$. In fact, $\sum_i a_i \prod_{\tau \in \Psi}(\tilde{x}_{\tau}-x_{\tau, n})^{i_{\tau}} \in D_{V_0} \otimes_{\mathcal{O}_{V}} \mathcal{O}\mathbb{B}_{\mathrm{dR}}^{+, \Psi-\mathrm{la}, G_{L'}}$ is clear from the definition and $Z(U(\mathfrak{g}_{\Psi}))$ acts on $\sum_i a_i \prod_{\tau \in \Psi}(\tilde{x}_{\tau}-x_{\tau, n})^{i_{\tau}}$ via $\chi_0$ because $\mathfrak{g}_{\Psi}$ acts on $\tilde{x}_{\tau}$ in the same way as $x_{\tau}$. (See \cite[{\S} 6.4.4]{PanII} for precise formulas.) By sending $f$ to $\sum_i a_i \prod_{\tau \in \Psi}(\tilde{x}_{\tau}-x_{\tau, n})^{i_{\tau}}$, we obtain $s_{0}$.
   
   By the definition of $s_0$, in order to prove Proposition \ref{1st}, it suffices to prove $\nabla(\tilde{x}_{\tau}) = \frac{(l_0 \otimes \mathrm{id})(\nabla(f_{\tau}) \otimes \begin{pmatrix}
    0 \\ 1 
   \end{pmatrix}_{\tau})l_0(f_{\tau} \otimes \begin{pmatrix}
    1 \\ 0 
   \end{pmatrix}_{\tau}) - l_0(f_{\tau} \otimes \begin{pmatrix}
    0 \\ 1 
   \end{pmatrix}_{\tau})(l_0 \otimes \mathrm{id})(\nabla(f_{\tau}) \otimes \begin{pmatrix}
    1 \\ 0 
   \end{pmatrix}_{\tau})}{l_0(f_{\tau} \otimes \begin{pmatrix}
    1 \\ 0 
   \end{pmatrix}_{\tau})^2} = 0$ in $\mathcal{O}_{K^p}^{\Psi-\mathrm{la}, (0, 0)_{\tau}, G_L} \otimes_{\mathcal{O}_{V_0}} \Omega_{V_0}^1$ for any $\tau \in \Psi$. This follows from the following.

\begin{prop}\label{p-adic Legendre} For $f_1, f_2 \in D_{\tau, V_0}^{\vee}$ and $v_1, v_2 \in V_{\tau}$, we have $l_0(f_1 \otimes v_1)l_0(f_2 \otimes v_2) - l_0(f_1 \otimes v_2)l_0(f_2 \otimes v_1) \in t\mathcal{O}\mathbb{B}_{\mathrm{dR}}^+$.

  \end{prop}

\begin{proof} We have the pairing $(\wedge^2D_{\tau})^{\vee} \otimes \wedge^2V_{\tau} \rightarrow \mathcal{O}\mathbb{B}_{\mathrm{dR}}^+, \ (f_1 \wedge f_2, v_1 \wedge v_2) \mapsto l_0(f_1 \otimes v_1)l_0(f_2 \otimes v_2) - l_0(f_1 \otimes v_2)l_0(f_2 \otimes v_1)$. Thus it suffices to prove $\wedge^2 V_{\tau} \otimes_L \mathcal{O}\mathbb{B}_{\mathrm{dR}}^+ \subset t\wedge^2D_{\tau, V_0} \otimes_{\mathcal{O}_{V_0}} \mathcal{O}\mathbb{B}_{\mathrm{dR}}^+$. This follows from the fact that $V_{\tau} \otimes \mathbb{B}_{\mathrm{dR}}^+/(D_{\tau, V_0} \otimes_{\mathcal{O}_{V_0}} \mathcal{O}\mathbb{B}_{\mathrm{dR}}^+)^{\nabla = 0}(1) = \omega_{\tau, K^p} \otimes_{\mathcal{O}_{K^p}^{\mathrm{sm}}} \wedge^2 D_{\tau, K_p}^{\mathrm{sm}}$ is a rank 1 direct summand of $D_{\tau, K^p} = (D_{\tau, V_0} \otimes_{\mathcal{O}_{V_0}} \mathcal{O}\mathbb{B}_{\mathrm{dR}}^+)^{\nabla = 0}/(D_{\tau, V_0} \otimes_{\mathcal{O}_{V_0}} \mathcal{O}\mathbb{B}_{\mathrm{dR}}^+)^{\nabla = 0}(1)$. \end{proof}

    %$\subset V_{\tau} \otimes_L \mathbb{B}_{\mathrm{dR}}^+$ $V_{\tau}(1) \otimes_L \mathbb{B}_{\mathrm{dR}}^+ \subset (D_{\tau, V_0} \otimes_{\mathcal{O}_{V_0}} \mathcal{O}\mathbb{B}_{\mathrm{dR}}^+)^{\nabla = 0}(1) \subset V_{\tau} \otimes_L \mathbb{B}_{\mathrm{dR}}^+$, $(D_{\tau, V_0} \otimes_{\mathcal{O}_{V_0}} \mathcal{O}\mathbb{B}_{\mathrm{dR}}^+)^{\nabla = 0}(1) / V_{\tau}(1) \otimes_L \mathbb{B}_{\mathrm{dR}}^+ = \omega_{\tau, K^p}^{-1}(1)$ and $V_{\tau} \otimes_L \mathbb{B}_{\mathrm{dR}}^+/(D_{\tau, V_0} \otimes_{\mathcal{O}_{V_0}} \mathcal{O}\mathbb{B}_{\mathrm{dR}}^+)^{\nabla = 0}(1) = \omega_{\tau, K^p} \otimes_{\mathcal{O}_{K^p}^{\mathrm{sm}}} \wedge^2 D_{\tau, K_p}^{\mathrm{sm}}$. 

\subparagraph{$k > 0$ case.}

We also have $\mathrm{Sym}^{k}V_{\tau} \otimes_L \mathcal{O}\mathbb{B}_{\mathrm{dR}}^+[\frac{1}{t}] \cong \mathrm{Sym}^{k}D_{\tau, V_0} \otimes_{\mathcal{O}_{V_0}} \mathcal{O}\mathbb{B}_{\mathrm{dR}}^+[\frac{1}{t}]$ and a map $l_{k} : \otimes_{\tau \in \Psi} (\mathrm{Sym}^{k}D_{\tau, V_0}^{\vee} \otimes_L \mathrm{Sym}^{k}V_{\tau}) \rightarrow \mathcal{O}\mathbb{B}_{\mathrm{dR}}^{+, \Psi-\mathrm{la}, G_L}$. Thus we have $l_{k}' := s_{0} \otimes l_{k} : D_{{\lambda^{\Psi}}, K^p}^{\Psi-\mathrm{la}, (0, 0)_{\tau}, G_L} \otimes_{\mathcal{O}_{V_0}} (\otimes_{\tau \in \Psi} (\mathrm{Sym}^{k}D_{\tau, V_0}^{\vee} \otimes_L \mathrm{Sym}^{k}V_{\tau})) \rightarrow D_{{\lambda^{\Psi}}, V_0} \otimes_{\mathcal{O}_{V_0}} \mathcal{O}\mathbb{B}_{\mathrm{dR}, k+2}^{+, \Psi-\mathrm{la}, G_L}$. By using the Hodge-Tate filtration and Hodge-de Rham filtration, we have quotient maps $D_{{\lambda^{\Psi}}, K^p}^{\Psi-\mathrm{la}, (0, 0)_{\tau}, G_L} \otimes_{\mathcal{O}_{V_0}} (\otimes_{\tau \in \Psi} (\mathrm{Sym}^{k}D_{\tau, V_0}^{\vee} \otimes_L \mathrm{Sym}^{k}V_{\tau})) \twoheadrightarrow D_{{\lambda^{\Psi}}, K^p}^{\Psi-\mathrm{la}, (0, 0)_{\tau}, G_L} \otimes_{\mathcal{O}_{V_0}} (\otimes_{\tau \in \Psi} (\omega_{\tau, V_0}^{-k} \otimes_{\mathcal{O}_V} (\wedge^2 D_{\tau, V_0})^{-k} \otimes_L \mathrm{Sym}^{k}V_{\tau})) \twoheadrightarrow D_{{\lambda^{\Psi}}, K^p}^{\Psi-\mathrm{la}, (0, k)_{\tau}, G_L}$. As we have seen in the construction of $d_{\tau}^{\lambda}$ and $\overline{d}_{\tau}^{\lambda}$, we have natural sections $D_{{\lambda^{\Psi}}, K^p}^{\Psi-\mathrm{la}, (0, k)_{\tau}, G_L} \xhookrightarrow{\psi_1} D_{{\lambda^{\Psi}}, K^p}^{\Psi-\mathrm{la}, (0, 0)_{\tau}, G_L} \otimes_{\mathcal{O}_{V_0}} (\otimes_{\tau \in \Psi} (\omega_{\tau, V_0}^{-k} \otimes_{\mathcal{O}_V} (\wedge^2 D_{\tau, V_0})^{-k} \otimes_L \mathrm{Sym}^{k}V_{\tau})) \xhookrightarrow{\psi_2} D_{{\lambda^{\Psi}}, K^p}^{\Psi-\mathrm{la}, (0, 0)_{\tau}, G_L} \otimes_{\mathcal{O}_{V_0}} (\otimes_{\tau \in \Psi} (\mathrm{Sym}^{k}D_{\tau, V_0}^{\vee} \otimes_L \mathrm{Sym}^{k}V_{\tau}))$ of the above map fitting in the following commutative diagram. 

($\psi_1$ is constructed in Lemma 4.5 and $\psi_2$ is (the section constructed after Lemma \ref{constdeRham2}) $\otimes_L (\otimes_{\sigma \in \Psi} \mathrm{Sym}^{k}V_{\sigma})$. In the following commutative diagram, the first right vertical map is the map induced by $\psi_1$ and the Kodaira Spencer map, the middle horizontal map is (the map constructed in after Lemma \ref{constdeRham2}) $\otimes_L (\otimes_{\sigma \in \Psi} \mathrm{Sym}^{k}V_{\sigma})$ and the second right vertical arrow is the natural inclusion. The commutativity of the following diagram follows from the constructions of $d_{\tau}^{\lambda}$ and $\psi_1$.) 

 \tiny

\xymatrix{
   D_{\lambda^{\Psi}, K^p}^{\Psi-\mathrm{la}, (0, k)_{\tau}, G_L} \ar[r]^{\oplus_{\tau} d_{\tau}^{\lambda}} \ar[d]^{\psi_1} & \oplus_{\tau \in \Psi} (D_{{\lambda^{\Psi}}, K^p}^{\Psi-\mathrm{la}, (0, k)_{\tau}, G_L} \otimes_{\mathcal{O}_{V_0}} \omega^{2k + 2}_{\tau, V_0} \otimes (\wedge^2D_{\tau, V_0})^{k + 1}) \ar[d] .\\
    D_{{\lambda^{\Psi}}, K^p}^{\Psi-\mathrm{la}, (0, 0)_{\tau}, G_L} \otimes_{\mathcal{O}_{V_0}} (\otimes_{\tau \in \Psi} (\omega_{\tau, V_0}^{-k} \otimes_{\mathcal{O}_V} (\wedge^2 D_{\tau, V_0})^{-k} \otimes_L \mathrm{Sym}^{k}V_{\tau})) \ar[r] \ar[d]^{\psi_2} & (\oplus_{\tau \in \Psi} (D_{{\lambda^{\Psi}}, K^p}^{\Psi-\mathrm{la}, (0, 0)_{\tau}, G_L} \otimes_{\mathcal{O}_{V_0}} \omega^{k + 2}_{\tau, V_0} \otimes (\wedge^2D_{\tau, V_0}))) \otimes_L (\otimes_{\sigma \in \Psi} \mathrm{Sym}^{k}V_{\sigma}) \ar[d] .\\
   D_{{\lambda^{\Psi}}, K^p}^{\Psi-\mathrm{la}, (0, 0)_{\tau}, G_L} \otimes_{\mathcal{O}_{V_0}} (\otimes_{\tau \in \Psi} (\mathrm{Sym}^{k}D_{\tau, V_0}^{\vee} \otimes_L \mathrm{Sym}^{k}V_{\tau})) \ar[r]^{\nabla} & D_{{\lambda^{\Psi}}, K^p}^{\Psi-\mathrm{la}, (0, 0)_{\tau}, G_L} \otimes_{\mathcal{O}_{V_0}} (\otimes_{\tau \in \Psi} (\mathrm{Sym}^{k}D_{\tau, V_0}^{\vee} \otimes_L \mathrm{Sym}^{k}V_{\tau})) \otimes_{\mathcal{O}_{V_0}} \Omega_{V_0}^1 .
    }

     \normalsize

We put $s_{k} := l_{k}' \circ \psi_2 \circ \psi_1$. By the property $\nabla(\tilde{x}_{\tau}) \in t\mathcal{O}\mathbb{B}_{\mathrm{dR}}^+$ for any $\tau \in \Psi$ (see Proposition \ref{p-adic Legendre}), we have the following commutative diagram.

 \scriptsize

\xymatrix{
   D_{\lambda^{\Psi}, K^p}^{\Psi-\mathrm{la}, (0, k)_{\tau}, G_L} \ar[r]^{\oplus_{\tau} d_{\tau}^{\lambda}} \ar[d]^{\psi_2 \circ \psi_1} & \oplus_{\tau \in \Psi} (D_{{\lambda^{\Psi}}, K^p}^{\Psi-\mathrm{la}, (0, k)_{\tau}, G_L} \otimes_{\mathcal{O}_{V_0}} \omega^{2k + 2}_{\tau, V_0} \otimes (\wedge^2D_{\tau, V_0})^{k + 1}) \ar[d] .\\
   D_{{\lambda^{\Psi}}, K^p}^{\Psi-\mathrm{la}, (0, 0)_{\tau}, G_L} \otimes_{\mathcal{O}_{V_0}} (\otimes_{\tau \in \Psi} (\mathrm{Sym}^{k}D_{\tau, V_0}^{\vee} \otimes_L \mathrm{Sym}^{k}V_{\tau})) \ar[r]^{\nabla} \ar[d]^{l_{k}'} & D_{{\lambda^{\Psi}}, K^p}^{\Psi-\mathrm{la}, (0, 0)_{\tau}, G_L} \otimes_{\mathcal{O}_{V_0}} (\otimes_{\tau \in \Psi} (\mathrm{Sym}^{k}D_{\tau, V_0}^{\vee} \otimes_L \mathrm{Sym}^{k}V_{\tau})) \otimes_{\mathcal{O}_{V_0}} \Omega_{V_0}^1 \ar[d]^{l_{k}' \otimes \mathrm{id}_{\Omega^1_{V_0}}} \\
   D_{\lambda^{\Psi}, V_0} \otimes_{\mathcal{O}_{V_0}} \mathcal{O}\mathbb{B}_{\mathrm{dR}, k+2}^{+, \Psi-\mathrm{la}}/t \ar[r]^{\nabla}  & D_{\lambda^{\Psi}, V_0} \otimes_{\mathcal{O}_{V_0}} \mathcal{O}\mathbb{B}_{\mathrm{dR}, k+1}^{+, \Psi-\mathrm{la}}/t \otimes_{\mathcal{O}_{K^p}} \Omega_{V_0}^{1}
    }

     \normalsize

We claim that the composition $m$ of the right vertical arrow is equal to the composition of the inclusion given by Lemma \ref{formal power} and $(-1)^k(KS)^{k+1}$, which implies the result. ($KS$ denotes the Kodaira-Spencer map.) Our claim follows from the following lemma.

%Note that $l_k' : D_{{\lambda^{\Psi}}, K^p}^{\Psi-\mathrm{la}, (0, 0)_{\tau}, G_L} \otimes_{\mathcal{O}_{V_0}} (\otimes_{\tau \in \Psi} (\mathrm{Sym}^{k}D_{\tau, V_0}^{\vee} \otimes_L \mathrm{Sym}^{k}V_{\tau})) \rightarrow D_{\lambda^{\Psi}, V_0} \otimes_{\mathcal{O}_{V_0}} \mathcal{O}\mathbb{B}_{\mathrm{dR}, k+2}^{+, \Psi-\mathrm{la}}/t$ is a $\mathcal{O}_{K^p}^{\Psi-\mathrm{la}, (0, 0)_{\tau}, G_L}$-linear extension of  

\begin{lem}\label{equality}

1 \ Let $m'$ be a map $(D_{\lambda^{\Psi}, V_0} \otimes_{\mathcal{O}_{V_0}} \omega_{\tau, V_0}^k \otimes_L \mathrm{Sym}^kV_{\tau}) \otimes_{\mathcal{O}_{V_0}} \mathcal{O}_{K^p} \hookrightarrow (D_{\lambda^{\Psi}, V_0} \otimes_{\mathcal{O}_{V_0}} \mathrm{Sym}^k D_{\tau, V_0}^{\vee} \otimes_L \mathrm{Sym}^kV_{\tau}) \otimes_{\mathcal{O}_{V_0}} \mathcal{O}_{K^p} \rightarrow D_{\lambda^{\Psi}, V_0} \otimes \mathcal{O}\mathbb{B}_{\mathrm{dR}, k+1}^+/t = D_{\lambda^{\Psi}, V_0} \otimes_{\mathcal{O}_{V_0}} (\oplus_{i=0}^{k} (\mathcal{O}_{K^p} \otimes_{\mathcal{O}_{V_0}} \mathrm{Sym}^i\Omega_{V_0}^1))$. Here, on the right hand side, we consider the $\mathcal{O}_{K^p}$-module structure induced by the identification $\mathcal{O}_{K^p} = \mathbb{B}_{\mathrm{dR}}^+/t$.

Then $m$ is the restriction of the composition $(m' \otimes \mathrm{id}_{\Omega_{V_0}^1}) \circ (\mathrm{id} \otimes KS)$.

2 \ $m'$ factors through $(D_{\lambda^{\Psi}, V_0} \otimes_{\mathcal{O}_{V_0}} \omega_{\tau, V_0}^k \otimes_L \mathrm{Sym}^kV_{\tau}) \otimes_{\mathcal{O}_{V_0}} \mathcal{O}_{K^p} \twoheadrightarrow D_{\lambda^{\Psi}, V_0} \otimes_{\mathcal{O}_{V_0}} \omega_{\tau, K^p}^{2k} \otimes_{\mathcal{O}_{K^p}} (\wedge^2D_{\tau, K^p})^{k}$ and the map $D_{\lambda^{\Psi}, V_0} \otimes_{\mathcal{O}_{V_0}} \omega_{\tau, K^p}^{2k} \otimes_{\mathcal{O}_{K^p}} (\wedge^2D_{\tau, K^p})^{k} \rightarrow D_{\lambda^{\Psi}, V_0} \otimes_{\mathcal{O}_{V_0}} (\oplus_{i=0}^{k} (\mathcal{O}_{K^p} \otimes_{\mathcal{O}_{V_0}} \mathrm{Sym}^i\Omega_{V_0}^1))$ is equal to $D_{\lambda^{\Psi}, V_0} \otimes_{\mathcal{O}_{V_0}} \omega_{\tau, K^p}^{2k} \otimes_{\mathcal{O}_{K^p}} (\wedge^2D_{\tau, K^p})^{k} \xrightarrow{(-KS)^k} D_{\lambda^{\Psi}, K^p} \otimes_{\mathcal{O}_{V_0}} \mathrm{Sym}^k\Omega_{V_0}^1 \hookrightarrow D_{\lambda^{\Psi}, V_0} \otimes_{\mathcal{O}_{V_0}} (\oplus_{i=0}^{k} (\mathcal{O}_{K^p} \otimes_{\mathcal{O}_{V_0}} \mathrm{Sym}^i\Omega_{V_0}^1))$.

\end{lem}

\begin{proof} We may assume $\lambda^{\Psi} = 0$.

1 \ Note that the image of the map $\omega_{\tau, V_0}^k \otimes_L \mathrm{Sym}^kV_{\tau} \hookrightarrow \mathrm{Sym}^k D_{\tau, V_0}^{\vee} \otimes_L \mathrm{Sym}^kV_{\tau} \rightarrow \mathcal{O}\mathbb{B}_{\mathrm{dR}, k+1}^+/t = \oplus_{i=0}^{k} (\mathcal{O}_{K^p} \otimes_{\mathcal{O}_{V_0}} \mathrm{Sym}^i\Omega_{V_0}^1)$ is contained in $\mathcal{O}_{K^p} \otimes_{\mathcal{O}_{V_0}} \mathrm{Sym}^k\Omega_{V_0}^1$ because the map $\mathrm{Sym}^k D_{\tau, V_0}^{\vee} \otimes_L \mathrm{Sym}^kV_{\tau} \rightarrow \mathcal{O}\mathbb{B}_{\mathrm{dR}, k+1}^+/t = \oplus_{i=0}^{k} (\mathcal{O}_{K^p} \otimes_{\mathcal{O}_{V_0}} \mathrm{Sym}^i\Omega_{V_0}^1)$ preserves the filterations on both sides. Note that the $\mathcal{O}_{K^p}^{\Psi-\mathrm{la}, (0, 0)_{\tau}, G_L}$-module structure on $\mathcal{O}_{K^p} \otimes_{\mathcal{O}_{V_0}} \mathrm{Sym}^k\Omega_{V_0}^1$ induced by $s_0$ is equal to that induced by the identification $\mathcal{O}_{K^p} = \mathbb{B}_{\mathrm{dR}}/t$ because $s_0$ is a section of the natural map $\mathcal{O}\mathbb{B}_{\mathrm{dR}, k+1}^{+, \Psi-\mathrm{la}, \chi_{\lambda_{\Psi}} G_L} \rightarrow \mathcal{O}_{K^p}^{\Psi-\mathrm{la}, (0, 0)_{\tau}, G_L}$. This implies 1.

2 As in the proof of 1 of this lemma, $m'$ is equal to $(\omega_{\tau, V_0}^k \otimes_L \mathrm{Sym}^kV_{\tau}) \otimes_{\mathcal{O}_{V_0}} \mathcal{O}_{K^p} \rightarrow \mathcal{O}_{K^p} \otimes_{\mathcal{O}_{V_0}} \mathrm{Sym}^k\Omega_{V_0}^1 \hookrightarrow \oplus_{i=0}^{k} (\mathcal{O}_{K^p} \otimes_{\mathcal{O}_{V_0}} \mathrm{Sym}^i\Omega_{V_0}^1)$. Thus the result follows from Proposition \ref{FHT}. \end{proof}

%By the characterization of decomposition $\mathcal{O}\mathbb{B}_{\mathrm{dR}, k+1}^+/t = \oplus_{i=0}^{k} \mathcal{O}_{K^p} \otimes_{\mathcal{O}_{V_0}} \mathrm{Sym}^i \Omega_{V_0}^1$, the result follows from $\nabla(x_{\tau}) = 0$ in $\mathcal{O}\mathbb{B}_{\mathrm{dR}, k+1}^+/t$ for any $\tau \in \Psi$.

%It suffices to prove that the composition $s_0 : \mathcal{O}_{K^p}^{\Psi-\mathrm{la}, (0, 0)_{\tau}, G_L} \rightarrow \mathcal{O}\mathbb{B}_{\mathrm{dR}, k+1}^{+} \rightarrow \mathcal{O}\mathbb{B}_{\mathrm{dR}, k+1}^+/t = \oplus_{i=0}^k \mathcal{O}_{K^p} \otimes_{\mathcal{O}_{V_0}} \mathrm{Sym}^i\Omega_{V_0}^1$ is equal to the natural inclusion $\mathcal{O}_{K^p}^{\Psi-\mathrm{la}, (0, 0)_{\tau}, G_L} \hookrightarrow \mathcal{O}_{K^p} \hookrightarrow \oplus_{i=0}^k \mathcal{O}_{K^p} \otimes_{\mathcal{O}_{V_0}} \mathrm{Sym}^i\Omega_{V_0}^1$. Note that the composition of this map and the quotient map $\oplus_{i=0}^k \mathcal{O}_{K^p} \otimes_{\mathcal{O}_{V_0}} \mathrm{Sym}^i\Omega_{V_0}^1 \rightarrow \mathcal{O}_{K^p}$ is the natural inclusion since $s_0$ is a section of the natural map $\mathcal{O}\mathbb{B}_{\mathrm{dR}, k+1}^{+, \Psi-\mathrm{la}, \chi_{\lambda_{\Psi}} G_L} \rightarrow \mathcal{O}_{K^p}^{\Psi-\mathrm{la}, (0, 0)_{\tau}, G_L}$. 

\paragraph{Proof of Proposition \ref{2nd}.}

We fix $\tau \in \Psi$. By Proposition \ref{FHT}, we have the following commutative diagram with exact rows.

\xymatrix{
    0 \ar[r] & \mathcal{O}_{K^p}(1) \ar[r] \ar[d]^{\mathrm{id}} & V_{\tau} \otimes_{L} \omega_{\tau, K^p}  \ar[r] \ar[d] & \omega_{\tau, K^p}^{2} \otimes_{\mathcal{O}_{V_0}} \wedge^2 D_{\tau, V_0} \ar[r] \ar[d]^{-\mathrm{KS}} & 0\\
    0 \ar[r] & \mathcal{O}_{K^p}(1) \ar[r] & \mathrm{gr}^1\mathcal{O}\mathbb{B}^+_{\mathrm{dR}} \ar[r]  & \mathcal{O}_{K^p} \otimes_{\mathcal{O}_{V_0}} \Omega_{V_0}^{1} \ar[r]& 0
    }

By taking $\mathrm{Sym}^{k+1}$, we obtain a map $\mathrm{Sym}^{k+1}V_{\tau} \otimes_{L} \omega_{\tau, {K^p}}^{k + 1} \rightarrow \mathrm{gr}^{k + 1}\mathcal{O}\mathbb{B}^+_{\mathrm{dR}}.$ Let $\chi_{\lambda, \tau} : Z(U(\mathfrak{gl}_2(L)_{\tau})) \rightarrow L$ be the infinitesimal character of $\mathrm{Sym}^{k}V_{\tau}$. We also consider the character $(0,-1) : \mathfrak{h}_{\tau} \rightarrow C$. Note that the horizontal action induces an action of $\prod_{\sigma \neq \tau}\mathfrak{h}_{\sigma}$ on $\mathrm{Sym}^{k+1}V_{\tau} \otimes_{L} \omega_{\tau, {K^p}}^{k + 1}$. Thus we can consider the following map $$E_{\lambda} := E_0((\mathrm{Sym}^{k+1}V_{\tau} \otimes_{L} \omega_{\tau, {K^p}}^{k + 1, \Pla, (0, -1)-\mathrm{nilp}})^{\chi_{\lambda, \tau}, (0, k)_{\sigma \neq \tau}, L}) \rightarrow E_0(\mathrm{gr}^{k + 1}\mathcal{O}\mathbb{B}^{+, \Pla, \chi_{\lambda_{\Psi}}, L}_{\mathrm{dR}})=:E_{\lambda}'.$$  Note that the above commutative diagram induces the following commutative diagram with exact rows. $$ \scriptsize \xymatrix{
    0 \ar[r] & D_{{\lambda^{\Psi}}, K^p}^{\Psi-\mathrm{la}, (0, k)_{\sigma \neq \tau}, (1+k, -1)_{\tau}}(k+1)^{G_L} \ar[r] \ar[d] & D_{\lambda^{\Psi}, V_0} \otimes_{\mathcal{O}_{V_0}} E_{\lambda} \ar[r] \ar[d] & \ar[d] D_{{\lambda^{\Psi}}, K^p}^{\Psi-\mathrm{la}, (0, k)_{\gamma}, G_L} \otimes_{\mathcal{O}_{V_0}} \omega_{\tau, V_0}^{2k + 2} \otimes_{\mathcal{O}_{V_0}} (\wedge^2D_{\tau, V_0})^{k + 1} \ar[r] & 0 \\
    0 \ar[r] & \oplus_{\sigma} D_{{\lambda^{\Psi}}, K^p}^{\Psi-\mathrm{la}, (0, k)_{\gamma \neq \sigma}, (1+k, -1)_{\sigma}}(k+1)^{G_L} \ar[r] & D_{\lambda^{\Psi}, V_0} \otimes_{\mathcal{O}_{V_0}} E_{\lambda}' \ar[r] & D_{{\lambda^{\Psi}}, K^p}^{\Psi-\mathrm{la}, (0, k)_{\gamma}, G_L} \otimes_{\mathcal{O}_{V_0}} \mathrm{Sym}^{k}\Omega_{V_0}^1 \otimes_{\mathcal{O}_{V_0}} \Omega_{V_0}^1 } \normalsize $$ The only non-trivial part is the exactness of the upper sequence. This follows from the fact that $\mathrm{Sym}^{k + 1}V_{\tau} \otimes_L \omega_{\tau, K^p}^{k + 1, \Psi-\mathrm{la}, (0, -1)-\mathrm{nilp}}$ has an increasing filtration whose graded pieces are $(\omega_{\tau, \Fl}^{k + 1 -2i} \otimes_L (\wedge^2 V_{\tau})^{k + 1 - i}) \otimes_{\mathcal{O}_{\Fl}}  \omega_{\tau, K^p}^{k + 1, \Psi-\mathrm{la}, (0, -1)-\mathrm{nilp}} = \omega_{\tau, K^p}^{2k + 2 - 2i, \Psi-\mathrm{la}, (i, k - i)-\mathrm{nilp}}(i) \otimes (\wedge^2 D_{\tau, K^p}^{\mathrm{sm}})^{k + 1 - i}$ for $i = 0, \cdots, k + 1$.

Thus to prove Proposition \ref{2nd}, it suffices to study the upper exact sequence instead of the lower exact sequence. Let $$N_{\lambda, \tau}' : D_{{\lambda^{\Psi}}, K^p}^{\Psi-\mathrm{la}, (0, k)_{\sigma}, G_L} \otimes_{\mathcal{O}_{V_0}} \omega_{\tau, V_0}^{2k + 2} \otimes_{\mathcal{O}_{V_0}} (\wedge^2D_{\tau, V_0})^{k + 1} \rightarrow D_{{\lambda^{\Psi}}, K^p}^{\Psi-\mathrm{la}, (0, k)_{\sigma \neq \tau}, (1+k, -1)_{\tau}}(k+1)^{G_L}$$ be the map induced by the Sen operator on $D_{\lambda^{\Psi}, V_0} \otimes_{\mathcal{O}_{V_0}} E_{\lambda}$ and the upper exact sequence. Let $z_{\tau} := \begin{pmatrix}
    1 & 0 \\
     0 & 1 \end{pmatrix} \in Z(\mathfrak{g}_{\tau})$ and $\Omega_{\tau} = \begin{pmatrix}
        0 & 1 \\
         0 & 0 \end{pmatrix}\begin{pmatrix}
            0 & 0 \\
             1 & 0 \end{pmatrix} + \begin{pmatrix}
                0 & 0 \\
                 1 & 0 \end{pmatrix}\begin{pmatrix}
                    0 & 1 \\
                     0 & 0 \end{pmatrix} + \frac{1}{2}\begin{pmatrix}
                        1 & 0 \\
                         0 & -1 \end{pmatrix}^2 \in Z(U(\mathfrak{g}_{\tau}))$ be the Casimir operator. We put $h_{\tau} := \begin{pmatrix}
                            1 & 0 \\
                             0 & -1 \end{pmatrix}$. Then $\mathrm{HC}(\Omega_{\tau}) = \frac{1}{2}h_{\tau}^2 + h_{\tau}$. We have $Z(U(\mathfrak{g}_{\tau})) = L[z_{\tau}, \Omega_{\tau}]$ and this is isomorphic to a two-variable polynomial ring over $L$. (See \cite[example after Theorem 5.44]{Knapp} for details.)

\begin{lem} \label{relation}

$(2\theta_{\mathrm{Sen}} + z_{\tau} + 1)^2 - 1 = 2 \Omega_{\tau}$ on $D_{\lambda^{\Psi}, K^p}^{\Psi-\mathrm{la}, (0, k)_{\sigma \neq \tau}}$.

\end{lem}

\begin{proof}

    By Proposition \ref{Sen operator}, we have $\theta_{\mathrm{Sen}} = \theta_{\mathfrak{h}}(\begin{pmatrix}
        -1 & 0 \\
         0 & 0 \end{pmatrix}_{\sigma \in \Psi})$ on $D_{\lambda^{\Psi}, K^p}^{\Psi-\mathrm{la}}$. By Lemma \ref{infinitesimal character and horizontal}, the result follows from the relation $(2\begin{pmatrix}
            -1 & 0 \\
             0 & 0 \end{pmatrix}_{\tau} + z_{\tau} + 1)^2 - 1 = h^2_{\tau} - 2h_{\tau} = \mathrm{Ad}(w_0)(h^2_{\tau} + 2h_{\tau})$ in $\mathrm{Sym}\mathfrak{h}_{\tau}$. \end{proof}

Note that the Hodge-Tate filtration $0 \rightarrow \omega_{\tau, K^p}^{-1, \Psi-\mathrm{la}}(1) \rightarrow V_{\tau} \otimes_{L} \mathcal{O}_{K^p}^{\Psi-\mathrm{la}} \rightarrow \omega_{\tau, K^p}^{\Psi-\mathrm{la}} \otimes_{\mathcal{O}_{V_0}} \wedge^2 D_{\tau, V_0} \rightarrow 0$ has a $G_L \times \prod_{\sigma \neq \tau}\mathfrak{h}_{\sigma} \times Z(\mathfrak{g}_{\tau})$-equivariant $\mathcal{O}_{K^p}^{\Psi-\mathrm{la}}$-linear splitting given by sending $e_{1, \tau}$ to $\begin{pmatrix}
    1 \\
     0 \end{pmatrix}_{\tau}$. Let $\psi_{i, \tau} : Z(\mathfrak{g}_{\tau}) \rightarrow L, \ \begin{pmatrix} 1 & 0 \\
     0 & 1 \end{pmatrix} \mapsto i$. 
     
     Then $(\mathrm{Sym}^{k+1}V_{\tau} \otimes_{L} \omega_{\tau, {K^p}}^{k + 1, \Pla, (0, k)_{\sigma \neq \tau}, \psi_{-1, \tau}})^{G_L}$ has an increasing filtration $\{ \mathrm{Fil}_i \}_{i=0}^{k+1}$ such that $\mathrm{gr}_i = \omega_{\tau, K^p}^{2i, \Psi-\mathrm{la}, (0, k)_{\sigma \neq \tau}, \psi_{k, \tau}} \otimes (\wedge^2 D_{\tau, K^p}^{\mathrm{sm}})^i(k + 1 - i)^{G_L} = \omega_{\tau, K^p}^{2i, \Psi-\mathrm{la}, (0, k)_{\sigma \neq \tau}, (k + 1 - i, i - 1)} \otimes (\wedge^2 D_{\tau, K^p}^{\mathrm{sm}})^i(k + 1 - i)^{G_L}$. Here, we used the relation $\theta_{\mathrm{Sen}} = \theta_{\mathfrak{h}}(\begin{pmatrix}
        -1 & 0 \\
         0 & 0 \end{pmatrix}_{\sigma \in \Psi})$ to deduce the above equality. See Proposition \ref{Sen operator}. Let $E_0'$ be the $\chi_{\lambda_{\Psi}}$-generalized eigenspace of $(\mathrm{Sym}^{k+1}V_{\tau} \otimes_{L} \omega_{\tau, {K^p}}^{k + 1, \Pla, (0, k)_{\sigma \neq \tau}, \psi_{-1, \tau}})^{G_L}$. Then we have the following exact sequence.
     
\begin{equation}\label{overlined}
0 \rightarrow D_{{\lambda^{\Psi}}, K^p}^{\Psi-\mathrm{la}, (0, k)_{\sigma \neq \tau}, (1+k, -1)_{\tau}}(k+1)^{G_L} \rightarrow D_{\lambda^{\Psi}, V_0} \otimes_{\mathcal{O}_{V_0}} E_0' \rightarrow D_{{\lambda^{\Psi}}, K^p}^{\Psi-\mathrm{la}, (0, k)_{\sigma}, G_L} \otimes_{\mathcal{O}_{V_0}} \omega_{\tau, V_0}^{2k + 2} \otimes_{\mathcal{O}_{V_0}} (\wedge^2D_{\tau, V_0})^{k + 1} \rightarrow 0. \end{equation}

 Thus the action of $(\Omega_{\tau} - \chi_{\lambda_{\Psi}}(\Omega_{\tau}))$ on $D_{\lambda^{\Psi}, V_0} \otimes_{\mathcal{O}_{V_0}} E_0'$ induces a map $$N_{\lambda, \tau}'' : D_{{\lambda^{\Psi}}, K^p}^{\Psi-\mathrm{la}, (0, k)_{\sigma}, G_L} \otimes_{\mathcal{O}_{V_0}} \omega_{\tau, V_0}^{2k + 2} \otimes_{\mathcal{O}_{V_0}} (\wedge^2D_{\tau, V_0})^{k + 1} \rightarrow D_{{\lambda^{\Psi}}, K^p}^{\Psi-\mathrm{la}, (0, k)_{\sigma \neq \tau}, (1+k, -1)_{\tau}}(k+1)^{G_L}.$$

\begin{prop}
    
$N_{\lambda, \tau}' = \frac{1}{2(k + 1)}N_{\lambda, \tau}''$.

\end{prop}

\begin{proof}

For $f \in D_{{\lambda^{\Psi}}, K^p}^{\Psi-\mathrm{la}, (0, k)_{\sigma}, G_L} \otimes_{\mathcal{O}_{V_0}} \omega_{\tau, V_0}^{2k + 2} \otimes_{\mathcal{O}_{V_0}} (\wedge^2D_{\tau, V_0})^{k + 1}$, we take lifts $f_1 \in D_{\lambda^{\Psi}, V_0} \otimes_{\mathcal{O}_{V_0}} E_{\lambda}$ and $f_2 \in D_{\lambda^{\Psi}, V_0} \otimes_{\mathcal{O}_{V_0}} E_0'$ of $f$. Then $N_{\lambda, \tau}'(f) = \theta_{\mathrm{Sen}}(f_1)$, $N_{\lambda, \tau}''(f) = (\Omega_{\tau} - \chi_{\lambda_{\Phi}}(\Omega_{\tau}))(f_2)$ and $f_1 - f_2 \in D_{{\lambda^{\Psi}}, K^p}^{\Psi-\mathrm{la}, (0, k)_{\sigma \neq \tau}, \psi_{k, \tau}}(k+1)$. By Lemma \ref{relation}, we obtain \begin{equation}\label{relationbet}((2\theta_{\mathrm{Sen}} - k - 1)^2 - 1)(f_1 - f_2) = 2 \Omega_{\tau}(f_1 - f_2)\end{equation}. The left hand side is $((2\theta_{\mathrm{Sen}} - k - 1)^2 - 1)(f_1 - f_2) = (4\theta_{\mathrm{Sen}}^2 -4\theta_{\mathrm{Sen}}(k + 1) + (k + 1)^2 - 1)(f_1 - f_2) = (-4\theta_{\mathrm{Sen}}(k + 1) + k^2 + 2k)(f_1 - f_2)$ because $\theta_{\mathrm{Sen}}(f_2) = 0$ and $\theta_{\mathrm{Sen}}^2(f_1) = 0$. Thus the equality (\ref{relationbet}) is equivalent to $-2(k + 1)\theta_{\mathrm{Sen}}(f_1 - f_2) = (\Omega_{\tau} - (\frac{1}{2}k^2 + k))(f_1 - f_2)$. Since $\chi_{\lambda_{\Psi}}(\Omega_{\tau}) = \frac{1}{2}k^2 + k$, we obtain the result. \end{proof}

Thus Proposition \ref{2nd} follows from the following.

\begin{prop}\label{uptoscalar}

There exists $c_{\tau} \in L^{\times}$ such that $N_{\lambda, \tau}'' = c_{\tau} \overline{d}_{\tau}^{\lambda}$.

\end{prop}

\begin{proof}

It suffices to compare the $C$-linear extensions of the above maps. In the following, we ignore Tate twists because we don't use Galois actions. 

Let $$W_{(k, 0)}^{\vee} := \{e_{1, \tau}^{k} \sum_{i_{\tau} \in \mathbb{Z}_{\ge 0}} a_{i_{\tau}}x_{\tau}^{i_{\tau}} \in (\omega_{\tau, \Fl} \otimes \wedge^2 V_{\tau})^{k}(U) \mid a_{i_{\tau}} \in L, a_i = 0 \mathrm{\ for \ almost \ all} \ i \},$$ $$W_{(-1, k + 1)}^{\vee} := \{f_{\tau}^{k+1}e_{1, \tau}^{-k-2} \sum_{i_{\tau} \in \mathbb{Z}_{\ge 0}} a_{i_{\tau}}x_{\tau}^{i_{\tau}} \in (\omega_{\tau, \Fl}^{-k-2} \otimes (\wedge^2 V_{\tau})^{-1})(U) \mid a_{i_{\tau}} \in L, a_i = 0 \mathrm{\ for \ almost \ all} \ i \},$$ $$V':= \{(\prod_{\sigma \neq \tau}e_{1, \sigma}^{k}) \sum_{i = (i_{\sigma})_{\sigma \neq \tau} \in \mathbb{Z}_{\ge 0}^{\Psi \setminus \{ \tau \}}} a_{i}(\prod_{\sigma \neq \tau} x_{\sigma}^{i_{\sigma}}) \in (\otimes_{\sigma \neq \tau} (\omega_{\sigma, \Fl} \otimes \wedge^2 V_{\sigma})^{k})(U) \mid a_{i} \in L, a_{i} = 0 \mathrm{\ for \ almost \ all} \ i  \}$$ and $$V := D_{\lambda^{\Psi}, K^p}^{\mathrm{sm}} \otimes (\otimes_{\sigma \neq \tau} (\omega_{\sigma, K^p}^{-k, \mathrm{sm}} \otimes (\wedge^2 D_{\tau, K^p}^{\mathrm{sm}})^{-k})) \otimes (\omega_{\tau, K^p}^{k+2, \mathrm{sm}} \otimes \wedge^2 D_{\tau, K^p}^{\mathrm{sm}}).$$ (See the discussions before Proposition \ref{mikami expansionII} for the definitions of $e_{1, \tau}$ and $f_{\tau}$.)

Then by the explicit description Proposition \ref{mikami expansionIII}, $V \otimes_L V' \otimes_L W_{(k, 0)}^{\vee}$ (resp. $V \otimes_L V' \otimes_L W_{(-1, k + 1)}^{\vee}$) is the dense subspace of $D_{{\lambda^{\Psi}}, K^p}^{\Psi-\mathrm{la}, (0, k)_{\sigma}, G_L} \otimes_{\mathcal{O}_{V_0}} \omega_{\tau, V_0}^{2k + 2} \otimes_{\mathcal{O}_{V_0}} (\wedge^2D_{\tau, V_0})^{k + 1}$ (resp. $D_{{\lambda^{\Psi}}, K^p}^{\Psi-\mathrm{la}, (0, k)_{\sigma \neq \tau}, (1+k, -1)_{\tau}}$) consisting of $\mathfrak{n}$-nilpotent vectors\footnote{We say that $v$ is a $\mathfrak{n}$-nilpotent vector if for any $X \in \mathfrak{n}$, there exist $n \in \mathbb{Z}_{>0}$ such that $X^nv = 0$.}. See the comments before Definition \ref{nilpotent} for the definition of $\mathfrak{n}$.

%. Note that $V \otimes_L V' \otimes_{L} W_{(k, 0)}^{\vee}$ (resp. $V \otimes_L V' \otimes_L W_{(-1, k + 1)}^{\vee}$) is the subspace of $D_{{\lambda^{\Psi}}, K^p}^{\Psi-\mathrm{la}, (0, k)_{\sigma}, G_L} \otimes_{\mathcal{O}_{V_0}} \omega_{\tau, V_0}^{2k + 2} \otimes_{\mathcal{O}_{V_0}} (\wedge^2D_{\tau, V_0})^{k + 1}$ (resp. $D_{{\lambda^{\Psi}}, K^p}^{\Psi-\mathrm{la}, (0, k)_{\sigma \neq \tau}, (1+k, -1)_{\tau}}$) 

By the continuity of the given maps, it suffices to check the equality on the above dense subspaces. We claim that the considered maps are given by taking $V \otimes_L V' \otimes_L \ $ of certain $\mathfrak{gl}_2(L)$-equivariant maps $W_{(k, 0)}^{\vee} \rightarrow W_{(-1, k+1)}^{\vee}$. By Lemma \ref{anti} and the construction of $\overline{d}^{\lambda}_{\tau}$, under the identification $dx_{\tau} = e_{1, \tau}^{-2}f_{\tau}$, the map $\overline{d}_{\tau}^{\lambda}$ is equal to $\mathrm{id}_{V} \otimes (\frac{\partial}{\partial x_{\tau}})^{k} \otimes (dx_{\tau})^{k}$ up to a scalar of $\mathbb{Q}_p^{\times}$. On the other hand, by taking the $\mathfrak{n}$-nilpotent vectors of the exact sequence of (\ref{overlined})$\widehat{\otimes}_L C$, we obtain the exact sequence $0 \rightarrow V \otimes_L V' \otimes_L W_{(k, 0)}^{\vee} \rightarrow ((D_{\lambda^{\Psi}, V_0} \otimes_{\mathcal{O}_{V_0}} E_0') \widehat{\otimes}_L C)^{\mathfrak{n}-\mathrm{nilp}}\footnote{$\mathfrak{n}-\mathrm{nilp}$ denotes the subspace consisting of $\mathfrak{n}$-nilpotent vectors.} \rightarrow V \otimes_L V'\otimes_L W_{(-1, k + 1)}^{\vee} \rightarrow 0$ by the same calculation as Lemma \ref{derived calculation1}. By construction, the space $(D_{\lambda^{\Psi}, V_0} \otimes_{\mathcal{O}_{V_0}} E_0') \widehat{\otimes}_L C$ is equal to the $\chi_{\lambda_{\Psi}}$-generalized eigenspace of $\mathrm{Sym}^{k+1}V_{\tau} \otimes_{L} \omega_{\tau, {K^p}}^{k + 1, \Pla, (0, k)_{\sigma \neq \tau}, (0, -1)}$. Thus $((D_{\lambda^{\Psi}, V_0} \otimes_{\mathcal{O}_{V_0}} E_0') \widehat{\otimes}_L C)^{\mathfrak{n}-\mathrm{nilp}}$ is equal to $V \otimes_L V' \otimes_{L} E_{\tau}$, where $E_{\tau}$ is the $\chi_{\lambda_{\Psi}}$-generalized eigenspace of $\mathrm{Sym}^{k}V_{\tau} \otimes_L \{ e_{1, \tau} \sum_{i_{\tau} \in \mathbb{Z}_{\ge 0}} a_{i_{\tau}}x_{\tau}^{i_{\tau}} \mid a_{\tau} \in L, a_{\tau} = 0 \mathrm{\ for \ almost \ all \ } i_{\tau} \}$. By the construction of $N_{\lambda, \tau}''$, $N_{\lambda, \tau}''$ is obtained by taking $V \otimes_L V' \otimes_L $ of the map induced from the exact sequence $0 \rightarrow W_{(k, 0)}^{\vee} \rightarrow E_{\tau} \rightarrow W_{(-1, k + 1)}^{\vee} \rightarrow 0$ and the action of $(\Omega_{\tau} - \chi_{\lambda_{\Psi}}(\Omega_{\tau}))$ on $E_{\tau}$.

By \cite[{\S} 1.5]{Hump}, $W_{(-1, k+1)}^{\vee}$ is irreducible and $W_{(k, 0)}^{\vee}$ has a finite-dimensional irreducible subrepresentation and the quotient is isomorphic to $W_{(-1, k+1)}^{\vee}$. Thus we have $$\mathrm{dim}_L\mathrm{Hom}_{\mathfrak{gl}_2(L)}(W_{(k, 0)}^{\vee}, W_{(-1, k+1)}^{\vee}) = 1.$$ Therefore we obtain the result if we prove that the given maps are non-zero. Clearly, we have $\overline{d}_{\tau}^{\lambda} \neq 0$. By \cite[proof of Lemma 6.5.12]{PanII} or \cite[{\S} 3.12]{Hump}, we also have $N_{\lambda, \tau}'' \neq 0$. \end{proof}

\subsubsection{Non-parallel weight case}

In this subsection, we study non-parallel weight cases under the assumption $|\Psi| = 2$.

We put $\Psi = \{ \tau_1, \tau_2 \}$, $\lambda_i := \lambda_{\tau_i}$ and assume $\lambda_1 < \lambda_2$. Again, we consider the following diagram, whose second and third rows and the left column are exact.

 \scriptsize

\xymatrix{
 0 & \\
D_{{\lambda^{\Psi}}, K^p}^{\Psi-\mathrm{la}, (0, \lambda_i)_{i}, G_L} \ar[u] \\
 E_0((V_{\lambda^{\Psi}} \otimes_L \mathbb{B}^+_{\mathrm{dR}, \lambda_2 + 2})^{\Psi-\mathrm{la}, \chi_{\lambda_{\Psi}}, L}) \ar[u] \ar@{^{(}-_>}[r] & \ar[lu] D_{\lambda^{\Psi}, V_0} \otimes_{\mathcal{O}_{V_0}} E_0( \mathcal{O}\mathbb{B}^{+, \Psi-\mathrm{la}, \chi_{\lambda_{\Psi}}, L}_{\mathrm{dR}, \lambda_2 + 2} ) \ar[r]^{\nabla} & D_{\lambda^{\Psi}, V_0} \otimes_{\mathcal{O}_{V_0}} E_0(\mathcal{O}\mathbb{B}^{+, \Psi-\mathrm{la}, \chi_{\lambda_{\Psi}}, L}_{\mathrm{dR}, \lambda_2 + 1}) \otimes_{\mathcal{O}_{V_0}} \Omega_{V_0}^1 \\
 E_0((V_{\lambda^{\Psi}} \otimes_L \mathrm{Fil}^1\mathbb{B}^+_{\mathrm{dR}, \lambda_2 + 2})^{\Psi-\mathrm{la}, \chi_{\lambda_{\Psi}}, L}) \ar@{^{(}-_>}[r] \ar[u] & D_{\lambda^{\Psi}, V_0} \otimes_{\mathcal{O}_{V_0}} E_0(\mathrm{Fil}^{1}\mathcal{O}\mathbb{B}^{+, \Psi-\mathrm{la}, \chi_{\lambda_{\Psi}}, L}_{\mathrm{dR}, \lambda_2+2}) \ar[r]^{\nabla} \ar@{^{(}-_>}[u] & D_{\lambda^{\Psi}, V_0} \otimes_{\mathcal{O}_{V_0}} E_0(\mathcal{O}\mathbb{B}^{+, \Psi-\mathrm{la}, \chi_{\lambda_{\Psi}}, L}_{\mathrm{dR}, \lambda_2 + 1}) \otimes_{\mathcal{O}_{V_0}} \Omega_{V_0}^1 \ar[u]^{\mathrm{id}} \\
 0 \ar[u]
 }

 \normalsize

Let $\chi_{\lambda_i} : Z(U(\mathfrak{g}_{\tau_i})) \rightarrow L$ be the infinitesimal character of $\mathrm{Sym}^{\lambda_i}V_{\tau_i}$. 

In the following, we will modify the above commutative diagram so that we obtain a similar diagram as in the parallel weight case. 

\vspace{0.5 \baselineskip}

By 3 of Proposition \ref{Hodge de Rham}, we have a map $V_{\tau_1} \otimes_{L} \mathcal{O}\mathbb{B}_{\mathrm{dR}}^+ \rightarrow \mathrm{Fil}^0( D_{\tau_1, V_0} \otimes_{\mathcal{O}_{V_0}} \mathcal{O}\mathbb{B}_{\mathrm{dR}}^+)$. By composing this with the quotient map $\mathrm{Fil}^0( D_{\tau_1, V_0} \otimes_{\mathcal{O}_{V_0}} \mathcal{O}\mathbb{B}_{\mathrm{dR}}^+) \rightarrow \mathrm{Fil}^0( D_{\tau_1, V_0} \otimes_{\mathcal{O}_{V_0}} \mathcal{O}\mathbb{B}_{\mathrm{dR}}^+)/\omega_{\tau_1, V_0} \otimes (\wedge^2 D_{\tau_1, V_0} ) \otimes \mathcal{O}\mathbb{B}_{\mathrm{dR}}^+ = \omega_{\tau_1, V_0}^{-1} \otimes_{\mathcal{O}_{V_0}} \mathrm{Fil}^1 \mathcal{O}\mathbb{B}_{\mathrm{dR}}^+$ and taking $\otimes_{\mathcal{O}_{V_0}} \omega_{\tau, V_0}$, we obtain $V_{\tau_1} \otimes_{L} \mathcal{O}\mathbb{B}_{\mathrm{dR}}^+ \otimes_{\mathcal{O}_{V_0}} \omega_{\tau_1, V_0} \rightarrow \mathrm{Fil}^1 \mathcal{O}\mathbb{B}_{\mathrm{dR}}^+$. By taking $\mathrm{Sym}^{\lambda_1 + 1}$ and $\otimes D_{\lambda^{\Psi}, V_0}$, we obtain $D_{\lambda^{\Psi}, V_0} \otimes_{\mathcal{O}_{V_0}} (\mathrm{Sym}^{\lambda_1 + 1} V_{\tau_1} \otimes_{L} \mathcal{O}\mathbb{B}_{\mathrm{dR}}^+ \otimes_{\mathcal{O}_{V_0}} \omega_{\tau_1, V_0}^{\lambda_1 + 1}) \rightarrow D_{\lambda^{\Psi}, V_0} \otimes_{\mathcal{O}_{V_0}} \mathrm{Fil}^{\lambda_1 + 1} \mathcal{O}\mathbb{B}_{\mathrm{dR}}^+$ and $D_{\lambda^{\Psi}, V_0} \otimes_{\mathcal{O}_{V_0}} (\mathrm{Sym}^{\lambda_1 + 1} V_{\tau_1} \otimes_{L} \mathcal{O}\mathbb{B}_{\mathrm{dR}, \lambda_2 - \lambda_1 + 1}^+ \otimes_{\mathcal{O}_{V_0}} \omega_{\tau_1, V_0}^{\lambda_1 + 1}) \rightarrow D_{\lambda^{\Psi}, V_0} \otimes_{\mathcal{O}_{V_0}} \mathrm{Fil}^{\lambda_1 + 1} \mathcal{O}\mathbb{B}_{\mathrm{dR}, \lambda_2 + 2}^+$. We will study the induced map \tiny \begin{align} \label{1} \phi : D_{\lambda^{\Psi}, V_0} \otimes_{\mathcal{O}_{V_0}} E_0((\mathrm{Sym}^{\lambda_1 + 1} V_{\tau_1} \otimes_{L} \mathcal{O}\mathbb{B}_{\mathrm{dR}, \lambda_2 - \lambda_1 + 1}^{+, \Pla, \chi_{\lambda_2}, \chi_{(0, -1)}-\mathrm{nilp}} \otimes_{\mathcal{O}_{V_0}} \omega_{\tau_1, V_0}^{\lambda_1 + 1})^{\chi_{\lambda_1}, L}) \rightarrow D_{\lambda^{\Psi}, V_0} \otimes_{\mathcal{O}_{V_0}} E_0(\mathrm{Fil}^{1}\mathcal{O}\mathbb{B}^{+, \Psi-\mathrm{la}, \chi_{\lambda_{\Psi}}, L}_{\mathrm{dR}, \lambda_2+2}) \end{align} \normalsize

Let $M := \mathrm{Im}(D_{\lambda^{\Psi}, V_0} \otimes_{\mathcal{O}_{V_0}} E_0((\mathrm{Sym}^{\lambda_1 + 1} V_{\tau} \otimes_{L} \mathrm{Fil}^1\mathcal{O}\mathbb{B}_{\mathrm{dR}, \lambda_2 - \lambda_1 + 1}^{+, \Pla, \chi_{\lambda_2}, \chi_{(0, -1)}-\mathrm{nilp}} \otimes_{\mathcal{O}_{V_0}} \omega_{\tau_1, V_0}^{\lambda_1 + 1})^{\chi_{\lambda_1, L}}) \rightarrow D_{\lambda^{\Psi}, V_0} \otimes_{\mathcal{O}_{V_0}} E_0(\mathrm{Fil}^{\lambda_1 + 1}\mathcal{O}\mathbb{B}^{+, \Psi-\mathrm{la}, \chi_{\lambda_{\Psi}}, L}_{\mathrm{dR}, \lambda_2+2}))$ and $N := \nabla(M) + D_{\lambda^{\Psi}, V_0} \otimes_{\mathcal{O}_{V_0}} E_0(t \mathcal{O}\mathbb{B}^{+, \Psi-\mathrm{la}, \chi_{\lambda_{\Psi}}, L}_{\mathrm{dR}, \lambda_2 + 1}) \otimes_{\mathcal{O}_{V_0}} \Omega_{V_0}^1$. 

Then we have the induced map $\psi : D_{\lambda^{\Psi}, V_0} \otimes E_0((\mathrm{Sym}^{\lambda_{1}+1}V_{\tau_1} \otimes \mathcal{O}^{\Psi-\mathrm{la}, \chi_{\lambda_2}, \chi_{(0, -1)}-\mathrm{nilp}}_{K^p} \otimes_{\mathcal{O}_{V_0}} \omega_{\tau_1, V_0}^{\lambda_1 + 1})^{\chi_{\lambda_1}, L}) \rightarrow D_{\lambda^{\Psi}, V_0} \otimes_{\mathcal{O}_{V_0}} E_0(\mathrm{Fil}^{1}\mathcal{O}\mathbb{B}^{+, \Psi-\mathrm{la}, \chi_{\lambda_{\Psi}}, L}_{\mathrm{dR}, \lambda_2+2})/M$. Note that we have the following exact sequence by the same calculation as the first part of the proof of Proposition \ref{2nd}.

\begin{lem}\label{exact sequence}

We have the following natural exact sequence.

$0 \rightarrow D_{{\lambda^{\Psi}}, K^p}^{\Psi-\mathrm{la}, (1+\lambda_1, -1), (0, \lambda_2)}(\lambda_{1}+1)^{G_L} \rightarrow D_{\lambda^{\Psi}, V_0} \otimes E_0((\mathrm{Sym}^{\lambda_{1}+1}V_{\tau_1} \otimes  \mathcal{O}^{\Psi-\mathrm{la}, \chi_{\lambda_2}, \chi_{(0, -1)_1}-\mathrm{nilp}}_{K^p} \otimes_{\mathcal{O}_{V_0}} \omega_{\tau_1, V_0}^{\lambda_1 + 1})^{\chi_{\lambda_1}, L}) \rightarrow D_{{\lambda^{\Psi}}, K^p}^{\Psi-\mathrm{la}, (0, \lambda_i)_{i}, G_L} \otimes_{\mathcal{O}_{V_0}} \omega_{\tau_1, V_0}^{2\lambda_{1} + 2} \otimes_{\mathcal{O}_{V_0}} (\wedge^2D_{\tau_1, V_0})^{\lambda_1 + 1} \rightarrow 0$.

\end{lem}

\begin{proof}

$\mathrm{Sym}^{\lambda_1 + 1}V_{\tau_1} \otimes_L \omega_{\tau_1, K^p}^{\lambda_1 + 1, \Psi-\mathrm{la}, \chi_{\lambda_2}, \chi_{(0, -1)_1}-\mathrm{nilp}}$ has an increasing filtration whose graded pieces are \begin{align}(\omega_{\tau_1, \Fl}^{\lambda_1 + 1 -2i} \otimes_L (\wedge^2 V_{\tau_1})^{\lambda_1 + 1 - i}) \otimes_{\mathcal{O}_{\Fl}}  \omega_{\tau_1, K^p}^{\lambda_1 + 1, \Psi-\mathrm{la}, \chi_{\lambda_2}, \chi_{(0, -1)_1}-\mathrm{nilp}} \nonumber \\
 = \oplus_{(a_2, b_2) = (0, \lambda_2), (\lambda_2 + 1, -1)} D_{\lambda^{\Psi}, K^p}^{\Psi-\mathrm{la}, (a_{2}, b_{2}), (i, \lambda_1 - i)_1-\mathrm{nilp}}(i) \otimes (\omega_{\tau_1, K^p}^{k + 1 - 2i, \mathrm{sm}} \otimes (\wedge^2 D_{\tau_1, K^p}^{\mathrm{sm}})^{\lambda_1 + 1 - i})\end{align} for $i = 0, \cdots \lambda_1 + 1$. Note that we can replace $\chi_{(0, -1)}-\mathrm{nilp}$ by $(0, -1)-\mathrm{nilp}$. (See 1 of Lemma \ref{elementary}.) Thus we obtain the desired exact sequence. \end{proof}

We will prove the following by studying graded pieces of the considered sheaves.

\begin{thm}\label{non-parallel compatibility}

1 \ The restriction of $\nabla \circ \phi$ to $$D_{\lambda^{\Psi}, V_0} \otimes_{\mathcal{O}_{V_0}} E_0((\mathrm{Sym}^{\lambda_1 + 1} V_{\tau} \otimes_{L} \mathrm{Fil}^1\mathcal{O}\mathbb{B}_{\mathrm{dR}, \lambda_2 - \lambda_1 + 1}^{+, \Pla, \chi_{\lambda_2}, \chi_{(0, -1)}-\mathrm{nilp}} \otimes_{\mathcal{O}_{V_0}} \omega_{\tau_1, V_0}^{\lambda_1 + 1})^{\chi_{\lambda_1, L}})$$ is injective. In particular, we have $E_0((V_{\lambda^{\Psi}} \otimes_L \mathrm{Fil}^1\mathbb{B}^+_{\mathrm{dR}, \lambda_2 + 2})^{\Psi-\mathrm{la}, \chi_{\lambda_{\Psi}}, L}) \cap M = 0$.

2 \  $N \cap \mathrm{Im}(\nabla) =\nabla(M)$, i.e., $(D_{\lambda^{\Psi}, V_0} \otimes_{\mathcal{O}_{V_0}} E_0(t \mathcal{O}\mathbb{B}^{+, \Psi-\mathrm{la}, \chi_{\lambda_{\Psi}}, L}_{\mathrm{dR}, \lambda_2 + 1}) \otimes_{\mathcal{O}_{V_0}} \Omega_{V_0}^1) \cap \mathrm{Im}(\nabla) \subset \nabla(M)$.

3 \ The restriction to the subspace $D_{{\lambda^{\Psi}}, K^p}^{\Psi-\mathrm{la}, (0, \lambda_2), (1+\lambda_1, -1)}(\lambda_{1}+1)^{G_L}$ of the map $\nabla \circ \psi$ is zero.

\end{thm}

Before proving this theorem, we see its consequences.

\begin{cor}\label{corollarynon-parallel}

(1) \ We have the following commutative diagram with exact columns, where the middle horizontal map is $\psi$ and the left column is given in Lemma \ref{exact sequence}.

\small

\xymatrix{
0 \ar[d] & 0 \ar[d] \\
D^{\Psi-\mathrm{la}, (\lambda_1 + 1, -1), (0, \lambda_2)}_{\lambda^{\Psi}, K^p}(\lambda_1 + 1)^{G_L} \ar[d] \ar[r] & E_0((V_{\lambda^{\Psi}} \otimes_L \mathrm{Fil}^1\mathbb{B}^+_{\mathrm{dR}, \lambda_2 + 2})^{\Psi-\mathrm{la}, \chi_{\lambda_{\Psi}}, L}) \ar[d] \\
D_{\lambda^{\Psi}, V_0} \otimes E_0((\mathrm{Sym}^{\lambda_{1}+1}V_{\tau_1} \otimes  \mathcal{O}^{\Psi-\mathrm{la}, (0, - 1)_{\tau_1}-\mathrm{nilp}, (0, \lambda_2)_{\tau_2}}_{K^p})^{\chi_{\lambda_1}, L}) \ar[d] \ar[r] & D_{\lambda^{\Psi}, V_0} \otimes_{\mathcal{O}_{V_0}} E_0(\mathrm{Fil}^{1}\mathcal{O}\mathbb{B}^{+, \Psi-\mathrm{la}, \chi_{\lambda_{\Psi}}, L}_{\mathrm{dR}, \lambda_2+2})/M \ar[d] \\
D^{\Psi-\mathrm{la}, (0, \lambda_i), G_L}_{\lambda^{\Psi}, K^p}\otimes \omega^{2\lambda_1 + 2}_{\tau_1, V_0} \otimes (\wedge^2 D_{\tau_1, V_0})^{\lambda_1 + 1} \ar[d] \ar[r] & D_{\lambda^{\Psi}, V_0} \otimes_{\mathcal{O}_{V_0}} E_0(\mathcal{O}\mathbb{B}^{+, \Psi-\mathrm{la}, \chi_{\lambda_{\Psi}}, L}_{\mathrm{dR}, \lambda_2 + 1}) \otimes_{\mathcal{O}_{V_0}} \Omega_{V_0}^1/N \\
0 & 
}

\normalsize

Moreover, the upper map is a section the quotient map $E_0((V_{\lambda^{\Psi}} \otimes_L \mathrm{Fil}^1\mathbb{B}_{\mathrm{dR}, \lambda_2+2}^{+})^{\Psi-\mathrm{la}, \chi_{\lambda}, L}) \twoheadrightarrow D_{\lambda^{\Psi}, K^p}^{\Psi-\mathrm{la}, (1 +\lambda_1, -1),(0, \lambda_2)}(\lambda_1 + 1)^{G_L}$. 

(2) \ The quotient $D_{\lambda^{\Psi}, V_0} \otimes_{\mathcal{O}_{V_0}} E_0(\mathcal{O}\mathbb{B}^{+, \Psi-\mathrm{la}, \chi_{\lambda_{\Psi}}, L}_{\mathrm{dR}, \lambda_2 + 1}) \otimes_{\mathcal{O}_{V_0}} \Omega_{V_0}^1/N$ is naturally a quotient module of $\oplus_{i=0}^{\lambda_2}( D_{{\lambda^{\Psi}}, K^p}^{\Psi-\mathrm{la}, (0, \lambda_{\tau})_{\tau}, G_L} \otimes_{\mathcal{O}_{V_0}} \mathrm{Sym}^i\Omega_{V_0}^1 \otimes_{\mathcal{O}_{V_0}} \Omega_{V_0}^1)$.
 
\end{cor}

\begin{proof} The exactness of the right column in the statement (1) follows from 1 and 2 of Theorem \ref{non-parallel compatibility}. Thus we have the above commutative diagram by 3 of Theorem \ref{non-parallel compatibility}. By the proof of Proposition \ref{FHT} and the construction of the map $\psi$, the upper map in the commutative diagram is a section of the natural quotient map. The statement (2) follows from $N \supset D_{\lambda^{\Psi}, V_0} \otimes_{\mathcal{O}_{V_0}} E_0(t \mathcal{O}\mathbb{B}^{+, \Psi-\mathrm{la}, \chi_{\lambda_{\Psi}}, L}_{\mathrm{dR}, \lambda_2 + 1}) \otimes_{\mathcal{O}_{V_0}} \Omega_{V_0}^1$. \end{proof}

By the above corollary, we obtain the following commutative diagram, which is very similar to the diagram in the parallel weight case and whose second and third rows and the left column are exact.

\tiny

\xymatrix{
 0 & \\
D_{{\lambda^{\Psi}}, K^p}^{\Psi-\mathrm{la}, (0, \lambda_i)_{i}, G_L} \ar[u] \\
 E_0((V_{\lambda^{\Psi}} \otimes_L \mathbb{B}^+_{\mathrm{dR}, \lambda_2 + 2})^{\Psi-\mathrm{la}, \chi_{\lambda_{\Psi}}, L}) \ar[u] \ar@{^{(}-_>}[r] & \ar[lu] D_{\lambda^{\Psi}, V_0} \otimes_{\mathcal{O}_{V_0}} E_0( \mathcal{O}\mathbb{B}^{+, \Psi-\mathrm{la}, \chi_{\lambda_{\Psi}}, L}_{\mathrm{dR}, \lambda_2 + 2} )/M \ar[r]^{\nabla} & D_{\lambda^{\Psi}, V_0} \otimes_{\mathcal{O}_{V_0}} E_0(\mathcal{O}\mathbb{B}^{+, \Psi-\mathrm{la}, \chi_{\lambda_{\Psi}}, L}_{\mathrm{dR}, \lambda_2 + 1}) \otimes_{\mathcal{O}_{V_0}} \Omega_{V_0}^1/N \\
 \oplus_{i=1,2} D_{\lambda^{\Psi}, K^p}^{\Psi-\mathrm{la}, (0, \lambda_j),(1 + \lambda_i, -1)}(\lambda_i + 1)^{G_L} \ar@{^{(}-_>}[r] \ar[u] & D_{\lambda^{\Psi}, V_0} \otimes_{\mathcal{O}_{V_0}} E_0(\mathrm{Fil}^{1}\mathcal{O}\mathbb{B}^{+, \Psi-\mathrm{la}, \chi_{\lambda_{\Psi}}, L}_{\mathrm{dR}, \lambda_2+2})/M \ar[r]^{\nabla} \ar@{^{(}-_>}[u] & D_{\lambda^{\Psi}, V_0} \otimes_{\mathcal{O}_{V_0}} E_0(\mathcal{O}\mathbb{B}^{+, \Psi-\mathrm{la}, \chi_{\lambda_{\Psi}}, L}_{\mathrm{dR}, \lambda_2 + 1}) \otimes_{\mathcal{O}_{V_0}} \Omega_{V_0}^1/N \ar@{^{(}-_>}[u] \\
 0 \ar[u]
 }

  \normalsize

Let $N_{\lambda}^0$ be the map induced by the exact sequence of the left column in the above commutative diagram and the Sen operator on $E_0((V_{\lambda^{\Psi}} \otimes_L \mathbb{B}^+_{\mathrm{dR}, \lambda_2 + 2})^{\Psi-\mathrm{la}, \chi_{\lambda_{\Psi}}, L})$. By this commutative diagram, an almost same argument as in the parallel weight case works and we can prove the following. 

\begin{cor}\label{non-parallel Fontaine}

There exists $c_1, c_2 \in L^{\times}$ such that $N_{\lambda}^0 = \oplus_{i=1, 2} c_i d^{\lambda}_{\tau_i} \circ \overline{d}^{\lambda}_{\tau_i}$.

\end{cor}

\begin{proof} Claerly, the counterpart of Proposition \ref{composition} holds in our situation. Thus we can reduce the proof of this corollary to the counterparts of Propositions \ref{1st} and \ref{2nd}. The proof of Proposition \ref{1st} works in our situation without any change except when we used the relation $\nabla(\tilde{x}_{\tau}) \in t\mathcal{O}\mathbb{B}_{\mathrm{dR}} \otimes \Omega_{V_0}^1$ for any $\tau \in \Psi$. (See the discussions before Proposition \ref{p-adic Legendre} and before Lemma \ref{equality}.) In our situation, we have the property (2) of Corollary \ref{corollarynon-parallel}. Thus the same argument works on $D_{\lambda^{\Psi}, V_0} \otimes_{\mathcal{O}_{V_0}} E_0(\mathcal{O}\mathbb{B}^{+, \Psi-\mathrm{la}, \chi_{\lambda_{\Psi}}, L}_{\mathrm{dR}, \lambda_2 + 1}) \otimes_{\mathcal{O}_{V_0}} \Omega_{V_0}^1/N$. We recall that we proved Proposition \ref{2nd} by fixing $\tau \in \Psi$. For the embedding $\tau_2$, the proof of Proposition \ref{2nd} works in our situation without any change. For the embedding $\tau_1$, by using the commutative diagram (1) of Corollary \ref{corollarynon-parallel} instead of the commutative diagram appeared in the first part of the proof of Proposition \ref{2nd}, exactly the same argument works as in the parallel weight case. \end{proof}

By using Corollary \ref{non-parallel Fontaine}, we can also deduce the following as in the proof of Corollary \ref{parallel comparison}.

\begin{cor}\label{non-parallel comparison}

Let $\mathfrak{m}$ be a decomposed generic non-Eisenstein ideal of $\mathbb{T}^S(K^p, \mathcal{O})$ such that $\overline{\rho}_{\mathfrak{m}}(G_{F})$ is not solvable and $s_{\mathfrak{m}} := \chi_{\mathfrak{m}}^c|_{G_{\tilde{F}}} \otimes (\otimes_{\tau \in \Psi} (\rho_{\mathfrak{m}}|_{G_{\tilde{F}}})^{\tau})$ satisfies the condition that $\overline{s}_{\mathfrak{m}}$ is absolutely irreducible. Let $\varphi : \mathbb{T}^S(K^p, \mathcal{O})_{\mathfrak{m}} \rightarrow \mathcal{O}$ be an $\mathcal{O}$-morphism such that $\rho_{\varphi}|_{G_{F_w}}$ is de Rham of $p$-adic Hodge type $\lambda_w$ for any $w \mid v$ and $\chi_{\varphi}|_{G_{F_{0, v^c}}}$ is de Rham of $p$-adic Hodge type $\lambda_0$. If $\widehat{H}^d(S_{K^p}, V_{\lambda^{\Psi}})_{\mathfrak{m}}^{\Psi-\mathrm{la}}[\varphi] \neq 0$, then $\mathrm{Ker} H^d(\Fl, N)_{\mathfrak{m}}[\varphi] \neq 0$.

\end{cor}

\paragraph{Proof of Theorem \ref{non-parallel compatibility}.} \

\vspace{0.5 \baselineskip}

We recall that we have constructed the map \begin{align}\phi : \small D_{\lambda^{\Psi}, V_0} \otimes_{\mathcal{O}_{V_0}} E_0((\mathrm{Sym}^{\lambda_1 + 1} V_{\tau_1} \otimes_{L} \mathcal{O}\mathbb{B}_{\mathrm{dR}, \lambda_2 - \lambda_1 + 1}^{+, \Pla, \chi_{\lambda_2}, \chi_{(0, -1)}-\mathrm{nilp}} \otimes_{\mathcal{O}_{V_0}} \omega_{\tau_1, V_0}^{\lambda_1 + 1})^{\chi_{\lambda_1}, L}) \nonumber \\
\rightarrow D_{\lambda^{\Psi}, V_0} \otimes_{\mathcal{O}_{V_0}} E_0(\mathrm{Fil}^{1}\mathcal{O}\mathbb{B}^{+, \Psi-\mathrm{la}, \chi_{\lambda_{\Psi}}, L}_{\mathrm{dR}, \lambda_2+2}) - (\ref{1}) \nonumber \end{align} before Lemma \ref{exact sequence}.

\vspace{0.5 \baselineskip}

In the following, we study graded pieces of both sides of (\ref{1}). 

First, we study the right hand side of (\ref{1}). We have a natural filtration on $D_{\lambda^{\Psi}, V_0} \otimes_{\mathcal{O}_{V_0}} E_0(\mathrm{Fil}^{1}\mathcal{O}\mathbb{B}^{+, \Psi-\mathrm{la}, \chi_{\lambda_{\Psi}}, L}_{\mathrm{dR}, \lambda_2+2})$ whose graded pieces are $D_{\lambda^{\Psi}, V_0} \otimes_{\mathcal{O}_{V_0}} E_0(\mathcal{O}_{K^p}^{\Pla, \chi_{\lambda_{\Psi}}, L}(j)) \otimes_{\mathcal{O}_{V_0}} \mathrm{Sym}^{i - j} \Omega_{V_0}$, where $i = 1, \cdots, \lambda_2 + 1$ and $0 \le j \le i$.  Note that for these $i, j$, when $D_{\lambda^{\Psi}, V_0} \otimes_{\mathcal{O}_{V_0}} E_0(\mathcal{O}_{K^p}^{\Pla, \chi_{\lambda_{\Psi}}, L}(j)) \otimes_{\mathcal{O}_{V_0}} \mathrm{Sym}^{i - j} \Omega_{V_0}^1 \neq 0$, we have $j = 0, \lambda_1 + 1$ or $\lambda_2 + 1$ and $D_{\lambda^{\Psi}, V_0} \otimes_{\mathcal{O}_{V_0}} E_0(\mathcal{O}_{K^p}^{\Pla, \chi_{\lambda_{\Psi}}, L}(j)) \otimes_{\mathcal{O}_{V_0}} \mathrm{Sym}^{i - j} \Omega_{V_0}^1$ is equal to $D_{\lambda^{\Psi}, V_0} \otimes_{\mathcal{O}_{V_0}} \mathcal{O}_{K^p}^{\Pla, (0, \lambda_i), L} \otimes_{\mathcal{O}_{V_0}} \mathrm{Sym}^{i} \Omega_{V_0}$, $D_{\lambda^{\Psi}, V_0} \otimes_{\mathcal{O}_{V_0}} \mathcal{O}_{K^p}^{\Pla, (\lambda_1 + 1, -1), (0, \lambda_2)}(\lambda_1 + 1)^{G_L} \otimes_{\mathcal{O}_{V_0}} \mathrm{Sym}^{i - \lambda_1 - 1} \Omega_{V_0}$ or $D_{\lambda^{\Psi}, V_0} \otimes_{\mathcal{O}_{V_0}} \mathcal{O}_{K^p}^{\Pla, (0, \lambda_1), (\lambda_2 + 1, -1)}(\lambda_2 + 1)^{G_L}$.

Thus we have the following exact sequences. 

\begin{align}\label{2}
 0 & \rightarrow  D_{\lambda^{\Psi}, K^p}^{\Pla, (\lambda_1 + 1, -1), (0, \lambda_2)}(\lambda_1 + 1)^{G_L} \otimes_{\mathcal{O}_{V_0}} \mathrm{Sym}^{i} \Omega_{V_0}^1 \rightarrow D_{\lambda^{\Psi}, V_0} \otimes_{\mathcal{O}_{V_0}} E_0(\mathrm{gr}^{i + \lambda_1 + 1}\mathcal{O}\mathbb{B}^{+, \Psi-\mathrm{la}, \chi_{\lambda_{\Psi}}, L}_{\mathrm{dR}}) \nonumber \\
  & \rightarrow D_{\lambda^{\Psi}, K^p}^{\Pla, (0, \lambda_1), (0, \lambda_2)} \otimes_{\mathcal{O}_{V_0}} \mathrm{Sym}^{i + \lambda_1 + 1} \Omega_{V_0}^1 \rightarrow 0
\end{align} for any $0 \le i \le \lambda_2 -\lambda_1 -1$.

%\begin{align}\label{3}
%0 & \rightarrow D_{\lambda^{\Psi}, K^p}^{\Pla, (\lambda_1 + 1, -1), (0, \lambda_2)}(\lambda_1 + 1)^{G_L} \otimes_{\mathcal{O}_{V_0}} \mathrm{Sym}^{i} \Omega_{V_0}^1 \rightarrow D_{\lambda^{\Psi}, V_0} \otimes_{\mathcal{O}_{V_0}} E_0(\mathrm{gr}^{i + \lambda_1 + 1}\mathcal{O}\mathbb{B}^{+, \Psi-\mathrm{la}, \chi_{\lambda_{\Psi}}, L}_{\mathrm{dR}}) \nonumber \\
%& \rightarrow D_{\lambda^{\Psi}, K^p}^{\Pla, (0, \lambda_1), (0, \lambda_2)} \otimes_{\mathcal{O}_{V_0}} \mathrm{Sym}^{i + \lambda_1 + 1} \Omega_{V_0}^1 \rightarrow 0 \end{align} 

\begin{align} \label{4}& 0  \rightarrow  D_{\lambda^{\Psi}, K^p}^{\Pla, (\lambda_1 + 1, -1), (0, \lambda_2)}(\lambda_1 + 1)^{G_L} \otimes_{\mathcal{O}_{V_0}} \mathrm{Sym}^{\lambda_2 -\lambda_1} \Omega_{V_0}^1 \rightarrow \nonumber \\
& D_{\lambda^{\Psi}, V_0} \otimes_{\mathcal{O}_{V_0}} E_0(\mathrm{gr}^{\lambda_2 + 1}\mathcal{O}\mathbb{B}^{+, \Psi-\mathrm{la}, \chi_{\lambda_{\Psi}}, L}_{\mathrm{dR}})/(D_{\lambda^{\Psi}, V_0} \otimes_{\mathcal{O}_{V_0}} \mathcal{O}_{K^p}^{\Pla, (0, \lambda_1), (\lambda_2 + 1, -1)}(\lambda_2 + 1)^{G_L}) \rightarrow \nonumber \\
& \ \ \ \ \ \ \ \ \ \ \ \ \ \ \ \ \ \ \  \ D_{\lambda^{\Psi}, K^p}^{\Pla, (0, \lambda_1), (0, \lambda_2)} \otimes_{\mathcal{O}_{V_0}} \mathrm{Sym}^{\lambda_2 + 1} \Omega_{V_0}^1 \rightarrow 0.\end{align}

\vspace{0.5 \baselineskip}

We also have the equality \begin{equation}\label{5}D_{\lambda^{\Psi}, V_0} \otimes_{\mathcal{O}_{V_0}} E_0(\mathrm{gr}^{i}\mathcal{O}\mathbb{B}^{+, \Psi-\mathrm{la}, \chi_{\lambda_{\Psi}}, L}_{\mathrm{dR}}) = D_{\lambda^{\Psi}, K^p}^{\Pla, (0, \lambda_1), (0, \lambda_2)} \otimes_{\mathcal{O}_{V_0}} \mathrm{Sym}^{i} \Omega_{V_0}^1 \end{equation} when $1 \le i \le \lambda_1$.

\vspace{0.5 \baselineskip}

Moreover by taking $\otimes_{\mathcal{O}_{V_0}} \Omega_{V_0}^1$, we obtain the following from (\ref{2}).

\begin{align}\label{6} & 0 \rightarrow D_{\lambda^{\Psi}, K^p}^{\Pla, (\lambda_1 + 1, -1), (0, \lambda_2)}(\lambda_1 + 1)^{G_L} \otimes_{\mathcal{O}_{V_0}} \mathrm{Sym}^{i-1} \Omega_{V_0}^1 \otimes_{\mathcal{O}_{V_0}} \Omega_{V_0}^1 \rightarrow \nonumber \\
&  D_{\lambda^{\Psi}, V_0} \otimes_{\mathcal{O}_{V_0}} E_0(\mathrm{gr}^{i + \lambda_1}\mathcal{O}\mathbb{B}^{+, \Psi-\mathrm{la}, \chi_{\lambda_{\Psi}}, L}_{\mathrm{dR}}) \otimes_{\mathcal{O}_{V_0}} \Omega_{V_0}^1 \rightarrow D_{\lambda^{\Psi}, K^p}^{\Pla, (0, \lambda_1), (0, \lambda_2)} \otimes_{\mathcal{O}_{V_0}} \mathrm{Sym}^{i + \lambda_1} \Omega_{V_0}^1 \otimes_{\mathcal{O}_{V_0}} \Omega_{V_0}^1 \rightarrow 0 \end{align} for any $1 \le i \le \lambda_2 -\lambda_1$.

\vspace{0.5 \baselineskip}

Note that the map $\nabla : D_{\lambda^{\Psi}, V_0} \otimes_{\mathcal{O}_{V_0}} E_0(\mathrm{Fil}^{1}\mathcal{O}\mathbb{B}^{+, \Psi-\mathrm{la}, \chi_{\lambda_{\Psi}}, L}_{\mathrm{dR}, \lambda_2+2}) \rightarrow D_{\lambda^{\Psi}, V_0} \otimes_{\mathcal{O}_{V_0}} E_0(\mathcal{O}\mathbb{B}^{+, \Psi-\mathrm{la}, \chi_{\lambda_{\Psi}}, L}_{\mathrm{dR}, \lambda_2+1}) \otimes_{\mathcal{O}_{V_0}} \Omega_{V_0}^1$ induces the following commutative diagram for any $0 \le i \le \lambda_2 -\lambda_1 -1$ whose upper maps and lower maps are induced from $\mathrm{Sym}^i\Omega_{V_0}^1 \hookrightarrow \mathrm{Sym}^{i-1}\Omega_{V_0}^1 \otimes_{\mathcal{O}_{V_0}} \Omega_{V_0}^1, \ x_1 \cdots x_i \mapsto \sum_{j = 1}^i x_1 \cdots \check{x_j} \cdots x_{i} \otimes x_j$. (We put $\mathrm{Sym}^{-1}\Omega_{V_0}^1 = \mathcal{O}_{V_0}$.)

\small

\xymatrix{
0 \ar[d] & 0 \ar[d] \\
D_{\lambda^{\Psi}, K^p}^{\Pla, (\lambda_1 + 1, -1), (0, \lambda_2)}(\lambda_1 + 1)^{G_L} \otimes_{\mathcal{O}_{V_0}} \mathrm{Sym}^{i} \Omega_{V_0}^1 \ar[d] \ar@{^{(}-_>}[r] & D_{\lambda^{\Psi}, K^p}^{\Pla, (\lambda_1 + 1, -1), (0, \lambda_2)}(\lambda_1 + 1)^{G_L} \otimes_{\mathcal{O}_{V_0}} \mathrm{Sym}^{i-1} \Omega_{V_0}^1 \otimes_{\mathcal{O}_{V_0}} \Omega_{V_0}^1 \ar[d] \\
 D_{\lambda^{\Psi}, V_0} \otimes_{\mathcal{O}_{V_0}} E_0(\mathrm{gr}^{i + \lambda_1 + 1}\mathcal{O}\mathbb{B}^{+, \Psi-\mathrm{la}, \chi_{\lambda_{\Psi}}, L}_{\mathrm{dR}}) \ar[d] \ar@{^{(}-_>}[r] & \ar[d] D_{\lambda^{\Psi}, V_0} \otimes_{\mathcal{O}_{V_0}} E_0(\mathrm{gr}^{i + \lambda_1}\mathcal{O}\mathbb{B}^{+, \Psi-\mathrm{la}, \chi_{\lambda_{\Psi}}, L}_{\mathrm{dR}}) \otimes_{\mathcal{O}_{V_0}} \Omega_{V_0}^1 \\
D_{\lambda^{\Psi}, K^p}^{\Pla, (0, \lambda_1), (0, \lambda_2)} \otimes_{\mathcal{O}_{V_0}} \mathrm{Sym}^{i + \lambda_1 + 1} \Omega_{V_0}^1 \ar@{^{(}-_>}[r] \ar[d] & D_{\lambda^{\Psi}, K^p}^{\Pla, (0, \lambda_1), (0, \lambda_2)} \otimes_{\mathcal{O}_{V_0}} \mathrm{Sym}^{i + \lambda_1} \Omega_{V_0}^1 \otimes_{\mathcal{O}_{V_0}} \Omega_{V_0}^1 \ar[d] \\
0 & 0
}

\normalsize

Moreover, we have the following commutative diagram. (We put $X := D_{\lambda^{\Psi}, V_0} \otimes_{\mathcal{O}_{V_0}} E_0(\mathrm{gr}^{\lambda_2 + 1}\mathcal{O}\mathbb{B}^{+, \Psi-\mathrm{la}, \chi_{\lambda_{\Psi}}, L}_{\mathrm{dR}})$ and $Y := D_{\lambda^{\Psi}, V_0} \otimes_{\mathcal{O}_{V_0}} \mathcal{O}_{K^p}^{\Pla, (0, \lambda_1), (\lambda_2 + 1, -1)}(\lambda_2 + 1)^{G_L}$.)

\footnotesize

\xymatrix{
0 \ar[d] & 0 \ar[d] \\
D_{\lambda^{\Psi}, K^p}^{\Pla, (\lambda_1 + 1, -1), (0, \lambda_2)}(\lambda_1 + 1)^{G_L} \otimes_{\mathcal{O}_{V_0}} \mathrm{Sym}^{\lambda_2 - \lambda_1} \Omega_{V_0}^1 \ar[d] \ar@{^{(}-_>}[r] & D_{\lambda^{\Psi}, K^p}^{\Pla, (\lambda_1 + 1, -1), (0, \lambda_2)}(\lambda_1 + 1)^{G_L} \otimes_{\mathcal{O}_{V_0}} \mathrm{Sym}^{\lambda_2 - \lambda_1 - 1} \Omega_{V_0}^1 \otimes_{\mathcal{O}_{V_0}} \Omega_{V_0}^1 \ar[d] \\
X/Y \ar[d] \ar@{^{(}-_>}[r] & D_{\lambda^{\Psi}, V_0} \otimes_{\mathcal{O}_{V_0}} E_0(\mathrm{gr}^{\lambda_2}\mathcal{O}\mathbb{B}^{+, \Psi-\mathrm{la}, \chi_{\lambda_{\Psi}}, L}_{\mathrm{dR}}) \otimes_{\mathcal{O}_{V_0}} \Omega_{V_0}^1 \ar[d] \\
D_{\lambda^{\Psi}, K^p}^{\Pla, (0, \lambda_1), (0, \lambda_2)} \otimes_{\mathcal{O}_{V_0}} \mathrm{Sym}^{\lambda_2 + 1} \Omega_{V_0}^1 \ar[d] \ar@{^{(}-_>}[r] & D_{\lambda^{\Psi}, K^p}^{\Pla, (0, \lambda_1), (0, \lambda_2)} \otimes_{\mathcal{O}_{V_0}} \mathrm{Sym}^{\lambda_2} \Omega_{V_0}^1 \otimes_{\mathcal{O}_{V_0}} \Omega_{V_0}^1 \ar[d] \\
0 & 0
}

\normalsize

\vspace{0.5 \baselineskip}

Next, we study the left hand side of (\ref{1}). We have a natural filtration on $$D_{\lambda^{\Psi}, V_0} \otimes_{\mathcal{O}_{V_0}} (\mathrm{Sym}^{\lambda_1 + 1} V_{\tau_1} \otimes_{L} \mathcal{O}\mathbb{B}_{\mathrm{dR}, \lambda_2 - \lambda_1 + 1}^{+, \Pla, \chi_{\lambda_1}, \chi_{(0, -1)}-\mathrm{nilp}} \otimes_{\mathcal{O}_{V_0}} \omega_{\tau_1, V_0}^{\lambda_1 + 1})$$ whose graded pieces are $D_{\lambda^{\Psi}, V_0} \otimes_{\mathcal{O}_{V_0}} (\mathrm{Sym}^{\lambda_1 + 1}V_{\tau_1} \otimes_L (\mathcal{O}_{K^p}^{\Pla, \chi_{\lambda_2}, \chi_{(0, -1)}-\mathrm{nilp}}(j) \otimes_{\mathcal{O}_{V_0}} \mathrm{Sym}^{i - j} \Omega_{V_0}^1 \otimes_{\mathcal{O}_{V_0}} \omega_{\tau_1, V_0}^{\lambda_1 + 1}))$, where $i = 0, \cdots, \lambda_2 - \lambda_1$ and $0 \le j \le i$. Moreover, $$D_{\lambda^{\Psi}, V_0} \otimes_{\mathcal{O}_{V_0}} (\mathrm{Sym}^{\lambda_1 + 1}V_{\tau_1} \otimes_L (\mathcal{O}_{K^p}^{\Pla, \chi_{\lambda_2}, \chi_{(0, -1)}-\mathrm{nilp}}(j) \otimes_{\mathcal{O}_{V_0}} \mathrm{Sym}^{i - j} \Omega_{V_0}^1 \otimes_{\mathcal{O}_{V_0}} \omega_{\tau_1, V_0}^{\lambda_1 + 1}))$$ has a natural filtration whose graded pieces are $$(\omega_{\tau_1, \Fl}^{-k} \otimes_{\mathcal{O}_{\Fl}} (\omega_{\tau_1, \Fl}^{\lambda_1+ 1 - k} \otimes_L (\wedge^2 V_{\tau_1})^{\lambda_1+ 1 - k})) \otimes_{\mathcal{O}_{\Fl}} (D_{\lambda^{\Psi}, K^p}^{\Pla, \chi_{\lambda_2}, \chi_{(0, -1)}-\mathrm{nilp}}(j) \otimes_{\mathcal{O}_{V_0}} \mathrm{Sym}^{i - j} \Omega_{V_0}^1 \otimes_{\mathcal{O}_{V_0}} \omega_{\tau_1, V_0}^{\lambda_1 + 1})$$
$$= \oplus_{(a_2, b_2)} D_{\lambda^{\Psi}, K^p}^{\Psi-\mathrm{la}, (k, \lambda_1 - k)-\mathrm{nilp}, (a_2, b_2)}(j+k) \otimes (\omega_{\tau_1, K^p}^{\lambda_1+ 1 - 2k, \mathrm{sm}} \otimes (\wedge^2 D_{\tau_1, K^p}^{\mathrm{sm}})^{\lambda_1+ 1 - k}) \otimes \mathrm{Sym}^{i - j} \Omega_{V_0}^1,$$ where $(a_2, b_2)$ runs through $(0, \lambda_2), (\lambda_2 + 1, -1)$ and $k = 0, \cdots, \lambda_1 + 1$. Note that by 1 of Lemma \ref{elementary}, we can replace $\chi_{(0, -1)}-\mathrm{nilp}$ by $(0, -1)-\mathrm{nilp}$.

Thus the $\chi_{\lambda_1}$-generalized eigenspace of $$((\omega_{\tau_1, \Fl}^{-k} \otimes_{\mathcal{O}_{\Fl}} (\omega_{\tau_1, \Fl}^{\lambda_1+ 1 - k} \otimes_L (\wedge^2 V_{\tau_1})^{\lambda_1+ 1 - k})) \otimes_{\mathcal{O}_{\Fl}} (D_{\lambda^{\Psi}, K^p}^{\Pla, \chi_{\lambda_2}, \chi_{(0, -1)}-\mathrm{nilp}}(j) \otimes_{\mathcal{O}_{V_0}} \mathrm{Sym}^{i - j} \Omega_{V_0}^1 \otimes_{\mathcal{O}_{V_0}} \omega_{\tau_1, V_0}^{\lambda_1 + 1}))$$ is nonzero if and only if $k = 0, \lambda_1 + 1$ and in these cases, this is equal to $$\oplus_{(a_2, b_2)} (\omega_{\tau_1, K^p}^{\lambda_1+ 1, \mathrm{sm}} \otimes_{\mathcal{O}_{K^p}^{\mathrm{sm}}} (\wedge^2 D_{\tau_1, K^p}^{\mathrm{sm}})^{\lambda_1+ 1}) \otimes_{\mathcal{O}_{K^p}^{\mathrm{sm}}} (D_{\lambda^{\Psi}, K^p}^{\Pla, (0, \lambda_1), (a_2, b_2)}(j) \otimes_{\mathcal{O}_{V_0}} \mathrm{Sym}^{i - j} \Omega_{V_0}^1)$$ or $$\oplus_{(a_2, b_2)} (\omega_{\tau_1, K^p}^{-\lambda_1 - 1, \mathrm{sm}}) \otimes_{\mathcal{O}_{K^p}^{\mathrm{sm}}} (D_{\lambda^{\Psi}, K^p}^{\Pla, (\lambda_1 + 1, -1), (a_2, b_2)}(j+\lambda_1 + 1) \otimes_{\mathcal{O}_{V_0}} \mathrm{Sym}^{i - j} \Omega_{V_0}^1).$$ 

Thus for $i, j$ as above, when $$D_{\lambda^{\Psi}, V_0} \otimes_{\mathcal{O}_{V_0}} E_0((\mathrm{Sym}^{\lambda_1 + 1} V_{\tau_1} \otimes_{L} \mathcal{O}_{K^p}^{\Pla, \chi_{\lambda_2}, \chi_{(0, -1)}-\mathrm{nilp}}(j) \otimes_{\mathcal{O}_{V_0}} \mathrm{Sym}^{i - j} \Omega_{V_0}^1 \otimes_{\mathcal{O}_{V_0}} \omega_{\tau_1, V_0}^{\lambda_1 + 1})^{\chi_{\lambda_1}, L})$$ is nonzero, we have $j = 0$ and thus we obtain $$D_{\lambda^{\Psi}, V_0} \otimes_{\mathcal{O}_{V_0}} E_0((\mathrm{Sym}^{\lambda_1 + 1} V_{\tau_1} \otimes_{L} \mathcal{O}\mathbb{B}_{\mathrm{dR}, \lambda_2 - \lambda_1 + 1}^{+, \Pla, \chi_{\lambda_2}, \chi_{(0, -1)}-\mathrm{nilp}})^{\chi_{\lambda_1}, L}) \otimes_{\mathcal{O}_{V_0}} \omega_{\tau_1, V_0}^{\lambda_1 + 1}$$ $$= \oplus_{i = 0}^{\lambda_2 - \lambda_1} D_{\lambda^{\Psi}, V_0} \otimes_{\mathcal{O}_{V_0}} E_0((\mathrm{Sym}^{\lambda_1 + 1} V_{\tau_1} \otimes_L \mathcal{O}_{K^p}^{\Pla, \chi_{\lambda_2}, \chi_{(0, -1)}-\mathrm{nilp}})^{\chi_{\lambda_1}, L}) \otimes_{\mathcal{O}_{V_0}} \mathrm{Sym}^i\Omega_{V_0}^1 \otimes_{\mathcal{O}_{V_0}} \omega_{\tau_1, V_0}^{\lambda_1 + 1}.$$

Moreover, we have the following exact sequence. (i.e. We only have the contribution of the term $(a_2, b_2) = (0, \lambda_2)$ in the above calculation.)

\begin{align}\label{7} & 0 \rightarrow D_{\lambda^{\Psi}, K^p}^{\Psi-\mathrm{la}, (1 +\lambda_1, -1),(0, \lambda_2)}(\lambda_1 + 1)^{G_L} \otimes_{\mathcal{O}_{V_0}} \mathrm{Sym}^i\Omega_{V_0}^1 \rightarrow \nonumber \\
& D_{\lambda^{\Psi}, V_0} \otimes_{\mathcal{O}_{V_0}} E_0((\mathrm{Sym}^{\lambda_1 + 1} V_{\tau} \otimes_{L} \mathrm{gr}^i\mathcal{O}\mathbb{B}_{\mathrm{dR}}^{+, \Pla, \chi_{\lambda_2}, \chi_{(0, -1)}-\mathrm{nilp}} \otimes_{\mathcal{O}_{V_0}} \omega_{\tau_1, V_0}^{\lambda_1 + 1})^{\chi_{\lambda_1, L}}) \rightarrow \nonumber \\
& D_{\lambda^{\Psi}, K^p}^{\Psi-\mathrm{la}, (0, \lambda_1),(0, \lambda_2), {G_L}} \otimes_{\mathcal{O}_{V_0}} \omega_{\tau_1, V_0}^{2\lambda_1 + 2} \otimes_{\mathcal{O}_{V_0}} (\wedge^2 D_{\tau_1, V_{0}} )^{\lambda_1 + 1} \otimes_{\mathcal{O}_{V_0}} \mathrm{Sym}^i\Omega_{V_0}^1 \rightarrow 0 \end{align}

\vspace{0.5 \baselineskip}

We recall that we have constructed the map \begin{align} \phi : D_{\lambda^{\Psi}, V_0} \otimes_{\mathcal{O}_{V_0}} E_0((\mathrm{Sym}^{\lambda_1 + 1} V_{\tau_1} \otimes_{L} \mathcal{O}\mathbb{B}_{\mathrm{dR}, \lambda_2 - \lambda_1 + 1}^{+, \Pla, \chi_{\lambda_2}, \chi_{(0, -1)}-\mathrm{nilp}} \otimes_{\mathcal{O}_{V_0}} \omega_{\tau_1, V_0}^{\lambda_1 + 1})^{\chi_{\lambda_1}, L}) \nonumber \\ 
\rightarrow D_{\lambda^{\Psi}, V_0} \otimes_{\mathcal{O}_{V_0}} E_0(\mathrm{Fil}^{\lambda_1 + 1}\mathcal{O}\mathbb{B}^{+, \Psi-\mathrm{la}, \chi_{\lambda_{\Psi}}, L}_{\mathrm{dR}, \lambda_2+2}) - (\ref{1}) \nonumber \end{align}. By taking graded pieces, we obtain a map \tiny \begin{align} \label{8} D_{\lambda^{\Psi}, V_0} \otimes_{\mathcal{O}_{V_0}} E_0((\mathrm{Sym}^{\lambda_1 + 1} V_{\tau} \otimes_{L} \mathrm{gr}^i\mathcal{O}\mathbb{B}_{\mathrm{dR}}^{+, \Pla, \chi_{\lambda_2}, \chi_{(0, -1)}-\mathrm{nilp}} \otimes_{\mathcal{O}_{V_0}} \omega_{\tau_1, V_0}^{\lambda_1 + 1})^{\chi_{\lambda_1, L}}) \rightarrow D_{\lambda^{\Psi}, V_0} \otimes_{\mathcal{O}_{V_0}} E_0(\mathrm{gr}^{\lambda_1 + i + 1}\mathcal{O}\mathbb{B}^{+, \Psi-\mathrm{la}, \chi_{\lambda_{\Psi}}, L}_{\mathrm{dR}}) \end{align} \normalsize for any $0 \le i \le \lambda_2 - \lambda_1$. Note that this is a map between the middle terms of the exact sequences of (\ref{2}), (\ref{4}) and (\ref{7}). (Precisely, the middle term of (\ref{4}) is a quotient of $D_{\lambda^{\Psi}, V_0} \otimes_{\mathcal{O}_{V_0}} E_0(\mathrm{gr}^{\lambda_2 + 1}\mathcal{O}\mathbb{B}^{+, \Psi-\mathrm{la}, \chi_{\lambda_{\Psi}}, L}_{\mathrm{dR}})$.)

\vspace{0.5 \baselineskip}

By taking the composition of this map with $D_{\lambda^{\Psi}, V_0} \otimes_{\mathcal{O}_{V_0}} E_0(\mathrm{Fil}^{1}\mathcal{O}\mathbb{B}^{+, \Psi-\mathrm{la}, \chi_{\lambda_{\Psi}}, L}_{\mathrm{dR}, \lambda_2+2}) \xrightarrow{\nabla} D_{\lambda^{\Psi}, V_0} \otimes_{\mathcal{O}_{V_0}} E_0(\mathcal{O}\mathbb{B}^{+, \Psi-\mathrm{la}, \chi_{\lambda_{\Psi}}, L}_{\mathrm{dR}, \lambda_2+1}) \otimes_{\mathcal{O}_{V_0}} \Omega_{V_0}^1$, we obtain a map \tiny \begin{align} \label{9} D_{\lambda^{\Psi}, V_0} \otimes_{\mathcal{O}_{V_0}} E_0((\mathrm{Sym}^{\lambda_1 + 1} V_{\tau} \otimes_{L} \mathrm{gr}^i\mathcal{O}\mathbb{B}_{\mathrm{dR}}^{+, \Pla, \chi_{\lambda_2}, \chi_{(0, -1)}-\mathrm{nilp}} \otimes_{\mathcal{O}_{V_0}} \omega_{\tau_1, V_0}^{\lambda_1 + 1})^{\chi_{\lambda_1, L}}) \rightarrow D_{\lambda^{\Psi}, V_0} \otimes_{\mathcal{O}_{V_0}} E_0(\mathrm{gr}^{i + \lambda_1}\mathcal{O}\mathbb{B}^{+, \Psi-\mathrm{la}, \chi_{\lambda_{\Psi}}, L}_{\mathrm{dR}}) \otimes_{\mathcal{O}_{V_0}} \Omega_{V_0}^1 \end{align} \normalsize \ for any $0 \le i \le \lambda_2 - \lambda_1$.

\vspace{0.5 \baselineskip}

Thus in order to prove Theorem \ref{non-parallel compatibility}, it suffices to prove the followings.

\begin{prop}\label{compatibility difficult} \

For any $1 \le i \le \lambda_2 - \lambda_1$, the map  
$D_{\lambda^{\Psi}, V_0} \otimes_{\mathcal{O}_{V_0}} E_0((\mathrm{Sym}^{\lambda_1 + 1} V_{\tau} \otimes_{L} \mathrm{gr}^i\mathcal{O}\mathbb{B}_{\mathrm{dR}}^{+, \Pla, \chi_{\lambda_2}, \chi_{(0, -1)}-\mathrm{nilp}} \otimes_{\mathcal{O}_{V_0}} \omega_{\tau_1, V_0}^{\lambda_1 + 1})^{\chi_{\lambda_1, L}}) \rightarrow D_{\lambda^{\Psi}, V_0} \otimes_{\mathcal{O}_{V_0}} E_0(\mathrm{gr}^{\lambda_1 + i}\mathcal{O}\mathbb{B}^{+, \Psi-\mathrm{la}, \chi_{\lambda_{\Psi}}, L}_{\mathrm{dR}}) \otimes_{\mathcal{O}_{V_0}} \Omega_{V_0}^1$ - (\ref{9}) gives the natural inclusion of the subspaces in the exact sequences (\ref{4}) and (\ref{7}) induced by $\mathrm{Sym}^i\Omega_{V_0}^1 \hookrightarrow \mathrm{Sym}^{i-1}\Omega_{V_0}^1 \otimes_{\mathcal{O}_{V_0}} \Omega_{V_0}^1, \ x_1 \cdots x_i \mapsto \sum_{j = 1}^i x_1 \cdots \check{x_j} \cdots x_{i} \otimes x_j$ and the injection of the quotient spaces induced by $(-\mathrm{KS})^{\lambda_1 + 1}$.

%2 \ We have the following commutative diagram.

% 

%$D_{\lambda^{\Psi}, V_0} \otimes_{\mathcal{O}_{V_0}} E_0(\mathrm{Fil}^{\lambda_1 + 1}\mathcal{O}\mathbb{B}^{+, \Psi-\mathrm{la}, \chi_{\lambda_{\Psi}}, L}_{\mathrm{dR}, \lambda_2+2}) \rightarrow D_{\lambda^{\Psi}, V_0} \otimes_{\mathcal{O}_{V_0}} E_0(\mathrm{Fil}^{\lambda_1}\mathcal{O}\mathbb{B}^{+, \Psi-\mathrm{la}, \chi_{\lambda_{\Psi}}, L}_{\mathrm{dR}, \lambda_2 + 1}) \otimes_{\mathcal{O}_{V_0}} \Omega_{V_0}^1 \rightarrow D_{\lambda^{\Psi}, K^p}^{\Psi-\mathrm{la}, (0, \lambda_i), G_L} \otimes_{\mathcal{O}_{V_0}} \mathrm{Sym}^{\lambda_1}\Omega_{V_0}^{1} \otimes_{\mathcal{O}_{V_0}} \Omega_{V_0}^1$.

%$ \  \ \ \ \ \ \ \ \ \ \ \ \ \ \ \ \ \ \ \ \ \ \ \ \ \ \ \ \uparrow  \ \ \  \ \ \  \ \  \ \ \ \ \  \ \ \ \  \ \ \ \ \ \ \ \ \ \ \ \ \ \ \ \ \ \ \ \ \ \ \ \ \ \ \ \ \ \ \ \ \ \ \ \ \ \ \ \ \ \ \ \ \ \ \ \ \ \ \ \ \ \ \ \ \ \ \ \ \ \ \ \ \ \ \ \ \ \ \ \ \ \ \ \ \ \ \ \ \ \ \ \ \ \ \ \ \ \ \ \ \ \uparrow $

%$E_0((\mathrm{Sym}^{\lambda_1 + 1} V_{\tau_1} \otimes_L D_{\lambda^{\Psi}, K^p}^{\Pla, \chi_{\lambda_2}, \chi_{(0, -1)}-\mathrm{nilp}})^{\chi_{\lambda_1}, L}) \otimes_{\mathcal{O}_{V_0}} \omega_{\tau_1, V_0}^{\lambda_1 + 1} \rightarrow D_{\lambda^{\Psi}, K^p}^{\Psi-\mathrm{la}, (0, \lambda_i), {G_L}} \otimes_{\mathcal{O}_{V_0}} \omega_{\tau_1, V_0}^{2\lambda_1 + 2} \otimes_{\mathcal{O}_{V_0}} (\wedge^2 D_{\tau_1, V_{0}} )^{\lambda_1 + 1}$.

% 

\end{prop}

\begin{prop}\label{zero map}

The restriction to the subspace $D^{\Psi-\mathrm{la}, (\lambda_1 + 1, -1), (0, \lambda_2)}_{\lambda^{\Psi}, K^p}(\lambda_1 + 1)^{G_L}$ of the map $D_{\lambda^{\Psi}, V_0} \otimes E_0((\mathrm{Sym}^{\lambda_{1}+1}V_{\tau_1} \otimes  \mathcal{O}^{\Psi-\mathrm{la}, \chi_{(0, - 1)}-\mathrm{nilp}, \chi_{\lambda_2}}_{K^p})^{\chi_{\lambda_1}, L}) \rightarrow D_{\lambda^{\Psi}, V_0} \otimes_{\mathcal{O}_{V_0}} E_0(\mathrm{Fil}^{1}\mathcal{O}\mathbb{B}^{+, \Psi-\mathrm{la}, \chi_{\lambda_{\Psi}}, L}_{\mathrm{dR}, \lambda_2+2})/M \rightarrow D_{\lambda^{\Psi}, V_0} \otimes_{\mathcal{O}_{V_0}} E_0(\mathcal{O}\mathbb{B}^{+, \Psi-\mathrm{la}, \chi_{\lambda_{\Psi}}, L}_{\mathrm{dR}, \lambda_2 + 1}) \otimes_{\mathcal{O}_{V_0}} \Omega_{V_0}^1/N$ is zero.

\end{prop}

In fact, Proposition \ref{zero map} implies that 3 of Theorem \ref{non-parallel compatibility}. Moreover, 1 and 2 of Theorem \ref{non-parallel compatibility} follow from Proposition \ref{compatibility difficult} because Proposition \ref{compatibility difficult} says that the graded maps of 1 is injective and thus the map of 1 is injective. Moreover, Proposition \ref{compatibility difficult} implies $(D_{\lambda^{\Psi}, V_0} \otimes_{\mathcal{O}_{V_0}} E_0(t \mathcal{O}\mathbb{B}^{+, \Psi-\mathrm{la}, \chi_{\lambda_{\Psi}}, L}_{\mathrm{dR}, \lambda_2 + 1}) \otimes_{\mathcal{O}_{V_0}} \Omega_{V_0}^1) \cap \mathrm{Im}(\nabla) \subset \nabla(M)$ because the map (\ref{8}) induces the isomorphism of the subspaces in the exact sequences (\ref{2}), (\ref{4}) and (\ref{7}). 

\vspace{0.5 \baselineskip}

In order to prove the above results, we may assume $\lambda^{\Psi} = 0$ because the given maps are equal to $\mathrm{id}_{D_{\lambda^{\Psi}, V_0}} \otimes_{\mathcal{O}_{V_0}}$ (the maps when $\lambda^{\Psi} = 0$). Proposition \ref{compatibility difficult} follows from the following proposition, which can be regarded as a version of the above Proposition \ref{compatibility difficult} before taking many operations and a generalization of Proposition \ref{FHT}.

\begin{prop}

For any $i \ge 1$, the map $\mathrm{Sym}^{\lambda_1 + 1} V_{\tau_1} \otimes_{L} \mathrm{gr}^i\mathcal{O}\mathbb{B}_{\mathrm{dR}}^{+} \otimes_{\mathcal{O}_{K^p}} \omega_{\tau_1, K^p}^{\lambda_1 + 1} \rightarrow \mathrm{gr}^{\lambda_1 + 1 + i}\mathcal{O}\mathbb{B}^{+}_{\mathrm{dR}}$ induces the identity of the subspaces $\mathrm{gr}^i\mathcal{O}\mathbb{B}_{\mathrm{dR}}^+(\lambda_1 + 1) \xrightarrow{\mathrm{id}} \mathrm{gr}^{i}\mathcal{O}\mathbb{B}_{\mathrm{dR}}^+(\lambda_1 + 1)$ and the map of the quotient spaces $(\omega_{\tau_1, K^p}^2 \otimes_{\mathcal{O}_{K^p}} \wedge^2 D_{\tau_1, K^p})^{\lambda_1 + 1} \otimes_{\mathcal{O}_{K^p}} \mathrm{Sym}^i\Omega_{K^p}^1 \xrightarrow{(-\mathrm{KS})^{\lambda_1 + 1}} \mathrm{Sym}^{\lambda_1 + i + 1}\Omega_{K^p}^1$.

%2 \ We have the following commutative diagram.

%\xymatrix{
%\mathrm{gr}^{\lambda_1 + 1}\mathcal{O}\mathbb{B}^{+}_{\mathrm{dR}} \ar[r] & \mathrm{Sym}^{\lambda_1}\Omega_{K^p}^{1} \otimes_{\mathcal{O}_{K^p}} \Omega_{K^p}^1  \\
%\mathrm{Sym}^{\lambda_1 + 1} V_{\tau_1} \otimes_L \mathcal{O}_{K^p} \otimes_{\mathcal{O}_{V_0}} \omega_{\tau_1, V_0}^{\lambda_1 + 1} \ar[u] \ar[r] &  \omega_{\tau_1, K^p}^{2\lambda_1 + 2} \otimes_{\mathcal{O}_{K^p}} (\wedge^2 D_{\tau_1, K^p})^{\lambda_1 + 1} \ar[u].
%}

\end{prop}

%The statement 2 is clear from Proposition \ref{FHT}. Thus we consider the statement 1. 

\begin{proof} We recall the construction of the given map. By 3 of Proposition \ref{Hodge de Rham}, we have a map $V_{\tau_1} \otimes_{L} \mathcal{O}\mathbb{B}_{\mathrm{dR}}^+ \rightarrow \mathrm{Fil}^0( D_{\tau_1, V_0} \otimes_{\mathcal{O}_{V_0}} \mathcal{O}\mathbb{B}_{\mathrm{dR}}^+) = D_{\tau_1, V_0} \otimes_{\mathcal{O}_{V_0}} \mathrm{Fil}^1\mathcal{O}\mathbb{B}_{\mathrm{dR}}^{+} + \omega_{\tau_1, V_0} \otimes_{\mathcal{O}_{V_0}} (\wedge^2 D_{\tau_1, V_0}) \otimes_{\mathcal{O}_{V_0}} \mathcal{O}\mathbb{B}_{\mathrm{dR}}^+ \rightarrow \mathrm{Fil}^0( D_{\tau_1, V_0} \otimes_{\mathcal{O}_{V_0}} \mathcal{O}\mathbb{B}_{\mathrm{dR}}^+)/\omega_{\tau_1, V_0} \otimes (\wedge^2 D_{\tau_1, V_0} ) \otimes \mathcal{O}\mathbb{B}_{\mathrm{dR}}^+ = \omega_{\tau_1, V_0}^{-1} \otimes_{\mathcal{O}_{V_0}} \mathrm{Fil}^1 \mathcal{O}\mathbb{B}_{\mathrm{dR}}^+$. Note that we can get the quotient map of the Hodge-Tate filtration by $V_{\tau_1} \otimes_{L} \mathcal{O}\mathbb{B}_{\mathrm{dR}}^+ \rightarrow D_{\tau_1, V_0} \otimes_{\mathcal{O}_{V_0}} \mathrm{Fil}^1\mathcal{O}\mathbb{B}_{\mathrm{dR}}^{+} + \omega_{\tau_1, V_0} \otimes_{\mathcal{O}_{V_0}} (\wedge^2 D_{\tau_1, V_0}) \otimes_{\mathcal{O}_{V_0}} \mathcal{O}\mathbb{B}_{\mathrm{dR}}^+ \rightarrow \omega_{\tau_1, V_0} \otimes_{\mathcal{O}_{V_0}} (\wedge^2 D_{\tau_1, V_0}) \otimes_{\mathcal{O}_{V_0}} \mathcal{O}_{K^p}$. By taking $\mathrm{Sym}^{\lambda_1 + 1}$, we obtain $\mathrm{Sym}^{\lambda_1 + 1} V_{\tau_1} \otimes_{L} \mathcal{O}\mathbb{B}_{\mathrm{dR}}^+ \rightarrow \omega_{\tau_1, V_0}^{-\lambda_1 - 1} \otimes_{\mathcal{O}_{V_0}} \mathrm{Fil}^{\lambda_1 + 1} \mathcal{O}\mathbb{B}_{\mathrm{dR}}^+$ and $\mathrm{Sym}^{\lambda_1 + 1} V_{\tau_1} \otimes_{L} \mathrm{gr}^i\mathcal{O}\mathbb{B}_{\mathrm{dR}}^+ \rightarrow \omega_{\tau_1, V_0}^{-\lambda_1 - 1} \otimes_{\mathcal{O}_{V_0}} \mathrm{gr}^{\lambda_1 + i + 1} \mathcal{O}\mathbb{B}_{\mathrm{dR}}^+ \xrightarrow{\nabla} \omega_{\tau_1, V_0}^{-\lambda_1 - 1} \otimes_{\mathcal{O}_{V_0}} \mathrm{gr}^{\lambda_1 + i} \mathcal{O}\mathbb{B}_{\mathrm{dR}}^+ \otimes_{\mathcal{O}_{V_0}} \Omega_{V_0}^1$. 

First, we prove the statement about the subspaces. Let $M := \mathrm{Fil}^0( D_{\tau_1, V_0} \otimes_{\mathcal{O}_{V_0}} \mathcal{O}\mathbb{B}_{\mathrm{dR}}^+)^{\nabla = 0}(1)$. Then $V_{\tau_1}(1) \otimes_L \mathbb{B}_{\mathrm{dR}}^+ \subset M \subset V_{\tau_1} \otimes_L \mathbb{B}_{\mathrm{dR}}^+$, $M/V_{\tau_1}(1) \otimes_L \mathbb{B}_{\mathrm{dR}}^+ = \omega_{\tau_1, V_0}^{-1}(1) \otimes \mathcal{O}_{K^p}$ and $V_{\tau_1} \otimes_L \mathbb{B}_{\mathrm{dR}}^+/M = \omega_{\tau_1, V_0} \otimes \wedge^2 D_{\tau_1, V_0} \otimes \mathcal{O}_{K^p}$ by Proposition \ref{Hodge de Rham}. Thus the statement about the subspace follows from $M = \mathrm{Fil}^0( D_{\tau_1, V_0} \otimes_{\mathcal{O}_{V_0}} t\mathcal{O}\mathbb{B}_{\mathrm{dR}}^+)^{\nabla = 0} \subset D_{\tau_1, V_0} \otimes_{\mathcal{O}_{V_0}} \mathrm{Fil}^1\mathcal{O}\mathbb{B}_{\mathrm{dR}}^+$ and the above construction of the map.

About the statement on the quotient spaces, first note that the composition of the following maps $V_{\tau_1} \otimes_{L} \mathcal{O}_{K^p} \rightarrow \omega_{\tau_1, V_0}^{-1} \otimes \mathrm{gr}^1\mathcal{O}\mathbb{B}_{\mathrm{dR}}^{+} \oplus \omega_{\tau_1, V_0} \otimes_{\mathcal{O}_{V_0}} (\wedge^2 D_{\tau_1, V_0}) \otimes_{\mathcal{O}_{V_0}} \mathcal{O}_{K^p} \rightarrow \omega_{\tau_1, V_0}^{-1} \otimes_{\mathcal{O}_{V_0}} \mathrm{gr}^1\mathcal{O}\mathbb{B}_{\mathrm{dR}}^{+} \xrightarrow{\nabla} \omega_{\tau_1, V_0}^{-1} \otimes_{\mathcal{O}_{V_0}} \mathcal{O}_{K^p} \otimes_{\mathcal{O}_{V_0}} \Omega_{V_0}^1$ is equal to $V_{\tau_1} \otimes_{L} \mathcal{O}_{K^p} \rightarrow \omega_{\tau_1, K^p} \otimes_{\mathcal{O}_{V_0}} (\wedge^2 D_{\tau_1, V_0}) \xrightarrow{-\mathrm{KS}} \omega_{\tau_1, K^p}^{-1} \otimes_{\mathcal{O}_{V_0}} \Omega_{V_0}^1$ because $\mathrm{Im}(V_{\tau_1} \otimes_{L} \mathcal{O}_{K^p} \rightarrow \omega_{\tau_1, V_0}^{-1} \otimes_{\mathcal{O}_{V_0}} \mathrm{gr}^1\mathcal{O}\mathbb{B}_{\mathrm{dR}}^{+} \oplus \omega_{\tau_1, V_0} \otimes_{\mathcal{O}_{V_0}} (\wedge^2 D_{\tau_1, V_0}) \otimes_{\mathcal{O}_{V_0}} \mathcal{O}_{K^p}) \subset (\omega_{\tau_1, V_0}^{-1} \otimes_{\mathcal{O}_{V_0}} \mathrm{gr}^1\mathcal{O}\mathbb{B}_{\mathrm{dR}}^{+} \oplus \omega_{\tau_1, V_0} \otimes_{\mathcal{O}_{V_0}} (\wedge^2 D_{\tau_1, V_0}) \otimes_{\mathcal{O}_{V_0}} \mathcal{O}_{K^p})^{\nabla = 0}$. (Actually, this was used in the proof of Proposition \ref{FHT}.) Thus the map $\mathrm{Sym}^{\lambda_1 + 1} V_{\tau_1} \otimes_{L} \mathcal{O}_{K^p} \rightarrow \omega_{\tau_1, V_0}^{-\lambda_1 - 1} \otimes_{\mathcal{O}_{V_0}} \mathcal{O}_{K^p} \otimes_{\mathcal{O}_{V_0}} \mathrm{Sym}^{\lambda_1 + 1}\Omega_{V_0}^1$ obtained by the quotient of the map $\mathrm{Sym}^{\lambda_1 + 1} V_{\tau_1} \otimes_{L} \mathcal{O}\mathbb{B}_{\mathrm{dR}}^+ \rightarrow \omega_{\tau_1, V_0}^{-\lambda_1 - 1} \otimes_{\mathcal{O}_{V_0}} \mathrm{Fil}^{\lambda_1 + 1} \mathcal{O}\mathbb{B}_{\mathrm{dR}}^+$ is equal to $\mathrm{Sym}^{\lambda_1 + 1} V_{\tau_1} \otimes_{L} \mathcal{O}_{K^p} \rightarrow \omega_{\tau_1, K^p}^{\lambda_1 + 1} \otimes_{\mathcal{O}_{V_0}} (\wedge^2 D_{\tau_1, V_0})^{\lambda_1 + 1} \xrightarrow{(-\mathrm{KS})^{\lambda_1 + 1}} \omega_{\tau_1, V_0}^{-\lambda_1 - 1} \otimes_{\mathcal{O}_{V_0}} \mathcal{O}_{K^p} \otimes_{\mathcal{O}_{V_0}} \mathrm{Sym}^{\lambda_1 + 1}\Omega_{V_0}^1$. By taking $\otimes_{\mathcal{O}_{V_0}} \mathrm{Sym}^i\Omega_{V_0}^1$, we obtain the result because we consider the $\mathcal{O}\mathbb{B}_{\mathrm{dR}}^{+}$-linear map $\mathrm{Sym}^{\lambda_1 + 1} V_{\tau_1} \otimes_{L} \mathcal{O}\mathbb{B}_{\mathrm{dR}}^{+} \rightarrow \mathrm{Fil}^{\lambda_1 + 1}\mathcal{O}\mathbb{B}^{+}_{\mathrm{dR}} \otimes_{\mathcal{O}_{K^p}} \omega_{\tau_1, K^p}^{-\lambda_1 - 1}$. \end{proof}

\textbf{Proof \ of \ Proposition \ \ref{zero map}}

We consider the map $\mathrm{Sym}^{\lambda_1 + 1} V_{\tau_1} \otimes_{L} \mathcal{O}\mathbb{B}_{\mathrm{dR}, \lambda_2 - \lambda_1 + 1}^+ \otimes_{\mathcal{O}_{V_0}} \omega_{\tau_1, V_0}^{\lambda_1 + 1} \rightarrow \mathrm{Fil}^{\lambda_1 + 1}\mathcal{O}\mathbb{B}_{\mathrm{dR}, \lambda_2 + 2}^+ \hookrightarrow \mathcal{O}\mathbb{B}_{\mathrm{dR}, \lambda_2 + 2}^+$ before (\ref{1}). By compositing $\nabla : \mathcal{O}\mathbb{B}_{\mathrm{dR}, \lambda_2 + 2}^+ \rightarrow \mathcal{O}\mathbb{B}_{\mathrm{dR}, \lambda_2 + 1}^+ \otimes_{\mathcal{O}_{V_0}} \Omega_{V_0}^1$ and taking $\otimes_{\mathbb{B}_{\mathrm{dR}}^+}\mathcal{O}_{K^p}$, we obtain the map $\oplus_{i=0}^{\lambda_2 - \lambda_1} (\mathrm{Sym}^{\lambda_1 + 1} V_{\tau_1} \otimes_{L} \mathcal{O}_{K^p} \otimes_{\mathcal{O}_{V_0}} \mathrm{Sym}^i \Omega_{V_0}^1 \otimes_{\mathcal{O}_{V_0}} \omega_{\tau_1, V_0}^{\lambda_1 + 1}) \rightarrow \oplus_{i = 0}^{\lambda_2} \mathcal{O}_{K^p} \otimes_{\mathcal{O}_{V_0}} \mathrm{Sym}^i\Omega_{V_0}^1 \otimes_{\mathcal{O}_{V_0}} \Omega_{V_0}^1$. Here, we use Lemma \ref{formal power}. By using the Hodge-Tate filtration, we have the $\mathcal{O}_{K^p} = \mathbb{B}_{\mathrm{dR}}^1$-submodule $\mathcal{O}_{K^p}(\lambda_1 + 1)$ of $(\mathrm{Sym}^{\lambda_1 + 1} V_{\tau_1} \otimes_{L} \mathcal{O}_{K^p} \otimes_{\mathcal{O}_{V_0}} \omega_{\tau_1, V_0}^{\lambda_1 + 1}) \subset \oplus_{i=0}^{\lambda_2 - \lambda_1} (\mathrm{Sym}^{\lambda_1 + 1} V_{\tau_1} \otimes_{L} \mathcal{O}_{K^p} \otimes_{\mathcal{O}_{V_0}} \mathrm{Sym}^i \Omega_{V_0}^1 \otimes_{\mathcal{O}_{V_0}} \omega_{\tau_1, V_0}^{\lambda_1 + 1})$. Thus Proposition \ref{zero map} holds if the restriction to $\mathcal{O}_{K^p}(\lambda_1 + 1)$ of the above map $\oplus_{i=0}^{\lambda_2 - \lambda_1} (\mathrm{Sym}^{\lambda_1 + 1} V_{\tau_1} \otimes_{L} \mathcal{O}_{K^p} \otimes_{\mathcal{O}_{V_0}} \mathrm{Sym}^i \Omega_{V_0}^1 \otimes_{\mathcal{O}_{V_0}} \omega_{\tau_1, V_0}^{\lambda_1 + 1}) \rightarrow \oplus_{i = 0}^{\lambda_2} \mathcal{O}_{K^p} \otimes_{\mathcal{O}_{V_0}} \mathrm{Sym}^i\Omega_{V_0}^1 \otimes_{\mathcal{O}_{V_0}} \Omega_{V_0}^1$ is zero because $N \supset D_{\lambda^{\Psi}, V_0} \otimes_{\mathcal{O}_{V_0}} E_0(t \mathcal{O}\mathbb{B}^{+, \Psi-\mathrm{la}, \chi_{\lambda_{\Psi}}, L}_{\mathrm{dR}, \lambda_2 + 1}) \otimes_{\mathcal{O}_{V_0}} \Omega_{V_0}^1$. This assumption follows from the following proposition. Let $\aS := \varprojlim_{K_1^p} \aS_{K_1^p}$ and $\pi_{\mathrm{HT}, \aS} : \aS \rightarrow \aS_{K^p} \rightarrow \Fl$ be the canonical map. We also write $\mathcal{O}_{\aS}$, $\omega_{\tau_1, \aS}$ and $D_{\tau_1, \aS}$ for the pushforward via $\pi_{\mathrm{HT}, \aS}$ of the pullback $\mathcal{O}_{\aS}$, $\omega_{\tau_1, \aS}$ and $D_{\tau_1, \aS}$ of $\mathcal{O}_{\aS_{K^p_1}}$, $\omega_{\tau_1, \aS_{K^p_1}}$ and $D_{\tau_1, \aS_{K^p_1}}$.

\begin{prop} Let $(i_{\tau})_{\tau} \in (\mathbb{Z}_{\ge 0})^{\Psi}$ such that $i_{\sigma} > 0$ for some $\sigma$.

Then any $\mathcal{O}_{\aS}$-linear map of $GU(\mathbb{A}_{\mathbb{Q}}^{\infty})$-equivariant sheaves $\mathcal{O}_{\aS} \rightarrow \otimes_{\tau \in \Psi} (\omega_{\tau_1, \aS}^2 \otimes \wedge^2D_{\tau_1, \aS})^{i_{\tau}}$ is zero.

\end{prop}

\begin{rem}

The author expects that we can prove Proposition \ref{zero map} without using this proposition by using purely $p$-adic geometric methods. In fact, if we can prove that the above map $\oplus_{i=0}^{\lambda_2 - \lambda_1} (\mathrm{Sym}^{\lambda_1 + 1} V_{\tau_1} \otimes_{L} \mathcal{O}_{K^p} \otimes_{\mathcal{O}_{V_0}} \mathrm{Sym}^i \Omega_{V_0}^1 \otimes_{\mathcal{O}_{V_0}} \omega_{\tau_1, V_0}^{\lambda_1 + 1}) \rightarrow \oplus_{i = 0}^{\lambda_2} \mathcal{O}_{K^p} \otimes_{\mathcal{O}_{V_0}} \mathrm{Sym}^i\Omega_{V_0}^1 \otimes_{\mathcal{O}_{V_0}} \Omega_{V_0}^1$ respects the graded pieces of both sides, then the above assumption follows from Proposition \ref{FHT}.

\end{rem}

\begin{proof}

The result says that $H^0(\Fl, \otimes_{\tau \in \Psi} (\omega_{\tau_1, \aS}^2 \otimes \wedge^2D_{\tau_1, \aS})^{i_{\tau}})^{GU(\mathbb{A}_{\mathbb{Q}}^{\infty})} = (\varinjlim_{K} H^0(\aS_{K}, \otimes_{\tau \in \Psi} (\omega_{\tau_1, \aS_{K}}^2 \otimes \wedge^2D_{\tau_1, \aS_{K}})^{i_{\tau}}))^{GU(\mathbb{A}_{\mathbb{Q}}^{\infty})} = (\varinjlim_{K} H^0(S_{K, \mathbb{C}}, \otimes_{\tau \in \Psi} (\omega_{\tau_1, S_{K, \mathbb{C}}}^2 \otimes \wedge^2D_{\tau_1, S_{K, \mathbb{C}}})^{i_{\tau}}))^{GU(\mathbb{A}_{\mathbb{Q}}^{\infty})} \otimes_{\mathbb{C}, \iota^{-1}} C = 0$. This follows from the fact that $H^0(S_{K, \mathbb{C}}, \otimes_{\tau \in \Psi} (\omega_{\tau_1, S_{K, \mathbb{C}}}^2 \otimes \wedge^2D_{\tau_1, S_{K, \mathbb{C}}})^{i_{\tau}})$ is decomposed by automorphic representations with a non-trivial central character determined by $(i_{\tau})_{\tau}$. (See \cite[Proposition 3.6, Theorem 5.3 and Proposition 4.3.2]{Harris}. Here, we used the assumption $i_{\sigma} > 0$ for some $\sigma$.) In fact, the weak approximation theorem (see \cite[Theorem 7.7]{PR}) implies that any automorphic representation contributing to $(\varinjlim_{K} H^0(S_{K, \mathbb{C}}, \otimes_{\tau \in \Psi} (\omega_{\tau_1, S_{K, \mathbb{C}}}^2 \otimes \wedge^2D_{\tau_1, S_{K, \mathbb{C}}})^{i_{\tau}}))^{GU(\mathbb{A}_{\mathbb{Q}}^{\infty})}$ has the trivial central character. \end{proof}

\section{Comparison of completed cohomologies of different unitary Shimura varieties}

Let $\mathfrak{m}$ be a decomposed generic non-Eisenstein ideal. For a moment, we assume that $p$ is inert in $F^+$ and $\Psi = \mathrm{Hom}_{F_0}(F, \overline{\mathbb{Q}}_p)$ for simplicity. Our purpose is to prove the classicality theorem, i.e., if $\varphi : \mathbb{T}^S(K^p, \mathcal{O})_{\mathfrak{m}} \rightarrow \mathcal{O}$ satisfies the condition that $\rho_{\varphi}$ and $\chi_{\varphi}$ are de Rham with $p$-adic Hodge type $0$ (for simplicity) and $\widehat{H}^d(S_{K^p}, L)_{\mathfrak{m}}[\varphi] \neq 0$, then $\varphi$ is actually a classical eigensystem. Note that we have $\widehat{H}^d(S_{K^p}, L)_{\mathfrak{m}}^{\mathrm{la}}[\varphi] \neq 0$ by Theorem \ref{density of locally analytic vectors}. As we have seen in Theorem \ref{induction step}, under appropriate technical assumptions on $\mathfrak{m}$ and assuming certain conjectures, we can prove that there exists $\tau_1 \in \Psi$ such that $\widehat{H}^d(S_{K^p}, L)_{\mathfrak{m}}^{\Psi \setminus \{ \tau_1 \}-\mathrm{la}}[\varphi] \neq 0$. If we believe that the cohomologies of Shimura varieties realize the Langlands correspondence, then the property $\widehat{H}^d(S_{K^p}, L)_{\mathfrak{m}}^{\Psi \setminus \{ \tau_1 \}-\mathrm{la}}[\varphi] \neq 0$ should be essentially independent of the choice of rank 2 unitary Shimura varieties. If this is actually true, by replacing $S_K$ by a rank 2 unitary Shimura variety $T_K$ associated with $\Psi \setminus \{ \tau_1 \}$ in the sense of {\S} 3.1, we obtain $\tau_2 \in \Psi \setminus\{ \tau_1 \}$ such that $\widehat{H}^{d-1}(T_{K^p}, L)_{\mathfrak{m}}^{\Psi \setminus \{ \tau_1, \tau_2 \}-\mathrm{la}}[\varphi] \neq 0$ by using Theorem \ref{induction step} again. By repeating this argument, we can get $\widehat{H}^{1}(T'_{K^p}, L)_{\mathfrak{m}}^{\mathrm{sm}}[\varphi] \neq 0$ for some unitary Shimura curve $T'_K$ and the classicality of $\varphi$. In this section, under technical assumptions on $\mathfrak{m}$, by using the works of \cite{CS}, \cite{kos} and \cite{zou} on mod $l$ cohomology of Shimura varieties, we will see that the above strategy actually works and complete the proof of our classicality theorem.

\subsection{Algebraic automorphic forms and completed cohomologies}

Note that the results of this subsection can be easily generalized to higher-rank unitary Shimura varieties.

\vspace{0.5 \baselineskip}

We consider the situations in {\S} 3.1 and 3.2. Let $\mathfrak{m}$ be a decomposed generic non-Eisenstein ideal of $\mathbb{T}^S$. Let $l \neq p$ be a prime splitting completely in $F$, $\mathbb{C} \Isom \overline{\mathbb{Q}}_l$ be an isomorphism of fields and $v_0 \mid l$ be a place of $F_0$ induced by this isomorphism. Let $\mathcal{M}_{K_{l}}$ be the Rapoport-Zink space corresponding to the basic element $b$ of $B(GU_{\mathbb{Q}_l}, \mu^{-1})$ of level $K_{l}$ over $\breve{M}$ for a sufficiently large finite subextension $M$ of $\overline{\mathbb{Q}}_l/\mathbb{Q}_l$, where $K_l$ is a pro-$l$ open compact subgroup of $G(\mathbb{Q}_l)$. We have a natural product decomposition $J_b(\mathbb{Q}_l) = F_{0, v_0^c}^{\times}  \times \prod_{w \mid v_0} J_b(\mathbb{Q}_l)_w$ such that $J_b(\mathbb{Q}_l)_w$ is $\mathrm{GL}_{2}(F_w)$ or $D_w^{\times}$ for some central division algebra $D_w$ over $F_w$ of dimension $4$. We assume that $\overline{\rho}_{\mathfrak{m}}|_{G_{F_w}}$ is irreducible and generic for any $w \mid l$ and if $J_b(\mathbb{Q}_l)_w = D_w^{\times}$, then the $w$-component of the minuscule $\mu$ is trivial. 

%Then $\mu$ induces the minuscule $\mathbb{G}_{m, \overline{\mathbb{Q}}_l} \rightarrow (\mathrm{Res}_{F/\mathbb{Q}}U)_{\overline{\mathbb{Q}}_l}$, for which we also write $\mu$. 

\begin{prop}\label{basic Mant}

We have the following natural isomorphism in $D(\mathcal{O}/\varpi^m[K \setminus GU(\mathbb{A}_{\mathbb{Q}}^{\infty})/K])$.   $$R\Gamma_{\et}(S_{K, \overline{F}}, \mathcal{O}/\varpi^m)_{\mathfrak{m}} \cong H^d_{\et}(S_{K, \overline{F}}, \mathcal{O}/\varpi^m)_{\mathfrak{m}}[-d] \cong \mathcal{A}_{GI}^{\mathrm{sm}}(K^{l}, \mathcal{O}/\varpi^{m})_{\mathfrak{m}} \otimes^{\mathbb{L}}_{\mathcal{H}(J_b(\mathbb{Q}_l))_{\mathcal{O}/\varpi^m}} R\Gamma_c(\mathcal{M}_{K_{l}}, \mathcal{O}/\varpi^m),$$  where $GI$ is a unitary similitude group satisfying $GI(\mathbb{A}_{\mathbb{Q}}^{\infty, l}) \cong GU(\mathbb{A}_{\mathbb{Q}}^{\infty, l})$, $GI(\mathbb{Q}_l) = \mathbb{Q}_l^{\times} \times J_b(\mathbb{Q}_l)$ and $GI(\mathbb{R}) = G(\prod_{\tau \in \Phi} U(0,2))$. We recall that $\mathcal{A}_{GI}^{\mathrm{sm}}(K^{l}, \mathcal{O}/\varpi^{m}) := \varinjlim_{K_l'} \mathcal{A}_{GI}(K^lK_l', \mathcal{O}/\varpi^m)$. (See the comment after Proposition \ref{de Rham uniformization}.)

\end{prop}

\begin{proof} The first isomorphism is obtained by \cite[Theorem 6.3.1]{CS} or \cite[Theorem 1.4]{kos}. We will show the second isomorphism.

By \cite[Theorem 7.1]{kos}, the Newton stratification induces a natural filtration on $R\Gamma_{\et}(S_{K, \overline{F}}, \mathcal{O}/\varpi^m)$ indexed by $b' \in B(GU_{\mathbb{Q}_l}, \mu^{-1})$, whose graded pieces are given by $R\Gamma_c(\mathrm{Igs}^{b'}_{K^l}, \mathcal{O}/\varpi^m) \otimes^{\mathbb{L}}_{\mathcal{H}(J_{b'}(\mathbb{Q}_l))_{\mathcal{O}/\varpi^m}} R\Gamma_c(\mathcal{M}_{(G,b',\mu), K_l}, \mathcal{O}/\varpi^m)$. Here, $\mathrm{Igs}^{b'}_{K^l}$ denotes the Igusa variety corresponding to $b'$ and $\mathcal{M}_{(G,b',\mu), K_l}$ denotes the Rapoport-Zink space corresponding to $(G, b', \mu)$ of level $K_l$. See \cite[{\S} 4]{CS} for details. Note that \cite[Theorem 7.1]{kos} assumed that the considering PEL Shimura varieties are unramified and our Shimura variety may not satisfy this assumption because $GU(\mathbb{Q}_l)$ may have a factor $D_w^{\times}$, where $D_w$ is a central division algebra over $\mathbb{Q}_l$. However we now assume that if $J_{b'}(\mathbb{Q}_l)_w = D_w^{\times}$, then the $w$-component of the minuscule $\mu$ is trivial. Thus the integral model as in \cite[{\S} 5]{Kottwitz} is smooth and the same proof as \cite[Theorem 7.1]{kos} works in our situation. We have $R\Gamma_c(\mathrm{Igs}^{b}_{K^l}, \mathcal{O}/\varpi^m) = \mathcal{A}_{GI}^{\mathrm{sm}}(K^{l}, \mathcal{O}/\varpi^{m})$ for the basic element $b$ by \cite[Proposition 11.27]{Igusa1}. 

By localizing at $\mathfrak{m}$, we obtain the filtration on $H^d(S_{K, \overline{F}}, \mathcal{O}/\varpi^m)_{\mathfrak{m}}[-d]$ indexed by $b' \in B(GU_{\mathbb{Q}_l}, \mu^{-1})$, whose graded pieces are given by $$R\Gamma_c(\mathrm{Igs}^{b'}_{K^l}, \mathcal{O}/\varpi^m)_{\mathfrak{m}} \otimes^{\mathbb{L}}_{\mathcal{H}(J_{b'}(\mathbb{Q}_l))_{\mathcal{O}/\varpi^m}} R\Gamma_c(\mathcal{M}_{(G,b',\mu), K_l}, \mathcal{O}/\varpi^m).$$ Let $Z^{\mathrm{spec}}((\mathrm{Res}_{F^+/\mathbb{Q}}U)_{\mathbb{Q}_l}, \mathcal{O}) = \prod_{w \mid v_0} Z^{\mathrm{spec}}(\mathrm{GL}_{2, F_w}, \mathcal{O})$ be the spectral Bernstein center defined in \cite[Definition I.9.2]{GLLC} and let $\mathfrak{n}$ be the maximal ideal of this corresponding to $\prod_{w \mid v_0} \overline{\rho}_{\mathfrak{m}}|_{G_{F_w}}$. 

Note that by the local-global compatibility Theorem \ref{Galois representation} and Theorem \ref{torsion vanishing}, the cohomology group $H^d(S_{K, \overline{F}}, \mathcal{O}/\varpi^m)_{\mathfrak{m}}$ is equal to its localization at $\mathfrak{n}$. - (*)

By localizing the above filtration at $\mathfrak{n}$, we obtain the filtration on $H^d(S_{K, \overline{F}}, \mathcal{O}/\varpi^m)_{\mathfrak{m}}[-d]$ indexed by $b' \in B((\mathrm{Res}_{F^+/\mathbb{Q}}U)_{\mathbb{Q}_l}, \mu^{-1})$, whose graded pieces are given by $R\Gamma_c(\mathrm{Igs}^{b'}_{K^l}, \mathcal{O}/\varpi^m)_{\mathfrak{m}} \otimes^{\mathbb{L}}_{\mathcal{H}(J_{b'}(\mathbb{Q}_l))_{\mathcal{O}/\varpi^m}} R\Gamma_c(\mathcal{M}_{(G,b',\mu), K_l}, \mathcal{O}/\varpi^m)_{\mathfrak{n}}$. By \cite[proof of Theorem IX.3.1]{GLLC}, we have $R\Gamma_c(\mathcal{M}_{(G,b',\mu), K_l}, \mathcal{O}/\varpi^m) = i^{b'*}T_{\mu}(j^1_!(c\mathrm{-}\mathrm{Ind}^{G(\mathbb{Q}_p)}_{K_l}\mathcal{O}/\varpi^m))$. (See \cite[proof of Theorem IX.3.1]{GLLC} for the definitions of these notations.) By \cite[Theorems IX.5.2 and IX.7.2]{GLLC}, the action of $Z^{\mathrm{spec}}((\mathrm{Res}_{F^+/\mathbb{Q}}U)_{\mathbb{Q}_l}, {\mathcal{O}})$ on $R\Gamma_c(\mathcal{M}_{(G,b',\mu), K_l}, \mathcal{O}/\varpi^m)$ factors through $Z^{\mathrm{spec}}((\mathrm{Res}_{F^+/\mathbb{Q}}U)_{\mathbb{Q}_l}, \mathcal{O}) \rightarrow Z^{\mathrm{spec}}(J_{b'}, \mathcal{O})$. Since $J_{b'}$ is an inner form of a proper Levi subgroup of $G$ for non-basic $b' \in B((\mathrm{Res}_{F^+/\mathbb{Q}}U)_{\mathbb{Q}_l}, \mu^{-1}) = B(GU_{\mathbb{Q}_l}, \mu^{-1})$ and $\overline{\rho}_{\mathfrak{m}}|_{G_{F_w}}$ is irreducible for any $w \mid v_0$, the maximal ideal $\mathfrak{n}$ is not contained in the image of $\mathrm{Spec} Z^{\mathrm{spec}}(J_{b'}, \mathcal{O}) \rightarrow \mathrm{Spec} Z^{\mathrm{spec}}((\mathrm{Res}_{F^+/\mathbb{Q}}U)_{\mathbb{Q}_l}, \mathcal{O})$. Thus $R\Gamma_c(\mathcal{M}_{(G,b',\mu), K_l}, \mathcal{O}/\varpi^m)_{\mathfrak{n}} = 0$ for non-basic $b' \in B(GU_{\mathbb{Q}_l}, \mu^{-1})$. This implies $$H^d_{\et}(S_{K, \overline{F}}, \mathcal{O}/\varpi^m)_{\mathfrak{m}}[-d] \cong \mathcal{A}_{GI}^{\mathrm{sm}}(K^{l}K_{l, 0}, \mathcal{O}/\varpi^{m})_{\mathfrak{m}} \otimes^{\mathbb{L}}_{\mathcal{H}(J_{b}(\mathbb{Q}_l))_{\mathcal{O}/\varpi^m}} R\Gamma_c(\mathcal{M}_{K_{l}}, \mathcal{O}/\varpi^m)_{\mathfrak{n}}.$$ This is equal to $$\mathcal{A}_{GI}^{\mathrm{sm}}(K^{l}K_{l, 0}, \mathcal{O}/\varpi^{m})_{\mathfrak{m}, \mathfrak{n}} \otimes^{\mathbb{L}}_{\mathcal{H}(J_{b}(\mathbb{Q}_l))_{\mathcal{O}/\varpi^m}} R\Gamma_c(\mathcal{M}_{K_{l}}, \mathcal{O}/\varpi^m)$$ because we have the isomorphism $Z^{\mathrm{spec}}((\mathrm{Res}_{F^+/\mathbb{Q}}U)_{\mathbb{Q}_l}, \mathcal{O}) \stackrel{\sim}{\rightarrow} Z^{\mathrm{spec}}(J_{b}, \mathcal{O})$ compatible with the actions on $R\Gamma_c(\mathcal{M}_{K_{l}}, \mathcal{O}/\varpi^m) = i^{b*}T_{\mu}(j^1_!(c\mathrm{-}\mathrm{Ind}^{G(\mathbb{Q}_p)}_{K_l}\mathcal{O}/\varpi^m))$ for the basic element $b$ as we have seen above. (See \cite[Definition I.9.2]{GLLC} and \cite[Theorems IX.5.2 and IX.7.2]{GLLC} again.) This is also equal to $\mathcal{A}_{GI}^{\mathrm{sm}}(K^{l}, \mathcal{O}/\varpi^{m})_{\mathfrak{m}} \otimes^{\mathbb{L}}_{\mathcal{H}(J_b(\mathbb{Q}_l))_{\mathcal{O}/\varpi^m}} R\Gamma_c(\mathcal{M}_{K_{l}}, \mathcal{O}/\varpi^m)$ because we have $\mathcal{A}_{GI}^{\mathrm{sm}}(K^{l}, \mathcal{O}/\varpi^{m})_{\mathfrak{m}, \mathfrak{n}} = \mathcal{A}_{GI}^{\mathrm{sm}}(K^{l}, \mathcal{O}/\varpi^{m})_{\mathfrak{m}}$ by the same reason as above (*). \end{proof}

For $w \mid v_0$, let $\overline{\pi}_w$ be the irreducible supercuspidal representation of $J_b(\mathbb{Q}_l)_w$ over $\mathbb{F}$ corresponding to $\overline{\rho}_{\mathfrak{m}}|_{G_{F_w}}$ via the mod $l$ Langlands correspondence (see \cite[Theorem 1.6]{Vigneras} and \cite[Theorem (1.2.4)]{Dat} for details). We can take an admissible lift of $\pi_w$ over $\mathcal{O}$ of $\overline{\pi}_w$ by \cite[Theorem 3.27 and Remark 3.28]{liftable}. Let $R_w$ be the universal deformation ring of $\overline{\rho}_{\mathfrak{m}}|_{G_{F_w}}$ over $\mathcal{O}$ and $\chi_w^{\mathrm{univ}} : G_{F_w} \rightarrow \mathcal{O}[[T]]^{\times}$ be the unramified character defined by $\mathrm{Frob}_w \mapsto 1 + T$. Note that we can consider the deformation ring by the irreducibility of $\overline{\rho}_{\mathfrak{m}}$. Note that $\mathrm{rec}_{F_{w}}(\pi_{w}|\mathrm{det}|_{w}^{-\frac{1}{2}})$ can be defined over $\mathcal{O}$ (again see \cite[Theorem 1.6]{Vigneras} and \cite[Theorem (1.2.4)]{Dat} or \cite[Theorem I.9.6]{GLLC} and \cite[Theorem 1.0.3]{HKW}) and this uniquely extends to a continuous representation $G_{F_w} \rightarrow \mathrm{GL}_2(\mathcal{O})$, which we also write for $\mathrm{rec}_{F_{w}}(\pi_{w}|\mathrm{det}|_{w}^{-\frac{1}{2}})$. 

\begin{lem}\label{formally smooth deformation}

1 \ The map $R_{w} \rightarrow \mathcal{O}[[T]]$ over $\mathcal{O}$ corresponding to the deformation $\chi_{w}^{\mathrm{univ}} \otimes_{\mathcal{O}} \mathrm{rec}_{F_{w}}(\pi_{w}|\mathrm{det}|_{w}^{-\frac{1}{2}})$ of $\overline{\rho}_{\mathfrak{m}}|_{G_{F_w}}$ is an isomorphism.

2 \  Let $R_{0}^{\mathrm{univ}}$ denote the universal deformation ring of $\overline{\chi}_{\mathfrak{m}}|_{G_{F_{0, v_0^c}}}$ over $\mathcal{O}$. Then $R_{0}^{\mathrm{univ}}$ is isomorphic to $\mathcal{O}[[X]]$ over $\mathcal{O}$.  

\end{lem}

\begin{proof}

1 \ By the genericity assumption $H^2(G_{F_w}, \mathrm{ad}\overline{\rho}_{\mathfrak{m}}) = H^0(G_{F_w}, \mathrm{ad}\overline{\rho}_{\mathfrak{m}}(1))^{\vee} = 0$ and the local Euler-Poincare characteristic formula, the natural inclusion $H^1_{\mathrm{ur}}(G_{F_w}, \mathrm{ad}\overline{\rho}_{\mathfrak{m}}|_{G_{F_w}}) \hookrightarrow H^1(G_{F_w}, \mathrm{ad}\overline{\rho}_{\mathfrak{m}}|_{G_{F_w}})$ is an isomorphism. By the irreducibility of $\overline{\rho}_{\mathfrak{m}}|_{G_{F_w}}$, we have $\mathrm{dim}_{\mathbb{F}}H^1_{\mathrm{ur}}(G_{F_w}, \mathrm{ad}\overline{\rho}_{\mathfrak{m}}|_{G_{F_w}}) = \mathrm{dim}_{\mathbb{F}}H^1(G_{F_w}/I_{F_w}, \mathrm{ad}\overline{\rho}_{\mathfrak{m}}^{I_{F_w}}|_{G_{F_w}}) = \mathrm{dim}_{\mathbb{F}}H^0(G_{F_w}, \mathrm{ad}\overline{\rho}_{\mathfrak{m}}|_{G_{F_w}}) = 1$. We claim that these result implies that $R_{w}$ is isomorphic to $\mathcal{O}[[T]]$ by the given map. Actually, for any surjection $\psi : R \twoheadrightarrow R'$ of Artin local rings over $\mathcal{O}$ with the kernel $I$ satisfying $I\mathfrak{m}_{R} = 0$ and any local map $\varphi : R_{w} \rightarrow R$ over $\mathcal{O}$, the deformations of $\psi \circ \varphi \circ (\chi_{w}^{\mathrm{univ}} \otimes_{\mathcal{O}} \mathrm{rec}_{F_{w}}(\pi_{w}|\mathrm{det}|_{w}^{-\frac{1}{2}}))$ to $R$ are parametrized by elements of $H^1(G_{F_w}, \mathrm{ad}\overline{\rho}_{\mathfrak{m}}|_{G_{F_w}}) \otimes_{\mathbb{F}} I \cong I$, which implies that any deformations to $R$ of $\overline{\rho}|_{G_{F_w}}$ is equal to an unramified twist of $\varphi \circ (\chi_{w}^{\mathrm{univ}} \otimes_{\mathcal{O}} \mathrm{rec}_{F_{w}}(\pi_{w}|\mathrm{det}|_{w}^{-\frac{1}{2}}))$.

2 \ This follows from the same (or easier) calculation as 1. \end{proof}

Let $\chi_{0}^{\mathrm{univ}} : G_{F_{0, v_0^c}} \rightarrow R_{0}^{\mathrm{univ}, \times}$ be the universal character, $R_l := R_0^{\mathrm{univ}} \widehat{\otimes} (\widehat{\otimes}_{w \mid v_0} R_w)$, $\chi_{l}^{\mathrm{univ}} := \chi_0^{\mathrm{univ}} \circ \mathrm{Art}_{F_{v_0^c}} \times \prod_{w \mid v_0} \chi_w^{\mathrm{univ}} \circ \mathrm{Art}_{F_w} : F_{v_0^c}^{\times} \times \prod_{w \mid v_0}F_w^{\times} \rightarrow R_l^{\times}$ and $\pi_l := \boxtimes_{w \mid l} \pi_w$ and $\pi_l^{\mathrm{univ}} := \pi_l \otimes_{\mathcal{O}} \chi_{l}^{\mathrm{univ}}$. Let $\mathbb{T}^S_{GU}(K, \mathcal{O})_{\mathfrak{m}} := \mathrm{Im}(\mathbb{T}^S \rightarrow \mathrm{End}_{\mathcal{O}}(H^d_{\et}(S_{K, \overline{E}}, \mathcal{O})_{\mathfrak{m}}))$ and $\mathbb{T}^S_{GI}(K^{l}K_{l}', \mathcal{O})_{\mathfrak{m}} := \mathrm{Im}(\mathbb{T}^S \rightarrow \mathrm{End}_{\mathcal{O}}(\mathcal{A}_{GI}(K^{l}K_{l}', \mathcal{O})_{\mathfrak{m}}))$.

Note that by constructing the Hecke algebra valued Galois representations as in Theorem \ref{Galois completed cohomology}, we have natural maps $R_{l} \rightarrow \mathbb{T}^S_{GU}(K, \mathcal{O})_{\mathfrak{m}}, \mathbb{T}^S_{GI}(K^{l}K_{l}', \mathcal{O})_{\mathfrak{m}}$. Thus $R_{l}$ naturally acts on $H^d_{\et}(S_{K, \overline{F}}, \mathcal{O}/\varpi^m)_{\mathfrak{m}}$ and $\mathcal{A}_{GI}^{\mathrm{sm}}(K^{l}, \mathcal{O}/\varpi^{m})_{\mathfrak{m}}$.

\begin{lem}\label{decomposition} Let $\Lambda := \mathcal{O}/\varpi^m, \mathcal{O}$ or $E$.

The evaluation map $\mathrm{ev}_{\Lambda} : \mathrm{Hom}_{R_l[J_b(\mathbb{Q}_l)]}(\pi^{\mathrm{univ}}_l, \mathcal{A}_{GI}^{\mathrm{sm}}(K^{l}, \Lambda)_{\mathfrak{m}}) \otimes_{R_l} \pi^{\mathrm{univ}}_l \rightarrow \mathcal{A}_{GI}^{\mathrm{sm}}(K^{l}, \Lambda)_{\mathfrak{m}}$ is an isomorphism.

\end{lem}
    
\begin{proof} This can be proved in the same way as Proposition \ref{residual irreducibility}. By the irreducibility of $\overline{\pi}_{l}^{\mathrm{univ}}$, we obtain the injectivity of $\mathrm{ev}_{\mathbb{F}}$ on the $\mathfrak{m}$-annihilated subspaces. Thus we obtain the injectivity of $\mathrm{ev}_{\mathbb{F}}$ because $\mathfrak{m}$ acts on $\mathrm{Ker}(\mathrm{ev}_{\mathbb{F}})$ nilpotently. We claim that $\mathrm{ev}_{\mathcal{O}}$ and $\mathrm{ev}_{\mathcal{O}/\varpi^m}$ are injective and the cokernel of $\mathrm{ev}_{\mathcal{O}}$ is torsion free over $\mathcal{O}$. We can check this claim by taking the fixed part for any pro-$l$ open compact subgroup $K_l$. Note that the admissibility of both spaces implies the $p$-completeness of the fixed parts. Thus the injectivity of $\mathrm{ev}_{\mathbb{F}}$ and $p$-torsion freeness of both spaces imply the claim. By using the decomposition of $\mathcal{A}_{GI}^{\mathrm{sm}}(K^{l}, E)_{\mathfrak{m}}$ by automorphic representations (see Theorem \ref{Zuker}) and the compatibility (see Theorem \ref{Galois representation}), we obtain the surjectivity of $\mathrm{ev}_{E}$ and thus the surjectivity of $\mathrm{ev}_{\mathcal{O}}$ and $\mathrm{ev}_{\mathcal{O}/\varpi^m}$. \end{proof}

By Proposition \ref{basic Mant} and Lemma \ref{decomposition}, we get  

\begin{align}\label{Mantdecomposition}
H^d(S_{K, \overline{F}}, \mathcal{O}/\varpi^m)_{\mathfrak{m}} \nonumber \\
\cong (\mathrm{Hom}_{R_l[J_b(\mathbb{Q}_l)]}(\pi^{\mathrm{univ}}_l, \mathcal{A}_{GI}^{\mathrm{sm}}(K^{l}, \mathcal{O}/\varpi^m)_{\mathfrak{m}}) \otimes_{R_l} \pi^{\mathrm{univ}}_l) \otimes^{\mathbb{L}}_{\mathcal{H}(J_b(\mathbb{Q}_l))_{\mathcal{O}/\varpi^m}} R\Gamma_c(\mathcal{M}_{K_{l}}, \mathcal{O}/\varpi^m)[d] \nonumber \\
= (\mathrm{Hom}_{R_l[J_b(\mathbb{Q}_l)]}(\pi^{\mathrm{univ}}_l, \mathcal{A}_{GI}^{\mathrm{sm}}(K^{l}, \mathcal{O}/\varpi^m)_{\mathfrak{m}}) \otimes_{R_l}^{\mathbb{L}} \pi^{\mathrm{univ}}_l) \otimes^{\mathbb{L}}_{\mathcal{H}(J_b(\mathbb{Q}_l))_{\mathcal{O}/\varpi^m}} R\Gamma_c(\mathcal{M}_{K_{l}}, \mathcal{O}/\varpi^m)[d] \nonumber \\
= \mathrm{Hom}_{R_l[J_b(\mathbb{Q}_l)]}(\pi^{\mathrm{univ}}_l, \mathcal{A}_{GI}^{\mathrm{sm}}(K^{l}, \mathcal{O}/\varpi^m)_{\mathfrak{m}}) \otimes^{\mathbb{L}}_{R_l} (\pi^{\mathrm{univ}}_l \otimes^{\mathbb{L}}_{\mathcal{H}(J_b(\mathbb{Q}_l))_{\mathcal{O}}} R\Gamma_c(\mathcal{M}_{K_{l}}, \mathcal{O})[d])
\end{align}

\begin{thm}\label{concentration}

$\pi^{\mathrm{univ}}_l \otimes^{\mathbb{L}}_{\mathcal{H}(J_b(\mathbb{Q}_l))_{\mathcal{O}}} R\Gamma_c(\mathcal{M}_{K_l}, \mathcal{O})[d]$ is concentrated at $0$ and a finite free $R_l$-module. Moreover, this is nonzero for sufficiently small $K_l$.

\end{thm}

\begin{proof} By \cite[Theorem IX.3.1]{GLLC} and \cite[Proposition 2.9]{LRC}, the complex $\pi^{\mathrm{univ}}_l \otimes^{\mathbb{L}}_{\mathcal{H}(J_b(\mathbb{Q}_l))_{\mathcal{O}}} R\Gamma_c(\mathcal{M}_{K_l}, \mathcal{O})[d]$ is a bounded above complex of finite free $R_l$-modules. The torsion Kottwiz conjecture \cite[Theorem 8.10]{zou} implies that $\mathbb{F} \otimes^{\mathbb{L}}_{R_l} \pi^{\mathrm{univ}}_l \otimes^{\mathbb{L}}_{\mathcal{H}(J_b(\mathbb{Q}_l))_{\mathcal{O}}} R\Gamma_c(\mathcal{M}_{K_l}, \mathcal{O})[d] = \overline{\pi}^{\mathrm{univ}}_l \otimes^{\mathbb{L}}_{\mathcal{H}(J_b(\mathbb{Q}_l))_{\mathbb{F}}} R\Gamma_c(\mathcal{M}_{K_l}, \mathbb{F})[d]$ is concentrated at degree $0$ and if $K_l$ is sufficiently small, then this is nonzero. Note that \cite[Theorem 8.10]{zou} state the Torsion Kottwitz conjecture only for $\mathrm{Res}_{L/\mathbb{Q}_p}\mathrm{GL}_{n, L}$, where $L$ is a finite extension of $\mathbb{Q}_p$. However, we now assume that if $J_b(\mathbb{Q}_l)_w = D_w^{\times}$, then the $w$-component of the minuscule $\mu$ is trivial. Note this implies that $G(\mathbb{Q}_l)_w = D_w^{\times}$. Thus the torsion Kottwiz conjecture is trivial in this case because the corresponding Rapoport-Zink space of level $K_w$ is simply a set $D_w^{\times}/K_w$ with a trivial Weil descent datum, canonical action of $J_b(\mathbb{Q}_l)_w = D_w^{\times}$ and a canonical Hecke action of $G(\mathbb{Q}_l)_w = D_w^{\times}$. This implies the result. \end{proof}

Let $\widehat{\mathcal{A}}_{GI}(K^{p}, \mathcal{O})_{\mathfrak{m}} := \varprojlim_{m} \varinjlim_{K_p} \mathcal{A}_{GI}(K^{p}K_p, \mathcal{O}/\varpi^m)_{\mathfrak{m}}$.

\begin{thm} \label{comparison1}

We have the following isomorphism of $GU(\mathbb{Q}_p) \times \mathbb{T}^S$-modules. $$\widehat{H}^d(S_{K^p}, \mathcal{O})_{\mathfrak{m}} \cong \mathrm{Hom}_{R_l[J_b(\mathbb{Q}_l)]}(\pi_l^{\mathrm{univ}}, \varinjlim_{K_l'}\widehat{\mathcal{A}}_{GI}(K^{p, l}K_l', \mathcal{O})_{\mathfrak{m}}) \otimes_{R_l} (\pi_{l}^{\mathrm{univ}} \otimes^{\mathbb{L}}_{\mathcal{H}(J_b(\mathbb{Q}_l))_{\mathcal{O}}} R\Gamma_{c}(\mathcal{M}_{K_{l}}, \mathcal{O})[d]).$$

Consequently, for sufficiently small $K_l$, there exists $k \in \mathbb{Z}_{>0}$ such that 

$\widehat{H}^d(S_{K^p}, \mathcal{O})_{\mathfrak{m}} \cong \mathrm{Hom}_{R_l[J_b(\mathbb{Q}_l)]}(\pi_l^{\mathrm{univ}}, \varinjlim_{K_l'}\widehat{\mathcal{A}}_{GI}(K^{p, l}K_l', \mathcal{O})_{\mathfrak{m}})^{\oplus k}$.

\end{thm}

\begin{proof} By Theorem \ref{concentration}, (\ref{Mantdecomposition}) and taking the limit $K_p \rightarrow 1$ and $m \rightarrow \infty$, we obtain $$\widehat{H}^d(S_{K^p}, \mathcal{O})_{\mathfrak{m}} \cong \mathrm{Hom}_{R_l[J_b(\mathbb{Q}_l)]}(\pi_l^{\mathrm{univ}}, \varprojlim_{m}\varinjlim_{K_l'}\widehat{\mathcal{A}}_{GI}(K^{p, l}K_l', \mathcal{O}/\varpi^m)_{\mathfrak{m}}) \otimes_{R_l} (\pi_{l}^{\mathrm{univ}} \otimes^{\mathbb{L}}_{\mathcal{H}(J_b(\mathbb{Q}_l))_{\mathcal{O}}} R\Gamma_{c}(\mathcal{M}_{K_{l}}, \mathcal{O})[d]).$$ Since the action of $J_b(\mathbb{Q}_l)$ on $\pi_l^{\mathrm{univ}}$ is smooth, we obtain $\mathrm{Hom}_{R_l[J_b(\mathbb{Q}_l)]}(\pi_l^{\mathrm{univ}}, \varinjlim_{K_l'}\widehat{\mathcal{A}}_{GI}(K^{p, l}K_l', \mathcal{O})_{\mathfrak{m}}) = \mathrm{Hom}_{R_l[J_b(\mathbb{Q}_l)]}(\pi_l^{\mathrm{univ}}, \varprojlim_{m}\varinjlim_{K_l'}\widehat{\mathcal{A}}_{GI}(K^{p, l}K_l', \mathcal{O}/\varpi^m)_{\mathfrak{m}})$. Theorem \ref{concentration} implies the second statement. \end{proof}

We fix $\tau \in \Psi$ and let $w_0 \mid v_0$ be a place $F$ induced by $\tau$. Assume that $\Psi \neq \{ \tau \}$ and $J_b(\mathbb{Q}_l)_{w_0} = D_{w_0}^{\times}$. Let $U'/F^+$ be a unitary group defined by the data $\Psi \setminus \{ \tau \}$ and $S(B) \sqcup \{ w_0, w_0^c \}$ as in {\S} 3.1. Thus the $w_0$-component of $GU'(\mathbb{Q}_l)$ is equal to $D_{w_0}^{\times}$. Let $T_{K}$ be the Shimura variety over $F$ defined by $GU'$ as in {\S} 3.2. Let $\mu': \mathbb{G}_{m, \overline{\mathbb{Q}}_l} \rightarrow GU'_{\overline{\mathbb{Q}}_l}$ be the minuscule defined by $GU'$ and $b'$ be the basic element of $B(GU'_{\mathbb{Q}_l}, \mu^{'-1})$. Then we have $J_b = J_{b'}$ by Lemma \ref{calculation of isocrystals}. Let $GI'$ be a unitary similitude group satisfying $GI'(\mathbb{A}_{\mathbb{Q}}^{\infty, l}) = GU'(\mathbb{A}_{\mathbb{Q}}^{\infty, l})$, $GI'(\mathbb{Q}_l) = J_{b'}(\mathbb{Q}_l)$ and $GI(\mathbb{R}) = G(\prod_{\tau \in \Phi} U(0,2))$. By Proposition \ref{unitary groups}, $GI$ and $GI'$ are identified by first taking an identification of division algebras and then taking an identification of involutions. Thus by Theorem \ref{comparison1}, we obtain the following comparison result.

\begin{cor}\label{comparison}
            
For any sufficiently small open compact subgroups $K_{l}$ and $K_{l}'$ of $GU(\mathbb{Q}_l)$ and $GU'(\mathbb{Q}_l)$ respectively, there exist positive integers $t, u \in \mathbb{Z}_{>0}$ and a $GU(\mathbb{Q}_p) \times \mathbb{T}^S$-equivariant isomorphism $$\widehat{H}^{d}(S_{K^{p, l}K_{l}}, \mathcal{O})^{\oplus t}_{\mathfrak{m}} \cong \widehat{H}^{d-1}(T_{K^{p, l}K'_{l}}, \mathcal{O})^{\oplus u}_{\mathfrak{m}}.$$

\end{cor}

\subsection{Proof of the classicality theorem}

We come back to the situation of {\S} 5.1. We recall that $v$ is the $p$-adic place of $F_0$ induced by the isomorphism $\iota : \overline{\mathbb{Q}}_p \Isom \mathbb{C}$. We fix $w \mid v$ and we assume $\Psi = \mathrm{Hom}_{\mathbb{Q}_p}(F_w, \overline{\mathbb{Q}}_p)$. Let $\mathfrak{m}$ be a decomposed generic non-Eisenstein ideal of $\mathbb{T}^S$ satisfying the following conditions.

\begin{itemize}
\item There exists a prime $l \neq p$ splitting completely in $F$, not lying below any place in $S(B)$ and $\overline{\rho}_{\mathfrak{m}}|_{G_{F_{w}}}$ is irreducible and generic for any $w \mid l$. 
\item $\overline{\rho}_{\mathfrak{m}}(G_{F})$ is not solvable.
\item $\otimes_{\tau \in \Psi} (\overline{\rho}_{\mathfrak{m}}|_{G_{\tilde{F}}})^{\tau}$ is absolutely irreducible. (We recall that $\tilde{F}$ denotes the Galois closure of $F$ over $\mathbb{Q}$. See comments before Theorem \ref{kottwitz conjecture} for the definition of $(\overline{\rho}_{\mathfrak{m}}|_{G_{\tilde{F}}})^{\tau}$.)
\end{itemize}

We fix $\overline{\mathbb{Q}}_l \Isom \mathbb{C}$. Let $v_0 \mid l$ be a place of $F_0$ induced by this and we have a natural injection $\Psi \hookrightarrow \{ w' : \mathrm{place \ of \ } F \mid w \mathrm{ \ divides  \ } v_0 \}, \ \tau \mapsto w_{\tau}$ by the assumption that $l$ splits completely in $F$. For any non-empty subset $\Psi' \subset \Psi$, let $S(B_{\Psi'}) = S(B) \sqcup \{ w_{\tau_i}, w_{\tau_i}^c \mid \tau_i \in \Psi \setminus \Psi' \}$. This induces a PEL datum $(B_{\Psi'}, *_{\Psi'}, V_{\Psi'}, \psi_{\Psi'}, h_{\Psi'})$ and thus a Shimura variety $S_{\Psi', K}$ defined by the PEL data $(B_{\Psi'}, *_{\Psi'}, V_{\Psi'}, \psi_{\Psi'}, h_{\Psi'})$ as in {\S} 3.1 and 3.2. We put $d_{\Psi'}:=|\Psi'|$. 

\begin{prop} \label{induction stepII}
    
Let $\Psi'$ be a subset of $\Psi$ such that $d_{\Psi'} \ge 2$.

Assume Conjectures \ref{classicality conjecture} and \ref{key diagram} for $S_{\Psi', K}$, $\mathfrak{m}$ and $\lambda$. Let $\varphi : \mathbb{T}^S(K^p, \mathcal{O})_{\mathfrak{m}} \rightarrow \mathcal{O}$ be an $\mathcal{O}$-morphism such that $\rho_{\varphi}|_{G_{F_{w'}}}$ is de Rham of $p$-adic Hodge-Tate type $\lambda_{w'}$ for any $w' \mid v$ and $\chi_{\varphi}|_{G_{F_0, v^c}}$ is de Rham of $p$-adic Hodge-Tate type $\lambda_0$. If we assume that $\widehat{H}^{d_{\Psi'}}(S_{\Psi', K^p}, V_{\lambda^{\Psi'}})_{\mathfrak{m}}^{\Psi'-\mathrm{la}}[\varphi] \neq 0$ and $\varphi$ is not a classical eigensystem of weight $\lambda$, then there exist $\tau \in \Psi'$ and $K_l'$ such that we obtain $\widehat{H}^{d_{\Psi' \setminus \{ \tau \}}}(S_{\Psi' \setminus \{ \tau \}, K_l'K^{l, p}}, V_{\lambda^{\Psi' \setminus \{ \tau \}}})^{\Psi' \setminus \{ \tau \}-\mathrm{la}}[\varphi] \neq 0$.

\end{prop}

\begin{proof}

By Theorem \ref{induction step}, there exists $\tau \in \Psi'$ such that $\widehat{H}^{d_{\Psi'}}(S_{\Psi', K^p}, V_{\lambda^{\Psi' \setminus \{ \tau \}}})_{\mathfrak{m}}^{{\Psi \setminus \{ \tau \}}-\mathrm{la}}[\varphi] \neq 0$. By Corollary \ref{comparison}, there exist positive integers $n, m$ and sufficiently small open compact subgroups $K_l, K_l'$ such that there exists a $GU(\mathbb{Q}_p) \times \mathbb{T}^S$-isomorphism $\widehat{H}^{d_{\Psi'}}(S_{\Psi', K_lK^{l, p}}, E)_{\mathfrak{m}}^{\oplus n} \cong \widehat{H}^{d_{\Psi' \setminus \{ \tau \}}}(S_{\Psi' \setminus \{ \tau \}, K_l'K^{l, p}}, E)_{\mathfrak{m}}^{\oplus m}$. In particular, $$(\widehat{H}^{d_{\Psi'}}(S_{\Psi', K_lK^{l, p}}, V_{\lambda^{\Psi' \setminus \{ \tau \}}})^{\Psi' \setminus \{ \tau \}-\mathrm{la}}_{\mathfrak{m}})^{\oplus n} \cong (\widehat{H}^{d_{\Psi' \setminus \{ \tau \}}}(S_{\Psi' \setminus \{ \tau \}, K_l'K^{l, p}}, V_{\lambda^{\Psi' \setminus \{ \tau \}}})^{\Psi' \setminus \{ \tau \}-\mathrm{la}}_{\mathfrak{m}})^{\oplus m}$$ and thus $\widehat{H}^{d_{\Psi' \setminus \{ \tau \}}}(S_{\Psi' \setminus \{ \tau \}, K_l'K^{l, p}}, V_{\lambda^{\Psi' \setminus \{ \tau \}}})^{\Psi' \setminus \{ \tau \}-\mathrm{la}}_{\mathfrak{m}}[\varphi] \neq 0$.  \end{proof}
 
We consider the following ``partially'' completed cohomologies.

$\widehat{H}^i(S_{K^w}, \mathcal{V}_{\lambda^w}) := \varprojlim_{n} \varinjlim_{K_w} H^i_{\et}(S_{K^wK_w, \overline{F}}, \mathcal{V}_{\lambda^w}/{\varpi}^n)$ and $\widehat{H}^i(S_{K^w}, V_{\lambda^w}) := \widehat{H}^i(S_{K^w}, \mathcal{V}_{\lambda^w})[\frac{1}{p}]$.

\begin{prop}\label{admissibility2}

$\widehat{H}^i(S_{K^w}, V_{\lambda^w})_{\mathfrak{m}}$ is an admissible representation of $\mathrm{GL}_2(F_w)$. 

\end{prop}

\begin{proof} We can directly prove this without localizing at $\mathfrak{m}$ by using the same argument as \cite[proof of Theorem 0.1]{emerton}. We can also deduce this from Proposition \ref{invariant part} by taking the $K_p^w$-invariant part of $\widehat{H}^d(S_{K^p}, \mathcal{V}_{\lambda^w})_{\mathfrak{m}}$. \end{proof}

\begin{thm} \label{conjectual classicality theorem} Assume Conjectures \ref{classicality conjecture} and \ref{key diagram} for $S_{\Psi', K}$, $\mathfrak{m}$ and $\lambda$ for any $\Psi' \subset \Psi$. Let $\varphi : \mathbb{T}^S(K^p, \mathcal{O})_{\mathfrak{m}} \rightarrow \mathcal{O}$ be an $\mathcal{O}$-morphism such that $\rho_{\varphi}|_{G_{F_{w'}}}$ is de Rham of $p$-adic Hodge-Tate type $\lambda_{w'}$ for any $w' \mid v$ and $\chi_{\varphi}|_{G_{F_0, v^c}}$ is de Rham of $p$-adic Hodge-Tate type $\lambda_0$. If $\widehat{H}^d(S_{\Psi, K^w}, V_{\lambda^w})_{\mathfrak{m}}[\varphi] \neq 0$, then $\varphi$ is a classical eigensystem of weight $\lambda$.

\end{thm}

\begin{proof} Assume that $\varphi$ is not a classical eigensystem of weight $\lambda$. By Theorem \ref{density of locally analytic vectors} and Proposition \ref{admissibility2}, we have $\widehat{H}^d(S_{\Psi, K^w}, V_{\lambda^{w}})_{\mathfrak{m}}^{\mathrm{la}}[\varphi] \neq 0$. In particular, we obtain $\widehat{H}^d(S_{\Psi, K^p}, V_{\lambda^{\Psi}})_{\mathfrak{m}}^{\Psi-\mathrm{la}}[\varphi] \neq 0$ by using Proposition \ref{invariant part}. By using Proposition \ref{induction stepII} repeatedly and using Theorem \ref{induction step}, we can find $\tau \in \Psi$ and $K_l'$ such that $\widehat{H}^1(S_{\{\tau \}, K^{p, l}K_l'}, V_{\lambda})_{\mathfrak{m}}^{\mathrm{sm}}[\varphi] = \varinjlim_{K_p} H^1(S_{\{\tau \}, K^{p, l}K_l'K_p}, V_{\lambda})_{\mathfrak{m}}[\varphi] \neq 0$. This implies that $\varphi$ is a classical eigensystem of weight $\lambda$ and this is a contradiction. \end{proof}

\begin{rem}

Note that unlike \cite[Theorems 7.3.2 and 7.3.7]{PanII}, our method doesn't give a nice description of $\widehat{H}^d(S_{\Psi, K^w}, V_{\lambda^w})_{\mathfrak{m}}^{\mathrm{la}}[\varphi]$ because we don't know the vanishing of $H^d(\Fl, Ker)$. (See {\S} 5.1 for the definition of $Ker$.) By using our method, we only know the description of $\widehat{H}^1(S_{\{ \tau \}, K^p}, V_{\lambda^{\tau}})_{\mathfrak{m}}^{\{ \tau \}-\mathrm{la}}[\varphi]$ for any $\tau \in \Psi$, which is very similar to \cite[Theorems 7.3.2 and 7.3.7]{PanII}.

\end{rem}

\begin{thm}\label{classicality theorem}

We assume that $[F_w : \mathbb{Q}_p] \le 2$ and $p$ is unramified in $F$.

Let $\varphi : \mathbb{T}^S(K^p, \mathcal{O})_{\mathfrak{m}} \rightarrow \mathcal{O}$ be an $\mathcal{O}$-morphism such that $\rho_{\varphi}|_{G_{F_{w'}}}$ is de Rham of $p$-adic Hodge-Tate type $\lambda_{w'}$ for any $w' \mid v$ and $\chi_{\varphi}|_{G_{F_0, v^c}}$ is de Rham of $p$-adic Hodge-Tate type $\lambda_0$. If we assume $\widehat{H}^d(S_{\Psi, K^w}, V_{\lambda^w})_{\mathfrak{m}}[\varphi] \neq 0$, then $\varphi$ is a classical eigensystem of weight $\lambda$.

\end{thm}

\begin{proof}

This follows from Theorem \ref{conjectual classicality theorem} because under our assumption, Conjectures \ref{classicality conjecture} and \ref{key diagram} are known by Corollary \ref{classicality of geometric}, Corollary \ref{parallel comparison} and Corollary \ref{non-parallel comparison}. \end{proof}

\section{Applications}

In this section, we give applications of the classicality theorem proved in the previous sections to the automorphy lifting theorem and the Breuil-M$\mathrm{\acute{e}}$zard conjecture. 

\subsection{Liftings of Galois representations with prescribed local properties}

Most parts of results in this subsection are contained in \cite{KP}, but we need to modify their result slightly for our purpose. Note that we can easily generalize the result of this section to higher-dimensional cases.

\vspace{0.5 \baselineskip}

We consider the following objects.

\begin{itemize}
    \item $p$ is an odd prime.
    \item $F^+$ is a totally real field.
    \item $F_0$ is an imaginary quadratic field.
    \item $F := F^+F_0$. We assume that $F$ is linearly disjoint from $\mathbb{Q}(\zeta_{p^{\infty}})$ over $\mathbb{Q}$ and assume that $F/F^+$ is unramified at all finite places.
    \item $\Phi := \mathrm{Hom}_{F_0}(F, \overline{\mathbb{Q}}_p)$.
    \item $E$ is a finite extension of $\mathbb{Q}_p$, $\mathcal{O}$ is the ring of integers, $\mathbb{F}$ is the residue field and $\varpi$ is a uniformizer of $\mathcal{O}$ such that $\varpi^2 \mid p$.
    \item $S$ is a finite set of finite places of $F^+$ and splitting in $F$ containing all $p$-adic places.
    \item For any $v \in S$, we fix a lift $\tilde{v}$ to $F$ of $v$ and let $\tilde{S} := \{ \tilde{v} \mid v \in S \}$. 
    \item $\chi : G_{F^+, S} \rightarrow \mathcal{O}^{\times}$ is a de Rham character such that $\chi(c_w) = -1$ for any $w \mid \infty$.
    \item $\psi : G_{F, S} \rightarrow \mathcal{O}^{\times}$ is a de Rham character such that $\psi\psi^c = \chi^{2}|_{G_{F, S}}$.
    \item $\overline{\rho} : G_{F, S} \rightarrow \mathrm{GL}_2(\mathbb{F})$ is a continuous representation such that there exists a perfect $G_{F, S}$-equivariant symmetric pairing $\overline{\rho} \times \overline{\rho}^c \rightarrow \overline{\chi}|_{G_{F}}$ and $\mathrm{det}\overline{\rho} = \overline{\psi}$.
\end{itemize}

First, we recall and modify some basic results of \cite{CHT}.

Let $\mathcal{G}_2$ be the group scheme over $\mathbb{Z}$ defined in \cite[{\S} 2]{CHT}. This is $\mathcal{G}_2 = (\mathrm{GL}_2 \times \mathrm{GL}_1) \rtimes \{ 1, \jmath \}$ defined by $\jmath (g, \mu)\jmath = (\mu ^t\!g^{-1}, \mu)$. Let $\nu : \mathcal{G}_2 \rightarrow \mathbb{G}_{m}$ be the morphism of group schemes defined by $(g, \mu) \mapsto \mu, \jmath \mapsto -1$ and $\mathcal{G}^0 := \mathrm{GL}_2 \times \mathrm{GL}_1$. Let $\mathrm{CNL}_{\mathcal{O}}$ be the category of complete Noetherian local rings whose residue field is $\mathbb{F}$ with local morphisms from $\mathcal{O}$. We have the following lemma.

\begin{lem}\label{correspondence}

For $R \in \mathrm{CNL}_{\mathcal{O}}$, we have the following bijection between the following two objects.

1 \ Continuous morphism $r : G_{F^+, S} \rightarrow \mathcal{G}_2(R)$ such that $\nu \circ r = \chi$, $r(c) \notin \mathcal{G}^0_2(R)$ and $r(G_{F, S}) \subset \mathcal{G}^0_2(R)$.

2 \ Continuous morphism $\rho : G_{F, S} \rightarrow \mathrm{GL}_2(R)$ with a perfect symmetric pairing $( \ , \ ) : \rho \times \rho^c \rightarrow \chi|_{G_{F, S}}$.

Explicitly, if we write $(x, y) = ^t\!xA^{-1}y$, then $(\rho, (\ , \ ))$ corresponds to $r$ such that $r(g) = (\rho(g), \chi(g))$ for any $g \in G_{F, S}$ and $r(c) = (A, -\chi(c))\jmath$.

\end{lem}

\begin{proof} See \cite[Lemma 2.1.]{CHT}. \end{proof}

In the following, we identify $r|_{G_{F, S}}$ with $G_{F, S} \xrightarrow{r|_{G_{F, S}}} \mathcal{G}_2^0(R) = \mathrm{GL}_2(R) \times \mathrm{GL}_1(R) \rightarrow \mathrm{GL}_2(R)$. Moreover, unless otherwise stated, we consider only continuous morphisms $r : G_{F^+, S} \rightarrow \mathcal{G}_2(R)$ satisfying the above condition 1 and $\mathrm{det}r|_{G_{F, S}} = \psi$. Let $\overline{r}$ be the continuous morphism corresponding to $\overline{\rho}$ with the pairing $\overline{\rho} \times \overline{\rho}^c \rightarrow \overline{\chi}$. Let $\mathrm{ad}^0\overline{r}$ denote the trace zero subrepresentation of $\mathrm{ad}\overline{r}$.

\begin{lem}\label{deformation theory}

Let $R,\ R' \in \mathrm{CNL}_{\mathcal{O}}$ be Artin local rings, $R' \twoheadrightarrow R$ be a local $\mathcal{O}$-morphism whose kernel $I$ satisfies $\mathfrak{m}_{R'}I = 0$ and $r : G_{F^+, S} \rightarrow \mathcal{G}_2(R)$ (resp. $\tilde{r} : G_{F^+, S} \rightarrow \mathcal{G}_2(R')$) be a lifting of $\overline{r}$ (resp. $r$) to $R$ (resp. $R'$).(As said above, we always assume $\mathrm{det}\tilde{r}|_{G_{F, S}} = \psi$.)

Then the map $H^1(G_{F^+, S}, \mathrm{ad}^0\overline{r}) \otimes_{\mathbb{F}} I \rightarrow \{ \mathrm{deformations} \ s \ \mathrm{of } \ r \ \mathrm{to} \ R' \ \mathrm{such \ that} \ \mathrm{det}s|_{G_{F}} = \psi \}, \ \varphi \mapsto (I_2 + \varphi)\tilde{r}$ is bijective.

\end{lem}

\begin{proof} Same as \cite[proof of Proposition 2.2.9]{CHT}. \end{proof}

\begin{lem}\label{obstruction theory}

Let $R,\ R' \in \mathrm{CNL}_{\mathcal{O}}$ be Artin local rings over $\mathcal{O}$, $R' \twoheadrightarrow R$ be a local $\mathcal{O}$-morphism whose kernel $I$ satisfies $\mathfrak{m}_{R'}I = 0$ and $r : G_{F^+, S} \rightarrow \mathcal{G}_2(R)$ be a lifting of $\overline{r}$. 

1 \ We can take a continuous map $\tilde{r} : G_{F^+, S} \rightarrow \mathcal{G}_2(R')$ (not necessarily group morphism) which is a lift of $r$ such that $\mathrm{det}\tilde{r}|_{G_{F, S}} = \psi$, $\nu \circ \tilde{r} = \chi$ and $\tilde{r}(\gamma c) = \tilde{r}(\gamma)\tilde{r}(c)$ for any $\gamma \in G_{F, S}$.

2 \ Let $\tilde{r} : G_{F^+, S} \rightarrow \mathcal{G}_2(R')$ be a continuous map as in the above 1 and we put $\phi(\gamma, \delta) := \tilde{r}(\gamma \delta) \tilde{r}(\delta)^{-1} \tilde{r}(\gamma)^{-1} - I_2$.

Then $\phi$ is a continuous 2-cocycle $G_{F^+, S} \times G_{F^+, S} \rightarrow \mathrm{ad}^0\overline{r} \otimes_{\mathbb{F}} I$ and the cohomology class $\mathrm{obs}(r, R')$ of $\phi$ is independent of the choice of $\tilde{r}$. Moreover, $\mathrm{obs}(r, R') = 0$ is equivalent to the existence of a lifting $r' : G_{F^+, S} \rightarrow \mathcal{G}_2(R)$ of $r$ such that $\mathrm{det}r'|_{G_{F, S}} = \psi$ and $\nu \circ r' = \chi^{-1}$.

\end{lem}

\begin{rem}

We use the assumption $\psi \psi^c = \chi^2|_{G_{F, S}}$ only in this proof and after Corollary \ref{lifting}.

\end{rem}

\begin{proof}

1 \ By the smoothness of $\mathrm{det} : \mathrm{GL}_2 \rightarrow \mathbb{G}_m$, there exists a continuous map $\tilde{r}' : G_{F, S} \rightarrow \mathcal{G}_2(R')$ (not necessarily group morphism) which is a lift of $r|_{G_{F, S}}$ such that $\mathrm{det}\tilde{r}|_{G_{F, S}} = \psi$. By chosing a lift $\tilde{r}(c)$ of $r(c)$, we obtain the desired $\tilde{r}$.

2 \ Note that $\mathrm{det}(\tilde{r}(\gamma \delta) \tilde{r}(\delta)^{-1} \tilde{r}(\gamma)^{-1}) = \mathrm{det}(I_2 + \phi(\gamma, \delta)) = 1 + \mathrm{tr}\phi(\gamma, \delta)$. 

Therefore, the condition $\mathrm{Im} \phi \in \mathrm{ad}^0\overline{r} \otimes_{\mathbb{F}} I$ is equivalent to $\mathrm{det}(\tilde{r}(\gamma \delta) \tilde{r}(\delta)^{-1} \tilde{r}(\gamma)^{-1}) = 1$ for any $\gamma, \delta \in G_{F^+, S}$. In the following, we will check this. We put $\tilde{r}(c) = (A, -\chi(c))\jmath$.

(1) \ If $\gamma, \delta \in G_{F,S}$, this follows from $\mathrm{det}\tilde{r}(\delta) = \psi(\delta)$.

(2) \ Assume $\gamma = \gamma'c$ and $\gamma', \delta \in G_{F, S}$. 

\begin{align*}
\mathrm{det}(\tilde{r}(\gamma \delta) \tilde{r}(\delta)^{-1} \tilde{r}(\gamma)^{-1}) \\
 = \mathrm{det}(\tilde{r}(\gamma'c \delta c) \tilde{r}(c) \tilde{r}(\delta)^{-1} \tilde{r}(c)^{-1} \tilde{r}(\gamma')^{-1}) \\
= \mathrm{det}(\tilde{r}(\gamma'c \delta c) A ^t\!\tilde{r}(\delta)\chi(\delta)^{-1} A^{-1}\tilde{r}(\gamma')^{-1}) \\
= \mathrm{det}(\tilde{r}(c \delta c) ^t\!\tilde{r}(\delta)\chi(\delta)^{-1}) \\
=  \psi^c(\delta) \chi(\delta)^{-2} \psi(\delta) = 1. 
\end{align*}

(3) \ Assume $\delta = \delta'c$ and $\gamma, \delta' \in G_{F, S}$.

$\mathrm{det}(\tilde{r}(\gamma \delta) \tilde{r}(\delta)^{-1} \tilde{r}(\gamma)^{-1}) = \mathrm{det}\tilde{r}(\gamma\delta') \mathrm{det}\tilde{r}(\delta')^{-1} \mathrm{det}\tilde{r}(\gamma)^{-1} = 1$.

(4) \ Assume $\delta = \delta'c$, $\gamma = \gamma'c$ and $\gamma', \delta' \in G_{F, S}$.

\begin{align}\mathrm{det}(\tilde{r}(\gamma \delta) \tilde{r}(\delta)^{-1} \tilde{r}(\gamma)^{-1}) = \mathrm{det}(\tilde{r}(\gamma' c \delta' c) \tilde{r}(c)^{-1} \tilde{r}(\delta')^{-1} \tilde{r}(c)^{-1} \tilde{r}(\gamma')^{-1}) \nonumber \\ 
= \mathrm{det}(\tilde{r}(\gamma' c \delta' c) ^t\!A \chi(\delta')^{-1} { }^t\!\tilde{r}(\delta') A^{-1} \tilde{r}(\gamma')^{-1}) = \psi^c(\delta') \psi(\delta') \chi(\delta)^{-2} = 1.\end{align}

In order to prove that $\phi$ is a $2$-cocycle, it suffice to prove that $\phi : G_{F^+, S} \times G_{F^+, S} \rightarrow \mathrm{ad}^0\overline{r} \otimes_{\mathbb{F}} I \hookrightarrow \mathrm{ad}\overline{r} \otimes_{\mathbb{F}} I$ is a $2$-cocycle and this is \cite[Lemma 2.2.11]{CHT}.

Another lift can be written as $\tilde{r}' = (I_2 + c)\tilde{r}$ for some continuous map $c : G_{F^+, S} \rightarrow \mathrm{ad}^0\overline{r} \otimes_{\mathbb{F}} I$. Let $\phi'$ be the $2$-cocycle associated with $\tilde{r}'$. We obtain $\phi'(\gamma, \delta) = (I_2 + c(\gamma \delta))\tilde{r}(\gamma \delta) \tilde{r}(\delta)^{-1} (1 - c(\delta))\tilde{r}(\gamma)^{-1}(I_2 - c(\gamma)) - I_2 = \phi(\gamma, \delta) + c(\gamma \delta) - c(\gamma) - \mathrm{ad}r(\gamma) c(\delta)$. Therefore, the cohomology class of $\phi$ is independent of the choice of $\tilde{r}$. Moreover, the same calculation says that any 2-cocycle $\psi : G_{F^+, S} \times G_{F^+, S} \rightarrow \mathrm{ad}^0\overline{r} \otimes_{\mathbb{F}} I$ which is equivalent to $\phi$ comes from some lifting of $\tilde{r}''$ as in 1. This implies the final property of $\mathrm{obs}(r, R')$. \end{proof}

\begin{lem} (Greenberg-Wiles formula) \label{Greenberg-Wiles formula} Let $M$ be a finite $\mathbb{F}[G_{F^+, S}]$-module and $( \mathcal{L}_v )_{v \in S}$ be a local condition of $M$, i.e., $\mathcal{L}_v$ is a subspace of $H^1(F_v^+, M)$. Then we have the following formula. (We put $\mathcal{L}^{\perp}:=(\mathcal{L}_v^{\perp})_{v \in S}$ and $\mathcal{L}_v^{\perp} \subset H^1(F_v^+, M^{\vee}(1)) \cong H^1(F_v^+, M)^{\vee}$ denotes the annihilator of $\mathcal{L}_v$.)

\small

$\mathrm{dim}_{\mathbb{F}}H^1_{\mathcal{L}}(G_{F^+, S}, M) - \mathrm{dim}_{\mathbb{F}}H^1_{\mathcal{L}^{\perp}}(G_{F^+, S}, M^{\vee}(1)) = \mathrm{dim}_{\mathbb{F}}H^0(G_{F^+, S}, M) - \mathrm{dim}_{\mathbb{F}}H^0(G_{F^+, S}, M(1)) + \sum_{v \in S}(\mathrm{dim}_{\mathbb{F}}\mathcal{L}_v - \mathrm{dim}_{\mathbb{F}}H^0(F_v^+, M)) - \sum_{v \mid \infty} \mathrm{dim}_{\mathbb{F}}H^0(F_v^{+}, M).$

\normalsize

\end{lem}

\begin{proof} See \cite[proof of Lemma 2.19]{NT}.   \end{proof}

\begin{cor} \label{Greenberg-Wiles} Assume that $\overline{\rho}|_{G_{F(\zeta_p)}}$ is irreducible. For a local condition $( \mathcal{L}_v )_{v \in S}$ of $\mathrm{ad}^0\overline{\rho}$, we have the following formula.

$\mathrm{dim}_{\mathbb{F}}H^1_{\mathcal{L}}(G_{F^+, S}, \mathrm{ad}^0\overline{r}) - \mathrm{dim}_{\mathbb{F}}H^1_{\mathcal{L}^{\perp}}(G_{F^+, S}, \mathrm{ad}^0\overline{r}(1)) = \sum_{v \in S}(\mathrm{dim}_{\mathbb{F}}\mathcal{L}_v - \mathrm{dim}_{\mathbb{F}}H^0(F_v^+, \mathrm{ad}^0 \overline{r})) - [F^+ : \mathbb{Q}]$.

\end{cor}

\begin{proof} By the irreducibility, we have $H^0(G_{F, S}, \mathrm{ad}^0\overline{\rho}) = H^0(G_{F, S}, \mathrm{ad}^0\overline{\rho}(1)) = 0$. The same proof as \cite[Lemma 2.1.3]{CHT} shows $\mathrm{dim}_{\mathbb{F}}H^0(F_v^{+}, \mathrm{ad}^0\overline{\rho}) = 1$ for any $v \mid \infty$ by using the symmetry of the pairing $\overline{\rho} \times \overline{\rho}^c \rightarrow \overline{\chi}|_{G_{F}}$, which corresponds to the ``oddness'' of $\overline{r}$.  \end{proof}

For a finite $\mathbb{F}[G_{F^+, S}]$-module $M$, let $\Sha_{S}^i(M) := \mathrm{Ker}(H^i(G_{F^+, S}, M) \rightarrow \prod_{v \in S} H^i(G_{F^+_v}, M))$.

\begin{prop} \label{PTE} (Poitou-Tate exact sequence) For a local condition $( \mathcal{L}_v )_{v \in S}$ of $\mathrm{ad}^0\overline{r}$, we have the following exact sequences.
    
1 \ $0 \rightarrow H^1_{\mathcal{L}}(G_{F^+, S}, \mathrm{ad}^0\overline{r}) \rightarrow H^1(G_{F^+, S}, \mathrm{ad}^0\overline{r}) \rightarrow \oplus_{v \in S} H^1(F_v^+, \mathrm{ad}^0\overline{r}|_{G_{F_v^+}})/\mathcal{L}_v \rightarrow H^1_{\mathcal{L}^{\perp}}(G_{F^+, S}, \mathrm{ad}^0\overline{r}(1))^{\vee} \rightarrow \Sha_{S}^1(\mathrm{ad}^0\overline{r}(1))^{\vee} \rightarrow 0.$

2 \ $0 \rightarrow \Sha^1_{S}(\mathrm{ad}^0\overline{r}(1))^{\vee} \rightarrow H^2(G_{F^+, S}, \mathrm{ad}^0\overline{r}) \rightarrow \oplus_{v \in S} H^2(F_v^+, \mathrm{ad}^0\overline{r}) \rightarrow 0$.

\end{prop}

\begin{proof} See \cite[proof of Lemma 2.3.4]{CHT}. \end{proof}

\begin{dfn}

Let $r : G_{F^+, S} \rightarrow \mathcal{G}_2(\mathcal{O}/\varpi^m)$ be a lift of $\overline{r}$ and $v \notin S$ be a finite place of $F^+$ splitting $v=uu^c$ in $F$. We say that $v$ is $r$-nice if the ramification index and the inertia degree of $F$ at $v$ are $1$, $q_v \not\equiv \pm 1 \mod p$ and $r(\mathrm{Frob}_u)$ has eigenvalues $\alpha, \beta \in (\mathcal{O}/\varpi^m)^{\times}$ such that $\alpha/\beta = q_v$.

\end{dfn}

\begin{rem}

In this section, we don't need the property that the ramification index and the inertia degree of $F$ at $v$ are $1$. In {\S} 7.3, we will use this property.

\end{rem}

\begin{lem}\label{Steinberg deformation}
    
    Let $l \neq p$ be a prime, $L/\mathbb{Q}_l$ be a finite extension such that $q_L \neq \pm 1 \mod p$ and $\overline{\rho}_L : G_L \rightarrow \mathrm{GL}_2(\mathbb{F})$ be an unramified representation such that the Frobenius eigenvalues $\alpha$, $\beta$ of $\overline{\rho}$ satisfy $\alpha/\beta = q_L$. Let $\phi : G_{L} \rightarrow \mathcal{O}^{\times}$ be an unramified character which is a lift of $\mathrm{det}\overline{\rho}_L$. Then the following results hold.
    
1 \ The Steinberg lifting ring $R^{\mathrm{st}, \phi}_{\overline{\rho}_L}$ with the determinant $\phi$ over $\mathcal{O}$ is formally smooth over $\mathcal{O}$ of dimension $4$. (See \cite[Proposition 3.1]{IA} for the definition of the Steinberg deformation ring $R^{\mathrm{st}}$ and see \cite[{\S} 2]{KP} for how to fix the determinant.)

2 \ Let $\mathcal{L}_v^{\mathrm{ur}}$ be the subspace of $H^1(L, \mathrm{ad}^0\overline{\rho})$ consisting of unramified deformations of $\overline{\rho}$ and $\mathcal{L}_{v}^{\mathrm{st}}$ be the subspace of $H^1(L, \mathrm{ad}^0\overline{\rho})$ consisting of Steinberg deformations of $\overline{\rho}$.

Then we have the following properties.

(a) \ $\mathrm{dim}_{\mathbb{F}}\mathcal{L}_v^{\mathrm{ur}} = \mathrm{dim}_{\mathbb{F}}\mathcal{L}_v^{\mathrm{st}} = 1$.

(b) \ $H^1(L, \mathrm{ad}^0\overline{\rho}) = \mathcal{L}_v^{\mathrm{ur}} \oplus \mathcal{L}_v^{\mathrm{st}}$.

(c) \ $H^1(L, \mathrm{ad}^0\overline{\rho}(1)) = \mathcal{L}_v^{\mathrm{ur}, \perp} \oplus \mathcal{L}_v^{\mathrm{st}, \perp}$, where $\mathcal{L}_v^{\mathrm{st}, \perp}$ (resp. $\mathcal{L}_v^{\mathrm{ur}, \perp}$) denotes the subspace of $H^1(L, \mathrm{ad}^0\overline{\rho}(1))$ annihilated by $\mathcal{L}_v^{\mathrm{st}}$ (resp. $\mathcal{L}_v^{\mathrm{ur}}$) via the local duality. 

\end{lem}

\begin{proof} By \cite[Lemma 2.4.27]{CHT}, we have a formally smooth quotient of dimension $5$ $R^{\mathrm{st}}_{\overline{\rho}_L} \twoheadrightarrow R'_{\overline{\rho}_L}$ inducing a bijection between $E'$-valued points of both sides for any finite extension $E'$ of $E$. Since $R^{\mathrm{st}}_{\overline{\rho}_L}$ is $p$-torsion free and reduced by \cite[Lemma 3.3]{IA} and $R^{\mathrm{st}}_{\overline{\rho}_L}[\frac{1}{p}]$ is a Jacobson ring and the residue field of any maximal ideal is a finite extension by \cite[Lemma 2.6]{IA}, we obtain $R^{\mathrm{st}}_{\overline{\rho}_L} \Isom R'_{\overline{\rho}_L}$ and thus the formally smoothness $R^{\mathrm{st}}_{\overline{\rho}_L}$. This implies that $R^{\mathrm{st}, \phi}_{\overline{\rho}_L}$ is formally smooth over $\mathcal{O}$ of dimension $4$. By the assumption $q_L \neq \pm 1 \mod p$ and the local Euler-Poincare characteristic formula, we have $H^1(L, \mathrm{ad}^0\overline{\rho}) = H^1(L, \mathbb{F}) \oplus H^1(L, \overline{\varepsilon_p}) = \mathcal{L}_v^{\mathrm{ur}} \oplus \mathcal{L}_v^{\mathrm{st}}$ and  $\mathrm{dim}_{\mathbb{F}}\mathcal{L}_v^{\mathrm{ur}} = \mathrm{dim}_{\mathbb{F}}\mathcal{L}_v^{\mathrm{st}} = 1$. Thus we obtain $(a)$ and $(b)$. $(c)$ follows from $(b)$ and the local duality.  \end{proof}

\begin{lem} \label{liftable Steinberg} Let $r_m : G_{F^+, S} \rightarrow \mathcal{G}_2(\mathcal{O}/\varpi^m)$ be a lift of $\overline{r}$, $v \notin S$ be an $r_{m}$-nice place and $\rho_{m+1, v} : G_{F^+_v, S} \rightarrow \mathrm{GL}_2(\mathcal{O}/\varpi^{m+1})$ be a lifting of $r_m|_{G_{F_v^+}}$. Fix a geometric Frobenius lift $\phi_v \in G_{F_v^+}$ and a lift $\gamma_v \in I_{F_v^+}$ of a topological generator $I_{F_v^+}/P_{F_v^+}$. Then we have the following results.

1 \ $G_{F_v^+}/\mathrm{Ker}\rho_{m+1, v}$ is generated by $\phi_v$ and $\gamma_v$.

2 \ After replacing $\rho_{m+1, v}$ by an equivalent representation if necessary, there exists $\alpha \in (\mathcal{O}/\varpi^{m+1})^{\times}$ and $A, B \in \mathbb{F}$ such that $\rho_{m+1, v}(\phi_v) = \begin{pmatrix}
q_v^{-1}\alpha + \varpi^mA & 0 \\
 0 & \alpha - \varpi^mA \end{pmatrix}$ and $\rho_{m+1, v}(\gamma_v) = \begin{pmatrix}
1 & \varpi^mB \\
 0 & 1 \end{pmatrix}$.

3 \ Let $A, B$ be as in 2. Then $\rho_{m+1, v}$ is liftable to $\mathcal{O}$ if and only if $B=0$ or $A=0$. 

4 \ Fix $f \in H^1(F_v^+, \mathrm{ad}^0\overline{r}|_{G_{F_v^+}})$ such that the image $f_{\mathbb{F}}$ of $f$ via the map $H^1(F^+_v, \mathrm{ad}^0\overline{r}) \twoheadrightarrow H^1(F_v^+, \mathbb{F})$ satisfies $f_{\mathbb{F}} \neq 0$. Then there exists $\alpha \in \mathbb{F}$ such that $(1 + \alpha\varpi^mf) \rho_{m+1, v}$ has a lifting to a point of $\mathrm{Spec}R^{\mathrm{st}}_{\overline{\rho}|_{G_{F_v}}}(\mathcal{O})$.

\end{lem}

\begin{proof}

    Since $\overline{r}|_{G_{F_v}}$ is unramified, $\rho_{m+1, v}$ is tamely ramified. This implies 1. 2 follows from (b) of 2 of Lemma \ref{Steinberg deformation}. About 3, ``if'' is clear. If $B$ is nonzero and $\rho_{m+1, v}$ is liftable to $\mathcal{O}$, then $q_v^{-1}\alpha + \varpi^mA = q_v^{-1}(\alpha - \varpi^mA)$ and thus $\varpi^m(1 + q_v^{-1})A = 0$. Thus we have $A=0$. 4 follows from 3. \end{proof}

We consider the following objects.

\begin{itemize}
\item For any place $w \mid p$ of $F^+$, let $\lambda_w \in (\mathbb{Z}_+^2)^{\mathrm{Hom}_{\mathbb{Q}_p}(F_{\tilde{w}}, \overline{\mathbb{Q}}_p)}$ and $\tau_w : I_{F_{\tilde{w}}} \rightarrow \mathrm{GL}_2(\mathcal{O})$ be an inertia type such that $\mathrm{WD}(\psi|_{G_{F_{\tilde{w}}}})|_{I_{F_{\tilde{w}}}} = \mathrm{det}\tau_{w}$ and $\psi|_{G_{F_{\tilde{w}}}}$ has $p$-adic Hodge type $(\lambda_{\tau, 1} + \lambda_{\tau, 2} + 1)_{\tau \in \mathrm{Hom}_{\mathbb{Q}_p}(F_{\tilde{w}}, \overline{\mathbb{Q}}_p)}$. (See definition \ref{inertia type} later for the definition of inertia types.)
\item For any $v \mid p$ of $F^+$, let $\mathcal{C}_v$ be an irreducible component of $\mathrm{Spec}R_{\overline{\rho}|_{G_{F_{\tilde{v}}}}}^{\mathrm{ss}, \lambda_v, \tau_v, \psi|_{G_{F_{\tilde{v}}}}}$ such that $\mathcal{C}_v \otimes_{\mathcal{O}} \mathcal{O}_{E'}$ is irreducible for any finite extension $E'/E$. ($R_{\overline{\rho}|_{G_{F_{\tilde{v}}}}}^{\mathrm{ss}, \lambda_v, \tau_v, \psi|_{G_{F_{\tilde{v}}}}}$ is the potentially semistable lifting ring  over $\mathcal{O}$ of weight $\lambda_v$, inertia type $\tau_v$ and determinant $\psi|_{G_{F_{\tilde{v}}}}$. This is equal to $R^{\Box, \psi|_{G_{F_{\tilde{v}}}}, \tau_v, \lambda_v}_{\overline{\rho}|_{G_{F_{\tilde{v}}}}}$ of \cite[Theorem A]{KP}.)
\item  For any $v \nmid p$ such that $v \in S$, let $\mathcal{C}_v$ be an irreducible component of $\mathrm{Spec}R^{\psi|_{G_{F_{\tilde{v}}}}}_{\overline{\rho}|_{G_{F_{\tilde{v}}}}}$ such that $\mathcal{C}_v \otimes_{\mathcal{O}} \mathcal{O}_{E'}$ is irreducible for any finite extension $E'/E$. ($R_{\overline{\rho}|_{G_{F_{\tilde{v}}}}}^{\psi|_{G_{F_{\tilde{v}}}}}$ is the universal lifting ring of determinant $\psi|_{G_{F_{\tilde{v}}}}$. This is equal to $R^{\Box, \psi|_{G_{F_{\tilde{v}}}}}_{\overline{\rho}|_{G_{F_{\tilde{v}}}}}$ of \cite[Theorem A]{KP}.)
\item For any $v \in S$, let $\rho_v : G_{F_{\tilde{v}}} \rightarrow \mathrm{GL}_2(\mathcal{O})$ be a generic representation which is contained in $\mathcal{C}_v$. (The genericity of $\rho_v$ means $H^0(W_{F_v}, \mathrm{ad}\mathrm{WD}(\rho_v)(1)) = 0$.)
\end{itemize}

The following is the main theorem of this subsection.

\begin{thm}\label{lifting of Galois}

    We assume that $\overline{\rho} : G_{F} \rightarrow \mathrm{GL}_2(\mathbb{F})$ is surjective.

Then for any $t \in \mathbb{Z}_{>0}$, there exist a finite set $Q$ of $\overline{r}$-nice finite places of $F^+$ and a lifting $r : G_{F^+, S \cup Q} \rightarrow \mathcal{G}_2(\mathcal{O})$ of $\overline{r}$ satisfying the following.   
    
    1 \ For any $v \in Q$, we have $r|_{G_{F^+_v}} \in \mathrm{Spec}R^{\mathrm{st}, \psi|_{G_{F^+_v}}}_{\overline{\rho}|_{G_{F^+_{v}}}}(\mathcal{O})$.
    
    2 \ For any $v \in S$, the representation $r|_{G_{F^+_v}}$ is contained in $\mathcal{C}_v$, generic and $r|_{G_{F_v^+}} \mod \ \varpi^t$ is equivalent to $\rho_v \mod \ \varpi^t$.

    \end{thm}

\begin{rem}

A very similar result was proved by \cite{KP} in very general situations, but their $v \in Q$ doesn't satisfy the condition that the Steinberg deformation ring of $\overline{r}|_{G_{F_v^+}}$ is formally smooth over $\mathcal{O}$. In {\S} 7.2 and 7.3, we will use a certain big $R=T$ theorem and in order to use that, we need the property that the given local lifting rings at all non-$p$-adic places are formally smooth over $\mathcal{O}$. (See {\S} 7.2.)

\end{rem}

Our strategy is exactly the same as \cite{KP}. Roughly speaking, that is a combination of the following two works. (It should be noted that they introduced many technical improvements to prove their results in very general situations.)

\vspace{0.5 \baselineskip}

(1) \cite{FS}: He proved the above result under the assumption that every local condition $\mathcal{L}_v$ has the expected dimension. Concretely, $$\mathrm{dim}_{\mathbb{F}}\mathcal{L}_v - \mathrm{dim}_{\mathbb{F}}H^0(F_v^+, \mathrm{ad}^0\overline{r}|_{G_{F_v^+}}) = \begin{cases} 0 \ \ \ \ \ \ \ \ \ \ \ \ \ \ \ \ \ \ \text{if $v \nmid p$} \\ 
  [F_v^+ : \mathbb{Q}_p] \ \ \ \ \ \ \ \ \text{if $v \mid p$.}
\end{cases}.$$
In this case, the given local lifting ring is formally smooth, but we want to give applications to the Breuil-M$\mathrm{\acute{e}}$zard conjecture in {\S} 7.3 and thus we want to study a non-formally smooth lifting ring $R^{\mathrm{ss}, \lambda_v, \tau_v}_{\overline{r}|_{G_{F_v^+}}}$ at $p$-adic places. Therefore, this method doesn't suit our purpose.

\vspace{0.5 \baselineskip}

(2) \cite{KR}: They constructed a lifting $r_{m} : G_{F^+, S \cup Q_m} \rightarrow \mathcal{G}_2(\mathcal{O}/\varpi^m)$ for any $m$ inductively such that $Q_m \subset Q_{m+1}$ and $r_{m}|_{G_{F^+_v}}$ is equivalent to $\rho_v \mod \varpi^m$ for any $v \in S$. However, the representation $\varprojlim_{m}r_m$ may be ramified at infinitely many places. Thus, this method also doesn' work to prove our main theorem.

\vspace{0.5 \baselineskip}

As stated above, each method by itself is inadequate for our main theorem. However, we can prove our main theorem by combining these two methods by using the following result, which was proved in \cite{KP}.

\begin{thm}\label{formally smooth} For any $v \in S$, there exist an open subset $Y_v$ of $\mathcal{C}_v(\mathcal{O})$ with respect to the $p$-adic topology containing $\rho_v$, a positive integer $n_v$ and a subspace $\mathcal{L}_v$ of $H^1(F_{\tilde{v}}, \mathrm{ad}^0 \overline{r})$ satisfying the following conditions. (In the following, $Y_{v, n}$ is the image of $Y_v$ in $\mathcal{C}_v(\mathcal{O}/\varpi^n)$ and $Z_v$ is the inverse image of $\mathcal{L}_v$ in $Z^1(F_{\tilde{v}}, \mathrm{ad}^0 \overline{r})$.)
    
1 \ For any $n \ge n_v$, $Y_{v, n+1} \rightarrow Y_{v, n}$ is a principal homogeneous space of $Z_v$.

2 \  $\mathrm{dim}_{\mathbb{F}}\mathcal{L}_v - \mathrm{dim}_{\mathbb{F}}H^0(F_{\tilde{v}}, \mathrm{ad}^0\overline{r}) = \begin{cases} 0\ \ \ \ \ \ \ \ \ \ \ \ \ \ \ \ \ \ \text{if $v \nmid p$.} \\ 
 [F_{\tilde{v}} : \mathbb{Q}_p] \ \ \ \ \ \ \ \ \text{if $v \mid p$.}
  \end{cases}$

3 \ Any $\rho_v' \in Y_v$ is generic.
    
\end{thm}

\begin{proof} See \cite[Proposition 4.7 and Lemma 4.9]{KP}.  \end{proof}

The above result Theorem \ref{formally smooth} roughly says that the above method (1) works for a lifting $r_m : G_{F^+, S} \rightarrow \mathcal{G}_2(\mathcal{O}/\varpi^m)$ of $\overline{r}$ for a sufficiently large $m$ and thus we can prove our main theorem by constructing a lifting $r_m : G_{F^+, S \cup Q_m} \rightarrow \mathcal{G}_2(\mathcal{O}/\varpi^m)$ of $\overline{r}$ for a sufficiently large $m$ by using the above method (2).

\vspace{0.5 \baselineskip}

From now on, we start the proof of Theorem \ref{lifting of Galois}.

\begin{prop}\label{independence}

Let $f_1, \cdots, f_r$ (resp. $g_1, \cdots, g_s$) be a basis of $H^1(G_{F^+, S}, \mathrm{ad}^0\overline{r})$ (resp. $H^1(G_{F^+, S}, \mathrm{ad}^0\overline{r}(1))$). Let $K$ be the fixed field of $\mathrm{ad}^0\overline{r}|_{G_{F^+(\zeta_p)}}$, $K_{f_i}$ be the fixed field of $f_i|_{G_{K}}$ and $K_{g_i}$ be the fixed field of $g_i|_{G_{K}}$. Then the following statements hold.

1 \ $f_i|_{G_K}$ (resp. $g_i|_{G_{K}}$) induces $\mathrm{Gal}(K_{f_i}/K) \ \Isom \ \mathrm{ad}^0\overline{r}$ (resp. $\mathrm{Gal}(K_{g_i}/K) \ \Isom \ \mathrm{ad}^0\overline{r}$) of $\mathrm{Gal}(K/F^+)$-modules.

2 \ For any $i$, $K_{f_i}$ (resp. $K_{g_i}$, $K(\zeta_{p^{\infty}})$) is linearly disjoint from $K_{f_1} \cdots \check{K}_{f_i} \cdots K_{f_r} K_{g_1} \cdots K_{g_s}(\zeta_{p^{\infty}})$ (resp. $K_{f_1} \cdots K_{f_r} K_{g_1} \cdots \check{K}_{g_i} \cdots K_{g_s}(\zeta_{p^{\infty}})$, $K_{f_1} \cdots K_{f_r} K_{g_1} \cdots K_{g_s}$) over $K$.

\end{prop}

\begin{proof} Same as \cite[Fact 5]{KR}. \end{proof}

In the following, we will fix a basis $f_1, \cdots, f_r, g_1, \cdots, g_s$ as above Proposition \ref{independence}.

\begin{lem}\label{shagroup}

There exists a finite set $Q$ of $\overline{r}$-nice places of $F^+$ such that $\Sha_{S \cup Q}^1 (\mathrm{ad}^0 \overline{r}(1))= 0$.
            
\end{lem}

\begin{proof} Same as \cite[Lemma 6]{KR}. \end{proof}

Let $Q$ be a finite set of $\overline{r}$-nice places of $F^+$ as in Lemma \ref{shagroup}. For any $v\in Q$, we fix a lift $\tilde{v}$ to $F$ and put $\tilde{Q} := \{ \tilde{v} \mid v \in Q \}$. %(In the following, when we take a $\overline{r}$-nice place of $F^+$, we always take a lift of $v$ to $F$.)

Note that for $v \in Q$, the Steinberg lifting ring $R^{\mathrm{st}, \psi|_{G_{F_{\tilde{v}}}}}_{\overline{r}|_{G_{F_{\tilde{v}}}}}$ is formally smooth of dimension $4$ by 1 of Lemma \ref{Steinberg deformation} and thus the properties of Theorem \ref{formally smooth} holds for $Y_v := \mathrm{Spec}R^{\mathrm{st}, \psi|_{G_{F_{\tilde{v}}}}}_{\overline{r}|_{G_{F_{\tilde{v}}}}}$, $\mathcal{L}_v = \mathcal{L}_v^{\mathrm{st}}$ and $n_v := 1$ except the property 3. We take an $\mathcal{O}$-rational point $\rho_v$ of $\mathrm{Spec}R^{\mathrm{st}, \psi|_{G_{F_{\tilde{v}}}}}_{\overline{r}|_{G_{F_{\tilde{v}}}}}$ for every $v \in Q$. We also take $n_v$ and $\mathcal{L}_v$ as in Theorem \ref{formally smooth} and take $n_0 \ge \max_{v \in S} \{n_v, 2\}$. For a moment, we assume the following.

\vspace{0.5 \baselineskip}

(*) \ There exists a lifting $r_{n_0} : G_{F^+, S \cup Q} \rightarrow \mathcal{G}_2(\mathcal{O}/\varpi^{n_0})$ satisfying the following conditions.

\begin{itemize}
\item The fixed field $P_{n_0}$ of $\mathrm{ad}r_{n_0}|_{G_{K(\zeta_p)}}$ is linearly disjoint from $K_{f_1} \cdots K_{g_s}(\zeta_{p^{\infty}})$ over $K$, where $f_1, \cdots, g_s$ are fixed basis appeared in Proposition \ref{independence}.
\item $r_{n_0}|_{G_{F_v}}$ is equivalent to $\rho_v \mod \varpi^{n_0}$ for any $v \in S \cup Q$. 
\item $\mathrm{ad}(r_{n_0}|_{G_{F, S}}) : G_{F, S} \rightarrow \mathrm{PGL}_2(\mathcal{O}/\varpi^{n_0})$ is surjective. 
\end{itemize}

For a finite set $Q'$ of $r_{n_0}$-nice places of $F^+$(note that by the definition of nice primes, we assume $Q' \cap (Q \cup S) = \emptyset$), let $\mathcal{L}_{Q'} := ((\mathcal{L}_{Q'})_v)_{v \in S \cup Q \cup Q'}$ be a local condition for $\mathrm{ad}^0\overline{r}$ defined by $$(\mathcal{L}_{Q'})_v:=\begin{cases} \mathcal{L}_v  \ \ \ \ \ \ \ \ \text{if $v \in S$} \\ 
\mathcal{L}_v^{\mathrm{st}} \ \ \ \ \ \ \ \ \text{if $v \in Q \cup Q'$}
     \end{cases}.$$ 

If $Q'$ is empty, we simply write $\mathcal{L}$ for $\mathcal{L}_{Q'}$. Note that we have $\mathrm{dim}H^1_{\mathcal{L}_{Q'}}(G_{F^+, S \cup Q \cup Q'}, \mathrm{ad}^0 \overline{r}) = \mathrm{dim}H^1_{\mathcal{L}_{Q'}^{\perp}}(G_{F^+, S \cup Q \cup Q'}, \mathrm{ad}^0 \overline{r}(1))$ by Corollary \ref{Greenberg-Wiles} and 2 of Lemma \ref{Steinberg deformation}.

\begin{lem}\label{obstruction} Under the above assumption (*), there exists a finite set $Q'$ of $r_{n_0}$-nice places of $F^+$ such that $H^1_{\mathcal{L}_{Q'}}(G_{F^+, S \cup Q \cup Q'}, \mathrm{ad}^0 \overline{r}) = 0$ and $H^1_{\mathcal{L}_{Q'}^{\perp}}(G_{F^+, S \cup Q \cup Q'}, \mathrm{ad}^0 \overline{r}(1)) = 0$.
    
\end{lem}

\begin{proof} We may assume that $H^1_{\mathcal{L}^{\perp}}(G_{F^+, S \cup Q}, \mathrm{ad}^0 \overline{r}(1)) \neq 0$. It suffices to prove that there exists a $r_{n_0}$-nice place $v$ such that $$\mathrm{dim}_{\mathbb{F}}H^1_{(\mathcal{L}^{\perp}, \mathcal{L}_{v}^{\mathrm{st}, \perp})}(G_{F^+, S \cup Q \cup \{ v \}}, \mathrm{ad}^0\overline{r}(1)) < \mathrm{dim}_{\mathbb{F^+}}H^1_{(\mathcal{L}^{\perp}, \mathcal{L}_{v}^{\mathrm{ur}, \perp})}(G_{F^+, S \cup Q \cup \{ v \}}, \mathrm{ad}^0\overline{r}(1)).$$ More strongly, we will prove the following. ($0_v$ denotes the zero subspace of $H^1(F_v, \mathrm{ad}^0\overline{r}(1))$.) \small$$H^1_{(\mathcal{L}^{\perp}, \mathcal{L}_{v}^{\mathrm{st}, \perp})}(G_{F^+, S \cup Q \cup \{ v \}}, \mathrm{ad}^0\overline{r}(1)) =_{1} H^1_{(\mathcal{L}^{\perp}, 0_v)}(G_{F^+, S \cup Q \cup \{ v \}}, \mathrm{ad}^0\overline{r}(1)) \subsetneq_{2} H^1_{(\mathcal{L}^{\perp}, \mathcal{L}_{v}^{\mathrm{ur}, \perp})}(G_{F^+, S \cup Q \cup \{ v \}}, \mathrm{ad}^0\overline{r}(1)) - (**).$$ \normalsize
    
By Corollary \ref{Greenberg-Wiles} and 2 of Lemma \ref{Steinberg deformation}, we obtain the following formula. ($\mathcal{L}_v := H^1(F_v, \mathrm{ad}^0\overline{r})$.)

\begin{gather*} \mathrm{dim}_{\mathbb{F}}H^1_{(\mathcal{L}^{\perp}, \mathcal{L}_{v}^{\mathrm{st}, \perp})}(G_{F^+, S \cup Q \cup \{ v \}}, \mathrm{ad}^0\overline{r}(1)) = \mathrm{dim}_{\mathbb{F}}H^1_{(\mathcal{L}, \mathcal{L}_{v}^{\mathrm{st}})}(G_{F^+, S \cup Q \cup \{ v \}}, \mathrm{ad}^0\overline{r}(1)) \\
\mathrm{dim}_{\mathbb{F}} H^1_{(\mathcal{L}^{\perp}, 0_v)}(G_{F^+, S \cup Q \cup \{ v \}}, \mathrm{ad}^0\overline{r}(1)) + 1 = \mathrm{dim}_{\mathbb{F}}H^1_{(\mathcal{L}, \mathcal{L}_v)}(G_{F^+, S \cup Q \cup \{ v \}}, \mathrm{ad}^0\overline{r}) \\
\mathrm{dim}_{\mathbb{F}}H^1_{(\mathcal{L}^{\perp}, \mathcal{L}_{v}^{\mathrm{ur}, \perp})}(G_{F^+, S \cup Q \cup \{ v \}}, \mathrm{ad}^0\overline{r}(1)) = \mathrm{dim}_{\mathbb{F}}H^1_{(\mathcal{L}, \mathcal{L}_{v}^{\mathrm{ur}})}(G_{F^+, S \cup Q \cup \{ v \}}, \mathrm{ad}^0\overline{r})
\end{gather*}

Note that $0 \le \mathrm{dim}_{\mathbb{F}}H^1_{(\mathcal{L}^{\perp}, \mathcal{L}_{v}^{\mathrm{st}, \perp})}(G_{F^+, S \cup Q \cup \{ v \}}, \mathrm{ad}^0\overline{r}(1)) - \mathrm{dim}_{\mathbb{F}}H^1_{(\mathcal{L}^{\perp}, 0_v)}(G_{F^+, S \cup Q \cup \{ v \}}, \mathrm{ad}^0\overline{r}(1)) \le \mathrm{dim}_{\mathbb{F}}\mathcal{L}_{v}^{\mathrm{st}, \perp} = 1$ and $0 \le \mathrm{dim}_{\mathbb{F}} H^1_{(\mathcal{L}^{\perp}, \mathcal{L}_{v}^{\mathrm{ur}, \perp})}(G_{F^+, S \cup Q \cup \{ v \}}, \mathrm{ad}^0\overline{r}(1)) - \mathrm{dim}_{\mathbb{F}}H^1_{(\mathcal{L}^{\perp}, 0_v)}(G_{F^+, S \cup Q \cup \{ v \}}, \mathrm{ad}^0\overline{r}(1)) \le 1$. Thus 1 (resp. 2) of $(**)$ is equivalent to $1'$ (resp. $2'$) in the following. $$H^1_{(\mathcal{L}, \mathcal{L}_{v}^{\mathrm{st}})}(G_{F^+, S \cup Q \cup \{ v \}}, \mathrm{ad}^0\overline{r}) \subsetneq_{1'} H^1_{(\mathcal{L}, \mathcal{L}_v)}(G_{F^+, S \cup Q \cup \{ v \}}, \mathrm{ad}^0\overline{r}) =_{2'} H^1_{(\mathcal{L}, \mathcal{L}_{v}^{\mathrm{ur}})}(G_{F^+, S \cup Q \cup \{ v \}}, \mathrm{ad}^0\overline{r}) - (***).$$ 

By Proposition \ref{independence} and our assumption that $P_{n_0}$ is linearly disjoint from $K_{f_1} \cdots K_{g_s}(\zeta_{p^{\infty}})$ over $K$, there exist a $r_{n_0}$-nice prime $v$ and $f \in H^1_{\mathcal{L}}(G_{F^+, S \cup Q}, \mathrm{ad}^0\overline{r}) = H^1_{(\mathcal{L}, \mathcal{L}_{v}^{\mathrm{ur}})}(G_{F^+, S \cup Q \cup \{ v \}}, \mathrm{ad}^0\overline{r})$ and $g \in H^1_{\mathcal{L}^{\perp}}(G_{F^+, S \cup Q}, \mathrm{ad}^0\overline{r}(1)) = H^1_{(\mathcal{L}^{\perp}, \mathcal{L}_{v}^{\mathrm{ur}, \perp})}(G_{F^+, S \cup Q \cup \{ v \}}, \mathrm{ad}^0\overline{r}(1))$ such that $f(\mathrm{Frob}_v) \neq 0$ and $g(\mathrm{Frob}_v) \neq 0$. $g(\mathrm{Frob}_v) \neq 0$ implies $$H^1_{(\mathcal{L}^{\perp}, 0_v)}(G_{F^+, S \cup Q \cup \{ v \}}, \mathrm{ad}^0\overline{r}(1)) \subsetneq_{2} H^1_{(\mathcal{L}^{\perp}, \mathcal{L}_{v}^{\mathrm{ur}, \perp})}(G_{F^+, S \cup Q \cup \{ v \}}, \mathrm{ad}^0\overline{r}(1)).$$ On the other hand, $f(\mathrm{Frob}_v) \neq 0$ implies $H^1_{(\mathcal{L}, \mathcal{L}_{v}^{\mathrm{st}})}(G_{F^+, S \cup Q \cup \{ v \}}, \mathrm{ad}^0\overline{r}) \subsetneq_{1'} H^1_{(\mathcal{L}, \mathcal{L}_v)}(G_{F^+, S \cup Q \cup \{ v \}}, \mathrm{ad}^0\overline{r})$ by $\mathcal{L}_v = \mathcal{L}_v^{\mathrm{ur}} \oplus \mathcal{L}_v^{\mathrm{st}}$. (See 2 of Lemma \ref{Steinberg deformation}.) This implies the result. \end{proof}

%Thus we have $H^1_{(\mathcal{L}, \mathcal{L}_v)}(G_{F^+, S \cup Q \cup \{ v \}}, \mathrm{ad}^0\overline{r}) =_{2'} H^1_{(\mathcal{L}, \mathcal{L}_{v}^{\mathrm{ur}})}(G_{F^+, S \cup Q \cup \{ v \}}, \mathrm{ad}^0\overline{r})$. 

We take $Q'$ as in the above Lemma \ref{obstruction} and we may replace $Q$ with $Q \cup Q'$.

\begin{cor}\label{lifting}

Under the above assumption (*), there exists a lifting $r : G_{F^+, S \cup Q} \rightarrow \mathcal{G}_2(\mathcal{O})$ of $r_{n_0}$ satisfying the conditions of Theorem \ref{lifting of Galois}.

\end{cor}

\begin{proof} 

Since $\Sha_{S \cup Q}^1(\mathrm{ad}^0\overline{r}(1)) = 0$, the vanishing of $\mathrm{obs}(r_{n_0}, \mathcal{O}/\varpi^{n_0+1})$ follows from the vanishing of $\mathrm{obs}(r_{n_0}, \mathcal{O}/\varpi^{n_0+1})$ at every $v \in S \cup Q$ by the exact sequence in 2 of Proposition \ref{PTE}. For any $v \in S \cup Q$, we have a lifting $\rho_v$ to $\mathcal{O}$ of $r_{n_0}|_{G_{F_{\tilde{v}}}}$. This implies that the vanishing of $\mathrm{obs}(r_{n_0}, \mathcal{O}/\varpi^{n_0+1})$ at every $v \in S \cup Q$ by a variant of 2 of Lemma \ref{obstruction theory} for the determinant fixed deformation with values in $\mathrm{GL}_2$ of $G_{F_{\tilde{v}}}$. (See \cite[p18]{CHT} for details.) Thus we have $\mathrm{obs}(r_{n_0}, \mathcal{O}/\varpi^{n_0+1}) = 0$. Consequently, we obtain a lifting $r_{n_0 + 1}' : G_{F^+, S \cup Q} \rightarrow \mathcal{G}_2(\mathcal{O}/\varpi^{n_0+1})$ of $r_{n_0}$. By Lemma \ref{deformation theory}, the set of equivalence classes of liftings $G_{F^+, S \cup Q} \rightarrow \mathcal{G}(\mathcal{O}/\varpi^{n_0 + 1})$ of $r_{n_0}$ is a $H^1(G_{F^+, S \cup Q}, \mathrm{ad}^0 \overline{r})$-torsor. By the exact sequence in 1 of Proposition \ref{PTE} and Lemma \ref{obstruction}, we can take a lifting $r_{n_0 + 1} : G_{F^+, S \cup Q} \rightarrow \mathcal{G}_2(\mathcal{O}/\varpi^{n_0+1})$ of $r_{n_0}$ such that $r_{n_0 + 1}|_{G_{F^+_v}} \in Y_{v, n_0+1}$. (See Theorem \ref{formally smooth} for the definition of $Y_{v, n_0 + 1}$.) By repeating this argument, we obtain a lifting $r : G_{F^+, S \cup Q} \rightarrow \mathcal{G}_2(\mathcal{O})$ of $r_{n_0}$ satisfying the conditions of Theorem \ref{lifting of Galois}. \end{proof}

By Corollary \ref{lifting}, it suffices to prove the existence $r_{n_0}$ as above (*).

\vspace{0.5 \baselineskip}

Since we assume $\varpi^2 \mid p$, by using $i : \mathbb{F} \hookrightarrow \mathcal{O}/\varpi^2$, we can take a lifting $r_2'' := i \circ \overline{r} : G_{F^+, S \cup Q} \rightarrow \mathcal{G}_2(\mathcal{O}/\varpi^2)$ of $\overline{r}$, but we may not have $\nu \circ r_2'' = \chi$ and $\mathrm{det}r_2''|_{G_{F, S}} = \psi$. By considering $r_2''|_{G_{F, S}} \otimes (i(\overline{\psi}^{-1}) \psi)^{\frac{1}{2}}$ and the correspondence of Lemma \ref{correspondence}, we obtain a lifting $r_2' : G_{F^+, S \cup Q} \rightarrow \mathcal{G}_2(\mathcal{O}/\varpi^2)$ of $\overline{\rho}$ satisfying $\nu \circ r_2' = \chi$, $\mathrm{det}r_2'|_{G_{F, S}} = \psi$. (Note that we used the relation $\chi^2|_{G_{F, S}} = \psi \psi^c$.) Moreover, we have the following.

\vspace{0.5 \baselineskip}

(1) \ $r_{2}'|_{G_{F_{\tilde{v}}}} \mod \varpi$ is equivalent to $\rho_v \mod \varpi$ for any $v \in S \cup Q$. 

(2) \ The fixed field $P_{2}$ of $\mathrm{ad}r_{2}'|_{G_{K}} : G_{K} \rightarrow \mathrm{PGL}_2(\mathcal{O}/\varpi^2)$ is linearly disjoint from $K_{f_1} \cdots K_{g_s}(\zeta_{p^{\infty}})$ over $K$ by $P_2 = K$.

(3) \ The image of $\mathrm{ad}r_{2}'|_{G_{F, S}} : G_{F, S} \rightarrow \mathrm{PGL}_2(\mathcal{O}/\varpi^{2})$ contains $\mathrm{PGL}_2(\mathbb{F})$.

\vspace{0.5 \baselineskip}

We may assume that There exists $T \subset S$ satisfying the following properties.

(A) \  $\rho_v$ is unramified for any $v \in T$ and $\langle \mathrm{ad}\rho_{v}(\mathrm{Frob}_{\tilde{v}}) \mod \varpi^2 \mid v \in T \rangle = \mathrm{Ker}(\mathrm{PGL}_2(\mathcal{O}/\varpi^{2}) \twoheadrightarrow \mathrm{PGL}_2(\mathbb{F}))$. 

(B) \ There exists $v \in T$ such that $\mathrm{ad}\rho_v(\mathrm{Frob}_{\tilde{v}}) \mod \varpi^2$ is nontrivial and $\mathrm{Frob}_{\tilde{v}}|_{K_{f_i}}$ and $\mathrm{Frob}_{\tilde{v}}|_{K_{g_i}}$ are trivial for any $f_i$ and $g_i$.

\vspace{0.5 \baselineskip}

We claim that in order to construct $r_{n_0}$ satisfying the conditions in $(*)$, it suffices to prove the following proposition. 

\begin{prop}\label{final proposition}

Let $n \ge 2$ be an integer. We assume that there exists a lifting $r_n' : G_{F^+, S \cup Q} \rightarrow \mathcal{G}_2(\mathcal{O}/\varpi^{n})$ of $\overline{r}$ satisfying the following conditions.

(1) \ $r_{n}'|_{G_{F_{\tilde{v}}}} \mod \varpi^{n-1}$ is equivalent to $\rho_v \mod \varpi^{n-1}$ for any $v \in S \cup Q$. 

(2) \ The fixed field $P_{n}$ of $\mathrm{ad}r_{n}'|_{G_{K}}$ is linearly disjoint from $K_{f_1} \cdots K_{g_s}(\zeta_{p^{\infty}})$ over $K$.

(3) \ $\mathrm{ad}r_{n}'|_{G_{F, S}} : G_{F, S} \rightarrow \mathrm{PGL}_2(\mathcal{O}/\varpi^{n})$ is surjective if $n \ge 3$ and the image of $\mathrm{ad}r_{2}'|_{G_{F, S}} : G_{F, S} \rightarrow \mathrm{PGL}_2(\mathcal{O}/\varpi^{2})$ contains $\mathrm{PGL}_2(\mathbb{F})$. 

Then there exists a finite set $Q'$ of $r_{n}'$-nice primes of $F^+$ and $r_n : G_{F^+, S \cup Q \cup Q'} \rightarrow \mathcal{G}_2(\mathcal{O}/\varpi^n)$ satisfying the following conditions.

(a) \ $r_{n} \mod \varpi^{n-1}$ and $r_n' \mod \varpi^{n-1}$ are equivalent.

(b) \ $r_{n}|_{G_{F_{\tilde{v}}}}$ is equivalent to $\rho_v \mod \varpi^{n}$ for any $v \in S \cup Q$.

(c) \ For any $v \in Q'$, $r_n|_{G_{F^+_{v}}}$ has a lifting to a point of $\mathrm{Spec}R^{\mathrm{st}, \psi|_{G_{F_{v}^+}}}_{\overline{r}|_{G_{F^+_v}}}(\mathcal{O})$.

\end{prop}

\textbf{Proposition \ref{final proposition} $\Rightarrow$ Construction of $r_{n}$ as in (*)}

\vspace{0.5 \baselineskip}

We will construct $r_{n}$ as in (*) by induction on $n \ge 2$. For $n = 2$, we already constructed $r_2'$ satisfying the conditions (1) $\sim$ (3) in Proposition \ref{final proposition}. Thus by Proposition \ref{final proposition}, we obtain $Q'$ and $r_2 : G_{F^+, S \cup Q \cup Q'} \rightarrow \mathcal{G}_2(\mathcal{O}/\varpi^2)$ satisfying the conditions (a) $\sim$ (c) in Proposition \ref{final proposition}. By taking a lifting of $r_n|_{G_{F^+_{v}}}$ to a point of $\mathrm{Spec}R^{\mathrm{st}, \psi|_{G_{F_{v}^+}}}_{\overline{r}|_{G_{F^+_v}}}(\mathcal{O})$ and replacing $Q \cup Q'$ with $Q$, $r_2$ satisfies the conditions in (*). If we have $r_{n}$ as in (*) for $n \ge 2$, then we can take a lifting $r_{n+1}' : G_{F, S} \rightarrow \mathcal{G}_2(\mathcal{O}/\varpi^{n+1})$ of $r_{n}$ by the liftability of $r_{n}|_{G_{F^+_v}}$ for any $v \in S \cup Q$ as used in the proof of Corollary \ref{lifting}. We claim that this $r_{n+1}'$ satisfies the conditions (1) $\sim$ (3) of Proposition \ref{final proposition}. (1) is clear. By the property (A), $\mathrm{ad}r_n|_{G_{F_{S \cup Q}}} \mod \varpi^2 : G_{F, S \cup Q} \rightarrow \mathrm{PGL}(\mathcal{O}/\varpi^2)$ is surjective. Thus $\mathrm{ad}r_n|_{G_{F_{S \cup Q}}}$ is also surjective by \cite[Lemma 7.13]{Feru2} and $[\mathrm{Ker}(\mathrm{PGL}_2(\mathcal{O}/\varpi^n) \rightarrow \mathrm{PGL}_2(\mathbb{F})) : \mathrm{Ker}(\mathrm{PGL}_2(\mathcal{O}/\varpi^n) \rightarrow \mathrm{PGL}_2(\mathbb{F}))] = \mathrm{Ker}(\mathrm{PGL}_2(\mathcal{O}/\varpi^n) \rightarrow \mathrm{PGL}_2(\mathcal{O}/\varpi^2))$. (This is used in \cite[p.3561]{KP}. Here, $[ \ : \ ]$ denotes the commutator group.) 

Note that we have $\mathrm{ad}r'_{n+1}|_{G_{K}} : G_K/(\mathrm{Ker}\mathrm{ad}r'_{n+1}|_{G_{K}}) \Isom \mathrm{Ker}(\mathrm{PGL}_2(\mathcal{O}/\varpi^n) \rightarrow \mathrm{PGL}_2(\mathbb{F}))$. In particular, we have $\mathrm{ad}r'_{n+1}|_{G_{K}} \mod \varpi^2 : G_K/(\mathrm{Ker}(\mathrm{ad}r'_{n+1}|_{G_{K}} \mod \varpi^2)) \Isom \mathrm{Ker}(\mathrm{PGL}_2(\mathcal{O}/\varpi^2) \rightarrow \mathrm{PGL}_2(\mathbb{F})) \cong \mathrm{ad}^0\overline{r}$ beucase we assume $p \neq 2$. Note also that $\mathrm{ad}^0\overline{r}$ is a simple $G_{F^+, S}$-module and $\mathrm{Gal}(K_{f_1} \cdots K_{g_s}(\zeta_{p^{\infty}})/K) \Isom (\prod_{i=1}^r \mathrm{Gal}(K_{f_i}/K)) \times (\prod_{j=1}^s \mathrm{Gal}(K_{g_j}/K)) \times \mathrm{Gal}(K(\zeta_{p^{\infty}})/K) \cong (\prod_{i=1}^r \mathrm{ad}^0\overline{r}) \times (\prod_{j=1}^s \mathrm{ad}^0\overline{r}) \times \mathrm{Gal}(K(\zeta_{p^{\infty}})/K)$ as $G_{F^+, S}$-modules. Thus if the fixed field $P_{n+1}$ of $\mathrm{ad}r_{n+1}'|_{G_{K}}$ was not linearly disjoint from $K_{f_1} \cdots K_{g_s}(\zeta_{p^{\infty}})$ over $K$ (i.e., $P_{n+1} \cap K_{f_1} \cdots K_{g_s}(\zeta_{p^{\infty}}) \neq K$), the fixed field $P_2'$ of $\mathrm{ad}r_{n+1}'|_{G_{K}} \mod \varpi^2$ would be equal to $K_{f_i}$ or $K_{g_i}$ over $K$ for some $i$. This contradicts to the condition (B). $\Box$

%By using the above equality $[\mathrm{Ker}(\mathrm{PGL}_2(\mathcal{O}/\varpi^n) \rightarrow \mathrm{PGL}_2(\mathbb{F})) : \mathrm{Ker}(\mathrm{PGL}_2(\mathcal{O}/\varpi^n) \rightarrow \mathrm{PGL}_2(\mathbb{F}))] = \mathrm{Ker}(\mathrm{PGL}_2(\mathcal{O}/\varpi^n) \rightarrow \mathrm{PGL}_2(\mathcal{O}/\varpi^2))$,

%We claim that in order to prove the property (2) for $r_{n+1}'$, it suffices to prove that the fixed field $P_2'$ of $\mathrm{ad}r_{n+1}'|_{G_{K}} \mod \varpi^2$ is linearly disjoint from $K_{f_1} \cdots K_{g_s}(\zeta_{p^{\infty}})$ over $K$.

%The fixed field $P_{n}$ of $\mathrm{ad}r_{n}'$ is linearly disjoint from $K_{f_1} \cdots K_{g_s}(\zeta_{p^{\infty}})$ over $K$.

\vspace{0.5 \baselineskip}

In the following, we will prove Proposition \ref{final proposition}. By the same method of \cite{KR}, we can prove the following.

\begin{prop}\label{obstruction removing}

    1 \ There exists a finite set $Q' = \{ v_1, \cdots, v_m \}$ of $r_{n}'$-nice places such that we have $H^1(G_{F^+, S \cup Q \cup Q'}, \mathrm{ad}^0\overline{r}) \Isom \oplus_{v \in S \cup Q} H^1(F^+_v, \mathrm{ad}^0\overline{r})$ and $H^1(G_{F^+, S \cup Q \cup Q'}, \mathrm{ad}^0\overline{r}(1)) \Isom \oplus_{v \in Q'} H^1(F^+_v, \mathrm{ad}^0\overline{r}(1))$.

    2 \ Let $Q'$ be as in 1. Then for any $1 \le k \le m$, there exists a finite set $T_k$ of $r_{n}'$-nice places and $f_{k} \in H^1(G_{F^+, S \cup Q \cup Q' \cup T_k}, \mathrm{ad}^0\overline{r})$ satisfying the following conditions.
    
    (1) \ $|T_k| = 1$ or $2$.
    
    (2) \ $(f_{k}|_{G_{F^+_w}})_{\mathbb{F}} = 0$ for any $w \in T_k$ and $(f_{k}|_{G_{F^+_{v_k}}})_{\mathbb{F}} \neq 0$. (See Lemma \ref{liftable Steinberg} for the definition of $(f_{k}|_{G_{F^+_w}})_{\mathbb{F}}$.)
    
    (3) \ For any $j < k$ and $t \in T_j$, $(f_k|_{G_{F^+_{t}}})_{\mathbb{F}} = 0$.
    
    (4) \ $f_k|_{G_{F^+_v}} = 0$ for any $v \in S \cup Q \cup Q' \setminus \{ v_k \}$.

\end{prop}

\begin{proof} Same proof as \cite[Lemma 9, Proposition 10 and Proposition 11]{KR}. \end{proof}

\begin{cor}\label{lifting prescribed}

There exist a finite set $Q'$ of $r_n'$-nice places and a lifting $r_n : G_{F^+, S \cup Q \cup Q'} \rightarrow \mathcal{G}_2(\mathcal{O}/\varpi^n)$ of $r_{n-1}$ satisfying the following.

(a) \ $r_{n} \mod \varpi^{n-1}$ and $r_n' \mod \varpi^{n-1}$ are equivalent.

(b) \ $r_n|_{G_{F_v^+}}$ is equivalent to $\rho_v \mod \varpi^n$ for any $v \in S \cup Q$.

(c) \ $r_n|_{G_{F_v^+}}$ has a lifting to a point of $\mathrm{Spec}R^{\mathrm{st}, \psi|_{G_{F^+_v}}}_{\overline{\rho}|_{G_{F_v^+}}}(\mathcal{O})$ for any $v \in Q'$.

\end{cor}

\begin{proof}

By 1 of Proposition \ref{obstruction removing}, we obtain $Q_1 = \{ v_1, \cdots, v_m \}$ and $g \in H^1(G_{F^+, S \cup Q \cup Q_1}, \mathrm{ad}^0\overline{r})$ such that $(1 + \varpi^{n-1} g)r_{n}'|_{G_{F_{\tilde{v}}}}$ is equivalent to $\rho_v \mod \varpi^n$ for any $v \in S \cup Q$. Moreover, by using 2 of Proposition \ref{obstruction removing} and 4 of Lemma \ref{liftable Steinberg}, we can take $\alpha_1, \cdots, \alpha_m \in \mathbb{F}$ such that $(1 + \varpi^{n-1} (g + \sum_{k = 1}^{m} \alpha_k f_k))r_{n}' : G_{F^+, S \cup Q \cup Q_1 \cup (\cup_{k=1}^m T_k)} \rightarrow \mathcal{G}_2(\mathcal{O}/\varpi^{n})$ satisfies the desired conditions by putting $Q' := Q_1 \cup (\cup_{k=1}^m T_k)$. \end{proof}

We complete the proof of Proposition \ref{final proposition} and thus we obtain our main theorem \ref{lifting of Galois}.

\subsection{Big $R=T$ and Gelfand-Kirillov dimension}

The result of this subsection is essentially contained in \cite{BGRT} and \cite{GEKI}. But they consider cohomologies of Shimura curves or Shimura sets coming from division algebras over totally real fields. Thus we need to slightly modify their results. 

We first recall some local results. Let $p$ be an odd prime and $L/\mathbb{Q}_p$ be a finite unramified extension, let $E$ be a finite extension of $\mathbb{Q}_p$ contained in $\overline{\mathbb{Q}}_p$ such that $\mathrm{Hom}_{\mathbb{Q}_p}(L, \overline{\mathbb{Q}}_p) = \mathrm{Hom}_{\mathbb{Q}_p}(L, E)$. Let $\mathcal{O}$ be the ring of integers of $E$ and $\mathbb{F}$ be the residue field of $\mathcal{O}$. Let $f := [L : \mathbb{Q}_p]$ and we fix an embedding $L \hookrightarrow E$ over $\mathbb{Q}_p$ and let $w_f : I_{L} \rightarrow I_{L^{\mathrm{ab}}/L} \xrightarrow{\mathrm{rec}_{L}^{-1}} \mathcal{O}_L^{\times} \twoheadrightarrow \mathbb{F}_{L}^{\times} \hookrightarrow \mathbb{F}^{\times}$ be the Serre's fundamental character of level $f$, where $I_{L^{\mathrm{ab}}/L}$ denotes the inertia group of the maximal abelian extension $L^{\mathrm{ab}}$ of $L$ over $L$.

\begin{dfn}

Let $\overline{\rho}_L : G_{L} \rightarrow \mathrm{GL}_2(\mathbb{F})$ be a continuous representation. 

1 \ If $\overline{\rho}_L$ is reducible, we say that $\overline{\rho}_L$ is generic if $\overline{\rho}_L|_{I_L}$ can be written as in the following up to twist. ($0 \le r_i \le p-3$ and $(r_0, \cdots, r_{f-1}) \neq (0, \cdots, 0), (p-3, \cdots, p-3)$. See \cite[Definition 11.7]{PB} for the definition of genericity of $\overline{\rho}_L$ when $\overline{\rho}_L$ is irreducible.)

$\begin{pmatrix}
\omega_{f}^{(r_0+1) + (r_1+1)p + \cdots + (r_{f-1}+1)p^{f-1}} & * \\
0 & 1
\end{pmatrix}$.

2 \ We say that $\overline{\rho}_L$ is reducible non-split sufficiently generic if $\overline{\rho}_L|_{I_L}$ can be written as in the following up to twist. ($2 \le r_i \le p-5$ and $* \neq 0$. This notion was introduced in \cite{GEKI}.)

$\begin{pmatrix}
\omega_{f}^{(r_0+1) + (r_1+1)p + \cdots + (r_{f-1}+1)p^{f-1}} & * \\
0 & 1
\end{pmatrix}$.

\end{dfn}

\begin{dfn}\label{inertia type}

Let $\tau : I_L \rightarrow \mathrm{GL}_2(\overline{\mathbb{Q}}_p)$ be a smooth representation. We say that $\tau$ is an inertia type if $\tau$ extends to a representation of $W_L$.
    
\end{dfn}

\begin{thm}\label{typetheory}

(1) \ For any inertia type $\tau : I_L \rightarrow \mathrm{GL}_2(\overline{\mathbb{Q}}_p)$, there exist irreducible representations $\sigma^{\mathrm{crys}}(\tau)$ and $\sigma^{\mathrm{ss}}(\tau)$ of $\mathrm{GL}_2(\mathcal{O}_{L})$ over $\overline{\mathbb{Q}}_p$ unique up to isomorphism satisfying the following conditions.

(a) \ For any irreducible smooth representation $\pi$ of $\mathrm{GL}_2(L)$ over $\overline{\mathbb{Q}}_p$, we have $$\mathrm{dim}_{\overline{\mathbb{Q}}_p}\mathrm{Hom}_{\overline{\mathbb{Q}}_p[\mathrm{GL}_2(\mathcal{O}_L)]}(\sigma^{\mathrm{crys}}(\tau), \pi) \le 1$$ and $$\mathrm{dim}_{\overline{\mathbb{Q}}_p}\mathrm{Hom}_{\overline{\mathbb{Q}}_p[\mathrm{GL}_2(\mathcal{O}_L)]}(\sigma^{\mathrm{ss}}(\tau), \pi) \le 1.$$
        
(b) \ $\mathrm{Hom}_{\overline{\mathbb{Q}}_p[\mathrm{GL}_2(\mathcal{O}_L)]}(\sigma^{\mathrm{crys}}(\tau), \pi) \neq 0$ if and only if $\mathrm{rec}_{L}(\pi)|_{I_{L}} = \tau$. 
        
(c) \ $\mathrm{Hom}_{\overline{\mathbb{Q}}_p[\mathrm{GL}_2(\mathcal{O}_L)]}(\sigma^{\mathrm{ss}}(\tau), \pi) \neq 0$ if and only if $\mathrm{rec}_{L}(\pi)^{\mathrm{ss}}|_{I_{L}} = \tau$ and $\pi$ is generic. 
        
(2) \ Let $D$ be a central division algebra over $L$ with $\mathrm{dim}_LD = 4$.

For any inertia type $\tau : I_L \rightarrow \mathrm{GL}_2(\overline{\mathbb{Q}}_p)$ which is equal to $\mathrm{rec}_{L}(\gamma)|_{I_L}$ for some discrete series representation $\gamma$ of $\mathrm{GL}_2(L)$ over $\overline{\mathbb{Q}}_p$, there exists an irreducible representation $\sigma_D(\tau)$ of $\mathcal{O}_{D}^{\times}$ over $\overline{\mathbb{Q}}_p$ satisfying the following conditions.

For any irreducible smooth representation $\pi$ of $D^{\times}$ over $\overline{\mathbb{Q}}_p$, we have $$\mathrm{dim}_{\overline{\mathbb{Q}}_p}\mathrm{Hom}_{\overline{\mathbb{Q}}_p[\mathcal{O}_D^{\times}]}(\sigma_D(\tau), \pi) \le 1$$ and this is the equality if and only if $\mathrm{rec}_{L}(\mathrm{JL}_{D}(\pi))|_{I_{L}} = \tau$, where $\mathrm{JL}_D$ denotes the Jacquet-Langlands correspondence.

\end{thm}
        
\begin{rem}

1 \ Note that $\sigma_{D}(\tau)$ is not unique up to isomorphism in general because $\mathcal{O}_{D}^{\times}$ is a normal subgroup of $D^{\times}$ and the conjugation of $\sigma_{D}(\tau)$ by a generator of $D^{\times}/\mathcal{O}_{D}^{\times}$ satisfies the same condition, but is not isomorphic to $\sigma_{D}(\tau)$ in general. 

2 \ This also holds if we replace $\overline{\mathbb{Q}}_p$ by $\overline{\mathbb{Q}_l}$ or $\mathbb{C}$.

\end{rem}

\begin{proof} See \cite[Theorem 8.2.1]{EGS} for 1. 2 is easy because $L^{\times}\mathcal{O}_{D}^{\times}$ is a normal subgroup of $D^{\times}$ of finite index. See \cite[comments before Lemma 3.10]{JLtype} for details.  \end{proof}

\vspace{0.5 \baselineskip}

Now we come back to a global situation. We consider the following objects.

\begin{itemize}
\item $F^+$ is a totally real field.
\item $F_0$ is an imaginary quadratic field.
\item $F := F^+F_0$ and assume that $F/F^+$ is unramified at all finite places. (Then we have $2 \mid [F^+:\mathbb{Q}]$.)
\item $\Phi := \mathrm{Hom}_{F_0}(F, \mathbb{C})$. 
\item $p$ is an odd prime which is unramified in $F$.
\item $\iota : \overline{\mathbb{Q}}_p \Isom \mathbb{C}$.
\item $E$ is a finite extension of $\mathbb{Q}_p$ contained in $\overline{\mathbb{Q}}_p$ such that $\mathrm{Hom}(F, \overline{\mathbb{Q}}_p) = \mathrm{Hom}(F, E)$.
\item $\mathcal{O}$ is the ring of integers of $E$ and $\mathbb{F}$ is the residue field of $\mathcal{O}$.
\end{itemize}

\begin{itemize}
\item $S$ is a finite set of finite places of $F^+$ containing all $p$-adic places such that every place in $S$ splits in $F$ and every place of $F$ ramified over $\mathbb{Q}$ splits over $F^+$.
\item $S_p$ is the set of $p$-adic places in $F^+$. 
\item $T_1, T_2, S(B)_1$ and $S(B)_2$ be subsets of $S$ such that $S = S_p \sqcup T_1 \sqcup T_2 \sqcup S(B)_1 \sqcup S(B)_2$ and $\frac{1}{2}|S(B)| \equiv 0 \pmod 2$, where $S(B)$ is the set of places of $F$ lying above a place in $S(B)_1 \sqcup S(B)_2$.
\end{itemize}

Let $U/F^+$ be the unitary group defined by $S(B)$ and $\Psi := \emptyset$ as in {\S} 3.1. Then $U(F^+ \otimes_{\mathbb{Q}} \mathbb{R}) \cong \prod_{v \mid \infty} U(0, 2)$. We fix a $p$-adic place $w$ of $F$. We also write $w$ for the place of $F^+$ lying below $w$. Moreover, for any $v \in S$, we fix a lifting $\tilde{v}$ to $F$ of $v$ and we put $\tilde{S} := \{ \tilde{v} \mid v \in S \}$. We assume the following conditions.
 
\begin{itemize} 
\item There exists a cohomological cuspidal automorphic representation $\pi$ of $\mathrm{GL}_2(\mathbb{A}_{F})$ such that $\pi^c \cong \pi^{\vee}$ and $\overline{\rho}|_{G_{F(\zeta_p)}}$ is irreducible. (We put $\overline{\rho} := \overline{r_{\iota}(\pi)}$.)
\item $\overline{\rho}$ is unramified outside $S$.
\item $T_1$ contains two places $v_1, v_2$ such that $\mathrm{char}\mathbb{F}_{v_1} \neq \mathrm{char}\mathbb{F}_{v_2}$ and $\overline{\rho}|_{G_{F_{\tilde{v}}}}$ is unramified and generic with distinct Frobenius eigenvalues for any $v \in T_1$.
\item For any $v \in T_2 \sqcup S(B)_2$, $\overline{\rho}|_{G_{F_{\tilde{v}}}}$ is irreducible and generic.
\item Every $v \in S(B)_1$ is $\overline{\rho}$-nice.
\item For any $p$-adic place $v \neq w$, the restriction $\overline{\rho}|_{G_{F_{\tilde{v}}}}$ is generic.
\item $\overline{\rho}|_{G_{F_w}}$ is reducible non-split sufficiently generic.
\end{itemize}

After extending $E$, we may assume that all eigenvalues of $\overline{\rho}$ are contained in $\mathbb{F}$. For any $p$-adic place $v \neq w$, there exists a Fontaine-Laffaille lift of $\overline{\rho}|_{G_{F_{\tilde{v}}}}$ by the definition of the genericity \cite[Definition 11.7]{PB} and let $\lambda_v$ be its $p$-adic Hodge type. Let $\mathcal{D} = (F/F^+, S, \tilde{S}, \mathcal{O}, \overline{\rho}, \varepsilon_p^{-1}, \{ \mathcal{D}_v \}_{v \in S})$ be the global deformation problem defined by the following conditions. (See \cite[{\S} 1.5]{CW} or \cite[Definition 4.2]{matsumoto} for details about global deformation problems.)

\begin{itemize}
\item For any $v \in T_1 \sqcup T_2 \sqcup S(B)_2$, the local deformation problem $\mathcal{D}_v$ consists of all deformations of $\overline{\rho}|_{G_{F_{\tilde{v}}}}$.
\item For any $v \in S(B)_1$, the local deformation problem $\mathcal{D}_v$ consists of the Steinberg deformations of $\overline{\rho}|_{G_{F_{\tilde{v}}}}$, i.e., defined by $\mathrm{Spec}R^{\mathrm{st}}_{\overline{\rho}|_{G_{F_{\tilde{v}}}}}.$
\item For any $p$-adic place $v \neq w$, the local deformation problem $\mathcal{D}_v$ consists of the Fontaine-Laffaille deformations of $p$-adic Hodge type $\lambda_v$ of $\overline{\rho}|_{G_{F_{\tilde{v}}}}$.
\item $\mathcal{D}_w$ consists of all deformations of $\overline{\rho}|_{G_{F_w}}$.
\end{itemize}

Note that for any $v \in S$, the local lifting ring $R_{\mathcal{D}_v}$ is formally smooth over $\mathcal{O}$. (See Lemma \ref{formally smooth} and \cite[Lemmas 2.4.1 and 2.4.9]{CHT}.)  Let $K = \prod_{v}K_v$ be a neat open compact subgroup of $U(\mathbb{A}_{F^+}^{\infty})$ defined by the following.

\begin{itemize}
\item For any $v \notin S$ which is ramified over $\mathbb{Q}$, let $K_v = \mathrm{GL}_2(\mathcal{O}_{F_{\tilde{v}}})$, where $\tilde{v}$ is a lift of $v$ to $F$.
\item For any $v \notin S$ which is unramified over $\mathbb{Q}$, let $K_v$ is a hyperspecial open compact subgroup.
\item For any $v \in T_2$, let $K_v = \mathrm{GL}_2(\mathcal{O}_{F_{\tilde{v}}})$.
\item For any $v \in T_1$, let $K_v = \mathrm{Iw}_{\tilde{v}, 1}$, i.e., pro-$l$ Iwahori subgroup of $\mathrm{GL}_2(F_{\tilde{v}})$. (We put $l := \mathrm{char} \mathbb{F}_v$.)
\item For any $v \in S(B)_1 \sqcup S(B)_2$, let $K_v = \mathcal{O}_{B_{F_{\tilde{v}}}^{\mathrm{op}}}^{\times}$.
\item For any $p$-adic place $v \neq w$, let $K_v = \mathrm{GL}_2(\mathcal{O}_{F_{\tilde{v}}})$.
\item $K_w$ be an arbitrary open compact.
\end{itemize}

Let $\mathbb{T}^S_U$ be the subalgebra of the Hecke algebra $\mathcal{H}(U(\mathbb{A}_{F^+}^{\infty}), K)_{\mathcal{O}}$ generated over $\mathcal{O}$ by $\mathcal{H}(U(F^+_{w'}), K_{w'})_{\mathcal{O}}$ for any prime $r$ splitting in $F_0$ not lying below $S$ and any $w' \mid r$. Let $\mathfrak{m}$ be the non-Eisenstein ideal of $\mathbb{T}^S_U$ corresponding to $\overline{\rho}$. (We define the notions of ``decomposed generic'' and ``non-Eisenstein'' by the same way as definition \ref{non-Eisenstein}.)

For any $v \in T_2 \cup S(B)_2$, we can take a lifting $\rho_v$ of $\overline{\rho}|_{G_{F_{\tilde{v}}}}$ to $\mathcal{O}$ unique up to unramified twists by the assumption that $\overline{\rho}_{\mathfrak{m}}|_{G_{F_{\tilde{v}}}}$ is generic and Lemma \ref{formally smooth deformation}. For any $v \in T_2$, let $\sigma_v := \sigma^{\mathrm{ss}}(\rho_v|_{I_{F_{\tilde{v}}}})$ be the representation given by Theorem \ref{typetheory}. For $v \in S(B)_2$, we choose a representation $\sigma_v := \sigma_{B^{\mathrm{op}, \times}_{\tilde{v}}}(\rho_{v}|_{I_{F_{\tilde{v}}}})$ given by Theorem \ref{typetheory}. After extending $E$ if necessary, we may assume that $\sigma_v$ is defined over $E$ and we take a $K_v$-stable lattice of $\sigma_v$ over $\mathcal{O}$ and we also write $\sigma_v$ for this. Let $\sigma := \otimes_{v \in T_2 \sqcup S(B)_2} \sigma_v$. Similary, for any $v \mid p$, we take a $K_v$-stable lattice $\mathcal{V}_{\lambda_v}$ over $\mathcal{O}$ of the irreducible algebraic representation of weight $\lambda_v$ of $\mathrm{GL}_{2, F_{\tilde{v}}^+}$ and we put $\mathcal{V}_{\lambda^w} := \otimes_{w \neq v \in S_p} \mathcal{V}_{\lambda_v}$. Then we have a $\mathcal{O}[K_S]$-module $\mathcal{V}_{\lambda^w}(\sigma) := \mathcal{V}_{\lambda^w} \otimes \sigma$ and thus we can consider $\mathcal{A}_{U}(K^wK_w, \mathcal{V}_{\lambda^w}(\sigma))$. (This is the degree zero cohomology of the locally symmetric space of $U$ of level $K_wK^w$ with the coefficient $\mathcal{V}_{\lambda^w}(\sigma)$. See \cite[{\S} 3.3]{CHT} for more explicit descriptions.)

Let $$\widehat{\mathcal{A}}_{U}(K^w, \mathcal{V}_{\lambda^w}(\sigma))_{\mathfrak{m}} := \varprojlim_{n} \varinjlim_{K_w} \mathcal{A}_U(K^wK_w, \mathcal{V}_{\lambda^w}(\sigma)/\varpi^n)_{\mathfrak{m}},$$ $$\widehat{\mathcal{A}}_{U}(K^w, V_{\lambda^w}(\sigma))_{\mathfrak{m}} := \widehat{\mathcal{A}}_{U}(K^w, \mathcal{V}_{\lambda^w}(\sigma))_{\mathfrak{m}}[\frac{1}{p}],$$ $$\mathbb{T}^S_U(K^wK_w, \mathcal{V}_{\lambda^w}(\sigma))_{\mathfrak{m}} := \mathrm{Im}(\mathbb{T}^S_U \rightarrow \mathrm{End}_{\mathcal{O}}(\mathcal{A}_U(K^wK_w, \mathcal{V}_{\lambda^w}(\sigma)))_{\mathfrak{m}})$$ and $$\mathbb{T}^S_U(K^w, \mathcal{V}_{\lambda^w}(\sigma))_{\mathfrak{m}} := \varprojlim_{K_w} \mathbb{T}^S_U(K^wK_w, \mathcal{V}_{\lambda^w}(\sigma))_{\mathfrak{m}}.$$

As in Theorem \ref{Galois completed cohomology}, by using the correspondence of Lemma \ref{correspondence},  we have a continuous map $\rho_{\mathfrak{m}} : G_{F^+, S} \rightarrow \mathcal{G}_2(\mathbb{T}^S_U(K^w, \mathcal{V}_{\lambda^w}(\sigma))_{\mathfrak{m}})$ such that $\rho_{\mathfrak{m}}(c) \notin \mathcal{G}^0(\mathbb{T}^S_U(K^w, \mathcal{V}_{\lambda^w}(\sigma))_{\mathfrak{m}})$, $\rho_{\mathfrak{m}}(G_{F, S}) \subset \mathcal{G}^0(\mathbb{T}^S_U(K^w, \mathcal{V}_{\lambda^w}(\sigma))_{\mathfrak{m}})$,  $\mathrm{det}(T - \rho_{\mathfrak{m}}(\mathrm{Frob}_w)) = P_w(T)$ for any prime $r$ not lying below any place in $S$ and any finite place $w \mid r$ of $F$ and $\nu \circ \rho_{\mathfrak{m}} = \varepsilon_p^{-1}$. (We recall that $c \in G_{F}$ denotes a complex conjugation.) Let $R_{\overline{\rho}, \mathcal{D}}$ denote the universal deformation ring defined by the global deformation problem $\mathcal{D}$ and $\rho^{\mathrm{univ}} : G_{F^+, S} \rightarrow \mathcal{G}_2(R_{\overline{\rho}, \mathcal{D}})$ be the universal deformation. (See \cite[{\S} 1.5]{CW} for details about $R_{\overline{\rho}, \mathcal{D}}$.)

\begin{prop}\label{deformation problem}

There exists a surjection of $\mathcal{O}$-algebras $f : R_{\overline{\rho}, \mathcal{D}} \twoheadrightarrow \mathbb{T}^S_U(K^w, \mathcal{V}_{\lambda^w}(\sigma))_{\mathfrak{m}}$ such that $f_*\rho^{\mathrm{univ}}$ is equivalent to $\rho_{\mathfrak{m}}$.

\end{prop}

\begin{proof} Since $\mathbb{T}^S_U(K^w, \mathcal{V}_{\lambda^w}(\sigma))_{\mathfrak{m}} = \varprojlim_{K_w}\mathbb{T}^S_U(K^wK_w, \mathcal{V}_{\lambda^w}(\sigma))_{\mathfrak{m}}$, it suffices to prove that for any $v \in S$ and any $K_w$, the restriction to $G_{F_{\tilde{v}}}$ of the representation $G_{F^+, S} \rightarrow \mathcal{G}_2(\mathbb{T}^S_U(K^w, \mathcal{V}_{\lambda^w}(\sigma))_{\mathfrak{m}}) \rightarrow \mathcal{G}_2(\mathbb{T}^S_U(K^wK_w, \mathcal{V}_{\lambda^w}(\sigma))_{\mathfrak{m}})$ is contained in $\mathcal{D}_v$. (Note that the surjectivity follows from the equation $\mathrm{det}(T - \rho_{\mathfrak{m}}(\mathrm{Frob}_w)) = P_w(T)$ for any prime $r$ not lying below any place in $S$ and any finite place $w \mid r$ of $F$.) If $v \in T_1 \cup T_2 \cup S(B)_2 \cup \{ w \}$, this property is clear because $\mathcal{D}_v$ consists of all deformations. If $v \neq w$ is a $p$-adic place, this property follows from Theorems \ref{base changeIII} and \ref{Galois representation}. If $v \in S(B)_1$, this property follows from the local Jacquet-Langlands correspondence (see \cite[before Lemma I.3.2]{HT}), again Theorems \ref{base changeIII} and \ref{Galois representation}. \end{proof}

\begin{prop}

For sufficiently small $K_w$, we have $\mathcal{A}_{U}(K^wK_w, \mathcal{V}_{\lambda^w}(\sigma))_{\mathfrak{m}} \neq 0$.

\end{prop}

\begin{proof}  
    
By Proposition \ref{congruence} later, we may assume that $\pi$ has weight $(0, 0)_{\tau} \in (\mathbb{Z}_+^2)^{\mathrm{Hom}(F, \mathbb{C})}$. Thus by \cite[Theorem 4.4.1]{CW} and the base change result Theorem \ref{base changeIII} between automorphic representations of $U(\mathbb{A}_{F^+})$ and $\mathrm{GL}_2(\mathbb{A}_{F})$, in order to prove $\mathcal{A}_{U}(K^wK_w, \mathcal{V}_{\lambda^w}(\sigma))_{\mathfrak{m}} \neq 0$ for sufficiently small $K_w$, it suffices to prove that $\mathrm{Spec}R_{\overline{\rho}|_{G_{F_{\tilde{v}}}}, \mathcal{D}_v}(\mathcal{O})$ is non-empty for any $v \in S$, $v \neq w$. This follows from the formally smoothness of our local lifting rings $R_{\overline{\rho}|_{G_{F_{\tilde{v}}}}, \mathcal{D}_v}$. \end{proof}

By using the same method as \cite{GEKI}, we can prove the following.

\begin{thm}\label{eigenspace}
    
1 \ $f : R_{\overline{\rho}, \mathcal{D}} \ \Isom \ \mathbb{T}^S_U(K^w, \mathcal{V}_{\lambda^w}(\sigma))_{\mathfrak{m}}$.
    
2 \ $\widehat{\mathcal{A}}_U(K^w, V_{\lambda^w}(\sigma))_{\mathfrak{m}}[\varphi] \neq 0$ for any $\mathcal{O}$-morphism $\varphi : R_{\overline{\rho}, \mathcal{D}} \rightarrow \mathcal{O}$.

\end{thm}

\begin{rem}

The author expects that we can also prove similar results for locally symmetric spaces of unitary (similitude) groups over $F^+$ by almost the same argument.

\end{rem}

\begin{proof}

This can be proved by exactly the same argument as \cite[Corollary 8.17]{GEKI}, which proved the same result for the completed cohomology of Shimura sets and Shimura curves over totally real fields. Our global situation is bit different from theirs and thus we need to check that the same proof also works in our situation. 

For $v \in T_1$, let $\alpha_v$ be a eigenvalue of $\overline{\rho}_{\mathfrak{m}}(\mathrm{Frob}_{\tilde{v}})$, let $U_{v}:=[\mathrm{Iw}_{\tilde{v}, 1} \begin{pmatrix}
    \varpi_{\tilde{v}} & 0 \\
    0 & 1 \end{pmatrix} \mathrm{Iw}_{\tilde{v}, 1}]$, where $\varpi_{\tilde{v}}$ denotes a uniformizer. Let $\mathfrak{m}'$ be the maximal ideal of $\mathbb{T}^S[U_v]_{v \in T_1}$ defined by the inverse image of $(\mathfrak{m}, U_v - \alpha_v)$ in $\mathbb{T}^S/\varpi[U_v]_{v \in T_1}$. Let $\gamma$ be the localization of $(\varinjlim_{K_w}\mathcal{A}_U(K^wK_w, \mathcal{V}_{\lambda^w}(\sigma)))[\mathfrak{m}]$ at $\mathfrak{m}'$. This is also nonzero as in \cite[{\S} 6.5]{EmertonGeeSavitt}. 
    
    Since our local lifting rings are formally smooth, in order to prove this theorem, it suffices to prove that the Gelfand-Kirillov dimension $\mathrm{dim}_{\mathrm{GL}_2(F_w)}\gamma$ of $\gamma$ is $d$ by the work of \cite{GN}. (See \cite[Theorem 8.15 and Corollary 8.17]{GEKI} for details.) \cite[Theorem 4.21]{GEKI} gave certain sufficient conditions (a), (b) and (c) for a mod $p$ representation $\gamma'$ of $\mathrm{GL}_2(F_w)$ to get $\mathrm{dim}_{\mathrm{GL}_2(F_w)}\gamma' = d$. We will check that $\gamma$ satisfies these conditions (a), (b) and (c).

First, we will prove that $\gamma$ satisfies the condition (a) of \cite[Theorem 4.21]{GEKI} by the same method as \cite{GEKI}. They used \cite[(i) of Proposition 8.9]{GEKI} to check the condition (a). \cite[(i) of Proposition 8.9]{GEKI} follows from the strong multiplicity one theorem, the automorphy lifting theorem in the Fontaine-Laffaille case and the weight part of the Serre conjecture, which is also known in our case by (e) of Theorem \ref{base changeIII}, \cite[Lemma 4.3.10 and (2) of Lemma 4.4.2]{GK} and \cite{GLS}. Note that in their situation, $T_2$ and $S(B)$ are empty, but exactly the same proof works. Thus we obtain the condition (a) of \cite[Theorem 4.21]{GEKI} for $\gamma$.

Second, we will prove that $\gamma$ satisfies the condition (b) of \cite[Theorem 4.21]{GEKI} by the same method as \cite{GEKI}. In order to prove the condition (b) of \cite[Theorem 4.21]{GEKI}, they used \cite[(i) of Proposition 8.9]{GEKI} and the fact that the patched module $M_{\infty}$ is a finite projective $S_{\infty}[[\mathrm{GL}_2(\mathcal{O}_{F_w})]]$-module as in the proof of \cite[Corollary 8.17]{GEKI}, which still holds in our situation by \cite[Lemma 3.3.1]{CHT} for example. (See \cite[{\S} 8.1]{GEKI} for the definitions of the patched module $M_{\infty}$ and $S_{\infty}$.) Thus we obtain the condition (b) of \cite[Theorem 4.21]{GEKI} for $\gamma$.

Finally, we will prove that $\gamma$ satisfies the condition (c) of \cite[Theorem 4.21]{GEKI} by the same method as \cite{GEKI}. They used \cite[Proposition 8.6, (ii) of Proposition 8.9 and Proposition 8.10]{GEKI} to check the condition (c) of \cite[Theorem 4.21]{GEKI}. In the following, by recalling the proofs of \cite[Proposition 8.6, (ii) of Proposition 8.9 and Proposition 8.10]{GEKI}, we will check that the same proof works in our situation.

\cite[Proposition 8.6]{GEKI} is a property about local deformation rings at $w$ and thus this also holds in our situation. Precisely, they fixed the determinant character, but exactly the same proof works.

To prove \cite[(ii) of Proposition 8.9]{GEKI}, they used \cite[Theorem 5.1]{wild}, which follows from \cite[Theorem 4.9]{wild}. We recall the proof of \cite[Theorem 4.9]{wild}. He used \cite[Propositions 4.6 and 4.7 and Lemma 4.8]{wild} to prove \cite[Theorem 4.9]{wild}. Note that he also used results in \cite[{\S} 2 and 3]{wild} in the proof, but they are purely local results on representations of $G_{F_{w}}$ and $\mathrm{GL}_2(F_w)$. Thus we can use them in our situation. \cite[Lemma 4.8]{wild} is a purely algebraic lemma and thus we can use this result. \cite[Proposition 4.7]{wild} follows from \cite[proof of Theorem 9.1.1]{EmertonGeeSavitt}, which also hold in our situation by \cite[Remark 9.1.2]{EmertonGeeSavitt}. \cite[Proposition 4.6]{wild} follows from \cite[Lemmas 4.3, 4.4 and 4.5]{wild}. \cite[Lemma 4.3]{wild} follows from \cite[proof of Theorem 10.1.1]{EmertonGeeSavitt}, which also hold in our situation by \cite[Remark 10.1.2]{EmertonGeeSavitt}. \cite[Proposition 4.4]{wild} follows from the results in \cite[{\S} 2 and 3]{wild}, which can be used in our situation as we have seen above. \cite[Lemma 4.5]{wild} is a purely algebraic lemma and thus we can use this result.

%and thus it suffices to check that the global arguments in \cite[{\S} 4]{wild} works. The global results in \cite[{\S} 4]{wild} are \cite[Lemma 4.3, Proposition 4.6, Proposition 4.7 and Theorem 4.9]{wild}. \cite[Proposition 4.6]{wild} is a formal consequence of \cite[Lemma 4.3]{wild} and purely local results proved in \cite{wild}. \cite[Proposition 4.9]{wild} is a formal consequece of \cite[Propositions 4.6 and 4.7]{wild} and purely local results. \cite[Lemma 4.3 and Proposition 4.7]{wild} follows from \cite[proof of Theorem 9.1.1 and Theorem 10.1.1]{EmertonGeeSavitt}, which also hold in our situation by \cite[Remark 9.1.2 and Remark 10.1.2]{EmertonGeeSavitt}. 

To prove \cite[Proposition 8.10]{GEKI}, they used \cite[Proposition 7.4 and (i) of Proposition 8.9]{GEKI} (and purely local results \cite[Proposition 5.16]{GEKI} on representations of $\mathrm{GL}_2(F_w)$, which can be used in our situation). We already gave some explanations about \cite[(i) of Proposition 8.9]{GEKI}. \cite[Proposition 7.4]{GEKI} follows from \cite[Corollary 7.40]{BD}, which was also stated to hold in our situation. Note that in \cite[Corollary 7.40]{BD}, they assumed $F_w = \mathbb{Q}_p$, but this assumption is not needed as explained in \cite[proof of Proposition 7.4]{GEKI}. \end{proof}

In the following, we come back to the situation of {\S} 6.2. In addition to the previous assumptions, we assume the following conditions.

\begin{itemize}
\item Every prime $r$ ramified in $F$ splits in an imaginary quadratic field in $F$.
\item There exists a prime $l \neq p$ splitting completely in $F$ such that all $l$-adic places $v$ of $F$ are contained in $S(B)_2$.
\item $\overline{\rho}_{\mathfrak{m}}(G_{F})$ is not solvable.
\item $\overline{\rho}_{\mathfrak{m}}$ is decomposed generic.
\item $\otimes_{\tau \in \Psi} (\overline{\rho}_{\mathfrak{m}}|_{G_{\tilde{F}}})^{\tau}$ is absolutely irreducible, where $\tilde{F}$ is the Galois closure of $F$ over $\mathbb{Q}$. See the comment before Theorem \ref{kottwitz conjecture} for the definition of $(\overline{\rho}_{\mathfrak{m}}|_{G_{\tilde{F}}})^{\tau}$.
\end{itemize}

We will compare the completed cohomology of $GU$ and $U$. Let $K_{GU} = \prod_{r}K_{GU, r}$ be an open compact subgroup of $GU(\mathbb{A}_{\mathbb{Q}}^{\infty})$ defined by the following.

\begin{itemize}
\item For any prime $r$ not lying below $S$ and unramified in $F$, let $K_{GU, r}$ be a hyperspecial maximal compact subgroup containing $K_{r}$.
\item For any prime $r$ not lying below $S$ and ramified in $F$, let $K_{GU, r} := \mathcal{O}_{F_{1, u}}^{\times} \times \prod_{w \mid u} \mathrm{GL}_2(\mathcal{O}_{F_{w}})$, where $F_1$ is a imaginary quadratic field in which $r$ splits, $u$ denotes a place of $F_1$ lying above $r$.
\item For any prime $r$ lying below $S$, let $K_{GU, r} = (1 + r^{2}\mathbb{Z}_r) \times \prod_{u \mid r}K_{u} \subset \mathbb{Q}_r^{\times} \times \mathrm{Res}_{F^+/\mathbb{Q}}U(\mathbb{Q}_r)$.
\end{itemize}

Note that we can regard $\mathcal{V}_{\lambda^w}(\sigma)$ as an $\mathcal{O}[K_{GU}]$-module. Note that $\mathbb{T}^S_{U}$ naturally acts on $\widehat{\mathcal{A}}_{GU}(K^w_{GU}, \mathcal{V}_{\lambda^w}(\sigma))$ and we have a natural $\mathbb{T}^S_U$-equivariant injection $\widehat{\mathcal{A}}_{U}(K^w, \mathcal{V}_{\lambda^{w}}(\sigma)) \hookrightarrow \widehat{\mathcal{A}}_{GU}(K^w_{GU}, \mathcal{V}_{\lambda^{w}}(\sigma))$ as in \cite[Lemma 2.1.1]{CS2}. 

Let $$\widehat{\mathcal{A}}_{GU}(K^w_{GU}, \mathcal{V}_{\lambda^w}(\sigma))_{\mathfrak{m}} := \varprojlim_{n} \varinjlim_{K_{GU, w}} \mathcal{A}_{GU}(K^w_{GU}K_{GU, w}, \mathcal{V}_{\lambda^w}(\sigma)/\varpi^n)_{\mathfrak{m}}$$ and $$\widehat{\mathcal{A}}_{GU}(K^w_{GU}, V_{\lambda^w}(\sigma))_{\mathfrak{m}} := \widehat{\mathcal{A}}_{GU}(K^w_{GU}, \mathcal{V}_{\lambda^w}(\sigma))_{\mathfrak{m}}[\frac{1}{p}].$$

Let $\mathbb{T}^S_{GU}(K^w_{GU}, \mathcal{V}_{\lambda^w}(\sigma))_{\mathfrak{m}} := \varprojlim_{K_{GU, w}} \mathrm{Im}(\mathbb{T}^S_{GU} \rightarrow \mathrm{End}_{\mathcal{O}}(\mathcal{A}_{GU}(K^w_{GU}K_{GU, w}, \mathcal{V}_{\lambda^w}(\sigma))_{\mathfrak{m}}))$ and $\mathbb{T}^S_{U}(K^w_{GU}, \mathcal{V}_{\lambda^w}(\sigma))_{\mathfrak{m}} := \varprojlim_{K_{GU, w}} \mathrm{Im}(\mathbb{T}^S_{U} \rightarrow \mathrm{End}_{\mathcal{O}}(\mathcal{A}_{GU}(K^w_{GU}K_{GU, w}, \mathcal{V}_{\lambda^w}(\sigma))_{\mathfrak{m}}))$, where $\mathbb{T}^S_{GU}$ is the subalgebra of the Hecke algebra $\mathcal{H}(GU(\mathbb{A}_{\mathbb{Q}}^{\infty}), K_{GU})_{\mathcal{O}}$ generated over $\mathcal{O}$ by $\mathcal{H}(GU(\mathbb{Q}_r), K_{GU, r})_{\mathcal{O}}$ for all primes $r$ splitting in $F_0$ and not lying below a place in $S$. Then we have the following canonical maps $\mathbb{T}^S_{U}(K^w_{U}, \mathcal{V}_{\lambda^w}(\sigma))_{\mathfrak{m}} \twoheadleftarrow \mathbb{T}^S_{U}(K^w_{GU}, \mathcal{V}_{\lambda^w}(\sigma))_{\mathfrak{m}} \hookrightarrow \mathbb{T}^S_{GU}(K^w_{GU}, \mathcal{V}_{\lambda^w}(\sigma))_{\mathfrak{m}}$.

\vspace{0.5 \baselineskip}

Note that we also have the universal representation $\rho_{\mathfrak{m}} : G_{F} \rightarrow \mathrm{GL}_2(\mathbb{T}^S_{U}(K^w_{GU}, \mathcal{V}_{\lambda^w}(\sigma))_{\mathfrak{m}})$ and the universal character $\chi_{\mathfrak{m}} : G_{F_0} \rightarrow \mathbb{T}^S_{GU}(K^w_{GU}, \mathcal{V}_{\lambda^w}(\sigma))_{\mathfrak{m}}^{\times}$ as in Theorem \ref{Galois completed cohomology}. We put $\rho_{\varphi} := \varphi_*\rho_{\mathfrak{m}}$ and $\chi_{\varphi} := \varphi_*\chi_{\mathfrak{m}}$ for an $\mathcal{O}$-morphism $\varphi : \mathbb{T}^S_{GU}(K^w_{GU}, \mathcal{V}_{\lambda^{w}}(\sigma))_{\mathfrak{m}} \rightarrow \mathcal{O}$ . Let $v \mid p$ be the place of $F_0$ induced by $\iota : \overline{\mathbb{Q}}_p \Isom \mathbb{C}$ and we fix a weight $\lambda_w \in (\mathbb{Z}_+^2)^{\mathrm{Hom}_{\mathbb{Q}_p}(F_w, E)}$.

\begin{lem} \label{Hecke eigensystem}

Let $\varphi : \mathbb{T}^S_U(K^w, \mathcal{V}_{\lambda^{w}}(\sigma))_{\mathfrak{m}} \rightarrow \mathcal{O}$ be an $\mathcal{O}$-morphism such that $\rho_{\varphi}|_{G_{F_{w'}}}$ is de Rham of $p$-adic Hodge type $\lambda_{w'}$ for any $w' \mid v$ and let $\mathfrak{p} := \mathrm{Ker}(\varphi)$. We regard $\mathfrak{p}$ as a prime ideal of $\mathbb{T}^S_{U}(K^w_{GU}, \mathcal{V}_{\lambda^w}(\sigma))$.

1 \ Then $\mathbb{T}^S_{GU}(K^w_{GU}, \mathcal{V}_{\lambda^w}(\sigma))_{\mathfrak{m}}/\mathfrak{p}\mathbb{T}^S_{GU}(K^w_{GU}, \mathcal{V}_{\lambda^w}(\sigma))_{\mathfrak{m}}$ is a finite $\mathcal{O}$-algebra and $\chi_{\mathfrak{p}} := \chi_{\mathfrak{m}} \mod \mathfrak{p} : G_{F} \rightarrow (\mathbb{T}^S_{GU}(K^w_{GU}, \mathcal{V}_{\lambda^w}(\sigma))_{\mathfrak{m}}/\mathfrak{p}\mathbb{T}^S_{GU}(K^w_{GU}, \mathcal{V}_{\lambda^w}(\sigma))_{\mathfrak{m}})^{\times}$ satisfies $\chi_{\mathfrak{p}} (g_{v^c}) = 1$ and $\chi_{\mathfrak{p}} (g_{v}) = \mathrm{Art}_{F_{v}}^{-1}(g_{v})^{\sum_{\tau \in \Phi}(\lambda_{\tau, 1} + \lambda_{\tau, 2})}$ for any $g_v \in U_v$ and $g_{v^c} \in U_{v^c}$, where $U_v$ and $U_{v^c}$ are some open subgroups of $I_{F_v}$ and $I_{F_{v^c}}$.

2 \ After enlarging $E$ if necessary, there exists an $E$-morphism $$\tilde{\varphi} : \mathbb{T}^S_{GU}(K^w_{GU}, \mathcal{V}_{\lambda^{w}}(\sigma))_{\mathfrak{m}}/\mathfrak{p}\mathbb{T}^S_{GU}(K^w_{GU}, \mathcal{V}_{\lambda^w}(\sigma))_{\mathfrak{m}} \rightarrow E$$ satisfing $\widehat{\mathcal{A}}_{GU}(K^w_{GU}, V_{\lambda^{w}}(\sigma))_{\mathfrak{m}}[\tilde{\varphi}] \neq 0$  

\end{lem}

\begin{proof}

1 \ The finiteness property in 1 follows from the description of $\chi_{\mathfrak{p}}$ in 1. In fact, by \cite[Lemma A.2.5]{CW}, after enlarging $E$ if necessary, we can construct a de Rham character $\psi : G_{F_0} \rightarrow \mathcal{O}^{\times}$ such that $\psi|_{U_{v}} = \chi_{\mathfrak{p}}|_{U_{v}}$ and $\psi|_{U_{v^c}} = \chi_{\mathfrak{p}}|_{U_{v^c}}$. In particular, $\chi_{\mathfrak{p}}\psi^{-1}$ has a finite ramification at $v$ and $v^c$. By our assumption on $S$ and Theorem \ref{base changeII}, the character $\chi_{\mathfrak{m}}$ has a finite ramification at any non-$p$-adic place. Thus $\chi_{\mathfrak{p}}\psi^{-1}$ has a finite image by global class field theory. Note that $\chi_{\mathfrak{p}}$ induces a surjection $\mathcal{O}[[G_{F_0, S}]] \twoheadrightarrow \mathbb{T}^S(K^w, \mathcal{V}_{\lambda^w}(\sigma))_{\mathfrak{m}}/\mathfrak{p}\mathbb{T}^S(K^w, \mathcal{V}_{\lambda^w}(\sigma))_{\mathfrak{m}}$ and the image of this map is invariant if we replace $\chi_{\mathfrak{p}}$ by $\chi_{\mathfrak{p}}\psi^{-1}$. This implies the finiteness property in 1.

Let us prove the description of $\chi_{\mathfrak{p}}$ in 1. By the construction of $\chi_{\mathfrak{m}} : G_{F_0, S} \rightarrow \mathbb{T}^S_{GU}(K^w_{GU}, \mathcal{V}_{\lambda^w})_{\mathfrak{m}}^{\times}$, we have $\chi_{\mathfrak{m}}(g_{v^c}) = 1$ for any $g_{v^c} \in U_{v^c}$, where $U_{v^c}$ is an open subgroup of $I_{F_{v^c}}$. By the relation $\chi_{\mathfrak{m}}^c/\chi_{\mathfrak{m}} \circ \mathrm{Art}_{F_{0}} = \varepsilon_p \mathrm{det}\rho_{\mathfrak{m}} \circ \mathrm{Art}_{F}|_{\mathbb{A}_{F_0}^{\infty}}$, we obtain an open subgroup $U_v$ of $I_{F_{0,v}}$ such that $\chi_{\mathfrak{p}}(g_{v}) = \mathrm{Art}_{F_{v}}^{-1}(g_{v})^{\sum_{\tau \in \Phi}(\lambda_{\tau, 1} + \lambda_{\tau,2})}$ for any $g_v \in U_v$. 

2 \ We have $\widehat{\mathcal{A}}_{GU}(K^w_{GU}, V_{\lambda^{w}}(\sigma))[\varphi] \neq 0$ by Theorem \ref{eigenspace}. Thus the result follows from the finiteness result 1. \end{proof}

By lemma \ref{Hecke eigensystem}, we can take a decomposed generic non-Eisenstein ideal $\mathfrak{m}'$ of $\mathbb{T}^S_{GU}$ lying above $\mathfrak{m}$. We consider the situation {\S} 6.2. Let $\Psi := \mathrm{Gal}(F_w/\mathbb{Q}_p)$. We fix $\overline{\mathbb{Q}}_l \Isom \mathbb{C}$, where $l \neq p$ denotes a prime splitting completely in $F$ such that all $l$-adic places $v$ of $F$ are contained in $S(B)_2$. Let $v_0 \mid l$ (resp. $w_0 \mid l$) be the place $F_0$ (resp. $F$) induced by this isomorphism and we have a natural injection $\Psi \hookrightarrow \{ w' : \mathrm{place \ of \ } F \mid w' \mathrm{\ divides \ } v_0 \}, \ \tau \mapsto w_0^{\tau}$. For any non-empty subset $\Psi' \subset \Psi$, let $S(B_{\Psi'}) = S(B) \setminus \{ w_0^{\tau}, w_0^{c\tau} \mid \tau \in \Psi' \}$. Then we have a PEL datum $(B_{\Psi'}, *_{\Psi'}, V_{\Psi'}, \psi_{\Psi'}, h_{\Psi'})$ and thus a Shimura variety $S_{\Psi', K}$ defined by the PEL data $(B_{\Psi'}, *_{\Psi'}, V_{\Psi'}, \psi_{\Psi'}, h_{\Psi'})$ as in {\S} 3.1 and 3.2. Let $\mu_{\tau} \in \mathbb{Z}_{\ge 0}$ for any $\tau \in \mathrm{Hom}_{\mathbb{Q}_p}(F_w, \overline{\mathbb{Q}}_p) = \Psi$ and let $\lambda$ be the element $(\lambda_{\tau})_{\tau}$ of $(\mathbb{Z}_+^2)^{\Phi}$ defined by $\lambda_{\tau} = (\lambda_{w'})_{\tau}$ for any $w' \mid v$, $w' \neq w$ and any $\tau \in \mathrm{Hom}_{\mathbb{Q}_p}(F_{w'}, \overline{\mathbb{Q}}_p)$ and $\lambda_{\tau} = (0, -\mu_{\tau})$ for any $\tau \in \mathrm{Hom}_{\mathbb{Q}_p}(F_w, \overline{\mathbb{Q}}_p)$.

\begin{cor} \label{BIGCL} Assume Conjectures \ref{classicality conjecture} and \ref{key diagram} for $S_{\Psi', K}$, $\mathfrak{m}'$ and $\lambda$ for any $\Psi' \subset \Psi$. Let $\rho : G_{F, S} \rightarrow \mathrm{GL}_2(\mathcal{O})$ lifting of $\overline{\rho}$ satisfying the following conditions.

(1) \ There exists a perfect symmetric $G_{F}$-equivalent pairing $\rho \times \rho^c \rightarrow \varepsilon_p^{-1}$.

(2) \ $\rho|_{G_{F_{\tilde{w'}}}}$ defines a point of $\mathrm{Spec} R^{\mathrm{st}}_{\overline{\rho}|_{G_{F_{\tilde{w'}}}}}(\mathcal{O})$ for $w' \in S(B)_1$.

(3) \ $\rho|_{G_{F_{\tilde{w'}}}}$ is Fontaine-Laffaille of $p$-adic Hodge type $\lambda_{w'}$ for any $w' \mid v$, $w' \neq w$.

(4) \ $\rho|_{G_{F_w}}$ is de Rham of $p$-adic Hodge type $\lambda_w$.

Then there exists a conjugate self-dual cohomological cuspidal automorphic representation $\pi$ of $\mathrm{GL}_2(\mathbb{A}_F)$ such that $\rho \cong r_{\iota}(\pi)$.

\end{cor}

\begin{proof} 

By Theorem \ref{eigenspace}, we obtain an $\mathcal{O}$-morphism $\varphi : R_{\overline{\rho}, \mathcal{D}} = \mathbb{T}^S_U(K^w, \mathcal{V}_{\lambda^w}(\sigma))_{\mathfrak{m}} \rightarrow \mathcal{O}$ corresponding to $\rho$. Then after extending $E$ if necessary, there exists an extension $\tilde{\varphi} : \mathbb{T}^S_{GU}(K^w_{GU}, \mathcal{V}_{\lambda^w}(\sigma))_{\mathfrak{m}} \rightarrow \mathcal{O}$ of $\varphi$ such that $\widehat{\mathcal{A}}_{GU}(K^w_{GU}, \mathcal{V}_{\lambda^{w}}(\sigma))_{\mathfrak{m}}[\tilde{\varphi}] \neq 0$ by Lemma \ref{Hecke eigensystem}. By shrinking the level $K_{S(B)_2 \cup T_2}$, we have $\widehat{\mathcal{A}}_{GU}(K^w_{GU}, \mathcal{V}_{\lambda^{w}}(\sigma))_{\mathfrak{m}} = \widehat{\mathcal{A}}_{GU}(K^w_{GU}, \mathcal{V}_{\lambda^{w}})_{\mathfrak{m}} \otimes \sigma$. Thus $\widehat{H}^{[F_w : \mathbb{Q}_p]}(S_{\Psi, K_1^w}, V_{\lambda^w})_{\mathfrak{m}}[\tilde{\varphi}] \neq 0$ for some level $K_1^w$ by Theorem \ref{comparison1}. The result follows from Theorem \ref{conjectual classicality theorem}. \end{proof}

\subsection{Applications to the automorphy lifting theorem and the Breuil-M$\mathrm{\acute{\textbf{e}}}$zard conjecture}

First, we recall the definition of Serre weights and the Breuil-M$\mathrm{\acute{e}}$zard conjecture. Let $p$ be an odd prime and $L/\mathbb{Q}_p$ be a finite extension and $E$, $\mathcal{O}$ and $\mathbb{F}$ as in the previous subsection such that $\mathrm{Hom}_{\mathbb{Q}_p}(L, \overline{\mathbb{Q}}_p) = \mathrm{Hom}_{\mathbb{Q}_p}(L, E)$.

Let $W_L$ be the subset of $(\mathbb{Z}_+^2)^{\mathrm{Hom}(\mathbb{F}_L, \mathbb{F}_E)}$ consisting of $k \in (\mathbb{Z}_+^2)^{\mathrm{Hom}(\mathbb{F}_L, \mathbb{F}_E)}$ satisfying that $k_{\gamma, 2}$ and $k_{\gamma, 1} - k_{\gamma, 2}$ are contained in $[0, p-1]$ for any $\gamma \in \mathrm{Hom}(\mathbb{F}_L, \mathbb{F}_E)$ and $k_{\gamma, 2} \neq p-1$ for some $\gamma \in \mathrm{Hom}(\mathbb{F}_L, \mathbb{F}_E)$. Note that $W_L$ has a bijection to the set of isomorphism classes of irreducible representations of $\mathrm{GL}_2(\mathbb{F}_L)$ via $k = (k_{\gamma, 1}, k_{\gamma, 2}) \mapsto \sigma_k := \otimes_{\gamma \in \mathrm{Hom}(\mathbb{F}_L, \mathbb{F}_E)} ((\wedge^2 \mathbb{F}_E^{\oplus 2})^{\otimes k_{\gamma, 2}} \otimes \mathrm{Sym}^{k_{\gamma, 1} - k_{\gamma, 2}}\mathbb{F}_E^{\oplus 2})$ by \cite[Appendix]{Her}. Here, we regard $(\wedge^2 \mathbb{F}_E^{\oplus 2})^{\otimes k_{\gamma, 2}} \otimes \mathrm{Sym}^{k_{\gamma, 1} - k_{\gamma, 2}}\mathbb{F}_E^{\oplus 2}$ as a representation of $\mathrm{GL}_2(\mathbb{F}_L)$ by $\tau : \mathbb{F}_L \hookrightarrow \mathbb{F}_E$.

For any inertia type $\tau$, after extending $E$ if necessary, we can take an $I_L$-stable lattice of $\tau$ and $\mathrm{GL}_2(\mathcal{O}_L)$-stable lattices of $\sigma^{\mathrm{crys}}(\tau)$ and $\sigma^{\mathrm{ss}}(\tau)$ over $\mathcal{O}$ and we also write $\tau$, $\sigma^{\mathrm{crys}}(\tau)$ and $\sigma^{\mathrm{ss}}(\tau)$ for them. For $\lambda \in (\mathbb{Z}^2_+)^{\mathrm{Hom}_{\mathbb{Q}_p}(L, \overline{\mathbb{Q}}_p)}$, we put $\overline{\mathcal{V}_{\lambda} \otimes_{\mathcal{O}} \sigma^{\mathrm{crys}}(\tau)}^{\mathrm{ss}} = \oplus_{k \in W_L} \sigma_k^{n_k^{\mathrm{crys}}(\lambda, \tau)}$ and $\overline{\mathcal{V}_{\lambda} \otimes_{\mathcal{O}} \sigma^{\mathrm{ss}}(\tau)}^{\mathrm{ss}} = \oplus_{k \in W_L} \sigma_k^{n_k^{\mathrm{ss}}(\lambda, \tau)}$. Here, $\mathcal{V}_{\lambda}$ is as before, i.e., a $\mathrm{GL}_2(\mathcal{O}_L)$-stable lattice over $\mathcal{O}$ of the irreducible algebraic representation of $\mathrm{Res}_{L/\mathbb{Q}_p}\mathrm{GL}_{2, L}(E)$ defined by $\lambda$. In the above formula, $\mathrm{ss}$ denotes the semisimplification as $\mathbb{F}_E[\mathrm{GL}_2(\mathcal{O}_L)]$-modules. Note that every irreducible continuous representation of $\mathrm{GL}_2(\mathcal{O}_L)$ over $\mathbb{F}_E$ comes from an irreducible representation of $\mathrm{GL}_2(\mathbb{F}_L)$ over $\mathbb{F}_E$ by \cite[Theorem 16]{Serre}.

For any $\gamma \in \mathrm{Hom}(\mathbb{F}_L, \mathbb{F}_E)$, we take a lift $\tilde{\gamma} \in \mathrm{Hom}_{\mathbb{Q}_p}(L, E)$ and let $W_L \hookrightarrow (\mathbb{Z}_+^2)^{\mathrm{Hom}_{\mathbb{Q}_p}(L, E)}, \ k \mapsto \tilde{k}$ be the embedding defined by for any $\gamma \in \mathrm{Hom}(\mathbb{F}_L, \mathbb{F}_E)$, $\tilde{k}_{\tilde{\gamma}} := k_{\gamma}$ and $\tilde{k}_{\gamma'} := (0, 0)$ if $\tilde{\gamma} \neq \gamma'$ and $\gamma'$ is a lift of $\gamma$.

\vspace{0.5 \baselineskip}

Let $\overline{\rho}_L : G_{L} \rightarrow \mathrm{GL}_2(\mathbb{F})$ be a continuous representation. 

\begin{dfn}\label{Serre weight}

Let $W_L(\overline{\rho}_L)$ be the subset of $W_L$ consisting of $k \in W_L$ such that there exists a crystalline lift $\rho : G_{L} \rightarrow \mathrm{GL}_2(\overline{\mathbb{Z}}_p)$ of $\overline{\rho}_L$ of $p$-adic Hodge type $\tilde{k}$. 
        
\end{dfn}

\begin{rem}

Actually, $W_{L}(\overline{\rho}_L)$ doesn't depend on the choice of $\tilde{\gamma}$ by the following theorem.

\end{rem}

    \begin{thm} \label{Serre weight2}
    
    1 \ There exists a unique element $(\mu_k(\overline{\rho}_L))_k$ of $(\mathbb{Z}_{\ge 0})^{W_L}$ such that $e(R_{\overline{\rho}_L}^{\mathrm{crys}, 0, \tau}/\varpi) = \sum_{k \in W_L} n^{\mathrm{crys}}_k(0, \tau) \mu_{k}(\overline{\rho}_L)$ for any inertia type $\tau$. Here, $e$ denotes the Hilbert-Samuel multiplicity. (See \cite[{\S} 1.3]{FMRI} for the definition.)
    
    2 \ We have $W_L(\overline{\rho}_L) = \{ k \in W_L \mid \mu_k(\overline{\rho}_L) \neq 0 \}$.

3 \ If $L/\mathbb{Q}_p$ is unramified and $k \in W_L$ satisfies $0 \le k_{\gamma, 1} - k_{\gamma, 2} \le p-3$ for any $\gamma \in \mathrm{Hom}(\mathbb{F}_L, \mathbb{F}_E)$, then $\mu_k(\overline{\rho}_L) = 0$ or $1$.

    \end{thm}
    
    \begin{proof} See \cite[Theorem A]{GK} and \cite[Theorem 6.1.8]{GLS}. \end{proof}

The following is the Breuil-M$\mathrm{\acute{e}}$zard conjecture in the two-dimensional case.

\begin{conj}\label{Breuil-Mezard conjecture}(Breuil-M$\acute{e}$zard conjecture)

For any weight $\lambda \in (\mathbb{Z}_+^2)^{\mathrm{Hom}_{\mathbb{Q}_p}(L, E)}$ and any inertia type $\tau$, we have the following.

1 \ $e(R^{\mathrm{crys}, \lambda, \tau}_{\overline{\rho}_L}/\varpi) = \sum_{k \in W_L} n_k^{\mathrm{crys}}(\lambda, \tau) \mu_k(\overline{\rho}_L)$.

2 \ $e(R^{\mathrm{ss}, \lambda, \tau}_{\overline{\rho}_L}/\varpi) = \sum_{k \in W_L} n_k^{\mathrm{ss}}(\lambda, \tau) \mu_k(\overline{\rho}_L)$.

\end{conj}

\begin{thm}\label{genericity}
    
1 \ For any $k \in W_L$ satisfying $k_{\gamma, 1} - k_{\gamma, 2} \neq p - 1$ for some $\gamma$, there exists a continuous representation $\overline{\rho}_k : G_{L} \rightarrow \mathrm{GL}_2(\overline{\mathbb{F}})$ such that $W_L(\overline{\rho}_k) = \{ k \}$ and $\overline{\rho}_{k}$ gives a smooth point of $\mathcal{X}^{k}_2$, where $\mathcal{X}^{k}_2$ is the irreducible component of the reduced Emerton-Gee stack labeled by $k$ defined in \cite[Theorem 6.5.1]{EGS}. Moreover, if we have $2 \le k_{\gamma, 1} - k_{\gamma, 2} \le p-5$ for any $\gamma$, we may assume that $\overline{\rho}_k$ is non-split reducible sufficiently generic.

2 \ Let $\lambda \in (\mathbb{Z}_+^2)^{\mathrm{Hom}_{\mathbb{Q}_p}(L, E)}$ be a weight and let $\tau$ be a inertia type. Assume that for any $k \in W_L(\overline{\rho}_L)$, we have $k_{\gamma, 1} - k_{\gamma, 2} \neq p - 1$ for some $\gamma$ and there exists $\overline{\rho}_k$ as in 1 such that the Breuil-M$\mathrm{\acute{e}}$zard conjecture holds for $\overline{\rho}_k$, $\lambda$ and $\tau$. Then the Breuil-M$\mathrm{\acute{e}}$zard conjecture for $\overline{\rho}_L, \lambda$ and $\tau$.

\end{thm}

\begin{proof} The first statement of 1 follows from the relation between irreducible components of the Emerton-Gee stack and Serre weights proved in \cite[Theorem 8.6.2]{EGS}. The second statement of 1 follows from the description of $\mathcal{X}^k$ given in \cite[2 of Theorem 5.5.12]{EGS}. 2 follows from exactly the same argument as \cite[{\S} 8.3]{EGS}. \end{proof}

Roughly speaking, \cite{FMRI} and \cite{BM} showed that the Breuil-M$\mathrm{\acute{e}}$zard conjecture is equivalent to the automorphy lifting theorem. In this subsection, we slightly modify their formulation.

\vspace{0.5 \baselineskip}

We consider the following objects and assumptions.

\begin{itemize}
    \item $p \ge 3$.
    \item $F_0$ is an imaginary quadratic field in which $p$ splits.
    \item $F^+/\mathbb{Q}$ is a totally real field such that $F^+ \neq \mathbb{Q}$.
    \item $F:=F_0F^+$ and assume that $F/F^+$ is unramified at any finite place and $\zeta_p \notin F$.
    \item $\Phi := \mathrm{Hom}_{F_0}(F, \mathbb{C})$. 
    \item $S$ is a finite set of finite places of $F^+$ splitting in $F$.
    \item $S_p$ is the set of $p$-adic places of $F^+$.
    \item $S_p^{\mathrm{ss}}$ and $S_{p}^{\mathrm{crys}}$ are subsets of $S_p$ such that $S_p = S_{p}^{\mathrm{ss}} \sqcup S_{p}^{\mathrm{crys}}$.
    \item $R$ is a subset of $S$ such that $S_p \cap R$ is empty.
    \item $T_1$ and $T_2$ are subsets of $S$ such that $S \setminus S_p \cup R = T_1 \sqcup T_2$.
    \item For any $v \in S$, we fix a lift $\tilde{v}$ to $F$ of $v$ and we put $\tilde{S}:=\{ \tilde{v} \mid v \in S \}$.
    \item For any $v \mid p$, we fix $\lambda_v \in (\mathbb{Z}_+^2)^{\mathrm{Hom}_{\mathbb{Q}_p}(F_{\tilde{v}}, E)}$ and an inertia type $\tau_v : I_{F_{\tilde{v}}} \rightarrow \mathrm{GL}_2(\mathcal{O})$.
    \item $\overline{\rho} : G_{F, S} \rightarrow \mathrm{GL}_2(\mathbb{F})$ is a continuous representation such that there exists a perfect $G_{F, S}$-equivariant symmetric pairing $\overline{\rho} \times \overline{\rho}^c \rightarrow \overline{\varepsilon_p^{-1}}$.
    \item $\overline{\rho}|_{G_{F(\zeta_p)}}$ is irreducible.
    \item $T_1$ contains two places $v_1, v_2$ such that $\mathrm{char}\mathbb{F}_{v_1} \neq \mathrm{char}\mathbb{F}_{v_2}$ and for any $v \in T_1$, $\overline{\rho}|_{G_{F_{\tilde{v}}}}$ is generic unramified.
    \item For any $v \in T_2$, $\overline{\rho}|_{G_{F_{\tilde{v}}}}$ is generic irreducible.
    \item For any $v \in R$, $\overline{\rho}|_{G_{F_{\tilde{v}}}}$ is trivial.
    \item There exists a conjugate self-dual cohomological automorphic representation $\pi$ of $\mathrm{GL}_2(\mathbb{A}_{F})$ such that $\overline{\rho} \cong \overline{r_{\iota}(\pi)}$.
    \end{itemize}
    
    Let $V := S_p^{\mathrm{crys}} \sqcup R$. Let $\mathcal{D}^{\lambda, \tau, V} = (F/F^+, S, \tilde{S}, \mathcal{O}, \overline{\rho}, \varepsilon_p^{-1}, \{ \mathcal{D}_v^{\lambda, \tau, V} \}_{v \in S})$ be the global deformation problem defined by the following conditions. (Precisely, in \cite{matsumoto}, he assumed that every local deformation problem $\mathcal{D}_v^{\lambda, \tau, V}$ is non-empty. In the following, we admit that $\mathcal{D}_v^{\lambda, \tau, V}$ may be empty and then we define the local lifting ring at $v$ and the global deformation ring to be the zero ring.)
        
    \begin{itemize}
    \item For any $v \in S \setminus (S_p \cup R)$, the local deformation problem $\mathcal{D}_v^{\lambda, \tau, V}$ consists of all deformations of $\overline{\rho}|_{G_{F_{\tilde{v}}}}$.
    \item For any $v \in S_p^{\mathrm{crys}}$, the local deformation problem $\mathcal{D}_v^{\lambda, \tau, V}$ is defined by $\mathrm{Spec}R_{\overline{\rho}|_{G_{F_{\tilde{v}}}}}^{\mathrm{crys}, \lambda_v, \tau_v}$.
    \item For any $v \in S_p^{\mathrm{ss}}$, the local deformation problem $\mathcal{D}_v^{\lambda, \tau, V}$ is defined by $\mathrm{Spec}R_{\overline{\rho}|_{G_{F_{\tilde{v}}}}}^{\mathrm{ss}, \lambda_v, \tau_v}$.
    \item For any $v \in R$, the local deformation problem $\mathcal{D}_v^{\lambda, \tau, V}$ is defined by $\mathrm{Spec}R_{\overline{\rho}|_{G_{F_{\tilde{v}}}}}^{\mathrm{unip}, \mathrm{red}}$, where $R_{\overline{\rho}|_{G_{F_{\tilde{v}}}}}^{\mathrm{unip}}$ is the unipotently ramified lifting ring and $R_{\overline{\rho}|_{G_{F_{\tilde{v}}}}}^{\mathrm{unip}, \mathrm{red}}$ is the reduction of $R_{\overline{\rho}|_{G_{F_{\tilde{v}}}}}^{\mathrm{unip}}$. ($R_{\overline{\rho}|_{G_{F_{\tilde{v}}}}}^{\mathrm{unip}}$ coincides with $R_{\overline{\rho}|_{G_{F_{\tilde{v}}}}}^{\mathrm{loc}}/\mathcal{I}^{(1, \cdots, 1)}$ of \cite[{\S} 3]{IA}. )
    \end{itemize}

Let $U/F^+$ be a unitary group defined by the data $S(B):=\emptyset$ and $\Psi:=\emptyset$ as in {\S} 3.1. Then $U(F^+ \otimes_{\mathbb{Q}} \mathbb{R}) \cong \prod_{v \mid \infty} U(0, 2)$. Let $K = \prod_{v} K_v$ be an open compact subgroup of $U(\mathbb{A}_{F^+}^{\infty})$ defined by the following.
        
\begin{itemize}
\item For any $v \in T_1$, let $K_v = \mathrm{Iw}_{\tilde{v}, 1}$.
\item For any $v \notin T_1$, let $K_v$ be a hyperspecial.
\end{itemize}

Let $\mathfrak{m}$ be a non-Eisenstein ideal of $\mathbb{T}^S$ corresponding to $\overline{\rho}$. For any $v \in T_2$, let $\sigma_v := \sigma^{\mathrm{ss}}(\rho_v|_{I_{F_{\tilde{v}}}})$ for a lift $\rho_v \in \mathrm{Spec}R_{\overline{\rho}|_{G_{F_{\tilde{v}}}}, \mathcal{D}_v}(\mathcal{O})$. Note that such a $\rho_v$ exists unique up to unramified twists by Lemma \ref{formally smooth deformation}. For $v \in R$, let $\sigma_v := \sigma^{\mathrm{ss}}(\bold{1})$ for the trivial representation $\bold{1}$.

We assume that for any $v \in S_p \cup R$, any irreducible component $\mathcal{C}_v$ of $\mathrm{Spec} R_{\overline{\rho}|_{G_{F_{\tilde{v}}}}, \mathcal{D}_v}$ and any finite extension $E'/E$, the base change $\mathcal{C}_v \otimes_{\mathcal{O}} \mathcal{O}_{E'}$ is irreducible. Moreover, for any $v \in T_2 \cup R$, we assume that $\sigma_v$ is defined over $E$ and we take a $K_v$-stable lattice of $\sigma_v$ over $\mathcal{O}$ and we also write $\sigma_v$ for this. We also assume that we can take a $\mathcal{O}$-lattice $\sigma^{\mathrm{ss}}(\tau_v)$ for any $v \in S_p^{\mathrm{ss}}$, $\sigma^{\mathrm{crys}}(\tau_v)$ for any $v \in S_p^{\mathrm{crys}}$ and $\mathcal{V}_{\lambda}$ as before. (Of course, these assumptions are satisfied if we can extend the field $E$. Later, we will need to be careful when we extend the field $E$ because we will use Theorem \ref{lifting prescribed}.) 

Let $\sigma_V := (\otimes_{v \in R \cup T_2} \sigma_v) \otimes (\otimes_{v \in S_p^{\mathrm{ss}}} \sigma^{\mathrm{ss}}(\tau_v)) \otimes (\otimes_{v \in S_p^{\mathrm{crys}}} \sigma^{\mathrm{crys}}(\tau_v))$. Then we can consider $\mathcal{A}_{U}(K, \mathcal{V}_{\lambda}(\sigma_V))_{\mathfrak{m}}$ and $\mathbb{T}^{S}(K, \mathcal{V}_{\lambda}(\sigma_V))_{\mathfrak{m}}$ as in {\S} 7.2. Moreover, as in Proposition \ref{deformation problem}, we have a natural surjection $R_{\overline{\rho}, \mathcal{D}^{\lambda, \tau, V}}^{\mathrm{red}} \twoheadrightarrow \mathbb{T}^S(K, \mathcal{V}_{\lambda}(\sigma_V))_{\mathfrak{m}}$ by using the property of $\sigma_v$'s. (See Theorem \ref{typetheory}.) 

\begin{prop}\label{congruence}

There exists a conjugate self-dual cohomological cuspical automorphic representation $\pi_1$ of $\mathrm{GL}_2(\mathbb{A}_{F})$ of weight $(0 , 0)_{\tau} \in (\mathbb{Z}^2_+)^{\mathrm{Hom}(F, \mathbb{C})}$ such that $\overline{r_{\iota}(\pi)} \cong \overline{r_{\iota}(\pi_1)}$.

\end{prop}

\begin{rem}

Note that $r_{\iota}(\pi_1)|_{G_{F_{\tilde{v}}}}$ is potentially diagonalizable if $r_{\iota}(\pi_1)|_{G_{F_{\tilde{v}}}}$ is potentially crystalline by \cite[Lemma 4.4.1]{GK} and $\iota$-ordinary otherwise by \cite[Lemma 5.6]{OG} for any $v \in S_p$. (See \cite[comments before Lemma 1.4.1]{CW} for the definition of potentially diagonalizable.)

\end{rem}

\begin{proof}  
    
The existence of $\pi$ implies that $\mathcal{A}_{U}(K', \mathcal{V}_{\lambda'})_{\mathfrak{m}} \neq 0$ for some $K'$ and some weight $\lambda'$ by using the base change Theorem \ref{base changeIII} and the decomposition of $\mathcal{A}_{U}(K', \mathcal{V}_{\lambda'})_{\mathfrak{m}}[\frac{1}{p}]$ by automorphic representations. Thus we obtain $\mathcal{A}_{U}(K', \mathcal{V}_{\lambda'}/\varpi)_{\mathfrak{m}} = \mathcal{A}_{U}(K', \mathbb{F})_{\mathfrak{m}} \otimes_{\mathbb{F}}  \mathcal{V}_{\lambda'}/\varpi \neq 0$ after replacing $K'$ by a sufficiently small open compact such that $K'$ acts on $\mathcal{V}_{\lambda'}/\varpi$ trivially. This implies $\mathcal{A}_{U}(K', \mathcal{O})_{\mathfrak{m}} \neq 0$ and we can take a cohomological cuspidal representation $\sigma$ of $U(\mathbb{A}_{F^+})$ contributing to $\mathcal{A}_{U}(K', \mathcal{O})_{\mathfrak{m}}[\frac{1}{p}]$. Note that $\pi_1 := \mathrm{BC}(\sigma)$ is cuspidal since $\mathfrak{m}$ is a non-Eisenstein ideal.  \end{proof}

\begin{prop}\label{equivalence}

We consider the following conditions.

1 \ $R_{\overline{\rho}, \mathcal{D}^{\lambda, \tau, V}}^{\mathrm{red}} \Isom \mathbb{T}^S(K, \mathcal{V}_{\lambda}(\sigma_V))_{\mathfrak{m}}$.

2 \ For any irreducible component $\mathcal{C} = \widehat{\otimes}_{v \in S_p \cup R}\mathcal{C}_v$ of $\mathrm{Spec}( \widehat{\otimes}_{v \in S_p \cup R}R_{\overline{\rho}|_{G_{F_{\tilde{v}}}}, \mathcal{D}_v^{\lambda, \tau, V}})$ such that there exists a conjugate self-dual cohomological cuspidal automorphic representation $\alpha$ of $\mathrm{GL}_2(\mathbb{A}_{F})$ such that $\overline{r_{\iota}(\alpha)} \cong \overline{\rho}$, $r_{\iota}(\alpha)|_{G_{F_{\tilde{v}}}} \in \mathcal{C}_v$ for any $v \in S_p \cup R$ and $\alpha$ is unramified outside $S$. Here, $\mathcal{C}_v$ is an irreducible component of $\mathrm{Spec}(R_{\overline{\rho}|_{G_{F_{\tilde{v}}}}, \mathcal{D}_v^{\lambda, \tau, V}})$ and $\widehat{\otimes}_{v \in S_p \cup R}\mathcal{C}_v$ denotes $\mathrm{Spec}\widehat{\otimes}_{v \in S_p \cup R}\Gamma(\mathcal{C}_v, \mathcal{O}_{\mathcal{C}_v})$. Note that this is irreducible and any irreducible component of $\mathrm{Spec}( \widehat{\otimes}_{v \in S_p \cup R}R_{\overline{\rho}|_{G_{F_{\tilde{v}}}}, \mathcal{D}_v^{\lambda, \tau, V}})$ has a form like this by \cite[Lemma 3.3]{Calabi} and our assumption that $\mathcal{C}_v \otimes_{\mathcal{O}} \mathcal{O}_{E'}$ is irreducible for any finite extension $E'/E$.

3 \ $e(R_{\overline{\rho}|_{G_{F_{\tilde{v}}}}}^{\mathrm{ss}, \lambda_v, \tau_v}/\varpi) = \sum_{k_v \in W_{F_v}} n_{k_v}^{\mathrm{ss}}(\lambda_v, \tau_v)\mu_{k_v}(\overline{\rho}|_{G_{F_{\tilde{v}}}})$ for any $v \in S_{p}^{\mathrm{ss}}$ and $e(R_{\overline{\rho}|_{G_{F_{\tilde{v}}}}}^{\mathrm{crys}, \lambda_v, \tau_v}/\varpi) = \sum_{k_v \in W_{F_v}} n_{k_v}^{\mathrm{crys}}(\lambda_v, \tau_v)\mu_{k_v}(\overline{\rho}|_{G_{F_{\tilde{v}}}})$ for any $v \in S_{p}^{\mathrm{crys}}$.

About the above conditions, we obtain the relations : $1 \Leftrightarrow 2 \Leftarrow 3$.

Moreover, we fix $w \mid p$ and assume the following condition. 

$\bullet$ $S_{p}^{\mathrm{ss}}$ is $\{ w \}$ or empty. Moreover, for any $p$-adic place $v \neq w$, we have the following conditions. 

(a) \ $F_v^+/\mathbb{Q}_p$ is unramified. 

(b) \ $\tau_v$ is the trivial representation $\bold{1}$. 

(c) \ $\lambda_{\gamma, 1} - \lambda_{\gamma, 2} \le p-3$ for any $\gamma \in \mathrm{Hom}_{\mathbb{Q}_p}(F_v, E)$. 

(d) \ $R_{\overline{\rho}|_{G_{F_{\tilde{v}}}}}^{\mathrm{crys}, \lambda, \bold{1}}$ is nonzero.

Then the above conditions $1, 2$ implies the following condition 3'.

3' \ $e(R_{\overline{\rho}|_{G_{F_{\tilde{w}}}}}^{\mathrm{ss}, \lambda_w, \tau_w}/\varpi) = \sum_{k_w \in W_{F_w}} n_{k_w}^{\mathrm{ss}}(\lambda_w, \tau_w)\mu_{k_w}(\overline{\rho}|_{G_{F_{\tilde{w}}}})$ if $w \in S_p^{\mathrm{ss}}$ and $e(R_{\overline{\rho}|_{G_{F_{\tilde{w}}}}}^{\mathrm{crys}, \lambda_w, \tau_w}/\varpi) = \sum_{k_w \in W_{F_w}} n_{k_w}^{\mathrm{crys}}(\lambda_w, \tau_w)\mu_{k_w}(\overline{\rho}|_{G_{F_{\tilde{w}}}})$ if $w \in S_{p}^{\mathrm{crys}}$.

\end{prop}

\begin{rem} Note that under the above assumption $\bullet$, $e(R_{\overline{\rho}|_{G_{F_{\tilde{w}}}}}^{\mathrm{ss}, \lambda_w, \tau_w}/\varpi) = \sum_{k_w \in W_{w}} n_{k_w}^{\mathrm{ss}}(\lambda_w, \tau_w)\mu_{k_w}(\overline{\rho}|_{G_{F_{\tilde{w}}}})$ implies $e(R_{\overline{\rho}|_{G_{F_{\tilde{w}}}}}^{\mathrm{crys}, \lambda_w, \tau_w}/\varpi) = \sum_{k_w \in W_w} n_{k_w}^{\mathrm{crys}}(\lambda_w, \tau_w)\mu_{k_w}(\overline{\rho}|_{G_{F_{\tilde{w}}}})$ because the condition 2 in the case $S_{p}^{\mathrm{ss}} = \{ w \}$ is stronger than the condition 2 in the case $S_{p}^{\mathrm{crys}} = \{ w \}$. \end{rem}

\begin{proof} The proof is the same as \cite[Lemma 4.3.9]{GK} or \cite[Lemma 5.5.1]{BM}, but our situation and statement are bit different from theirs. Thus we recall the proof.

We fix a representative of the universal deformation over $R_{\overline{\rho}, \mathcal{D}^{\lambda, \tau, V}}$ and get a map $R_{\overline{\rho}|_{G_{F_{\tilde{v}}}}, \mathcal{D}_v^{\lambda, \tau, V}} \rightarrow R_{\overline{\rho}, \mathcal{D}^{\lambda, \tau, V}}$. By the usual patching argument (see \cite[{\S} 5]{BM} for details), we obtain a complete regular local ring $S_{\infty}$ over $\mathcal{O}$ with a prime ideal $\mathfrak{a}$ such that $\mathcal{O} \Isom S_{\infty}/\mathfrak{a}$, a surjection of complete Noetherian local $S_{\infty}$-algebras $R_{\infty, \mathcal{D}^{\lambda, \tau, V}} \twoheadrightarrow \mathbb{T}_{\infty}(\mathcal{V}_{\lambda}(\sigma_V))_{\mathfrak{m}}$ and a finite faithful $\mathbb{T}_{\infty}(\mathcal{V}_{\lambda}(\sigma_V))_{\mathfrak{m}}$-module $M_{\infty, \lambda, \tau,  V}$ which is a finite free $S_{\infty}$-module satisfying the following.

(1) \ We have surjections $R_{\infty, \mathcal{D}^{\lambda, \tau, V}} \twoheadrightarrow R_{\overline{\rho}, \mathcal{D}^{\lambda, \tau, V}}$ and $\mathbb{T}_{\infty}(\mathcal{V}_{\lambda}(\sigma_V))_{\mathfrak{m}} \twoheadrightarrow \mathbb{T}^S(K, \mathcal{V}_{\lambda}(\sigma_V))_{\mathfrak{m}}$ of $S_{\infty}$-algebras fitting in the following commutative diagram. Here, we regard $\mathcal{O}$-algebras as $S_{\infty}$-algebras via $S_{\infty} \twoheadrightarrow S_{\infty}/\mathfrak{a} \Isom \mathcal{O}$.

\xymatrix{
R_{\infty, \mathcal{D}^{\lambda, \tau, V}} \ar[r] \ar[d] & \ar[d] R_{\overline{\rho}, \mathcal{D}^{\lambda, \tau, V}} \\
\mathbb{T}_{\infty}(\mathcal{V}_{\lambda}(\sigma_V))_{\mathfrak{m}} \ar[r] & \mathbb{T}^S(K, \mathcal{V}_{\lambda}(\sigma_V))_{\mathfrak{m}}
}

(2) \ We have an isomorphism $M_{\infty, \lambda, \tau, V} \otimes_{S_{\infty}} \mathcal{O} \cong \mathcal{A}_{U}(K, \mathcal{V}_{\lambda}(\sigma_V))_{\mathfrak{m}}$ which is compatible with $\mathbb{T}_{\infty}(\mathcal{V}_{\lambda}(\sigma_V))_{\mathfrak{m}} \twoheadrightarrow \mathbb{T}^S(K, \mathcal{V}_{\lambda}(\sigma_V))_{\mathfrak{m}}$.

(3) \ $R_{\infty, \mathcal{D}^{\lambda, \tau, V}}$ is equal to $(\widehat{\otimes}_{v \in S_p \cup R}R_{\overline{\rho}|_{G_{F_{\tilde{v}}}}, \mathcal{D}_v^{\lambda, \tau, V}})[[Y_1, \cdots, Y_s]]$ and this is equidimensional of dimension $\mathrm{dim}S_{\infty}$. Moreover, the map $R_{\infty, \mathcal{D}^{\lambda, \tau, V}} \twoheadrightarrow R_{\overline{\rho}, \mathcal{D}^{\lambda, \tau, V}}$ is a morphism of $\widehat{\otimes}_{v \in S_p \cup R}R_{\overline{\rho}|_{G_{F_{\tilde{v}}}}, \mathcal{D}_v^{\lambda, \tau, V}}$-algebras.

We claim that the image of $\mathrm{Spec}\mathbb{T}_{\infty}(\mathcal{V}_{\lambda}(\sigma_V))_{\mathfrak{m}} \hookrightarrow \mathrm{Spec}R_{\infty, \mathcal{D}^{\lambda, \tau, V}}$ is a union of irreducible components. - (*) To see this, we may assume that $\mathcal{A}_{U}(K, \mathcal{V}_{\lambda}(\sigma_V))_{\mathfrak{m}} \neq 0$. Thus we have \begin{align}\label{20} \mathrm{dim}R_{\infty, \mathcal{D}^{\lambda, \tau, V}} \ge \mathrm{dim}\mathbb{T}_{\infty}(\mathcal{V}_{\lambda}(\sigma_V))_{\mathfrak{m}} \ge \mathrm{inf}_{\mathfrak{p} \in \mathrm{Ass}_{\mathbb{T}_{\infty}(\mathcal{V}_{\lambda}(\sigma_V))_{\mathfrak{m}}}(M_{\infty, \lambda, \tau, V})} \mathrm{dim}\mathbb{T}_{\infty}(\mathcal{V}_{\lambda}(\sigma_V))_{\mathfrak{m}}/\mathfrak{p} \nonumber \\
  \ge \mathrm{depth}_{\mathbb{T}_{\infty}(\mathcal{V}_{\lambda}(\sigma_V))_{\mathfrak{m}}}(M_{\infty, \lambda, \tau, V}) \ge \mathrm{depth}_{S_{\infty}}(M_{\infty, \lambda, \tau, V}) = \mathrm{dim}S_{\infty} = \mathrm{dim}R_{\infty, \mathcal{D}^{\lambda, \tau, V}}.\end{align} This implies our claim because $M_{\infty, \lambda, \tau, V}$ is a faithful $\mathbb{T}_{\infty}(\mathcal{V}_{\lambda}(\sigma_V))_{\mathfrak{m}}$-module. 

In the following, we will prove that the conditions 1 and 2 are equivalent to the following condition.

\vspace{0.5 \baselineskip}

4 \ The closed immersion $\mathrm{Spec}\mathbb{T}_{\infty}(\mathcal{V}_{\lambda}(\sigma_V))_{\mathfrak{m}} \hookrightarrow \mathrm{Spec}R_{\infty, \mathcal{D}^{\lambda, \tau, V}}$ is a homeomorphism. 

\vspace{0.5 \baselineskip}

($2 \Rightarrow 4$) \ The condition 2 implies the homeomorphism $\mathrm{Spec}\mathbb{T}_{\infty}(\mathcal{V}_{\lambda}(\sigma_V))_{\mathfrak{m}} \Isom \mathrm{Spec}R_{\infty, \mathcal{D}^{\lambda, \tau, V}}$ because we have the above property (*), any automorphic Galois representation $r_{\iota}(\gamma)$ satisfying the condition 2 defines a regular point of $R_{\infty, \mathcal{D}^{\lambda, \tau, V}}$ by \cite[(2) of Proposition 1.2.2 and Theorem 1.2.7]{Allen} and contributes to $\mathrm{Spec}\mathbb{T}(K, \mathcal{V}_{\lambda}(\sigma_V))_{\mathfrak{m}}$ by the local-global compatibility Theorem \ref{Galois representation} and the base change Theorem \ref{base changeIII}.

($4 \Rightarrow 1$) \ Assume that the closed immersion $\mathrm{Spec}\mathbb{T}_{\infty}(\mathcal{V}_{\lambda}(\sigma_V))_{\mathfrak{m}} \hookrightarrow \mathrm{Spec}R_{\infty, \mathcal{D}^{\lambda, \tau, V}}$ is a homeomorphism. Then $\mathrm{Ann}_{R_{\infty, \mathcal{D}^{\lambda, \tau, V}}}M_{\infty, \lambda, \tau, V}$ is nilpotent. This implies that $\mathrm{Ann}_{R_{\overline{\rho}, \mathcal{D}^{\lambda, \tau, V}}}\mathcal{A}_{U}(K, \mathcal{V}_{\lambda}(\sigma_V))_{\mathfrak{m}}$ is nilpotent and thus we have $R_{\overline{\rho}, \mathcal{D}^{\lambda, \tau, V}}^{\mathrm{red}} \Isom \mathbb{T}^S(K, \mathcal{V}_{\lambda}(\sigma_V))_{\mathfrak{m}}$. 

(1 $\Rightarrow$ 2) \ We assume the condition 1. Let $\mathcal{C} = \widehat{\otimes}_{v \in S_p \cup R} \mathcal{C}_v$ be an irreducible component of $\mathrm{Spec}( \widehat{\otimes}_{v \in S_p \cup R}R_{\overline{\rho}|_{G_{F_{\tilde{v}}}}, \mathcal{D}_v^{\lambda, \tau, V}})$. Note that \cite[(2) of Proposition 1.5.1]{CW} says that the quotient ring $R_{\mathcal{C}}$ of $R_{\overline{\rho}, \mathcal{D}^{\lambda, \tau, V}}^{\mathrm{red}}$ defined by the ideal generated by the images of the minimal ideals of $R_{\overline{\rho}|_{G_{F_{\tilde{v}}}}, \mathcal{D}_v}$ corresponding to $\mathcal{C}_v$ has dimension $\ge 1$. Thus the condition 1 implies that $R_{\mathcal{C}}$ is finite over $\mathcal{O}$ and has dimension 1. Thus there exists a global representation $\rho \in \mathrm{Spec}R_{\mathcal{C}}(\mathcal{O})$ after extending $\mathcal{O}$. Then, by using the condition 1 again, we get $\alpha$ corresponding to $\rho$ as in the condition 2.

In the following, we will prove $4 \Rightarrow 3'$ under the assumption $\bullet$ and $3 \Rightarrow 4$. (We recall that the conditions 1 and 2 are equivalent to the condition 4.) First note that by the usual patching argument (see \cite[{\S} 5]{BM} for details), for any weight $\lambda'$, inertia type $\tau'$, $S_{p}^{\mathrm{crys}'}$ and $S_{p}^{\mathrm{ss}'}$, we can also construct the objects $R_{\infty, \mathcal{D}^{\lambda', \tau', V'}}$, $\mathbb{T}_{\infty}(\mathcal{V}_{\lambda'}(\sigma_{V'}))_{\mathfrak{m}}$, $M_{\infty, \lambda', \tau',  V'}$ as above. Moreover, for any $k = (k_v)_v \in \prod_{v \mid p} W_{F_v}$, there exist an $S_{\infty}/\varpi$-algebra $\overline{\mathbb{T}}_{\infty, k, R}$ and a faithful $\overline{\mathbb{T}}_{\infty, k, R}$-module $\overline{M}_{\infty, k, R}$ which is a finite free $S_{\infty}/\varpi$-module and doesn't depend on $\lambda'$, $\tau'$, $S_{p}^{\mathrm{crys}'}$ and $S_{p}^{\mathrm{ss}'}$ satisfying the following. 

\vspace{0.5 \baselineskip}

(4) \ For any $\lambda'$, $\tau'$, $S_{p}^{\mathrm{crys}'}$ and $S_{p}^{\mathrm{ss}'}$, we have an $S_{\infty}$-algebra surjection $\mathbb{T}_{\infty}(\mathcal{V}_{\lambda'}(\sigma_{V'}))_{\mathfrak{m}} \twoheadrightarrow \overline{\mathbb{T}}_{\infty, k, R}$ and we have a $\mathbb{T}_{\infty}(\mathcal{V}_{\lambda'}(\sigma_{V'}))_{\mathfrak{m}}$-stable exhaustive filtration $\{ \mathrm{Fil}^t \}_{t=1}^m$ on $M_{\infty, \lambda', \tau',  V'}/\varpi$ such that $\oplus_{t=1}^m \mathrm{gr}^i \cong \oplus_{k = (k_v)_v \in \prod_{v \mid p} W_{F_v}} \overline{M}_{\infty, k, R}^{\oplus ((\prod_{v \in S_{p}^{\mathrm{crys}}} n_{k_v}^{\mathrm{crys}}(\lambda'_v, \tau'_v))(\prod_{v \in S_{p}^{\mathrm{ss}}} n_{k_v}^{\mathrm{ss}}(\lambda'_v, \tau'_v)))}$ as $\mathbb{T}_{\infty}(\mathcal{V}_{\lambda'}(\sigma_{V'}))_{\mathfrak{m}}$-modules. 

\vspace{0.5 \baselineskip}

We claim that if $\mathcal{A}_{U}(K, \mathcal{V}_{\lambda'}(\sigma_{V'}))_{\mathfrak{m}} \neq 0$, the rank of the finite free $\mathbb{T}^{S}(K, \mathcal{V}_{\lambda'}(\sigma_{V'}))_{\mathfrak{m}}[\frac{1}{p}]$-module $\mathcal{A}_{U}(K, \mathcal{V}_{\lambda'}(\sigma_{V'}))_{\mathfrak{m}}[\frac{1}{p}]$ is $2^{|T_1|}$. By using the decomposition by automorphic representations Theorem \ref{Zuker}, Theorem \ref{typetheory} and (e) of Theorem \ref{base changeIII}, it suffices to prove that for $v \in T_1$ and unramified characters $\chi_1, \chi_2 : F_{\tilde{v}}^{\times} \rightarrow \mathbb{C}^{\times}$, we have $\mathrm{dim}_{\mathbb{C}}(\mathrm{n}$-$\mathrm{Ind}^{\mathrm{GL}_2(F_{\tilde{v}})}_{B_2(F_{\tilde{v}})} \chi_1 \boxtimes \chi_2)^{\mathrm{Iw}_{\tilde{v}, 1}} = 2$. (Here, $B_2$ denotes the Borel subgroup of $\mathrm{GL}_2$ consisting of upper triangular matrices. Note that we assume that $\overline{\rho}|_{G_{F_{\tilde{v}}}}$ is generic unramified for any $v \in T_1$. Thus every lifting of $\overline{\rho}|_{G_{F_{\tilde{v}}}}$ to $\mathcal{O}$ is generic unramified by $H^1_f(G_{F_{\tilde{v}}}, \mathrm{ad}\overline{\rho}|_{G_{F_{\tilde{v}}}}) \Isom H^1(G_{F_{\tilde{v}}}, \mathrm{ad}\overline{\rho}|_{G_{F_{\tilde{v}}}})$.) By Iwasawa decomposition $\mathrm{GL}_2(F_{\tilde{v}}) = B_2(F_{\tilde{v}}) \mathrm{GL}_2(\mathcal{O}_{F_{\tilde{v}}})$, we have $\mathrm{n}$-$\mathrm{Ind}^{\mathrm{GL}_2(F_{\tilde{v}})}_{B_2(F_{\tilde{v}})} \chi_1 \boxtimes \chi_2 \cong \mathrm{Ind}^{\mathrm{GL}_2(\mathcal{O}_{F_{\tilde{v}}})}_{B_2(\mathcal{O}_{F_{\tilde{v}}})}\chi_1|_{\mathcal{O}_{F_{\tilde{v}}}^{\times}} \boxtimes \chi_2|_{\mathcal{O}_{F_{\tilde{v}}}^{\times}}$ as a representation of $\mathrm{GL}_2(\mathcal{O}_{F_{\tilde{v}}})$. By Iwahori decomposition $\mathrm{GL}_2(\mathcal{O}_{F_{\tilde{v}}}) = \mathrm{Iw}_{\tilde{v}} \sqcup \mathrm{Iw}_{\tilde{v}} \begin{pmatrix}
    0 & 1 \\
     1 & 0 \end{pmatrix} \mathrm{Iw}_{\tilde{v}} =  B_2(\mathcal{O}_{F_{\tilde{v}}})\mathrm{Iw}_{\tilde{v}, 1} \sqcup  B_2(\mathcal{O}_{F_{\tilde{v}}}) \begin{pmatrix}
    0 & 1 \\
     1 & 0 \end{pmatrix} \mathrm{Iw}_{\tilde{v}, 1}$, the dimension of $(\mathrm{Ind}^{\mathrm{GL}_2(\mathcal{O}_{F_w})}_{B_2(\mathcal{O}_{F_w})}\chi_1|_{\mathcal{O}_{F_w}^{\times}} \boxtimes \chi_2|_{\mathcal{O}_{F_w}^{\times}})^{\mathrm{Iw}_{\tilde{v}, 1}}$ is equal to 2. Thus we obtain the claim.

For any irreducible component $X$ of $\mathrm{Spec} \mathbb{T}_{\infty}(\mathcal{V}_{\lambda'}(\sigma_{V'}))_{\mathfrak{m}}[\frac{1}{p}]$, a natural finite map $X \rightarrow \mathrm{Spec}S_{\infty}[\frac{1}{p}]$ is a surjection by counting the dimensions. Note that the fiber of $\mathfrak{a}_{\infty}$ in $\mathrm{Spec} \mathbb{T}_{\infty}(\mathcal{V}_{\lambda'}(\sigma_{V'}))_{\mathfrak{m}}[\frac{1}{p}]$ is homeomorphic to $\mathrm{Spec}\mathbb{T}^{S}(K, \mathcal{V}_{\lambda'}(\sigma_{V'}))_{\mathfrak{m}}[\frac{1}{p}]$. By taking a fiber at $\mathfrak{a}_{\infty}$, we have $$X \cap \mathrm{Spec}\mathbb{T}^{S}(K, \mathcal{V}_{\lambda'}(\sigma_{V'}))_{\mathfrak{m}}[\frac{1}{p}] \neq \emptyset.$$ By \cite[(2) of Proposition 1.2.2 and Theorem 1.2.7]{Allen} again, every point of $X \cap \mathrm{Spec}\mathbb{T}^{S}(K, \mathcal{V}_{\lambda'}(\sigma_{V'}))_{\mathfrak{m}}[\frac{1}{p}]$ defines a regular point of $\mathrm{Spec} \mathbb{T}_{\infty}(\mathcal{V}_{\lambda'}(\sigma_{V'}))_{\mathfrak{m}}[\frac{1}{p}] \subset \mathrm{Spec}R_{\infty, \mathcal{D}^{\lambda', \tau', V'}}$. Thus $M_{\infty, \lambda', \tau', V'}$ is a generically finite free $\mathbb{T}_{\infty}(\mathcal{V}_{\lambda'}(\sigma_{V'}))_{\mathfrak{m}}$-module of rank $2^{|T_1|}$. 

By the same calculation as above (\ref{20}), we see that the image of $\mathrm{Spec}\overline{\mathbb{T}}_{\infty, k, R} \hookrightarrow \mathrm{Spec}\mathbb{T}_{\infty}(\mathcal{V}_{\lambda'}(\sigma_{V'}))_{\mathfrak{m}}/\varpi$ is a union of the irreducible components.

%Note that $\mathbb{T}_{\infty}(\mathcal{V}_{\lambda'}(\sigma'_V))_{\mathfrak{m}}$ is generically reduced because the image of $\mathrm{Spec}\mathbb{T}_{\infty}(\mathcal{V}_{\lambda'}(\sigma_V'))_{\mathfrak{m}} \hookrightarrow \mathrm{Spec}R_{\infty, \mathcal{D}^{\lambda', \tau', V}}$ is a union of irreducible components and $R_{\infty, \mathcal{D}^{\lambda', \tau', V}}$ is generically reduced by \cite[Proposition 3.1.]{IA}. Thus $M_{\infty, \lambda', \tau', V}$ is a generically finite projective $\mathbb{T}_{\infty}(\mathcal{V}_{\lambda'}(\sigma'_V))_{\mathfrak{m}}$-module because $M_{\infty, \lambda', \tau', V}$ is a faithful $\mathbb{T}_{\infty}(\mathcal{V}_{\lambda'}(\sigma'_V))_{\mathfrak{m}}$-module. 

By \cite[Corollary 1.3.5]{FMRI} and the fact that $\mathbb{T}_{\infty}(\mathcal{V}_{\lambda'}(\sigma_{V'}))_{\mathfrak{m}}$ is $\varpi$-torsion free, we obtain $$e(M_{\infty, \lambda', \tau', V'}/\varpi M_{\infty, \lambda', \tau', V'}, \mathbb{T}_{\infty}(\mathcal{V}_{\lambda'}(\sigma_{V'}))_{\mathfrak{m}}/\varpi) = 2^{|T_1|}e(\mathbb{T}_{\infty}(\mathcal{V}_{\lambda'}(\sigma_{V'}))_{\mathfrak{m}}/\varpi).$$ Moreover, by \cite[Proposition 1.3.2]{FMRI} and the definition of $e$, we have \begin{align*}e(M_{\infty, \lambda', \tau', V'}/\varpi M_{\infty, \lambda', \tau', V'}, \mathbb{T}_{\infty}(\mathcal{V}_{\lambda'}(\sigma_{V'}))_{\mathfrak{m}}/\varpi) \\
 = \sum_{k = (k_v)_v \in \prod_{v \mid p} W_{F_v}} ((\prod_{v \in S_{p}^{\mathrm{crys}}} n_{k_v}^{\mathrm{crys}}(\lambda'_v, \tau'_v))(\prod_{v \in S_{p}^{\mathrm{ss}}} n_{k_v}^{\mathrm{ss}}(\lambda'_v, \tau'_v)))e(\overline{M}_{\infty, k, R}, \mathbb{T}_{\infty}(\mathcal{V}_{\lambda'}(\sigma_{V'}))_{\mathfrak{m}}/\varpi) \\
 = \sum_{k = (k_v)_v \in \prod_{v \mid p} W_{F_v}} ((\prod_{v \in S_{p}^{\mathrm{crys}}} n_{k_v}^{\mathrm{crys}}(\lambda'_v, \tau'_v))(\prod_{v \in S_{p}^{\mathrm{ss}}} n_{k_v}^{\mathrm{ss}}(\lambda'_v, \tau'_v)))e(\overline{M}_{\infty, k, R}, \overline{\mathbb{T}_{\infty, k, R}}).\end{align*}

Thus we obtain \begin{equation}\label{30}e(\mathbb{T}_{\infty}(\mathcal{V}_{\lambda'}(\sigma_{V'}))_{\mathfrak{m}}/\varpi) = \frac{1}{2^{|T_1|}}\sum_{k = (k_v)_v \in \prod_{v \mid p} W_{F_v}} ((\prod_{v \in S_{p}^{\mathrm{crys}}} n_{k_v}^{\mathrm{crys}}(\lambda'_v, \tau'_v))(\prod_{v \in S_{p}^{\mathrm{ss}}} n_{k_v}^{\mathrm{ss}}(\lambda'_v, \tau'_v)))e(\overline{M}_{\infty, k, R}, \overline{\mathbb{T}_{\infty, k, R}}).\end{equation}

On the other hand, by \cite[Proposition 1.3.7]{FMRI}, we have \begin{align}\label{40} e(R_{\infty, \mathcal{D}^{\lambda', \tau', V'}}/\varpi) = \prod_{v \in S}e(R_{\overline{\rho}|_{G_{F_{\tilde{v}}}}, \mathcal{D}^{\lambda', \tau', V'}_v}/\varpi) \nonumber \\
= (\prod_{v \in S_p^{\mathrm{crys}}}e(R_{\overline{\rho}|_{G_{F_{\tilde{v}}}}}^{\mathrm{crys}, \lambda_{v}', \tau_v'}/\varpi))(\prod_{v \in S_p^{\mathrm{crys}}}e(R_{\overline{\rho}|_{G_{F_{\tilde{v}}}}}^{\mathrm{ss}, \lambda_{v}', \tau_v'}/\varpi))N, \end{align} where $N$ is a positive integer not depending on $\lambda'$, $\tau'$, $S_p^{\mathrm{crys}'}$ and $S_p^{\mathrm{ss}'}$.

Note that by the definition of $e$, the above property (*) and the $\varpi$-torsion freeness of $\mathbb{T}_{\infty}(\mathcal{V}_{\lambda'}(\sigma_{V'}))_{\mathfrak{m}}$, we have $e(\mathbb{T}_{\infty}(\mathcal{V}_{\lambda'}(\sigma_{V'}))_{\mathfrak{m}}/\varpi) \le e(R_{\infty, \mathcal{D}^{\lambda', \tau', V'}}/\varpi)$ and this is an equality if and only if the surjection $R_{\infty, \mathcal{D}^{\lambda', \tau', V'}} \twoheadrightarrow \mathbb{T}_{\infty}(\mathcal{V}_{\lambda'}(\sigma_{V'}))_{\mathfrak{m}}$ is an isomorphism. (Note that $R_{\infty, \mathcal{D}^{\lambda', \tau', V'}}$ is reduced by \cite[Lemma 3.3]{Calabi}. Thus the isomorphy of $R_{\infty, \mathcal{D}^{\lambda', \tau', V'}} \twoheadrightarrow \mathbb{T}_{\infty}(\mathcal{V}_{\lambda'}(\sigma_{V'}))_{\mathfrak{m}}$ is equivalent to the homeomorphy of $\mathrm{Spec} R_{\infty, \mathcal{D}^{\lambda', \tau', V'}} \hookrightarrow \mathrm{Spec} \mathbb{T}_{\infty}(\mathcal{V}_{\lambda'}(\sigma_{V'}))_{\mathfrak{m}}$.)

Then by Proposition \ref{congruence} and \cite[Theorem 4.4.1]{CW}, we obtain the condition 2 under the assumption that $S_p^{\mathrm{crys}}$ is equal to $S_p$ and $\lambda$ is trivial. Thus (\ref{30}) and (\ref{40}) imply \begin{align*} \frac{1}{2^{|T_1|}}\sum_{k = (k_v)_v \in \prod_{v \mid p} W_{F_v}} (\prod_{v \in S_{p}} n_{k_v}^{\mathrm{crys}}(0, \tau'_v))e(\overline{M}_{\infty, k, R}, \overline{\mathbb{T}_{\infty, k, R}}) \\
= (\prod_{v \in S_p}e(R_{\overline{\rho}|_{G_{F_{\tilde{v}}}}}^{\mathrm{crys}, 0, \tau'_v}/\varpi))N = (\prod_{v \in S_p}(\sum_{k_v \in W_{F_v}} n^{\mathrm{crys}}_{k_v}(0, \tau_v') \mu_{k_v}(\overline{\rho}|_{G_{F_{\tilde{v}}}})))N \\
= \sum_{k = (k_v) \in \prod_v W_{F_v}} (\prod_{v \in S_p}n^{\mathrm{crys}}_{k_v}(0, \tau_v')\mu_{k_v}(\overline{\rho}|_{G_{F_{\tilde{v}}}}))N 
\end{align*} by 1 of Theorem \ref{Serre weight2}. By using \cite[Theorem 33]{Serre} as in \cite[proof of Lemma 4.5.2]{GK}, we obtain $\frac{1}{2^{|T_1|}}e(\overline{M}_{\infty, k, R}, \overline{\mathbb{T}_{\infty, k, R}}) 
= N\prod_{v \mid p}\mu_{k_v}(\overline{\rho}|_{G_{F_{\tilde{v}}}}).$

Thus (\ref{30}) implies $$e(\mathbb{T}_{\infty}(\mathcal{V}_{\lambda'}(\sigma_{V'}))_{\mathfrak{m}}/\varpi) =  (\prod_{v \in S_{p}^{\mathrm{ss}}}(\sum_{k_v \in W_{F_v}} n_{k_v}^{\mathrm{ss}}(\lambda_v', \tau_v')\mu_{k_v}(\overline{\rho}|_{G_{F_{\tilde{v}}}})))(\prod_{v \in S_{p}^{\mathrm{crys}}}(\sum_{k \in W_{F_v}} n_{k_v}^{\mathrm{crys}}(\lambda_v', \tau_v')\mu_{k_v}(\overline{\rho}|_{G_{F_{\tilde{v}}}})))N.$$ 

Thus $3$ implies 4. We assume the condition $\bullet$. Then by the Fontaine-Laffaille theory (see \cite[Lemma 2.4.1]{CHT}), for any $p$-adic place $v \neq w$, we have that $R_{\overline{\rho}|_{G_{F_{\tilde{v}}}}}^{\mathrm{crys}, \lambda_v, \bold{1}}$ is a formal power series ring over $\mathcal{O}$. Thus $e(R_{\overline{\rho}|_{G_{F_{\tilde{v}}}}}^{\mathrm{crys}, \lambda_v, \bold{1}}/\varpi) = 1$. On the other hand, we have $\sum_{k \in W_{F_v}} n_{k_v}^{\mathrm{crys}}(\lambda_v, \bold{1})\mu_{k_v}(\overline{\rho}|_{G_{F_{\tilde{v}}}}) = 1$ by the definition of $n_{k_v}^{\mathrm{crys}}(\lambda_v, \bold{1})$ and 2 and 3 of Theorem \ref{Serre weight2}. Thus we obtain the implication $4 \Rightarrow 3'$. \end{proof}

In the following, we will prove the Breuil-M$\mathrm{\acute{e}}$zard conjecture in some 2-dimensional cases by using the above proposition under some conjectures. 

We consider the following objects. 

\begin{itemize}
\item $L$ is an unramified extension of $\mathbb{Q}_p$.
\item $d := [L, \mathbb{Q}_p]$.
\item $E, \mathcal{O}$ and $\mathbb{F}$ are as in the begining of this subsection.
\item $\overline{\rho}_L : G_{L} \rightarrow \mathrm{GL}_2(\mathbb{F})$ is a sufficiently generic reducible non-split representation. 
\item $\lambda_L := (0, -\lambda_{\tau}) \in (\mathbb{Z}_+^2)^{\mathrm{Hom}_{\mathbb{Q}_p}(L, \overline{\mathbb{Q}}_p)}$.
\item $\tau_L : I_{L} \rightarrow \mathrm{GL}_2(\overline{\mathbb{Q}}_p)$ is an inertia type.
\end{itemize}

After extending $E$, we may assume the following conditions.

\begin{itemize}
\item Every irreducible component $\mathcal{C}$ of $\mathrm{Spec}R_{\overline{\rho}_L}^{\mathrm{ss}, \lambda_L, \tau_L}$ satisfies the condition that $\mathcal{C} \otimes_{\mathcal{O}} \mathcal{O}_{E'}$ is irreducible for any finite extension $E'/E$.
\item For any irreducible component $\mathcal{C}$ of $\mathrm{Spec}R_{\overline{\rho}_{L}}^{\mathrm{ss}, \lambda_L, \tau_L}$, there exists a generic lifting $G_{L} \rightarrow \mathrm{GL}_2(\mathcal{O})$ of $\overline{\rho}_{L}$ contained in $\mathcal{C}$.
\item $\tau_L$, $\sigma^{\mathrm{ss}}(\tau_L)$ and $\sigma^{\mathrm{crys}}(\tau_L)$ are defined over $\mathcal{O}$.
\end{itemize}

We fix pairwise different odd primes $r_1$ and $r_2$ such that $r_i \neq p$ and $r_i^2 \not\equiv 1 \mod p$ for any $i = 1, 2$. By using the following lemma, after extending $E$, we obtain an irreducible generic representation $\overline{\rho}_1 : G_{\mathbb{Q}_{r_1}} \rightarrow \mathrm{GL}_2(\mathbb{F})$.

\begin{lem}

Let $l$ be a prime satisfying that $l \neq p$ and $l^2 \not\equiv 1 \mod p$. Then there exists an irreducible generic representation $\overline{\rho} : G_{\mathbb{Q}_{l}} \rightarrow \mathrm{GL}_2(\overline{\mathbb{F}})$.

\end{lem}

\begin{proof} We fix a Frobenius lift $\phi \in G_{\mathbb{Q}_l}$. Let $\psi : G_{\mathbb{Q}_{l^2}} \rightarrow \overline{\mathbb{F}}^{\times}$ be a character $G_{\mathbb{Q}_{l^2}} \rightarrow G_{\mathbb{Q}_{l^2}}^{\mathrm{ab}} \Isom \widehat{\mathbb{Z}} \times \mathbb{Z}_{l^2}^{\times} \xrightarrow{\mathrm{pr}_2} \mathbb{Z}_{l^2}^{\times} \twoheadrightarrow \mathbb{F}_{l^2}^{\times} \hookrightarrow \overline{\mathbb{F}}$, where the second isomorphism is induced by $\phi^2$ and the last embedding is one arbitrary embedding $\mathbb{F}_{l^2}^{\times} \hookrightarrow \overline{\mathbb{F}}$. Then $\overline{\rho} := \mathrm{Ind}^{G_{\mathbb{Q}_l}}_{G_{\mathbb{Q}_{l^2}}}\psi$ is irreducible. Moreover, we have $\mathrm{Hom}_{\overline{\mathbb{F}}[G_{\mathbb{Q}_l}]} (\overline{\rho}, \overline{\rho}(1)) = \mathrm{Hom}_{\overline{\mathbb{F}}[G_{\mathbb{Q}_{l^2}}]} (\psi, \psi(1) \oplus \psi^{\phi}(1))$ by the Frobenius reciprocity, where $\sigma \in \mathrm{Gal}(\mathbb{Q}_{l^2}/\mathbb{Q}_l)$ denotes the nontrivial element. This is zero by the assumption $l^2 \not\equiv 1 \mod p$. Thus $\overline{\rho}$ is generic.  \end{proof}

We fix a Fontaine-Laffaille lift $\rho_L$ of $\overline{\rho}_L$ over $\mathcal{O}$, which exists by our genericity assumption on $\overline{\rho}_L$ and let $\lambda_L' \in (\mathbb{Z}_+^2)^{\mathrm{Hom}_{\mathbb{Q}_l}(L, \overline{\mathbb{Q}}_l)}$ denotes its $p$-adic Hodge type. By \cite[proofs of A.1 $\sim$ A.7]{BM}, after extending $E$ if necessary, we obtain the following globalization data of $L$ and $\overline{\rho}_L$.

\begin{itemize}
\item $F_0$ is an imaginary quadratic field in which $p$, $r_1$ and $r_2$ split.
\item $F^+/\mathbb{Q}$ is a Galois totally real field such that $F^+ \neq \mathbb{Q}$ such that $F_u^+ \cong L$ over $\mathbb{Q}_p$ for any $u \mid p$ and $r_i$ splits completely in $F^+$ for any $i=1, 2$. We fix $F_u^+ \cong L$ over $\mathbb{Q}_p$ for any $u \mid p$ in the following.
\item $F:=F_0F^+$ satisfies the condition that every ramified prime in $F$ splits in an imaginary quadtratic field in $F$. (We use \cite[Lemma 4.44]{matsumoto} to obtain this property.)
\item $\iota : \overline{\mathbb{Q}}_p \Isom \mathbb{C}$. Moreover, we also fix $\iota_i : \overline{\mathbb{Q}}_{r_i} \Isom \mathbb{C}$ for any $i = 1, 2$.
\item $\Phi := \mathrm{Gal}(F/F_0)$.
\item $w$ (resp. $v$) is a $p$-adic place of $F$ (resp. $F_0$) induced by $\iota$. 
\item For $i = 1, 2$, $w_i \mid r_i$ (resp. $v_i \mid r_i$) is the place of $F$ (resp. $F_0$) induced by $\iota_i$.
\item Let $S$ be the set of finite places of $F^+$ which lies above $p, r_1$ and $r_2$.
\item $\Psi := \mathrm{Gal}(F_w/\mathbb{Q}_p)$. We have $\Phi \neq \Psi$. We regard $\Psi$ as a subset of $\Phi$.
\item $\lambda$ is an element of $(\mathbb{Z}_+^{2})^{\Phi}$ defined by $\lambda_w = \lambda_L$ under the identification $L = F_w$ and $\lambda_{w'} = \lambda_L'$ for any $w' \mid v$, $w' \neq w$. We put $\tau_w := \tau_L$.
\item $\psi : G_{F, S} \rightarrow \mathcal{O}^{\times}$ such that $\psi \psi^c = \varepsilon_p^{-2}$ is a de Rham character and for any place $w \neq w' \mid p$ of $F^+$, $\psi|_{G_{F_{w}}}$ is crystalline, $\mathrm{WD}(\psi|_{G_{F_{w}}})|_{I_{F_{w}}} = \mathrm{det}\tau_w$ and $\psi|_{G_{F_{w'}}}$ has $p$-adic Hodge type $(\lambda_{\tau, 1} + \lambda_{\tau, 2} + 1)_{\tau \in \mathrm{Hom}_{\mathbb{Q}_p}(F_{w'}, \overline{\mathbb{Q}}_p)}$ for any place $w \neq w' \mid p$ of $F^+$. (We use \cite[Lemma A.2.5]{CW} to obtain this character and we need to extend $E$ only for this. Note that we may not have $\overline{\psi} = \mathrm{det}\overline{\rho}$ for the following $\overline{\rho}$.)
\item $\overline{\rho} : G_{F, S} \rightarrow \mathrm{GL}_2(\mathbb{F})$ is a continuous surjection such that there exists a perfect $G_{F, S}$-equivariant symmetric pairing $\overline{\rho} \times \overline{\rho}^c \rightarrow \overline{\varepsilon_p^{-1}}$, which is unramified outside finite places lying above $p, r_1$ and $r_2$.
\item For any $w' \mid v$, we have $\overline{\rho}|_{G_{w'}} \cong \overline{\rho}_{L}$.
\item For any $w' \mid v_1$, we have $\overline{\rho}|_{G_{F_{w'}}} \cong \overline{\rho}_{1}$.
\item $\otimes_{\tau \in \Psi} \overline{\rho}_{\mathfrak{m}}^{\tau}$ is absolutely irreducible. (We obtain this property if there exists a prime $l$ splitting completely in $F$ such that $\overline{\rho}|_{G_{F_{v}}}$ is irreducible for one $v \mid l$ and for any $1 \neq \tau \in \Psi$, $\overline{\rho}|_{G_{F_{{v}^{\tau}}}}$ is trivial. Note that by replacing $F^+$ by an appropriate solvable extension of $F^+$, we may assume that $\overline{\rho}$ is unramified outside finite places lying above $p, r_1$ and $r_2$.) \footnote{The author doesn't know whether we can take $\overline{\rho}$ satisfying that for any $1 \neq \tau \in \Phi$, $\overline{\rho}|_{G_{F_{{v}^{\tau}}}}$ is trivial because we used the Moret-Bailly theorem in the proof of \cite[Proposition 3.2]{Calegari}, which was used in the proof of \cite[Proposition A.2]{BM}.}
\item For one $w' \mid v_2$, $\overline{\rho}|_{G_{F_{w'}}}$ is a generic unramified representation with distinct Frobenius eigenvalues. Note that this implies that if $\overline{\rho} \cong \overline{\rho}_{\mathfrak{m}}$ for some non-Eisenstein ideal $\mathfrak{m}$, then $\mathfrak{m}$ is decomposed generic.
\item There exists a conjugate self-dual cohomological automorphic representation $\pi$ of $\mathrm{GL}_2(\mathbb{A}_{F})$ such that $\overline{\rho} \cong \overline{r_{\iota}(\pi)}$.
\end{itemize}

Let $i : \mathbb{F}^{\times} \hookrightarrow \mathcal{O}^{\times}$ be the Teichmuller map. By replacing $\psi$ by $\psi i(\overline{\psi}^{-1} \mathrm{det}\overline{\rho})$, we may assume that $\overline{\psi} = \mathrm{det}\overline{\rho}$. By Theorem \ref{lifting prescribed}, for any irreducible component $\mathcal{C}_w$ of $\mathrm{Spec}R_{\overline{\rho}|_{G_{F_w}}}^{\mathrm{ss}, \lambda_w, \tau_w}$, we can take a finite set $T_{\mathcal{C}_w, 3}$ of $\overline{\rho}$-nice places of $F^+$ and a lifting $\rho_{\mathcal{C}_w} : G_{F, S \cup T_{\mathcal{C}_w, 3}} \rightarrow \mathrm{GL}_2(\mathcal{O})$ of $\overline{\rho}$ satisfying the following.

\begin{itemize}
\item There exists a perfect symmetric pairing $\rho_{\mathcal{C}_w} \times \rho_{\mathcal{C}_w}^c \rightarrow \varepsilon_p^{-1}$.
\item $\rho_{\mathcal{C}_w}|_{G_{F_w}}$ is contained in $\mathcal{C}_w$.
\item For any $w' \mid v$, $w' \neq w$, $\rho_{\mathcal{C}_w}|_{G_{F_{w'}}}$ is Fontaine-Laffaille lift of $\overline{\rho}|_{G_{F_{w'}}}$ of $p$-adic Hodge type $\lambda_{w'}$.
\item $\rho_{\mathcal{C}_{w}}|_{G_{F^+_{v}}}$ defines a point of $\mathrm{Spec} R_{\overline{\rho}|_{G_{F^+_{v}}}}^{\mathrm{st}}(\mathcal{O})$ for any $v \in T_{\mathcal{C}_w, 3}$.
\end{itemize}

We consider a similar situation as {\S} 7.2. Let $S(B) := \{ w' \mathrm{ \ place \ of \ } F \mid w' \mathrm{ \ divides \ } r_1 \}$. We have a natural injection $\Psi \hookrightarrow \{ w' : \mathrm{place \ of \ } F \mid w' \mathrm{\ divides \ } v_1 \}, \ \tau \mapsto w_1^{\tau}$. For any non-empty subset $\Psi' \subset \Psi$, let $S(B_{\Psi'}) = S(B) \setminus \{ w_1^{\tau}, w_1^{c\tau} \mid \tau \in \Psi' \}$. From the data $\Psi' \subset \Psi$ and $S(B_{\Psi'})$, we have a unitary Shimura variety $S_{\Psi', K}$ as in {\S} 3.1 and 3.2. Moreover, for any irreducible component $\mathcal{C}_w$ of $\mathrm{Spec}R_{\overline{\rho}|_{G_{F_w}}}^{\mathrm{ss}, \lambda_w, \tau_w}$, we take a subset $S(B_{\mathcal{C}_w})' \subset S(B) \setminus \{ w_1^{\tau}, w_1^{c \tau} \mid \tau \in \Psi \}$ satisfying $S(B_{\mathcal{C}_w})' = S(B_{\mathcal{C}_w})'^c$ and $\frac{1}{2}|\tilde{T}_{3, \mathcal{C}_w} \sqcup S(B_{\mathcal{C}_w})'| \equiv 0 \mod 2$, where $\tilde{T}_{3, \mathcal{C}_w}$ denotes the set of places of $F$ lying above a place in $T_{3, \mathcal{C}_w}$. We put $S(B_{\mathcal{C}_w, \Psi'}) := \tilde{T}_{3, \mathcal{C}_w} \sqcup S(B_{\mathcal{C}_w})' \sqcup S(B_{\Psi'})$. From data $\Psi'$ and $S(B_{\mathcal{C}_w, \Psi'})$, we also obtain a unitary Shimura variety $S_{\mathcal{C}_w, \Psi', K}$.

\begin{prop}\label{BIGRT} Assume that for any $\Psi' \subset \Psi$ and $\mathcal{C}_w$, Conjectures \ref{classicality conjecture} and \ref{key diagram} hold for $\lambda$, $S_{\Psi', K}$, $S_{\Psi', \mathcal{C}_w, K}$ and $\mathfrak{m}$ corresponding to $\overline{\rho}$.
    
Then the Breuil-M$\acute{e}$zard conjecture holds for $\overline{\rho}_L$, $\lambda_L$ and $\tau_L$. \end{prop}

\begin{proof} We apply Proposition \ref{equivalence} to our situation by putting $S_p^{\mathrm{ss}} := \{ w \}$, $T_1 := \{ w' \mid r_3 \}$, $T_2 := \{ w' \mid r_1 \}$ and $R := \emptyset$. Note that the condition $\bullet$ holds in our situation by putting $\tau_{w'} = \bold{1}$ for any $w' \mid p$, $w' \neq w$. Thus for any irreducible component $\mathcal{C}_w$ of $\mathrm{Spec}R_{\overline{\rho}|_{G_{F_w}}}^{\mathrm{ss}, \lambda_w, \tau_w}$, it suffices to construct a conjugate self-dual cohomological cuspidal automorphic representation $\pi'_{\mathcal{C}_w}$ of $\mathrm{GL}_2(\mathbb{A}_F)$ such that $r_{\iota}(\pi'_{\mathcal{C}_w}) \in \mathcal{C}_w$ and $\pi'_{\mathcal{C}_w}$ is unramified outside $S$.

We fix $\mathcal{C}_w$. By Theorem \ref{BIGCL}, we obtain a conjugate self-dual cohomological cuspidal automorphic representation $\pi_{\mathcal{C}_w}$ of $\mathrm{GL}_2(\mathbb{A}_{F})$ such that $r_{\iota}(\pi_{\mathcal{C}_w}) \cong \rho_{\mathcal{C}_w}$. By using \cite[Theorem A]{ECG} \footnote{In this situation, we can prove this more easily by using the same method as \cite[Proposition 4.1.1 and Theorem 4.2.1]{CW}.}, there exists a lifting $\rho'_{\mathcal{C}_w} \in \mathrm{Spec}R_{\mathcal{D}^{\lambda, \tau, \emptyset}}(\mathcal{O})$ such that $\rho'_{\mathcal{C}_w}|_{G_{F_{w}}}$ is contained in $\mathcal{C}_w$ after extending $\mathcal{O}$ if necessary. Again by using Theorem \ref{BIGCL}, we obtain a conjugate self-dual cohomological cuspidal automorphic representation $\pi'_{\mathcal{C}_w}$ such that $r_{\iota}(\pi'_{\mathcal{C}_w}) \cong \rho'_{\mathcal{C}_w}$. This implies the result. \end{proof}

\begin{thm}(Breuil-M$\acute{e}$zard conjecture)\label{Breuil-Mezard}
    
Let $\overline{\rho} : G_{\mathbb{Q}_{p^2}} \rightarrow \mathrm{GL}_2(\overline{\mathbb{F}}_p)$ be a continuous representation such that any $k \in W_{\mathbb{Q}_{p^2}}(\overline{\rho})$ satisfies $2 \le k_{\gamma, 1} - k_{\gamma, 2} \le p-5$ for any $\gamma \in \mathrm{Hom}(\mathbb{F}_{p^2}, \overline{\mathbb{F}}_p)$. Then for any weight $\lambda$ and any inertia type $\tau$, the Breuil-M$\acute{e}$zard conjecture for $\overline{\rho}$, $\lambda$, $\tau$ holds.

\end{thm}

\begin{rem}

By the same method, we can prove the Breuil-M$\mathrm{\acute{e}}$zard conjecture in some $\mathrm{GL}_2(\mathbb{Q}_p)$ cases. However in the $\mathrm{GL}_2(\mathbb{Q}_p)$ case, the Breuil-M$\mathrm{\acute{e}}$zard conjecture was already completely solved by \cite{FMRI}, \cite{HUTU}, \cite{TunI} and \cite{TunII}.

\end{rem}

\begin{proof}

After twisting by some character, we may assume that $\lambda = (0, -\lambda_{\tau})_{\tau} \in (\mathbb{Z}_+^2)^{\mathrm{Hom}_{\mathbb{Q}_p}(\mathbb{Q}_{p^2}, \overline{\mathbb{Q}}_p)}$ for some $\lambda_{\tau} \ge 0$. By Theorem \ref{genericity}, we may assume that $\overline{\rho}$ is sufficiently generic reducible non-split. Then the result follows from Proposition \ref{BIGRT} because for the field $\mathbb{Q}_{p^2}$, Conjecture \ref{classicality conjecture} follows from Corollary \ref{classicality of geometric} and Conjecture \ref{key diagram} follows from Corollary \ref{parallel comparison} and Corollary \ref{non-parallel comparison}. \end{proof}

Let $F$ be a CM field such that $F_v = \mathbb{Q}_{p^2}$ or $\mathbb{Q}_p$ and $\Psi$ be the set of $p$-adic places $v$ such that $F_v = \mathbb{Q}_{p^2}$.

\begin{thm}(Automorphy lifting theorem) \label{automorphy lifting theorem}

    Let  $\rho : G_{F} \rightarrow \mathrm{GL}_2(\overline{\mathbb{Q}_p})$ be an irreducible continuous representation. We assume the following conditions.

    1 \ There exist a continuous character $\chi : G_{F^+} \rightarrow \overline{\mathbb{Q}_p}^{\times}$ satisfying $\chi(c_v) = \chi(c_w)$ for any complex conjugations $c_v, c_w$ at $v, w \mid \infty$ and a perfect $G_{F}$-equivariant symmetric pairing $\rho \times \rho^c \rightarrow \chi|_{G_{F}}$. 
    
    2 \ $\rho$ is unramified at almost all finite places.
    
    3 \ $\rho$ is de Rham and Hodge-Tate regular at all $p$-adic places.
    
    4 \ $\overline{\rho}|_{G_{F(\zeta_p)}}$ is absolutely irreducible.
    
    5 \ There exists an essentially conjugate self-dual cohomological cuspidal automorphic representation $\sigma$ of $\mathrm{GL}_2(\mathbb{A}_F)$ such that $\overline{\rho} \cong \overline{r_{\iota}(\sigma)}$.
    
    6 \ For any $v \in \Psi$ and for any $k \in W(\overline{\rho}|_{G_{F_v}})$, we have $2 \le k_{\tau, 1} - k_{\tau, 2} \le p-5$ for any $\tau \in \mathrm{Hom}(\mathbb{F}_v, \overline{\mathbb{F}}_p)$.
    
    Then there exists an essentially conjugate self-dual cohomological cuspidal automorphic representation $\pi$ of $\mathrm{GL}_2(\mathbb{A}_F)$ such that $r \cong r_{\iota}(\pi)$.
    
\end{thm}

    \begin{proof}

        By using the solvable base change, we may assume all conditions after Theorem \ref{genericity} as in \cite[proof of theorem 2.2.18]{FMRI} or \cite[proof of Theorem 5.2]{IA}. Thus this theorem follows from Proposition \ref{equivalence}, Theorem \ref{Breuil-Mezard} and the Breuil-M$\mathrm{\acute{e}}$zard conjecture in the $\mathrm{GL}_2(\mathbb{Q}_p)$ case. (See \cite{FMRI}, \cite{HUTU}, \cite{TunI} and \cite{TunII}.) \end{proof}

If $F$ is an imaginary quadratic, the assumption 5 follows from the assumption 1 by the following proposition and \cite{SEI}.

\begin{prop}\label{residual automorphy}

Let $F$ be a CM field and $\overline{\rho} : G_{F} \rightarrow \mathrm{GL}_2(\overline{\mathbb{F}}_p)$ be an irreducible continuous representation. Assume that there exists a continuous character $\overline{\chi} : G_{F^+} \rightarrow \mathbb{F}^{\times}$ such that $\overline{\chi}(c_v) = \overline{\chi}(c_w)$ for any $v, w \mid \infty$ and there exists a perfect symmetric $G_{F}$-equivariant pairing $\overline{\rho} \times \overline{\rho}^c \rightarrow \overline{\chi}|_{G_{F}}$. Then there exists a continuous representation $\overline{r} : G_{F^+} \rightarrow \mathrm{GL}_2(\mathbb{F})$ such that $\overline{r}|_{G_{F}} = \overline{\rho}$ and $\mathrm{det}\overline{r}(c_v) = -1$ for any $v \mid \infty$. 

\end{prop}

\begin{proof}

See \cite[proof of Theorem 8.1]{ALTRR}. \end{proof}

    \begin{thm}(Potential automorphy theorem) \label{potential automorphy}
    
    Let $\rho : G_{F} \rightarrow \mathrm{GL}_2(\overline{\mathbb{Q}_p})$ be a continuous representation. We assume that $\rho$ satisfies the conditions of Theorem \ref{automorphy lifting theorem} except $5$.
    
    Then there exist a finite Galois CM extension of $E/F$ and a conjugate self-dual cohomological cuspidal automorphic representation $\pi$ of $\mathrm{GL}_2(\mathbb{A}_E)$ such that $\rho|_{G_E} \cong r_{\iota}(\pi)$.
    
    \end{thm}

    \begin{proof}

By \cite[Proposition A.6]{BM}, there exists a finite Galois CM extension $E/F$ such that $\rho|_{G_{E}}$ satisfies the conditions of Theorem \ref{automorphy lifting theorem}. Thus we obtain the result. \end{proof}

\tiny

\printbibliography

\end{document}